\documentclass[11pt]{amsart}

\usepackage{amsmath,amssymb,amsfonts,amsthm,graphicx}

\usepackage{pinlabel}
\usepackage[usenames,dvipsnames]{color}

\usepackage[usenames,dvipsnames]{color}

\newtheorem{dummy}{dummy}[section]
\newtheorem{lemma}[dummy]{Lemma}
\newtheorem{theorem}[dummy]{Theorem}

\newtheorem{corollary}[dummy]{Corollary}
\newtheorem{proposition}[dummy]{Proposition}

\theoremstyle{definition}
\newtheorem{definition}[dummy]{Definition}
\newtheorem{convention}[dummy]{Convention}
\newtheorem{notation}[dummy]{Notation}

\newtheorem{remark}[dummy]{Remark}

\newtheorem{construction}[dummy]{Construction}
\newtheorem{observation}[dummy]{Observation}

\newtheorem{property}{Property}

\newcommand{\Dee}{D}
\newcommand{\You}{U}
\newcommand{\Are}{R} 
\newcommand{\El}{L} 
 
\newcommand{\C}{\mathbb{C}}

\newcommand{\sgn}{\mbox{sgn}}

\newcommand{\dd}{\partial}

\newcommand{\Z}{\mathbb{Z}}
\newcommand{\R}{\mathbb{R}}
\newcommand{\A}{\mathcal{A}}

\newcommand{\cE}{\mathcal{E}}
\newcommand{\tL}{\tilde{L}}
\newcommand{\lchA}{\mathcal{A}_{\mathit{LCH}}}

\newcommand{\e}{\epsilon}
\newcommand{\sq}{\mathit{Sq}}

\newcommand{\thmAlg}{2.1~}%{X.X~}
\newcommand{\ea}{\epsilon_1}

\newcommand{\ec}{\epsilon_2}%{\textcolor{YellowOrange}{\epsilon_3}}
\newcommand{\ed}{\epsilon_3}%{\textcolor{OrangeRed}{\epsilon_4}}
\newcommand{\grad}{\delta}%{\textcolor{OrangeRed}{\delta}}

\newcommand{\tra}{\pitchfork}

\allowdisplaybreaks[1] %This allows the very long table and arrays to be broken into different pages.

\begin{document}

\title[Cellular LCH for sufaces, II]{Cellular Legendrian contact homology for surfaces, part II}

\thanks{We thank Tobias Ekholm for explaining his transversality argument.  The second author is supported by grant 317469 from the Simons Foundation. He thanks the Centre de Recherches Mathematiques for hosting him while some of this work was done.}

\author{Dan Rutherford}
\address{Ball State University}
\email{rutherford@bsu.edu}

\author{Michael Sullivan}
\address{University of Massachusetts, Amherst}
\email{sullivan@math.umass.edu}

\begin{abstract}
This article is a continuation of \cite{RuSu1}.  For Legendrian surfaces in  $1$-jet spaces, we prove that the Cellular DGA defined in \cite{RuSu1} is stable tame isomorphic to the Legendrian contact homology DGA.
\end{abstract}

\maketitle

{\small \tableofcontents}

\section{Introduction}

This article is a continuation of \cite{RuSu1}.  Together the two articles provide a formulaic computation of the Legendrian contact homology (LCH) differential graded algebra (DGA)  for any Legendrian surface in a $1$-jet space.
% for which all differentials are given by explicit matrix formulas.  
In \cite{RuSu1} we compute the LCH DGA for several families of examples, and in \cite{RuSu3} we apply our results to study augmentations and generating families.

Let $S$ be a surface, and let $L \subset J^1(S)$ be a closed
%\footnote{\dr{4-11: Do we have anything to say about the non-closed case?} \ms{7/13/16: added fgi case}} 
Legendrian surface in the $1$-jet space of $S.$  
%(Our results could apply to non-compact Legendrians with ``finite-geometry at infinity," a sufficient condition for Gromov compactness of pseudo-holomorphic curves, however, we do not do this here.)
In \cite{RuSu1}, we defined the   cellular DGA of $L$, denoted here as $(\A_{\mathit{cell}},\partial)$.  The definition of $(\A_{\mathit{cell}},\partial)$ requires as input a cellular decomposition of the base projection of $L$ that contains in its $1$-skeleton the projection of 
%the singular set of $L$, i.e. 
%in its $1$-skeleton contains
all crossings arcs, cusp edges, and swallowtail points from the front projection.   For each cell, we have one generator of $\A_{\mathit{cell}}$ for each pair of sheets above the cell that and do not cross one another.  After collecting generators into upper triangular matrices, the differential $\partial$ is characterized by simple matrix formulas as summarized in Figure \ref{fig:ReebChords}.  (Additional terms appear when swallowtail points are present.)  

We have established in Theorem 4.1 of \cite{RuSu1} that $(\A_{\mathit{cell}},\partial)$ is determined, up to stable tame isomorphism, by the Legendrian $L$, i.e., it does not depend on the cellular decomposition or other choices involved in the definition.  
In the present article, we prove the following:
\begin{theorem} \label{thm:CellularLCH}  Let $L \subset J^1(S)$ be a closed Legendrian surface.  
The cellular DGA of $L$ is stable tame isomorphic to the Legendrian contact homology DGA of $L$ with coefficients in $\Z/2$.
\end{theorem}
In particular, the cellular DGA is an invariant of the Legendrian isotopy class of $L$.  We remark that it would be interesting to have an independent, low-tech proof of invariance that skirts the theory of holomorphic curves.  

\subsubsection{Outline of proof}
The proof of Theorem \ref{thm:CellularLCH} occupies this entire article, and  
%Given this background the argument is entirely self contained.  
% While the idea of the proof is quite straightforward, the execution is somewhat lengthy.  
we now provide an outline of the argument.  In this introduction, we denote the LCH DGA  as $(\mathcal{A}_{\mathit{LCH}}, \partial)$.

%and take the time to consider one illustrative case in detail.  

%In this introduction, we denote the LCH DGA  as $(\mathcal{A}_{\mathit{LCH}}, \partial)$.  
Our starting point is the work of Ekholm \cite{Ekholm07} that allows the differential in $(\mathcal{A}_{\mathit{LCH}}, \partial)$, that was originally defined in \cite{EkholmEtnyreSullivan05b} and \cite{EkholmEtnyreSullivan07} by a count of holomorphic disks, to be computed via a count of rigid {\it gradient flow trees} (abbrv. GFTs).  In Section \ref{sec:Background}, we recall relevant background about the LCH DGA and GFTs.  
%In Section \ref{sec:transverse}, 
In Section \ref{sec:transverse}, we construct a cellular decomposition of $S$ into squares, $\mathcal{E}_\pitchfork$, 
%subject to several restrictions.   In particular, the  $1$-skeleton of $\mathcal{E}_\pitchfork$ is transverse to $\pi_x(L)$, and 
so that, in the front projection of $L$, the form of the singular set above each square matches one of 14 standard types.  (See Figure \ref{fig:generators}, below.)  %Unlike the cellular decompositions used in defining the cellular DGA, the  $1$-skeleton of $\mathcal{E}_\pitchfork$ is transverse to the singular set of $L$.  
%;, and hence $\mathcal{E}_\pitchfork$ is not suitable from the point of view of the cellular DGA.  Accordingly, 
In addition, we construct a related cellular decomposition, $\mathcal{E}_\parallel$, that is suitable from the point of view of the cellular DGA.

For computing LCH, we modify $L$ by a Legendrian isotopy to a Legendrian $\tilde{L}$ that has several properties that aid in the explicit enumeration of rigid GFTs. Most importantly, the sheets of $\tilde{L}$ are pinched together above the $1$-skeleton of $\mathcal{E}_\pitchfork$  with the pinching most exaggerated above the $0$-skeleton.  As a result, above each cell of dimension $0 \leq d \leq 2$, each pair of sheets are connected by a Reeb chord that when viewed as a critical point of the difference of $z$-coordinates of sheets has index $d$.  (Some additional Reeb chords that we call exceptional generators appear as well and are discussed below.)  Labeling Reeb chords between the $i$-th and $j$-th sheet above $0$-cells, $1$-cells, and $2$-cells in the form $a_{i,j}$, $b_{i,j}$, and $c_{i,j}$, gives a rough indication of the correspondence between generators of $\A_{\mathit{cell}}$ and $\mathcal{A}_{\mathit{LCH}}$ underlying the isomorphism of Theorem 1.1.  See Figure \ref{fig:ReebChords}.
%the cellular DGA and LCH DGA that are local mins, saddle points, and maxs that appear above a cell of $\mathcal{E}_\pitchfork$ as $S_1, \ldots, S_n$
%As a result, critical points of differences of local defining functions for $\tilde{L}$, have local minima and saddles appearing above neighborhoods of $0$-cells and $1$-cells 
As an additional consequence of the pinching near the $1$-skeleton, 
the GFTs used in computing the differential of $(\mathcal{A}_{\mathit{LCH}}, \partial)$  become localized and must remain entirely within a neighborhood of the square that they begin in.  The Reeb chords that appear above suitable neighborhoods of any closed cell from $\mathcal{E}_\pitchfork$ then generate a sub-DGA, so that computation of $(\mathcal{A}_{\mathit{LCH}}, \partial)$ may be carried out on a square-by-square basis.

\begin{figure}
\labellist
%\small
%\pinlabel $1/4$ [t] at 338 21
\small
\pinlabel $C=(c_{i,j})$  at 170 192
\pinlabel $B_U$ [b] at 170 350
\pinlabel $B_L$ [r] at 0 232
\pinlabel $B_R$ [l] at 328 232
\pinlabel $B_D$ [t] at 170 20
\pinlabel $A_{+,+}$ [l] at 328 330
\pinlabel $A_{+,-}$ [l] at 328 36
\pinlabel $A_{-,+}$ [r] at 0 330
\pinlabel $A_{-,-}$ [r] at 0 36
\pinlabel $c_{i,j}$ [l] at 640 240
\pinlabel $b^R_{i,j}$ [l] at 820 240
\pinlabel $a^{-,-}_{i,j}$ [bl] at 504 46
\endlabellist

\centerline{ \includegraphics[scale=.5]{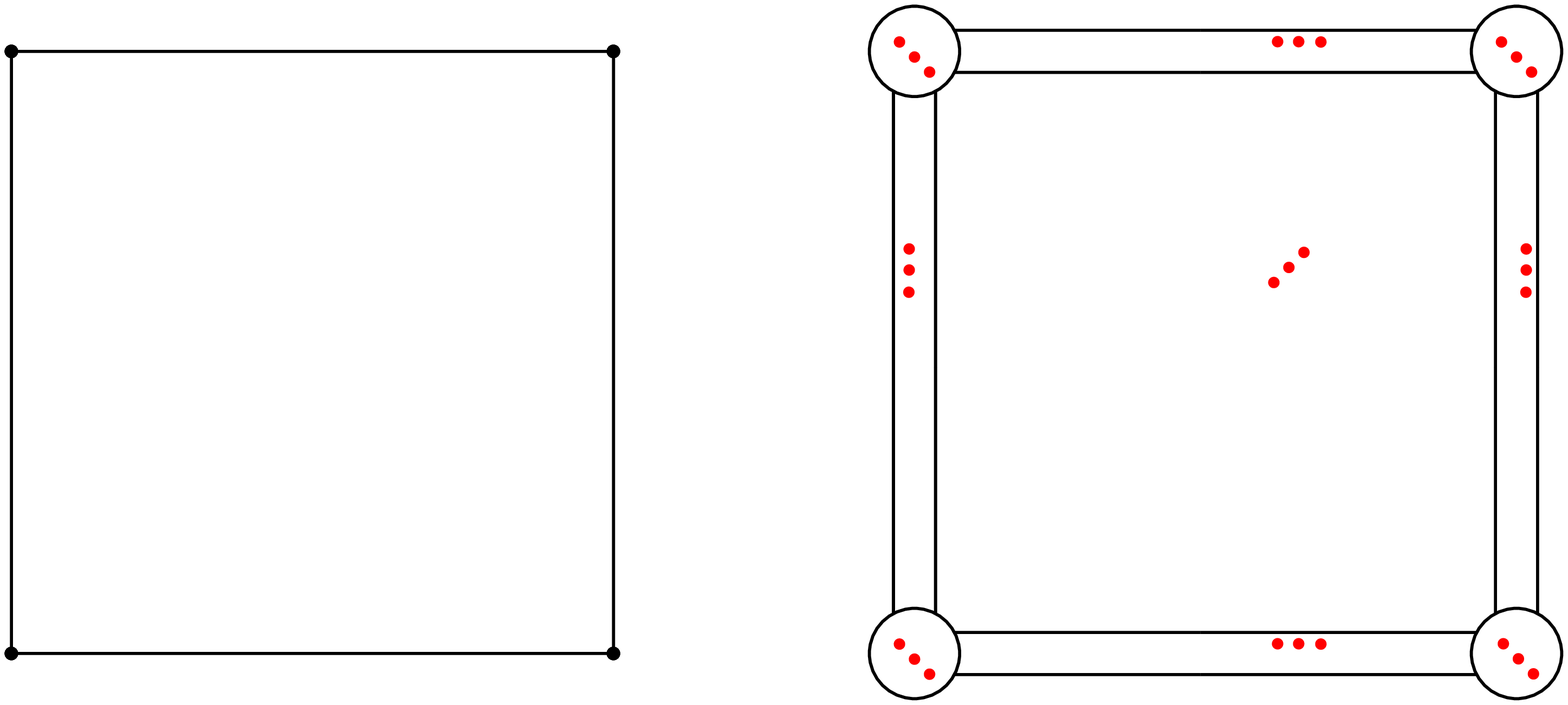}  }

\quad

$\partial A= A^2$;  \quad \quad $\partial B = A_+(I+B) +(I+B)A_-$;

$\partial C = A_{-,-}C + CA_{+,+} + (I+B_U)(I+B_L) + (I+B_R)(I+B_D)$.
\caption{(left)  Matrix formulas for the differential in the Cellular DGA.  (right) Geometrically, the generators of the Cellular DGA correspond to Reeb chords $a_{i,j}$, $b_{i,j}$, and $c_{i,j}$ that are local minima, saddle points, and local maxima located in neighborhoods of the $0$-, $1$-, and $2$-cells of a decomposition of $S$ into squares.}

\label{fig:ReebChords}
\end{figure}

The enumerations of GFTs required for computation of the sub-DGAs corresponding to individual cells of $\mathcal{E}_\pitchfork$ are carried out in Sections  \ref{sec:CompLCH}-\ref{sec:SwallowComp}.  A common argument applies for all squares that do not contain swallowtail points.  To orient the reader for this general computation, the logically independent Section \ref{sec:3.5} outlines in detail the computation for two of the most basic square types with technical details simplified or ignored.  %orient the reader to the general proof, we present a computation of the differential in a square with out crossing or cusp a simplified 
In Section \ref{sec:CompLCH}, we compute the sub-DGAs associated to $0$-cells and $1$-cells.  In Section \ref{sec:Comp2Cells}, we compute in a uniform manner the sub-DGAs of all $2$-cells that do not contain swallowtail points.  In Section \ref{sec:SwallowComp}, we address squares with swallowtail points.

To be more precise, in Sections \ref{sec:CompLCH}-\ref{sec:SwallowComp} we only give a partial calculation of the differential in $(\mathcal{A}_{\mathit{LCH}}, \partial)$.  
 The reason for this is that 
the presence of crossing arcs above a square forces  
 additional Reeb chords that we call {\it exceptional generators} to appear.  In turn, these exceptional generators lead to GFTs that can be quite complicated, so that a complete calculation  would require special arguments for many of the different square types.  To keep the computation as uniform as possible, in Sections \ref{sec:CompLCH}-\ref{sec:SwallowComp} we only compute $\partial$ up to terms belonging to an ideal $I$ involving these exceptional generators.  Aside from these terms, for each of the first 12 square types the differential is given by a common matrix formula where certain matrix entries are replaced with $0$'s and $1$'s in a way that is determined by the location of crossing arcs and cusp edges.  The final two square types contain swallow tail points, and as a result differentials in these square types have a form that is genuinely distinct from the other twelve square types.

%To complete the proof of Theorem, 
To avoid having to explicitly identify the remaining terms involving exceptional generators, we show in the first half of Section \ref{sec:Iso}, that the quotient of $(\A_{LCH},\partial)$ by the differential ideal $I$ is stable tame isomorphic to $(\A_{LCH},\partial)$.  Moreover, our partial calculations from Sections \ref{sec:CompLCH}-\ref{sec:SwallowComp} give a complete calculation of the differential in $\A_{LCH}/I$.  Section \ref{sec:Iso} is then concluded by establishing Theorem \ref{thm:CellularLCH} via a stable tame isomorphism between $(\A_{LCH}/I)$ and the cellular DGA.

%up to terms that involve certain Reeb chords that we call {\it exceptional generators}.  
%However, an exact calculation 
%The presence of 
%the crossing arcs force 
% additional Reeb chords to appear.  
%  that do not appear above simpler square types.  
% that do not appear above squares without crossing arcs.  
%Especially, in squares with more than one crossing arc, 
%The extra Reeb chords then allow for GFTs that can be quite complicated, especially in squares with more than one crossing arc, and  
% for squares with more than one crossing arc.  
%identifying all of these GFTs would require special arguments for each of the different square types.
%., and would likely produce a  calculation of the differential highly involved.  

%To avoid doing this, we %In actuality, in Sections \ref{sec:CompLCH}-\ref{sec:SwallowComp}, we only give a calculation of the differential up to terms that involve certain Reeb chords that we call {\it exceptional generators}.  

%instead identify a suitable stable tam

%show in the first half of Section \ref{sec:Iso}, that the quotient of $(\A_{LCH},\partial)$ by a suitable ideal $I$ involving the exceptional generators is stable tame isomorphic to $(\A_{LCH},\partial)$.  Moreover, our partial calculations of differentials from Sections \ref{sec:CompLCH}-\ref{sec:SwallowComp} give a complete calculation in $\A_{LCH}/I$.  Section \ref{sec:Iso} is then concluded by establishing Theorem \ref{thm:CellularLCH} via a stable tame isomorphism between $(\A_{LCH}/I)$ and the cellular DGA.

The final Sections \ref{sec:Constructions}-\ref{sec:ProofSetup} give a rather explicit coordinate model for $\tilde{L}$  that is used for the calculation of   $(\A_{\mathit{LCH}},\partial)$   in the earlier sections of the paper. 
%The GFTs are trees  .  In carrying out the square-by-square computation of LCH, we make use of several properties of gradient vector fields that are arranged via Legendrian isotopy.  To prove that the properties can all simultaneously be obtained, we construct a rather explicit coordinate model for $L$ in the sections.  
This construction is technical (although elementary).  To allow the reader to understand the computation of LCH without worrying about the details of this coordinate construction, we have listed 
%in Sections \ref{sec:CompLCH}-\ref{sec:SwallowComp} 
those properties of $\tilde{L}$ that are required for the arguments as  Properties \ref{pr:14models}-\ref{pr:SwitchBarriers} appearing in Sections \ref{sec:CompLCH}-\ref{sec:SwallowComp}.  %We believe that all of these properties are quite reasonable, and 
We expect that most readers will not find it surprising that a Legendrian satisfying Properties \ref{pr:14models}-\ref{pr:SwitchBarriers} exists.  However, as many properties involving the critical points and gradient vector fields of local defining functions for $\tilde{L}$ and their difference functions are demanded simultaneously,  we did not feel that the existence of $\tilde{L}$ was entirely obvious.  
%This is why we have included the detailed construction.  
An outline of the construction appears at the start of Section \ref{sec:Constructions}.  

%To illustrate the isomorphism as well as some of the challenges that need to be addressed, we outline the computation for two of the simplest square types with many technical details simplified or ignored.  

\subsubsection{Strategies for reading.}  For a first reading, the details of the construction from Section \ref{sec:Constructions}-\ref{sec:ProofSetup} may be omitted.  The reader may also reduce the length of the proof by considering only the case of Legendrians without swallowtail points.  This eliminates Sections \ref{sec:SwallowComp} and \ref{sec:ConstructionsST}, and simplifies the isomorphism in Section \ref{sec:Iso} as well as parts of Section \ref{sec:Properties}.  In fact, for some purposes considering only Legendrian surfaces without swallowtail points may suffice.  By work of Entov \cite[Corollary 1.8 and Section 1.8]{Entov99}, assuming the base surface $S$ and Legendrian $L$ are orientable, it is always possible to remove all swallowtail points from $L$ via a Legendrian isotopy.  
%Thus, in applications where $L$ only needs to be considered if one  for such $S$ and $L$ Thus, if one only needs to know that there is some Legendrian .   
However, in practice, it can be difficult to explicitly find such an isotopy for a given Legendrian, and the number of generators of the cellular DGA  may be greatly increased by removing swallowtail points.  Thus, aside from removing any assumptions on the orientability of $L$ and $S$, the extra work involved to include swallow tail points is justified in order to make the cellular DGA a flexible tool for computations. For example, a configuration of four swallowtail points can collapse to a cone point, whose DGA we compute in Section 5.3 of \cite{RuSu1}, and the cone points appear in the work-in-progress by Casals and Murphy on Lefschetz fibrations.

\section{Background}  \label{sec:Background}

We assume familiarity with \cite[Sections 2 and 3]{RuSu1}.  In particular, the reader should be familiar with the generic appearance of the base and front projection of a Legendrian surface in a $1$-jet space; the latter may include {\bf cusp edges}, {\bf crossing arcs}, and (upward and downward) {\bf swallowtail points}.  We refer to this subset of the front projection (cusp edges, crossing arcs and swallowtail points) and sometimes its base projection as the {\bf singular set}.  The definition of the {\bf cellular DGA} from \cite[Section 3]{RuSu1} as well as algebraic results on \textbf{DGAs} (differential graded algebras) and {\bf stable tame isomorphism} as found in \cite[Section 2]{RuSu1} are required for the proof of Theorem \ref{thm:CellularLCH}.  
In this section, we recall background on Legendrian contact homology, focusing on Ekholm's work on gradient flow trees which is fundamental for this article.

Let $S$ be a surface.  Denote the $1$-jet space of $S$ as  $J^1S = T^*S \times \R$.  Define the {\bf front projection} $\pi_{xz}: J^1S \rightarrow J^0S = S \times \R,$ the {\bf base projection} $\pi_{x}: J^1S \rightarrow S$ and the {\bf Lagrangian projection} $\pi_{xy}: J^1S \rightarrow T^*S.$ 
In coordinates $(x_1,x_2,y_1,y_2,z) \in J^1S$ arising from a choice of local coordinates $(x_1,x_2)$ on $S$, the standard contact structure on $J^1 S$ is the kernel of the contact form $dz -y\,dx := dz - y_1 \,dx_1 - y_2\, dx_2.$  Let $\omega$ be the standard symplectic structure on $T^*S,$ where locally $\omega = dx \wedge dy.$

\subsection{LCH}
Let $L \subset J^1S$ be a Legendrian surface.  The Legendrian contact homology (LCH) of $L$ is the homology of a differential graded algebra (DGA) denoted $(\mathcal{A}_{LCH}, \partial)$ whose definition is based on the theory of $J$-holomorphic curves.  The LCH DGA was first introduced in \cite{EliashbergGiventalHofer00} and was then rigorously defined in \cite{EkholmEtnyreSullivan05b, EkholmEtnyreSullivan07}.  In reviewing the definition, it is natural to use $\Z/2[H_1(L)]$-coefficients and introduce a $\Z$-grading on $\mathcal{A}_{LCH}$, but in the remainder of the article we will specialize to coefficients in $\Z/2$ (all group ring generators are replaced with $1$) and collapse the grading modulo the Maslov number of $L$. 

\subsubsection{Reeb chords and the algebra $\mathcal{A}_{LCH}$}

%The restriction of $\pi_{xy}$ to $L$ is an exact Lagrangian immersion, which means that $\int y\,dx$ vanishes on the image of $\pi_1(L)$ projected to $\pi_{xy}(L).$
The Reeb vector field $X_\alpha$ associated to a contact form $\alpha$ is defined as the unique vector field
which satisfies $\alpha(X_\alpha) = 1$ and  $d\alpha(X_\alpha, \cdot)=0.$ 
For our contact form on $J^1S$, the Reeb vector field is  $\partial_z.$
The {\bf{Reeb chords}} $R(L)$ for our Legendrian $L$ are the (non-constant) flows lines of the Reeb vector field that begin and end on the Legendrian.

Assuming that $L$ has a generic front projection, we can write $L = L_0 \sqcup L_1 \sqcup L_2$ where $L_i \subset L$ has codimension $i$ and the front projection $\pi_{xz}|_L$ 
satisfies:
\begin{itemize}
\item $\pi_{xz}|_L$ is an immersion when restricted to $L_0$;
\item $\pi_{xz}|_L$ has a cusp edge singularity along $L_1$;
\item $\pi_{xz}|_L$ has (upward or downward) swallowtail singularities at points in $L_2$.
\end{itemize}
Any $p \in L_0$ has a neighborhood $U \subset L$ that is the $1$-jet of a function $F: \pi_x(U) \rightarrow \R$.  We say that $F$ is a {\bf local defining function} for $L$ at $p$.  
The  Reeb chords $R(L)$ are in one-to-one correspondence with the critical points $F_i - F_j$ of positive critical value, where $F_i$ and $F_j$ are two such local defining functions for $L.$  [The critical points are in the intersection of the domains of $F_i$ and $F_j.$]
As an additional generic assumption on $L$ the front projection of $L$, we assume that the critical points of the $F_i-F_j$ are Morse.

The restriction of $\pi_{xy}$ to $L$ is an exact Lagrangian immersion.  Reeb chords are also in one-to-one correspondence with the double points of $\pi_{xy}(L)$, which are transverse due to the previous assumption on the $F_i-F_j$.

The underlying algebra $\mathcal{A}_{LCH} = \mathcal{A}_{LCH}(L)$ of the LCH DGA is defined to be the free unital associative (non-commutative) over the group ring $\Z/2[H_1(L)]$ with generating set $R(L)$.  
 This is the tensor algebra of the $\Z/2[H_1(L)]$-module generated by $R(L).$  

%%%%%OLD REWRITTEN FOR MULTI-COMPONENT CASE
%We review the (Maslov) grading of a Reeb chord and the (minimal) Maslov number of the Legendrian, as formulated in \cite[Section 3]{EkholmEtnyreSullivan05a}.
%Let $\gamma \in C^\infty([0,1], L)$ be such that $\gamma([0,1]) \subset L_0$ except for at isolated points
%where $\gamma$ intersects $L_1$ transversely.
%Define $D(\gamma)$ (resp. $U(\gamma)$) to be the number of intersections of $\gamma$ with $L_1$ where $\pi_{xz}(\gamma)$
%(oriented by $[0,1]$) goes from the upper (resp. lower) sheet to the the lower (resp. upper) sheet of the pair of cusping sheets. 
%Let $\mu(\gamma) := U(\gamma) - D(\gamma).$
%Define the {\bf{Maslov number}} of $L,$ $\mu(L),$  to be the minimum non-negative integer value achieved by $\mu(\gamma)$ over all loops $\gamma \in C^\infty(S^1, L).$
%Given a Reeb chord $c \in R(L),$ let $c \cap L = \{ c^+, c^-\}$ where $c^+$ has the greater $z$-coordinate of the two.
%Choose any $\gamma_c \in C^\infty([0,1], L),$ called a capping path for $c,$ such that $\gamma_c(0) = c^-$ and $\gamma_c(1) = c^+.$
%Let $c_S = \pi_x(c^\pm) \in S.$ 
%Define\footnote{\dr{Are we doing the $\Z$ grading (which we may as well do if we use $H_1$ coeff.)? Changed $\Z/m(L)\Z$ to $\Z$ in this definition.}} the {\bf{grading}} of $c,$ $|c| \in \Z$ by
%$$
%|c| = \mu(\gamma_c) + \mbox{Ind}(c_S, F_i-F_j) -1
%$$
%where $F_i$ and $F_j$ are local defining functions for the sheets containing $c^+$ and $c^-$, and $\mbox{Ind}(c_S, F_i-F_j)$ denotes the Morse index of $c_S$ as a critical point of $F_i - F_j$.

%\dr{REWRITE FOR MULTI-COMPONENT CASE}
\subsubsection{The grading on $\mathcal{A}_{LCH}$}

We review the (Maslov) grading of a Reeb chord and the (minimal) Maslov number of the Legendrian, as formulated in \cite[Section 3]{EkholmEtnyreSullivan05a}.
Let $\gamma \in C^\infty([0,1], L)$ be such that $\gamma([0,1]) \subset L_0$ except for at isolated points
where $\gamma$ intersects $L_1$ transversely.
Define $D(\gamma)$ (resp. $U(\gamma)$) to be the number of intersections of $\gamma$ with $L_1$ where $\pi_{xz}(\gamma)$
(oriented by $[0,1]$) goes from the upper (resp. lower) sheet to the the lower (resp. upper) sheet of the pair of cusping sheets. 
Let %\footnote{\dr{2-19:  Changed this formula by multiplication by $(-1)$ to match the KCH paper, and the 1-dim case, and correspondingly changed the capping path to be oriented from upper end point to lower end point. OK?  Maybe not.  This would change $\mu(A)$ in the formula for dim of $M(a;b_1,..)$ by a negative sign... Is there an error in KCH exposition?  OK.  Looking at the Etnyre-Ng-Sabloff paper where they do the $Z$-coefficients, and E-E-S, my impression is that computation of $\mu$ for a closed curve $A$ is $\mu(A) = D(A) -U(A)$ not the other way around.  So, I'll leave the change that I made for now.  Where did you get the $U(\gamma)- D(\gamma)$ definition from?} \ms{2/29/16: must have been a typo since it is always $D-U$ and not $U-D.$}} 
$\mu(\gamma) := D(\gamma) - U(\gamma).$
%If $\gamma(0) = \gamma(1),$ then $\mu(\gamma)$ agrees with the Maslov index of the path of  Lagrangian tangent planes along $\gamma,$ viewed as a loop in the (unoriented) Lagrangian Grassmanian.
%\footnote{\dr{Question:  Does this make sense for $L \subset J^1S$ if interpreted properly?  How do we identify all of the tangent spaces (or I guess the contact planes) with a single symplectic vector space?} \ms{you are correct. commented out sentence}}
Define the {\bf{Maslov number}} of $L,$ $\mu(L),$  to be the minimum non-negative integer value achieved by $\mu(\gamma)$ over all loops $\gamma \in C^\infty(S^1, L).$
%\footnote{\dr{Since we want to compare with the grading in the cellular DGA, defined in terms of a maslov potential, may be good to introduce a maslov potential at this point.} \ms{what is definition? do you choose base point as well as base point for each sheet and then connect with path?} \dr{2-17-16:  Maslov potential= function from complement of cusp set to $Z/\mu(L)$ that increases by 1 when you go from lower sheet to upper sheet at a cusp.  It is then clear that the formula you wrote becomes $|c| = \mu{c^+}- \mu{c^-} + Morse Index - 1$ in the 1-component case.  I think this is in your paper  E-E-S.  Should rewrite to allow multi-component case.  For multi-component case, fix a base point on each component, and choose paths from $c^\pm$ to the base point of the corresponding component.  This will be useful for fully non-commutative DGA anyway.   For details, look at a background section in a paper with Lenny, perhaps Knot Contact Homology.}} 

Denote the connected components of $L$ as $L = \bigsqcup_{i=1}^{\ell} L^i$, and choose base points $p_i \in L^i$.
Given a Reeb chord $c \in R(L)$, 
let $c \cap L = \{ c^+, c^-\}$ where $c^+$ has the greater $z$-coordinate of the two.  Supposing that $c^+ \in L^i$ and $c^- \in L^j$, we choose  generic {\bf base point paths} $\gamma^\pm_c \in C^\infty([0,1], L),$ such that $\gamma^\pm_c(0) = c^\pm$, $\gamma^+_c(1) = p_i$, and $\gamma^-_c(1) = p_j$, and let $\gamma_c = \gamma^+_c \cup -\gamma^-_c$ which is a single oriented path from $c^+$ to $c^-$ when $i=j$ and a disjoint union of oriented paths when $i \neq j$.    
%Let $c_S = \pi_x(c^\pm) \in S.$ 
Define the {\bf{grading}} of $c,$ $|c| \in \Z $ by
\begin{equation} \label{eq:pathgrading}
|c| = \mu(\gamma_c) + \mbox{Ind}(c, F_i-F_j) -1
\end{equation}
where $F_i$ and $F_j$ are local defining functions for the sheets containing $c^+$ and $c^-$ respectively, and $\mbox{Ind}(c, F_i-F_j)$ denotes the Morse index of $c$ as a critical point of $F_i - F_j$.

We extend the grading of Reeb chords to a grading of $\mathcal{A}_{LCH}$ as a direct sum of $\Z/2$-modules, $\mathcal{A}_{LCH}= \oplus_{k \in \Z} \mathcal{A}_{k}$, by requiring that for a homology class $A \in H_1(L)$, $e^A \in \Z/2[H_1(L)]$ has $|e^A| = -\mu(A),$
 and that $|x\cdot y| = |x| + |y|$ holds when $x$ and $y$ are homogeneous elements.

\begin{remark} \label{rem:gradingchange}
For Legendrian isotopy invariance of $(\mathcal{A}_{LCH}, \partial)$ in the multiple component case where $\ell \geq 2$, we need to allow for an overall grading change determined by a choice of integers $m = (m_1, \ldots, m_\ell) \in \Z^{\ell}$.  Given $m$, we obtain an alternate grading of Reeb chords by replacing each $|q|$ with $|q| + m_j - m_i$ when $q$ has lower endpoint on $L^j$ and upper endpoint on $L^i$.
\end{remark}

%\footnote{\dr{Check that this is compatible with the way the grading is discussed in Section 7... The grading is not discussed in Section 7 yet.  Will need to make sure everything is compatible with Article 1.}}  
Alternatively, a grading of Reeb chords by $\Z/ m(L)$ may be defined without reference to base point paths as follows.  Fix a choice of {\bf Maslov potential}, i.e. a locally constant function
$\mu:  L_0 \rightarrow \Z/m(L)$ whose value increases by $1$ when passing from the lower sheet to the upper sheet at cusp edges.  Then, 
\begin{equation}  \label{eq:gradingdef}
|c| = \mu(c^+) - \mu(c^-) + \mbox{Ind}(c, F_i-F_j) -1 \quad \quad (\mbox{mod $m(L)$}).
\end{equation}
Note that when $\Lambda$ is connected, the grading is independent of the choice of Maslov potential and is the reduction of the grading from (\ref{eq:pathgrading}) mod $m(L)$.    In the multi-component case, gradings resulting from different Maslov potentials are related to one another and to the grading from (\ref{eq:pathgrading}) as in Remark \ref{rem:gradingchange}.

%\dr{Can also throw in a note that to define the $Z$ grading with a Maslov potential, at least when $L$ is orientable,  we may fix embedded curves that represent a symplectic basis $\mu_1, \lambda_1, \ldots, \mu_g, \lambda_g$ for $H_1(L)$.  There removal cuts $L$ into a Legendrian disk for which a $\Z$ valued Maslov potential may be defined.  This agrees with the grading from (\ref{eq:pathgrading}) provided base point paths are chosen to be disjoint from the $\mu_i$ and $\lambda_i$.  Moreover, this allows us to compute the class $A$ that appears in the differential via intersections with the $\mu_i$ and $\lambda_i$...}

\subsubsection{The LCH differential}
For any non-negative integer $n,$ let $D_{n+1}$ be the space of closed unit disks in $\C$ centered on the origin, with
punctures $p_0, \ldots, p_n$ removed from the boundary in counter-clockwise order such that $p_0 =1,$ $p_1 = i,$ $p_2 = -1$ (if $p_1,p_2$ exist).
 Let $\mathcal{J}$ be the set of $\omega$-tame almost complex structures $\mathcal{J} := \{ J \in \mbox{End}(TT^*M)\,\, | \,\, J^2 = -Id, \omega (J v, v) \ge 0\},$
 with equality holding if and only if $v=0.$ 
 As a technical condition, we also need to assume for $J \in \mathcal{J}$ there exists some neighborhood of the double points in $\pi_{xy}(L)$ where $J$ looks like the standard complex structure of $\C^2$ \cite[page 3305]{EkholmEtnyreSullivan07}.
Let $J \in \mathcal{J}$ and let $i$ denote the complex structure on any $\Delta \in D_{n+1}$ inherited from $\C.$
Fix $a, b_1, \ldots, b_n \in R(L)$ with base point paths $\gamma^\pm_a, \gamma^\pm_{b_1}, \ldots, \gamma^\pm_{b_n}.$
Fix an element $A \in H_1(L).$
For $n \ge 0,$ let $\mathcal{M}(a; b_1, \ldots b_n; A)$ be the moduli space of equivalence class of pairs $(u,\Delta)$ such that the following hold.
\begin{itemize}
\item
The pair is comprised of $\Delta  \in D_{n+1}$ and $u:(\Delta, \partial \Delta) \rightarrow (T^*M, \pi_{xy}(L)).$
\item
The image under $u$ of a sequence of points $\{z_n\} \subset \Delta$ converging to the  puncture $p_0$ (resp. $p_j$)
 converges to the double point $\pi_{xy}(a)$ (resp. $\pi_{xy}(b_j)$).  
Moreover, the image of $u(\partial \Delta)$ near $p_0$ (resp. $p_j$) in the positive (resp. negative) oriented boundary direction runs from the branch (of the two branches that intersect at the double point $\pi_{xy}(a)$ (resp. $\pi_{xy}(b_j)$))
of $\pi_{xy}(L)$ with the lower $z$-coordinate lift in $L$ to the branch with the higher $z$-coordinate lift.
\item   
Combining the continuous lift of $u(\partial \Delta \setminus \{p_0, p_1, \ldots, p_n\})$ in $L$ with the base point paths $-\gamma^+_a\cup \gamma^-_a$ and  $\gamma^+_{b_i}\cup -\gamma^-_{b_i}$ produces a representative for $A \in H_1(L).$
%\footnote{\dr{2-19: For fully non-commutative case, we take $A= (A_0, A_1, \ldots, A_n)$ to be a sequence of homology classes with $A_i$ in the appropriate $H_1(L)$ and require that each $A_i$ is represented by the union of the lift of the portion of $\partial \Delta$ between $p_i$ and $p_{i+1}$, $\gamma_i$, with the (properly oriented) base point paths associated to the Reeb chord end points that $\gamma_i$ limits to.} }

\item
The map $u$ restricted to the interior of $\Delta$ is $J$-holomorphic, $J  \circ du = du \circ i.$
\item
Two pairs $(u,\Delta)$ and $(u',\Delta')$ are equivalent if there exists a biholomorphic map $\psi$ of the unit disk taking the punctures of $\Delta$ to those of $\Delta'$ such that $u = u' \circ \phi.$
\end{itemize}
Define the {\bf{formal dimension}} of the moduli space $\mathcal{M}(a; b_1, \ldots b_n; A)$ to be
$$\mbox{fdim} \mathcal{M}(a; b_1, \ldots b_n; A) = |a| - \sum_{j=1}^n |b_j| + \mu(A) -1.$$
For any $a, b_1, \ldots b_n, A$ such that $\mbox{fdim} \mathcal{M}(a; b_1, \ldots b_n; A) \le 1,$  there exists an open dense subset $ \mathcal{J}_{\mathit{reg}} \subset \mathcal{J}$ 
such that $\mathcal{M}(a; b_1, \ldots b_n; A)$ is a manifold, compact in the sense of Gromov (see \cite{EkholmEtnyreSullivan05b}, for example) whose dimension equals its formal dimension \cite{EkholmEtnyreSullivan07}.
So for  $J \in \mathcal{J}_{\mathit{reg}},$ $\mathcal{M}(a; b_1, \ldots b_n; A)$ is a finite set of points when $\mbox{fdim} \mathcal{M}(a; b_1, \ldots b_n; A) = 0$, cf. \cite{EkholmEtnyreSullivan05b, EkholmEtnyreSullivan07}.

%We can now discuss the LCH DGA.
%Let $R$ be the group ring $\Z/2[H_1(L)].$
%Define the associative algebra $\mathcal{A}_{LCH} = \mathcal{A}_{LCH}(L)$ to be the unital tensor algebra\footnote{\dr{This is probably correct as is, but usually we speak of the tensor algebra of a vector space or $R$-module, right?} \ms{corrected} \dr{I'm not sure.  Should ``vector space'' be replaced with ``$\Z/2[H_1(L)]$-module''?}} with $\Z/2[H_1(L)]$-coefficients, over the vector space generated by $R(L).$ We denote the unit by $1.$
%For any monomial $A x_1 \cdots x_l$ where $A \in H_1(L)$ and $x_j \in R(L),$ define the grading by 
%$|A x_1 \cdots x_l| = \mu(A) + \sum_j |x_j|.$ This implies $|1| = 0$ and makes $\mathcal{A}_{LCH}$ a graded algebra.
We can now define the  differential $\partial: \mathcal{A}_{LCH} \rightarrow \mathcal{A}_{LCH}$.  For Reeb chords $a \in R(L)$, set
$$\partial a = \sum_{ |e^Ab_1 \cdots b_n| = |a| -1} \sharp \mathcal{M}(a; b_1, \ldots, b_n; A) e^A b_1 \cdots b_n,$$
and then extend $\partial$ to all of $\mathcal{A}_{LCH}$ with the ($\Z/2[H_1(L)]$-linear) Leibniz rule.
Here $\sharp \mathcal{M}(a; b_1, \ldots, b_n; A)$ is a $\Z/2$-count.
If $n=0,$ then the empty word is $b_1 \cdots b_n = 1.$
\begin{theorem}[\cite{EkholmEtnyreSullivan05b, EkholmEtnyreSullivan07}]
The differential has degree $-1$ and satisfies $\partial^2=0.$
Moreover, if $L$ and $L'$ are Legendrian isotopic, then for any $J  \in \mathcal{J}_{\mathit{reg}}(L)$, $J' \in \mathcal{J}_{\mathit{reg}}(L')$ and any choice of base point paths, 
the DGAs $(\mathcal{A}_{LCH}, \partial) = (\mathcal{A}_{LCH}(L), \partial(L,J))$ and
$(\mathcal{A}_{LCH}', \partial') = (\mathcal{A}_{LCH}(L'), \partial(L',J'))$ are stable tame isomorphic (after possibly applying a  grading change as in Remark \ref{rem:gradingchange}).
\end{theorem}
In particular, the homology $H_*(\mathcal{A}_{LCH}, \partial),$ called the {\bf{Legendrian contact homology}}, is 
a well-defined Legendrian isotopy invariant.

\subsection{Gradient flow trees}
\label{ssec:GFT}

%\dr{Terminologies to define.  ``flow line, GFT, PFT, input/output vertices  ($e$-vertices are considered outputs), word of a GFT:  $w(\Gamma)$'',  $Y_0,Y_1, sw-, e-$ vertices.  The point that $Y_1$ and $sw$-vertices cannot appear at points where gradint points to the side of the cusp locus where fewer sheets exist.  Switch points, non-degenerate switch points.  Internal vertices, External vertices.}

In \cite{Ekholm07}, Ekholm showed that the count of homolorphic disks appearing in the definition of LCH may be replaced with a count of suitable gradient flow trees (abbrv. GFTs).  This is a significant extension of earlier results of Floer \cite{Fl}  and Fukaya-Oh \cite{FukayaOh} that relate holomorphic strips and disks to gradient trajectories and flow trees in the case of graphical Legendrians.  
In this article, we only consider GFTs with a single positive puncture.  This allows a definition of GFT that appears somewhat different from the general definition given in \cite{Ekholm07}, but more closely parallels the presentation of flow trees from \cite{FukayaOh}.  See Remark \ref{rem:Compare}.

We consider {\bf{rooted trees}}, i.e. trees with a finite number of edges and with a chosen $1$-valent {\textbf{initial vertex}}, $v_0.$ 
(For those familiar with \cite{Ekholm07}, we recast $2$-valent initial vertices as discussed later in Remark \ref{rem:2ValentMinMax}.)
 We orient the edges of a rooted tree so that all edges are oriented away from $v_0$.  Note that any other vertex $v \neq v_0$ has precisely one incoming edge that is oriented towards $v$ while any other adjacent edges are oriented away from $v$.  Vertices with valence $1$ (resp. valence greater than $1$) are called {\bf external} (resp. {\bf internal}).  Sometimes, we refer to $v_0$ as the {\bf input} vertex, and to the remaining external vertices as {\bf output} vertices.  Edges with both endpoints at internal vertices are {\bf internal}, while edges with at least one endpoint at an external vertex are {\bf external}.  

%The domains of GFTs are trees with the following additional structure.  
%\begin{definition}
A \textbf{metric tree} with $1$ input is a rooted tree $\Gamma$ equipped with the following additional structure:
\begin{enumerate}
\item Each edge $e$ is assigned a length $l(e) \in (0, +\infty]$ such that (i) all internal edges have finite length, and (ii) the initial edge that starts at $v_0$ has length $+\infty$.  
\item At each internal vertex an ordering of the outgoing edges is chosen.  
\end{enumerate}

%\end{definition}
To each edge, $e$,  of a metric tree, we assign an interval $[c,d] \subset [-\infty, +\infty]$ according to:
\begin{itemize}
\item The initial edge beginning at $v_0$ has $[c,d] = [-\infty, 0]$ unless it is the only edge of $\Gamma$ in which case $[c,d] = [-\infty, +\infty]$.
\item All other edges, $e$, have $[c,d] = [0, l(e)]$.
\end{itemize}

Fix a Riemannian metric $g$ on $S$.  
\begin{definition} \label{def:GFT}
A {\bf gradient flow tree} (abbrv.  {\bf GFT}) of $L$ (with respect to the metric $g$) is a metric tree $\Gamma$ together with for each edge $e$ a continuous map 
\[
\gamma: [c,d] \rightarrow S
\]
 together with a pair of $1$-jet lifts 
\[
\gamma^+: [c,d] \rightarrow L \quad \mbox{and} \quad \gamma^-: [c,d] \rightarrow L, \quad \pi_x \circ \gamma^\pm = \gamma
\] 
satisfying:
\begin{enumerate}
\item For all, $t \in (c,d), \, z(\gamma^+(t)) > z(\gamma^-(t))$.
% such that if $z(\gamma_e^+(t)) = z(\gamma_e^-(t))$ then $t = b$ and $\gamma_e^+(t) = \gamma_e^-(t) \in L_1$,
\item For all $t \in (c,d)$,  
 $\gamma'(t) = -\nabla(f^+-f^-)(\gamma(t))$ where $f^+$ and $f^-$ are local defining functions for $L$ at $\gamma^+(t)$ and $\gamma^-(t)$, respectively.
\item If $e$ is an internal edge, then $\gamma$ is non-constant.
\item At each internal vertex, $v$, write the edges adjacent to $v$ as $e_0, e_1, \ldots, e_{r}$, $r \geq 1$, where $e_0$ is the unique edge oriented into $v$ and $e_1, \ldots, e_{r}$ are the outgoing edges at $v$ appearing according to their order.  Notate corresponding edge maps as $\gamma_i: [c_i,d_i] \rightarrow S$, $0 \leq i \leq r$. We require that the $1$-jet lifts $\gamma^\pm_i$ 
%of the edge maps $\gamma_i: [a_i,b_i] \rightarrow S$, $0 \leq i \leq r$, 
fit together continuously at $v$ as follows
\[
\gamma^+_{0}(d_0) = \gamma^+_{1}(c_1);  \quad \gamma^-_{i}(c_i) = \gamma^+_{{i+1}}(c_{i+1}) \mbox{ for $1 \leq i \leq r-1$}; \quad \gamma^-_{{r}}(c_{r}) = \gamma^-_{0}(d_0).
\] 
(See Figure \ref{fig:GFTdef}.)
\item The initial edge, $e$, that begins at $v_0$, satisfies $\displaystyle \lim_{t \rightarrow -\infty} \gamma(t) = a$ where the Reeb chord $a$ is a critical point of $f_+ - f_-$.  
\item An external vertex at the end of an edge $e$, we either have
\begin{enumerate}
\item $l(e) = +\infty$ and $\displaystyle \lim_{t \rightarrow +\infty} \gamma(t) = b$ where the Reeb chord $b$ is a critical point of $f_+-f_-$, or 
\item the images $\gamma^+(b)=\gamma^-(b)$ agree and belong to a cusp edge of $L$.
\end{enumerate}
\end{enumerate}

A {\bf partial flow tree} (abbrv.  {\bf PFT}) is a rooted tree $\Gamma$ satisfying all of the properties of a GFT except that the initial edge starting at $v_0$ is parametrized with $[c,d]$ where $ -\infty < c < 0$, and no requirement is placed on the image at $c$ other than $z(\gamma^+(c)) > z(\gamma^-(c))$.  In particular, this input edge is no longer required to begin at a  Reeb chord.
\end{definition}

We introduce some additional terminology associated with a GFT or PFT $\Gamma$.  When the input vertex $v_0$  begins at the Reeb chord $a$ we say that the vertex $v_0$ is a {\bf positive puncture} at $a$.  When an output vertex $v$ limits to a Reeb chord $b$ as $t \rightarrow +\infty$, we say that $v$ is a {\bf negative puncture} at $b$.  We refer to output vertices that end at a point on the cusp edge of $L$ as {\bf $e$-vertices}.

Often, we will provide a labeling of the sheets of $L$ above some subset $U \subset S$ as $S_1, \ldots, S_n$, and denote the corresponding local defining functions by $F_1, \ldots, F_n$.  (We emphasize that such an enumerations of sheets of $L$ only makes sense locally.)  Typically, we shorten notation for the local {\bf difference functions} to 
\[
F_{i,j} := F_i-F_j.
\]
When the $1$-jet lifts, $\gamma^+$ and $\gamma^-$, of an edge $e$ of a GFT or PFT
%, with associated map $\gamma:[c,d] \rightarrow S$ $1$-jet lifts $\gamma^+$ and $\gamma^-$ 
have their images in $S_i$ and $S_j$ respectively, we say that $S_i$ (resp.  $S_j$) is the {\bf upper sheet} (resp. {\bf lower sheet}) of $e$.  Note that in this case $\gamma$ is a trajectory for $-\nabla F_{i,j}$, and we say that $\gamma$ is an {\bf $(i,j)$-flow line}.

\begin{observation}
\begin{enumerate}
\item[(i)] For any point $x$ belonging to the interior of one of the edges of a GFT or PFT $\Gamma$, there is a unique PFT whose initial vertex is $x$ that consists of the portion of $\Gamma$ that lies below $x$ (with respect to the orientation of edges of the domain tree of $\Gamma$).  Usually, we refer to this PFT as simply {\bf the PFT starting with $x$}.

\item[(ii)]  External edges may be mapped to Reeb chords as constant maps.  In fact this must be the case for GFTs that have a positive (resp. negative) puncture at a Reeb chord that is a local minimum (resp. maximum) of $f^+-f^-$.     

\item[(iii)]  The sheet difference $z(\gamma^+(t)) - z(\gamma^-(t))$ decreases not only along every non-constant edge (as a standard consequence of the negative gradient flow equation), but also when passing internal vertices (from the incoming edge to an outgoing edge).  Thus, if $x$ and $y$ are any points in the domain of a PFT $\Gamma$ with $x$ above $y$ (with respect to the orientation of edges of $\Gamma$), then the sheet difference at $x$ is larger than at $y$.

\end{enumerate}
\end{observation}

\begin{remark}  \label{rem:Compare}
\begin{enumerate}
\item[(i.)]  Our definition of PFT is much more restrictive than in \cite{Ekholm07}, where more than one exceptional vertex is allowed. 
\item[(ii.)]  In \cite{Ekholm07}, the definition of GFT does not allow constant external edges (mapped to Reeb chords).  However, \cite{Ekholm07} does allow punctures to occur at internal vertices, as we explicitly do not.   The equivalence of these conventions is easily seen:  A constant edge may be replaced with a puncture at an internal vertex in a unique way and the converse statement holds also.  
\item[(iii.)]  The condition $z(\gamma^+) > z(\gamma^-)$ is not included in the definition of GFT from Ekholm.  Indeed, this condition may fail for GFTs with more than one positive puncture.  However, Lemma 2.8 from \cite{Rizell11}, establishes this condition for GFTs with $1$ positive puncture.
%(Note Ekholm's definition of puncture.)  
%This convention is different from Ekholm's convention.  However, he allows punctures to occur at internal vertices.
\end{enumerate}
\end{remark}

\subsubsection{Words of GFTs and PFTs}  \label{sec:GFTPFT}

We define the {\bf word}, $w(\Gamma)$, of a PFT or GFT, $\Gamma$, using induction on the number of internal vertices.  If $\Gamma$ has no internal vertices then, its domain is a single edge with one output vertex, $v$.  In this case, declare
\[
w(\Gamma) = \left\{ \begin{array}{cr} b, & \mbox{if $v$ has a negative puncture at $b$}  \\
1,  & \mbox{if $v$ is an $e$-vertex.} \end{array} \right.
\]
Now, supposing $w(\Gamma)$ has been defined for PFTs with fewer internal vertices, we consider the first internal vertex, $v$, of $\Gamma$.  Let the outgoing edges at $v$ be ordered as $e_1, \ldots, e_r$, for some $r \geq 1$.  Let $\Gamma_1, \ldots, \Gamma_r$ denote the PFTs of $\Gamma$ that begin at interior points of the edges $e_1, \ldots, e_r$ respectively, and define
\[
w(\Gamma) = w(\Gamma_1) \cdots w(\Gamma_r).
\]
 
An alternate formulation of $w(\Gamma)$ is as follows:  The ordering of outgoing edges at internal vertices, allows us to embed the domain tree of $\Gamma$ into $D^2$ by requiring that the outgoing edges $e_1, \ldots, e_r$ at an internal vertex $v$ appear in counterclockwise order when we travel around $v$ starting at the incoming edge.  Moreover, we can arrange the external vertices to sit on $\partial D^2$, and their cyclic ordering is well defined.  We then take the product of all negative punctures, $b_1, \ldots, b_n$, as they appear in counter-clockwise order around $\partial D^2$, starting at $v_0$.  That is,
\[
w(\Gamma) = b_1 \cdots b_n.
\]

Yet another perspective on $w(\gamma)$ is the following:  
The union of the Lagrangian projections of the one jet lifts $\gamma^+$ and $-\gamma^-$  (the negative indicates orientation reverse) of all edges of $\Gamma$ fit together to give a continuous closed  curve, $C$, with image in $\pi_{xy}(L)$.  We call $C$ the {\bf boundary curve} of $\Gamma$.  See Figure \ref{fig:GFTdef}.  When instead considered as a curve on $L$, $C$ has discontinuities that are in bijection with the positive and negative punctures of $\Gamma$.  (The $z$-coordinate increases (resp. decreases) 
when passing a positive (resp. negative) puncture.)  From this perspective, $w(\Gamma)$ is the product of negative punctures as they appear along the curve $C$, when starting with the upper lift $\gamma^+$ of the initial edge of $\Gamma$.

When $\Gamma$ is a GFT with positive puncture at $a$ and $w(\Gamma) = b_1\cdots b_n$, we say that $\Gamma$ is a GFT {\bf from $a$ to $b_1, \ldots, b_n$}.

\begin{figure}
\labellist
%\small
%\pinlabel $1/4$ [t] at 338 21
\small
\pinlabel $\gamma^+$ [r] at 44 144
\pinlabel $-\gamma^-$ [l] at 90 144
\pinlabel $a$ [b] at 334 200
\pinlabel $b_1$ [t] at 254 12
\pinlabel $1$ [t] at 304 4
\pinlabel $b_2$ [t] at 344 -2
\pinlabel $b_3$ [t] at 386 4
\pinlabel $b_4$ [t] at 434 20
\pinlabel $e_0$ [b] at 72 174
\pinlabel $e_1$ [t] at -2 30
\pinlabel $e_2$ [t] at 72 36
\pinlabel $e_3$ [t] at 146 30
\endlabellist

\centerline{ \includegraphics[scale=.6]{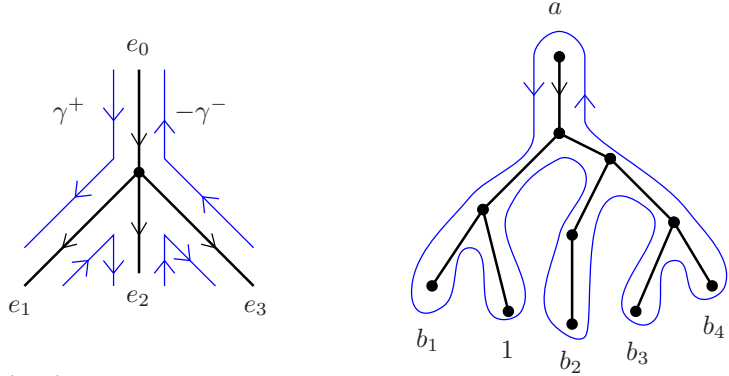}  }
\caption{  (left) As specified in Definition \ref{def:GFT}, the $1$-jet lifts of adjacent edges fit together continuously at internal vertices.  (right) A GFT with $w(\Gamma) = b_1b_2b_3b_4$.  All edges of $\Gamma$ are oriented downward, and the output vertex labelled with a $1$ is an $e$-vertex.  The $1$-jet lifts fit together as indicated to form a continuous closed curve, $C$, on $\pi_{xy}(L)$.
%The outputor each edge of $\Gamma$, picture the domain of the $1$-jet lifts $\gamma^+$ and $\gamma^-$ by shifting $e$ a small amount to the right and left (with respect to the orientation of $e$).   The resulting curve is continuous on $\pi_{xz}(L)$.
}
\label{fig:GFTdef}
\end{figure}

%Caption:  As specified in Definition ???, the $1$-jet lifts of adjacent edges fit together continuously at internal vertices and $e$-vertices.  The cotangent projections fit together  at external vertices.  For each edge of $\Gamma$, picture the domain of the $1$-jet lifts $\gamma^+$ and $\gamma^-$ by shifting $e$ a small amount to the right and left (with respect to the orientation of $e$).   The resulting curve is continuous on $\pi_{xz}(L)$.

%We define the {\bf word} of $\Gamma$,
%\[
%w(\Gamma) = b_1 b_2 \cdots b_n
%\]
%where the Reeb chords $b_1, \ldots, b_n$ are the negative punctures of $\Gamma$  output vertices appear in counterclockwise order starting from $v_0$ when the domain of $\Gamma$ is drawn in $D^2$ with exterior vertices on $\partial D^2$.

\subsection{GFTs and Holomorphic Disks}

The connection between GFTs and Legendrian contact homology is the following.

In Section 3 of \cite{Ekholm07}, Ekholm associates a {\bf formal dimension} to a GFT, $\Gamma$, from $a$ to $b_1, \ldots, b_n$ written, $\mathit{fdim}(\Gamma)$.  It is equal to the formal dimension of the moduli space of holomorphic disks $\mathcal{M}(a; b_1, \ldots, b_n; [\widehat{C}])$, where $\widehat{C}$ is the lift of the boundary curve of $\Gamma$ to $L$ (this is discontinuous at punctures) together with the base point paths $-\gamma^+_a\cup \gamma^-_a$ and  $\gamma^+_{b_i}\cup -\gamma^-_{b_i}$.  The pair consisting of the Legendrian and metric $(L, g)$ is {\bf $1$-regular} when (i) there are no GFTs with $\mathit{fdim}(\Gamma)<0$, (ii) all trees with $\mathit{fdim}(\Gamma)= 0$ are transversally cut out (i.e. they are obtained via iterated transverse intersection of certain suitably defined ascending and descending manifolds), and (iii) the set of trees with $\mathit{fdim}(\Gamma) = 0$ is finite.  (The ``$1$'' in $1$-regular refers to the number of positive punctures.)  When $(L,g)$ is $1$-regular, GFTs with formal dimension $0$ are called {\bf rigid}.

%all GFTs of formal dimension $\leq 0$ with $1$ positive puncture satisfy a suitable transversality condition (see \ref{sec:GFTTran} below), and 
It is shown in \cite[Theorem 1.1 (a)]{Ekholm07} that any $(L,g)$ can be made $1$-regular after a $C^\infty$-small perturbation.  For the Legendrians that we consider we will use a mild strengthening of this statement; see Proposition \ref{prop:LemmaA1}.

\begin{theorem}[\cite{Ekholm07}, Theorem 1.1]  \label{thm:EkholmMain}
Suppose that $(L,g)$ is $1$-regular.  There exists a Legendrian $\widehat{L} \subset J^1S$ such that
\begin{enumerate}
\item $\widehat{L}$ is Legendrian isotopic to $L$, and has its Reeb chords canonically identified with those of $L$, and
\item there exists $J \in \mathcal{J}_{\mathit{reg}}$ for $\widehat{L}$ such that when $\mathit{fdim}(\mathcal{M}(a; b_1, \ldots, b_n; A)) =0$, disks in $\mathcal{M}(a; b_1, \ldots, b_n; A)$ for $\widehat{L}$ are in bijection with the set of (equivalence classes of) rigid GFTs for $L$ from $a$ to $b_1, \ldots, b_n$  whose boundary curves satisfy $[\widehat{C}] = A$. 
\end{enumerate}
In particular, the differential in the LCH DGA of $L$ may be computed by
\[
\partial a = \sum_{\Gamma} [\widehat{C}] \cdot w(\Gamma)
\]
where the sum is over all rigid GFTs for $L$ that begin at $a$.
\end{theorem}

\subsection{Generic behavior at vertices}

In Section 3 of \cite{Ekholm07}, it is shown that when $(L,g)$ is $1$-regular 
%satisfying a suitable transversality condition (see ???, below) 
the internal vertices that appear in rigid GFTs are limited to the following types:

\begin{itemize} 
\item  A {\bf $Y_0$-vertex}, $v$, is a $3$-valent vertex with image in $S$ disjoint from the cusp locus of $L$.  At $v$, there exist sheets $S_i, S_m, S_j$ with defining functions satisfying $F_i(v) > F_m(v) > F_j(v)$, such that the incoming edge $e_0$ at $v$ is an $(i,j)$-flow line while the outgoing edges $e_1$ and $e_2$ are respectively $(i,m)$- and $(m,j)$-flow lines.  

\item  A {\bf $Y_1$-vertex}, $v$, is a $3$-valent vertex with image at a point of $S$ belonging to the cusp locus of $L$ such that the sheets corresponding to edges adjacent to $v$ are as follows.  Near $v$ there are sheets $S_i, S_{k}, S_{k+1}, S_j$ of $L$ such that  sheets $S_k$ and $S_{k+1}$ meet at a cusp edge so that $S_k$ has larger $z$-coordinate than $S_{k+1}$;  $S_i$ sits above the cusp edge; and $S_j$ sits below the cusp edge.  The incoming edge $e_0$ is an $(i,j)$-flow line while the out going edges $e_1$ and $e_2$ are respectively $(i,k+1)$- and $(k,j)$-flow lines.  In addition, $-\nabla F_{i,j},  -\nabla F_{i,k+1}$  and $-\nabla F_{k,j}$ are all transverse to the cusp locus of $L$ at $v$, and necessarilly point to the side of the cusp locus where $S_k$ and $S_{k+1}$ exist. 

\item  A {\bf $\mathit{sw}$-vertex} (also called a {\bf switch vertex}), $v$, is a $2$-valent vertex with image at the base projection of a cusp edge where sheets $S_k$ and $S_{k+1}$ meet with $S_k$ the upper sheet of the cusp edge.  In addition, one  of the following holds:
\begin{itemize}
\item[(i)] The incoming (resp. outgoing) edge at $v$ is an $(i,k)$-flow line (resp. an $(i,k+1)$-flow line) for some sheet $S_i$ that sits above the cusp edge at $v$.  Note that  $-\nabla F_{i,k}= -\nabla F_{i,k+1}$ must be tangent to the cusp locus at $v$. 
\item[(ii)] The incoming (resp. outgoing) edge at $v$ is a $(k+1,j)$-flow line (resp. a $(k,j)$-flow line) for some sheet $S_j$ that sits below the cusp edge.  Note that $-\nabla F_{k,j}= -\nabla F_{k+1,j}$ is tangent to the cusp locus at $v$.  
\end{itemize}
\end{itemize}
See Figure \ref{fig:Y0}.

We use the terminology {\bf switch point} for a point along the projection of a cusp edge between sheets $S_{k}$ and $S_{k+1}$  where a third sheet $S_{i}$ or $S_{j}$ as in (i) or (ii) has $-\nabla F_{i,k}$ or $-\nabla F_{k,j}$ tangent to the projection of the cusp edge. 
A switch point is {\bf non-degenerate} when the tangency of the restriction of the corresponding $-\nabla F_{i,k}$ or $-\nabla F_{k,j}$ to the cusp locus has order $1$.  Of course, switch vertices can only have their images at switch points.

\begin{figure}
\labellist
%\small
%\pinlabel $1/4$ [t] at 338 21
\large
\pinlabel $\mathbf{Y_0:}$ [r] at -50 102
\small
\pinlabel $-\nabla F_{i,j}$ [l] at 50 126
\pinlabel $-\nabla F_{i,m}$ [br] at 14 50
\pinlabel $-\nabla F_{m,j}$ [bl] at 62 50
\pinlabel $S_i$ [l] at 388 154
\pinlabel $S_m$ [l] at 388 90
\pinlabel $S_j$ [l] at 388 26
\endlabellist

\centerline{ \includegraphics[scale=.5]{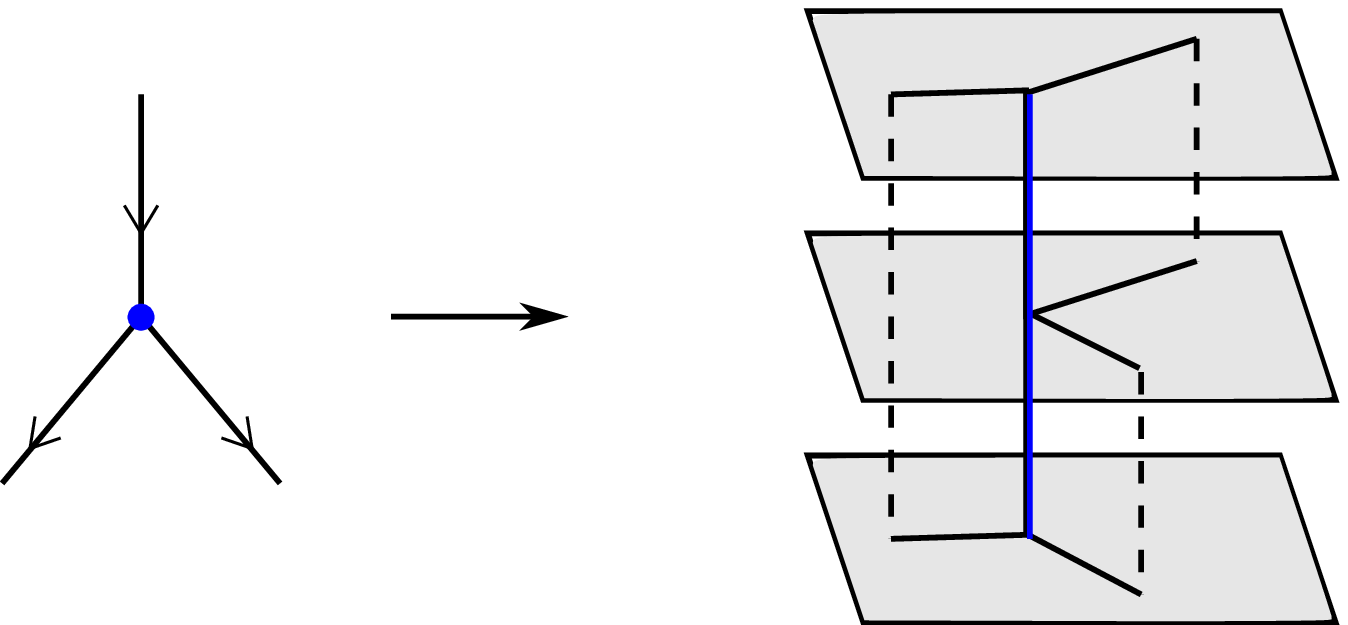}   }

\bigskip

\labellist
%\small
%\pinlabel $1/4$ [t] at 338 21
\large
\pinlabel $\mathbf{Y_1:}$ [r] at -50 102
\small
\pinlabel $-\nabla F_{i,j}$ [l] at 50 126
\pinlabel $-\nabla F_{i,k+1}$ [br] at 14 50
\pinlabel $-\nabla F_{k,j}$ [bl] at 62 50
\pinlabel $S_i$ [l] at 398 154
\pinlabel $S_{k}$ [l] at 398 108
\pinlabel $S_{k+1}$ [l] at 398 60
\pinlabel $S_j$ [l] at 398 26
\endlabellist

\centerline{ \includegraphics[scale=.5]{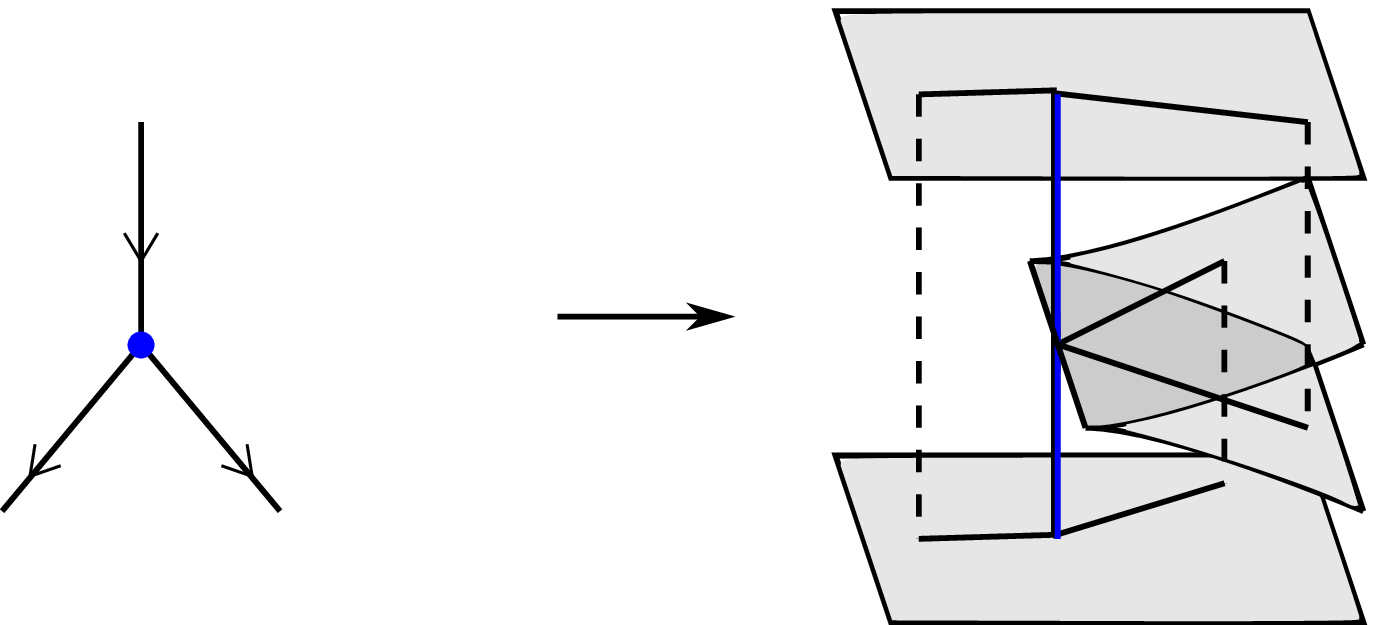}  }

\bigskip

\labellist
%\small
%\pinlabel $1/4$ [t] at 338 21
\large
\pinlabel $\mathbf{SW:}$ [r] at -60 102
\small
\pinlabel $-\nabla F_{k+1,j}$ [l] at 18 88
\pinlabel $-\nabla F_{k,j}$ [l] at 18 44
%\pinlabel $-\nabla F_{k,j}$ [bl] at 62 50
%\pinlabel $S_i$ [l] at 398 154
\pinlabel $S_{k}$ [l] at 360 108
\pinlabel $S_{k+1}$ [l] at 360 60
\pinlabel $S_j$ [l] at 360 18
\endlabellist

\centerline{ \includegraphics[scale=.5]{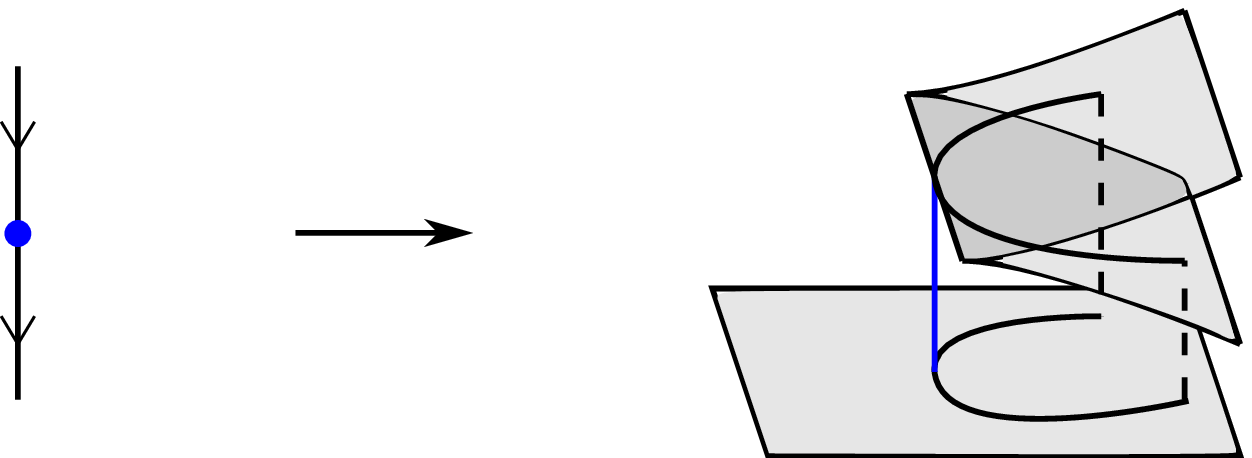}  }
\caption{  The generic internal vertices pictured via their domain trees and $1$-jet lifts.
%The outputor each edge of $\Gamma$, picture the domain of the $1$-jet lifts $\gamma^+$ and $\gamma^-$ by shifting $e$ a small amount to the right and left (with respect to the orientation of $e$).   The resulting curve is continuous on $\pi_{xz}(L)$.
}
\label{fig:Y0}
\end{figure}

%A generic condition on $(L,g)$ from \cite{Ekholm}, that we call \emph{GFT-transversality} is defined below in Section ???.

\begin{theorem}[\cite{Ekholm07}]   \label{thm:1regV}
Assume $(L,g)$ is 1-regular.  Then, all internal vertices of rigid GFTs must be either $Y_0$-, $Y_1$-, or $\mathit{sw}$-vertices.  (External vertices may be punctures or $e$-vertices.) 
\end{theorem}

\begin{proof} 
In Section 3 of \cite{Ekholm07}, GFTs are assigned a geometric dimension, $\mathit{gdim}(\Gamma)$.  
%, such that rigid trees must have $\mathit{gdim}(\Gamma) = \mathit{fdim}(\Gamma) =0$.  (The meaning of being transversally cut out for GFTs is implicit from the proof of 
\cite[Proposition 3.14]{Ekholm07} shows that when $\Gamma$ is transversally cut out, a space of nearby GFTs with similar geometric properties, eg. homeomorphic domain trees with corresponding branches having the same upper and lower sheet in $L$, is a manifold of dimension $\mathit{gdim}(\Gamma)$.  In particular, a transversally cut out tree has $\mathit{gdim}(\Gamma) \geq 0$.  Moreover, \cite[Lemma 3.7]{Ekholm07} shows that in general $\mathit{gdim}(\Gamma) \leq \mathit{fdim}(\Gamma)$, so that a rigid tree must have $\mathit{gdim}(\Gamma) = \mathit{fdim}(\Gamma) =0$.  Possible vertices of trees with  $\mathit{gdim}(\Gamma) = \mathit{fdim}(\Gamma)$ are listed in \cite[Lemma 3.7]{Ekholm07}.
%, with all internal vertices shown to be $Y_0$-, $Y_1$-, or $\mathit{sw}$-vertices.          
\end{proof}

\begin{remark}
\label{rem:2ValentMinMax}
In \cite[Lemma 3.7]{Ekholm07}, $2$-valent vertices with a positive (resp. negative) puncture at a local minimum (resp. local maximum) are also listed as possible internal vertices.  However, with our conventions for GFTs, such punctures become $Y_0$'s with an adjacent external edge mapped to a constant trajectory at the corresponding critical point.  See Remark \ref{rem:Compare} (ii).
\end{remark}

Moreover, for GFTs with these generic vertices, formal dimension may be computed from the count of vertices of various types and the Morse index of punctures.

\begin{proposition} \label{prop:EY1SWFormula}  Suppose that $(L,g)$ is $1$-regular.  Let $\Gamma$ is a rigid GFT with a postive puncture at $a$ and negative punctures at $b_1, \ldots, b_m$.  Let $E$, $Y_1$, and $SW$  denote the number of $e-$, $Y_1-$, and $sw$-vertices of $\Gamma$ respectively, and $\mbox{Ind}(a)$  and $\mbox{Ind}(b_i)$ denote the Morse indices of $a$ and $b_i$ as critical points of  local difference functions $F_i-F_j$ with $F_i > F_j$.  
Then, we have
\begin{equation} \label{eq:EY1SWFormula}
2 - \mbox{Ind}(a) = \sum_{i=1}^m \left(1 - \mbox{Ind}(b_i)\right) +  E - Y_1 - SW.
\end{equation}
\end{proposition}

\begin{proof}  
The first formula for the formal dimension, $\mathit{fdim}(\Gamma)$, from \cite[Definition 3.4]{Ekholm07}  reads
\[
\dim(\Gamma) = (n-3) + \sum_{p \in P(\Gamma)}(I(p)-(n-1)) - \sum_{q \in Q(\Gamma)}(I(q)-1) + \sum_{r\in R(\Gamma)} \mu(r).
\]
Comparing the notation there with ours, $\dim(\Gamma) = \mathit{fdim}(\Gamma) = 0,$ 
$n=2,$ $\sum_{p \in P(\Gamma)}(I(p)-(n-1)) = \mbox{Ind}(a)-1,$ 
$\sum_{q \in Q(\Gamma)}(I(q)-1) = \sum_{i =1}^m(\mbox{Ind}(b_i) -1) $ and 
 $\sum_{r\in R(\Gamma)} \mu(r)=E - Y_1 - SW.$
The equation (\ref{eq:EY1SWFormula}) follows.  
\end{proof}

\section{Transverse Square Decompositions}  \label{sec:transverse}

Let $L \subset J^1S$ be a Legendrian with generic front projection and base projection, and let $\Sigma$ denote the singular set  (cusps, crossings and swallowtails) of the front projection of $L$.
% of the singular set $\Sigma \subset L$ is generic.  
In this section, we construct a cell decomposition of $S$ into squares.  The $1$-skeleton will be transverse to $\pi_x(\Sigma)$, so we will denote the decomposition as $\mathcal{E}_\pitchfork$.  Above each square of $\mathcal{E}_\tra$, $\Sigma$ is required to match one of a finite collection of standard forms that are introduced in Section \ref{ssec:ElemSq} below, and a few additional technical requirements are imposed on $\mathcal{E}_\tra$ for later use.  In addition, we construct a related cellular decomposition $\mathcal{E}_{||}$ that will be used in Section \ref{sec:Iso} when working with the Cellular DGA.    
%In Sections \ref{sec:Constructions} and \ref{sec:ConstructionsST}, we will construct a Legendrian, $\widetilde{L}$, given by explicit coordinate models in each square so that $\widetilde{L}$ is Legendrian isotopic to $L$.    

\subsection{Regularity requirements}  \label{sec:RegReq}

% In order for the square-by-square definition of $\tilde{L}$ to fit together smoothly to produce a globally defined Legendrian, 
We impose some regularity requirements on our square decomposition.
In this section we use the notation $I = [-1,1]$.

For the transverse cell decomposition $\mathcal{E}_\pitchfork= \{e^i_\alpha\}$ constructed below, we always require the following conditions.  Note that these requirements may be obtained from an arbitrary polygonal decomposition (as in Section 3.1 of \cite{RuSu1}) into squares by altering the original characteristic maps first near $0$-cells and then in a neighborhood of $1$-cells.

\begin{enumerate}
\item We have characteristic maps
\[
c^i_\alpha : I^i \stackrel{\cong}{\rightarrow} \overline{e^i_\alpha} \subset S
\] 
which are homeomorphisms and are smooth  with smooth inverse except 
%that the differential is not invertible\footnote{Possibly not differentiable at all at the corners since the polar coordinate map may not be smooth at $r=0$.  Not important, but should be stated correctly.} 
possibly at the corners of $2$-cells. (Compare with requirement (2) below). 

\item Each $0$-cell $e^0_\alpha$ has a disk coordinate neighborhood $U \subset S \rightarrow D^2$ so that if $N$ distinct $2$-cells meet at $e^0_\alpha$ then near the corner each of their characteristic maps rescales the angle.  More precisely, for a given $2$-cell $e^2_\beta$ containing $e^0_\alpha$ as a corner, we use polar coordinates on $I^2$ centered at the corner of $I^2$ corresponding to $e^0_\alpha$ with the angle coordinate rotated so that $0 \leq \theta \leq \pi/2$ parametrizes a neighborhood of this corner.  We require that the characteristic map for $e^2_\beta$ takes $\{ 0 \leq r \leq 1/8\}$ into $U  \cong D^2$ and has the form 
%near the corner of $I^2$, 
\[
c(r,\theta) = \left(r, \pm\frac{4}{N}\theta + (m 2 \pi)/N\right)
\]
for $0 \leq r \leq 1/8$, for some $0 \leq m < N$.

\item Along each $1$-cell, the characteristic maps of the $2$-adjacent faces may be glued together to give smooth maps $(-\epsilon, \epsilon) \times (-1,1) \rightarrow S$.
\end{enumerate}

\subsubsection{Where do the technical requirements come into play?}

In Sections \ref{sec:Constructions}-\ref{sec:ProofSetup}, we will define a front projection for a Legendrian $\tilde{L}$ that is isotopic (but not literally equal to) the original front projection of $L$.  This is done by providing local defining functions $F_i: I^2 \rightarrow \R$ for the sheets of $\tilde{L}$ above each square using coordinates given by the characteristic maps.  This explains (1).  Near corners each local defining function will have the form $a r^2 + b$ in the  polar coordinate system centered at the corner of $I^2$ as described in (2) with the coefficients $a$ and $b$ agreeing for sheets in distinct squares that meet at the $0$-cell.  Therefore, the requirements from (2) show that the front will be smooth in a neighborhood of a $0$-cell.  Moreover, the construction of sheets will be standard near edges as well, so that (3) will guarantee the smoothness of the pieced together front above $1$-cells.  

In addition, the conditions allow the following:

\begin{construction} \label{construct:metric} Let $L \subset J^1S$ be a Legendrian with local defining functions given by $a r^2+b$ in the coordinate neighborhoods of $0$-cells prescribed in (2).  There exists a metric on $S$ with respect to which the gradients of local defining functions (and hence also local difference functions) agree with the gradients taken in each square with respect to the standard Euclidean metric on $I^2$.  
\end{construction}
\begin{proof}
The Euclidean metrics on squares piece together to give a metric on the complement of the $0$-cells of $S$.  In a $D^2$ coordinate neighborhood (as in (2)) of a $0$-cell adjacent to $N$ faces, the metric  appears in polar coordinates as $dr^2 + (\frac{4}{Nr})^2 d\theta^2$.  In order to extend smoothly to $r=0$, we alter the metric inside  $D^2$ to $dr^2 + [(1- \alpha(r))(\frac{4}{Nr})^2 + \alpha(r) (\frac{1}{r})^2]d\theta^2$ where $\alpha(r)$ is a bump function that is $0$ outside values of $r$ where all local defining functions have the form $a r^2 + b$ and is $1$ in a neighborhood of $0$.   

This change to the metric does not alter gradients of functions of the form $f(r,\theta) = g(r)$.  
\end{proof}

\subsection{Elementary squares}  \label{ssec:ElemSq}

In Figure \ref{fig:generators}, some standard forms for the base projection of the crossing and cusp loci of a Legendrian $L$ to a square $[-1,1] \times [-1,1]$ are presented and divided into Types (1)-(14).    
%The boundary edges of these squares are always oriented from the lower left corner to the upper right corner.  
Note
that parametrizations of squares may be orientation reversing (supposing $S$ is oriented, which we do not require in general) so that the images of the cusp and crossing locus in $S$ may appear reflected across $x_1=x_2$ in some squares.

%that reflections across the line $y=x$ are also allowed and will be denoted as Types (1')-(14').  

For some edges of the squares of type (5), (6), (8) or (12), a pair of arcs of the crossing locus intersect a single edge.  Along such an edge either the $x_1$ or $x_2$-coordinate goes from $-1$ to $1$ while the other coordinate remains fixed.  Denote the coordinate that varies by $x_i$.
%In these cases, we always require that when such an edge is parametrized from the lower left corner to the upper right corner of the square, i.e. with increasing $x_1$ or $x_2$-coordinate,  
We require that the two crossings that appear along the restriction of $L$ to that edge arise from a strand with larger $z$-coordinate at $x_i=-1$ passing through two consecutive strands with lesser $z$-coordinate at $x_i=-1$ as $x_i$ increases.  See the image labelled (2Cr) in Figure \ref{fig:EdgeTypesB}.  Thus, when viewed with this prescribed orientation, above each edge of a square of Type (1)-(14) the singular set $L$ matches $1$ of the  $4$ forms pictured in Figure \ref{fig:EdgeTypesB}.

We refer to a square in $S$ above which $L$ has one of the types (1)-(14) 
%or (1')-(14') 
as an {\bf elementary square} for $L$.

\begin{figure}
%\centerline{\includegraphics[scale=.4]{images/SVanilla}  \quad \includegraphics[scale=.4]{images/SCusp} \quad \includegraphics[scale=.4]{images/SCr}
%\quad \includegraphics[scale=.4]{images/SCrCorner}  \quad \includegraphics[scale=.4]{images/SCr2}} 

\centerline{$\begin{array}{cccccc} \includegraphics[scale=.4]{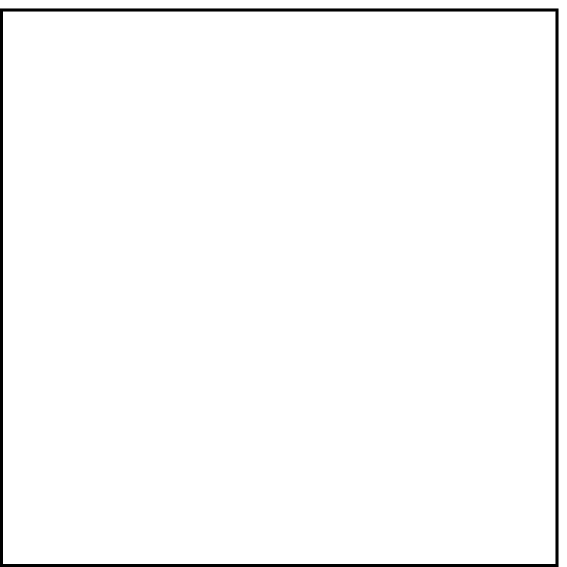}  & \includegraphics[scale=.4]{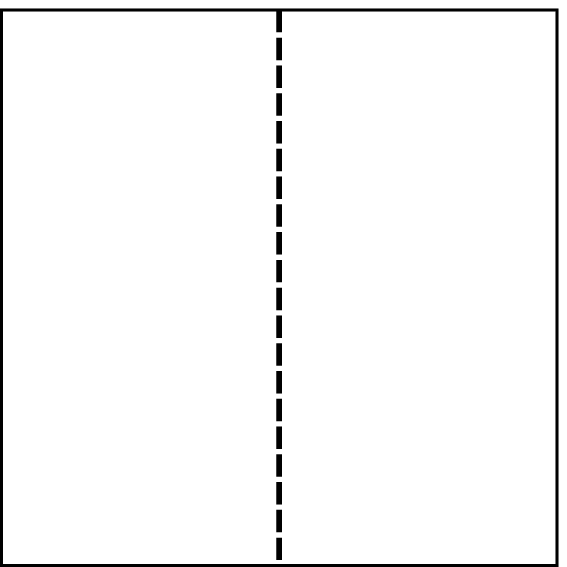} & \includegraphics[scale=.4]{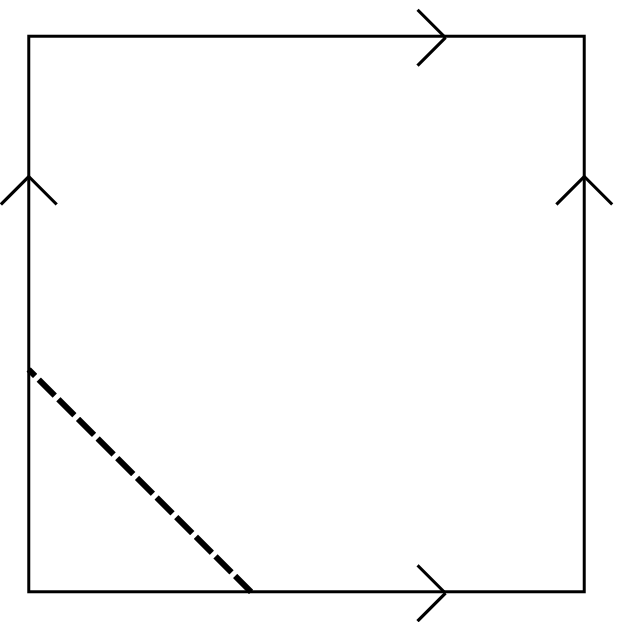}
& \includegraphics[scale=.4]{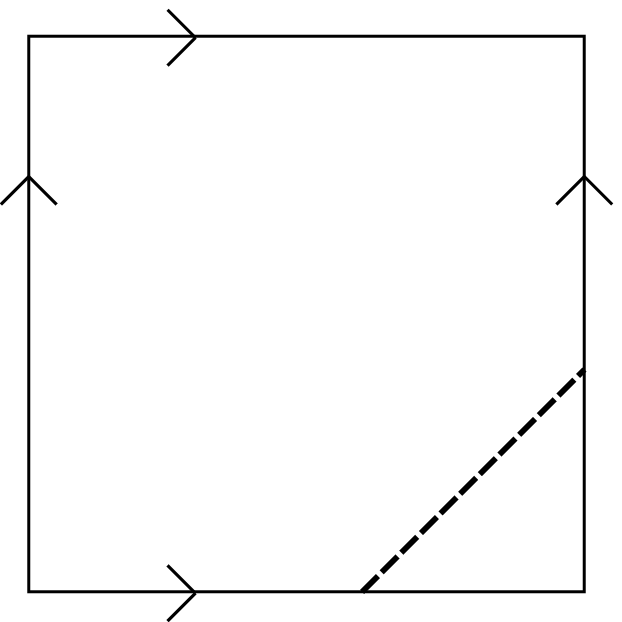}  & \includegraphics[scale=.4]{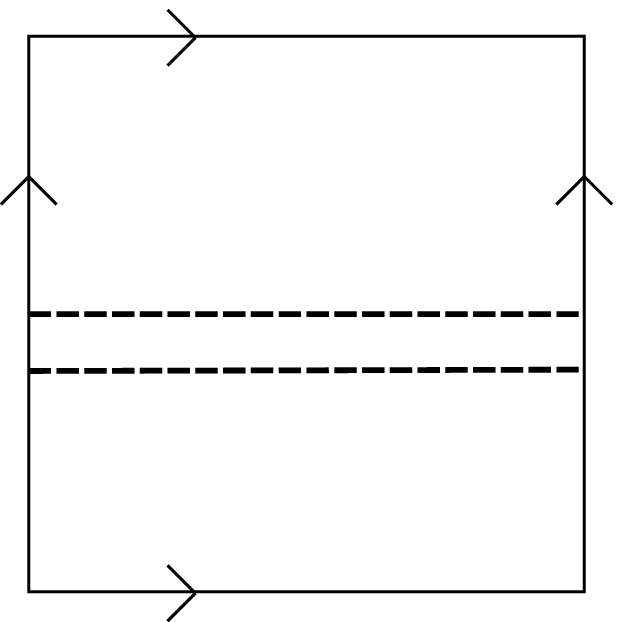} & \includegraphics[scale=.4]{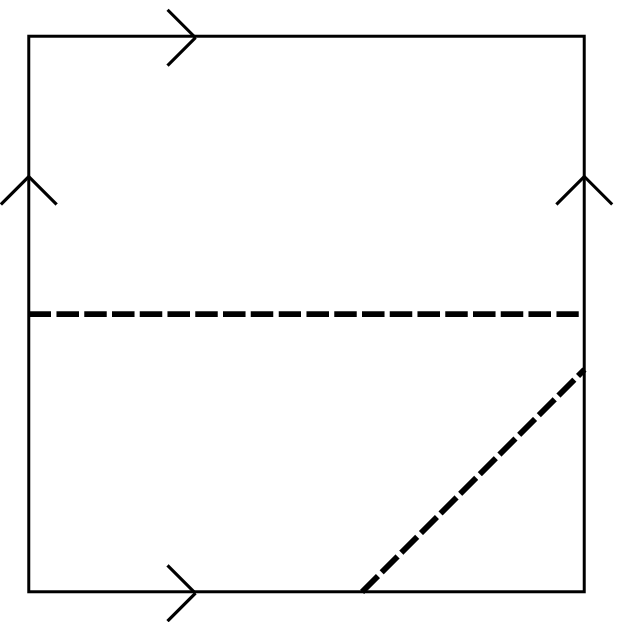} \\
 \mbox{(1)} & \mbox{(2)} & \mbox{(3)} & \mbox{(4)} & \mbox{(5)} & \mbox{(6)} \end{array}$}

\bigskip

\centerline{$\begin{array}{cccccc}  \includegraphics[scale=.4]{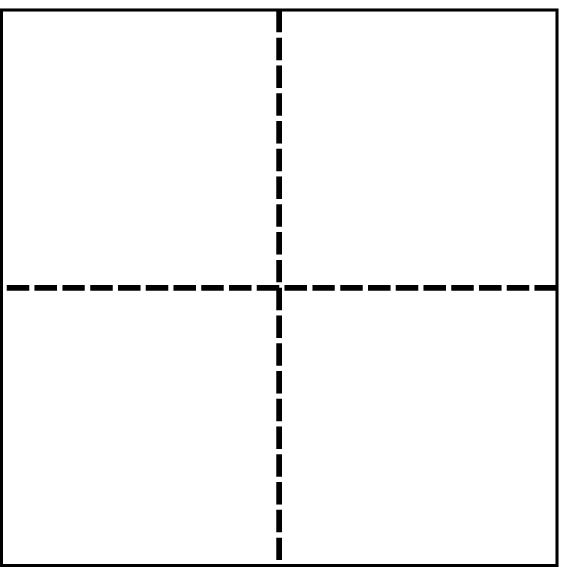} & \includegraphics[scale=.4]{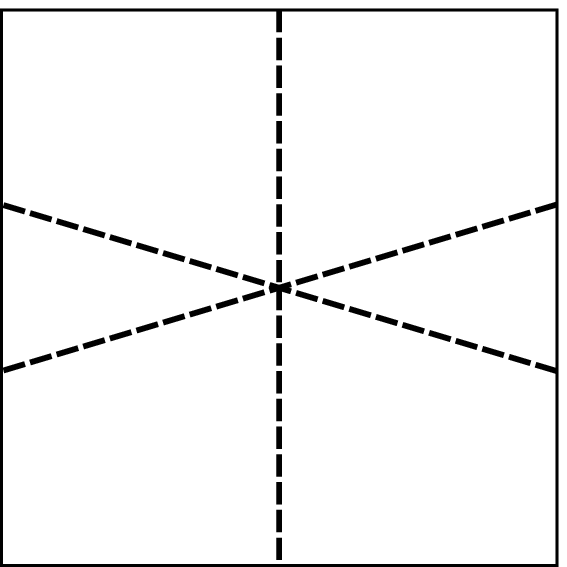} & \includegraphics[scale=.4]{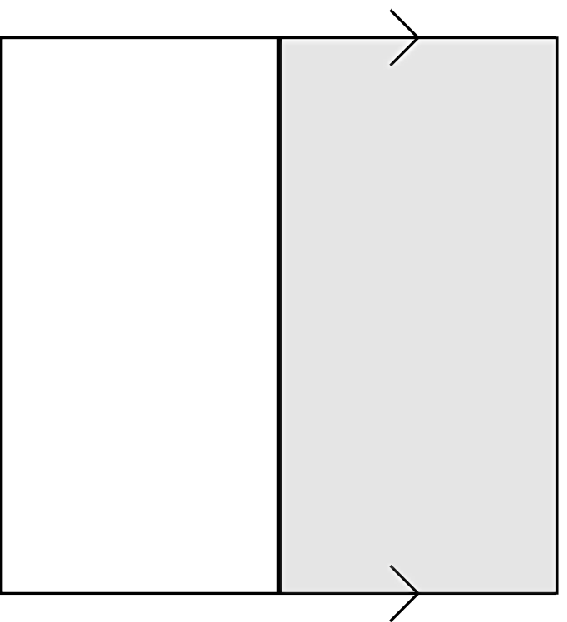} & \includegraphics[scale=.4]{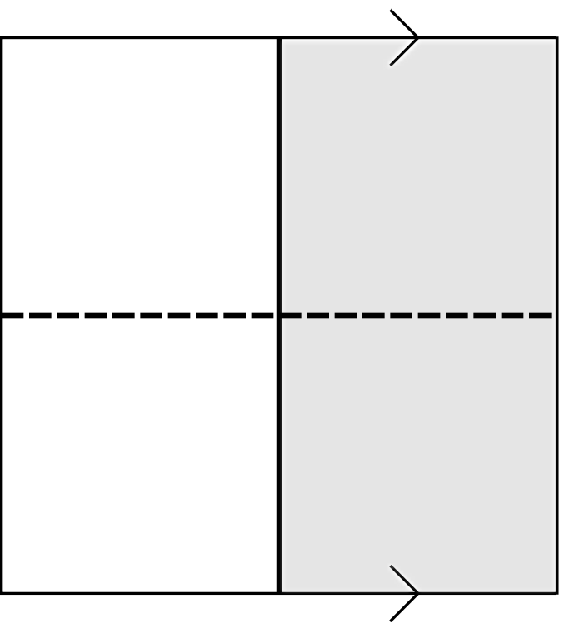}
& \includegraphics[scale=.4]{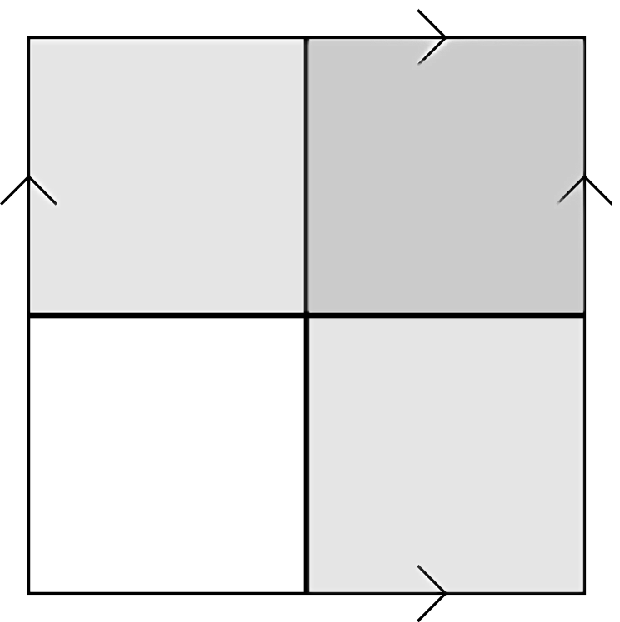} & \includegraphics[scale=.4]{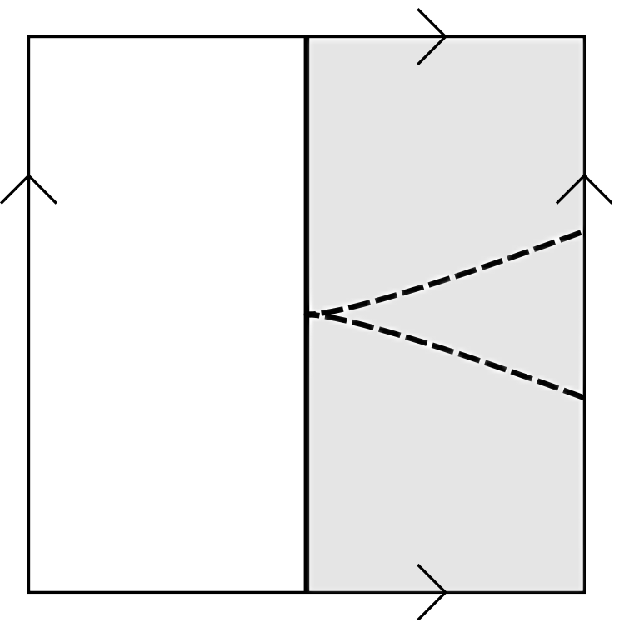} \\  \mbox{(7)} & \mbox{(8)} &\mbox{(9)} & \mbox{(10)} & \mbox{(11)} & \mbox{(12)}  \end{array}$}

%\bigskip

%\centerline{$\begin{array}{cccc} \includegraphics[scale=.4]{images/S9} & \includegraphics[scale=.4]{images/S10}
%& \includegraphics[scale=.4]{images/S11} & \includegraphics[scale=.4]{images/S12} \\ \mbox{(9)} & \mbox{(10)} & \mbox{(11)} & \mbox{(12)}  \end{array}$} 

\bigskip

\centerline{$\begin{array}{ccc} \includegraphics[scale=.4]{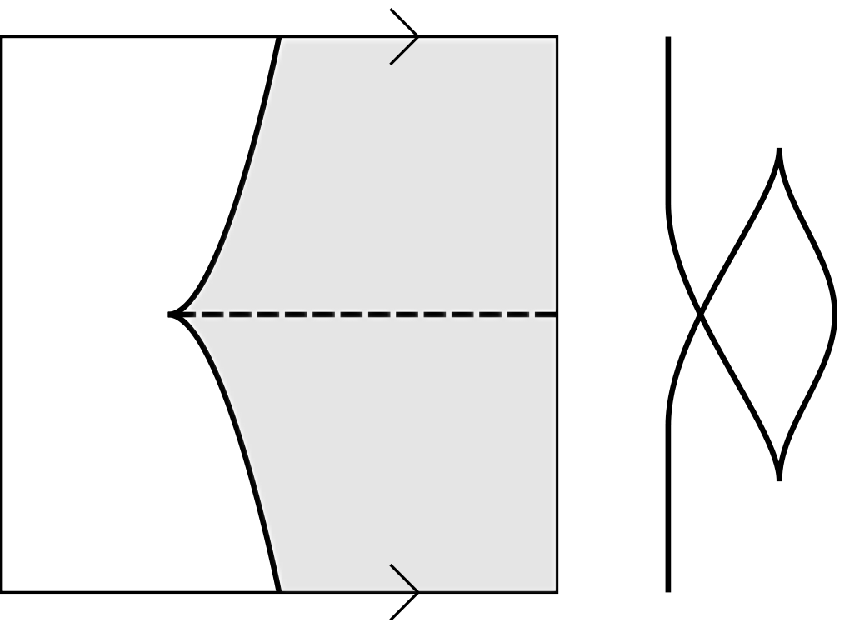}  & \includegraphics[scale=.4]{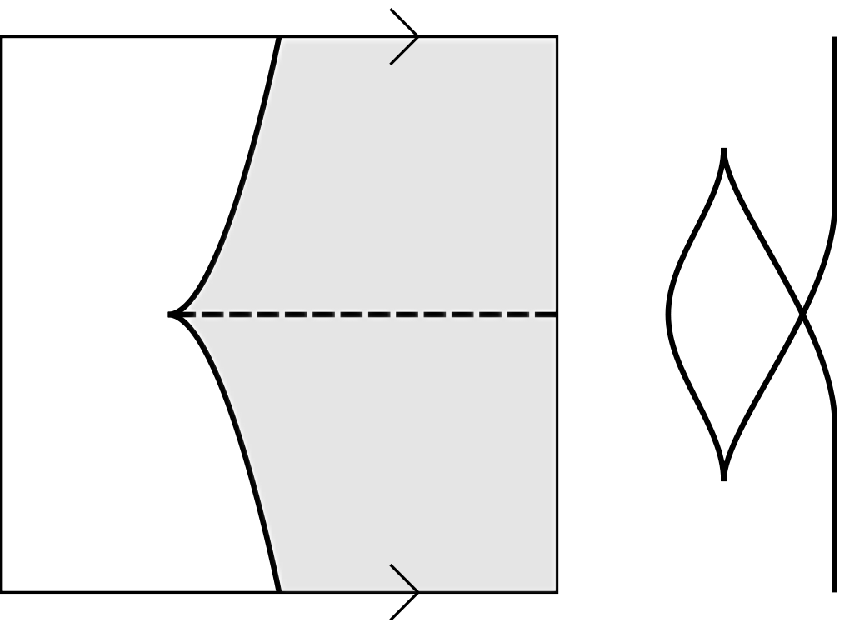}  \\ \mbox{(13)} & \mbox{(14)}  \end{array}$}

\caption{Allowed base projections of the cusp and crossing locus for elementary squares.  Solid lines denote cusp edges and dotted lines denote intersection of sheets.  The side of the cusp locus where the cusp sheets lie is shaded.    Codimension $2$ phenomena appear in (7), (8) and (10)-(14).  In (7), (10), and (11) the two crossing and/or cusp lines involve $4$ distinct sheets.  The squares (9), (12), and (13)-(14) are base projections of a triple point, a sheet crossing a cusp edge, and the upward and downward  swallow tails respectively. }
\label{fig:generators}
\end{figure}

%\begin{figure}
%\centerline{ \includegraphics[scale=.5]{images/2Cr1D}} 
%\caption{The form of $L$ above edges that intersect two arcs of the crossing locus.  If strands are numbered from top to bottom as they appear at $x_i=-1$, crossings occur first between the $k$ and $k+1$ strands and then between the $k+1$ and $k+2$ strands.  }
%\label{fig:2Cr1D}
%\end{figure}

\begin{figure}

\labellist
%\small
%\pinlabel $1/4$ [t] at 338 21
\pinlabel (PV) [t] at 100 0
\pinlabel (1Cr) [t] at 370 0
\pinlabel (2Cr) [t] at 640 0
\pinlabel (Cu) [t] at 910 0
\small
\pinlabel $k$ [l] at 468 96
\pinlabel $k+1$ [l] at 468 56
\pinlabel $k$ [l] at 738 96
\pinlabel $k+1$ [l] at 738 56
\pinlabel $k+2$ [l] at 738 16
\pinlabel $k$ [l] at 1008 96
\pinlabel $k+1$ [l] at 1008 56
\endlabellist

\centerline{ \includegraphics[scale=.4]{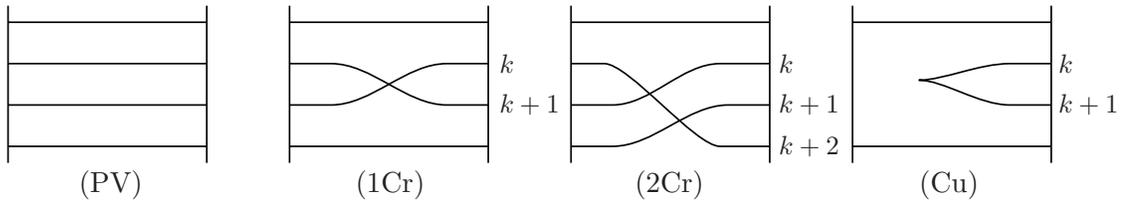}} 

\quad

\caption{The four possibilities for the singular set of $L$ above edges of elementary squares, pictured in the front projection:  (PV) Plain vanilla, i.e. no crossings or cusps; (1Cr) 1 Crossing; (2Cr) 2 Crossings; and (Cu) a single left cusp.  The left and right vertical lines denote the endpoints of the edge where $x_i=-1$ and $x_i=+1$ respectively for either $i=1$ or $2$.  There may be many additional strands that are disjoint from those pictured and do not have any crossings or cusps.  Note that in the (2Cr) square if strands are numbered from top to bottom as they appear at $x_i=+1$, crossings occur first between the $k+2$ and $k$ strands and then between the $k+2$ and $k+1$ strands.  }
\label{fig:EdgeTypesB}
\end{figure}

\subsection{Requirements for the transverse square decomposition $\mathcal{E}_\tra$}  \label{ssec:TransverseSqRe}

We will want to construct a transverse square decomposition $\mathcal{E}_\pitchfork$ for $L$ together with a choice of orientation for each of the $1$-cells of $\mathcal{E}_\pitchfork$ so that a list of conditions (A1)-(A4) is satisfied.  
We state (A1)-(A3) now, and postpone (A4) until \ref{sec:parallel}.

\begin{itemize}
\item[(A1)] All $2$-cells are elementary squares for $L$ parametrized so that all orientations of $1$-cells are from the lower left corner to upper right corner. %has one of the forms pictured in Figure \ref{fig:generators}.

%\item[(A2)]  Opposite edges of squares have the same orientation.  Moreover, it should be possible to identify each square of $\mathcal{E}_\tra$ with an elementary square from Figure \ref{fig:generators} (or with the reflection of such a square across $y=x$) so that all orientations of $1$-cells are from the lower left corner to upper right corner.

%\item[(A3)]Edges intersecting the cusp locus of $L$ are oriented in the direction where the number of sheets increases.
%\item  Above each square $L$ has one of the forms pictured in Figure \ref{fig:generators}.  

\item[(A2)] For any $0$-cell, $e^0_\alpha$, and pair of sheets $S_{i}$, $S_{j}$ above $e^0_\alpha$  consider the set of $1$-cells, 
\[
T(e^0_\alpha, S_i,S_j) =\{e^1_\beta \, | \,  \mbox{$e^0_\alpha$ is the initial vertex of $e^1_\beta$; and $S_i$ and $S_j$ intersect above $e^1_\beta$}\}.
\]
The cardinality of $T(e^0_\alpha, S_i,S_j)$ satisfies
\[
|T(e^0_\alpha, S_i,S_j)| \leq 2,
\]
and if $|T(e^0_\alpha, S_i,S_j)| = 2$ then the two $1$-cells in $T(e^0_\alpha, S_i,S_j)$ form the bottom and left edge of a single Type (3) square that has $e^0_\alpha$ as its lower left vertex.
\item[(A3)]  No two Type (3) squares share a common boundary edge or vertex.

\end{itemize} 
%Note that (A3) is actually a consequence of (A2).  

\subsection{Shifting $\pi_x(\Sigma)$ into the $1$-skeleton} \label{sec:parallel}
In order to later make contact with the cellular DGA, we will want to be able to associate an $L$-compatible cell decomposition to $\mathcal{E}_\pitchfork$ which we will denote $\mathcal{E}_\parallel$.  (Recall from \cite{RuSu1} that a polygonal decomposition is $L$-compatible if its $1$-skeleton contains the projection of the singular set of $L$.)  The remaining condition (A4) is designed to ensure that such a procedure is successful.  

To begin, we shift the base projection of the singular set, denoted $\pi_x(\Sigma)$, into the $1$-skeleton of $\mathcal{E}_\pitchfork$.  This is done one square at a time.  Except in the case of squares of type (12)-(14), the intersection of an arc of the cusp or crossing locus with a given elementary square has its endpoints in the interiors of some $1$-cells at the boundary of the square.  We move these endpoints so that they sit at the initial vertex (with respect to the orientation) of the given $1$-cell.  This specifies a unique edge of the square that we shift the arc into, or in the case of a square of type (3) the arc is shifted into a single vertex.    For Type (12)-(14) cells, we place the codimension $2$ point in the lower left corner of the square and then proceed as above.  See Figure \ref{fig:Homotopy}.

\begin{figure}
\centerline{ \includegraphics[scale=.4]{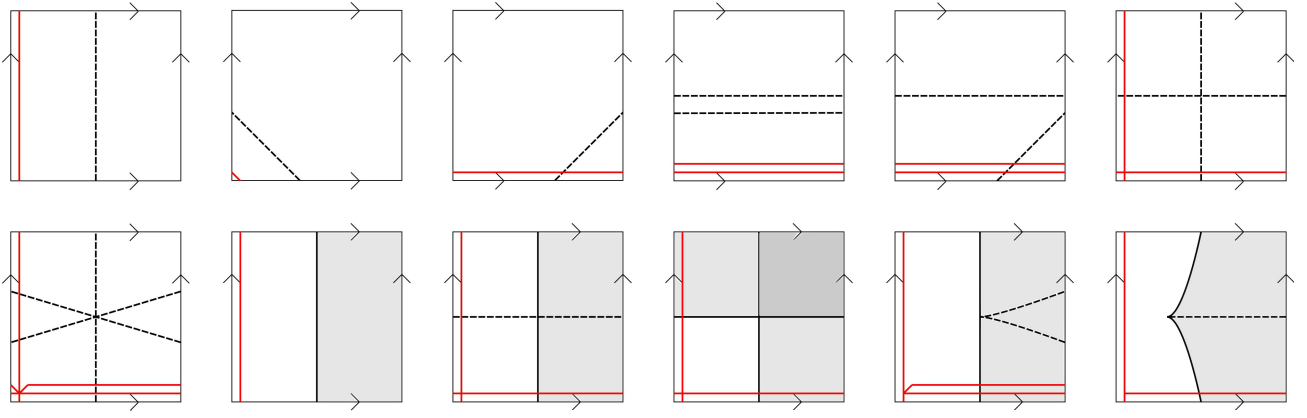} }
%\centerline{ \includegraphics[scale=.8]{images/SubdivideSq} }
\caption{Homotoping the singular set into the $1$-skeleton.  Red lines are drawn next to the edge or vertex that an arc of $\pi_x(\Sigma)$ is moved to.}
\label{fig:Homotopy}
\end{figure}

As the relocation of endpoints is determined by the orientation of the $1$-cells, this shifting of segments pieces together to give a global homotopy of $\pi_x(\Sigma)$ into the $1$-skeleton of $\mathcal{E}_\pitchfork$.  In general, this homotopy cannot be realized by an isotopy of $S$ since for elementary squares of type (5), (6), (8), and (12) two distinct arcs are homotoped to the same edge of a square.  Also, it is possible that a closed component of the crossing locus only intersects squares of type (3) and hence is collapsed to a point.  We can now state the remaining requirement  (A4).

\begin{itemize}
\item[(A4)]  In the above homotopy, no closed component of $\pi_x(\Sigma)$ is shrunk to a point, and the only distinct arcs that are placed on the same edge are those arising from squares of type (5), (6), (8), and (12) as described above.
\end{itemize} 
For instance, we do not allow two squares of type (2) to be located next to one another with the orientation of horizontal borders pointing away from a shared vertical edge.  See Figure \ref{fig:A4Ex}.

\begin{figure}
\centerline{ \includegraphics[scale=.4]{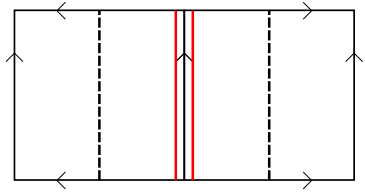} }
%\centerline{ \includegraphics[scale=.8]{images/SubdivideSq} }
\caption{An example where crossing arcs from distinct squares are shifted to the same $1$-cell so that (A4) fails.}
\label{fig:A4Ex}
\end{figure}

%complete construction of $\mathcal{E}_\parallel$.

\subsection{Construction of $\mathcal{E}_\pitchfork$}

After constructing $\mathcal{E}_\pitchfork$ we construct the associated $L$-compatible decomposition $\mathcal{E}_\parallel$ in \ref{sec:mathcalE} below.

%be a Legendrian in the $1$-jet space, $J^1 S$ where  $S$ is a surface.  Assume the front projection of $L$ is generic.

\begin{proposition}  \label{prop:Etra}
For any $L \subset J^1S$ with generic front and base projection, we can find a transverse square decomposition $\mathcal{E}_\pitchfork$  of a neighborhood of $S$ satisfying (A1)-(A4).
\end{proposition}

\begin{proof} Let $A, B \subset S$ denote the (base projection of the) cusp and crossing locus respectively.   Generically, they have the following form.

The cusp locus, $A$, is a union of transversely intersecting and self-intersecting closed curves which are immersed except for cusp points at the projections of swallow tail points.  The crossing locus, $B$, is a union of closed curves and arcs with endpoints at swallow tails.  The crossing locus has cusp points at some points of $A\cap B$ where a sheet crosses a cusp edge as in Figure \ref{fig:generators} (12).  In addition, there are a finite number of transverse double points between parts of the crossing locus and itself as well as transverse double points between crossing and cusp curves.  Finally, there are a finite number of triple points of $B$, as in Figure \ref{fig:generators} (8), that correspond to an intersection of $3$ sheets of $L$.

We will construct the required decomposition of $S$ into squares in the following steps:  
\begin{itemize}
\item[1.] At the codimension $2$ parts of $A \cup B$. 
\item[2.] In a neighborhood of $A$.  
\item[3.] In a neighborhood of $A\cup B$.
\item[4.] The rest of $S$.
\item[5.] Subdivide the result, and assign orientations to edges.   
\end{itemize}

 \textbf{Step 1.} Begin with squares around all codimension $2$ singularities, with various edges intersecting arcs of $A$ and $B$ as perscribed in Figure \ref{fig:generators} (7)-(8) and (10)-(14).  

%\smallskip

 \textbf{Step 2.}  Next, square the remaining portion of a regular neighborhood $N(A)$ of the cusp locus so that boundaries of individual squares have two opposite boundary edges perpendicular to $A$ and the other two opposite edges on $\partial N(A)$.  (We choose the neighborhood $N(A)$ so that this is also the case for the squares surrounding codimension $2$ points except for the case of a transverse intersection of two cusp edges for which only the vertices of the square lie on the boundary of $N(A)$.)  If necessary, add extra edges perpendicular to $A$ to ensure that no two codim $2$ squares share a common edge.  

%At this point we assign orientations to all edges that intersect $A$;  orient these edges so that they point from the side of the cusp locus with fewer sheets to the side with more sheets.  Note that the orientation of opposite edges does agree.

%\smallskip

 \textbf{Step 3.}
Now we extend the squaring to a neighborhood of $A\cup B$.  Of the squares containing codimension $2$ front singularities there are two types with edges having two intersections with the crossing locus; see  Figure \ref{fig:generators} (8) and (12).  Begin by placing an adjacent square  next to each of these edges with the boundary edges chosen so that the crossing locus appears as in the first half of Figures \ref{fig:SpecialSub} and \ref{fig:SpecialSub2} below.  (These new squares are as in Figure 1 (6).)  Now, the remainder of $B$ should be a union of arcs with boundaries on edges that contain just one intersection with $B$.  Use a single Type (2) square for the neighborhoods of each of these arcs.  (Note that we do require that between any two codimension $2$ points of the crossing locus there is  one square of type (2).)

%\smallskip

 \textbf{Step 4.}
We now have squares in a neighborhood of $A\cup B$, $N(A\cup B)$.  The complement of the interior of this neighborhood, $T = S \setminus \mathit{Int}(N(A\cup B))$, is a surface with boundary with a triangulation of the boundary already constructed.  We can extend this triangulation of $\partial T$ to a triangulation of all of $T$ (into triangles not squares).  To see that this is possible, first, triangulate $T$ up to a collar neighborhood of the boundary.  Then, observe that a triangulation of the boundary of an annulus (which we now have along the boundary of each component of this collar) can always be extended to a triangulation of the annulus.

%\smallskip

 \textbf{Step 5.}
Now, we subdivide our current polygonal decomposition into squares and orient the corresponding edges.  Subdivide the triangles in $T$ into $3$ squares by dividing each edge in half and adding edges from the midpoints of edges to a new vertex in the center of the triangle.  Orient the edges of these squares so that they point away from the center of the triangle.  
This is indicated in Figure \ref{fig:TriangleSub}.  Subdivide each of the squares in $N(A\cup B) - \mathit{N(A)}$ into $4$ smaller squares by cutting the square into quarters.  Again, orient all edges to point away from the center of the original square; see Figure \ref{fig:TriangleSub}.   Finally, cut squares in $N(A)$ in half along the direction that is perpendicular to the cusp edge with the exception of those squares that contain double points of $A$ which we do not subdivide at all.  
%\footnote{\ms{7/22/15: `` we do not subdivide at all" or `` we do not NEED TO subdivide at all"?} \dr{7/23:  1st one.  If the reader for some reason decides to take the initiative to go ahead and subdivide these squares then he/she will mess up the construction (because then some of the bordering squares will now be pentagons and so will have to be subdivided again).}}
Orient the edges that are perpendicular to the cusp edge  so that they point from the side of the cusp locus with fewer sheets to the side with more sheets.  Orient the remaining edges away from the newly added edge in the middle of the original squares.

\begin{figure}
\centerline{ \includegraphics[scale=.5]{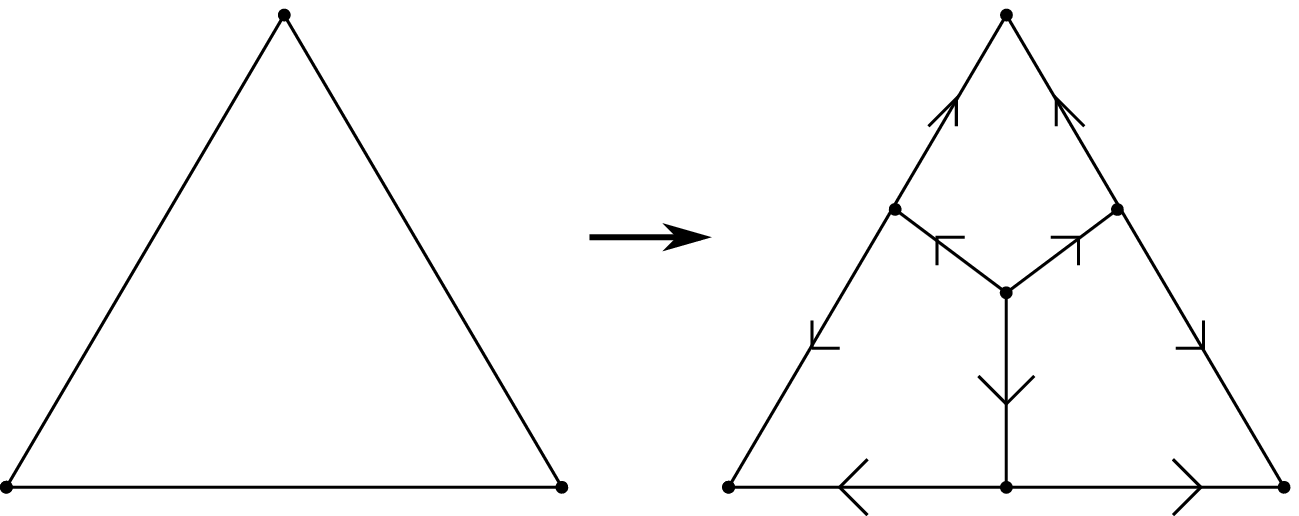} \quad \quad \includegraphics[scale=.5]{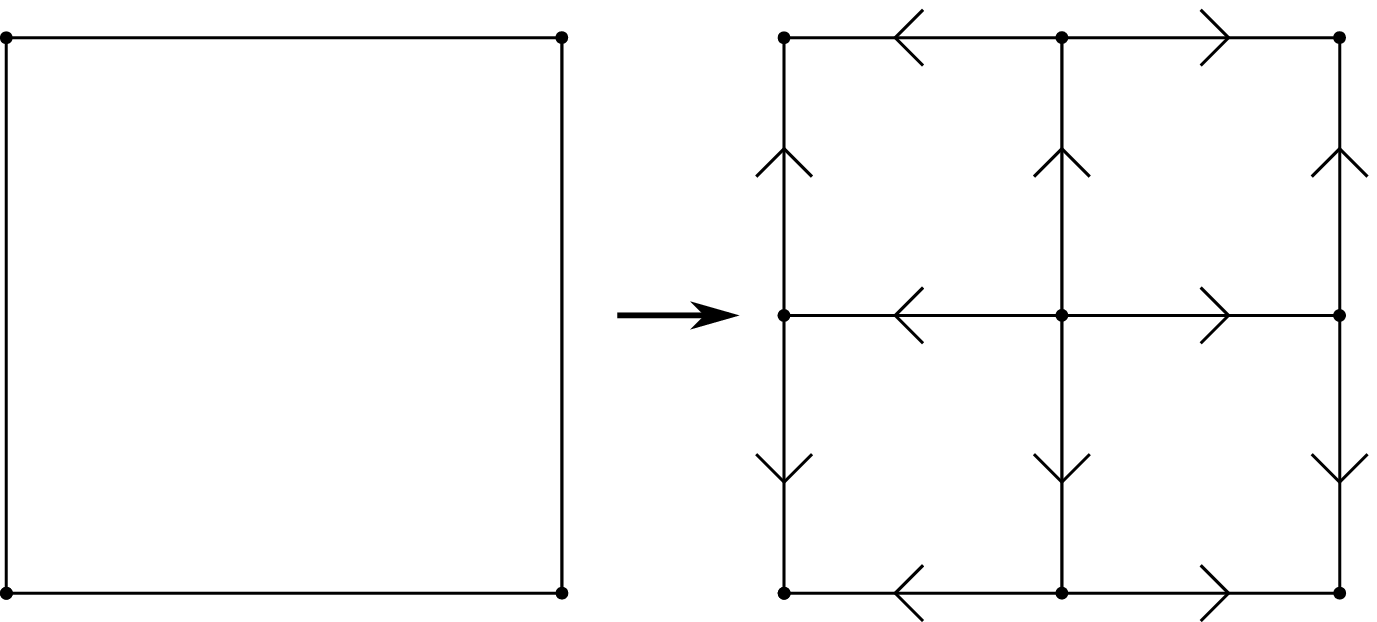}  }
%\centerline{ \includegraphics[scale=.8]{images/SubdivideSq} }
\caption{Subdividing triangles and squares.  Orientations for edges are  pictured after the subdivision.}
\label{fig:TriangleSub}
\end{figure}

Each edge of the original decomposition is shared by two squares, so it is important to note that the subdivision and orientation prescribed for these edges is the same with respect to either face.  This is the case since all edges are subdivided in half except for edges that intersect the cusp locus; the orientation of the two subdivided pieces is always away from the new vertex.  Also, note that all squares have opposite edges oriented in the same direction. 

Finally, one should take some care in subdividing squares that intersect the crossing locus, $B$, so that the subdivided pieces all fit one of the forms specified in Figure \ref{fig:generators}.  This is particularly important at squares of the form (8) and (12) which contain triple points and sheets intersecting the cusp edge.  We subdivide these squares and their adjacent squares as pictured in Figure \ref{fig:SpecialSub} and \ref{fig:SpecialSub2} respectively.  The edges that intersect the crossing locus twice can be made to have the form indicated in  Figure \ref{fig:EdgeTypesB} (2Cr) by the following considerations.  For squares of type (8), choose the identification of the given square (before subdivision) with the pictured one so that the vertical sides have this property when oriented from lower to upper corners.  For squares of type (12), exactly
one of the subdivisions pictured in Figure \ref{fig:SpecialSub2} has the desired property, and this is the subdivision that we use.  

Note that our subdivisions prescribe particular halves of the original edges that the crossing arc should leave these squares along.  This extra restriction can be met, since in our original decomposition we required one square containing a single arc of type (2) between any two codimension $2$ portions of $B$.  The subdivision of such a square can be chosen so that the crossing arc will leave along any perscribed halves of the bottom and top edges, and so that the subdivision uses only squares of type (2), (3), and (4).  See Figure \ref{fig:CrossingSub}.

\begin{figure}
\centerline{ \includegraphics[scale=.6]{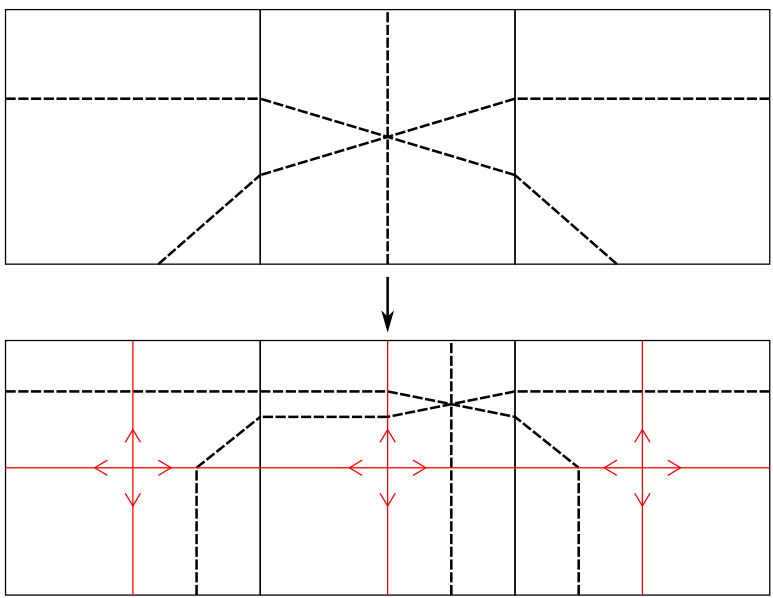}  }
\caption{Subdivision near a triple point.}
\label{fig:SpecialSub}
\end{figure}

\begin{figure}
\centerline{ \includegraphics[scale=.6]{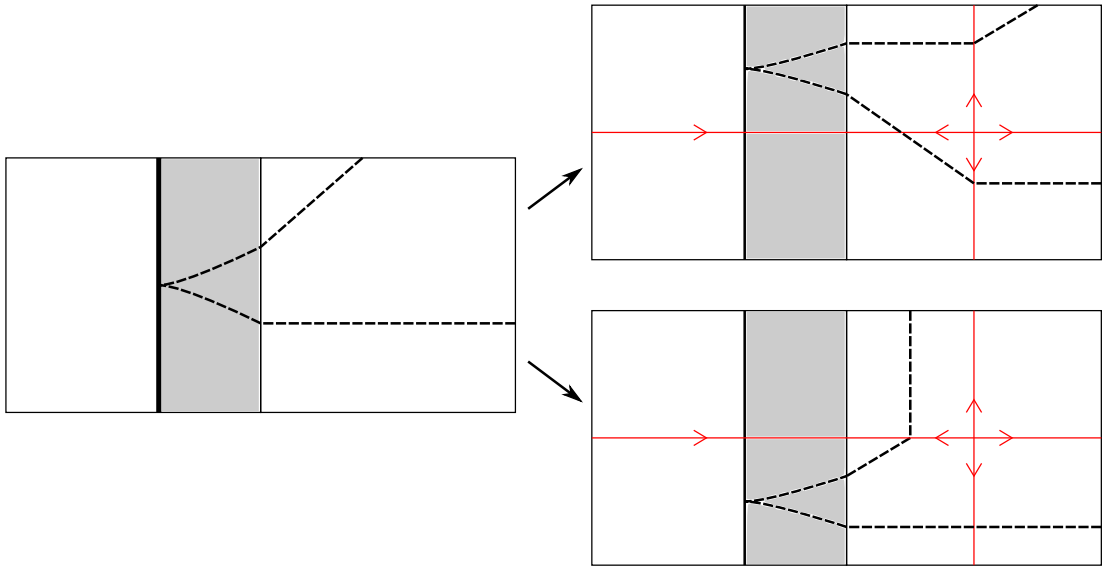}  }
\caption{Subdivision near a sheet intersecting a cusp edge.  Note that the left square contains a cusp edge, so it is only subdivided into two smaller squares.  Of the two subdivisions pictured, we use the one that results in the edge with two crossings matching the form given in Figure \ref{fig:EdgeTypesB} (2Cr).  (This way, only one version of (12) is necessary.)}
\label{fig:SpecialSub2}
\end{figure}

\begin{figure}
\centerline{ \includegraphics[scale=.6]{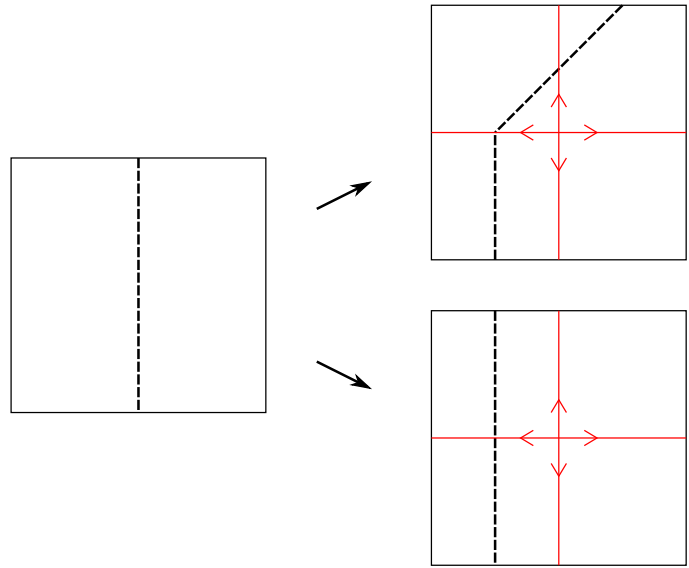}  }
\caption{Subdividing a Type (2) square.  Using one of the pictured subdivisions or their reflections across a vertical axis, we arrange for the crossing arc to intersect any perscribed halves of the boundary.}
\label{fig:CrossingSub}
\end{figure}

With the construction complete, we verify that (A1)-(A4) holds.  Property (A1) is clear from the construction, keeping in mind that in our figures the singular set in the image of an elementary square may appear in a manner that differs from the depiction in Figure \ref{fig:generators} by reflected across $x_1=x_2$.  
Property (A2) is seen to hold by considering the appearance of the crossing locus in relation to the orientations of squares after Step 5. is completed.  (In particular, examine Figures \ref{fig:SpecialSub}, \ref{fig:SpecialSub2}, and \ref{fig:CrossingSub}.)

For (A3), note that Type (3) squares only appear when some of the Type (12) and Type (2) squares are subdivided during Step 5.  (See the Figures \ref{fig:SpecialSub2}  and \ref{fig:CrossingSub}.)  Examining the placement of squares after the subdivision  shows that no two Type (3) squares share an edge or vertex.  

To verify (A4), first check that the condition holds near codimension $2$ points (examine Figures \ref{fig:SpecialSub} and \ref{fig:SpecialSub2}). Before the subdivision, all remaining squares that intersect $A\cup B$ have type (2) or (9) and are bordered on the right and left by squares of type (1).  After the subdivision and resulting orientation of $1$-cells, the crossing arc for a subdivided Type (2) square will be homotoped to the center two $1$-cells of the four squares arising from the subdivision.  The cusp arcs for subdivided Type (9) squares are homotoped to the left sides of these squares (which are bordered by squares of type (1)).  

\end{proof}

\begin{remark}
As constructed, some squares of $\mathcal{E}_\pitchfork$ may be disjoint from $\pi_x(L)$.  These squares can be safely ignored so that we view $\mathcal{E}_\pitchfork$ as a square decomposition of a neighborhood of $\pi_x(L)$. 
\end{remark}

\subsection{Construction of $\mathcal{E}_\parallel$} \label{sec:mathcalE}

From this construction of $\mathcal{E}_\tra$ we can now produce a related $L$-compatible polygonal decomposition of $\pi_x(L) \subset S$ denoted $\mathcal{E}_\parallel$.  As in \ref{sec:parallel}, we have a map of $\pi_x(\Sigma)$ into the $1$-skeleton of $\mathcal{E}_\tra$ that is not realized by an isotopy of $S$ only because in  squares of type (5), (6), (8) and (12) two crossing arcs are mapped to the same edge.  Note that all apperances of these squares in $\mathcal{E}_\parallel$ are as  specified in Figures \ref{fig:SpecialSub} and \ref{fig:SpecialSub2}).  To arrive at the decomposition $\mathcal{E}_\parallel$ we simply seperate these edges by placing new $2$-cells and $1$-cells between them in an appropriate manner.  See Figures \ref{fig:EParallel} and \ref{fig:EParallel12}.

More precisely, the relationship between the cells of $\mathcal{E}_\pitchfork$ and $\mathcal{E}_\parallel$ is the following:

\begin{proposition}  \label{prop:EparProp}
The $L$-compatible polygonal decomposition $\mathcal{E}_\parallel$ is obtained by applying a homeomorphism to the cell decomposition  $\mathcal{E}_\tra$ and then making the following subdivisions:
\begin{itemize}
\item  The cell decompositions of closed squares of type (5), (8), and (12) from $\mathcal{E}_\tra$ are subdivided by placing an extra $0$-cell in the interior of the edge of the square that corresponds to the right edge in Figure \ref{fig:generators} and then  connecting this new $0$-cell to the lower left corner of the square with a new $1$-cell.  
\item  The cell decompositions of closed squares of type (6) are subdivided by placing an extra $0$-cell on both the left and right edges and then connecting these new $0$-cells with a new $1$-cell.
\end{itemize}
\end{proposition}

\begin{figure}
\centerline{ \includegraphics[scale=.6]{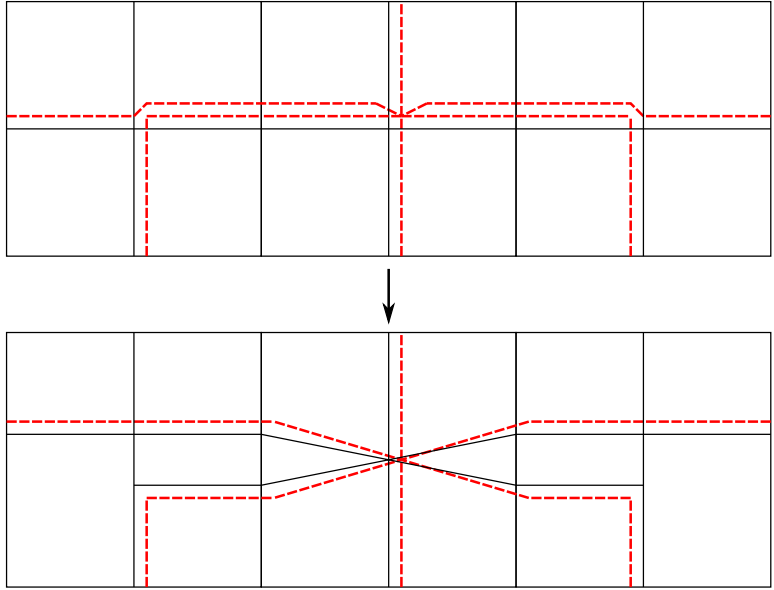}  }
\caption{(Top) The result of shifting $\pi_x(\Sigma)$ into the $1$-skeleton of $\mathcal{E}_\pitchfork$ is pictured near a Type (8) square.  (Bottom) We obtain $\mathcal{E}_\parallel$ by subdividing squares of type (5),(6),(8) and (12) and seperating the crossing arcs that previously belonged to the same $1$-cells.  
%Near a triple point this is done as pictured, and the procedure near a Type (12) square is similar.  
The dotted red lines represent the crossing locus, and are intended to coincide with parts of the $1$-skeleton (black).}
\label{fig:EParallel}
\end{figure}

\begin{figure}
\centerline{ \includegraphics[scale=.6]{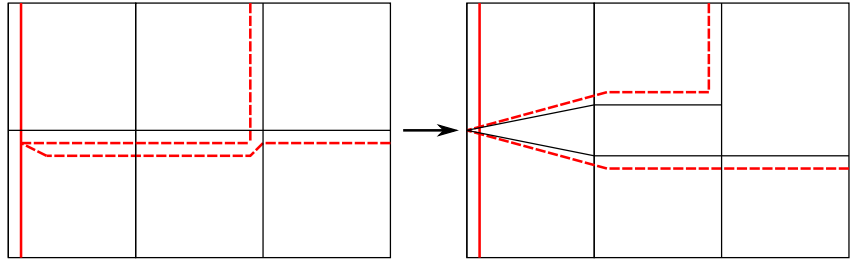}  }
\caption{ Shifting $\pi_x(\Sigma)$ into the $1$-skeleton and forming $\mathcal{E}_\parallel$ near a Type (12) square.   
%(Bottom) We obtain $\mathcal{E}_\parallel$ by subdividing squares of type (5),(6),(8) and (12) and seperating the crossing arcs that previously belonged to the same $1$-cells.  Near a triple point this is done as pictured, and the procedure near a Type (12) square is similar.  
The solid red line is the cusp locus while the dotted red lines represent the crossing locus.   Both are intended to coincide with parts of the $1$-skeleton (black).}
\label{fig:EParallel12}
\end{figure}

\section{Blueprint for enumeration of GFTs}  \label{sec:3.5}

%Geometrically the generators of the Cellular DGA correspond to Reeb chords $a_{i,j}$, $b_{i,j}$, and $c_{i,j}$ that are local minima, saddle points, and local maxima located in neighborhoods of the $0$-, $1$-, and $2$-cells of a decomposition of $S$ into squares.

%After Section \ref{sec:transverse}, we have decomposed a neighborhood of the base projection of $L$ into squares, and above each square $L$ matches one of the types (1)-(14) from Figure \ref{fig:generators}.  
The next step of the proof is to perform a (partial) computation of the LCH DGA above each square of the decomposition $\mathcal{E}_\pitchfork$ using a Legendrian isotopic surface, $\tilde{L}$, for which the Reeb chords near any particular square form sub-DGAs.  
This is carried out in Sections \ref{sec:CompLCH}-\ref{sec:SwallowComp}.  To compute the LCH differential using Theorem \ref{thm:EkholmMain}, we will need to enumerate relevant rigid GFTs for each of the 14 square types.    To help prepare the reader for the general enumeration procedure, in this section we outline a computation for the Type (1) square using a simplified model for $\tilde{L}$.  This should serve as a blue print for the unified approach used for square types (1)-(12).  (The (13) and (14) squares which have swallowtail points require an additional argument.) 
%uniformally in Section \ref{sec:Comp2Cells}.
%where no crossings or cusps are present 
%using a simplified model for $\tilde{L}$, and 
We then explain some of the challenges that are later addressed in extending this approach to other square types.

Section \ref{sec:3.5} is intended as a guide for the proof of Theorem \ref{thm:CellularLCH}.  The remainder of the article is independent of Section \ref{sec:3.5}, both logically and in exposition.  

\subsection{Type (1) square as an example}
%To give an idea for the arguments involved in our computations of LCH in particular squares, and make to illuminate somewhat the geometric connection between the cellular DGA and LCH, 
%We conclude this introduction with a computation in an illustrative special case.

%We start with an explicit model, although because of the possible need to allow a perturbation to achieve the transversality of GFTs, the model used in Sections ??-?? is not quite this specific.  (See below.)

We consider a square, in coordinates $[-1,1]\times[-1,1]$, above which $L$ consists of $n$ non-intersecting sheets without cusp edges.  (This is a Type (1) square in Figure \ref{fig:generators}.)  Above the square, the front projection of $L$ appears as a subset of $[-1,1]\times [-1,1] \times \R$ that consists of the union of the graphs of $n$ functions $F_1, \ldots, F_n:[-1,1]\times[-1,1] \rightarrow \R$, which we label to satisfy $F_1 > \ldots > F_n$.  We take $F_i$ of the form
\[
F_i(x_1,x_2) = f_i(x_1)+f_i(x_2)
\]
where $f_i :[-1,1] \rightarrow \R$ are $1$-dimensional functions designed so that for any $1 \leq i < j \leq n$, the difference function $f_{i} -f_j$ has local minima at $-1$ and $1$ and a single local maximum $\beta_{i,j} \in (-1,1)$.  Moreover, we can arrange that the locations of the local maxima are lexicographically ordered in $i$ and $j$,
\[
\beta_{1,2} < \beta_{1,3} < \ldots <\beta_{2,3} < \beta_{2,4} < \ldots < \beta_{n-1,n}.
\]
(Figure \ref{fig:LinStair}, below illustrates how this ordering of critical points may be arranged.)
The Reeb chords in $[-1,1]\times[-1,1]$, are critical points of difference functions $F_{i,j} := F_i -F_j$ with $i<j$, and for each $F_{i,j}$ 
we have $9$ critical points
% for each of the $2$-dimensional difference functions
\[
a^{\pm,\pm}_{i,j}, b^U_{i,j}, b^R_{i,j}, b^D_{i,j}, b^L_{i,j},  c_{i,j}, \quad \quad 1 \leq i <j \leq n
\]
located as pictured in Figure \ref{fig:Intro}.  Critical points labeled with $a$'s, $b$'s, and $c$'s are respectively local minima, saddle points, and local maxima.  The horizontal and vertical segments connecting the $b^X_{i,j}$ to $c_{i,j}$ are flow lines for $-\nabla F_{i,j}$ that divide the square into four closed {\bf $(i,j)$-quadrants}.  We call the upper right and the lower left of these quadrants the $1$-st and $3$-rd $(i,j)$-quadrant.

\begin{figure}

\quad

\quad

\labellist
\small
\pinlabel $c_{i,j}$ [tr] at 142 142
\pinlabel $b^R_{i,j}$ [l] at 254 150
\pinlabel $b^U_{i,j}$ [b] at 150 252
\pinlabel $b^D_{i,j}$ [t] at  150 0
\pinlabel $b^L_{i,j}$ [r] at  0 150
\pinlabel $a^{+,+}_{i,j}$ [bl] at 254 254
\pinlabel $a^{-,+}_{i,j}$ [br] at 0 254
\pinlabel $a^{-,-}_{i,j}$ [tr] at  0 0
\pinlabel $a^{+,-}_{i,j}$ [tl] at  254 0
\endlabellist
\centerline{ \includegraphics[scale=.6]{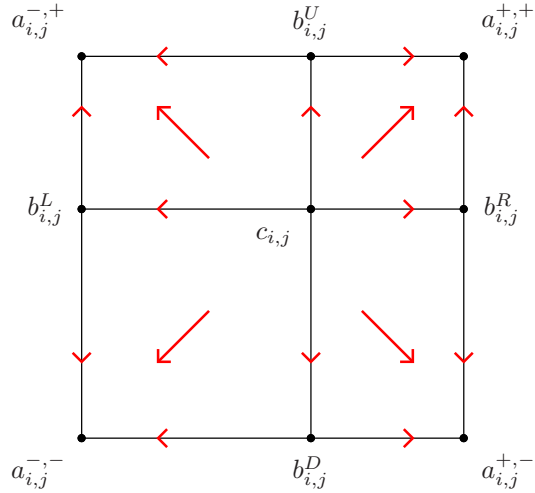} }
%\centerline{ \includegraphics[scale=.8]{images/SubdivideSq} }

\quad 

\caption{Location of the critical points of $F_{i,j}$.  Red arrows indicate the sign of the components of $-\nabla F_{i,j}$.  }
\label{fig:Intro}
\end{figure}

We now compute the differential $\partial c_{i,j}$ (with $\Z/2$ coefficients) which counts rigid GFTs begining at $c_{i,j}$.  Due to the absence of cusp edges, internal vertices can only be  $Y_0$-vertices.  (We assume the $1$-regular condition is satisfied with respect to the Euclidean metric on $[-1,1]\times[-1,1]$.)  Thus, by Proposition \ref{prop:EY1SWFormula}, a GFT beginning at $c_{i,j}$ is rigid if and only if the number of outputs at local minima, $a^{\pm,\pm}_{k,l}$, is the same as the number of outputs at local maxima, $c_{k,l}$.
For each, $X\in \{U,L,R,D\}$, the single edge tree that is the $-\nabla F_{i,j}$ flow line from $c_{i,j}$ to $b^X_{i,j}$ is rigid, and $4$ other families of rigid GFTs are pictured in Figure \ref{fig:4TreesB}.  We will show that these are the only rigid GFTs, so that 
\[
\partial c_{i,j} = b^U_{i,j} + b^L_{i,j} + b^R_{i,j} + b^D_{i,j}+ \sum_{i<m<j} \left(a^{+,+}_{i,m}c_{m,j} + c_{i,m}a^{-,-}_{m,j} + b^U_{i,m} b^L_{m,j} + b^R_{i,m}b^D_{m,j}\right) 
\]
Putting $C =(c_{i,j})$,  $B_U=(b^U_{i,j})$,  etc., (with $0$'s for entries when the corresponding Reeb chord does not exist), we get the matrix equation 
\begin{equation}  \label{eq:vanillaEx}
\partial C = A_{+,+}C+CA_{-,-} + (I+B_U)(I+B_L)+(I+B_R)(I+B_D)
\end{equation}
which is identical to the differential of the cellular DGA for a square without crossings or cusps above its boundary.

\begin{figure}

\quad

\quad

\labellist
\small
\pinlabel $c_{i,j}$ [br] at 134 144
\pinlabel $c_{m,j}$ [br] at 182 192
\pinlabel $a^{+,+}_{i,m}$ [b] at 248 248
\pinlabel $c_{i,j}$ [br] at  454 144
\pinlabel $c_{i,m}$ [br] at  414 104
\pinlabel $a^{-,-}_{m,j}$ [t] at 318 -2
\pinlabel (A) [t] at 123 -8
\pinlabel (B) [t] at 453 -8
\endlabellist
\centerline{ \includegraphics[scale=.4]{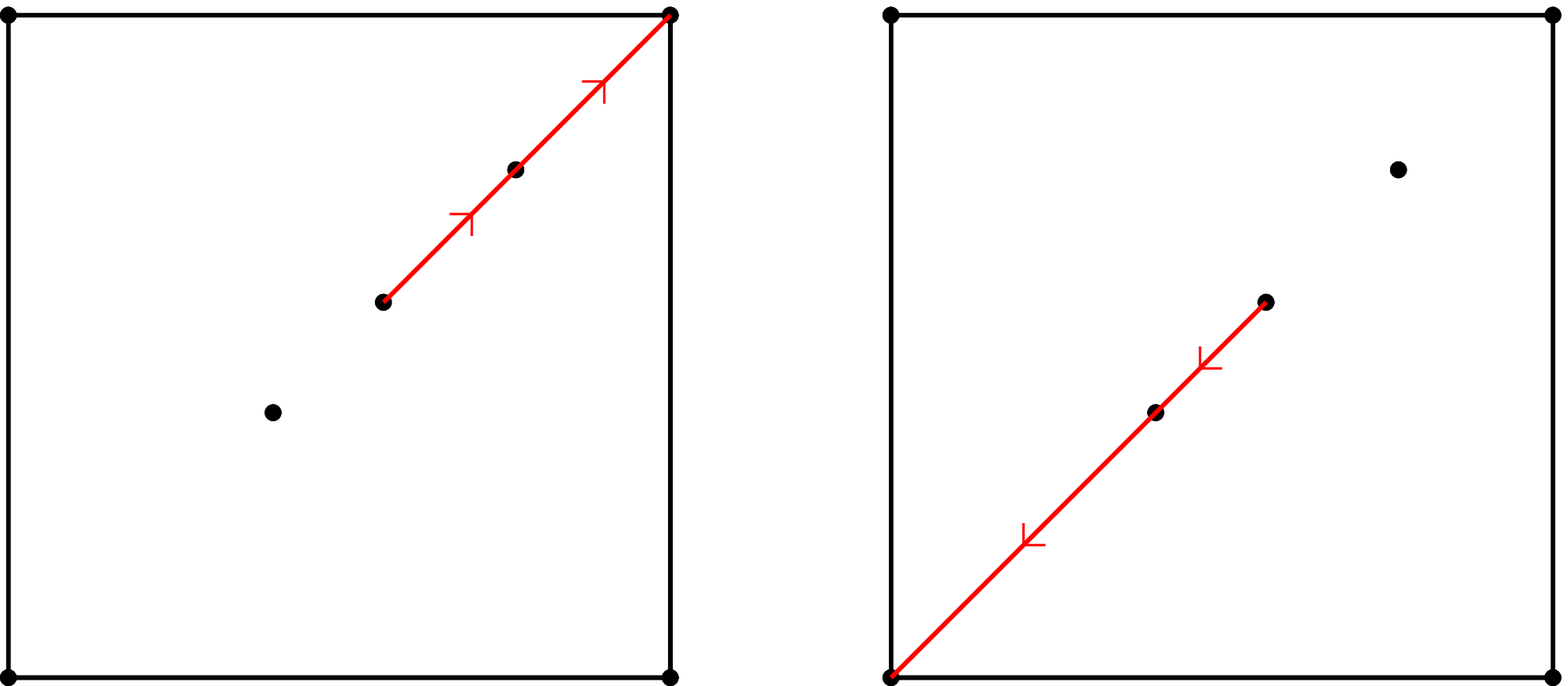} 
 \quad \quad \quad
\labellist
\small
\pinlabel $c_{i,j}$ [tl] at 144 134
\pinlabel $b^L_{m,j}$ [r] at -2 186
\pinlabel $b^U_{i,m}$ [b] at 98 248
\pinlabel $c_{i,j}$ [br] at 454 144
\pinlabel $b^D_{m,j}$ [t] at  508 -2
\pinlabel $b^R_{i,m}$ [l] at  568 98
\pinlabel (C) [t] at 123 -8
\pinlabel (D) [t] at 453 -8
\endlabellist
\includegraphics[scale=.4]{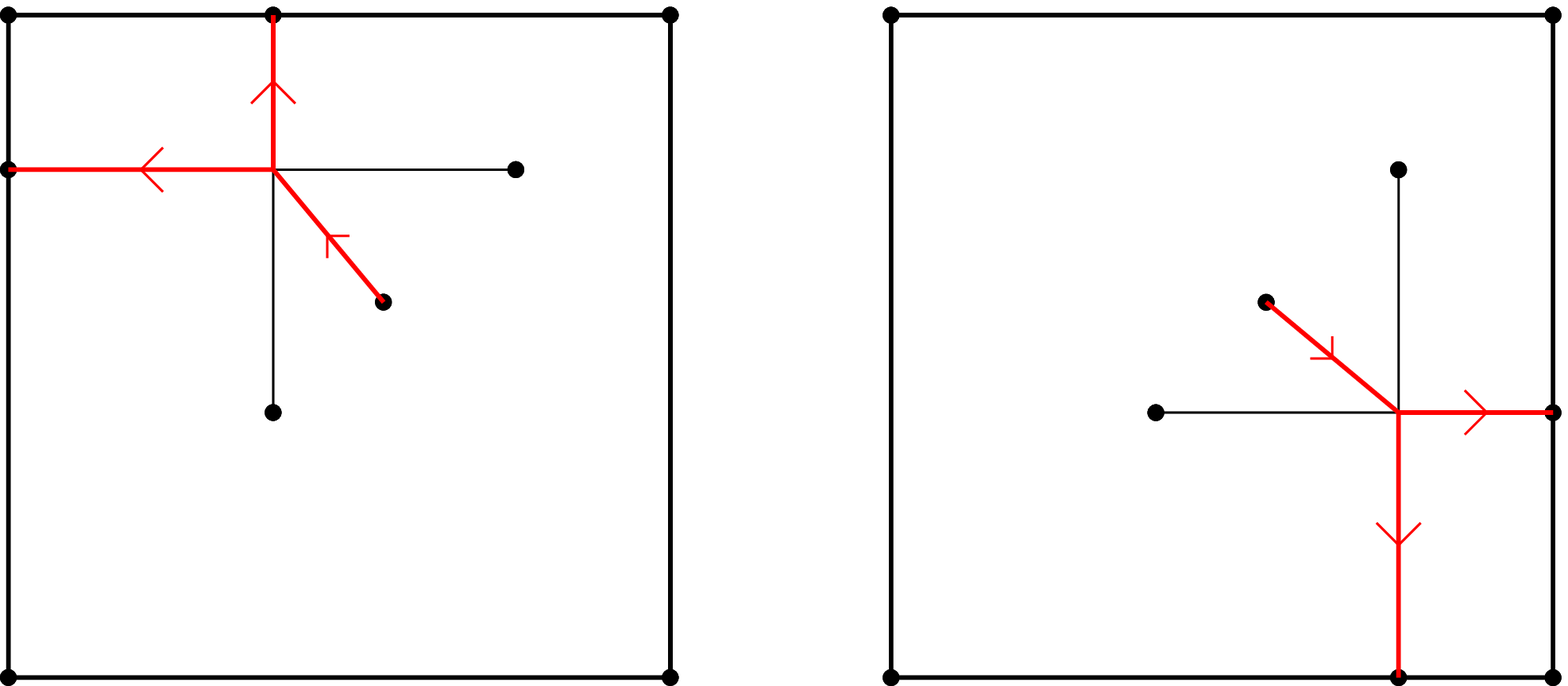}
}

\quad 

\quad

\caption{For each $i<m<j$ there are $4$ rigid GFTs with a single $Y_0$-vertex, and two outputs.  Note that edges that precede outputs at $c$'s are constant.  }
\label{fig:4TreesB}
\end{figure}

%admit the following simplified description.  Beginning with the root at $c$, the initial edge of $\Gamma$ is a flow line for $-\nabla F_{i,j}$.   $\Gamma$ may have internal vertices at which (with all edges oriented away from the root) a flow line for some $-\nabla F_{l_1,l_2}$ splits into flow lines for $-\nabla F_{l_1,m}$ and $-\nabla F_{m,l_2}$.  All output vertices limit to critical points, $x_1, \ldots, x_N$.  Such a tree is rigid provided that the number of outputs at local minima, $a^{\pm,\pm}_{i,j}$, is the same as the number of outputs at local maxima, $c_{i,j}$, and each such rigid tree produces a term of the form $\partial c = x_1\cdots x_N + \cdots$.  

%We examine here the differential of the $c_{i,j}$ generators, which is given by the simple matrix equation
%\[
%\partial C= A^{+,+} C + C A^{-,-} + (I+B_U)(I+B_L) + (I+B_R)(I+B_D).
%\]
%Note that other than slightly different notation for superscripts of generators, this is exactly the part of the Cellular DGA associated to a square without crossings or cusps above its boundary.   

\begin{proof}[Proof that there are no other rigid GFTs beginning at $c$'s]

\begin{figure}

\quad

\labellist
\small
\pinlabel $\vdots$ [t] at 0 20
\pinlabel $\vdots$ [b] at 16 70
\pinlabel $\vdots$ [b] at 144 70
\pinlabel $\vdots$ [b] at 448 70
\pinlabel $c_{m,j}$ [t] at  48 -2
\pinlabel $c_{m,j}$ [t] at  176 -2
\pinlabel $c_{m,j}$ [t] at 312 -2
\pinlabel $a^{+,+}_{i,m}$ [t] at  112 -2
\pinlabel $a^{+,+}_{i,m}$ [t] at  248 -2
\pinlabel $b^U_{i,m}$ [t] at 416 -2
\pinlabel $b^U_{i,m}$ [t] at 544 -2
\pinlabel $b^L_{m,j}$ [t] at 480 -2
\pinlabel $b^L_{m,j}$ [t] at 608 -2

\pinlabel $c_{i,j}$ [b] at 280 102
\pinlabel $c_{i,j}$ [b] at 576 102

\endlabellist
\centerline{ \includegraphics[scale=.6]{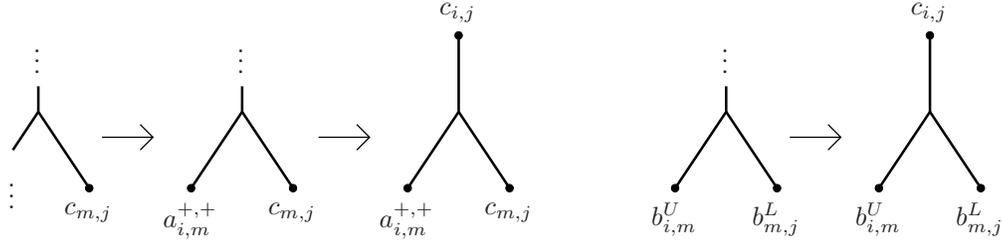} }
%\centerline{ \includegraphics[scale=.8]{images/SubdivideSq} }

\quad 

\caption{Pictorial summary of the arguments in Case 1 and Case 2. }
\label{fig:ExampleSteps}
\end{figure}

The steps of the proof are summarized in Figure \ref{fig:ExampleSteps}.  We will use
%that the above are the only rigid GFTs beginning at $c_{i,j}$ is based on 
the following:
\begin{lemma}[1st and 3rd Quadrant Lemma]  \label{lem:1st3rd}
Suppose that an edge $\gamma \subset \Gamma$ on a GFT is a $-\nabla F_{i,j}$ flow line.   
\begin{enumerate}
\item If $\gamma$ has a point that is mapped to the $1$-st $(j,l)$-quadrant for some $i<j<l$.  Then, all outputs on the part of $\Gamma$ below $\gamma$ are at $a^{+,+}$ critical points.   
\item If $\gamma$ has a point that is mapped to the $3$-rd $(h, i)$-quadrant for some $h<i<j$.  Then, all outputs on the part of $\Gamma$ below $\gamma$ are at $a^{-,-}$ critical points.
\end{enumerate}
\end{lemma}

\begin{proof}  We prove (1) as (2) is similar.  Any edge of $\Gamma$ below $\gamma$ (with respect to the orientation of the domain tree of $\Gamma$) must be an $(i',j')$-flow line (i.e. a trajectory for $-\nabla F_{i',j'}$) for some $i\leq i' < j' \leq j$.  For all such $(i',j')$,  the $1$-st $(j,l)$-quadrant is contained in the $1$-st $(i',j')$-quadrant where $-\nabla F_{i',j'}$ has both components positive.  Thus, all edges below $\gamma$ must remain entirely within the $1$-st $(j,l)$-quadrant.  Moreover, since the only critical points of the $F_{i',j'}$ in the $1$-st $(j,l)$-quadrant are the $a^{+,+}_{i',j'}$, (1) follows.  
\end{proof}

With Lemma \ref{lem:1st3rd} in hand, let $\Gamma$ be a rigid GFT beginning at some $c$ Reeb chord.  

\medskip

\noindent{\bf Case 1:}  $\Gamma$ has at least one output at a $c$.

\medskip

%Since the only $-\nabla F_{i,j}$ flowline ending at $c_{i,j}$ is the constant flowline at $c_{i,j}$ itself, 
The only way that a $c_{m,l}$ can occur as an output of a GFT is when a $Y_0$-vertex has its image at $c_{m,l}$ with one of the outgoing edges the constant $(m,l)$-flow line.   Supposing the other outgoing edge is an $(i,m)$-flow line (resp. an $(l,j)$-flow line), (using the lexicographical ordering of the $c_{i,j}$ along the diagonal)
 the $1$-st and $3$-rd quadrant lemma shows that all outputs that occur below this edge can only be at $a^{+,+}$'s (resp. $a^{-,-}$'s).  Since the number of outputs at $c$'s and $a$'s must agree, we see that: 
\begin{itemize}
\item[(i)] At a $Y_0$ with one outgoing edge having an output at a  $c$, the other outgoing edge must proceed directly to an $a$ without branching.  
\item[(ii)] Moreover, any edge with an output at an $a$ must begin with such a $Y_0$-vertex at a $c$-vertex.  
\end{itemize}

So far, we know $\Gamma$ has a $Y_0$-vertex $x$ where one of the outgoing edges is a constant map limiting to a $c$.  Suppose that the  incoming edge at $x$ is an $(i,j)$-flow line.  We consider the case where for some $i<m<j$ the constant outgoing edge of $x$ is at $c_{m,j}$, so that (i) implies 
%with one of the outgoing edges constant, ending at a puncture at $c_{m,j}$ 
the other outgoing edge limits to a puncture at some $a^{+,+}_{i,m}$.  
%We consider the case where the other outgoing flow line limits to $a^{+,+}_{i,m}$.     
We show that the incoming edge $\alpha$ at $x$ must begin at a positive puncture at $c_{i,j}$, so that $\Gamma$ is as in (A) of Figure \ref{fig:4TreesB}.  (In the case where the outgoing edges have punctures at $c_{i,m}$ and $a^{-,-}_{m,j}$ a similar argument will show that $\Gamma$ is as in (B) of Figure \ref{fig:4TreesB}.)

Suppose that this is not the case, so that $\alpha$ begins at another $Y_0$ vertex, $y$.  As $t$ increases, $\alpha$ traces out some part of the $-\nabla F_{i,j}$ trajectory from $c_{i,j}$ to $c_{m,j}$, so the image of $\alpha$ is entirely contained in the intersection of the $1$-st $(i,j)$-quadrant and the $3$-rd $(m,j)$-quadrant.  
\begin{itemize}
\item  Subcase:  The other outgoing edge at $y$ is a $(j,l)$-flow line for some $j<l$.   Then, since $y$ is in the $3$-rd $(m,j)$-quadrant, Lemma \ref{lem:1st3rd} (2) applies to contradict (ii).
\item  Subcase:  The other outgoing edge at $y$ is a $(h,i)$-flow line for some $h<i$.   Then, since $y$ is in the $1$-st $(i,j)$-quadrant, Lemma \ref{lem:1st3rd} (1) applies to contradict (ii).
\end{itemize}

\medskip

\noindent{\bf Case 2:}  $\Gamma$ does not have any outputs at $c$'s.

Since the number of outputs at $a$'s and $c$'s is equal, $\Gamma$ has all outputs at $b$'s.  The rigid GFTs with only one edge (these are simply gradient trajectories of $-\nabla F_{i,j}$) are as specified.  Assuming more than one edge, we can find a ``lower most'' $Y_0$ vertex, $x$, where for some $i<j$, an $(i,j)$-flow line branches into two edges that limit to critical points of the form $b^X_{i,m}$ and $b^Y_{m,j}$.  For this to occur, $x$ must have its image at the intersection of flow lines from $c_{i,m}$ and $c_{m,j}$ to  $b^X_{i,m}$ and $b^Y_{m,j}$, and since these flow lines are horizontal or vertical line segments, we see that either $(X,Y) = (U,L)$ or $(X,Y)=(R,D)$.  We show that the incoming edge $\alpha$ at $x$ must begin at a positive puncture at $c_{i,j}$, so that $\Gamma$ is as in (C) or (D) of Figure \ref{fig:4TreesB}. 

Suppose that instead $\alpha$ begins at a $Y_0$ vertex, $y$.  Since $\alpha$ is a portion of the $(i,j)$-flow line from $c_{i,j}$ to the image of $x$, the image of $\alpha$ is entirely contained in the intersection of the $1$-st $(i,m)$-quadrant and the $3$-rd $(m,j)$-quadrant.  Then, arguing as in the sub-cases above, considering the other outgoing edge at $y$ and applying Lemma \ref{lem:1st3rd} contradicts (ii).

\end{proof}

\subsection{Issues with extending the argument to other square types}
The biggest difficulty with extending the argument to remaining squares of type (1)-(12) is that for squares with crossing arcs additional Reeb chords are present.  As an example, consider the Type (2) square where a crossing between sheets $k$ and $k+1$ runs vertically through the center of the square.  To the left of the crossing locus, an additional local maximum $\tilde{c}_{k+1,k}$ appears and two new saddle points $\tilde{b}^U_{k+1,k}$ and $\tilde{b}^D_{k+1,k}$ appear along the edges $U$ and $D$.  (We order the subscripts, so that the first subscript indicates the upper endpoint of the Reeb chord.)  The new Reeb chords can appear as outputs in GFTs beginning at some $c_{i,j}$ that do not have the form identified in Figure \ref{fig:4TreesB}.  For instance, $\partial c_{i,j}$ may have terms of the form $c_{i,k+1}\tilde{b}_{k+1,k}^U a^{-,-}_{k,j}$;  see Figure \ref{fig:2tree}.

\begin{figure}

\quad

\quad

\labellist
\small
\pinlabel $\tilde{c}_{k+1,k}$ [l] at 360 242
\pinlabel $c_{i,k+1}$ [tl] at 496 214
\pinlabel $\tilde{b}^U_{k+1,k}$ [b] at 356 328
\pinlabel $\tilde{b}^D_{k+1,k}$ [t] at 356 -2
\pinlabel $a^{-,-}_{k,j}$ [r] at 268 2
\pinlabel $c_{i,j}$ [l] at 544 246
\pinlabel $b^L_{k+1,k}$ [r] at 268 242
\pinlabel $a^{-,-}_{k,j}$ [t] at 138 102 
\pinlabel $\tilde{b}^U_{k+1,k}$ [t] at 44 102
\pinlabel $c_{i,k+1}$ [t] at 4 140
\pinlabel $c_{i,j}$ [b] at 52 246

\endlabellist
\centerline{ \includegraphics[scale=.4]{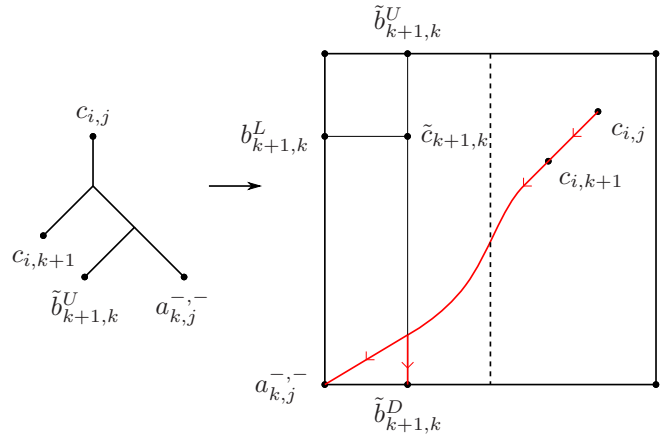} }
%\centerline{ \includegraphics[scale=.8]{images/SubdivideSq} }

\quad 

\caption{A rigid GFT in a Type (2) square giving $\partial c_{i,j} = c_{i,k+1}\tilde{b}^U_{k+1,k} a^{-,-}_{k,j}+ \cdots$. }
\label{fig:2tree}
\end{figure}

To avoid having to identify these additional trees, we instead observe the (more easily derived) formulas 
\[
\partial \tilde{c}_{k+1,k} = \tilde{b}^U_{k+1,k} + \tilde{b}^D_{k+1,k} + b^L_{k+1,k}; \quad \partial \tilde{b}^U_{k+1,k} = a^{-,+}_{k+1,k}; \quad \partial \tilde{b}^D_{k+1,k} = a^{-,-}_{k+1,k}.
\]
Then, (using \cite[Theorem 2.1]{RuSu1}) the quotient of the DGA by the ideal generated by the {\it exceptional generators} $\tilde{c}_{k+1,k}, \tilde{b}^U_{k+1,k}, \tilde{b}^D_{k+1,k}, b^L_{k+1,k}, a^{-,+}_{k+1,k}, a^{-,-}_{k+1,k}$ is stable tame isomorphic to the original DGA.  In this quotient, terms in the differential arising from the new GFTs with endpoints at exceptional generators all become $0$, so that the formula (\ref{eq:vanillaEx}) remains valid with the caveat that the matrices $A_{-,+}, B_L$ and $A_{-,-}$ have the $(k,k+1)$-entry equal to $0$.  This is precisely the form of the part of the cellular DGA associated to the corresponding square of $\mathcal{E}_{||}$ (where the crossing arc now sits directly above the left edge of the square).  In Section \ref{sec:Iso}, we carry out a similar quotient procedure for all squares, with care taken when canceling $a$ or $b$ Reeb chords that appear in the boundary of more than one $2$-cell.  (This is the reason that the technical requirements (A2)-(A4) were imposed on the square decomposition $\mathcal{E}_\pitchfork$.)

We mention briefly some of the other adjustments to the above argument that appear in Sections \ref{sec:CompLCH}-\ref{sec:SwallowComp}. 
\begin{enumerate}
\item  The construction of $\tilde{L}$ needs to allow for a perturbation to obtain the $1$-regular condition; see Section \ref{sec:ProofSetup} for  details.  As a result, in Sections \ref{sec:Constructions}-\ref{sec:ProofSetup} the location of Reeb chords is only specified up to $\epsilon$, and the behavior of individual flow lines is not as precisely known.  Typically, to restrict the locations of flow lines we require that  $-\nabla F_{i,j}$ points transversally to particular curves in $[-1,1]\times[-1,1]$, as such conditions are preserved by perturbation.   

\item  No major changes are required for the $1$-st Quadrant Lemma.  However, extra care is required in the $3$-rd Quadrant Lemma, for instance to allow the possibility that GFTs terminate at an $e$-vertex along a cusp edge.
 
\item  The location of the gradient trajectories connecting the $c_{i,j}$ to surrounding saddle points $b^X_{i,j}$ is no longer precisely known.  (The relative location of saddle points $b^X_{i,j}$ and $b^Y_{i,j}$ along edges seperated by a crossing or cusp arc are different, so that, even ignoring the perturbation, we should no longer expect these trajectories to be straight lines.)     
%and in some squares it is not clear that we can provide bounds on their location without imposing properties on $\tilde{L}$ that become tedious to construct.  
%Instead, when 
In identifying rigid GFTs without endpoints at $c$'s we make a more radical departure, and give a topological argument to show that the $\Z/2$-count of such GFTs is independent of the precise location of these flow lines.  See Section \ref{ssec:112btrees}.  (This added flexibility is also important in the case of (13)-(14) squares.)  

\end{enumerate}

\section{Computation of LCH, Part 1: $0$-cells and $1$-cells}
\label{sec:CompLCH}

In this section, we replace $L$ with the Legendrian isotopic surface $\tilde{L}$.  We then establish in Corollary \ref{cor:SVK} that the %LCH differential preserves the sub-algebras generated by 
 Reeb chords of $\tilde{L}$ located in a suitable neighborhood of any given cell of $\mathcal{E}_\pitchfork$ generate a sub-algebra that is preserved by the LCH differential.   In Propositions \ref{prop:LCH0comp} and \ref{prop:LCH1comp}, we compute the sub-DGAs associated to $0$-cells and $1$-cells.

\subsection{The Legendrian $\tL$}

\label{ssec:AxiomaticProperties}

Recall from Section \ref{sec:transverse} that we have found a cell decomposition, $\mathcal{E}_\pitchfork$, of a neighborhood of $\pi_x(L) \subset S$ into squares so that the projection of the crossing and cusp locus of $L$ to each square matches (up to isotopy) one of 14 model squares.  
Moreover, the closed squares are parametrized by $[-1,1]^2$, and (except possibly at the corners) 
 the parametrizations provide smooth local coordinates on $S$ that we notate as $(x_1,x_2)$.  Coordinates of neighboring squares fit together at $1$-cells and $0$-cells as specified in Section \ref{sec:transverse}.

\begin{theorem}  \label{thm:PropertiesofLtilde} Let $\mathcal{E}_\pitchfork$ be a transverse square decomposition for the front generic Legendrian $L \subset J^1(S)$. Then, there exists a Legendrian $\tilde{L} \subset J^1(S)$ and a Riemannian metric $g$ on $S$ such that: 
\begin{itemize}
\item Above each of the squares of $\mathcal{E}_\pitchfork$, the singular sets of $L$ and $\tilde{L}$ have the same topological form, i.e. they match the same square type with the same sheets meeting at crossing and cusp arcs and swallowtail points.  In particular, $\tilde{L}$ and $L$ are Legendrian isotopic.
\item The Legendrian and metric pair $(\tilde{L},g)$ is $1$-regular.
%\item With respect to the Riemannian metric from Construction \ref{construct:metric},  $\tilde{L}$ is $1$-regular.
\item In the local coordinates given by $\mathcal{E}_\pitchfork$, the local defining functions of $\tilde{L}$ and their gradients with respect to $g$ satisfy Properties \ref{pr:14models}-\ref{pr:SwitchBarriers} as stated below. 
\end{itemize}
\end{theorem}

%\footnote{\ms{7/22/15: This next paragraph should be expanded and put in a proof environment.} \dr{7-22:  Putting ``The proof of Theorem ?? will be given in Sections ...'' in a proof seems contradictory.  We are explicitly saying that the proof is there--and not here.  The rest of the paragraph is just pointing the reader to where the Properties are all stated.   Go ahead and edit if you want though.}}

The proof of Theorem \ref{thm:PropertiesofLtilde} will be given in Sections \ref{sec:Constructions}-\ref{sec:ProofSetup}.  The Properties \ref{pr:14models}-\ref{pr:1cmono} of Theorem \ref{thm:PropertiesofLtilde} are stated in the current section; Properties \ref{pr:Reeb2}-\ref{pr:CuspTransversality} appear in Section \ref{sec:Comp2Cells}; and Properties \ref{pr:monotonicityIIST}-\ref{pr:SwitchBarriers} which concern the form of $\tilde{L}$ above the swallow tail squares (13) and (14) are stated in Section \ref{sec:SwallowComp}.

\begin{notation}
 We denote the LCH DGA of $\tilde{L}$ with coefficients in $\Z/2$ (specialize all homology classes in $\Z/2[H_1(\tilde{L})]$ to $1$) and grading reduced modulo $m(L)$ by $(\lchA,\partial)$.  
\end{notation}

From Theorem \ref{thm:EkholmMain}, 
%there is a further Legendrian $L'$ that is also isotopic to $L$, such that 
the differential of $(\lchA,\partial)$ can be computed by summing over rigid GFTs of $\tilde{L}$. 
This DGA is computed up to stable tame isomorphism over the course of Sections \ref{sec:Comp2Cells}-\ref{sec:Iso} with the end result stated in Proposition \ref{prop:lchAI}  below.

\begin{notation}
We use the notation $\nabla F = (\grad_{x_1} F, \grad_{x_2} F)$ for the components of the gradient of a function $F$ with respect to the metric $g$ from Theorem \ref{thm:PropertiesofLtilde}. 
\end{notation}

\subsubsection{Cusp locus, crossing locus}  
\label{sssec:AxiomNotation}

As stated in Theorem \ref{thm:PropertiesofLtilde}, the projection of the singular set of $\tL$ to each square of $\cE_\tra$ matches one of the square types (1)-(14) up to an ambient isotopy of $[-1,1]\times[-1,1]$.  The following property provides a stronger restriction on the location of the cusp and crossing locus.
\begin{property}[The square models]\label{pr:14models}
%The crossing loci and cusp loci appear as drawn in Section \ref{sec:transverse}.
All crossing loci are contained in the region $\{-1/4 < x_1 < 1/4\} \cup \{-1/4 < x_2 < 1/4\}.$
All cusp loci are straight lines in $\{x_1 = -3/8\}$ or $\{x_2 = -3/8\},$
except possibly in disks of radii $3/32$ containing the swallowtail points.
\end{property}

%\subsection{The sub-DGAs $\lchA(e^d_\alpha)$}

%To each cell $e^d_\alpha$, $d=0,1,2$ of $\mathcal{E}_\pitchfork$ we associate a sub-DGA of $(\lchA,\partial)$.  

\subsection{Invariant neighborhoods of $0$-cells and $1$-cells and the sub-DGAs $\lchA(e^d_\alpha)$}  
%\label{sssec:AxiomNotation}
\label{sssec:FlowLinesBoundary}

%In Section \ref{sec:RegReq}, various compatibility requirements were placed on the parametrization of closed squares that share $0$ or $1$-cells in their boundaries.  
Recall from Section \ref{sec:RegReq} that each $0$-cell, $e^0_\alpha$, has a disk neighborhood consisting of the union of balls of radius $1/16$ (with respect to the euclidean metric used in the domain of parametrization $[-1,1]^2$) centered at all corners of $2$-cells where $e^\alpha_0$ appears.  Denote this neighborhood as %$B(e^0_\alpha, 1/16)$, 
$N(e^0_\alpha)$, and note that requirement (2) of Section \ref{sec:RegReq} shows that $\partial N(e^0_\alpha)$ is a smooth circle in $S$.

\begin{property}[$0$-cells]  \label{pr:0cells}  Let $e^0_\alpha$ be a $0$-cell.
For all local defining functions $F_i$ and $F_j$ with $F_{i} > F_j$ in $N(e^0_\alpha)$, the gradient $-\nabla F_{i,j}$ points inward along the boundary  $\partial N(e^0_\alpha)$.
\end{property}

\begin{property}[$1$-cells]  \label{pr:1cells}

  Let $e^1_\alpha$ be a $1$-cell with endpoints at the $0$-cells $e^0_-$ and $e^0_+$.  Then, $e^1_\alpha$ has a neighborhood  $N(e^1_\alpha) \subset S$ with the following features.

\begin{enumerate}	
\item We have $N(e^0_-),N(e^0_+) \subset N(e^1_\alpha)$ and the boundary of $N(e^1_\alpha)$ is a piecewise smooth curve consisting of the union of an arc from the boundary of $N(e^0_-)$, an arc from the boundary of $N(e^0_+)$, and two paths $P_1$ and $P_2$.  

\item Each of $P_l$ is contained in the interior of a single $2$-cell that contains $e^1_\alpha$ as an edge. In $[-1,1]\times[-1,1]$ coordinates, $P_l$ is piecewise linear, contained within a distance of $1/32$ from $e^1_\alpha,$ 
and monotonically increasing in the coordinate, $x_i$, that is parametrizes $e^1_\alpha$, and is parallel to $e^1_\alpha$ for $1/2 \leq x_i \leq 3/4$.

\item 
For any $(x_1,x_2) \in P_l$ and $i,j$ such that $F_{i,j}(x_1,x_2) >0,$ we have that $-\nabla F_{i,j}(x_1,x_2)$ is transverse to $P_l$ at $(x_1,x_2)$ and points into $N(e^1_\alpha)$.  (If $(x_1,x_2)$ is a non-smooth point of $P_l$, then $-\nabla F_{i,j}(x_1,x_2)$ should be transverse to both segments of $P_l$ that meet at $(x_1,x_2)$.)

\item Suppose the sheets defined by $F_i$ and $F_{j}$ cross at $(x_1,x_2) \in P_l$.  Then $(x_1,x_2)$ is a non-smooth point of $P_l$ where two line segments of $P_l$ meet.  Label these segments $Q_1$ and $Q_2$ so that $F_{i} \geq F_j$ along $Q_1$ and $F_{j} \geq F_{i}$ along $Q_2$.  Then, at $(x_1,x_2)$,  $-\nabla F_{i,j}$ (resp. $-\nabla F_{j,i}$) is transverse to $Q_1$ (resp. $Q_2$) and points to the side of $Q_1$ (resp. $Q_2$) that borders $N(e^1_\alpha)$.

\end{enumerate}
\end{property}

\begin{remark}
The item (4) in Property \ref{pr:1cells} is included to simplify the proof of Theorem \ref{thm:PropertiesofLtilde} and is not needed for computing $(\lchA,\partial)$.
\end{remark}

For any $2$-cell, $e^2_\alpha$, we define a neighborhood $N(e^2_\alpha)$ that is the union of $e^2_\alpha$ and the neighborhoods of its boundary $0$- and $1$-cells given by Properties \ref{pr:0cells} and \ref{pr:1cells}.  See Figure \ref{fig:InvariantN}.  

\begin{figure}

\labellist
\small
\endlabellist
\centerline{ \includegraphics[scale=.6]{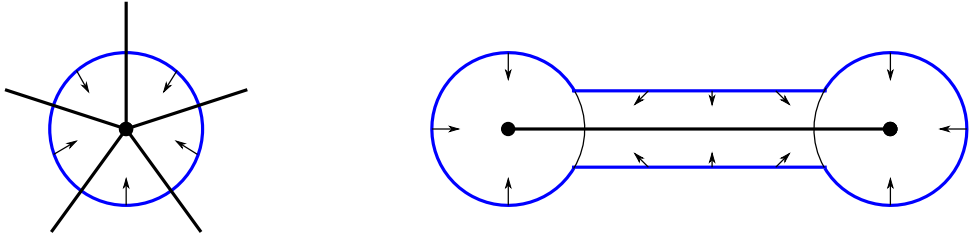} }

\caption{Neighborhoods of $0$-cells and $1$-cells.  (left) The neighborhood $N(e^0_\alpha)$ is pictured (boundary in blue)  along with the $1$-skeleton near $e^0_\alpha$. (right) The boundary of $N(e^1_\alpha)$ appears in blue with the polygonal paths $P_1$ and $P_2$ pictured as horizontal segments. The division of $N(e^1_\alpha)$ into $N(e^0_-)$, $\widehat{N}(e^1_\alpha)$, and $N(e^0_+)$ is indicated by the thin black lines.  The arrows indicate the direction of negative gradients $-\nabla F_{i,j}$ with $F_{i}>F_j$ as specified by Properties \ref{pr:0cells} and \ref{pr:1cells}.}
\label{fig:InvariantN}
\end{figure}

The following allows us to localize the computation of LCH to a square-by-square calculation.
\begin{lemma}  \label{lem:N} Let $N$ denote any of the neighborhoods $N(e^0_\alpha)$, $N(e^1_\alpha)$ or $N(e^2_\alpha)$.  Any (partial or complete) gradient flow tree for $\tilde{L}$ that starts at a point in $N$ must have its entire image contained in $N$.
\end{lemma}
\begin{proof}
From Properties \ref{pr:0cells} and \ref{pr:1cells}, we have that at any point on the boundary of $N$ all negative gradients of positive local difference functions point into $N$.  Therefore, no branch of the tree can ever cross out of $N$ as $t$ increases.  
\end{proof}

For any cell $e^d_\alpha$, $d = 0,1,$ or $2$, of $\cE_\tra,$ let $\lchA(e^d_\alpha)$ denote the sub-algebra of $\lchA$ generated by Reeb chords belonging to $N(e^d_\alpha)$.

\begin{corollary}
\label{cor:SVK}
For any cell $e^d_\alpha$, of $\cE_\tra,$ the sub-algebra $\lchA(e^d_\alpha)$ 
is a sub-DGA of $(\lchA,\partial)$.
\end{corollary}  
\begin{proof}
Since the differential of $(\lchA,\partial)$ may be computed via a count of rigid GFTs of $\tilde{L}$, the corollary follows from Lemma \ref{lem:N}.
\end{proof}

%In computing $(\mathcal{A}_{\mathit{LCH}}, \partial ),$ therefore,  we can compute differentials of generators in the neighborhood of each cell separately.  In the remainder of Section \ref{sec:CompLCH} we focus primarily on $2$-cells of Type (1)-(12), while the computation of the subalgebras associated to swallow tail squares of Type (13) and (14) is postponed until Section \ref{sec:SwallowComp}.

\subsubsection{Notations for sheets and Reeb chords}

 %In the following paragraphs, we will enumerate the Reeb chords in all $0$ and $1$-cells as well as in $2$-cells of Type (1)-(12).  %It is extremely awkward to assign a  fixed notation to the Reeb chords and the sheets of $\tilde{L}$ in the front projection, once and for all.  
%many formulas are simplified by allowing the same Reeb chord to have different notations depending on the context.    
%Rather than attempt to do so, 

In computing the sub-DGAs $\lchA(e^d_\alpha)$, we will use a notation for Reeb chords arising from a numbering of the sheets restricted to $e^d_\alpha$ as $S_1, \ldots, S_n$.  The notation for a Reeb  chord uses two subscripts to indicate the endpoints of the Reeb chord by the following:

\begin{convention}  \label{conv:subscript}
The order of sub-scripts for a Reeb chord $x_{i,j}$ is  always chosen so that the first (resp. second) subscript indicates the upper (resp. lower) sheet of the Reeb chord 
\end{convention}

In general, the enumeration of sheets cannot be carried out in a consistent manner globally, and, as a result, the same Reeb chord may be denoted with different notations in the context of different $\lchA(e^d_\alpha)$.  %In Section \ref{sec:CompLCH}-???, 
%We typically work with one single closed $2$-cell at a time, and during such local considerations the notation for Reeb chords is fixed.  However, we caution that a single $0$-cell or $1$-cell belongs to the closure of more than one $2$-cell, and the notation provided to a Reeb chord above a $0$ or $1$-cell in the context of different closed $2$-cells may not agree.  
We pay special attention to this point in Section \ref{sec:Iso} where we combine our local calculations of the $\lchA(e^d_\alpha)$ to compute the full DGA $(\lchA,\partial)$.% a global isomorphism between the LCH and cellular DGAs.  

\begin{remark} This issue is parallel to the choice of total orderings of sheets required in defining the differential in the cellular DGA. 
\end{remark}

\subsection{Computation of $(\lchA, \partial)$ near the $0$-skeleton}

Recall from Section \ref{ssec:GFT} that we denote local defining functions for sheets of $\tilde{L}$ labeled $S_1, \ldots, S_n$ as $F_1, \ldots, F_n$, and that a Reeb chord starting on sheet $S_j$ and ending on sheet $S_i$ (with respect to the Reeb vector field $\partial_z$) corresponds to a critical point of $F_{i,j}:=F_i-F_j$ with positive critical value.  Trajectories of $-\nabla F_{i,j}$ are referred to as $(i,j)$-flow line.  The abbreviations GFT and PFT respectively refer to gradient flow trees and partial flow trees of $\tilde{L}$.

Consider a $0$-cell, $e^0_\alpha$. By Property \ref{pr:14models}, above $N(e^0_\alpha)$, $\tilde{L}$ is a union of non-intersecting sheets that we label $S_1, \ldots, S_n$ with (pointwise) decreasing $z$-coordinate, i.e. $F_1>F_2> \ldots > F_n$ above $N(e^0_\alpha)$.
\begin{property}[Reeb chords above $0$-cells] \label{pr:Reeb0}  
For any pair of sheets, $S_i, S_j$, above $e^0_\alpha$ with $F_i > F_j$, there is a unique Reeb chord $a_{i,j}$ in $N(e^0_\alpha)$.  Each $a_{i,j}$ is a non-degenerate local minimum of $F_{i,j}$, and there are no other Reeb chords in $N(e^0_\alpha)$. 
\end{property}

%[COMPUTE THE DIFFERENTIAL ABOVE $0$-cells here? For now anyway.]

\begin{proposition} \label{prop:LCH0comp} For some $0$-cell, $e^0_\alpha$, let $a_{i,j}$, $1 \leq i < j \leq n$, be the Reeb chords $N(e^0_\alpha)$.  In $(\lchA(e^0_\alpha), \partial)$ we have
\[
\partial a_{i,j} = \sum_{i<k<j} a_{i,k}a_{k,j}.
\]
\end{proposition}
\begin{proof}
By Theorem \ref{thm:EkholmMain} and Theorem \ref{thm:PropertiesofLtilde}, $\partial a_{i,j}$ is computed as a sum over rigid GFTs with a postive puncture at $a_{i,j}$.  Any such flow tree, $\Gamma$, is contained entirely in $N(e^0_\alpha)$ (by Lemma \ref{lem:N}.  Since the cusp locus of $\tilde{L}$ is disjoint from $N(e^0_\alpha)$, 
%$\Gamma$ can have only punctures at other $a_{kl}$ Reeb chords and $Y_0$-vertices.   
internal vertices of $\Gamma$ can only be $Y_0$-vertices, and all outputs of $\Gamma$ must be at Reeb chords in $N(e^0_\alpha)$.
Thus, Proposition \ref{prop:EY1SWFormula} shows that $\Gamma$ must have two negative punctures and, therefore, a single $Y_0$.  As the initial branch of $\Gamma$ must be constant (since $a_{i,j}$ is a local minimum), we have exactly one such tree for each intermediate sheet $S_k$ with $F_i > F_k > F_j$ as follows:  After the $Y_0$ the outgoing edges of $\Gamma$ are flow lines for $-\nabla F_{i,k}$ and $-\nabla F_{k,j}$ that must limit to $a_{i,k}$ and $a_{k,j}$ respectively.  See Figure \ref{fig:0Trees}.
\end{proof}

\begin{figure}

\labellist
\small
%\pinlabel $1/4$ [t] at 338 21
\pinlabel $a_{i,j}$ [b] at 32 100
\pinlabel $a_{i,k}$ [t] at 0 -2
\pinlabel $a_{k,j}$ [t] at 64 -2
\endlabellist
\centerline{ \includegraphics[scale=.6]{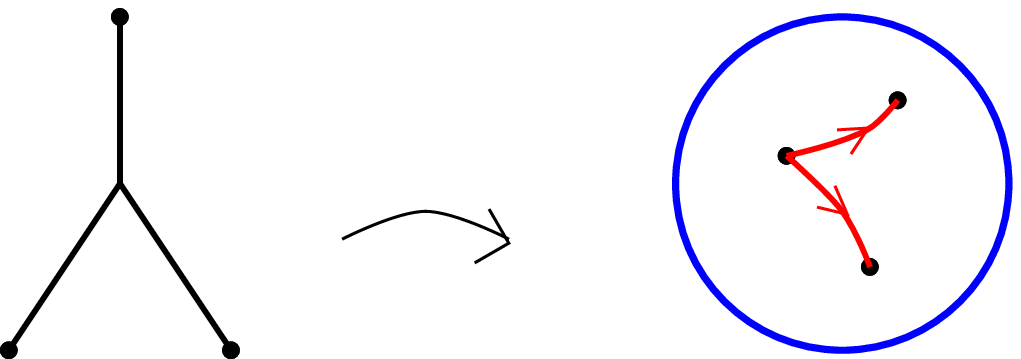} }

\quad 

%%\centerline{ \includegraphics[scale=.8]{images/SubdivideSq} } %Dan had already commented out this line

\caption{A rigid GFT $\Gamma$ beginning at $a_{i,j}$ with domain and image pictured.  The top branch of the tree maps to the constant $-\nabla F_{i,j}$ at $a_{i,j}$.}
\label{fig:0Trees}
\end{figure}

\subsection{Computation of $(\lchA, \partial)$ near the $1$-skeleton}
\label{ssec:Compute1cell}

Consider now a $1$-cell $e^1_\alpha$ parametrized by $[-1,1]$, and label the sheets above $N(e^1_\alpha)$ as $S_1, \ldots, S_n$ with descending $z$-coordinate in the order they appear above $+1$.  
As in Section \ref{ssec:ElemSq}, above $e^1_\alpha$ the front projection of $\tL$ has one of four forms that we refer to with the abbreviations
\begin{enumerate}
\item[(PV)]  Plain Vanilla:  The sheets do not intersect. 
\item[(1Cr)]  $1$-crossing:  Sheets $k$ and $k+1$ cross one another somewhere in $[-1/4,1/4]$. 
\item[(2Cr)]  $2$-crossing:  Sheet $k+2$ crosses sheets $k+1$ and $k$ somewhere in $[-1/4,1/4]$ as $x \in [-1,1]$ decreases from $x=+1$ to $x=-1$.
\item[(Cu)]  Left cusp:  Sheets $k$ and $k+1$ meet at a cusp at $x = -3/8$.
\end{enumerate}  
 See Figure \ref{fig:EdgeTypesB}.  As specified by Property \ref{pr:14models}, the location of the cusp point in (Cu) is at $x = -3/8$, and the location of crossings in (1Cr) and (2Cr) are between $x= -1/4$ and $x=1/4$.  

Let $\widehat{N}(e^1_\alpha)$ denote $N(e^1_\alpha)$ with the the interior of each $N(e^0_\pm)$ removed where $e^0_+$ and $e^0_-$ are the endpoints of $e^1_\alpha$.  Then, $N(e^1_\alpha) = N(e^0_+) \cup N(e^0_-) \cup \widehat{N}(e^1_\alpha)$ with all subsets compact, see Figure \ref{fig:InvariantN}.  
Moreover, $\widehat{N}(e^1_\alpha)$ consists of a portion of $e^1_\alpha$ together with subsets of two squares that contain $e^1_\alpha$ in their boundary.  In each of these squares, the intersection of $\widehat{N}(e^1_\alpha)$  lies, with respect to the parametrization by $[-1,1]^2$, within a strip of distance $1/32$ from $e^1_\alpha$.  According to the third regularity requirement stated in Section \ref{sec:RegReq}, the parametrizations of these bordering squares can be combined in a suitable manner to realize $\widehat{N}(e^1_\alpha)$ as a subset of $(-1, 1) \times [-1/32, 1/32]$.  The boundary of $\widehat{N}(e^1_\alpha)$ consist of some subsets of $\partial N(e^0_\pm)$ which are circular arcs centered at $(\pm1,0)$ together with polygonal paths $P_1$ and $P_2$ which respectively belong to $[-1,1] \times (0,1/32]$ and $[-1,1] \times [-1/32, 0)$ and project to the first component in a one-to-one manner.  

Label sheets above $N(e^1_\alpha)$ as $S_1, \ldots, S_n$ with decreasing $z$-coordinate as they appear above $e^+_0$.  For $p \in \widehat{N}(e^1_\alpha) \subset [-1,1] \times [-1/32,1/32]$ we let $x_1(p) \subset [-1,1]$ denote the first coordinate of $p$.

\begin{property}[Reeb chords above $1$-cells] \label{pr:Reeb1} The only Reeb chords in $\widehat{N}(e^1_\alpha)$ are as follows.  
\begin{itemize}
\item  For every $i<j$, there is a Reeb chord $b_{i,j}$ with  $x_1(b_{i,j}) \in [1/2,3/4]$.
\item  If sheets $i$ and $j$ with $i <j$ cross above $e^1_\alpha$, then there is a Reeb chord $\tilde{b}_{j, i}$ with  $x_{1}(\tilde{b}_{j i}) \in [-3/4,-1/2]$.
\end{itemize}
Moreover, all of the $b_{i,j}$ and $\tilde{b}_{j,i}$ are non-degenerate saddle points.
\end{property}
(We follow Convention \ref{conv:subscript}; in all cases the first and second subscripts refer respectively to the upper and lower sheets of the Reeb chord.)

Note that the only Reeb chords of the second type are $\tilde{b}_{k+1,k}$ for a (1Cr) edge and $\tilde{b}_{k+2,k+1}, \tilde{b}_{k+2,k}$ for a (2Cr) edge.

\subsubsection{Statement of the differential}

Let $a^-_{i,j}, a^+_{i,j}, b_{i,j}, \tilde{b}_{i,j}$ denote Reeb chords in $N(e^0_-)$, $N(e^0_+)$, and $N(e^{1}_\alpha)$ respectively, where subscripts of all Reeb chords %indices $i$ and $j$ refer respectively to the upper and lower sheets of the Reeb chord where 
now correspond to the ordering of sheets as $S_1, \ldots, S_n$  above $e^+_0$.  

We place these generators into strictly upper triangular matrices $A_-$, $A_+$, and $B$ whose respective $(i,j)$-entries with $i<j$ are given by $a^-_{i,j}, a^+_{i,j},$ and $b_{i,j}$, when they exist.  (The $\tilde{b}_{i,j}$ do not appear in any of these matrices.)  All other entries are $0$ with the exception that if $S_k$ and $S_{k+1}$ meet at a cusp edge above $e^1_\alpha$, then the $(k,k+1)$-entry of $A_-$ is $1$.  (Note that for $i<j$, if $S_i$ and $S_j$ cross above $e^1_\alpha$ or exactly one of these sheets ends at a cusp edge, then the $(i,j)$-entry of $A_-$ is $0$.) 
%with $i\leq j$ are $0$.  
 %make the convention that if $F_i < F_j$ at $e^0_-$ or either $F_i$ or $F_j$ is not defined at $e^0_-$, then $a^-_{i,j} =0$ with 
% the exception that if $S_k$ and $S_{k+1}$ meet at a cusp edge above $e^1_\alpha$, then $a^-_{k,k+1}=1$.

%\dr{Could be better organization to move this statement to appear before Property 6.}

\begin{proposition}  \label{prop:LCH1comp}
Let $e^1_\alpha$ be a $1$-cell with Reeb chords $a^-_{i,j}, a^+_{i,j}, b_{i,j}, \tilde{b}_{i,j}$ as above.  In $(\lchA(e^1_\alpha), \partial)$ we have
\[
\partial B = A_+(I+B)  + (I+B) A_- + X
\]
where all entries of $X$ belongs to the ideal generated by the $\tilde{b}_{i,j}$.
\end{proposition}

After discussing relevant properties of $\tilde{L}$ above $1$-cells, the proof of Proposition \ref{prop:LCH1comp} is given at the conclusion of this section.

\subsubsection{Properties of Reeb chords and gradients above $1$-cells}

Locations of Reeb chords are specified more precisely in the
%\footnote{\dr{Using different notation for the $b_{i,j}$ because of a different numbering at the upper right corner of a $2$-cell should probably result in a different notation of the $\beta_{i,j}$ along the edge.  This point will probably only come up in the Construction section.  Actually, the superscript $X$ may clear up any ambiguity.}} 
following.
\begin{property}[Location of $1$-cell Reeb chords] \label{pr:Location1} For a $1$-cell $e^1_\alpha$ with $n$ sheets above $x =+1$, there exists a collection of numbers $\beta_{i,j} \in [1/2, 3/4]$, for any $i<j$ and $\epsilon >0$, 
  depending only on $n$, such that 
\begin{itemize}
\item  the intervals $[\beta_{i,j} - \epsilon, \beta_{i,j} + \epsilon]$ are all disjoint; 
\item the $\beta_{i,j}$ appear in lexicographic order, i.e.  
\begin{equation}
\label{eq:StaircaseAxiom}
\beta_{i,j} < \beta_{i',j'} \quad \mbox{whenever $i < i'$, or $i=i'$ and $j < j'$;  and}
\end{equation}
\item $x_1(b_{i,j}) \in (\beta_{i,j} - \epsilon, \beta_{i,j} + \e)$ 
\end{itemize}
Moreover, in the case of a (1Cr) or (2Cr) edge respectively, there are in addition $\tilde{\beta}_{k+1,k} \in [-3/4, -1/2]$ or $\tilde{\beta}_{k+2,k+1},\tilde{\beta}_{k+2,k} \in [-3/4, -1/2]$ such that  $x_1(\tilde{b}_{j,i}) \in (\tilde{\beta}_{j,i}-\e, \tilde{\beta}_{j,i}+\e)$ and, in the (2Cr) case, the $[\tilde{\beta}_{i,j} - \epsilon, \tilde{\beta_{i,j}} + \epsilon]$ are disjoint with  $\tilde{\beta}_{k+2,k} < \tilde{\beta}_{k+2,k+1}$.
\end{property}

Property \ref{pr:Location1} is illustrated in Figure \ref{fig:1staircase}.

\begin{figure}

\quad

\quad 

\labellist
\small
\pinlabel $b_{1,2}$ [b] at 31 52
\pinlabel $(\beta_{i,j}-\epsilon,\beta_{i,j}+\epsilon)$ [t] at 112 0
\pinlabel $b_{1,3}$ [b] at 72 52
\pinlabel $b_{1,4}$ [b] at 112 52
\pinlabel $b_{2,3}$ [b] at 168 52
\pinlabel $b_{2,4}$ [b] at 208 52
\pinlabel $b_{3,4}$ [b] at 264 52
\endlabellist
\centerline{ \includegraphics[scale=.6]{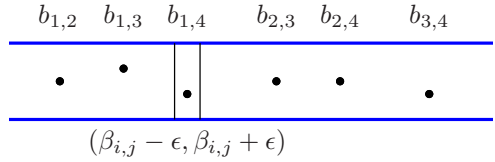} }

%%\centerline{ \includegraphics[scale=.8]{images/SubdivideSq} } %Dan had already commented out this line

\caption{The $b_{i,j}$ are lexicographically ordered by their $x_1$-coordinates.}
\label{fig:1staircase}
\end{figure}

\begin{property}[Monotonicity above $1$-cells]  \label{pr:1cmono}
Let $(x_1,x_2) \in \widehat{N}(e^1_\alpha)$, and $\nabla F_{i,j} = (\grad_{x_1}F_{i,j}, \grad_{x_2}F_{i,j})$.  

Suppose $F_i > F_j$ when $x_1 \in [1/4, 1].$
\begin{itemize}
\item If $ \beta_{i,j} + \epsilon \leq x_1$, then $-\grad_{x_1} F_{i,j}(x_1,x_2) > 0$. 
\item If $  x_1 \in [1/4, \beta_{i,j} - \epsilon]$, then $-\grad_{x_1} F_{i,j}(x_1,x_2) < 0$.
\end{itemize}

Suppose $F_i > F_j$ when $x_1 \in [-1,-1/4]$ and that sheets $S_i$ and $S_j$ do not cross in $\widehat{N}(e^1_\alpha)$.  
\begin{itemize}
\item If  $ x_1 \leq -1/4$, then  $-\grad_{x_1} F_{i,j}(x_1,x_2) < 0$
\end{itemize}

Suppose $F_i > F_j$ when $x_1 \in [-1,-1/4]$ and that sheets $S_i$ and $S_j$ cross in $\widehat{N}(e^1_\alpha)$.  
\begin{itemize}
\item If  $  x_1 \in [\tilde{\beta}_{i,j} + \epsilon , -1/4]$, then  $-\grad_{x_1} F_{i,j}(x_1,x_2) > 0$.
\item If  $ x_1 \leq \tilde{\beta}_{i,j} - \epsilon$, then  $-\grad_{x_1} F_{i,j}(x_1,x_2) < 0$.
\end{itemize}
\end{property}

See Figure \ref{fig:Monotone1Sk}.  

\begin{figure}

\labellist
\small
\pinlabel $b_{i,j}$ [b] at 430 229
\pinlabel $(\beta_{i,j}-\epsilon,\beta_{i,j}+\epsilon)$ [t] at 433 173
\pinlabel $1/4$ [t] at 338 173
\pinlabel $-1/4$ [t] at 273 173
\pinlabel $\cdots$  at 305 201
\pinlabel $\tilde{b}_{i,j}$ [b] at 177 77
\pinlabel $(\tilde{\beta}_{i,j}-\epsilon,\tilde{\beta}_{i,j}+\epsilon)$ [t] at 177 21
%\pinlabel $1/4$ [t] at 338 21
\pinlabel $-1/4$ [t] at 273 21
\pinlabel $\cdots$  at 305 49
\endlabellist
\centerline{ \includegraphics[scale=.6]{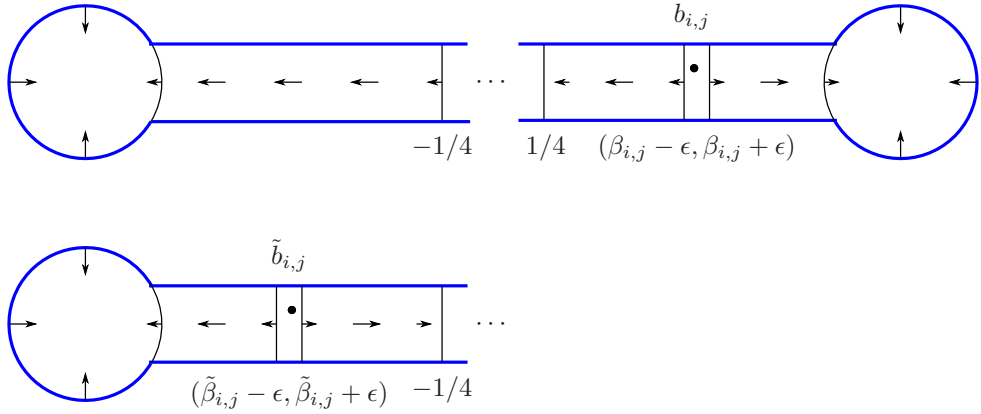} }

%%\centerline{ \includegraphics[scale=.8]{images/SubdivideSq} } %Dan had already commented out this line

\caption{The horizontal component of $-\nabla F_{i,j}$ with $F_i> F_j$ in the neighborhood of a $1$-cell when (top) the $S_i$ and $S_j$ do not cross or when (bottom) sheets $S_i$ and $S_j$ cross in the $1$-cell with $S_i$ above $S_j$ when $x_1 \leq -1/4$.}
\label{fig:Monotone1Sk}
\end{figure}

\subsubsection{Stable and unstable manifolds of the $b_{i,j}$}

The $b_{i,j}$ are saddle points of $F_i-F_j$, and thus their {\it stable} and {\it unstable} manifolds with respect to $-\nabla F_{i,j}$  are $1$-dimensional open disks.  
The stable (resp. unstable) manifold of $b_{i,j}$ consists of $b_{i,j}$ together with two non-constant $(i,j)$-flow lines ending (resp. beginning) at $b_{i,j}$, i.e. limiting to $b_{i,j}$ as $t \rightarrow +\infty$ (resp. $t \rightarrow -\infty$).  By Properties \ref{pr:0cells} and \ref{pr:1cells}, a flow line belonging to the unstable manifold of $b_{i,j}$ is entirely contained in $N(e^1_\alpha)$ as long as it remains in a region where $F_i -F_j$ is positive.  Thus, as $t$ increases towards $+\infty$, such a flow line may either limit to a Reeb chord above $N(e^0_{\pm})$, reach the cusp edge at an $e$-vertex, or {\it terminate at the cusp or crossing locus}.     In more detail, for $1$-cells of Type (1Cr) or (2Cr), the flowline may pass the crossing locus and enter a region where $F_i-F_j <0$; for $1$-cells of  Type (Cu) the flowline may reach the cusp locus where, if the graphs of $F_i$ and $F_j$ meet at the cusp edge, it ends as a $e$-vertex, or, when exactly one of the sheets $S_i$ and $S_j$ is a cusp sheet, the difference function $F_i-F_j$ ceases to be defined.  
%Notice that with our conventions for the meaning of $a^-_{i,j}$ these rigid flow trees contribute exactly the term $a^+_{i,j}+ a^-_{i,j}$ to $\partial b_{i,j}$.

\begin{proposition} \label{prop:bUS} Consider a Reeb chord $b_{i,j}$ in $N(e^1_\alpha)$ where  $e^1_\alpha$ is a $1$-cell with endpoints $e^0_\pm$.
\begin{enumerate}
\item One of the non-constant flow lines of the unstable manifold of $b_{i,j}$ limits to a Reeb chord in $N(e^0_+)$, while the other either limits to a Reeb chord in $N(e^0_-)$; reaches an $e$-vertex; or terminates at the cusp or crossing locus.
\item Any point $p$ in the intersection of the stable manifold of $b_{i,j}$ with $N(e^1_\alpha)$ satisfies $x_1(p) \in [\beta_{i,j} - \epsilon, \beta_{i,j} + \e]$ with respect to the coordinates $\widehat{N}(e^1_\alpha) \subset [-1,1] \times [-1/32,1/32]$.
\item The two flowlines of the stable manifold of $b_{i,j}$ exit $N(e^1_\alpha)$ along opposite sides of the boundary.  
\end{enumerate}
\end{proposition}

See Figure \ref{fig:StableM1Sk}.

\begin{figure}

\labellist
\small
\pinlabel $b_{i,j}$ [b] at 375 92
%\pinlabel $(\beta_{i,j}-\epsilon,\beta_{i,j}+\epsilon)$ [t] at 433 173
\pinlabel $a^+_{i,j}$ [t] at 537 72
\pinlabel $a^-_{i,j}$ [b] at 56 63
%\pinlabel $\cdots$  at 305 201
%\pinlabel $\tilde{b}_{i,j}$ [b] at 177 77
%\pinlabel $(\tilde{\beta}_{i,j}-\epsilon,\tilde{\beta}_{i,j}+\epsilon)$ [t] at 177 21
%\pinlabel $1/4$ [t] at 338 21
%\pinlabel $-1/4$ [t] at 273 21
%\pinlabel $\cdots$  at 305 49
\endlabellist
\centerline{ \includegraphics[scale=.6]{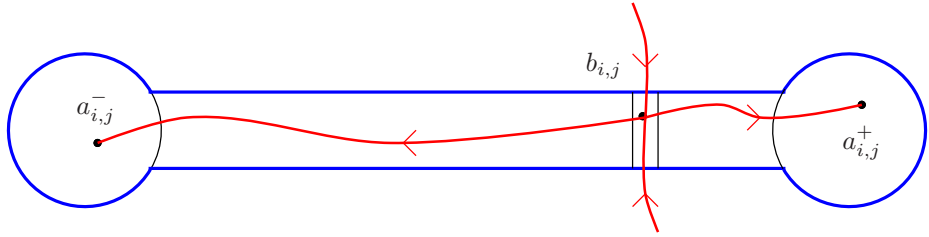} }

%%\centerline{ \includegraphics[scale=.8]{images/SubdivideSq} } %Dan had already commented out this line

\caption{The stable and unstable manifold of $b_{i,j}$.  If sheet $S_i$ is no longer above $S_j$ when $x_1 \leq -1/4$ due to a crossing or cusp, then the left flowline of the unstable manifold ends at the crossing or cusp locus. }
\label{fig:StableM1Sk}
\end{figure}

\begin{proof}
To verify (1), suppose that instead both flow lines of the unstable manifold limit to the same point or that they both end at the cusp/crossing locus. Then, the unstable manifold of $b_{i,j}$ together with the limiting critical point, or some appropriate portion of the crossing or cusp locus, bounds a region $R \subset N(e^1_\alpha)$ that does not contain any critical points of $F_{i,j}$ in its interior.  As the unstable manifold of $b_{i,j}$ seperates the two non-constant flow lines of the stable manifold of $b_{i,j}$,
one of the two $(i,j)$-flow lines from the stable manifold of $b_{i,j}$  would have to be entirely contained in $R$.  This is impossible since as $t$ decreases toward $-\infty$ this flow line cannot limit to critical points above $N(e^0_\pm)$, since they are local minimums, and cannot run into the crossing locus (since $F_{i,j}$ increases as $t$ decreases) or the cusp locus (by Property \ref{pr:1cmono}).

For (2), Property \ref{pr:1cmono} shows that as $t$ decreases from $+\infty$ a flowline of the stable manifold cannot cross the vertical lines $x_1=\beta_{i,j}\pm\epsilon$.  Finally, (3) follows from (1) and (2), since the unstable manifold of $b_{i,j}$ seperates $\widehat{N}(e^1_\alpha)$ into two halves.
\end{proof}

%\begin{property}[Intersections of unstable and stable manifolds of the $b_{i,j}$]  \label{pr:IntersectB} 
%For any $i<j$ and $i'<j'$, the unstable manifold of $b_{i,j}$ intersects the stable manifold of $b_{i',j'}$ transversally in a unique point $(x_1,x_2) \in \widehat{N}(e^1_\alpha) \subset [-1,1] \times [-1/32,1/32]$ 
% with $x_1 \in (\beta_{i',j'}-\epsilon, \beta_{i',j'} + \epsilon)$.  
%\end{property}

\begin{lemma}  \label{lem:IntersectB}
For any $i<j$ and $i'<j'$ with $(i,j) \neq (i',j')$, the unstable manifold of $b_{i,j}$ intersects the stable manifold of $b_{i',j'}$ transversally in an odd number of points.  Any such intersection point $(x_1,x_2) \in \widehat{N}(e^1_\alpha) \subset [-1,1] \times [-1/32,1/32]$ 
 satisfies $x_1 \in (\beta_{i',j'}-\epsilon, \beta_{i',j'} + \epsilon)$.  
\end{lemma}

\begin{proof}
The unstable manifold of $b_{i,j}$ is entirely contained in $N(e^1_\alpha)$ (by Lemma \ref{lem:N}) while the intersection of the stable manifold of  $b_{i',j'}$ with $N(e^1_\alpha)$ belongs to the subset where $x_1 \in (\beta_{i',j'}-\epsilon, \beta_{i',j'} + \epsilon)$ (by Proposition \ref{prop:bUS}).  Thus, the location of any intersection points must be as claimed.

That the unstable and stable manifolds intersect transversally follows from the $1$-regularity of $\tilde{L}$ stated in Theorem \ref{thm:PropertiesofLtilde}.  Finally, the intersections points are odd in number since the unstable manifold of $b_{i,j}$  (resp. the stable manifold of $b_{i',j'}$) intersects $T =N(e^1_\alpha) \cap \left([\beta_{i',j'}-\epsilon, \beta_{i',j'} + \epsilon] \times[-1/32,1/32]\right)$ in a single line segment with endpoints on the left and right boundary of $T$ (resp. on the upper and lower boundary of $T$) by Property \ref{pr:1cmono} (resp. by Proposition \ref{prop:bUS} and Property \ref{pr:1cells}).    

\end{proof}

\subsubsection{Computation of $\partial b_{i,j}$}

%Let $a^-_{i,j}, a^+_{i,j}, b_{i,j}, \tilde{b}_{i,j}$ denote Reeb chords in $N(e^0_-)$, $N(e^0_+)$, and $N(e^{1}_\alpha)$ respectively, where in all cases the indices correspond to labeling of sheets as $S_1, \ldots, S_n$ with decreasing $z$-coordinate as they appear above $e^+_0$.  We make the convention that if $F_i < F_j$ at $e^0_-$ or either $F_i$ or $F_j$ is not defined at $e^0_-$, then $a^-_{i,j} =0$ with the exception that if $S_k$ and $S_{k+1}$ meet at a cusp edge above $e^1_\alpha$, then $a^-_{k,k+1}=1$.

%\dr{Could be better organization to move this statement to appear before Property 6.}

%\begin{proposition}  \label{prop:LCH1comp}
%Let $e^1_\alpha$ be a $1$-cell with  $a^-_{i,j}, a^+_{i,j}, b_{i,j}, \tilde{b}_{i,j}$ as above.  For any $i<j$, in $(\lchA(e^1_\alpha), \partial)$ we have
%\[
%\partial b_{i,j} = a^+_{i,j} + a^-_{i,j}+\sum_{i<m<j} ( a_{i,m}^+ b_{m,j} + b_{i,m} a^-_{m,j}) + x
%\]
%where $x$ belongs to the ideal generated by the $\tilde{b}_{i,j}$.
%\end{proposition}

\begin{proof}[Proof of Proposition \ref{prop:LCH1comp}]
In this proof, we make the convention that if $F_i < F_j$ at $e^0_-$ or either $F_i$ or $F_j$ is not defined at $e^0_-$, then $a^-_{i,j} =0$ with the exception that if $S_k$ and $S_{k+1}$ meet at a cusp edge above $e^1_\alpha$, then $a^-_{k,k+1}=1$.
Then, the stated formula for $\partial B$ is equivalent to the term-by-term formula, 
\[
\partial b_{i,j} = a^+_{i,j} + a^-_{i,j}+\sum_{i<m<j} ( a_{i,m}^+ b_{m,j} + b_{i,m} a^-_{m,j}) + x
\]
where $x$ belongs to the ideal generated by the $\tilde{b}_{l,l'}$.

We first identify an odd number of rigid GFTs contributing each of the terms, other than $x$, that appear in the formula for $\partial b_{i,j}$.  Then, to complete the proof we show that any other rigid GFT would have to have an endpoint at one of the $\tilde{b}_{i,j}$.

First, consider rigid flow trees beginning at $b_{i,j}$ that do not have any internal vertices, i.e. flow lines of $-\nabla F_{i,j}$ that start at $b_{i,j}$ and limit to a Reeb chord as $t\rightarrow +\infty$ or terminate at an $e$-vertex.  There are two non-constant $(i,j)$-flow lines beginning at $b_{i,j}$, and, from Proposition \ref{prop:bUS}, one of these two flow lines must limit to $a^+_{i,j}$, while the other either limits to $a^-_{i,j}$; ends at an $e$-vertex; or terminates at the cusp or crossing locus.  Notice that with our conventions for the meaning of $a^-_{i,j}$ these rigid flow trees contribute exactly the term $a^+_{i,j}+ a^-_{i,j}$ to $\partial b_{i,j}$, and there are no other rigid flow trees without internal vertices.

We now specify some rigid flow trees with exactly $1$ internal vertex.  For any $i<m<j$, consider a flow tree $\Gamma$ whose initial branch is a portion of the $-\nabla F_{i,j}$  flowline beginning at $b_{i,j}$ that  intersects the stable manifold of $b_{i,m}$  (resp. $b_{m,j}$), as specified by Lemma   \ref{lem:IntersectB}.  Next, a $Y_0$ occurs in $\Gamma$ at one of the odd number of intersection points;  the outgoing edges are $(i,m)$-flows and $(m,j)$-flows.  If the $Y_0$ point is on the stable manifold of $b_{i,m}$, then the outgoing branch that is an $(i,m)$-flow line limits to $b_{i,m}$, while other outgoing branch is an $(m,j)$-flow line that has initial point $(x_1,x_2)$ with 
\[
x_1 \leq \beta_{i,m} +\epsilon < \beta_{m,j}-\epsilon  \quad \quad \mbox{by (2) of Proposition \ref{prop:bUS} and equation (\ref{eq:StaircaseAxiom}).}
\]
Property \ref{pr:1cmono} applies to show that this $(m,j)$-flow must remain to the left of $x_1 = \beta_{m,j}-\epsilon$ in $N(e^1_\alpha)$, and will therefore either limit to $a^-_{m,j}$; reach an $e$-vertex at the cusp locus; or terminate at a crossing arc or cusp edge.  (In the final case, $\Gamma$ is not actually a flow tree.)  As we use $\Z/2$ coefficients, trees of this type contribute precisely the term $b_{i,m} a^-_{m,j}$ to $\partial b_{i,j}$, since we have set $a^-_{m,j} = 0$  (resp. $1$) exactly when the $(m,j)$-flow terminates at the cusp or crossing edge (resp. ends at an $e$-vertex).  In a similar manner, when the $Y_0$ point belongs to the stable manifold of $b_{m,j}$, the outgoing edges limit to $a^+_{i,m}$ and $b_{m,j}$ respectively, and these trees give the $a^+_{i,m} b_{m,j}$ terms.
See Figure \ref{fig:1Trees}.

\begin{figure}

\labellist
\small
\pinlabel $b_{i,j}$ [b] at 35 104
\pinlabel $b_{i,m}$ [t] at 2 0
\pinlabel $a^-_{m,j}$ [t] at 66 2
\pinlabel $a^-_{m,j}$ [t] at 222 10
\pinlabel $b_{i,m}$ [t] at 409 22
\pinlabel $b_{i,j}$ [t] at 460 22
\pinlabel $b_{m,j}$ [t] at 521 22

%\pinlabel $\cdots$  at 305 201
%\pinlabel $\tilde{b}_{i,j}$ [b] at 177 77
%\pinlabel $(\tilde{\beta}_{i,j}-\epsilon,\tilde{\beta}_{i,j}+\epsilon)$ [t] at 177 21
%\pinlabel $1/4$ [t] at 338 21
%\pinlabel $-1/4$ [t] at 273 21
%\pinlabel $\cdots$  at 305 49
\endlabellist
\centerline{ \includegraphics[scale=.6]{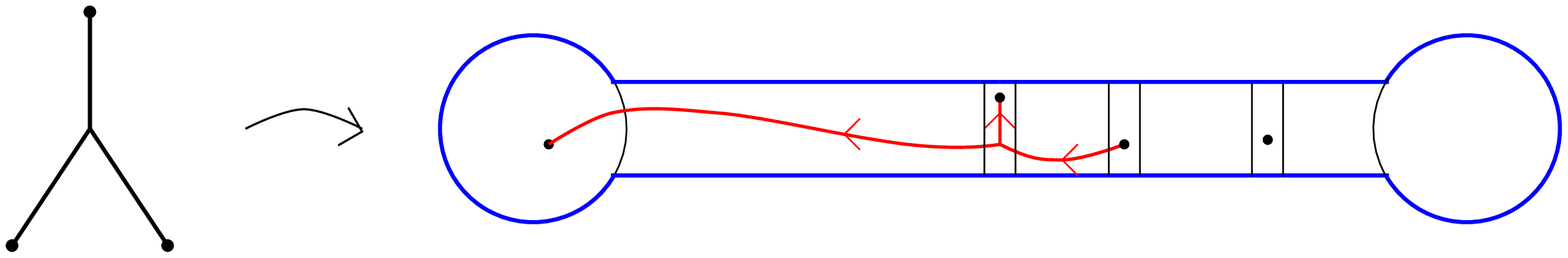} }

%%\centerline{ \includegraphics[scale=.8]{images/SubdivideSq} } %Dan had already commented out this line

\quad

\quad

\quad

\labellist
\small
\pinlabel $b_{i,j}$ [b] at 35 104
\pinlabel $a^+_{i,m}$ [t] at 2 2
\pinlabel $b_{m,j}$ [t] at 66 0
\pinlabel $a^+_{i,m}$ [t] at 608 10
\pinlabel $b_{i,m}$ [t] at 409 22
\pinlabel $b_{i,j}$ [t] at 460 22
\pinlabel $b_{m,j}$ [t] at 521 22

%\pinlabel $\cdots$  at 305 201
%\pinlabel $\tilde{b}_{i,j}$ [b] at 177 77
%\pinlabel $(\tilde{\beta}_{i,j}-\epsilon,\tilde{\beta}_{i,j}+\epsilon)$ [t] at 177 21
%\pinlabel $1/4$ [t] at 338 21
%\pinlabel $-1/4$ [t] at 273 21
%\pinlabel $\cdots$  at 305 49
\endlabellist
\centerline{ \includegraphics[scale=.6]{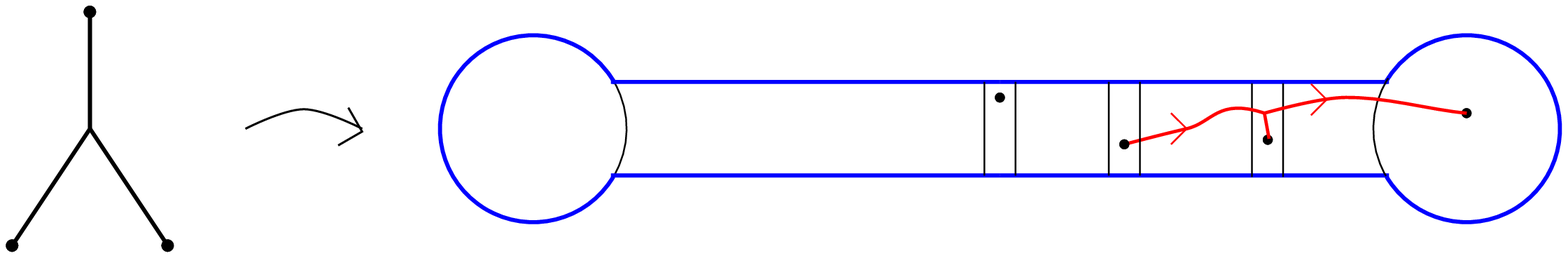} }

\caption{The domain and image of rigid flow trees beginning at $b_{i,j}$ and having $1$ $Y_0$-vertex.}
\label{fig:1Trees}
\end{figure}

To complete the proof, we need to show that any rigid GFT $\Gamma$ begining at some $b_{i,j}$ with at least $1$ internal vertex, and without endpoints at any $\tilde{b}_{i,j}$ must be one of the trees from the previous paragraph.  Such a flow tree can only have endpoints at Reeb chords of the form $a^\pm_{l,l'}$, $b_{l,l'}$ or $e$-vertices, and cannot have any $Y_1$ or $sw$-vertices because Properties  \ref{pr:14models} and \ref{pr:1cmono} show that all $-\nabla F_{i,j}$ with $F_i>F_j$ point transversally to the cusp edge in the direction where the cusp sheets are not defined.  Therefore, Proposition \ref{prop:EY1SWFormula} gives 
\[
1 = A +E 
\] 
where $A$ and $E$ are respectively the number of endpoints of $\Gamma$ at $a$ Reeb chords and at $e$-vertices.  Therefore, $\Gamma$ has all but one endpoint at $b$ Reeb chords, with the remaining endpoint at an $a$ Reeb chord or an $e$-vertex.  

Consider the edge that has its endpoint at an $a$ Reeb chord or an $e$-vertex.  The initial vertex of this edge, $x$, must be a $Y_0$.

\medskip

\noindent{\bf Step 1.}  Suppose that the incoming edge at $x$ is an $(i,j)$-flow with the outgoing edges $(i,m)$ and $(m,j)$-flows.  If the $(i,m)$ edge (resp. the $(m,j)$ edge) ends at an $a$ Reeb chord or $e$-vertex, then the $(m,j)$ edge (resp. the $(i,m)$ edge) ends at $b_{m,j}$ (resp. at $b_{i,m}$).  See Figure \ref{fig:BStep1}.

\begin{figure}

\quad

\quad 

\labellist
\small
%\pinlabel $1/4$ [t] at 338 21
\pinlabel $\vdots$ [b] at 35 231
\pinlabel $\vdots$ [b] at 163 231
\pinlabel $\vdots$ [b] at 339 231
\pinlabel $\vdots$ [b] at 467 231
\pinlabel $\vdots$ [b] at 35 70
\pinlabel $\vdots$ [b] at 339 70
\pinlabel $\vdots$ [t] at 320 188
\pinlabel $\vdots$ [t] at 53 188
\pinlabel $a_{i,m}$ [t] at 2 155
\pinlabel $a_{i,m}$ [t] at 130 155
\pinlabel $b_{m,j}$ [t] at 194 155
\pinlabel $a_{m,j}$ [t] at 371 155
\pinlabel $b_{i,m}$ [t] at 434 155
\pinlabel $a_{m,j}$ [t] at 498 155
\pinlabel $a_{i,m}$ [t] at 2 155
\pinlabel $a_{i,m}$ [t] at 130 0
\pinlabel $b_{m,j}$ [t] at 194 0
\pinlabel $a_{m,j}$ [t] at 371 0
\pinlabel $b_{i,m}$ [t] at 434 0
\pinlabel $a_{m,j}$ [t] at 498 0
\pinlabel $a_{i,m}$ [t] at 2 0
\pinlabel $b_{m,j}$ [t] at 66 0
\pinlabel $b_{i,m}$ [t] at 306 0
\pinlabel $b_{i,j}$ [b] at 162 104
\pinlabel $b_{i,j}$ [b] at 466 104
\pinlabel Step~2: [r] at -10 34
\pinlabel Step~1: [r] at -10 194

\endlabellist
\centerline{ \includegraphics[scale=.6]{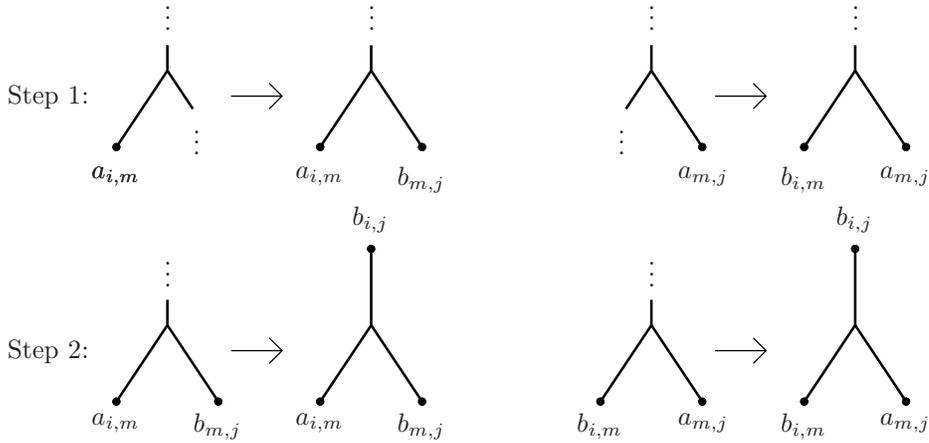} }

\quad 

%%\centerline{ \includegraphics[scale=.8]{images/SubdivideSq} } %Dan had already commented out this line

\caption{Steps 1 and 2 from the proof of Proposition \ref{prop:LCH1comp}.}
\label{fig:BStep1}
\end{figure}

\medskip 

The partial flow tree that starts with the edge in question has all of its endpoints at $b$ Reeb chords.  Therefore, Step 1. follows from 
\begin{lemma} \label{lem:AllBs1Sk}
Any PFT that begins in $N(e^1_\alpha)$ and has all endpoints at Reeb chords of the form $b_{i',j'}$ 
 has no internal vertices.
\end{lemma}
\begin{proof}[Proof of Lemma \ref{lem:AllBs1Sk}]
Suppose that such a PFT has some internal vertex.  Then, by the combinatorics of rooted trees, we can find a pair of edges that end at Reeb chords $b_{i_1,j_1}$ and $b_{i_2,j_2}$ respectively and start at a common $Y_0$-vertex.   The image of this $Y_0$-vertex would be an intersection in $N(e^1_{\alpha})$ of the stable manifolds of $b_{i_1,j_1}$ and $b_{i_2,j_2}$.  However, no such intersection point can exist by (2) of Proposition \ref{prop:bUS} since the intervals $[\beta_{i_1,j_1}-\epsilon,\beta_{i_1,j_1}+\epsilon]$ and  $[\beta_{i_2,j_2}-\epsilon, \beta_{i_2,j_2}+\epsilon]$ are disjoint.
\end{proof}

As a consequence of Step 1., note that since $x$ belongs to the stable manifold of $b_{i,m}$ or $b_{m,j}$ it is to the right of any crossings in the case $e^1_\alpha$ is a (1Cr) or (2Cr) edge.  Thus, the sheets $S_1, \ldots, S_n$ appear with $z$-coordinates strictly descending at $x$, so it follows that $i<m<j$.

\medskip

\noindent{\bf Step 2.}  The incoming edge at $x$ must begin at $b_{i,j}$.  See Figure \ref{fig:BStep1}.

\medskip

To establish Step 2., we need to show that the $(i,j)$ edge, $E$, that ends at $x$ cannot begin at a $Y_0$ vertex.  From Step 1., we know that $x$ is located on either the stable manifold of $b_{i,m}$ or $b_{m,j}$.  Therefore by Properties \ref{pr:Location1} and \ref{pr:1cmono} and Proposition \ref{prop:bUS}, any point in the image of $E$ must have $x_1$-coordinates in $[b_{i,m}-\epsilon, b_{m,j}+\epsilon]$.  If the initial endpoint of $E$ were at a $Y_0$, $y$, then the other outgoing edge of $y$, $F$, would either be a $(j,l)$-flow or an $(h,i)$-flow for some $h < i <j< l$.  By Lemma \ref{lem:AllBs1Sk}, $F$ would have to belong to the stable manifold of either $b_{j,l}$ or $b_{h,i}$, but, by Proposition \ref{prop:bUS} and Property \ref{pr:Location1}, this is impossible since $\beta_{h,i}+\epsilon < \beta_{i,m}-\epsilon$ and $\beta_{m,j}+\epsilon < \beta_{j,l} - \epsilon$.

With Step 1. and Step 2. established, we see that the $Y_0$ vertex $x$ is at one of the odd number of  intersection of the unstable manifold of $b_{i,j}$ with either the stable manifold of $b_{i,m}$ or $b_{m,j}$.  There are no other internal vertices in the tree, so $\Gamma$ is indeed one of the previously identified GFTs.
\end{proof}

\section{Computation of LCH, Part 2: $2$-cells without swallowtail singularities}
\label{sec:Comp2Cells}

In this section, we compute in Theorem \ref{thm:SquareComp}
 the sub-DGAs associated to $2$-cells of type (1)-(12).

\subsection{Reeb chords above a $2$-cell:  Types (1)-(12)}

Let $e^2_\alpha$ be a $2$-cell of type (1)-(12).  The closure of $e^2_\alpha$ is parametrized by $[-1,1]^2$, and we use $e^0_{+,+}, e^0_{-,+}, e^0_{-,-}, e^0_{+,-}$ to denote the $0$-cells that appear above the corners of $e^2_\alpha$ at $(x_1,x_2) = (1,1), (-1,1), (-1,-1), (1,-1)$ respectively.  We denote the $1$-cells that form the edges $\{-1\}\times [-1,1]$, \, $[-1,1] \times \{-1\}$, \, $\{+1\}\times [-1,1]$, \, $[-1,1] \times \{+1\}$ by $e^1_L$, $e^1_{D}$, $e^1_{R}$, $e^1_U$ respectively.  (Here and elsewhere, the subscripts $L, D, R, U$ are intended to convey ``Left'', ``Down'', ``Right'', and ``Up''.)
  We recall that, since the transverse square decomposition   $\cE_\tra$ satisfies the requirement (A1) from Section \ref{ssec:TransverseSqRe}, the orientation of both $e^1_L$ and $e^1_R$ (resp. $e^1_D$ and $e^1_U$) is from the bottom edge to the top edge (resp. the left edge to the right edge) of $[-1,1]^{2}$. 

We define a subset $\widehat{N}(e^2_\alpha) \subset N(e^2_\alpha)$ to satisfy
\[
N(e^2_\alpha) = N(e^1_L) \cup N(e^1_{D})  \cup N(e^1_{R}) \cup N(e^1_U) \cup \widehat{N}(e^2_\alpha)
\]
so that $\widehat{N}(e^2_\alpha)$ is compact and only intersects neighborhoods of $1$-cells along subsets of their boundaries.  In coordinates, $\widehat{N}(e^2_\alpha)$ is a subset of the interior of $[-1,1]^2$.  Note in addition that the earlier parametrizations of the subsets $\widehat{N}(e^1_X)$ for $X\in \{L,D,R,U\}$ by subsets of $[-1,1] \times[-1/32,1/32]$ can be suitably modified and pieced together with the parametrization of $e^2_\alpha$ to realize  $N(e^2_\alpha)$ with the interiors of the $N(e^0_{\pm,\pm})$ removed as a subset of $[-1-1/32,1+1/32]\times [-1-1/32,1+1/32]$.  Denote this subset of $N(e^2_\alpha)$ by $\widetilde{N}(e^2_\alpha)$.  See Figure \ref{fig:Neighbor1}.

\begin{figure}

\labellist
\small
\pinlabel $N(e^2_\alpha)$ [t] at 80 0
%\pinlabel $(\beta_{i,j}-\epsilon,\beta_{i,j}+\epsilon)$ [t] at 433 173
\pinlabel $\widehat{N}(e^2_\alpha)$ [t] at 320 0
\pinlabel $\widetilde{N}(e^2_\alpha)$ [t] at 560 0
%\pinlabel $\cdots$  at 305 201
%\pinlabel $\tilde{b}_{i,j}$ [b] at 177 77
%\pinlabel $(\tilde{\beta}_{i,j}-\epsilon,\tilde{\beta}_{i,j}+\epsilon)$ [t] at 177 21
%\pinlabel $1/4$ [t] at 338 21
%\pinlabel $-1/4$ [t] at 273 21
%\pinlabel $\cdots$  at 305 49
\endlabellist
\centerline{ \includegraphics[scale=.6]{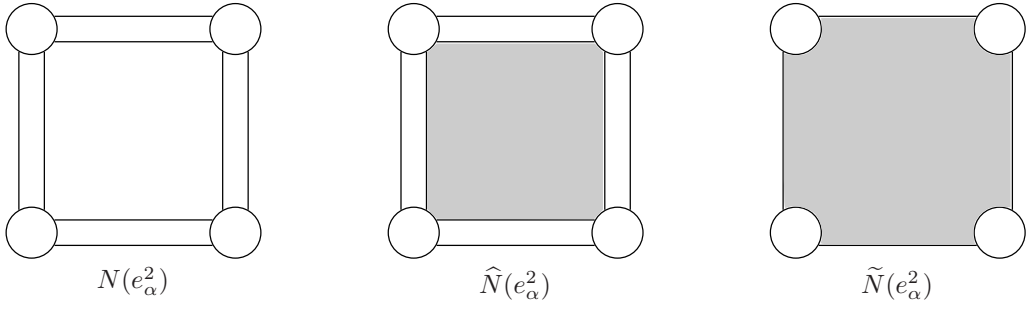} }

%%\centerline{ \includegraphics[scale=.8]{images/SubdivideSq} } %Dan had already commented out this line

\quad

\quad

\caption{The neighborhood $N(e^2_\alpha)$ (left) with the subsets $\widehat{N}(e^2_\alpha)$ (center) and $\widetilde{N}(e^2_\alpha)$ (right) shaded. }
\label{fig:Neighbor1}
\end{figure}

\subsubsection{Reeb chords above $N(e^2_\alpha)$}

Above $e^2_\alpha$, the components of the complement of the cusp edge of $\tilde{L}$ each project in a one-to-one manner to $[-1,1]$.  Thus, within $e^2_\alpha$ we can label the sheets of $\tilde{L}$ as $S_1, \ldots, S_n$ so that above $(x_1,x_2) =  (+1,+1)$ their $z$-coordinates appear in decreasing order.  For $X \in \{L,D,R,U\}$, we have Reeb chords in neighborhoods of $N(e^0_{\pm,\pm})$ and $N(e^1_X)$ as specified by Properties \ref{pr:Reeb0} and \ref{pr:Reeb1}, which we denote by $a^{\pm,\pm}_{i,j}$,  $b^X_{i,j}$, and $\tilde{b}^X_{i,j}$.  We follow Convention \ref{conv:subscript} so that the subscript $i,j$  indicates a Reeb chord whose upper (resp. lower) endpoint is on $S_i$ (resp. $S_j$).  The location of $b^X_{i,j}$ is determined by Property \ref{pr:Location1} to lie within the subset of $N(e^1_X)$ that is specified in our current coordinate system by a requirement of the form
\[
x_1(b^X_{i,j}) \in (\beta_{i,j}^X-\e, \beta_{i,j}^X+\e)  \mbox{   when $X \in \{D,U\}$, or } \quad x_2(b^X_{i,j}) \in  (\beta_{i,j}^X-\e, \beta_{i,j}^X+\e)  \mbox{   when $X \in \{L,R\}$,}
\]
for appropriate $\beta^X_{i,j} \in [1/2,3/4]$.  
Similarly, we have $\tilde{\beta}^X_{i,j} \in [-3/4,-1/2]$ that specify the location of the $\tilde{b}^X_{i,j}$ should they exist.

\begin{remark}  \label{rem:betanotate}

\begin{itemize}
\item[(i)] In general, $\beta^U_{i,j} = \beta^R_{i,j} = \beta_{i,j}$.  (Recall, that the $\beta_{i,j}$ specify the position of Reeb chords in neighborhoods of $1$-cells and depend only on $n$ and not the type of the $1$-cell.)   However, we caution that, 
\[
\beta^L_{i,j} = \beta_{\sigma_L(i), \sigma_L(j)} \mbox{  and  } \beta^D_{i,j} = \beta_{\sigma_D(i), \sigma_D(j)}
\] where $\sigma_L(i)$ and $\sigma_D(i)$ indicate the position (with respect to descending $z$-coordinate) that sheet $S_{i}$ appears above $(x_1,x_2) = (-1,+1)$ and $(x_1,x_2) = (+1,-1)$  respectively.  Thus, if a crossing or cusp appears above $U$ or $R$, then for some values of $i,j$ we may have $\beta^L_{i,j} \neq \beta_{i,j}$, or $\beta^D_{i,j} \neq \beta_{i,j}$.

\item[(ii)] In parallel to (i), note that the indices provided on Reeb chords of the form $b^L_{i,j}$ or $b^D_{i,j}$ may be distinct from the indices that would be used to specify the same Reeb chord when considering $e^1_L$ or $e^1_D$ in the setting of Section \ref{ssec:Compute1cell}.  This is because if a crossing or cusp appears along edges $U$ or $R$, then the sheets $S_1, \ldots, S_n$ are labeled in a different manner above $e^1_L$ and $e^1_D$ than in Section \ref{ssec:Compute1cell}.  For instance, if sheets $k$ and $k+1$ cross above $e^1_U$, then there is a Reeb chord $b^L_{k+1,k}$ with $x_2(b^L_{k+1,k}) \in [1/2,3/4]$, while in Section \ref{ssec:Compute1cell}  all Reeb chords above the upper half of $N(e^1_L)$ would be labeled with subscripts $i,j$ with $i<j$.  The use of different labellings of Reeb chords in the context of the sub-DGAs $\mathcal{A}(e^d_\alpha)$ for different $e^d_\alpha$ is addressed in Section \ref{sec:Iso}.

%We caution that as a consequence, the $\beta^L_{i,j}$ and $\beta^D_{i,j}$ may not satisfy the lexicographic ordering of equation (\ref{eq:StaircaseAxiom}).

\end{itemize}

\end{remark} 

\begin{property} \label{pr:Reeb2} The only Reeb chords in $\widehat{N}(e^2_\alpha)$ are as follows.
\begin{itemize}  
\item For any $i<j$, there is a unique Reeb chord $c_{i,j}$ within $[1/2,3/4]^2$.  The location of $c_{i,j}$ is more precisely specified by
\[
c_{i,j} \in (\beta^U_{i,j}-\e, \beta^U_{i,j}+\e) \times (\beta^R_{i,j}-\e, \beta^R_{i,j}+\e).
\]
\item For each pair of sheets $S_i$ and $S_j$ with $i<j$ that cross one another somewhere in $[-1,1]^2$, there is a unique Reeb chord with upper sheet $S_j$ and lower sheet $S_i$ that we denote $\tilde{c}_{j,i}$.
\end{itemize}  
Moreover, all of the $c_{i,j}$ and $\tilde{c}_{j,i}$ are non-degenerate local maxima.
\end{property}

See Figure \ref{fig:Stairc}.   
%STAIRCASE ORDERING OF C's ALONG MAIN DIAGONAL  Caption: ... The critical points, $c_{i,j}$, that occur in $[1/4,1]\times[1/4,1]$ have both coordinates increasing with $i$ and, for a fixed $i$, increasing with $j$.  

\begin{figure}

\quad

\quad

\labellist
\small
\pinlabel $c_{1,2}$ [br] at 34 46
%\pinlabel $(\beta_{i,j}-\epsilon,\beta_{i,j}+\epsilon)$ [t] at 433 173
\pinlabel $c_{1,3}$ [br] at 68 74
\pinlabel $c_{1,4}$ [br] at 86 122
\pinlabel $c_{2,3}$ [br] at 144 150
\pinlabel $c_{2,4}$ [br] at 174 182
\pinlabel $c_{3,4}$ [br] at 210 216
\pinlabel $\beta^R_{1,4}-\epsilon$ [l] at 248 122
\pinlabel $\beta^R_{1,4}+\epsilon$ [l] at 248 86
\pinlabel $\beta^U_{1,4}-\epsilon$ [br] at 90 246
\pinlabel $\beta^U_{1,4}+\epsilon$ [bl] at 120 246
\endlabellist
\centerline{ \includegraphics[scale=.6]{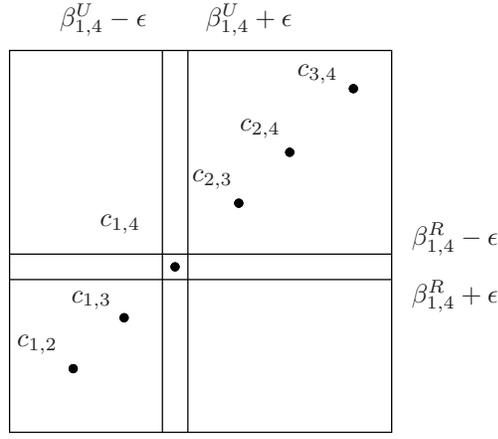} }

%%\centerline{ \includegraphics[scale=.8]{images/SubdivideSq} } %Dan had already commented out this line

\caption{The critical points, $c_{i,j}$, that occur in $[1/2,3/4]\times[1/2,3/4]$ have both coordinates increasing with $i$ and, for a fixed $i$, increasing with $j$. }
\label{fig:Stairc}
\end{figure}

To specify the location of $\tilde{c}_{j,i}$ we make the following observation.  By considering the squares of type (1)-(12) in a case-by-case manner, we see that if sheets $S_i$ and $S_j$ cross somewhere in $[-1,1]^2$,  then they either cross above exactly one of $e^1_U$ or $e^1_L$, or, as happens only in the case of the Type (3) square, they cross above both $e^1_D$ and $e^1_R$.  

\begin{property} \label{pr:Location2} Let $i<j$.
\begin{itemize} 
\item If sheets $S_i$ and $S_j$ cross above $e^1_R$, then
\[
\tilde{c}_{j,i} \in (\beta^D_{j,i}-\e, \beta^D_{j,i}+\e) \times (\tilde{\beta}^R_{j,i}-\e, \tilde{\beta}^R_{j,i}+\e).
\]
\item If sheets $S_i$ and $S_j$ cross above $e^1_U$, then
\[
\tilde{c}_{j,i} \in (\tilde{\beta}^U_{j,i}-\e, \tilde{\beta}^U_{j,i}+\e) \times (\beta^L_{j,i}-\e, \beta^L_{j,i}+\e).
\]
\item If sheets $S_i$ and $S_j$ cross above both $e^1_L$ and $e^1_D$, then
\[
\tilde{c}_{j,i} \in (\tilde{\beta}^D_{j,i}-\e, \tilde{\beta}^D_{j,i}+\e) \times (\tilde{\beta}^L_{j,i}-\e, \tilde{\beta}^L_{j,i}+\e).
\]
\end{itemize}
\end{property}

\subsubsection{Exceptional generators}  
The equivalence between the cellular DGA and $(\lchA, \partial)$ of Theorem \ref{thm:CellularLCH} will actually be realized by comparing the cellular DGA with a certain stable tame isomorphic quotient of $(\lchA, \partial)$.  In this quotient, we will remove the following generators from the generating set.   

\begin{definition} \label{def:ex-gen}
We say that a Reeb chord is an {\bf exceptional generator} in $N(e^2_\alpha)$ if either of the following conditions applies:
\begin{enumerate}
\item  The generator has the form $\tilde{b}_{i,j}$ or $\tilde{c}_{i,j}$, i.e. the Reeb chord sits in the half of an edge where $-3/4 <x_i < -1/2$ or in the interior of the square but outside of the region $[1/2,3/4]\times [1/2,3/4]$.  OR
\item  When the singular set is pushed into the boundary of $e^2_\alpha$ as in Section \ref{sec:parallel}, a crossing arc between the sheets that form the end points of the Reeb chord sits above the $0$-cell or $1$-cell that the chord corresponds to. 
\end{enumerate}
\end{definition}

\begin{remark}
In particular, according to (2) of Definition \ref{def:ex-gen}, any generator of the form $b^X_{i,j}$ or $a^{\pm,\pm}_{i,j}$ such that $i>j$ will be an exceptional generator.  However, in some cases there are also exceptional generators of the form $b^X_{i,j}$ or $a^\pm_{i,j}$ with $i<j$, eg. above a Type (4) square $a^{-,-}_{k,k+1}$ is an exceptional generator.  A complete square-by-square list of exceptional generators appears in Section \ref{ssec:DiffExceptional}.  
\end{remark}

\subsection{Computation of $(\lchA(e^2_\alpha), \partial):$ the statement}

For a given $2$-cell, $e^2_\alpha$, of type (1)-(12) we define strictly upper triangular matrices $A_{+,+},$ $A_{-,-},$ $B_X$ for $X \in \{L, D, R, U\}$, and $C$.  All $(i,j)$-entries with $i\geq j$ are $0$.  For $i<j$, the $(i,j)$-entries are respectively the Reeb chords of the form $a^{+,+}_{i,j}$, $a^{-,-}_{i,j}$, $b^X_{i,j}$, and $c_{i,j}$.  If no such Reeb chord exists, the $(i,j)$-entry is $0$ with the following exceptions:  For squares of type (9)-(12), a pair of sheets $S_k$ and $S_{k+1}$ meet at a cusp before reaching the lower left side of the square.  In these cases, the matrices $B_L$ and $A_{-,-}$ have all entries in the $k$ and $k+1$ rows and columns equal to $0$ except for the $(k,k+1)$-entry of $A_{-,-}$ which is taken to be $1$.  For a Type (11) square an additional, similar adjustment is made to the two columns of $B_{D}$ and $A_{-,-}$ that correspond to the sheets that meet at a cusp before reaching the bottom edge of the square.

\begin{remark}
Reeb chords of the form $\tilde{b}^X_{i,j}$ or $\tilde{c}_{i,j}$ do not appear as entries of these matrices, nor do Reeb chords of the form $a^{\pm,\pm}_{i,j}$ or $b^X_{i,j}$ with $i>j$.  For instance, if $e^2_\alpha$ has type (2), then both the $(k,k+1)$- and $(k+1,k)$-entries are $0$ in the matrices $A^{-,-}$ and $B_L$.
\end{remark}

\begin{theorem} \label{thm:SquareComp}
For a $2$-cell $e^2_\alpha$ of type (1)-(12),  in $(\lchA(e^2_\alpha), \partial)$ we have 
\begin{equation}\label{eq:SquareComp}
\partial C = A_{+,+} C + C A_{-,-} + (I+B_R)(I+B_D) + (I+B_U)(I+B_L)+ X.
\end{equation}
where all of the entries of the matrix $X$ belong to the ideal generated by the exceptional generators in $N(e^2_\alpha)$.
\end{theorem}

The proof of Theorem \ref{thm:SquareComp} appears at the conclusion of this section after relevant properties of gradients have been stated and several preliminary results about  GFTs in $N(e^2_\alpha)$ have been established.

\subsection{Properties of gradients above $2$-cells}

The following Properties \ref{pr:monotonicityI} and \ref{pr:monotonicityII} may be viewed as extensions of the monotonicity property of the $-\nabla F_{i,j}$ above $1$-cells, i.e. Property \ref{pr:1cmono}, to the interiors of $2$-cells.

For a given $2$-cell $e^2_\alpha$ with boundary edges $e^1_L, e^1_D, e^1_R, e^1_U,$
we label the four components of
$$\left\{ (x_1, x_2) \in [-1,1]^2 \,\, \left| \,\, |x_1| \ge \frac{1}{4} \, \mbox{and} \, |x_2| \ge \frac{1}{4} \right. \right\} \setminus \left(N(e^1_L) \cup N(e^1_D) \cup N(e^1_R) \cup N(e^1_U)\right)$$
as $C_{R,U}, \, C_{L, U}, \, C_{L,D}, \, C_{R, D}$ so that for $X \in \{L,R\}$ and $Y\in \{D, U\}$ the subset $C_{X,Y}$ shares some portion of its boundary with $N(e^1_X)$ and $N(e^1_Y)$.

\begin{property}[Monotonicity in corners] \label{pr:monotonicityI}
For $X \in \{L,R\}$ and $Y\in \{D, U\}$,  
suppose that $F_i > F_j$ in $C_{X,Y}$ and that $S_i$ and $S_j$ are both defined in all of $C_{X,Y}$ (i.e., neither sheet meets a cusp edge in $C_{X,Y}$).  Write $\nabla F_{i,j} = (\grad_{x_1}F_{i,j}, \grad_{x_2}F_{i,j})$.  

\begin{itemize}
\item  Suppose there is a Reeb chord in $\widehat{N}(e^1_X)$ with upper sheet $S_i$ and lower sheet $S_j$.  If the Reeb chord has the form $b^X_{i,j}$  (resp. $\tilde{b}^X_{i,j}$),  set $\alpha^X_{i,j} =\beta^X_{i,j}$ (resp. $\alpha^X_{i,j} = \tilde{\beta}_{i,j}^X$).  Then, 
\[
-\grad_{x_2} F_{i,j}(x_1,x_2) >0  \quad \mbox{ if $x_2 \geq \alpha^X_{i,j} + \e$, and}
\] 
\[
-\grad_{x_2} F_{i,j}(x_1,x_2) <0 \quad \mbox{ if $x_2 \leq \alpha^X_{i,j} - \e$.}
\]

\item Suppose there is no Reeb chord in $\widehat{N}(e^1_X)$ with upper sheet $S_i$ and lower sheet $S_j$.  Then,
\[
-\grad_{x_2} F_{i,j}(x_1,x_2) <0. 
\]
\end{itemize}

In a similar manner, the sign of $-\grad_{x_1} F_{i,j}(x_1,x_2)$ is specified in regions of $C_{X,Y}$ that are determined by the location of Reeb chords with upper sheet $S_i$ and lower sheet $S_j$ in $\widehat{N}(e^1_Y)$.  
\end{property}

See Figure \ref{fig:MonoCRU1} for an illustration of Property \ref{pr:monotonicityI} in $C_{R,U}$, and see Figure \ref{fig:MonoExample} for an example.

\begin{figure}

\labellist
\small
%\pinlabel $1/4$ [t] at 338 21
\pinlabel $C_{R,U}$ [t] at 368 -2
\pinlabel $C_{R,U}$ [t] at 56 -2
\pinlabel $b^U_{i,j}$ [b] at 88 182
\pinlabel $b^R_{i,j}$ [l] at 486 88
\pinlabel $c_{i,j}$ [r] at 78 88
\pinlabel $c_{i,j}$ [b] at 392 100
\endlabellist
\centerline{ \includegraphics[scale=.6]{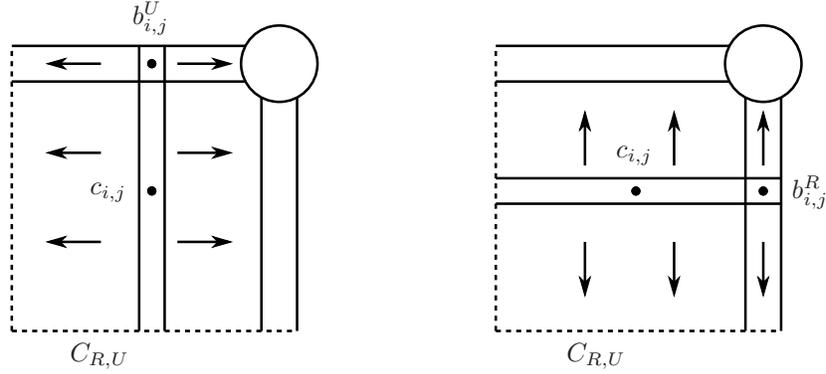} }

\quad 

%%\centerline{ \includegraphics[scale=.8]{images/SubdivideSq} } %Dan had already commented out this line

\caption{The direction of the $x_1$- and $x_2$-components of $-\nabla F_{i,j}$ in $C_{R,U}$ as determined by Property \ref{pr:monotonicityI}.}
\label{fig:MonoCRU1}
\end{figure}

\begin{remark}
The location of Reeb chords and signs of partial derivatives specified in Properties \ref{pr:Reeb2}-\ref{pr:monotonicityI} are explained by the following considerations. 
%Any of the corners $C_{X,Y}$ shares borders with a part of the neighborhoods of the $1$-cells $N(e^1_X)$ and $N(e^1_Y)$.  
Up to a $C^\infty$-small perturbation, the defining functions $F_{i}$ (and also the difference functions $F_{i,j}$) for sheets of $\tilde{L}$ above any of the corners $C_{X,Y}$ are sums 
\begin{equation} 
F_{i}(x_1,x_2) = F^Y_i(x_1) +F^X_i(x_2)
\end{equation}
where, up to additive constants, $F^Y_i$ and $F^X_i$ are the restriction of $F_i$ to the $1$-cells $e^1_Y$ and $e^1_X$.    Thus, we have the following:  
The Reeb chords in $C_{X,Y}$ are precisely at points $(x_1,x_2) \in C_{X,Y}$ such that $d_x F^Y_{i,j}(x_1)=0$ and $d_xF^X_{i,j}(x_2)=0$ where Reeb chords $b^Y_{i,j}$ and $b^X_{i,j}$ sit along the $1$-skeleton.   In $C_{X,Y}$, the horizontal line through $b^X_{i,j}$ and vertical line through $b^Y_{i,j}$ are flow lines of $-\nabla F_{i,j}$, and they divide $C_{X,Y}$ into regions where the signs of the components of $-\nabla F_{i,j}$ are uniform.  

The set up from Properties \ref{pr:Reeb2}-\ref{pr:monotonicityI} is just the portion of this picture that remains after an arbitrary suitably small perturbation.  See Section \ref{sec:Constructions} for details.

%Such $x_1$ and $x_2$ correspond precisely to the location of the $b$ Reeb chords along $e^1_Y$ and $e^1_X$, respectively.  When restricted to the $1$-skeleton they are at local maxima of the $F^X_{i,j}$ and $F^Y_{i,j}$, so (\ref{eq:remarkCxy}) explains the signs of partial derivatives of $F_{i,j}$ from Property \ref{pr:monotonicity1}.  

%The sign of the $-\partial_{x_1} F_{i,j}(x_1,x_2)$  (resp. $-\partial_{x_2} F_{i,j}(x_1,x_2)$ is as it appears above the projection of $(x_1,x_2)$ to corresponding points of $e^1_X$ (resp. $e^1_Y$).
\end{remark}

\begin{property}[Monotonicity in half spaces] \label{pr:monotonicityII}
Let $1 \leq i< j \leq n$, and write $\nabla F_{i,j} = (\grad_{x_1}F_{i,j}, \grad_{x_2}F_{i,j})$.  
  
\begin{itemize}
\item Suppose $S_i$ and $S_j$ do not cross above $e^1_U$ and at most one of them ends at a cusp edge above $e^1_U$.
Then, for $(x_1,x_2) \in \widehat{N}(e^2_\alpha)$ we have
\begin{equation} \label{eq:monoIIeq1}
-\grad_{x_2} F_{i,j}(x_1,x_2) <0,  \quad \mbox{  whenever  $x_2 =1/2$.}
\end{equation}
\item Similarly, if $S_i$ and $S_j$ do not cross above $e^1_R$ and at most one of them ends at a cusp edge above $e^1_R$.
Then, for $(x_1,x_2) \in \widehat{N}(e^2_\alpha)$ we have
\begin{equation} \label{eq:monoIIeq2}
-\grad_{x_1} F_{i,j}(x_1,x_2) <0,  \quad \mbox{  whenever  $x_1=1/2$.}
\end{equation}
\end{itemize}
\end{property}

See Figure \ref{fig:Mono2Example}.

\begin{figure}

\labellist
\small
\pinlabel $c_{i,j}$ [tr] at 332 348
\pinlabel $b^R_{i,j}$ [l] at 440 360
\pinlabel $b^U_{i,j}$ [b] at 345 439
\pinlabel $b^L_{i,j}$ [r] at 8 322
\pinlabel $b^D_{i,j}$ [t] at 345 10
\endlabellist
\centerline{ \includegraphics[scale=.6]{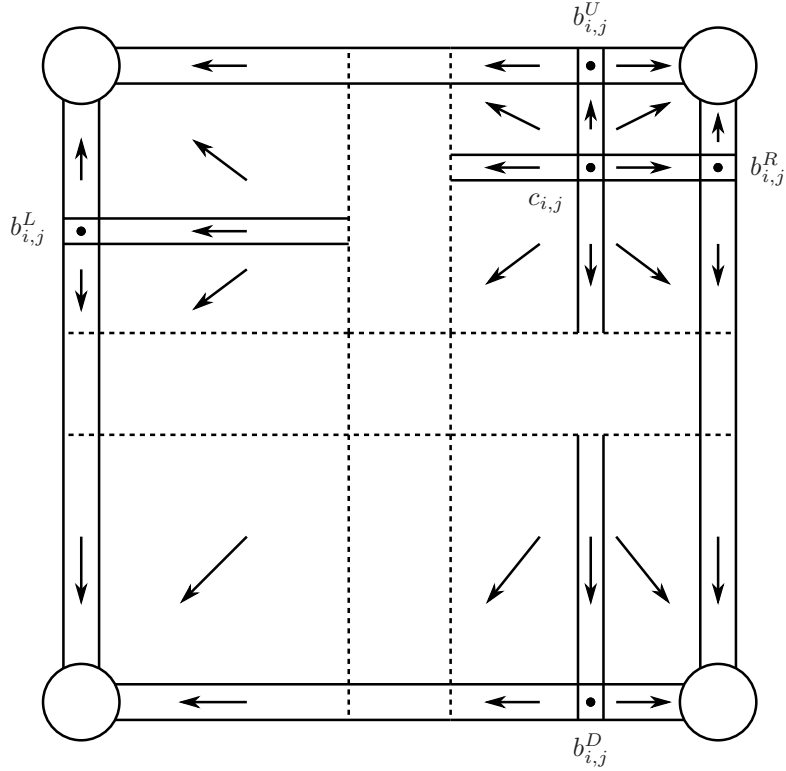} }

\quad 

\caption{ Signs of the components of $-\nabla F_{i,j}$ specified by Properties \ref{pr:1cmono} and \ref{pr:monotonicityI}.  A diagonal arrow indicates the sign of both components while a horizontal or vertical arrow indicates only one component.  No sign has been specified in regions without an arrow.  The dotted lines are $x_1= \pm 1/4$ and $x_2 =\pm 1/4$.  
In this example, $\beta^R_{i,j} < \beta^L_{i,j}$ as may happen, for instance, if the square has Type (2) with $i<k$ and $j = k+1$.}
\label{fig:MonoExample}
\end{figure}

\begin{figure}

\labellist
\small
\pinlabel $x_1=1/2$ [t] at 305 9
\pinlabel $x_2=1/2$ [r] at 9 305
\endlabellist
\centerline{ \includegraphics[scale=.4]{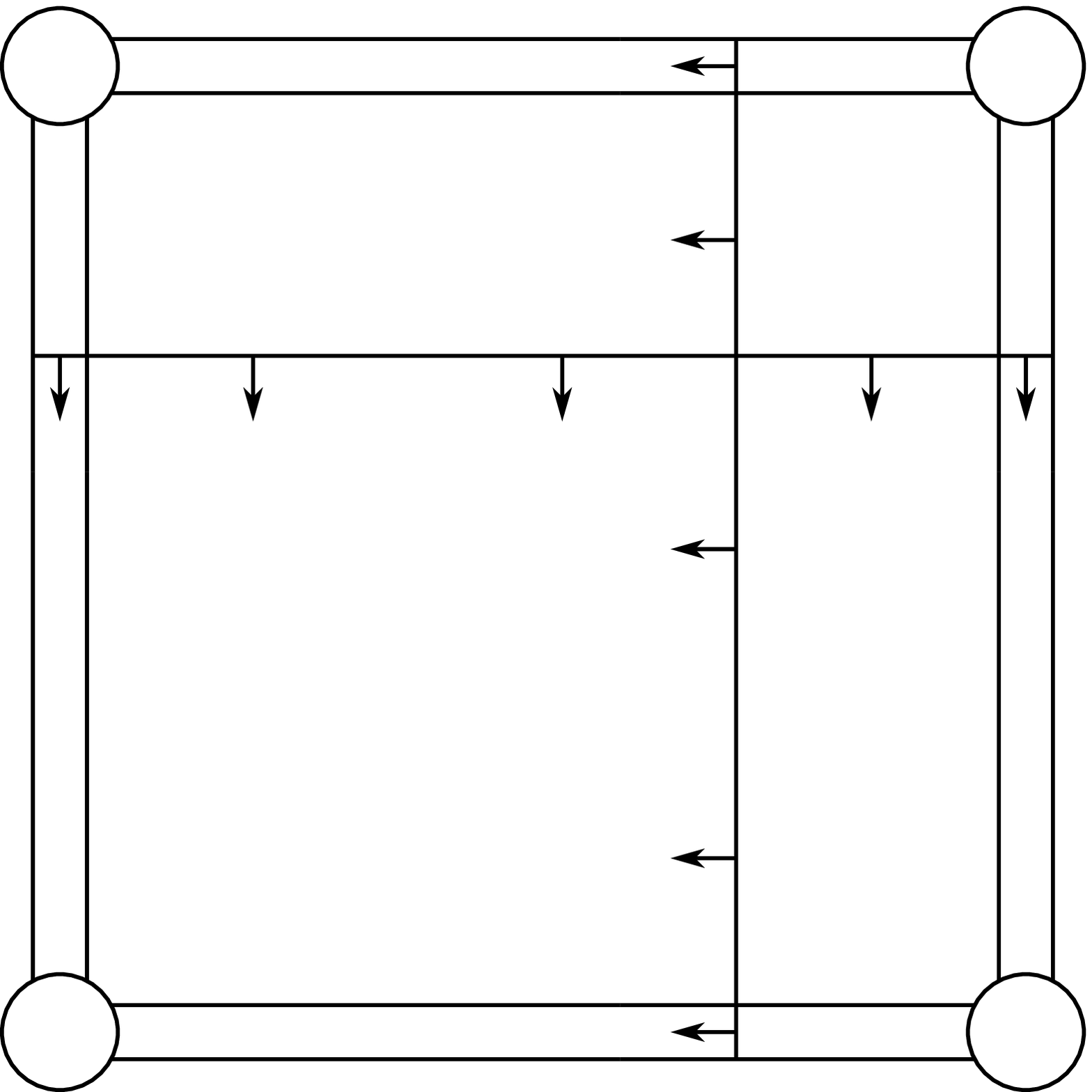} }

\quad 

\caption{ Components of $-\nabla F_{i,j}$ point down along $x_1 = 1/2$ and left along $x_2 =1/2$ (in most cases) as specified by Properties \ref{pr:monotonicityII} and \ref{pr:1cmono}.   }
\label{fig:Mono2Example}
\end{figure}

\begin{property} \label{pr:monotonicityIV}
There are line segments $B_1$ and $B_2$, such that the following hold.
\begin{itemize}
\item  The segment $B_1$ begins on $x_1 = 1/2$, ends on $x_1 = 3/4$ and is contained in the rectangle $[1/2,3/4]\times (1/4, 1/2)$.
\item  For all $i<j$, and $(x_1,x_2) \in B_1$ with $1/2 \leq x_1 \leq \beta^U_{i,j} -\e$,  $-\nabla F_{i,j}(x_1,x_2)$ is transverse to $B_1$ and points in to the region above $B_1$.
\item  The segment $B_2$ satisfies analogous conditions with all locations reflected across $x_1=x_2$.
\end{itemize}
\end{property}
See Figure \ref{fig:PABC}, below, where portions of the segments $B_1$ and $B_2$ are pictured.

\subsubsection{Transversality at cusps}
\label{sssec:TransCusp}

Recall from Section \ref{ssec:GFT} that if sheets $S_k$ and $S_{k+1}$ meet at a cusp, then a point along the $(k,k+1)$-cusp edge where $\nabla F_{i,k}$ or $\nabla F_{k+1,j}$ is tangent to the (base projection) of the cusp locus for some $S_i$ above $S_k$ or $S_j$ below $S_{k+1}$ is called an $(i,k)$- or $(k+1,j)$-switch point.  Note also that along the cusp locus $\nabla F_{i,k}=\nabla F_{i,k+1}$ and $\nabla F_{k,j}= \nabla F_{k+1,j}$.

\begin{property}[Cusp transversality] \label{pr:CuspTransversality}
For a square of type (1)-(11), suppose sheets $S_k, S_{k+1}$ form a cusp edge.  Then, at any point, $(x_1,x_2)$ where $S_i$ (resp. $S_j$) is above (resp. below) the 
cusp edge 
the vector fields  $-\nabla F_{i,k+1}$, $-\nabla F_{k, j}$, and $-\nabla F_{i,j}$ are transverse to the cusp locus and point towards the region with fewer sheets.  

In each Type (12) square, the above condition holds with the following exception:  With sheets labeled $S_1, \ldots, S_n$ as they appear above $(+1,+1)$ such that $S_k$ and $S_{k+1}$ meet at the cusp edge, there is a unique non-degenerate $(k+1,k+2)$-switch point, $P$, located with $x_2$-coordinate above the point $Q$  where  the crossing locus and the cusp locus intersect and below $x_2 = 1/2$.  Between $P$ and $Q$, $-\nabla F_{k+1,k+2}$ is transverse to the cusp locus and points towards the region with more sheets.
\end{property}

See Figure \ref{fig:CuspTrans12}.

\begin{figure}

\labellist
\small
%\pinlabel $1/4$ [t] at 338 21
\pinlabel $P$ [bl] at 64 134
\pinlabel $Q$ [r] at 50 65
\pinlabel $k+1,k+2$ [l] at 198 98
\pinlabel $k,k+2$ [l] at 198 34
\endlabellist
\centerline{ \includegraphics[scale=.6]{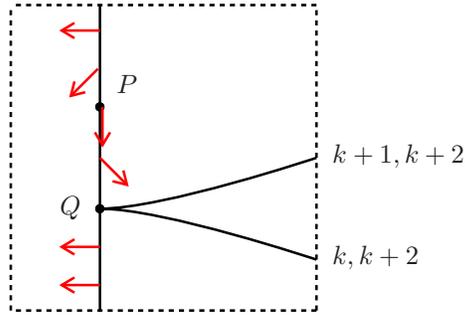} }

\caption{The unique $(k+1,k+2)$-switch point at $P$.    The vertical line denotes the $(k,k+1)$-cusp locus.  The upper (resp. lower) branch of the horizontal curve is the $(k+1,k+2)$-crossing locus (resp. the $(k,k+2)$-crossing locus). Above $Q$ the arrows denote $-\nabla F_{k+1,k+2}$, while below $Q$ the arrows denote $-\nabla F_{k+2,k+1}$. }
\label{fig:CuspTrans12}
\end{figure}

\subsubsection{Stable manifolds of the $b^X_{i,j}$}
\label{sssec:blines}

For $X \in \{L,D, R, U\}$, any Reeb chord $b^X_{i,j}$ or $\tilde{b}^X_{i,j}$ is a saddle point of $F_{i,j}$, and has  stable manifold, with respect to $-\nabla F_{i,j}$, consisting of the point itself and two non-constant flow lines that limit to the  point as $t \rightarrow +\infty$.   Moreover, Proposition \ref{prop:bUS} shows that precisely one of these flow lines  
is contained entirely in $N(e^2_\alpha)$.  (The other flow line is contained in the other $2$-cell that borders $e^1_X$.)    We refer to this non-constant flow line as the {\bf $b^X_{i,j}$-line} or the {\bf $\tilde{b}^X_{i,j}$-line} in $N(e^2_\alpha)$.

\begin{figure}

\labellist
\small
%\pinlabel $1/4$ [t] at 338 21
\pinlabel $b^L_{i,j}$ [r] at 10 64
\pinlabel $c_{i,j}$ [l] at 358 104
\endlabellist
\centerline{ \includegraphics[scale=.6]{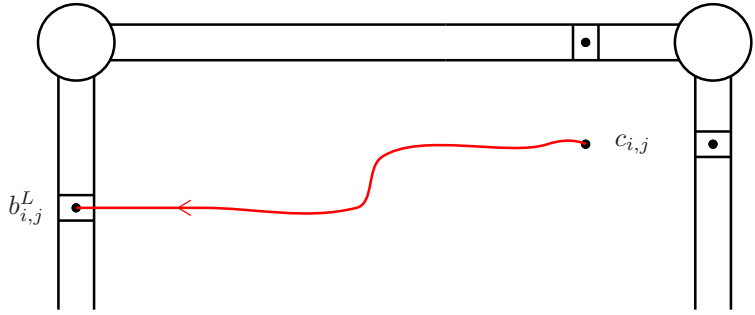} }

\caption{ The $b^L_{i,j}$-line.  }
\label{fig:Bfolds3}
\end{figure}

\begin{lemma}[Limits of  $b$-lines] \label{lem:blimits}  For squares of type (1)-(12),
as $t \rightarrow -\infty,$ the $b^X_{i,j}$-line limits to $c_{i,j}$ (resp. $\tilde{c}_{i,j}$) if $i<j$ (resp. if $i>j$),  while (if it exists)   
 the $\tilde{b}^X_{i,j}$-line  limits to $c_{i,j}$ (resp. $\tilde{c}_{i,j}$) if $i<j$ (resp. if $i>j$).
\end{lemma}

See Figure \ref{fig:Bfolds3}.
 
\begin{proof}
Follow the flow line, $\gamma$, as $t$ decreases towards $-\infty$.  Once $\gamma$ enters $\widehat{N}(e^2_\alpha)$ it cannot leave (by Property \ref{pr:1cells}).  Moreover, as $t\rightarrow -\infty$, $\gamma$ must either limit to a critical point of $F_{i,j}$, or terminate at a cusp edge.  (Since $F_{i,j}$ increases as $t$ decreases, the flow line cannot terminate at a crossing.)  The only critical point of $F_{i,j}$ in $\widehat{N}(e^2_\alpha)$ with $F_{i,j} >0$ is $c_{i,j}$ (resp. $\tilde{c}_{i,j}$), so the result follows once we show that $\gamma$ cannot terminate at a cusp edge.  

Suppose $\gamma$ does terminate at a point $x$ of the cusp edge.  
Then
  $-\nabla F_{i,j}(x)$ must point to the side of the cusp locus where more sheets exist, and one of $i$ or $j$ must be a cusp sheet.  [Near a cusp locus, $-\nabla F_{i,j}$ is only defined on the side where $S_i$ and $S_j$ both exist.  Therefore if $\gamma$ runs into a cusp locus involving $S_i$ or $S_j$ it must approach this cusp locus from the side where $S_i$ and $S_j$ both exist, and since we follow $\gamma$ in the direction where $t$ decreases it follows that $-\nabla F_{i,j}(x)$ points to the side of the cusp locus where $S_i$ and $S_j$ both exist.]  By Property \ref{pr:CuspTransversality}, the square is then of type (12) with $i = k+1$ and $j=k+2$, so $\gamma$ belongs to either the stable manifold of $b^U_{k+1,k+2}$ or $b^R_{k+1,k+2}$.  [These are the only Reeb chords with upper and lower sheets $S_{k+1}$ and $S_{k+2}$ respectively.]  However, by Proposition \ref{prop:bUS}, these stable manifolds enter $\widehat{H}(e^2_\alpha)$ in the upper right corner where $x_1 \geq 1/4$ and $x_2 \geq 1/4$.  Then,  Property \ref{pr:monotonicityI} shows that the stable manifolds of $b^U_{k+1,k+2}$ or $b^R_{k+1,k+2}$ must remain in the region where $x_1 \geq 1/4$ and $x_2 \geq 1/4$, so they cannot reach the cusp locus. 
\end{proof}

\subsubsection{Non-existence of switches and $Y_1$ vertices} \label{ssec:preliminary}

\begin{proposition}
\label{prop:NoY1NoSW}
Consider any 2-cell of $e^2_\alpha$ of type (1)-(12).
No GFT in $N(e^2_\alpha)$ has a $Y_1$- or $sw$-vertex.
\end{proposition}

\begin{proof}
The image of a $Y_1$-vertex must be at a point on the cusp locus.  Moreover, at this image there there must be sheets $S_{i}$ and $S_j$ that respectively lie above and below the cusp edge and have the property that $-\nabla F_{i,j}$ points into the side of the cusp locus where the number of sheets increases by $2$.  Such a point cannot exist by Property \ref{pr:CuspTransversality}.

Property \ref{pr:CuspTransversality} also states $N(e^2_\alpha)$ has a switch point only when $e^2_\alpha$ is a Type (12) square.  Then, there is a unique $(k+1,k+2)$ switch point with $x_2$-coordinate above the intersection of the  crossing and cusp loci.  We show that a switch cannot occur at this point in any PFT.  If this were to happen, the outgoing edge of the switch would be a $(k,k+2)$-flow line, $\gamma$, that begins at the switch point.  By Property \ref{pr:monotonicityI}, $\gamma$ cannot cross into the region $T$ where $x_1 \geq 1/2$ and $x_2 \geq 1/2$.  There are no critical points of $F_{k,k+2}$ in the part of $N(e^2_\alpha)$  with $F_{k,k+2} >0$ outside of $T$, and $S_{k}$ and $S_{k+2}$ do not meet at a cusp edge.  Thus, if it did not end at an internal vertex, $\gamma$ would  only be able terminate at the cusp or crossing locus, and would not be part of a valid PFT.  Now, if $\gamma$ ends at an internal vertex it can only be a $Y_0$ that occurs above the $(k+1,k+2)$-crossing locus.  
 The two outgoing branches of the $Y_0$ will be a $(k,k+1)$-flow line and a $(k+1,k+2)$-flow line.  However, the $(k+1,k+2)$-branch cannot be completed to a valid PFT.  Indeed, Property \ref{pr:monotonicityI} applies to show the flowline stays in the complement of $T$ where there are no critical points of $F_{k+1,k+2}$ with $F_{k+1,k+2}>0$, and there is no $(k+1,k+2)$-cusp edge.  Thus, if it were part of a PFT, then this edge would have to end with an internal vertex.  However, $S_{k+1}$ and $S_{k+2}$ have no sheets between them in the region where $F_{k+1,k+2} >0$, so this edge cannot end at a $Y_0$- or $Y_1$-vertex.  Moreover, the edge cannot end at the $(k+1,k+2)$-switch point, since the difference between $z$-coordinates of sheets strictly decreases along all non-constant edges of a PFT.
\end{proof}

\begin{proposition}  \label{prop:ACE112}
For a square $e^2_\alpha$ of type (1)-(12), any GFT $\Gamma$ starting at one of the generators $c_{i,j}$ or $\tilde{c}_{i,j}$ is rigid if and only if 
\[
A-C+E = 0
\]
where $A$ and $C$ denote the number of endpoints at generators of the form $a^{\pm,\pm}_{i,j}$ and $c_{i,j}$ or $\tilde{c}_{i,j}$ respectively, and $E$ denotes the number of endpoints at $e$-vertices.
\end{proposition}

\begin{proof}
This follows from Proposition \ref{prop:NoY1NoSW} and Proposition \ref{prop:EY1SWFormula}.
\end{proof}

\subsection{Rigid GFTs with at least one $c_{i,j}$ output}

The computation of $\partial c_{i,j}$, as stated in Theorem \ref{thm:SquareComp}, is accomplished in two parts.  In this section, we verify those terms in $\partial C$ that correspond to rigid GFTs with at least one output at a $c$ Reeb chord.  Then, in \ref{ssec:112btrees} we identify the remaining terms in $\partial C$.

For $1 \leq i < j\leq n$, we introduce the notation
\[
\partial_c c_{i,j} = \sum_{\Gamma} w(\Gamma)
\]
where the sum is over those rigid GFTs, $\Gamma$, with an input at $c_{i,j}$ and having at least one output vertex at a Reeb chord of the form $c_{r,s}$ or $\tilde{c}_{r,s}$.

\begin{theorem} \label{thm:112ctrees}
For a $2$-cell $e^2_\alpha$ of type (1)-(12), and for any $1 \leq i < j \leq n$, we have
\begin{equation} \label{eq:112ctrees}
\partial_c C = A^{+,+} C + C A^{-,-} + X_1
%\sum_{i<m<j} a^{+,+}_{i,m} c_{m,j} + \sum_{i<m<j} c_{i,m} a^{+,+}_{m,j} 
\end{equation}
where the matrices $C, A^{+,+}$, and $A^{-,-}$ are as in Theorem \ref{thm:SquareComp}, and $X_1$ is a matrix whose entries belong to the $2$-sided ideal generated by exceptional generators in $N(e^2_\alpha)$.
\end{theorem}

The proof of Theorem \ref{thm:112ctrees} is given below after the preliminary Lemmas \ref{lem:1stQuad} and \ref{lem:3rdQuad}.

\subsubsection{The First and Third Quadrant Lemmas}
\label{sssec:QuadInv}

Recall that we have coordinates $\widetilde{N}(e^2_\alpha) \subset [-1-1/32,1+1/32] \times [-1-1/32,1+1/32]$ where $\widetilde{N}(e^2_\alpha) \subset N(e^2_\alpha)$ is obtained by removing interiors of the $N(e^0_{\pm,\pm})$.  
Now, take the full upper right corner of $N(e^2_\alpha)$,  
\[
\{(x_1,x_2) \in \widetilde{N}(e^2_\alpha) \, \left| \,  x_1\geq 1/4, \,\, x_2 \geq 1/4  \}\right. \bigcup N(e^0_{+,+}),
\]
and remove the vertical and horizontal bands 
\[
(\beta^U_{i,j} - \e,\beta^U_{i,j} + \e ) \times [1/4,1+1/32] \quad \mbox{ and } \quad [1/4,1+1/32] \times (\beta^R_{i,j} - \e,\beta^R_{i,j} + \e ).
\]
This results in four regions that we call the {\bf $(i,j)$-quadrants} and denote by $Q1(i,j), Q2(i,j), Q3(i,j), Q4(i,j)$ where 
\begin{align*}
Q1(i,j) \, =\, \{ (x_1,x_2)  \in \widetilde{N}(e^2_\alpha)  \, \left| \, x_1 \geq \beta^U_{i,j} +\e, \,\, 
 x_2 \geq \beta^R_{i,j}+\e\}\right. \bigcup N(e^0_{+,+}) 
\end{align*} 
and the numeration proceeds in a counter-clockwise manner.    In addition, let $N_\epsilon Q1(i,j)$ and $N_{\epsilon} Q3(i,j)$ denote thickened versions of $Q1(i,j)$ and $Q3(i,j)$ defined by
\[
N_\epsilon Q1(i,j) \, =\, \{ (x_1,x_2)  \in \widetilde{N}(e^2_\alpha)  \, \left| \, x_1 \geq \beta^U_{i,j} -\e, \,\, 
 x_2 \geq \beta^R_{i,j}-\e\}\right. \bigcup N(e^0_{+,+});
\] 
\[
N_\epsilon Q3(i,j) \, =\, \{ (x_1,x_2)  \in \widetilde{N}(e^2_\alpha)  \, \left| \, 1/4 \leq x_1 \leq \beta^U_{i,j} +\e, \,\, 
 1/4 \leq x_2 \leq \beta^R_{i,j}+\e\}\right..
\] 
See Figure \ref{fig:NeQ1}.

\begin{figure}

\labellist
\small
\pinlabel $N_\e Q3(i,j)$ [r] at -2 50 
\pinlabel $N_\e Q1(i,j)$ [l] at 490 126
\endlabellist
\centerline{ \includegraphics[scale=.6]{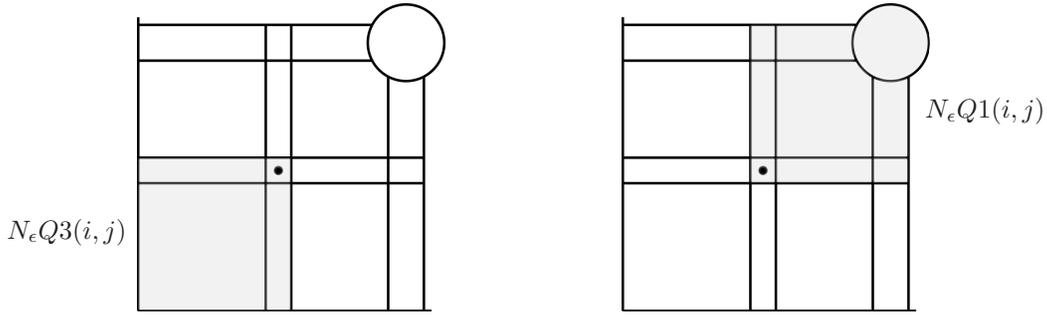} }

\caption{The enlarged first and third quadrants.  }
\label{fig:NeQ1}
\end{figure}

\begin{lemma}[First Quadrant Lemma] \label{lem:1stQuad}  Let $i< m_1 \leq m_2 < j$.
   Any PFT, $\Gamma$, that begins with an $(i,m_1)$-flow line at a point in $N_\e Q1(m_2,j)$ has all endpoints at $a^{+,+}$ generators.  Moreover, the image of $\Gamma$ is entirely contained in $N_\e Q1(m_2,j)$. 
\end{lemma}

\begin{proof}  
%According to Lemma \ref{lem:QuadrantFlows}, an $(i,m+1)$-flow cannot\footnote{Somewhere around here would be nice to illustrate the staircase ordering of $c_{i,j}$ and some quadrants.} leave $\widehat{Q}1(m_2,j) = Q1(m_2,j)$.  
For any $i \leq l_1 <l_2\leq m_1$, (\ref{eq:StaircaseAxiom}) implies that $N_\e Q1(m_2,j) \subset Q1(l_1,l_2)$   Moreover, 
%the monotonicity statement in Lemma \ref{lem:monotone} then implies 
Properties \ref{pr:1cmono} and \ref{pr:monotonicityI} imply that an $(l_1,l_2)$-flow line that begins in $N_\e Q1(m_2,j)$ cannot leave $N_\e Q1(m_2,j)$ as $t$-increases
%that $Q1(m_2,j)$ is invariant in positive time for $(l_1,l_2)$-flows.
(since $-\nabla F_{l_1,l_2}$ points inward along all portions of the boundary of $N_\e Q1(m_2,j)$).
  Within $N_\e Q_1(m_2,j)$, sheets appear in the same order that they are indexed, so any branch of the PFT tree that is below the initial  $(i,m_1)$-flow line can only be an $(l_1,l_2)$-flow with $i \leq l_1 <l_2\leq m_1$.    (Any $Y_0$-vertices, including punctures, can only increase the first subscript and decrease the second.)  
%Applying Lemma \ref{lem:QuadrantFlows} again, we see that  $Q1(m_2,j)$-quadrant is invariant for any of these flows.  
The result then follows, since the only Reeb chord with upper and lower sheets $S_{l_1}$ and $S_{l_2}$ contained in $N_\e Q1(m_2,j)$ is $a^{+.+}_{l_1,l_2}$.
\end{proof}

\begin{lemma}[Third Quadrant Lemma] \label{lem:3rdQuad}
Suppose $i< m_1 \leq m_2 <j$ and $x  \in N_{\e}Q3(i,m_1) \cap [1/2,1]\times [1/2,1]$.

Then any PFT beginning with an $(m_2,j)$-flow at $x$ has  
\begin{itemize}
\item[(1)] all endpoints at $a^{-,-}$'s and $e$-vertices, or

\item[(2)] at least $1$ endpoint at an exceptional generator.
\end{itemize}
\end{lemma}

\begin{proof}  We begin by establishing the following Lemma.

\begin{lemma} \label{lem:14121} Consider a PFT, $\Gamma$, that begins with a $(m_2,j)$-flow line at a point  $x \in N_{\epsilon}Q3(i,m_1) \cap [1/2,1]\times [1/2,1]$, for some $i < m_1 \leq m_2 < j$.  Let $\gamma$ be a path beginning at $x$ and consisting of several consecutive edges of $\Gamma$, (each parametrized according to their orientation in $\Gamma$).  The path $\gamma$ must cross into $[1/4,1/2]\times [1/4,1/2]$ before leaving $N_{\epsilon}Q3(i,m_1)$.
\end{lemma}

\begin{proof}[Proof of Lemma \ref{lem:14121}]

\begin{figure}
\labellist
\small
\pinlabel $A_1$ [l] at 214 158
\pinlabel $A_2$ [b] at 158 214
\pinlabel $B_1$ [l] at 172 42
\pinlabel $B_2$ [b] at 42 172
\pinlabel $C_1$ [r] at 122 80
\pinlabel $C_2$ [t] at 80 122
\pinlabel $1/4$ [t] at 32 -2
\pinlabel $1/4$ [r] at -2 32
\pinlabel $1/2$ [t] at 128 -2
\pinlabel $1/2$ [r] at -2 128
\pinlabel $\beta^U_{i,m_1}+\e$ [t] at 208 -2
\pinlabel $\beta^R_{i,m_1}+\e$ [r] at -2 208
\pinlabel $P$ [r] at 160 160
\endlabellist

\quad

\quad

\quad 

\centerline{ \includegraphics[scale=.6]{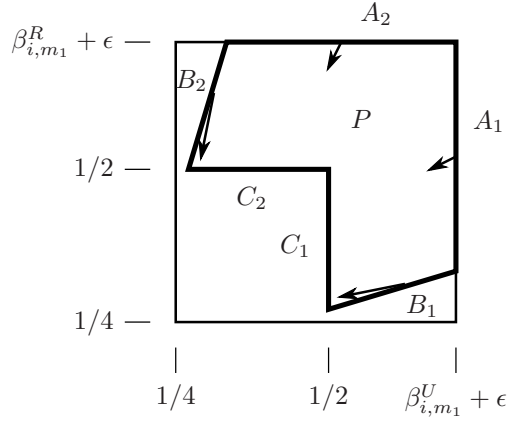} }

\quad

\caption{The polygonal region $P$ and its bounding segments.  The negative gradients $-\nabla F_{l_1,l_2}$ with $m_1 \leq l_1 < l_2 \leq j$ point into $P$ along $A_1\cup A_2 \cup B_1 \cup B_2$.}
\label{fig:PABC}
\end{figure}

\medskip

Consider the (non-convex) hexagonal subset $P$ of  
$N_\e Q3(i,m_1)$
bounded by line segments $A_1,A_2, B_1,B_2, C_1,C_2.$ 
Here,  $B_1 $ and $B_2 $ are the segments from Property \ref{pr:monotonicityIV};  
$A_1$ is the vertical segment from $(\beta_{i,m_1}^U+\e, \beta_{i,m_1}^U+\e) = (\beta_{i,m_1}^R+ \e, \beta_{i,m_1}^R+\e)$ 
to the right endpoint of $B_1$; and $C_1$ is   
the vertical segment from $(1/2,1/2)$ to the left endpoint of $B_1$.  
The remaining segments $A_2,B_2,C_2$ are obtained from reflecting $A_1,B_1,C_1$ across $x_1=x_2$.  See Figure \ref{fig:PABC}.

The region $P$ satisfies 
\[
N_{\epsilon}Q3(i,m_1) \cap [1/2,1]\times [1/2,1] \subset P \subset N_{\epsilon}Q3(i,m_1) \cap [1/4,1]\times [1/4,1].
\]
  Therefore, the path $\gamma$ which begins in $N_{\epsilon}Q3(i,m_1) \cap [1/2,1]\times [1/2,1]$ must leave $P$ to leave $N_{\epsilon}Q3(i,m_1)$.   Initially, $\gamma$ is a $(m_2,j)$-flow line with $i<m_1\leq m_2 <j$.  Moreover, until we leave $N_{\epsilon}Q3(i,m_1) \cap [1/4,1]\times [1/4,1]$ the only type of vertices that may appear in $\Gamma$ are $Y_0$-vertices.  [This is because the cusp locus is disjoint from $N_{\epsilon}Q3(i,m_1) \cap [1/4,1]\times [1/4,1]$.]  Within $N_{\epsilon}Q3(i,m_1) \cap [1/2,1]\times [1/2,1]$ the sheets of $\tilde{L}$ appear in the order that they are indexed, thus after a $Y_0$, either outgoing edge is, for some $m_2 \leq l_1 < l_2 \leq j$, an $(l_1,l_2)$-flow line. In particular, the edges that constitute $\gamma$ consist of such flows as long as $\gamma$ remains in $N_{\epsilon}Q3(i,m_1) \cap [1/4,1]\times [1/4,1]$. 
Property \ref{pr:monotonicityIV} then shows that when leaving $P$ the path $\gamma$ cannot cross $B_1$ or $B_2$, and Property \ref{pr:monotonicityI} shows $\gamma$ cannot cross our of $P$ along $A_1$ or $A_2$.  Therefore, $\gamma$ must leave $P$ along one of the segments $C_1$ or $C_2$, and immediately afterward its image lies in $[1/4, 1/2] \times[1/4,1/2]$.
\end{proof}

We return now to the proof of Lemma \ref{lem:3rdQuad}.  Consider an endpoint of a partial flow tree $\Gamma$ satisfying the stated hypothesis.  Let $\gamma$ denote the portion of the tree that travels from $x$ to this endpoint. Applying Lemma \ref{lem:14121}, at some point $y$ the path $\gamma$ is in $[1/4,1/2]\times[1/4,1/2]$. The proof will now be completed by applying the following Lemma \ref{lem:1412} to the PFT that consists of the subset of $\Gamma$ below $y$.

%reach an edge where the image of the path crosses out of $P$ for the first time. [Within $P$ only $Y_0$-vertices can occur since, for any $m_1\leq l_1<l_2\leq j$, $F_{l_1,l_2}$ has no critical points in $P$, and the new edges appearing after a $Y_0$-vertex are always $(l_1,l_2)$-flows with $m_1 \leq l_1 < l_2 \leq j$.]  The claim shows that this crossing must occur along one of the segments $C_1$ or $C_2$, and so the image lies in $[1/4, 1/2] \times[1/4,1/2]$.  The proof will now be completed by the following Lemma \ref{lem:1214}.
\end{proof}

\begin{lemma}\label{lem:1412}
Any partial flow tree with initial point in $[1/4,1/2] \times [1/4,1/2]$ has
\begin{itemize}
\item[(1)] all endpoints  at $a^{-,-}$'s or $e$-vertices,
or
\item[(2)] at least $1$ endpoint  at an exceptional generator.
\end{itemize}
\end{lemma}

\begin{proof}  Let $L(1/2) = \{(x_1,x_2) \in \widetilde{N}(e^2_\alpha) \, |\, x_1 \leq 1/2\} \bigcup N(e^0_{-,+}) \bigcup N(e^0_{-,-})$ and $D(1/2) = \{(x_1,x_2) \in \widetilde{N}(e^2_\alpha) \, |\, x_2 \leq 1/2\} \bigcup N(e^0_{+,-}) \bigcup N(e^0_{-,-})$ so that $L(1/2) \cap D(1/2) = \{(x_1,x_2) \in \widetilde{N}(e^2_\alpha) \, |\, x_1 \leq 1/2, \, x_2 \leq 1/2,\} \bigcup N(e^0_{-,-})$.  

We prove the stronger statement that any PFT with initial point in $L(1/2) \cap D(1/2)$ satisfies (1) or (2).

\medskip

\noindent {\bf Claim 1}.  If sheets $i,j$ with $i<j$ do not have a crossing arc on the $R$ (resp. $U$) and do not meet at a cusp above the $R$ (resp. $U$) edge, then 
$L(1/2)$  (resp. $D(1/2)$) is invariant for the $(i,j)$-flow, i.e. is $-\nabla F_{i,j}$ invariant in positive time. 

\medskip

This follows since together Properties \ref{pr:0cells}, \ref{pr:1cells}, \ref{pr:1cmono}, and \ref{pr:monotonicityII} show that  $-\nabla F_{i,j}$ points inward along all portions of the boundary.

\medskip

\noindent {\bf Claim 2}.  Any partial flow tree starting with a $(j,i)$-flow line with $j>i$ has an endpoint at an exceptional generator.

\medskip

At the first $Y_0$-vertex, the $(j,i)$-flow will split into edges that are $(j,m)$-flows and $(m,i)$-flows.  Since $j>i$, we will have either $j>m$ or $m>i$.  Proceeding inductively, we can then find an endpoint that limits to a Reeb chord whose upper endpoint is on a sheet with larger label than the lower endpoint.  All such Reeb chords are exceptional.

\medskip

\noindent {\bf Claim 3}.  Consider a square that has $2$ crossings along the edge $R$, i.e. of type (5), (6), (8) or (12), so that along $R$ sheet $k+2$ crosses sheets $k+1$ and $k$.  [Sheets are labelled as they appear at $(+1,+1)$.]  
   Then, any partial flow tree that begins with a $(k+1,k+2)$-flow in $D(1/2)$ has at least one endpoint at an exceptional Reeb chord.

\medskip
 
From Claim 1, we see that $D(1/2)$ is $(k+1,k+2)$-invariant.  [In any of the 12 squares, it is never the case that the same two sheets cross on both the edges $R$ and $U$.]  Observe that there are in fact no non-exceptional $(k+1,k+2)$ Reeb chords in $D(1/2)$ (and, in the case of a Type (12) square, no $(k+1,k+2)$ cusp edge).  Thus, such a partial flow tree without $Y_0$-vertices would necessarily contain an exceptional generator.  In all cases except for Type (8), sheets $k+1$ and $k+2$ are always adjacent in $D(1/2) \cap \{F_{k+1,k+2} >0\}$, so no $Y_0$-vertex can occur.  For a Type (8) square, the only portion of $D(1/2) \cap \{F_{k+1,k+2} >0\}$  where a $Y_0$-vertex may occur in a $(k+1,k+2)$-flow is in the portion of the square where sheets appear in the order $z(S_{k+1}) > z(S_{k}) >z(S_{k+2})$.  See Figure \ref{fig:TripPointOrder}.  One of the outgoing branches of such a $Y_0$-vertex would be a $(k+1,k)$-flow, and Claim 2 would then imply that the tree has an endpoint at an exceptional generator.

\begin{figure}

\quad

\quad

\labellist
\small
\pinlabel \fbox{$\begin{array}{l}k\\k+1\\k+2\end{array}$} [l] at 228 160
\pinlabel \fbox{$\begin{array}{l}k\\k+2\\k+1\end{array}$} [l] at 228 80
\pinlabel \fbox{$\begin{array}{l}k+2\\k\\k+1\end{array}$} [l] at 228 0
\pinlabel \fbox{$\begin{array}{l}k+1\\k\\k+2\end{array}$} [r] at -4 160
\pinlabel \fbox{$\begin{array}{l}k+1\\k+2\\k\end{array}$} [r] at -4 80
\pinlabel \fbox{$\begin{array}{l}k+2\\k+1\\k\end{array}$} [r] at -4 0
\endlabellist
\centerline{ \includegraphics[scale=.6]{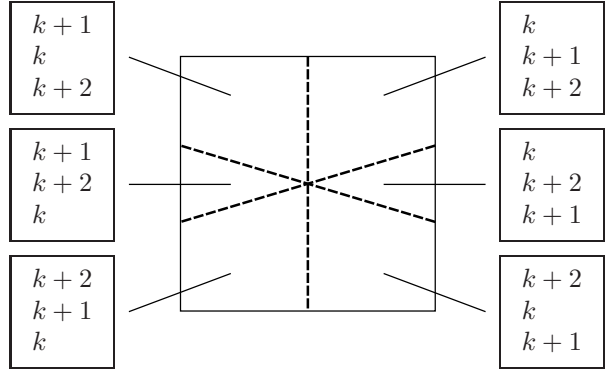} }

%%\centerline{ \includegraphics[scale=.8]{images/SubdivideSq} } %Dan had already commented out this line

\quad 

\quad

\caption{The order in which the $3$ sheets that meet at a triple point appear in the $6$ portions of a Type (8) square.}
\label{fig:TripPointOrder}
\end{figure}

\medskip

With Claims 1-3 in hand, suppose that we are given a partial flow tree in a square of type (1)-(12) that begins with an $(i,j)$-flow at $x \in L(1/2) \cap D(1/2)$.   In establishing the conclusion of the lemma, by Claim 2, we may assume $i <j$.  We first treat all of the special cases in which the $i$ and $j$ sheets cross or meet at a cusp edge, and then address the remaining cases inductively.  Note that an $(i,j)$ crossing or cusp arc cannot have endpoints on both the $R$ and $U$ edge.  
\begin{itemize}
\item Case 1:  The $(i,j)$ crossing arc does have a single endpoint on edge $R$ or $U$.    
This includes the squares of type (2), (4)-(8), (10), (12) where $(i,j)$ can correspond to any of the crossing arcs.  Note that for any of the squares (5), (6), (12) with $(i,j) = (k+1,k+2)$, Claim 3 applies.  The remaining cases are as follows:
\begin{itemize}
\item Subcase 1: The square is any of (2), (4), (7), (10) with $(i,j)$ any of the pictured arcs.  

First consider the case where the $(i,j)$ crossing occurs on edge $R$.  
Then, from Claim 1, $D(1/2)$ is $(i,j)$-invariant.  Moreover, in the subset of $D(1/2)$ where $F_{i,j}>0$, the sheets $i$ and $j$ are adjacent so that there are no possibilities for a partial flow starting at $x$ to have $Y_0$-vertices.  Finally, observe that the only non-exceptional $(i,j)$ Reeb chord in $D(1/2) \cap \{F_{i,j}>0\}$ is at $(-1,-1)$.  [For a general square, the only non-exceptional $(i,j)$ Reeb chords in $D(1/2)$ are $a^{-,-}_{i,j}, b^D_{i,j}$ and $a^{+,-}_{i,j}$, but in the present case there are no Reeb chords of the form $b^D_{i,j}$ and $a^{+,-}_{i,j}$ as they would necessarily on the other side of the $(i,j)$ crossing arc.]

If the crossing occurs on edge $U$, a similar argument (reflected across $x_1=x_2$) applies.
  
\item Subcase 2: The square is (5), (6), or (12) with $(i,j) = (k,k+2)$.  Again, $D(1/2)$ is $(k,k+2)$ invariant, and the only non-exceptional $(k,k+2)$ Reeb chord in $D(1/2)$ would have the form $a^{-,-}_{i,j}$.  Thus, in the case there are no $Y_0$-vertices the result follows.   If the flow tree has $Y_0$-vertices, then at the first $Y_0$-vertex the tree splits into a $(k,k+1)$-flow and a $(k+1,k+2)$-flow in the region above the $(k+1,k+2)$-crossing arc.  Now, an application of Claim 3 shows that the flow tree must contain an exceptional generator.  

\item Subcase 3:  The triple point square (8).  
The case $(i,j) = (k+1,k+2)$ is covered by Claim 3.
For  $(i,j) = (k,k+1)$, 
the partial flow tree must have an end point at an exceptional generator.   This follows since $L(1/2)$ 
is $(i,j)$-invariant and contains no non-exceptional Reeb chords of type $(i,j).$  The only possibility for a $Y_0$-vertex would produce an outgoing edge that is a $(k+2,k+1)$-flow 
to which we apply Claim 2.  (See Figure \ref{fig:TripPointOrder}.)

Finally, consider $(i,j) = (k,k+2)$.   The $(k,k+2)$-flow is $D(1/2)$-invariant with all $(i,j)$ Reeb chords in $D(1/2)$ exceptional.  If a $Y_0$-vertex occurs one out going edge will be a $(k+1,k+2)$-flow in $D(1/2)$ which must have an endpoint at an exceptional generator by Claim 3.

\end{itemize}
\item Case 2:  The $(i,j)$ crossing arc has its endpoints on edges $L$ and $D.$
This only occurs in (3) with $(i,j)=(k,k+1)$.  Here, Claim 1 gives that $L(1/2) \cap D(1/2)$ is $(k,k+1)$-invariant.  In this region, sheets $k$ and $k+1$ are adjacent and all $(k,k+1)$ Reeb chords are exceptional.

\item Case 3:  The $(i,j)$ cusp edge has its endpoint on $U$ or $R$.  Supposing the cusp edge lies on $U$, Claim 1 shows that $L(1/2)$ is $(i,j)$ invariant.  As they are adjacently indexed, any $Y_0$ would produce an outgoing edge that is a $(m,n)$-flow with $m>n$, and Claim 2 would then produce an exceptional generator.  Since there are no $(i,j)$ Reeb chords in $L(1/2)$, if the PFT has no $Y_0$'s, then the $(i,j)$-flow necessarily reaches an $e$-vertex.
\end{itemize}

With the special cases out of the way we handle the general case using induction on the number of $Y_0$-vertices in the partial flow tree.  Since we can assume that the $i,j$ sheets do not cross, Claim 1 shows that the quadrant $L(1/2)\cap D(1/2)$ is $(i,j)$-invariant.  The only $(i,j)$-Reeb chord in $L(1/2)\cap D(1/2)$ is $a^{-,-}_{i,j}$, so in the case of no $Y_0$-vertices the result holds.  For the inductive step, at the first $Y_0$-vertex, an $(i,j)$-flow line splits into a $(i,l)$-flow line and a $(l,j)$-flow line for some $l$.  Since $L(1/2) \cap D(1/2)$ is $(i,j)$ invariant, the inductive hypothesis applies to both partial flow trees below the $Y_0$-vertex.  
%If $k$ is not between $i$ and $j$, then one of these two flow lines is an $(m,n)$-flow with $m>n$, so Claim 2 guarantees an endpoint at an exceptional generator. 
\end{proof}

\begin{proof}[Proof of Theorem \ref{thm:112ctrees}.]

The matrix formula (\ref{eq:112ctrees}) is equivalent to the term-by-term formula 
\begin{equation}\label{eq:termbyterm}
\partial_c c_{i,j} =\sum_{i < m < j} \left( a^{+,+}_{i,m}c_{m,j} +c_{i,m}a^{-,-}_{m,j}\right) + x
\end{equation}
where any of the terms $a^{-,-}_{i,m}$ which correspond to Reeb chords that do not exist or are exceptional generators are replaced with $0$, except that $a^{-,-}_{k,k+1}$ is replaced with $1$ if $S_{k}$ and $S_{k+1}$ meet at a cusp edge in $N(e^2_\alpha)$;  the $x$ is an element of the ideal generated by exceptional generators in $N(e^2_\alpha)$.  We show that, aside from $x$, each of the non-zero terms in (\ref{eq:termbyterm}) corresponds to an odd number of rigid GFTs  beginning at $c_{i,j}$ and that any other rigid GFTs with an output at some $c_{r,s}$ or $\tilde{c}_{r,s}$ must have endpoints at exceptional generators.

For each of the (non-zero) quadratic terms of (\ref{eq:termbyterm}) we  now observe a single GFT beginning at $c_{i,j}$ (and rigid by Proposition \ref{prop:ACE112}) whose word agrees with the given term.

\medskip

\noindent \textbf{Term 1:  $a^{+,+}_{i,m}c_{m,j}$ for $i<m<j$.}

Travel along the unique $(i,j)$-flow line beginning at $c_{i,j}$ that intersects $c_{m,j}$.  Once this point is reached a puncture occurs.  That is, a $Y_0$-vertex appears where the two outgoing edges are an $(i,m)$-flow line and the constant $(m,j)$-flow line at $c_{m,j}$.  Since $c_{m,j} \in N_\e Q1(m,j)$, %(by Property \ref{pr:Reeb2} and equation (\ref{eq:StaircaseAxiom}))  
 the non-constant outgoing edge is an $(i,m)$-flow line that limits to an output vertex at $a^{+,+}_{i,m}$ by Lemma \ref{lem:1stQuad}.  See Figure \ref{fig:112Term1}.

\noindent \textbf{Term 2:  $c_{i,m}a^{-,-}_{m,j}$ for $i<m<j$.}
Travel along the $(i,j)$-flow line that hits $c_{i,m}$, and then puncture.
Since $c_{i,m}\in N_\e Q3(i,m)$, Lemma \ref{lem:3rdQuad} shows that 
the outgoing $(m,j)$-flow line  behaves as follows:
\begin{enumerate}
\item It limits to $a^{-,-}_{m,j}$ as long as $S_m$ and $S_j$ do not cross and they both exist at $(x_1,x_2) =(-1,-1)$.  
\item If $S_m$ and $S_j$ cross, then if the crossing locus separates $c_{i,m}$ from $(-1,-1)$ the flow line runs into the crossing locus (and is not part of a valid GFT).  In the case that the crossing locus does not separate $c_{i,m}$ from $(-1,-1)$, as can happen for Type (4) or (6) squares, the flow line may either run into the crossing locus or end at $a^{-,-}_{m,j}$.  For such squares, $a^{-,-}_{m,j}$ is an exceptional generator so the precise behavior is not relevant.
\item If $S_m$ or $S_j$ meets a third sheet in a cusp edge, then the $(m,j)$-flow line ceases to exist when the cusp edge is reached (and is not part of a valid GFT).
\item If $S_m$ and $S_j$ meet one another at a cusp, then the flow line reaches an $e$-vertex at the cusp edge.  In this case we do have a valid rigid GFT.
\end{enumerate}
Note that the results described in (2) and (3) are consistent with the $0$ entries that appear in $A^{-,-}$ as a result of crossings or cusps.  Replacing the $(k,k+1)$-entry of $A^{-,-}$ with a $1$ when $S_k$ and $S_{k+1}$ meet at a cusp edge is consistent with (4). 
 See Figure \ref{fig:112Term1}.

\begin{figure}
\labellist
\small
\pinlabel $c_{i,j}$ [b] at 40 164
\pinlabel $c_{m,j}$ [t] at 56 92
\pinlabel $a^{+,+}_{i,m}$ [t] at 0 68
\pinlabel $c_{i,j}$ [r] at 162 48
\pinlabel $c_{m,j}$ [br] at 220 110
\pinlabel $a^{+,+}_{i,m}$ [l] at 286 162
\pinlabel $c_{i,j}$ [b] at 408 164
\pinlabel $c_{i,m}$ [t] at 392 92
\pinlabel $a^{+,+}_{m,j}$ [t] at 448 68
\pinlabel $c_{i,j}$ [br] at 644 141
\pinlabel $c_{i,m}$ [t] at 640 110
\pinlabel $a^{+,+}_{m,j}$ [r] at 506 10
\endlabellist
\centerline{ \includegraphics[scale=.6]{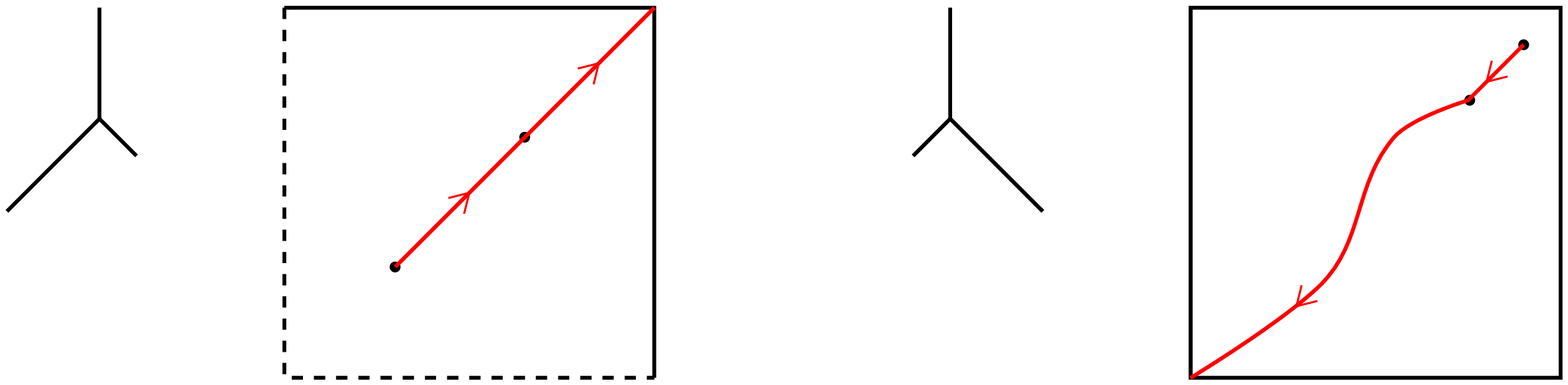} }
\caption{The domain and image (in red) of the trees from Terms 1 and 2.  The short edges in the domain trees are constant when mapped into $N(e^2_\alpha)$. }
%The arrows indicate the direction along edges of all $(i,j)$-flows (negative gradients) other than the $(k+2,k+1)$ flow.  
%(right) The partial flow trees in $A$ described in (2) of Lemma \ref{lem:RegionA}.  
\label{fig:112Term1}
\end{figure}

\medskip

\noindent \textbf{Proof of uniqueness.}  We have thus identified a single
 rigid GFT for each of the (non-zero) terms in (\ref{eq:termbyterm}) aside from $x$.  
  To complete the proof, we show that these are the only rigid GFTs beginning at $c_{i,j}$ with at least one output at some $c_{r,s}$ or $\tilde{c}_{r,s}$ and without any outputs at exceptional generators.  We follow the outline indicated in Figure \ref{fig:CStep1}.

Let $\Gamma$ be a rigid GFT that begins at some $c_{i',j'}$, $i'< j'$,  without endpoints at exceptional generators and having at least one puncture at some $c_{r,s}$ or $\tilde{c}_{r,s}$.  In the latter case, we are done since $\tilde{c}_{r,s}$ is an exceptional generator.

\medskip

\noindent \textbf{Step 1.}  If $\Gamma$ has an $(i,j)$-edge that punctures at $c_{i,m}$ (resp. $c_{m,j}$), then there are no further internal vertices on the non-constant branch that follows the puncture.  Thus, after the puncture $\Gamma$ agrees with the GFT found in Step 1 (resp. Step 2). 

\medskip

The critical point $c_{i,m}$ is located in $N_\e Q3(i,m)\cap [1/2,1]\times[1/2,1]$.  Therefore, the PFT starting with the non-constant branch following the puncture begins with a $(m,j)$-flow in  $N_\e Q3(i,m)\cap [1/2,1]\times[1/2,1]$, and  Lemma \ref{lem:3rdQuad} implies all endpoints below the puncture must be at $a^{-,-}$ Reeb chords or $e$-vertices.    A similar argument with Lemma \ref{lem:1stQuad} used in place of Lemma \ref{lem:3rdQuad} shows that all endpoints below a puncture at $c_{k,j}$ must be at $a^{+,+}$ Reeb chords. 
Since Proposition  \ref{prop:NoY1NoSW} shows that any internal vertices of $\Gamma$ must be $Y_0$'s, the number of outputs of the PFT that starts below the puncture will be equal to the number of $Y_0$'s plus $1$.  
However, Proposition \ref{prop:ACE112} states that the total number of outputs of $\Gamma$ at $a$ Reeb chords and $e$-vertices must be the same as the number of $c$ outputs.  Thus, there can be only one $a$ or $e$-vertex below any $c$ puncture, so after the puncture there are no $Y_0$-vertices and Step 1 follows. 

\medskip 

In addition, Proposition \ref{prop:ACE112} together with Step 1 show that: 

\begin{equation}\label{eq:ac}
\mbox{Any edge with its endpoint at an $a$ or $e$-vertex must begin with a $Y_0$ that is a puncture at a $c$.}
\end{equation}

[As we have seen, any puncture at a $c$  Reeb chord has its non-constant branch ending at an $a$ or an $e$-vertex.  This, gives an injection from the set of $c$ endputs of $\Gamma$ to the set of endpoints of $\Gamma$ at $a$ Reeb chords and at $e$-vertices.  Proposition \ref{prop:ACE112} shows that this injection is actually a bijection.]

\begin{figure}

\quad

\quad 

\labellist
\small
%\pinlabel $1/4$ [t] at 338 21
\pinlabel $\vdots$ [b] at 35 229
\pinlabel $\vdots$ [b] at 163 229
\pinlabel $\vdots$ [b] at 339 229
\pinlabel $\vdots$ [b] at 467 229
\pinlabel $\vdots$ [b] at 35 68
\pinlabel $\vdots$ [b] at 339 68
\pinlabel $\vdots$ [t] at 320 186
\pinlabel $\vdots$ [t] at 53 186
\pinlabel $c_{i,m}$ [t] at 2 153
\pinlabel $c_{i,m}$ [t] at 130 153
\pinlabel $a^{-,-}_{m,j}$ [t] at 194 153
\pinlabel $c_{m,j}$ [t] at 371 153
\pinlabel $a^{+,+}_{i,m}$ [t] at 434 153
\pinlabel $c_{m,j}$ [t] at 498 153
\pinlabel $c_{i,m}$ [t] at 2 -2
\pinlabel $c_{i,m}$ [t] at 130 -2
\pinlabel $a^{-,-}_{m,j}$ [t] at 194 -2
\pinlabel $c_{m,j}$ [t] at 371 -2
\pinlabel $a^{+,+}_{i,m}$ [t] at 434 -2
\pinlabel $c_{m,j}$ [t] at 498 -2
\pinlabel $a^{-,-}_{m,j}$ [t] at 66 -2
\pinlabel $a^{+,+}_{i,m}$ [t] at 306 -2
\pinlabel $c_{i,j}$ [b] at 162 102
\pinlabel $c_{i,j}$ [b] at 466 102
\pinlabel Step~2: [r] at -10 32
\pinlabel Step~1: [r] at -10 192
%\pinlabel Step~3: [r] at -10 34
%\pinlabel $\vdots$ [b] at 186 70
%\pinlabel $b_{i,m}$ [t] at 154 -2
%\pinlabel $b_{m,j}$ [t] at 218 -2
%\pinlabel $b_{i,m}$ [t] at 282 -2
%\pinlabel $b_{m,j}$ [t] at 346 -2
%\pinlabel $c_{i,j}$ [b] at 314 102
\endlabellist
\centerline{ \includegraphics[scale=.6]{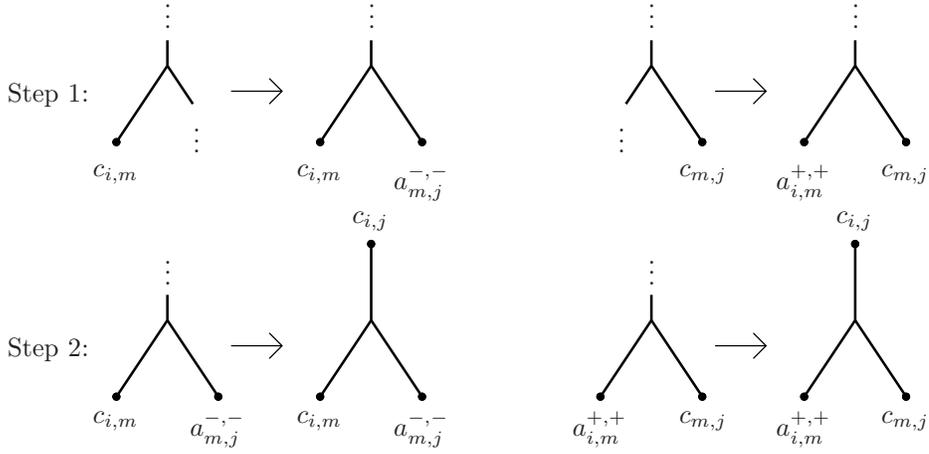} }

\quad 

%%\centerline{ \includegraphics[scale=.8]{images/SubdivideSq} } %Dan had already commented out this line

\caption{Steps 1 and 2 in the Proof of Uniqueness.}
\label{fig:CStep1}
\end{figure}

\medskip

%%%\noindent \textbf{Step 2.}  Any edge terminating at an $a$ must be immediately preceeded by a $c$.   

%%%This is for index reasons.  There must be a bijection between $c$'s and $a$'s.  Part 1 gives an injection from $c$'s to $a$'s and it must be bijective.

\noindent \textbf{Step 2.}  In (the domain of) $\Gamma$ a puncture at a $c$ cannot appear below any internal vertex, $x$.  

\medskip

By Proposition \ref{prop:NoY1NoSW} the internal vertex $x$ can only be a $Y_0$-vertex. 

\begin{itemize}
\item[ Case 1:] Assume $i<j<l$, and an $(i,l)$-flow branches at $x$ into the $(i,j)$-flow to some $c$ and a $(j,l)$-flow.  
\begin{itemize}
\item[$\bullet$] Subcase $c=c_{i,m}$:  Introduce the notation 
\[
\mathit{Sq}(c_{i,m}, c_{i,j}) = [\beta^U_{i,m} - \e, \beta^U_{i,j}+\e] \times [\beta^R_{i,m}-\e, \beta^R_{i,j} +\e],
\]
which may be approximately viewed as a square with lower left corner at $c_{i,m}$ and upper right corner $c_{i,j}$.  See Figure \ref{fig:Sqcij}.   According to Property \ref{pr:monotonicityI}, $-\nabla F_{i,j}$ points outward along each of the edges of $\mathit{Sq}(c_{i,m}, c_{i,j})$.  Therefore, as we follow the $(i,j)$-edge with decreasing $t$, from its end at $c_{i,m}$ towards its beginning at $x$, the edge cannot leave the square.  Thus, 
\[
x \in \mathit{Sq}(c_{i,m}, c_{i,j}) \subset  N_\e Q3(i,j) \cap \left( [1/2,1]\times[1/2,1] \right),
\]
and we apply Lemma \ref{lem:3rdQuad} to the $(j,l)$-branch that starts at $x$ to arrive at a contradiction of (\ref{eq:ac}).

\begin{figure}

\labellist
\small
\pinlabel $c_{i,j}$ [l] at 94 90
\pinlabel $c_{r,s}$ [r] at 155 159
\endlabellist
\centerline{ \includegraphics[scale=.6]{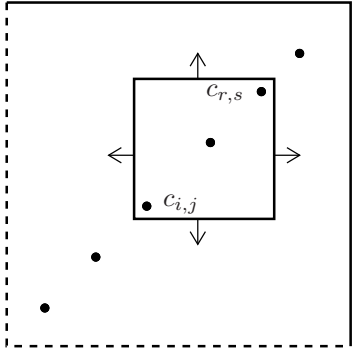} }

\caption{The square, $\mathit{Sq}(c_{i,j},c_{r,s})$.  For any $(p,q)$ such that $(i,j) \preceq (p,q) \preceq (r,s)$ with respect to lexicographic order, $-\nabla F_{p,q}$ points outward along all edges of $\mathit{Sq}(c_{i,j}, c_{r,s})$. }
%The arrows indicate the direction along edges of all $(i,j)$-flows (negative gradients) other than the $(k+2,k+1)$ flow.  
%(right) The partial flow trees in $A$ described in (2) of Lemma \ref{lem:RegionA}.  
\label{fig:Sqcij}
\end{figure}

\item[$\bullet$] Subcase $c=c_{m,j}$:  A similar argument shows
\[
x \in \mathit{Sq}(c_{i,j}, c_{m,j}) \subset  N_\e Q3(m,j) \cap \left( [1/2,1]\times[1/2,1] \right),
\] 
and we again apply Lemma \ref{lem:3rdQuad} to the $(j,l)$-branch starting at $x$ to deduce a contradiction.
\end{itemize}
\item[Case 2:] Assume $h<i<j$, and an $(h,j)$-flow branches at $x$ into an $(h,i)$-flow and the $(i,j)$-flow to $c$.  

An argument similar to that of Case 1 applies with Lemma \ref{lem:1stQuad} used in place of Lemma \ref{lem:3rdQuad}.
%in a similar way to deduce a contradiction.  (This time pay attention to the $1$st quadrant.) 
\end{itemize}

\medskip

At this point, as a consequence of Steps 1 and 2, we see that 
%if $\Gamma$ has one endpoint at a $c$ (or equivalently, at least one endpoint at an $a$ or an $e$-vertex), then $
$\Gamma$ must be one of the trees identified in Terms 1 and 2 above.  This completes the proof.
%Indeed, Steps 1 and 2 show that  they agree with the trees described above in connection with the terms $a^{+,+}_{i,m}c_{m,j}$ and $c_{i,m} a^{-,-}_{m,j}$.   

\end{proof}

%\input{LCH5b}

%\subsection{New approach}

%Use the new conditions on $d_xf_{i,i+1}$ in $[1/4,1/2]$ (with the interval where $d_xf= \epsilon_4$ extended) to establish.

%\begin{theorem}
%Suppose $i< m_1 \leq m_2 <j$ and $x = (X,Y) \in [1/2,1]\times [1/2,1]$.

%If in addition $x$ belongs to the third $(i,m_1)$-quadrant, then any partial flow tree beginning with a $(m_2,j)$-flow at $x$ has 
%\begin{itemize}
%\item[(1)] all endpoints of the partial flow tree are at $a^{-,-}$'s or cusp edge points,
%OR
%\item[(2)] at least $1$ endpoint is at an exceptional generator.
%\end{itemize}

%If instead $x$ belongs to the $1$-st $(m_2,j)$-quadrant, then any partial flow tree beginning with a $(i,m_1)$-flow at $x$ has all endpoints at %$a^{+,+}$'s.
%\end{theorem}

%Idea of Proof:  For first half, show that all $(k,l)$-flows with $m_2\leq k < l \leq j$ cannot cross a union of line segments $x_1= x_1(c_{i,m_1})$, $x_2 =x_1(c_{i,m_1})$ and appropriate lines of slope $\epsilon_3$.  This forces all such flows into the region $[1/4,1/2] \times [1/4,1/2]$.  Thus it suffices to establish the conclusion for flows in this region.  This should basically follow from what is written above.

%\begin{lemma}\label{lem:1412}
%Any partial flow tree with initial point in $[1/4,1/2] \times [1/4,1/2]$ has
%\begin{itemize}
%\item[(1)] all endpoints of the partial flow tree are at $a^{-,-}$'s or cusp edge points,
%OR
%\item[(2)] at least $1$ endpoint is at an exceptional generator.
%\end{itemize}
%\end{lemma}

\subsection{Rigid GFTs without outputs at any $c_{i,j}$}  \label{ssec:112btrees}

For $1\leq i < j \leq n$, let
\[
\partial_b c_{i,j} = \sum_{\Gamma} w(\Gamma)
\]
where the sum is over rigid GFTs that do \emph{not} have outputs at Reeb chords of the form $c_{r,s}$ or $\tilde{c}_{r,s}$.   

\begin{theorem} \label{thm:112btrees}
For a $2$-cell $e^2_\alpha$ of type (1)-(12), and for any $1 \leq i < j \leq n$, we have
\begin{equation} \label{eq:112btrees}
\partial_b C = (I+B_R)(I+B_D) + (I+B_U)(I+B_L) + X_2
%\sum_{i<m<j} a^{+,+}_{i,m} c_{m,j} + \sum_{i<m<j} c_{i,m} a^{+,+}_{m,j} 
\end{equation}
where the matrices $C$, and $B_R, B_D, B_U, B_L$ are as in Theorem \ref{thm:SquareComp}, while $X_2$ is a matrix whose entries belong to the $2$-sided ideal generated by exceptional generators in $N(e^2_\alpha)$.
\end{theorem}

We prove Theorem \ref{thm:112btrees} at the conclusion of this section after reducing the computation to a purely topological problem.  

\subsubsection{$b$-trees}

\begin{lemma}  \label{lem:112reduce}
In $N(e^2_\alpha)$ where $e^2_\alpha$ has type (1)-(12), any rigid GFT starting at a $c_{i,j}$ and without outputs at Reeb chords of the form $c_{r,s}$ or $\tilde{c}_{r,s}$ must have
\begin{enumerate}
\item all outputs at Reeb chords of the form $b^X_{i,j}$ or $\tilde{b}^X_{i,j}$ with $X \in \{R, D, L, U\}$; and
\item all internal vertices are $Y_0$'s.
\end{enumerate}
\end{lemma}
\begin{proof}
This is a consequence of Propositions \ref{prop:NoY1NoSW} and \ref{prop:ACE112}.
\end{proof}
\begin{definition}
A {\bf $b$-tree} is a PFT or GFT that satsifies conditions (1) and (2) from Lemma \ref{lem:112reduce} and also has
\begin{enumerate}
\item[(3)] no outputs at exceptional generators.
\end{enumerate}
Define a {\bf $b^{LU}$-tree} to be a $b$-tree with all endpoints at $b^L_{i,j}$ and/or $b^U_{i,j}$ Reeb chords.  Define a { \bf $b^{RD}$-tree} in a similar manner.
\end{definition}
%Rephrase the computation of $\partial c_{i,j}$ as in swallowtail section by writing $ \partial_c c_{i,j} + \partial_b c_{i,j} +\partial_x c_{i,j}$.  Observe that $\partial_b c_{i,j}$ is a sum over {\bf $b$-trees} which we define as PFTs or GFTs with only $Y_0$'s and all endpoints at (non-exceptional) $b^X_{i,j}$ Reeb chords.  Define a {\bf $b^{LU}$-tree} to be a $b$-tree with all endpoints at $b^L$ and/or $b^U$ Reeb chords.  Define {\bf $b^{RD}$-tree} in a similar manner.  

For $1 \leq i < j \leq n$, define a region $B^{LU}(i,j) \subset \widehat{N}(e^2_\alpha)$ to be the closed subset that is 
\begin{itemize}
\item  bounded on the right by $x_1 = \beta^U_{i,j} + \e$;
\item bounded below by (i) the horizontal segment from  $(\beta^U_{i,j}+\e, \beta^R_{i,j}-\e)$ to the barrier $B_2$ from Property \ref{pr:monotonicityIV}; (ii)  the part of the barrier $B_2$ that lies between $x_2 = \beta^R_{i,j}-\e$ and $x_2 = 1/2$; and (iii) the horizontal segment at $x_2 = 1/2$ from $B_2$ to the left boundary of $\widehat{N}(e^2_\alpha)$; and  
\item bounded on the left and above by $\partial \widehat{N}(e^2_\alpha)$.
\end{itemize}
In addition, define regions $B^{RD}(i,j)$ to be the reflection of $B^{LU}(i,j)$ across the line $x_1 = x_2$.  
See Figure \ref{fig:BLUij}.
%\footnote{\ms{7/22/15: Can you label $x_1 = 1/4$ in Figure \ref{fig:BLUij}?} \dr{7/24:  Of course the figure is not to scale at all.  Why do you want $x_1=1/4$?  To remind readers of the approximate location of $B_2$ (which is actually only narrowed down to $x_1 \in (1/4,1/2)$ in the statement of the property)?  Nothing done yet, but this is an easy thing to add if you convince me its a good idea.}}

\begin{figure}

\quad

\labellist
\small
%\pinlabel $1/4$ [t] at 338 21
\pinlabel $c_{i,j}$ [br] at 328 142
\pinlabel $x_2=\beta^R-\e$ [l] at 380 128
\pinlabel $x_1=\beta^U_{i,j}+\e$ [t] at 344 92
\pinlabel $B_2$ [l] at 238 32
\pinlabel $x_2=1/2$ [r] at -2 0
%\pinlabel $I_{2,3}$ [r] at 360 208
\endlabellist
\centerline{ \includegraphics[scale=.6]{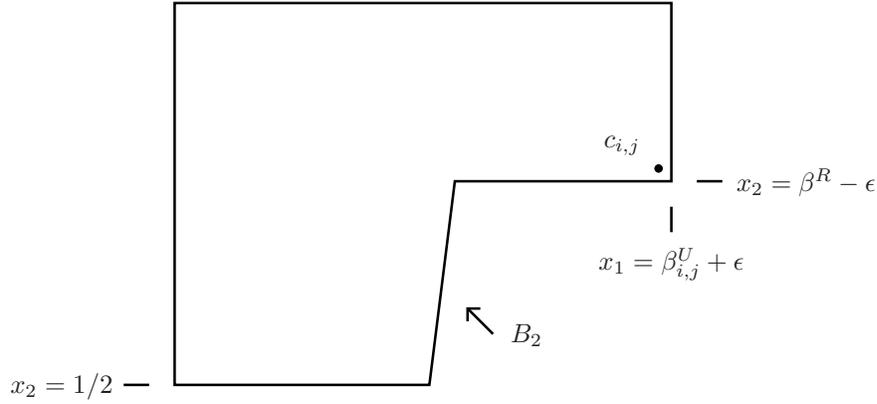} }

%%\centerline{ \includegraphics[scale=.8]{images/SubdivideSq} } %Dan had already commented out this line

\caption{The region $B^{LU}(i,j)$.  In the later Section \ref{sec:bLU13trees}, we will modify the definition of $B^{LU}(i,j)$ so that the segment $B_2$ from Property \ref{pr:monotonicityIV} is replaced with the Swallow Tail Barrier, $P$, from Property \ref{pr:STBnew}. 
%pictured with flow lines from the stable manifolds of $b^U_{i,j}$ and $b^L_{i,j}$.
}
\label{fig:BLUij}
\end{figure}

\begin{lemma} \label{lem:BLUij}   %The intersection of any edge, $\gamma$, of a $b^{LU}$-tree (resp. a $b^{RD}$-tree) with $\widehat{N}(e^2_\alpha)$ must be an $(i,j)$-flow line for some $i<j$, and must have its image contained in the region $B^{LU}(i,j)$ (resp. $B^{RD}(i,j)$).  
 %The intersection of A
 Any edge, $\gamma$, of a $b^{LU}$-tree (resp. a $b^{RD}$-tree) must 
 \begin{enumerate}
 \item be an $(i,j)$-flow line for some $i<j$ and 
\item have the intersection of its image with $\widehat{N}(e^2_\alpha)$ contained in the region $B^{LU}(i,j)$ (resp. $B^{RD}(i,j)$).  Moreover, the image of $\gamma$ is entirely in $\widehat{N}(e^2_\alpha)$ unless $\gamma$ is part of a $b^X_{i,j}$-line for $X \in \{L,D,R,U\}$.    
\end{enumerate}
\end{lemma}
(The terminology \emph{$b^X_{i,j}$-line} from (2) is as in \ref{sssec:blines}.)

\begin{proof}  First, in some cases, we need to focus on certain subsets $A^{LU}(i,j) \subset B^{LU}(i,j)$ defined by
\begin{enumerate}
\item $A^{LU}(i,j) = B^{LU}(i,j) \cap \{x_1 \geq 1/4\}$
%\{(x_1,x_2) \in \widehat{N}(e^2_\alpha) \, |\, \beta^U_{i,j} -\e \leq x_1 \leq \beta^U_{i,j} +\e, \quad \beta^R_{i,j}-\e \leq x_2\}$ 
if either $S_i$ and $S_j$ cross above $e^1_U$ or meet one another at a cusp above $e^1_U$;
\item $A^{LU}(i,j) = B^{LU}(i,j) \cap \{x_1 \geq -3/8\}$ when exactly one of $S_i$ and $S_j$ ends at a cusp edge above $e^1_U$; and
\item $A^{LU}(i,j) = B^{LU}(i,j)$ in all other cases.
\end{enumerate}
Observe that that $-\nabla F_{i,j}$ points outward along all segments of $\partial A^{LU}(i,j)$.  [With $A^{LU}(i,j)$ as in (1), for boundary segments that are not part of $\partial \widehat{N}(e^2_\alpha)$, just use Properties \ref{pr:monotonicityI}, \ref{pr:monotonicityII} and \ref{pr:monotonicityIV}; as in (3) add Properties \ref{pr:monotonicityII}; as in (2) add Property \ref{pr:CuspTransversality}.]

We prove the proposition by establishing the same result with each $B^{LU}(i,j)$ replaced with $A^{LU}(i,j)$ (this is a stronger statement) using induction on $N$, the number of $Y_0$'s in the PFT beginning with $\gamma$.
By Proposition \ref{prop:bUS}, all of the $b^L_{i,j}$- and $b^U_{i,j}$-lines intersect  $\partial \widehat{N}(e^2_\alpha)$ at a point in $\partial A^{LU}(i,j)$.  [Note that if $A^{LU}(i,j)$ is as in (1) or (2) then there is no $b^L_{i,j}$-line.]   It follows that as $t$ decreases they must remain in $A^{LU}(i,j)$.  Thus, the case $N = 0$ holds.  
%Note also that the part of the $b^L_{i,j}$- and $b^U_{i,j}$-lines that lie outside of $\widehat{N}(e^2_\alpha)$ are contained in disjoint strips in $e^1_L$ and $e^1_U$.

For the inductive step, consider the $Y_0$ at the end of $\gamma$. Both of the PFTs that begin with the outgoing edges of the $Y_0$ are themselves $b^{LU}$-trees.  Thus, by the inductive hypothesis, this $Y_0$ has its image in $A^{LU}(i,m) \cap A^{LU}(m,j)$ for some $m$ with $i < m < j$.  [Note that the parts of the $b^L_{i,j}$- and $b^U_{i,j}$-lines that lie outside of $B^{LU}_{i,j}$ are contained in disjoint strips in $e^1_L$ and $e^1_U$ and cannot intersect one another.]
Thus, the result follows once we verify that:

\medskip

\noindent {\bf Claim.} For any Type (1)-(12) square, $e^2_\alpha$, for all $i < m <j$, we have  $A^{LU}(i,m) \cap A^{LU}(m,j) \subset A^{LU}(i,j)$.

\medskip

That 
\[
A^{LU}(i,m) \cap A^{LU}(m,j) \subset B^{LU}(i,m) \cap B^{LU}(m,j) \subset B^{LU}(i,j)
\] holds in general.  [The second inclusion is easily verified using the lexicographic ordering of the $\beta_{i,j}$.  See Figure \ref{fig:BLUijPf}.]  Thus, when $A^{LU}(i,j) = B^{LU}(i,j)$ the Claim follows.  Now, if $A^{LU}(i,j)$ is as in (2), then at least one of the pairs $(i,m)$ and $(m,j)$ includes a sheet that ends at a cusp edge above $e^1_U$, and so at least one of $A^{LU}(i,m)$ or $A^{LU}(m,j)$ lies entirely to the right of $x_1=-3/8$ so that 
\[
A^{LU}(i,m) \cap A^{LU}(m,j) \subset (B^{LU}(i,m) \cap B^{LU}(m,j)) \cap \{x_1 \geq -3/8\} \subset 
\]
\[
B^{LU}(i,j) \cap \{x_1 \geq -3/8\} = A^{LU}(i,j).
\] 
Finally, if $A^{LU}(i,j)$ is as in (1), then the only way that there exists $m$ with $i<m<j$ is when $\tilde{L}$ has two crossings, i.e. Type (2Cr), above $N(e^1_U)$, and $i = k, m=k+1, j=k+2$.   Since sheets $k+1$ and $k+2$ also cross in this case, $A^{LU}(k+1,k+2)$ is also as in (1), and we see that
\[
A^{LU}(k,k+1) \cap A^{LU}(k+1,k+2)  \subset (B^{LU}(i,m) \cap B^{LU}(m,j)) \cap \{x_1 \geq 1/4\} \subset 
\]
\[
B^{LU}(i,j) \cap \{x_1 \geq 1/4\} = A^{LU}(i,j).
\] 
\end{proof}

\begin{figure}

\quad

\quad

\labellist
\small
%\pinlabel $1/4$ [t] at 338 21
\pinlabel $c_{m,j}$ [l] at 252 152
\pinlabel $c_{i,m}$ [l] at 156 56
\pinlabel $c_{i,j}$ [l] at 204 104
\pinlabel $x_1=1/4$ [b] at 0 198
%\pinlabel $I_{2,3}$ [r] at 360 208
\endlabellist
\centerline{ \includegraphics[scale=.6]{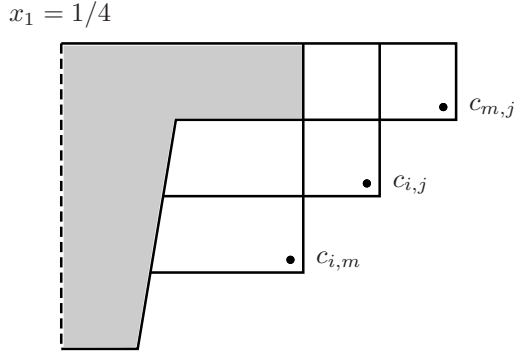} }

%%\centerline{ \includegraphics[scale=.8]{images/SubdivideSq} } %Dan had already commented out this line

\caption{The regions $B^{LU}(i,m)$, $B^{LU}(m,j)$, and $B^{LU}(i,j)$ all agree to the left of $x_1 =1/4$.  There bordering segments are pictured, with $B^{LU}(i,m) \cap B^{LU}(m,j)$ shaded.  }
\label{fig:BLUijPf}
\end{figure}

\begin{proposition} \label{prop:bLUbRD}  Any $b$-tree is either a $b^{LU}$-tree or a $b^{RD}$-tree.
\end{proposition}

\begin{proof}
Use induction on the number of $Y_0$'s and Lemma \ref{lem:BLUij}.  At the inductive step, note that the regions $B^{LU}(i,j)$ and $B^{RD}(i',j')$ are all disjoint except when $(i,j) = (i',j')$. 
\end{proof}

\subsubsection{Topological preliminary}

Let $D$  be a (region diffeomorphic to) a  closed disk in $\mathbb{R}^2$.  Suppose that we are given a collection of points $\{p_1, \ldots, p_M\} \subset \partial D$ and a collection of closed intervals $\{I_{i,j} \, | \, 1 \leq i < j \leq n\}$ with $I_{i,j} \subset \partial D$.  In addition, suppose that
\begin{itemize}
\item the points $p_m$ and intervals $I_{i,j}$ are all pairwise disjoint from one another, and
\item each point $p_m$ is assigned an {\bf upper index} $i(p_m)$ and a {\bf lower index} $j(p_m)$ that are integers with $1 \leq i(p_m) < j(p_m) \leq n$.
\end{itemize}
We refer to such a triple $(D, \{p_m\}, \{I_{i,j}\})$ as a {\bf disk datum}.

\begin{definition} \label{def:genbmfd} A {\bf generalized $b$-manifold} for the disk datum $(D, \{p_m\}, \{I_{i,j}\})$ is a collection $A = \bigsqcup_{r=1}^{n-1} A_{r}$ of (embedded) paths in $D$ with the following properties:
\begin{enumerate}

%\item The paths in $A$ all intersect transversally. These intersections are all disjoint from other paths in $A$ and are disjoint from the endpoints...  
\item Each path $\gamma \in A_r$ has upper and lower indices, $1 \leq i(\gamma) < j(\gamma) \leq n$, satisfying $r =j(\gamma)- i(\gamma)$.
%such that $j(\gamma_1) = i(\gamma_2)$, then $\gamma_1$ and $\gamma_2$ are transverse and are disjoint from one another at their boundary points.  

\item For any $\gamma \in A$, the initial point of $\gamma$ belongs to $I_{i(\gamma), j(\gamma)}$.

%The paths in $A$ all intersect transversally. These intersections are all disjoint from other paths in $A$ and are disjoint from the endpoints...  

%\item Paths in $A_1$ are in bijection with those points $p_m$, with $1 =j(p_m)- i(p_m)$.  The path, $\gamma_m$, corresponding to $p_m$ has $\left(i(\gamma_m), j(\gamma_m)\right) = \left(i(p_m), j(p_m)\right)$.  The endpoint of $\gamma_m$ is at $p_m$.
\item For $r \geq 1$, paths in $A_r$ are in bijection with (i) points $p_m$, with $r =j(p_m)- i(p_m)$, and (ii) triples $(\gamma_1, \gamma_2, x)$ where $\gamma_1 \in A_{r_1}$, $\gamma_2 \in A_{r_2}$, $x \in \gamma_1 \cap \gamma_2$, that satisfy $r_1 +r_2 = r$, and $j(\gamma_1) = i(\gamma_2)$.  

\item A path $\gamma \in A_r$ corresponding to a point $p_m$ as in (3) (i) has its endpoint at $p_m$ and satisfies $\left(i(\gamma), j(\gamma)\right) = \left(i(p_m), j(p_m)\right)$.  

\item A path $\gamma \in A_r$ corresponding to a triple $(\gamma_1,\gamma_2,x)$ as in (3) (ii) has its endpoint at $x$ and satisfies $(i(\gamma), j(\gamma)) = (i(\gamma_1), j(\gamma_2))$.

\item All intersections between distinct $\gamma_1, \gamma_2 \in A$ are transverse.  Moreover, the only case when intersections can occur at endpoints of a path, or when three paths meet at a common point is as specified in (5).  There are no points where four or more paths intersect.

\end{enumerate}
\end{definition}

See Figure \ref{fig:GenBMfld} for an example of a generalized $b$-manifold.

For each path $\gamma \in A$, we assign a {\bf word}, $w(\gamma)$, that is a formal product of points from $\{p_1, \ldots, p_M\}$, and a {\bf sign}, $\sigma(\gamma) \in \{+1,-1\}$, as follows.  
\begin{itemize}
\item For $\gamma$ corresponding to a point $p_m$ as in (3) (i),  let $w(\gamma) = p_m$ and $\sigma(\gamma)=+1$.  
\item For $\gamma$ corresponding to a triple $(\gamma_1, \gamma_2, x)$ as in (3) (ii), let  $w(\gamma) = w(\gamma_1)\cdot w(\gamma_2)$ and $\iota(\gamma_1,\gamma_2,x) \sigma(\gamma_1)\sigma(\gamma_2)$.  
Here, $\iota(\gamma_1,\gamma_2,x)$ is the intersection sign of $\gamma_1 \cap \gamma_2$ at $x$.  (We assume an orientation on $D$ is fixed.)
\end{itemize}
For $1 \leq i < j \leq n$, we define an element $ \partial_A I_{i,j} \in  \Z \langle p_1, \ldots, p_M \rangle$ in the free associative $\Z$-algebra generated by $\{p_1, \ldots, p_M\}$ by
\begin{equation} \label{eq:pbIij}
\partial_A I_{i,j} = \sum_{\gamma} \sigma(\gamma)w(\gamma)
\end{equation}
where the sum is over those $\gamma \in A$ that begin on $I_{i,j}.$
(We introduce here a $\Z$-algebra instead of a $\Z_2$-algebra because we hope to generalize our main DGA computations to $\Z$-coefficients in a future paper. For this article, however, the reader can think of $ \partial_A I_{i,j} \in  \Z_2 \langle p_1, \ldots, p_M \rangle.$)

\begin{figure}

\quad

\quad

\labellist
\small
%\pinlabel $1/4$ [t] at 338 21
\pinlabel $a$ [tr] at 28 12
\pinlabel $b$ [tl] at 132 12
\pinlabel $I_{1,2}$ [bl] at 148 134
\pinlabel $I_{1,3}$ [b] at 80 164
\pinlabel $I_{2,3}$ [br] at 12 134
\pinlabel $\gamma_2$ [bl] at 103 43
\pinlabel $\gamma_1$ [br] at 57 43
\pinlabel $\gamma_3$ [r] at 80 118

\endlabellist
\centerline{ \includegraphics[scale=.6]{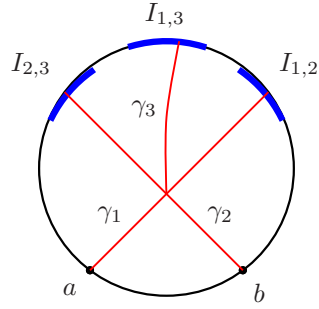} }

%%\centerline{ \includegraphics[scale=.8]{images/SubdivideSq} } %Dan had already commented out this line

\caption{A generalized $b$-manifold, $A$, for the disk datum $(D, \{a,b\}, \{I_{1,2}, I_{1,3}, I_{2,3}\})$ with $(i(a), j(a)) = (1,2)$ and $(i(b),j(b)) = (2,3)$.  Here, $A = A_1 \sqcup A_2$ with $A_1 = \{\gamma_1, \gamma_2\}$ and $A_2 = \{\gamma_3\}$.  We have $\partial_A I_{1,2} = a$, $\partial_AI_{2,3} = b$, and $\partial_A I_{1,3} = a b$. }
\label{fig:GenBMfld}
\end{figure}

\medskip

\noindent {\bf Main Example.} 
Consider a Type (1)-(12) square, $e^2_\alpha$, such that $\tilde{L}$ has $n$ sheets above $N(e^2_\alpha)$.  We let $D^{LU} \subset [-1,1] \times [-1,1]$ be obtained from $\bigcup_{1 \leq i < j \leq n} B^{LU}(i,j)$ by removing the small squares $Sq_{i,j} := (\beta^U_{i,j}-\e, \beta^U_{i,j}+\e) \times (\beta^R_{i,j}-\e, \beta^R_{i,j}+\e)$  that contain the $c_{i,j}$.  Let $I_{i,j} \subset \partial D^{LU}$ be the upper and left sides of the closed square $\overline{Sq_{i,j}}$.  By a slight abuse of notation, for $i<j$, we use $b^L_{i,j}$ and $b^U_{i,j}$ to denote the unique intersection point of the stable manifolds of these critical points with the boundary of $\widehat{N}(e^2_\alpha)$.  (Note that $b^L_{i,j}$ is only defined for those $i<j$ such that $S_i$ sits above $S_j$ along the upper half of $e^1_L$.)  The triple $(D^{LU}, \{b^L_{i,j}, b^U_{i,j}\}, \{I_{i,j}\})$ is a disk datum, where the upper and lower indices of the points $\{b^L_{i,j}, b^U_{i,j}\}$ are given by their subscripts.  See Figure \ref{fig:DiskDLU}.

\begin{figure}

\quad

\quad

\quad

\labellist
\small
%\pinlabel $1/4$ [t] at 338 21
\pinlabel $c_{1,2}$ [tl] at 254 64
\pinlabel $c_{1,3}$ [tl] at 318 128
\pinlabel $c_{2,3}$ [tl] at 382 192
\pinlabel $I_{1,2}$ [br] at 236 86
\pinlabel $I_{1,3}$ [br] at 306 156
\pinlabel $I_{2,3}$ [r] at 360 208
\pinlabel $b^U_{2,3}$ [b] at 376 254
\pinlabel $b^U_{1,3}$ [b] at 312 254
\pinlabel $b^U_{1,2}$ [b] at 248 254
\pinlabel $b^L_{1,3}$ [r] at -2 72
\pinlabel $b^L_{1,2}$ [r] at -2 136
\endlabellist
\centerline{ \includegraphics[scale=.6]{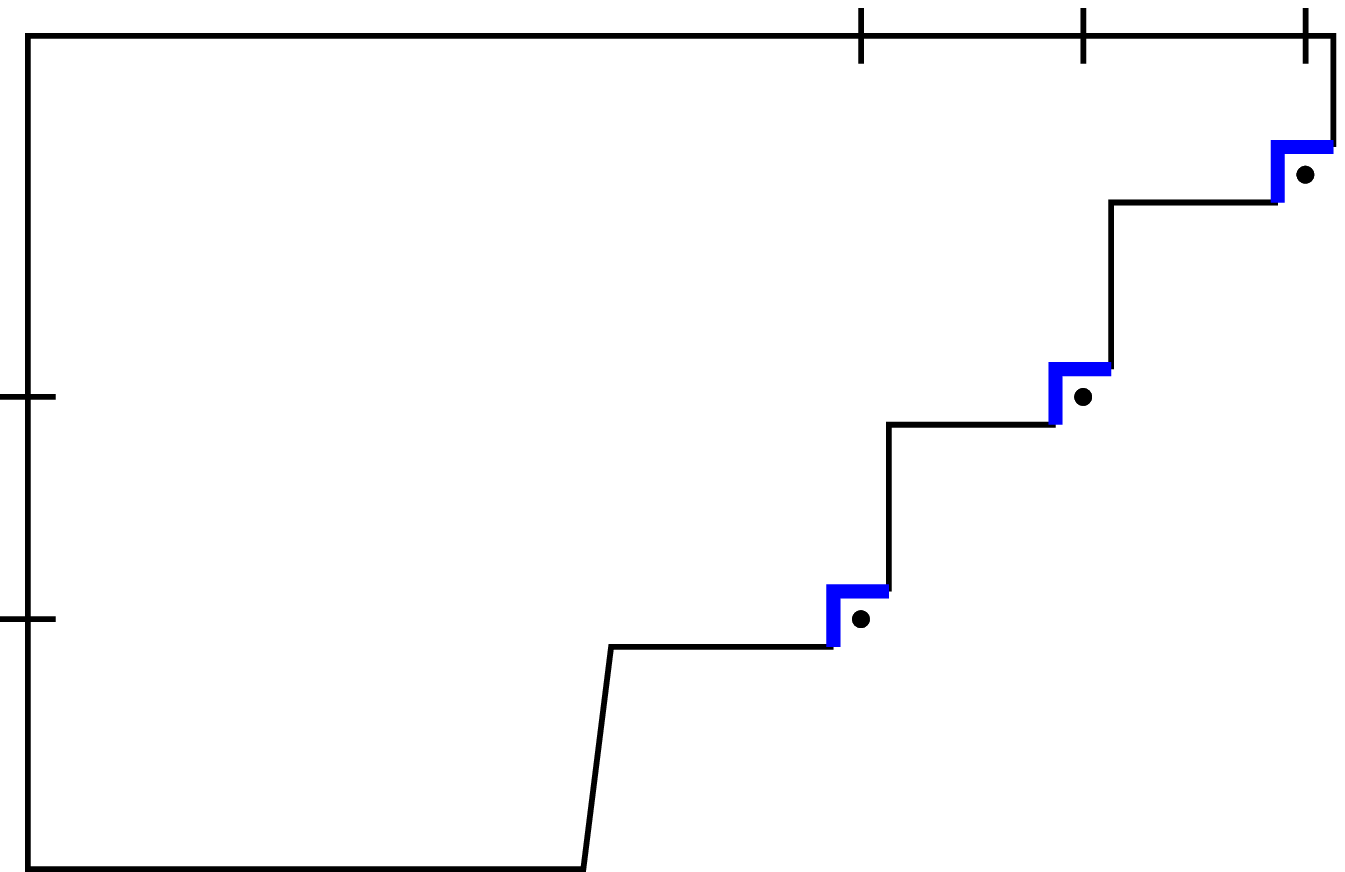} }

%%\centerline{ \includegraphics[scale=.8]{images/SubdivideSq} } %Dan had already commented out this line

\caption{The disk datum $(D^{LU}, \{b^L_{i,j}, b^U_{i,j}\}, \{I_{i,j}\})$.  In the pictured case the number of sheets is $n = 3$, and sheets $2$ and $3$ cross one another above $e^1_U$.  Thus, $b^L_{2,3}$ does not exist.}
\label{fig:DiskDLU}
\end{figure}

We will define a generalized $b$-manifold for the disk datum $(D^{LU}, \{b^L_{i,j}, b^U_{i,j}\}, \{I_{i,j}\})$.  Suppose that  $\widehat{\gamma}$ is the top branch of a GFT that is a $b^{LU}$-tree beginning at some $c_{i,j}$, and let $\gamma = \widehat{\gamma} \cap D^{LU}$.  In addition, define upper and lower indices $1 \leq i(\gamma) < j(\gamma) \leq n$, so that $\gamma$ is a trajectory for $-\nabla F_{i(\gamma),j(\gamma)}$, i.e. $S_{i(\gamma)}$ and $S_{j(\gamma)}$ are the sheets of $\widetilde{L}$ that correspond to $\widehat{\gamma}$ in the GFT that it belongs to.   

% that are obtained from $\widehat{\gamma}$ is the top branch of some GFT that is a $b$-tree beginning at $c_{i,j}$.  
\begin{proposition}  \label{prop:MainGenB}
Let $A = \bigsqcup_{r} A_r$ denote the collection of such paths $\gamma$ with $i(\gamma)$ and $j(\gamma)$ as above.  Then, $A$ is a generalized $b$-manifold for the disk datum $(D^{LU}, \{b^L_{i,j}, b^U_{i,j}\}, \{I_{i,j}\})$.  Moreover, 
\[
\overline{\partial_A I_{i,j}} = \sum_{\Gamma} w(\Gamma)
\]
where the sum on the right is over $b^{LU}$-trees that start at $c_{i,j}$, and $\overline{\partial_A I_{i,j}}$ denotes $\partial_AI_{i,j}$ with coefficients reduced mod $2$.

An analogous result holds for $b^{RD}$-trees.
\end{proposition}
(Recall that the word, $w(\Gamma)$, of a GFT, $\Gamma$, was defined in Section \ref{sec:GFTPFT}.)

\begin{proof}
The condition (1) in Definition \ref{def:genbmfd} is clear.  Condition (6) is one of the transversality conditions that is implied by $\tilde{L}$ being $1$-regular. 

  The top most branch $\widehat{\gamma}$ of a $b^{LU}$-tree, $\Gamma$, must start at some $c_{i,j}$.  Since $c_{i,j} \in Sq_{i,j}$, and $\gamma \subset B^{LU}(i,j)$ (by Lemma \ref{lem:BLUij}), we see that, as $t$ increases from $-\infty$, $\widehat{\gamma}$ must enter $D^{LU}$ along a unique point in $I_{i,j}$.  [Uniqueness follows since $-\nabla F_{i,j}$ points outward along the entire boundary of $Sq_{i,j}$.]  Thus, (2) holds.
	
	Now, as $t \rightarrow +\infty$, $\widehat{\gamma}$ must either (a) limit to a critical point $b^{L}_{i,j}$ or $b^{U}_{i,j}$ or (b) end at a $Y_0$ that occurs at the intersection of two edges of $\Gamma$, $\gamma_1$ and $\gamma_2$, with $j(\gamma_1) = i(\gamma_2)$.  Clearly, those $\gamma$ satisfying (a) are in bijection with $\{b^L_{i,j}, b^U_{i,j}\}$ as in (3) (i) and satisfy (4).  When $\gamma$ satisfies (b), note that, when extended to allow $t \rightarrow -\infty$, the branches $\gamma_1$ and $\gamma_2$ are themselves the top branches of unique $b^{LU}$-trees.  Thus, the $\gamma$ satisfying (b) are in bijection with intersection points as in (3) (ii) and satisfy (5).

The equality $\overline{\partial_b I_{i,j}} = \partial_b c_{i,j}$ is clear since the word associated to a path $\gamma \in A$ that is a top branch of the $b^{LU}$-tree $\Gamma$ is precisely the word $w(\Gamma)$.  [This is verified by an obvious induction.]

%\dr{Cite something about $b$-trees only having $Y_0$-vertices.}
%By Lemma \ref{lem:BLUij}, $\widehat{\gamma}$ has its image in $B^{LU}(i,j)$ except possibly if $\widehat{\gamma}$ limits to $b^{L}_{i,j}$ or $b^U_{i,j}$ as $t\rightarrow \infty$, in which case $\widehat{\gamma}$ exits $B^{LU}_{i,j}$ at the point on $\partial D^{LU}$ that we have also labelled $b^L_{i,j}$ or $b^U_{i,j}$.  Now, $B^{LU}_{i,j}$
\end{proof}

The following key proposition will allow us to compute $\partial_b c_{i,j}$ without requiring detailed knowledge of the locations of $b^X$-lines.

\begin{proposition} \label{prop:dAind}
For any disk datum $(D, \{p_m\}, \{I_{i,j}\})$, the sums $\partial_A I_{i,j}$ are independent of the choice of generalized $b$-manifold, $A$, associated to $(D, \{p_m\}, \{I_{i,j}\})$. 
\end{proposition}

\begin{proof}
We prove the following statement by induction on $R$ with $R$ decreasing from $n$ to $1$:

\medskip

\noindent {\bf Inductive Statement:}  Let $A = \sqcup_{r=1}^{n-1} A_r$ and $A' = \sqcup_{r=1}^{n-1} A'_r$ be two generalized $b$-manifolds for $(D, \{p_m\}, \{I_{i,j}\})$ such that 
\[
\mbox{$A_r = A'_r$ for all $1 \leq r < R$.}
\]  Then, for all $1 \leq i < j \leq n$,
\[
\mbox{$\partial_A I_{i,j} = \partial_{A'} I_{i,j}$}
\]
 where $\partial_A I_{i,j}$ and $\partial_{A'} I_{i,j}$ are elements of $\Z\langle p_1, \ldots, p_M \rangle$ as defined in (\ref{eq:pbIij}).

\medskip  

The case $R=n$ is tautological.  For the inductive step, suppose that the statement is known for larger values of $R$, 
and suppose $A_r = A'_r$ for $r< R$.  Notice that paths in $A_R$ and $A'_R$ are in bijection, and have the same endpoints and indices.  [The number of paths in $A_R$ as well as  their endpoints and indices are determined by the $\{p_m\}$ and intersections of paths in $\bigsqcup_{r<R} A_r$.]   Thus, without loss of generality, we may assume that $A_R$ and $A'_R$ are identical except for a single pair of paths $\gamma \in A_R$ and $\gamma' \in A_R$ with the same endpoint and equal indices.  

Since $D$ is a topological disk, and $\gamma$ and $\gamma'$ are smoothly embedded arcs with the same endpoints, they are isotopic via an isotopy of arcs $\gamma^t$, $0 \leq t \leq 1$, with $\gamma^0 = \gamma$ and $\gamma^1 = \gamma'$ that preserves endpoints.  

Set 
\[
X = \left( \bigsqcup_{r\leq R} A_r \right)\setminus\{\gamma\}.  
%\bigcup \left \{ \beta \in A \, | \, \mbox{ $j(\beta) = i(\gamma)$ or $j(\gamma) = i(\beta)$} \right\}.
\]
By standard transversality results, after possibly modifying the isotopy $\gamma^t$, we may assume that there exist $0 = t_0 < t_1 < \ldots < t_T = 1$ such that, the $\gamma^{t_{s}}$ are transverse to all paths in $X$, and  for $1 \leq s \leq T$ either:
\begin{enumerate}
\item[{\bf O.}]  The paths in $X \cup \{\gamma^{t_{s}}\}$ are obtained from the corresponding paths in $X \cup \{\gamma^{t_{s-1}}\}$ by applying an orientation preserving diffeomorphism $\varphi: D \stackrel{\cong}{\rightarrow} D$.  
\end{enumerate}
Or, when $t_{s-1} < t < t_{s}$, $\gamma^t$ remains fixed except in a small disk $D' \subset D$ where one of the following four moves occurs:
\begin{enumerate}
\item[{\bf I.}]  %When $t_{i-1} < t < t_{i}$, $\gamma^t$ remains constant except in a small 
The disk $D' \subset D$ intersects a single path $\beta \in X$ in a properly embedded arc.  For some $t_{s-1} < t < t_{s}$,  $\gamma^t$ has a tangential intersection with $\beta$ that results in the creation or cancellation of two transverse intersections.    

\item[{\bf II.}] Suppose $\gamma$ has its endpoint at a transverse crossing point, $x$, of paths $\beta_1, \beta_2 \in X$.  The disk $D' \subset D$ contains $x$, and for some $t_{s-1} < t < t_{s}$,  $\gamma^t$ becomes tangent to either $\beta_1$ or $\beta_2$ at $x$ resulting in the addition or removal of a single transverse intersection between $\gamma^t$ and $\beta_1$ or $\beta_2$. 

\item[{\bf III.}] The disk $D' \subset D$ contains a crossing, $x$, of two paths $\beta_1, \beta_2 \in X$. Note that if $\beta_1$ and $\beta_2$ have a common upper and lower index, then $x$ may be the endpoint of a third path $\beta_3 \in X$.  
%As $t$ increases from $t_{s-1}$ to $t_s$, $\gamma^t$ passes $x$, so that both
Both  $\gamma^{t_{s-1}}$ and $\gamma^{t_s}$ have a single intersection point  with $\beta_1$ in $D'$, and these intersection points lie on opposite sides of $x$.  The same holds for intersections of $\gamma^{t_{s-1}}$ and $\gamma^{t_s}$ with $\beta_2$.  If $\beta_3$ exists, then precisely one of $\gamma^{t_{s-1}}$ and $\gamma^{t_s}$ intersects $\beta_3$ in $D'$, and this intersection is unique.  

\item[{\bf IV.}] The disk $D' \subset D$ contains on its boundary part of the interval $I_{i,j}$  where $\gamma$ has its initial point.  For $t_{s-1} < t < t_{s}$ the endpoint of $\gamma$ passes the location of the endpoint of another path $\beta \in X$ whose endpoint also lies on $I_{i,j}$.  This creates or removes a crossing between $\gamma$ and $\beta$.

\end{enumerate} 

See Figure \ref{fig:GenBMoves} for a depiction of Moves I-IV.

\begin{figure}

\quad

\labellist
\small
%\pinlabel $1/4$ [t] at 338 21
\pinlabel \textbf{I.} [r] at -6 300
\pinlabel \textbf{II.} [r] at 612 300
\pinlabel \textbf{III.} [r] at -6 80
\pinlabel \textbf{IV.} [r] at 612 80

\endlabellist
\centerline{ \includegraphics[scale=.4]{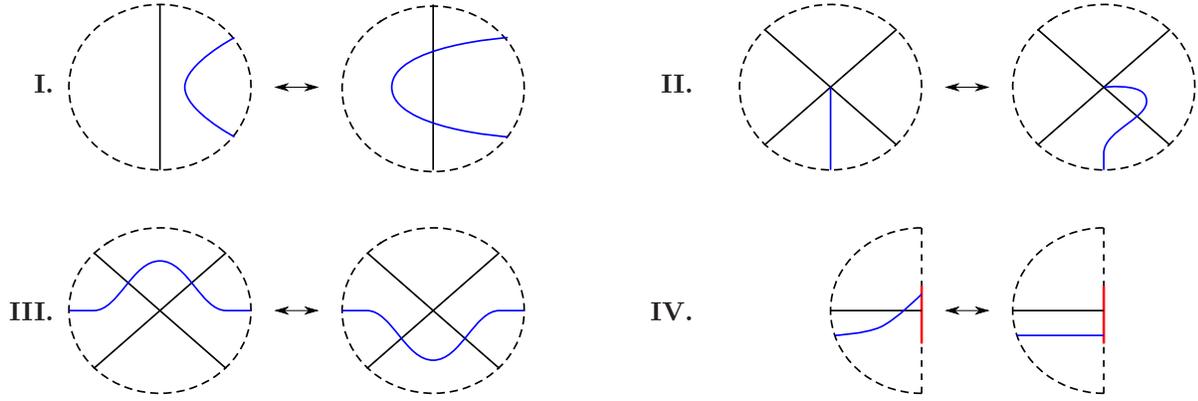} }

%%\centerline{ \includegraphics[scale=.8]{images/SubdivideSq} } %Dan had already commented out this line

\caption{Moves I-IV.  The blue path denotes $\gamma^{t}$ while the black paths belong to $X$.  In Move III, a third path from $X$ may have an endpoint at the crossing of the two black curves.  In Move IV, the red segment is part of the interval $I_{i,j} \subset \partial D$.}
\label{fig:GenBMoves}
\end{figure}

Making use of the inductive hypothesis, it suffices to establish the following:

\medskip

\noindent {\bf Claim.} For any $1 \leq s \leq T$, there exist some generalized $b$-manifolds, $C$ and $C'$ for $(D, \{p_m\}, \{I_{i,j}\})$ such that $\partial_C I_{i,j} = \partial_{C'} I_{i,j}$;
\[
\bigsqcup_{r \leq R} C_r = %\left( \bigsqcup_{r\leq R} A_R \right)\setminus\{\gamma\}  
X \bigcup \left \{ \gamma^{t_{s-1}} \right\}; \quad \mbox{and} \quad 
\bigsqcup_{r \leq R} C'_r = %\left( \bigsqcup_{r\leq R} A_R \right)\setminus\{\gamma\}   
X \bigcup \left \{ \gamma^{t_{s}} \right\}.
\] 

\medskip
We prove the claim when $\gamma^{t_{s-1}}$ and $\gamma^{t_{s}}$ are related by Moves O and I-III in a case-by-case manner.  
% and $\gamma^{t_{s}}$ appear respectively as in the diagrams on the left 

\medskip

\noindent {\bf Move O.}  Extend $X \cup \{\gamma^{t_{s-1}}\}$ to obtain a generalized $b$-manifold $C$ as follows.   Take $C_r = A_r$ for $r<R$, and $C_R = \left( A_R \setminus \{\gamma\} \right) \bigcup \left \{ \gamma^{t_{s-1}} \right\}$.  We then define $C_r$
 with $r >R$ inductively.  With all $C_{r'}$ having $r' < r$ already defined, the required number of paths in $C_{r}$ as well as their endpoints and indices are specified by $\bigsqcup_{r'<r} C_{r'}$ and $\{p_m\}$ according to Definition \ref{def:genbmfd}.  
 We then define $C_{r}$
  by choosing some collection of paths with the desired endpoints that are transverse to all paths in $\bigsqcup_{r'<r} C_{r'}$ and to one another.  
 [Note that in order for Definition \ref{def:genbmfd} to be met, any new intersection points $x \in \beta_1 \cap \beta_2$ with $\beta_1 \in C_{r}$ and $\beta_2 \in C_{r'}$, for some $r' \leq r$, and such that either $j(\beta_1)= i(\beta_2)$ or  $i(\beta_1)= j(\beta_2)$ will have to be the endpoint of some path in $C_{r''}$ with $r'' = r+r'$.  Since $r'' > r$, this will be arranged at a later step of the induction, and in particular will not require any modification of the already defined subsets $\bigsqcup_{r'\leq r} C_{r'}$.]

 The construction of $C$ is thus completed after finitely many such steps. We then arrive at $C'$ by applying the diffeomorphism $\varphi$ to each of the paths in $C$.

\medskip

In considering Moves I-IV, we assume without loss of generality that $\gamma^{t_{s-1}}$ (resp. $\gamma^{t_{s}}$) corresponds to the diagram on the left (resp. right) as depicted in Figure \ref{fig:GenBMoves}.  To simplify notation, we let $\gamma^L$ and $\gamma^R$ denote $\gamma^{t_{s-1}}$ and $\gamma^{t_s}$ as they appear on the left and right side of Moves I-IV.

\medskip

\noindent {\bf Move I.}   As in Move O., extend $X \cup \{\gamma^{L}\}$ to a generalized $b$-manifold $C$, but at the inductive step impose the additional requirement that all new paths are disjoint from $D'$.  Unless $j(\beta) = i(\gamma)$ or $i(\beta) = j(\gamma)$, $C'$ is obtained from $C$ by simply replacing $\gamma^{L}$ with $\gamma^{R}$.

Consider now the case where  $j(\beta) = i(\gamma)$ or $i(\beta) = j(\gamma)$.  Form $C'_0$ from $C$ by replacing $\gamma^{L}$ with $\gamma^{R}$.  We need to add two additional paths $\alpha_1$ and $\alpha_2$ to $C'_0$ that end at the intersection points $x_1$ and $x_2$ of $\gamma^R$ and $\beta$ that appear on the right side of Move II.  Do so by choosing a single path $\alpha$ from the right side of $\partial D'$ to $I_{i,j}$ (where $i$ and $j$ are as specified by the indices of $\gamma^R$ and $\beta$).  Then define $\alpha_1$ and $\alpha_2$ to travel from $x_1$ and $x_2$ to the initial point of $\alpha$, and then follow parallel to $\alpha$ until they reach $I_{i,j}$.  We take $\alpha_1$ and $\alpha_2$ to be close enough to one another so that their intersections with paths in $C'_0$ are are in bijection with one another with the same signs, and are located in small neighborhoods of the intersection points of $\alpha$.  Set $C'_1 = C'_0 \cap\{ \alpha_1,\alpha_2\}$.  See Figure \ref{fig:GenBPf1}.

Now, the new paths $\alpha_1$ and $\alpha_2$ may have corresponding intersection points $y_1 \in \alpha_1 \cap \nu$ and $y_2 \in \alpha_2 \cap \nu$ with some $\nu \in C'_1$, such that for $p =1,2$,  $i(\alpha_p) = j(\nu)$ or $j(\alpha_p) = i(\nu)$.  For each such pair of intersections, we connect $y_1$ and $y_2$ to the appropriate $I_{i,j}$ using a pair of sufficiently close paths, and  then denote the union of $C'_1$ with all such pairs of paths as $C'_2$.  We proceed inductively, and note that at each step the difference $j-i$ increases for all new paths.  Thus, this process terminates after less than $n$ steps, and we take $C'$ to be the resulting generalized $b$-manifold.  

To verify that $\partial_C I_{i,j} = \partial_{C'} I_{i,j}$, note that $(C\setminus\{\gamma^L\}) \subset (C'\setminus\{\gamma^R\})$.  Now, $w(\gamma^L) = w(\gamma^R)$, and  paths in $C' \setminus (C \cup \{\gamma^R\})$ that begin at $I_{i,j}$ come in pairs (as constructed in the previous paragraph),  $\tau_1$ and $\tau_2$, such that $w(\tau_1)= w(\tau_2)$ and $\sigma(\tau_1)=-\sigma(\tau_2)$ (since the two intersections of $\gamma^R$ with $\beta$ have opposite signs).  Such pairs cancel out in the sum from (\ref{eq:pbIij}) that defines $\partial_{C'} I_{i,j}$.

\begin{figure}

\quad

\labellist
\small
%\pinlabel $1/4$ [t] at 338 21
\pinlabel $\beta$ [r] at 86 142
\pinlabel $\gamma^R$ [r] at 46 84
\pinlabel $\alpha_1$ [b] at 144 88
\pinlabel $\alpha_2$ [t] at 144 72
\endlabellist
\centerline{ \includegraphics[scale=.6]{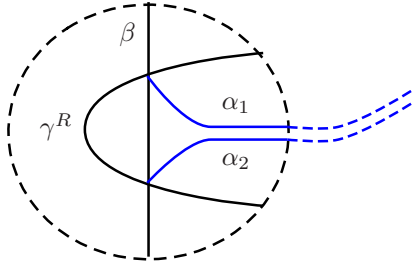} }

%%\centerline{ \includegraphics[scale=.8]{images/SubdivideSq} } %Dan had already commented out this line

\caption{The paths $\alpha_1$ and $\alpha_2$ from $C'_1$. }
\label{fig:GenBPf1}
\end{figure}

%We can choose $\alpha$ to be transverse to all paths in $C'_0$.  We then take $\alpha'$ to be a very small shift of $\alpha$ off of itself, so that intersections of $\alpha$ and $\alpha'$ with any curve $\nu \in C^'_0$ are in bijection and bound a segment of $\nu$ that is disjoint from all other curves in $C^'_0$.     and we can thicken $\alpha$ to an embedded strip, $S \cong [0,1] \times[-\delta, \delta]$, such that $[0,1] \times \{0\}$ is $\alpha$ and all intersections of the push offs $\alpha_-$ $[0,1] \times \{-\delta\}$ and $[0,1] \times \{-delta\}$

%Next, choose a   At each step ??? $R$ increases so this process does not continue indefinitely.  Moreover, each pair of paths that is added clearly have equal words, so they cancel out in $\partial_{C'} I_{i,j}$...

\medskip

\noindent {\bf Move II.}  Since the indices of the $\beta_p$, $p = 1,2$, must be of the form $(a,b)$ and $(b,c)$ for some $a<b<c$, and $(i(\gamma), j(\gamma))= (a,c)$, we have $i(\gamma) \neq j(\beta_p)$ and $j(\gamma) \neq i(\gamma)$.  Thus, we can take $C$ to be any extension of $X \cup \{\gamma^{L}\}$ with all new paths disjoint from $D'$, and set $C' = (C\setminus  \{\gamma^{L}\}) \cup \{\gamma^{R}\}$.

\medskip

\noindent {\bf Move III.} Let $\beta_1, \beta_2$ denote the two paths of $X$ that cross in $D'$.  In our figures we picture $\beta_1$ with negative slope and $\beta_2$ with positive slope.  
Label the intersection points in $D'$ as $x \in \beta_1 \cap \beta_2$, $y^{K} \in \beta_1 \cap \gamma^K$, and $z^K \in \beta_2 \cap \gamma^K$ for $K \in \{L, R\}$.  See Figure \ref{fig:GenBPf2}.

Introduce the following notation for indices:  
\[
(a,b) = (i(\beta_1), j(\beta_1)); \quad (c,d) = (i(\beta_2), j(\beta_2)); \quad \mbox{and} \quad (p,q) = (i(\gamma), j(\gamma)).
\]

Recall that there may  be an additional path $\beta_3 \in X$ that intersects $D'$ and has its endpoint at $x$.  If it exists, in $D'$, $\beta_3$ intersects precisely one of $\gamma^L$ and $\gamma^R$, and this intersection is in a single point.  Without loss of generality we assume that, within $D'$, $\beta_3$ is disjoint from $\gamma^L$ and intersects $\gamma^R$ precisely once.  [If this is not the case, then interchange the left and right side of Move III and rotate the diagrams by $180^{\circ}$.]

%Note that in the case that $a=d$ or $b=c$, any generalized $b$-manifold containing $X$ must have a path, $\beta_3$, with endpoint at $x$ (the pictured crossing of $\beta_1$ and $\beta_2$).  (Such a path may or may not be contained in $X$ depending on if $i(\beta_3)-j(\beta_3)$ is larger than $R$ or not.)  Without loss of generality we assume that, at least within $D'$, $\beta_3$ is disjoint from $\gamma$ in the left diagram of Move III, and intersects $\beta_3$ precisely once in the right diagram.  [If this is not the case, then interchange left and right side of Move III and rotate the diagrams by $180^{\circ}$.]

To begin, choose an extension of $X\bigsqcup \{\gamma^L\}$ to a generalized $b$-manifold, $C$, so that any paths starting at $x$, $y^L$, and $z^L$ (should they exist) are respectively disjoint from $\gamma^L$, $\beta_2$, and $\beta_1$.  Denote these paths by $\alpha_x$, $\alpha_y$, $\alpha_z$, and note that it may be the case that $\alpha_x = \beta_3$.  See Figure \ref{fig:GenBPf2}.

\begin{figure}

\labellist
\small
%\pinlabel $1/4$ [t] at 338 21
\pinlabel $\beta_1$ [br] at 22 140
\pinlabel $\beta_2$ [bl] at 156 140
\pinlabel $\gamma^L$ [r] at -2 82
\pinlabel $y^L$ [b] at 57 114
\pinlabel $z^L$ [b] at 120 114
\pinlabel $x$ [t] at 88 74
\pinlabel $y^R$ [t] at 384 44
\pinlabel $z^R$ [t] at 322 44
\pinlabel $x$ [b] at 352 88
\pinlabel $\gamma^R$ [l] at 444 80
\pinlabel $\beta_1$ [br] at 284 140
\pinlabel $\beta_2$ [bl] at 420 140

\endlabellist
\centerline{ \includegraphics[scale=.6]{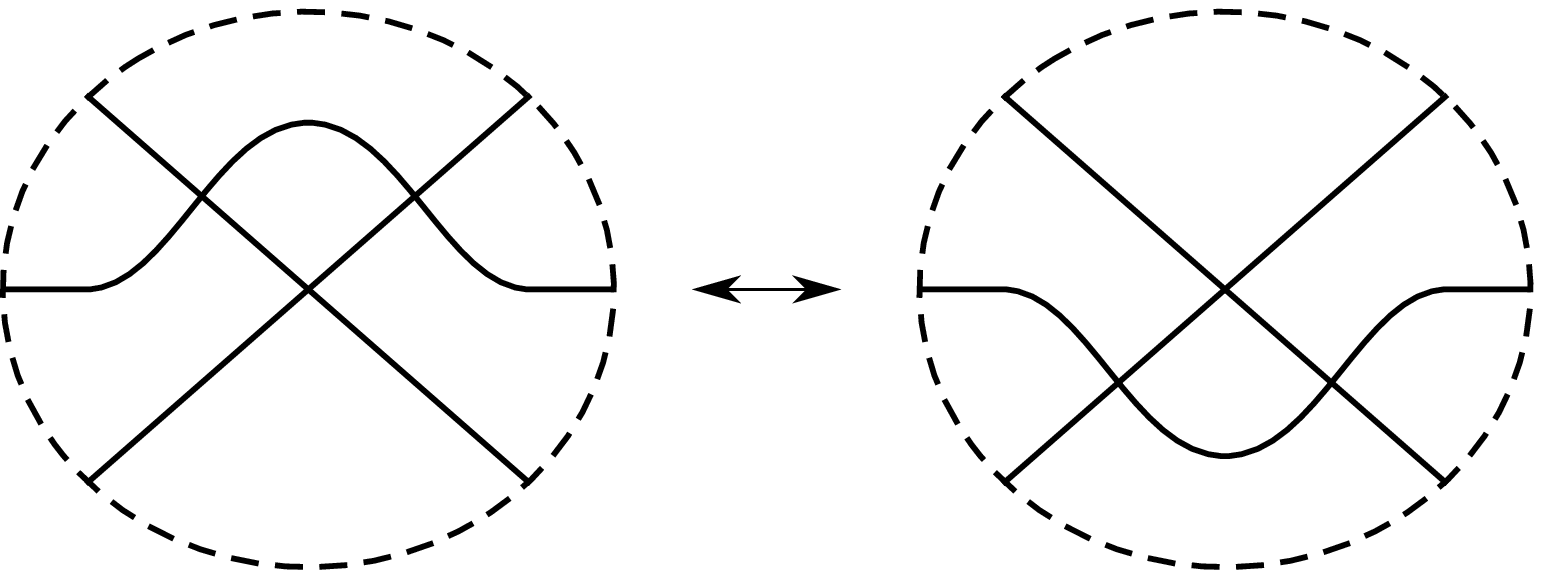} }

\quad

\quad

\labellist
\small
%\pinlabel $1/4$ [t] at 338 21
%\pinlabel $E_{p,q}$ [r] at -2 102
%\pinlabel $E_{a,b}$ [br] at 40 178
%\pinlabel $E_{c,d}$ [bl] at 172 178
\pinlabel $\alpha_x$ [tr] at 6 50
\pinlabel $\alpha_y$ [b] at 104 198
\pinlabel $\alpha_z$ [tl] at 202 60
\pinlabel $\alpha_x$ [tr] at 410 50
\pinlabel $\alpha'_y$ [b] at 508 198
\pinlabel $\alpha'_z$ [tl] at 606 60
%\pinlabel $[E_{a,b},E_{c,d}]$ [tr] at 6 50
%\pinlabel $[E_{p,q},E_{a,b}]$ [b] at 104 198
%\pinlabel $[E_{c,d},E_{p,q}]$ [tl] at 202 60
\pinlabel $u$ [b] at 452 92
\pinlabel $v$ [b] at 546 142
\pinlabel $w$ [b] at 572 48
\endlabellist
\centerline{ \includegraphics[scale=.6]{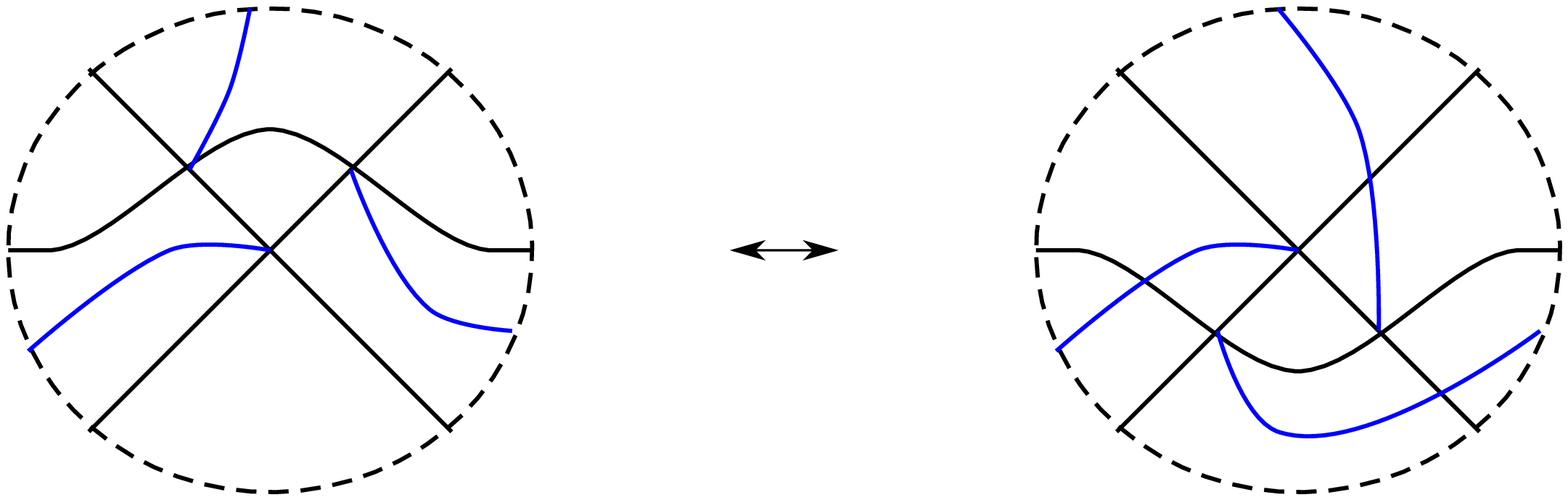} }

%%\centerline{ \includegraphics[scale=.8]{images/SubdivideSq} } %Dan had already commented out this line

\caption{(top) The notations used for Move III.  (bottom left) The paths $\alpha_x$, $\alpha_y$, and $\alpha_z$ are pictured in blue.  (bottom right) The paths $\alpha_x$, $\alpha_y'$, and $\alpha_z'$ along with intersection points $u, v$, and $w$. }
\label{fig:GenBPf2}
\end{figure}

To keep track of indices, it is convenient to assign a coefficient to each path $\nu \in C$, $F_\nu$, that is the strictly upper triangular matrix $F_{\nu} = \sigma(\nu)E_{i(\nu), j(\nu)} \in \mathit{Mat}(n, \Z)$  whose $(i(\nu), j(\nu))$-entry is $\sigma(\nu)$, and with other entries equal to $0$.  We then make the observation that if two paths $\nu, \nu' \in C$  
%coefficients $E_{i,j}$ and $E_{k,l}$ 
intersect at a point $w$ then Definition \ref{def:genbmfd} specifies that the coefficient of the path in $C$ that ends at $w$ is 
\[
\iota(\nu,\nu',w)[ F_\nu, F_{\nu'}] = \iota(\nu',\nu,w)[ F_\nu', F_{\nu}]
%[E_{i,j}, E_{k,l}] := \delta_{j,k} E_{i,l} - \delta_{l,i} E_{k,j},
\]
where if $[F_\nu, F_{\nu'}] =0$,
%$[E_{i,j}, E_{k,l}]=0$, 
then no such path exists in $C$.  Thus, we interpret a coefficient of $0$ to indicate that the path is not part of $C$.  

Above, $[\cdot,\cdot]$ is the commutator of matrices in $\mathit{Mat}(n, \Z)$.  Up to a sign, $[F_\nu, F_{\nu'}]$ has the form
\[
[E_{i,j}, E_{k,l}] = \delta_{j,k} E_{i,l} - \delta_{l,i} E_{k,j}.
\]
 Note that since the subscripts satisfy $i<j$ and $k<l$, only one of the terms on the right can be non-zero.  This implies the following:

\medskip

\noindent {\bf Observation 1.} Any iterated commutator of matrices of the form $E_{i,j}$ with $i<j$ again retains this form or is $0$. 

\medskip

 Note also that because a product of the form $E_{i,j} U E_{i,j}$ must be $0$ when $i<j$ and $U$ is upper triangular, we must have:

\medskip

\noindent {\bf Observation 2.} If the same $E_{i,j}$ appears more than once in such an iterated commutator, then the result is $0$.

\medskip

Abusing notation by writing $\nu$ for its coefficient $F_{\nu}$, the coefficients of $\alpha_x, \alpha_y,$ and $\alpha_z$ are respectively 
\[
\alpha_x= \iota(\beta_2,\beta_1, x)[\beta_2,\beta_1], \quad \alpha_y= \iota(\beta_1,\gamma^L, y^L)[\beta_1,\gamma^L], \quad \alpha_z= \iota(\gamma^L, \beta_2, z^L)[\gamma^L, \beta_2]. 
\]
%$[E_{a,b}, E_{c,d}]$, $[E_{p,q}, E_{a,b}]$, and $[E_{c,d}, E_{p,q}]$ 

%as indicated in Figure \ref{fig:GenBPf2}.

To construct $C'$, start with $C'_0$ as follows.  Take $\left( C \setminus\{\gamma^L\} \right) \cup \{\gamma^R\}$, but replace the paths $\alpha_y$, $\alpha_z$ (if they exist) with new paths $\alpha_y'$ and $\alpha_z'$ that, outside of $D'$, agree respectively with $\alpha_y$ and $\alpha_z$.  [Since the point $x$ remains unchanged, we can retain the path $\alpha_x$, which is the only one of the three that may belong to $X$.]  We can arrange  that $\alpha_x$, $\alpha_y'$, $\alpha_z'$, respectively have unique intersection points in $D'$ with $\gamma^R$, $\beta_2$, and $\beta_1$.  Denote these intersection points as 
\[
u \in \gamma^R \cap \alpha_x; \quad v \in \beta_2 \cap \alpha_y'; \quad \mbox{and} \quad w \in \beta_1 \cap \alpha_z'.
\]

Now, it may be the case that additional paths starting at $u,v$, and/or $w$ must be added to $C'_0$ to form a generalized $b$-manifold.  Whether such paths are necessary, and what their indices should be are determined by the coefficients associated to these intersection points.  They are
\[
\begin{array}{rcr} \mbox{Coeff. of $u$:} & & \iota(\alpha_x,\gamma^L,u)\iota(\beta_2,\beta_1,x)  \left[[\beta_2,\beta_1], \gamma^L\right]; \\
\mbox{Coeff. of $v$:} & & \iota(\alpha'_y,\beta_2,v)\iota(\beta_1,\gamma^L,y^L)   \left[[\beta_1,\gamma^L], \beta_2\right]; \\
\mbox{Coeff. of $w$:} & & \iota(\alpha'_z,\beta_1,w)\iota(\gamma^L,\beta_2,z^L)  \left[[\gamma^L, \beta_2], \beta_1\right]. \\ 
\end{array}
\]
%\[
%\begin{array}{rcc} \mbox{Coeff. of $u$:} & & \left[[E_{a,b},E_{c,d}], E_{p,q}\right]; \\
%\mbox{Coeff. of $v$:} & & \left[[E_{p,q},E_{a,b}], E_{c,d}\right]; \\
%\mbox{Coeff. of $w$:} & & \left[[E_{c,d},E_{p,q}], E_{a,d}\right]. 
%\end{array}
%\]
[Note that at the level of coefficients, $\alpha_y=\alpha'_y$, $\alpha_z=\alpha'_z$, and $\gamma^L=\gamma^R$.]
Moreover, the initial signs are all equal
\[
\iota(\alpha_x,\gamma^L,u)\iota(\beta_2,\beta_1,x) = \iota(\alpha'_y,\beta_2,v)\iota(\beta_1,\gamma^L,y^L) = \iota(\alpha'_z,\beta_1,w)\iota(\gamma^L,\beta_2,z^L).
\]
[To verify this, note that $\alpha_x, \alpha'_y,$ and $\alpha'_z$ are all oriented toward their  endpoints which appear in $D'$.  Moreover, reversing the orientation of any one of $\beta_1, \beta_2, \gamma^L$ changes all three quantities by a sign.  Therefore, it is enough to check the equality for a single choice of orientation of $\beta_1, \beta_2, \gamma^L$.  If we orient these three curves to the right as pictured, the three products become
\[
(-1)(-1)=(+1)(+1)=(+1)(+1). \quad]
\]

Therefore, by the Jacobi identity, the sum of these three coefficients is 
\[
u+v+w = 0 \in \mathit{Mat}(n, \Z).
\]  Since each of them is either $0$ or a single matrix of the form $\pm E_{i,j}$ with $i<j$ (by Observation 1), we have two possibilities:
\begin{enumerate}
\item  {\bf All three coefficients are $0$.}  This means that no additional paths are necessary, so $C' = C'_0$ is a generalized $b$-manifold with the desired properties.  
\item  {\bf One coefficient is $0$ and the other two are negatives of one another.}  This means that we need to add two paths starting at two of the points $u,v,$ and $w$ and with the same indices but opposite signs.  
We choose these paths so that outside of $D'$ they are small shift of a single path from the $\partial D'$ to the appropriate $I_{i,j}$.   Note that the coefficients of any intersection points in $D'$ between these two new paths and existing paths  must be $0$ because of the Observation 2.
Then, proceeding as in Move I, we construct $C'$ inductively so that all paths not corresponding to paths in $C$ cancel in pairs when computing $\partial_{C'}I_{i,j}$.  
%no additional paths are necessary, so $C' = C'_0$ is a generalized $b$-manifold with the desired properties.
\end{enumerate}

%\noindent {\bf Case: $p \notin \{b,d\}$ and $q \notin \{a,c\}$.}  Then, no paths need to begin at the intersections of $\gamma^L$ or $\gamma^R$ with $\beta_a$ for $a=1,2,3$.  Simply extend $X \cup \{\gamma^L\}$ to form $C$, and then put $C' = (C \setminus \{\gamma^L\}) \cup \{\gamma^R\}$. 

%\medskip

%\noindent {\bf Case: $p \in \{b,d\}$ and $q \notin \{a,c\}$.}  Without loss of generality we can assume $p = b$, since we can reflect across a vertical axis to interchange $\beta_1$ and $\beta_2$.

%medskip

%\noindent {\bf Case: $p \notin \{b,d\}$ and $q \in \{a,c\}$.}

%\medskip

%\noindent {\bf Case: $p \in \{b,d\}$ and $q \in \{a,c\}$.}

\medskip

\noindent {\bf Move IV.}  Since $(i(\gamma), j(\gamma))= (i(\beta), j(\beta))$, Move IV may be treated as in Move II.

\end{proof}

\begin{proof}[Proof of Theorem \ref{thm:112btrees}]
From Propositions \ref{prop:bLUbRD}, \ref{prop:MainGenB}, and \ref{prop:dAind} we have
\[
\partial_b c_{i,j}= \overline{\partial_{A^{LU}}I_{i,j}} + \overline{\partial_{A^{RD}} I_{i,j}} + X
\]
where $A^{LU}$ (resp. $A^{RD}$) is any generalized $b$-manifold associated to the disk datum $(D^{LU}, \{b^L_{i,j}, b^U_{i,j}\}, \{I_{i,j}\})$  (resp. $(D^{RD}, \{b^R_{i,j}, b^D_{i,j}\}, \{I_{i,j}\})$).

Thus, Theorem \ref{thm:112btrees} follows from:

\medskip

\noindent {\bf Claim.}  There exists generalized $b$-manifolds $A^{LU}$ and $A^{RD}$ for the above disk data such that
\[
\overline{\partial_{A^{LU}}I_{i,j}} = b^U_{i,j} + b^L_{i,j} + \sum_{i<m<j} b^{U}_{i,m} b^L_{m,j}; \quad \mbox{and} \quad \overline{\partial_{A^{RU}} I_{i,j}} = b^R_{i,j} + b^D_{i,j}+ \sum_{i<m<j} b^{R}_{i,m} b^D_{m,j}.
\]
(Here, any of the $b^L_{i,j}$ or $b^D_{i,j}$ that do not exist are replaced with $0$.)

\medskip

We construct $A^{LU}$ as the construction of $A^{RD}$ is identical after reflecting across $x_1=x_2$.
\begin{enumerate}
\item Connect each $b^U_{i,j}$
 to $I_{i,j}$ with a vertical (i.e. slope $\infty$) line segment, $\gamma^U_{i,j}$. 
\item Connect each $b^L_{i,j}$ to $I_{i,j}$ by a piecewise linear segment (smoothed at the corner), $\gamma^L_{i,j}$, consisting of (i) the line segment from $b^L_{i,j}$ to $(0, \beta^R_{i,j})$, followed by (ii) the horizontal segment (i.e. slope $0$) from $(0,\beta^R_{i,j})$ to $I_{i,j}$.  
\item For each $1\leq i<m<j \leq n$ such that $b^L_{m,j}$ exists, the paths $\gamma^U_{i,m}$ and $\gamma^L_{m,j}$ 
%associated to $b^U_{i,m}$ in (1) and to $b^L_{m,j}$ in (2) 
 intersect in a unique point.  (This follows from the location of the intervals $I_{i,j}$.)  Connect this intersection point to $I_{i,j}$ via a straight line segment, $\gamma_{i,m,j}$.
\end{enumerate}
See Figure \ref{fig:112Term2}.

\begin{figure}

\quad 

\quad 

\labellist
\small
\pinlabel $c_{m,j}$ [l] at 280 112
\pinlabel $b^L_{m,j}$ [r] at -2 88
\pinlabel $b^U_{i,m}$ [b] at 200 162
\pinlabel $c_{i,j}$ [l] at 238 71
\pinlabel $c_{i,m}$ [l] at 206 40 
\pinlabel $\gamma^L_{m,j}$ [b] at 90 106
\pinlabel $\gamma^U_{i,m}$ [r] at 198 142
\pinlabel $\gamma_{i,m,j}$ [l] at 222 98
\endlabellist
\centerline{ \includegraphics[scale=.6]{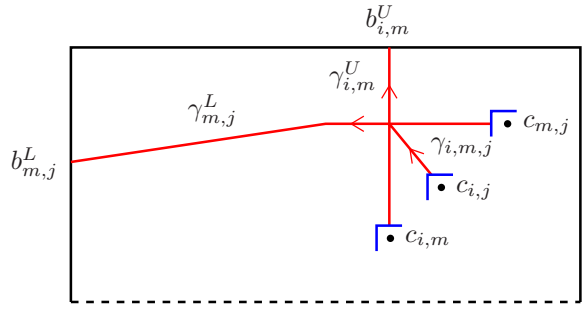} }

\quad

\quad

\caption{The paths $\gamma^U_{i,m}$, $\gamma^L_{m,j}$, and $\gamma_{i,m,j}$.}
%The arrows indicate the direction along edges of all $(i,j)$-flows (negative gradients) other than the $(k+2,k+1)$ flow.  
%(right) The partial flow trees in $A$ described in (2) of Lemma \ref{lem:RegionA}.  
\label{fig:112Term2}
\end{figure}

It is clear from definitions that  $\overline{\partial_{A^{LU}}I_{i,j}}$ has the desired form, so to complete the proof we only need to verify that $A^{LU}$ is in fact a generalized $b$-manifold.  To this end, it suffices to check that the only intersections between paths $\gamma_1, \gamma_2 \in A^{LU}$ for which 
\begin{equation} \label{eq:jgammaigamma}
j(\gamma_1) = i(\gamma_2)
\end{equation} 
are as in (3).  

Clearly, all the $\gamma^U_{i,j}$ are disjoint from one another.  When $x_1>0$, the $\gamma^L_{i,j}$ are horizontal lines, so because of the relative positioning of the $I_{i,j}$ (lexicographically ordered along the diagonal, $x_1=x_2$) 
$\gamma^U_{m,j}$ and $\gamma^L_{i,m}$ are disjoint.  

To see that $\gamma_{i,m,j}$ is disjoint from any path with lower index $i$, note that any path of the form $\gamma^U_{h,i}, \gamma^L_{h,i}$ or $\gamma_{h,g, i}$ is, at least when $x_1>0$, entirely above $x_2 = \beta^R_{h,i}-\e$ and to the left of $x_1 = \beta^U_{h,i}+\e$.  However, $\gamma_{i,m,j}$ is entirely to the right of $x_1 = \beta^U_{i,m} - \e$ and below $x_2= \beta^R_{m,j} +\e$.  These regions are disjoint (since $(h,i)$ preceds both $(i,m)$ and $(m,j)$ lexicographically).  A similar argument shows that $\gamma_{i,m,j}$ is disjoint from paths with upper index $j$.

Finally, we show that for any $i<m<j$, $\gamma^L_{i,m}\cap \gamma^L_{m,j} =\emptyset$ (provided both of these paths exist).  As in (2) (i), the first half of the $\gamma^L_{i,j}$ are line segments that, for all $i<j$, begin and end at on a common pair of vertical lines.  [Recall that, by Property (2) of \ref{pr:1cells} the left side of $\partial \widehat{N}(e^2_\alpha)$ is a vertical line when $1/2 \leq x_2 \leq 3/4$.] 
Two such segments intersect if and only if the ordering of their $x_2$-coordinates is opposite along the verical lines where the left and right endpoints sit.   The left endpoint of $\gamma^L_{i,j}$ satisfies $\beta^L_{i,j}-\e \leq x_2 \leq \beta^L_{i,j}+\e$, while the right endpoint is at $x_2=\beta^R_{i,j} = \beta_{i,j}$.  Now, $\beta^L_{i,j} = \beta_{\sigma_L(i), \sigma_L(j)}$ where $\sigma_L(i)$ is the position of the sheet $S_i$ above $(-1,1)$.  Thus, in order for $\gamma^L_{i,m}$ and $\gamma^L_{m,j}$ to cross we would need to have $(\sigma_L(m),\sigma_L(j)) \prec (\sigma_L(i),\sigma_L(m))$ with respect to lexicographic order.  Since $\sigma_L(m) \neq \sigma_L(i)$ this would mean that $\sigma_L(m) < \sigma_L(i)$ which can only happen if sheets $S_i$ and $S_m$ cross above $e^1_U$.  However, in this case $b^L_{i,m}$ doesn't exist.
%\footnote{\dr{7/26:  Need to come back and add one more (trivial) Move to the long proof.}}

\end{proof}

\begin{proof}[Proof of Theorem \ref{thm:SquareComp}] As $\partial C = \partial_c C +\partial_b C$, the result follows from Theorems \ref{thm:112ctrees}  and  \ref{thm:112btrees}.
\end{proof}

\section{Computation of LCH, Part 3: $2$-cells with swallowtail singularities}  \label{sec:SwallowComp}

In this section, we compute in Theorem \ref{thm:SwallowComp}
 the sub-DGAs $(\lchA(e^2_\alpha), \partial)$ when $e^2_\alpha$ is a square of type (13) or (14).  
% proceeds in much the same manner as in Section ???.  First, a collection of exceptional generators is analyzed.  After canceling these generators, 
A key feature of the Type (13) and (14) squares is that some rigid GFTs contain switch vertices.  This results in a differential that departs substantially from the common formula from Theorem \ref{thm:SquareComp} that applies to squares of type (1)-(12).  

\subsubsection{Upward and downward swallowtails}
Recall from Section \ref{ssec:ElemSq}  that the Type (13) square contains a single upward swallowtail point while the Type (14) contains a single downward swallowtail point.   The projections of the crossing and cusp locus in Type (13) and (14) squares are identical.   

 To simplify exposition, we consider only the case of the Type (13) square in detail while an analogous construction and computation holds for a Type (14) square.  (See Remark \ref{rem:about14} for more about translating the discussion of the Type (13) square to the case of a Type (14) square.)

\subsubsection{Labeling of sheets and Reeb chords above the Type (13) and (14) squares}  Let $e^2_\alpha$ be a $2$-cell of type (13).  %There is a small issue with labeling the sheets as the two sheets that cross both correspond to a single component of $L$ above $N(e^2_\alpha)$ after removing the cusp set.  
Near $(x_1,x_2) = (+1,+1)$ we label the sheets of $L$ above $e^2_\alpha$ as $S_1, \ldots, S_n$ with descending $z$-coordinate, so that $S_k,S_{k+1},S_{k+2}$ denote the three sheets that all meet at the swallowtail point.   %Extending this labeling to all of $N(e^2_\alpha)$ requires a small adjustment.  
The 
sheets $S_{k+1}$ and $S_{k+2}$ that cross one another near the swallow tail point both correspond to the same sheet to the left of the cusp locus, so we cannot extend this labelling to the entire square.  Instead, we use the following:   
\begin{convention} \label{conv:stlabeling}
\begin{enumerate}
\item For $m \notin \{k+1,k+2\}$, we let $S_m$ denote the entire sheet that sits in position $m$ 
%(ordered with decreasing $z$-coordinate) 
at $(+1,+1)$.  Note that $S_k$ only exists to the right of the cusp locus.   
\item We use $S_{k+1}$ and $S_{k+2}$ to denote only the closed subset of the sheets with position $k+1$ and $k+2$ at $(+1,+1)$ that lies on or to the right of the cusp locus.  More precisely, the base projections of $S_{k+1}$ and $S_{k+2}$ lie on or to the right of the cusp locus.
%with position $k+1$ and $k+2$ at $(+1,+1)$.  Specifically,  let $S_{k+1}$ and $S_{k+2}$ denote the subset of these sheets 
%whose base projection lies on or to the right of the cusp locus.  
\item We use $\widetilde{S}_k$ to denote the sheet with base projection on or to the left of the cusp locus that extends to agree with $S_{k+2}$ (resp. $S_{k+1}$) when $x_2$ is above (resp. below) the swallowtail point; denote the defining function of $\widetilde{S}_k$ by $\widetilde{F}_k$.  
\end{enumerate}
\end{convention}
Thus, to the left of the cusp locus the sheets appear with decreasing $z$-coordinate in the order $S_1, \ldots, S_{k-1}, \widetilde{S}_k, S_{k+3}, \ldots, S_n$.

The Reeb chords above the neighborhoods of the boundary edges and corners of $e^2_\alpha$ are specified by Properties \ref{pr:0cells} and \ref{pr:1cells}.  Reeb chords above $\widehat{N}(e^2_\alpha)$, and their locations are as in Properties \ref{pr:Reeb2} and \ref{pr:Location2}:
\begin{itemize}
\item For all $1 \leq i <j \leq n$, there is a Reeb chord $c_{i,j}$ with 
\[
c_{i,j} \in (\beta^U_{i,j} - \e, \beta^U_{i,j} + \e) \times (\beta^R_{i,j} - \e, \beta^R_{i,j} + \e).
\]
\item Sheets $S_{k+1}$ and $S_{k+2}$ cross, so there is a Reeb chord $\tilde{c}_{k+2,k+1}$ with
\[
\tilde{c}_{k+2,k+1} \in (\beta^D_{k+2,k+1}- \e, \beta^D_{k+2,k+1}+ \e) \times (\tilde{\beta}^R_{k+2,k+1}- \e, \tilde{\beta}^R_{k+2,k+1}+ \e).
\]
\end{itemize}
% continue to hold to specify the locations of Reeb chords above $\widehat{N}(e^2_\alpha)$.    

As long as neither endpoint of a Reeb chord lies on $\widetilde{S}_k$ we follow our earlier Convention \ref{conv:subscript} so that $a^{\pm,\pm}_{i,j}$, $b^L_{i,j}$, $b^D_{i,j}$, $b^U_{i,j}$, $b^R_{i,j}$, and $c_{i,j}$ denote Reeb chords above the various vertices, edges, and interior of the square with upper endpoint on $S_{i}$ and lower endpoint on $S_{j}$.  There are two Reeb chords $\tilde{b}^R_{k+2,k+1}$ 
and $\tilde{c}_{k+2,k+1}$ that are decorated with tildes.  They respectively lie in the lower half of the right edge and in the lower right portion of $\widehat{N}(e^2_\alpha)$.
  
\begin{convention}  \label{c:Sktilde} If a Reeb chord has an endpoint on the sheet $\widetilde{S}_k$, then we use $k+2$ for the corresponding subscript when $x_2 >0$, i.e. for the the $b^L$ and $a^{-,+}$ Reeb chords, and $k+1$ when $x_2 <0$, i.e. for the $a^{-,-}$ Reeb chords.  
\end{convention}
%Probably label sheets at $-,-$ with $k,k+2$ omitted.  What about sheets at $-,+$.

\begin{remark} \label{rem:about14}
A similar enumeration of sheets and Reeb chords applies to the Type (14) square.  There the roles of the swallowtail sheets are reversed, so that exposition of the computation of GFTs for the Type (14) square would require translating considerations applied to the sheets $\tilde{S}_k$, $S_{k}$, $S_{k+1}$, and $S_{k+2}$ from the Type (13) square  to considerations for sheets of the form $\tilde{S}_l, S_l, S_{l-1}$, and $S_{l-2}$.   In particular, for Reeb chords in a Type (14) square with endpoints on $\tilde{S}_l$, we use the subscript $l-2$ when $x_2 >0$, and $l-1$ when $x_2<0$.
\end{remark}

\subsubsection{Statement of the differential}  %To provide a condensed formula for the differential of the $c_{i,j}$ generators, we 
We form matrices containing Reeb chords in $N(e^2_\alpha)$ where $e^2_\alpha$ has type (13).  Let $A_{+,+}$, $A_{-,-}$, $B_X$ for $X \in \{L, D, R, U\}$, and $C$ denote strictly upper triangular matrices with $(i,j)$-entries, for $i<j$, given by the corresponding Reeb chord provided that it exists.  For example, take  $a^{+,+}_{i,j}$ for the $(i,j)$-entry of $A_{+,+}$ when $i<j$.  In view of Convention \ref{c:Sktilde}, there are no $a^{-,-}$ Reeb chords with $k$ or $k+2$ as a subscript.  We set the $(k,k+2)$-entry of $A_{-,-}$ to be $1$, while all remaining entries in rows and columns $k$ and $k+2$ are set to $0$.    Convention \ref{c:Sktilde} also shows that there are no $b^L$ Reeb chords with subscript  $k$ or $k+1$, 
 and we set all entries of the $k$ and $k+1$ rows and columns of $B_L$ to be $0$.  We also place a $0$ in the $(k+1,k+2)$-entry of $B_D$.  The Reeb chords $b^D_{k+2,k+1}, \tilde{b}^R_{k+2,k+1}, \tilde{c}_{k+2,k+1}$ do not appear in any of these matrices.

%we replace the $k+2,k+1$ entry with $0$ even though a Reeb chord $b^D_{k+2,k+1}$ exists.  In addition, set the $k,k+2$ entry of $A_{-,-}$ to be $1$.  All remaining entries are set to $0$. We note that  except for the $1$ in the $k,k+2$ position, all other entries in rows and columns $k$ and $k+2$ of $A_{-,-}$ are $0$.  In addition, the $k$ and $k+1$ columns of $B_L$ have all entries $0$ and the $k+1,k+2$ entry of $\tilde{B}_D$ is $0$.  All remaining entries above the main diagonal are non-zero. 

A similar procedure is applied to form matrices from the Reeb chords of a Type (14) square.  In this case, $A_{-,-}$ has $(l-2,l)$ entry $1$, while all remaining entries in the $l-2$ and $l$ rows and columns are $0$; $B_L$ has $0$ entries in the $l$ and $l-1$ rows and columns.  

Recall that $E_{i,j}$ denotes the matrix with $(i,j)$-entry $1$ and all other entries $0$.

\begin{theorem} \label{thm:SwallowComp}  For  a $2$-cell $e^2_\alpha$ of type (13), in $(\lchA(e^2_\alpha), \partial)$ we have
\[
\begin{array}{rl}
\partial C = & A_{+,+} C + C (I+E_{k+2,k+1}) A_{-,-} (I+E_{k+2,k+1}) + \\ & (I+B_U)(I+B_L)(I+ A_{-,-}E_{k+1,k}+ E_{k+1,k+2}) + (I+B_R)(I+B_D + B_D E_{k+2,k+1}) + X, 
\end{array}
\]
where $X$ denotes a strictly upper triangular matrix with entries in the ideal generated by $\tilde{b}^R_{k+2,k+1}, \tilde{c}_{k+2,k+1}$.  In particular, the Reeb chord $b^{D}_{k+2,k+1}$ does not appear in the differential of any of the $c_{i,j}$.

 For  a $2$-cell $e^2_\alpha$ of Type (14), in $(\lchA(e^2_\alpha), \partial)$ we have
\[
\begin{array}{rl}
\partial C = & A_{+,+} C + C (I+E_{l-1,l-2}) A_{-,-} (I+E_{l-1,l-2}) + \\ & (I+B_U)(I+B_L)(I+ E_{l, l-1}A_{-,-}+ E_{l-2, l-1}) + (I+B_R)(I+B_D + B_D E_{l-1,l-2} ) + X, 
\end{array}
\]
where $X$ denotes a strictly upper triangular matrix with entries in the ideal generated by $\tilde{b}^R_{l-1,l-2}, \tilde{c}_{l-1,l-2}$.  In particular, the Reeb chord $b^{D}_{l-1,l-2}$ does not appear in the differential of any of the $c_{i,j}$.
\end{theorem}

Theorem \ref{thm:SwallowComp} is proved at the conclusion of this section as a consequence of Theorem \ref{thm:dc} and Propositions
\ref{prop:dbc} and \ref{prop:bRDcomp}.   In the remainder of this section we present only the case of a Type (13) square, i.e. an upward swallowtail, as the Type (14) square is analogous.

\subsection{Properties of gradients}  \label{sec:61}
In the following we continue to use the terminology {\it $(i,j)$-flow line} to denote a flow line for $-\nabla F_{i,j}$, while  {\it $(\tilde{k}, j)$-flow line} or {\it $(i, \tilde{k})$-flow line} refers to  a flow line of $-\nabla( \tilde{F}_k - F_j)$ or $-\nabla( \tilde{F}_k - F_j)$ respectively.  Notice that a single edge of a GFT may change, for instance, from being an $(i,k+2)$-flow line to an $(i,\tilde{k})$-flow line if its image crosses the upper half of the cusp locus.

Many of the properties of squares of type (1)-(12) stated in Section \ref{sec:Comp2Cells} continue to hold for the Type (13) square.  Specifically, as previously discussed,  Properties \ref{pr:Reeb2} and \ref{pr:Location2} hold. In addition,  Property \ref{pr:monotonicityIV} continues to hold as stated.  The statements from Property \ref{pr:monotonicityI} also hold and apply as well to difference functions $F_i-\tilde{F}_k$ and $\tilde{F}_k-F_j$ provided  $\beta^L_{i,k+2}$ and $\beta^L_{k+2,j}$ are used when providing the restrictions on $x_2$ that appear in the statements of these properties.  (This is consistent with the Convention \ref{c:Sktilde}.)

\begin{lemma} \label{lem:holdovers}
The Lemmas \ref{lem:1stQuad},  and \ref{lem:14121} continue to hold for the Type (13) square.  
\end{lemma}

\begin{proof}
The proofs of these Lemmas use only Properties \ref{pr:monotonicityI} and \ref{pr:monotonicityIV} from Section \ref{sec:Comp2Cells}, and do not depend on the precise form of the singular set above $N(e^2_\alpha)$.
\end{proof}

Property \ref{pr:monotonicityII} requires the following modification, and strengthening.  Recall the notation $\nabla F= (\grad_{x_1}F,\grad_{x_2}F)$.

\begin{property}[Monotonicity in half spaces, Type (13) square]  \label{pr:monotonicityIIST} The defining functions in a Type (13) square satisfy:
\begin{itemize}
\item Suppose that $1 \leq i < j \leq n$, and $\{i,j\} \not\subset \{k,k+1,k+2\}$.  Then, (where defined)
\[
-\grad_{x_2} F_{i,j} <0, \quad \mbox{when $x_2 = 1/2$}.
\]
Moreover, for $i<k$ and $k+2<j$, we also have (where defined)
\[
\mbox{$-\grad_{x_2}(F_{i}-\widetilde{F}_k) <0$ and $-\grad_{x_2}(\widetilde{F}_k-F_j) <0$,    when $x_2=1/2$.}
\]
\item Suppose that $1 \leq i < j \leq n$, and $(i,j) \neq (k+1,k+2)$.  Then, for $(x_1,x_2) \in \widehat{N}(e^2_\alpha)$ we have
\begin{equation} \label{eq:monoIIeq1ST}
-\grad_{x_1} F_{i,j}(x_1,x_2) <0,  \quad \mbox{  whenever  $x_1  \in \left[1/4,\, \mbox{Min}\{ \beta^U_{i,j}, \beta^D_{i,j} \}- \e\right]$, and }
\end{equation}
\begin{equation} \label{eq:monoIIeq2ST}
-\grad_{x_1} F_{i,j}(x_1,x_2) >0,  \quad \mbox{  whenever  $x_1  \in \left[\mbox{Max}\{\beta^U_{i,j}, \beta^D_{i,j}\}+ \e, 3/4\right]$.}
\end{equation}
\end{itemize}
\end{property}

\begin{remark}
Since sheets $S_{k+1}$ and $S_{k+2}$ cross above $e^1_R$, we have $\beta^D_{i,j}= \beta^U_{\sigma_D(i), \sigma_D(j)}$ where $\sigma_D$ transposes $k+1$ and $k+2$.  Compare with Remark \ref{rem:betanotate}.
\end{remark}

On the other hand, Property \ref{pr:CuspTransversality} requires a serious modification for Type (13) squares.  This leads to many new rigid gradient flow trees.

\begin{property}[Existence and uniqueness of switch points] \label{pr:switches}
Let $O_1$ %and $O_2$ 
denote the closed disk centered at $(-3/8,0)$ with radius $R_1 = 1/16$,  
%and $R_2 = 3/32$, 
and let $\Sigma \subset N(e^2_\alpha)$ denote (the base projection of) the cusp locus within a Type (13) square.  
The projection of $\Sigma$ to the $x_2$-coordinate is injective.
The swallowtail point, $Q$, is located within $O_1$ along the horizontal line $x_2=0$ at a point with $-7/16 < x_1  <-3/8$.  At $Q$ two branches of $\Sigma$ meet at a semi-cubical cusp, where the $x_1$-coordinate is minimized along $\Sigma$.  For $x_2 \geq 0$ sheets $S_k$ and $S_{k+1}$ meet along $\Sigma$, while for $x_2 \leq 0$ sheets $S_k$ and $S_{k+2}$ meet along $\Sigma$.  For $|x_2| \leq 1/32$, the $(k,k+1)$-cusp locus and $(k,k+2)$-cusp locus each project injectively to the $x_1$-axis.  Outside of $O_1$, $\Sigma$ agrees with the vertical line $x_2 = -3/8$.

The defining functions $F_{i}$ 
 satisfy the following.
\begin{enumerate}
\item For any $i <k$, there is a unique $(i,k)$-switch point along $\Sigma$.  This switch point has $x_2>0$.
\item For any $k+2 \leq j$, there is a unique $(k+1,j)$-switch point in $\Sigma \cap \{x_2 \geq 0\}$.  %This point is located within $O_1$.
\item For any $k+2 < j$, there are no $(k+2,j)$-switch points in $\Sigma \cap \{x_2 \leq 0\}$.  
\item There is a unique $(k+2,k+1)$-switch point in  $\Sigma \cap \{x_2 \leq 0\}$.  %This point is located within $O_1$.
\end{enumerate}
Moreover, all of these switch points are non-degenerate and are located in $O_1$ with $|x_2| \leq 1/32$.
\end{property}

\begin{definition}
We refer to the portion of the cusp locus with $x_2 >0$ (resp. $x_2<0$) where sheets $S_k$ and $S_{k+1}$ (resp. $S_k$ and $S_{k+2}$) meet as the {\bf $(k,k+1)$-cusp locus} (resp. the {\bf $(k,k+2)$-cusp locus}).
\end{definition}

\begin{corollary}  \label{cor:mnPQ}
\begin{enumerate}

\item  
For $(r,s) = (i,k)$, or $(r,s)= (k+1,j)$, with $i<k$, $k+1<j$,   $-\nabla F_{r,s}$ is transverse to the $(k,k+1)$-cusp locus  at all points  besides the $(r,s)$-switch point, $P$.  Moreover, between $P$ and $Q$ (the swallowtail point), $-\nabla F_{r,s}$ points in the direction where the number of sheets increases, while at all other points $-\nabla F_{r,s}$ points in the direction where the number of sheets decreases.

\item For $(r,s) = (i,k)$, or $(r,s)= (k+2,j)$, with $i<k$, $j>k+2$,   $-\nabla F_{r,s}$ is transverse to the $(k,k+2)$-cusp locus everywhere and  points in the direction where the number of sheets increases.  The same statement applies for $(r,s) = (k+2,k+1)$ except between the $(k+2,k+1)$-switchpoint, $P$, and the swallowtail point $Q$ where $-\nabla F_{k+2,k+1}$ points in the direction where the number of sheets increases.

\end{enumerate}

\end{corollary}
See Figure \ref{fig:13Points2}.

\begin{proof}
As we precede along the $(k,k+1)$- or $(k,k+2)$-cusp locus, the direction of transversality of  $-\nabla F_{r,s}$ to  $\Sigma$ changes exactly when either 
$-\nabla F_{r,s}$ becomes tangent to the cusp locus; for $(r,s)$ as in the statement of the Corollary, such points of tangency are exactly at the $(r,s)$-switch points. 
[The direction of transversality \emph{must} change at a switch point because all switch points are non-degenerate.]
Moreover, when $x_2$ is near $\pm 1$, the direction of transversality for all positive difference functions is from right to left, by Property \ref{pr:1cmono}.  The result follows.
\end{proof}

\begin{figure}
\labellist
\small
\pinlabel $P$ [l] at 108 176
\pinlabel $P_1$ [l] at 426 176
\pinlabel $P_2$ [l] at 426 96
\pinlabel $Q$ [r] at 60 136
\endlabellist
\centerline{ \includegraphics[scale=.6]{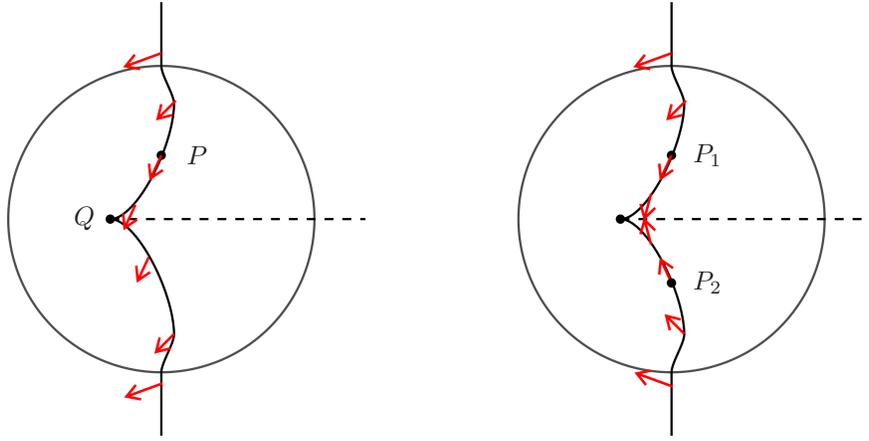} }
\caption{(left)  The $(r,s)$-switch point $P$ for $(r,s) = (i,k)$ with $i<k$ or $(r,s)= (k+1,j)$ with $k+2<j$. The vector field $-\nabla F_{r,s}$ is schematically pictured in red.   (right)   The $(k+1,k+2)$-switch point, $P_1$, and the $(k+2,k+1)$-switch point, $P_2$.  The vector field $-\nabla F_{k+1,k+2}$ (resp. $-\nabla F_{k+2,k+1}$) is pictured along the $(k,k+1)$-cusp locus (resp. along the $(k,k+2)$-cusp locus).  
}  
\label{fig:13Points2}
\end{figure}

\subsubsection{Additional monotonicity properties for squares containing swallowtails}

The following Properties \ref{pr:STBnew}-\ref{pr:SwitchBarriers} provide more detailed constraints on the direction of gradient vector fields.  They are used at various places during the course of the proof of Theorem \ref{thm:SwallowComp}.

\begin{property}[The Swallowtail Barrier] \label{pr:STBnew}  For the Type (13) square, there is a polygonal path $P= P_1 \cup P_2$ that satisfies the following.
\begin{enumerate}
\item The subset $P_1$ is a single line segment that starts on $x_2 = -3/8$; ends on $x_2 = 3/4$; and is entirely contained in $(1/4,1/2) \times [-3/8,3/4]$.
\item For all $1 \leq i < j \leq n$ and for all $(x_1,x_2) \in P_1$ with $-3/8 \leq x_2 \leq \beta^R_{i,j} -\e$, 
\[
\mbox{$-\nabla F_{i,j}(x_1,x_2)$ points transversally into the region to the right of $P_1$.}
\]
\item The subset $P_2$ has monotonic $x_1$-coordinate with its right endpoint at the lower endpoint of $P_1$ and its left endpoint on the left boundary of $N(e^2_\alpha)$.  Moreover, $P_2 \subset [-1-1/32, 1/2] \times [-3/8, -1/4]$.
\item For all $(x_1,x_2) \in P_2$, and $(i,j) \neq (k+2,k+1)$ such that $F_{i,j}(x_1,x_2) >0$, 
\[
\mbox{$-\nabla F_{i,j}(x_1,x_2)$ points transversally into the region below $P_2$.}  
\]
(We include the case of $-\nabla(F_i- \widetilde{F}_k)$ and $-\nabla(\widetilde{F}_k - F_j)$.)
\end{enumerate}
\end{property}
See Figure \ref{fig:NewSTBarrier}.

\begin{figure}
\labellist
\small
\pinlabel $x_2=\beta^R_{i,j}-\e$ [l] at 188 138
\pinlabel $P$ [l] at 102 50
\endlabellist
\centerline{ \includegraphics[scale=.6]{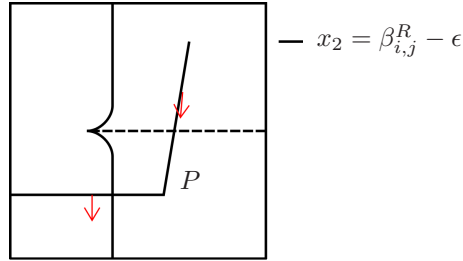} }
\caption{  The Swallowtail Barrier, $P$.  For $i<j$, $-\nabla F_{i,j}$, pictured in red, points to the region below and to the right along $P$ when $x_2 \leq \beta^R_{i,j} - \e$.
}  
\label{fig:NewSTBarrier}
\end{figure}

\begin{property}[Barrier at $x_1 = -17/64$]
\label{pr:leftC} For
 any $i <k$, we have  $-\grad_{x_1} F_{i,k}<0$, $-\grad_{x_1} F_{i,k+1}<0$, and $-\grad_{x_1} F_{i,k+2}<0$ at all points along $ x_1 = -17/64$.
\end{property}

\begin{property}[Downward gradients when $x_2 \leq -1/4$]   \label{pr:monoLtilde}
 For
$i< k$ and $j>k+2 $, we have  $-\grad_{x_2}(F_i -w)(x_1,x_2) <0$ and $-\grad_{x_2}(w-F_j)(x_1,x_2) <0$ for all $(x_1,x_2) \in (\mbox{domain of $w$})\cap\{ (x_1,x_2) \in \widetilde{N}(e^2_\alpha) \, |\, x_1 \leq -1/4 \mbox{  and  } -1/4 \leq x_2  \leq 1/4 \}$  where $w$ is any of $F_k,F_{k+1},F_{k+2}$ or $\widetilde{F}_k$.  Moreover, at points with $x_1 \leq -1/4$ that belong to the crossing locus, $-\nabla (F_i-w)$ points into the region where $F_{k+2} > F_{k+1}$.
\end{property}

\begin{property}[Switch point barriers] \label{pr:SwitchBarriers}

There exists a line segment $L_{k+1,k+2}$ contained entirely in $O_1$
 with upper endpoint at the  $(k+1,k+2)$-switch point and lower endpoint on the $(k,k+2)$ cusp locus such that 
\begin{itemize} 
\item $-\nabla F_{k+1,k+2}$ points from right to left at all points of $L_{k+1,k+2}$ above or on the crossing locus; and 
\item $-\nabla F_{k,k+2}$ points from right to left along all of $L_{k+1,k+2}$ except for the endpoint on the $(k,k+2)$ cusp edge.
\end{itemize}
Similarly, there is a line segment $L_{k+2,k+1}$ contained in $O_1$ with endpoints at the $(k+2,k+1)$-switch point and on the $(k,k+1)$ cusp edge such that
\begin{itemize} 
\item $-\nabla F_{k+2,k+1}$ points from right to left at all points of $L_{k+2,k+1}$ below or on the crossing locus; and 
\item $-\nabla F_{k,k+1}$ points from right to left along all of $L_{k+2,k+1}$ except for the endpoint on the $(k,k+2)$ cusp edge.
\end{itemize}

\end{property}

See Figure \ref{fig:FlowDirect}.

\subsubsection{The $(i,k)$, $(k+1,k+2)$, and $(k+2,k+1)$ switch flow lines and GFTs}

Flow lines that end at switch points play a role in the Type (13) square that is similar to the stable manifolds of $b$ Reeb chords.  In this subsection, we examine the flow lines ending at the $(i,k)$-, $(k+1,k+2)$-, and $(k+2,k+1)$-switch points.

For each $(r,s)$-switch point identified in Property \ref{pr:switches}, we call the (unique) $(r,s)$-flow line that ends at the switch point the {\bf{$(r,s)$-switch flow line}}.

\begin{lemma}  \label{lem:k1k2switchflow}
\begin{enumerate}
\item The $(k+1,k+2)$-switch flow line (resp. $(k+2,k+1)$-switch flow line) limits to $c_{k+1,k+2}$ (resp. $\tilde{c}_{k+2,k+1}$) as $t \rightarrow -\infty$. 
\item The $(k,k+2)$-flow line (resp. $(k,k+1)$-flow line) starting at the $(k+1,k+2)$-switch point (resp. the $(k+2,k+1)$-switch point) ends at an $e$-vertex at a point on the $(k,k+2)$-cusp locus (resp. $(k,k+1)$-cusp locus).
\end{enumerate}
\end{lemma}

\begin{proof}
For (1), recall that any flow line must either limit to a Reeb chord or terminate at the cusp edge as $t\rightarrow -\infty$.  (See the discussion around Proposition \ref{prop:bUS}.) If the $(k+1,k+2)$-switch flow line were to terminate at the cusp edge as $t$ decreases, it could only do so at a point of the $(k,k+1)$-cusp edge where $-\nabla F_{k+1,k+2}$ points in the direction where the number of sheets increases.  From  Corollary \ref{cor:mnPQ}, such a point must be located between the $(k+1,k+2)$-switch point,  $P$, and the swallow tail point, $Q$.  However, Property \ref{pr:SwitchBarriers} prevents the $(k+1,k+2)$-switch flow line from passing to the left of $L_{k+1,k+2}$ 
 as $t$ decreases, and the portion of the cusp locus between $P$ and $Q$  lies entirely to the left of $L_{k+1,k+2}$.  Thus, as $t\rightarrow -\infty$, the $(k+1,k+2)$-switch flow line limits to a $(k+1,k+2)$-Reeb chord in $\widehat{N}(e^2_\alpha)$; the only such Reeb chord is $c_{k+1,k+2}$.  

A similar argument (using the statement about $-\nabla F_{k+2,k+1}$ and the segment $L_{k+2,k+1}$ in Property \ref{pr:SwitchBarriers}) applies to show the $(k+2,k+1)$-switch flow line begins at $\tilde{c}_{k+2,k+1}$.  

For (2), consider only the $(k,k+1)$-flow line starting at the $(k+2,k+1)$-switch point as a similar proof will establish the statement about the $(k,k+2)$-flow line.   Along the segment, $L_{k+2,k+1}$, that starts at the $(k+2,k+1)$-switch point,   
%the vertical segment through the switch point, 
$-\nabla F_{k,k+1}$ points to the left (by Property \ref{pr:SwitchBarriers}).  Therefore, as $t$ increases, the flow must exit the triangle-like region formed by $L_{k+2,k+1}$ and the cusp edge somewhere along the cusp edge.  Along the lower half of the cusp locus, to the left of the $(k+2,k+1)$-switch point, $-\nabla F_{k,k+1}$ points across the cusp edge from the side with fewer sheets to the side with more sheets.  [This follows from  Corollary \ref{cor:mnPQ} since $\nabla F_{k,k+1}$ and $\nabla F_{k+2,k+1}$ agree along the $(k,k+2)$-cusp edge.]  Thus, as $t$ increases the flow line can only reach the upper half of the cusp edge (which is part of the $(k,k+1)$-cusp locus) where it ends at an $e$-vertex.  See Figure \ref{fig:FlowDirect} (left).
\end{proof}

\begin{figure}
\labellist
\small
\pinlabel $L_{k+2,k+1}$ [l] at 70 96
\endlabellist
\centerline{ \includegraphics[scale=.6]{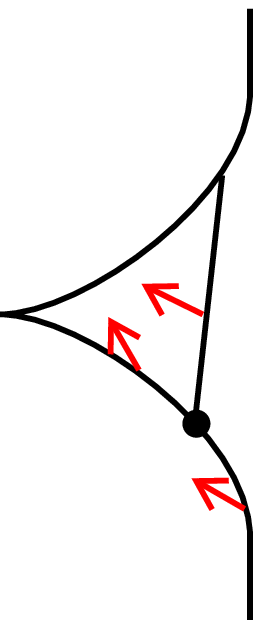} 
\quad \quad \quad \quad \quad \quad \quad \quad 
\labellist
\small
\pinlabel $c_{k+1,k+2}$ [r] at 136 144
\pinlabel $\tilde{c}_{k+2,k+1}$ [r] at 362 14
\endlabellist
\includegraphics[scale=.6]{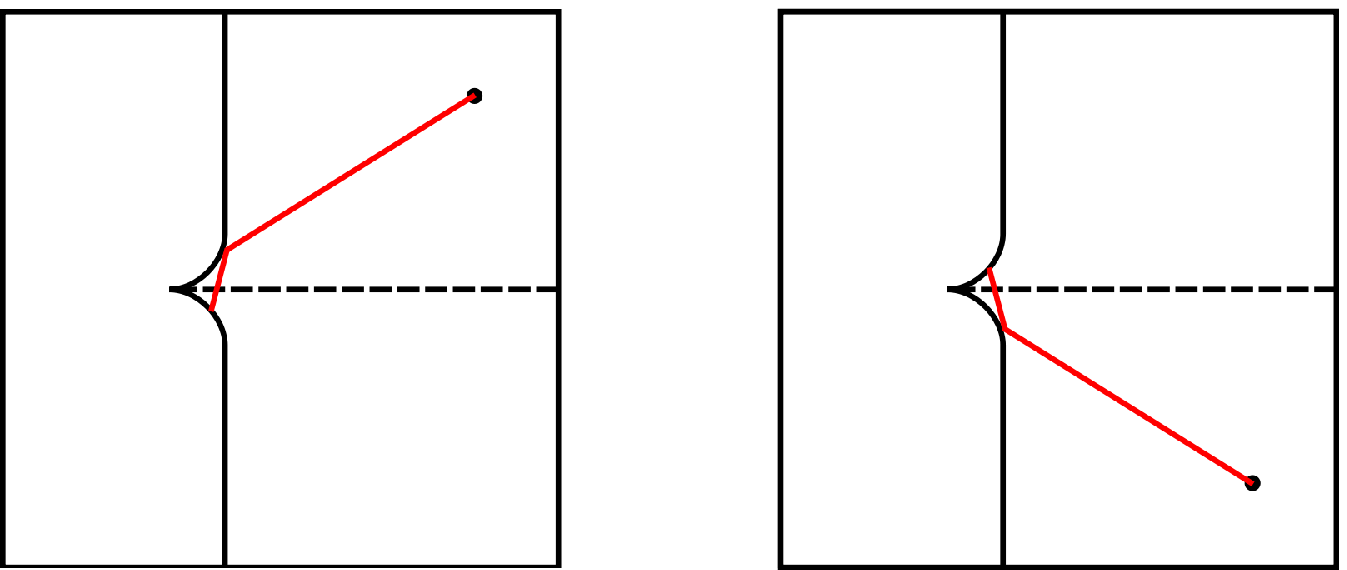}
}
\caption{(left)  The $(k+2,k+1)$-switch point appears as a dot.  Arrows indicate the direction of $-\nabla F_{k,k+1}$ along the lower half of the cusp locus and the line segment $L_{k+2,k+1}$ as used in Proof of Lemma \ref{lem:k1k2switchflow}.  (right)  The images of the GFTs resulting from Lemma \ref{lem:k1k2switchflow}.
}  
\label{fig:FlowDirect}
\end{figure}

The flow lines identified in (1) and (2) of Lemma \ref{lem:k1k2switchflow} fit together in pairs to produce gradient flow trees starting at $c_{k+1,k+2}$ and $c_{k+2,k+1}$ respectively.   Both of these GFTs contain a single switch as an internal vertex and  have a single output at an $e$-vertex.  See Figure \ref{fig:FlowDirect}.  We refer to these gradient flow trees as the {\bf $(k+1,k+2)$- and $(k+2,k+1)$-switch GFTs}.

We will make use of the following constraint on the image of the $(k+1,k+2)$-switch flow line.  Let $R(\mathit{sw}_{k+1,k+2})\subset \widehat{N}(e^2_\alpha)$ denote the region bounded
\begin{enumerate}
\item on the left by $L_{k+1,k+2}$ and the $(k,k+1)$-cusp locus;
\item below by a path constructed from left to right from (i) the $(k+1,k+2)$ crossing locus from $L_{k+1,k+2}$ to the Swallowtail Barrier $P$ (from Property \ref{pr:STBnew});  (ii) the part of $P$ below $x_2 = \beta^R_{k+1,k+2}-\e$; and (iii) the part of the line segment $\{x_2=\beta^R_{k+1,k+2}-\e\}$ between $P$ and $x_1 = \beta^U_{k+1,k+2}+\e$;
\item on the right by $\{x_1 = \beta^U_{k+1,k+2}+\e\}$;
\item and above by $\partial \widehat{N}(e^2_\alpha)$.
\end{enumerate}
See Figure \ref{fig:Rswk1k2}.

\begin{figure}
\labellist
\small
\pinlabel $L_{k+1,k+2}$ [l] at 60 104
\pinlabel $P$ [l] at 244 128
\pinlabel $c_{k+1,k+2}$ [r] at 322 214
%\pinlabel $x_2=\beta^R_{k+1,k+2}-\e$ [l] at 316 200
%\pinlabel $x_1=\beta^U_{k+1,k+2}+\e$ [tl] at 288 170
\endlabellist
\centerline{ \raisebox{16ex}{\includegraphics[scale=.5]{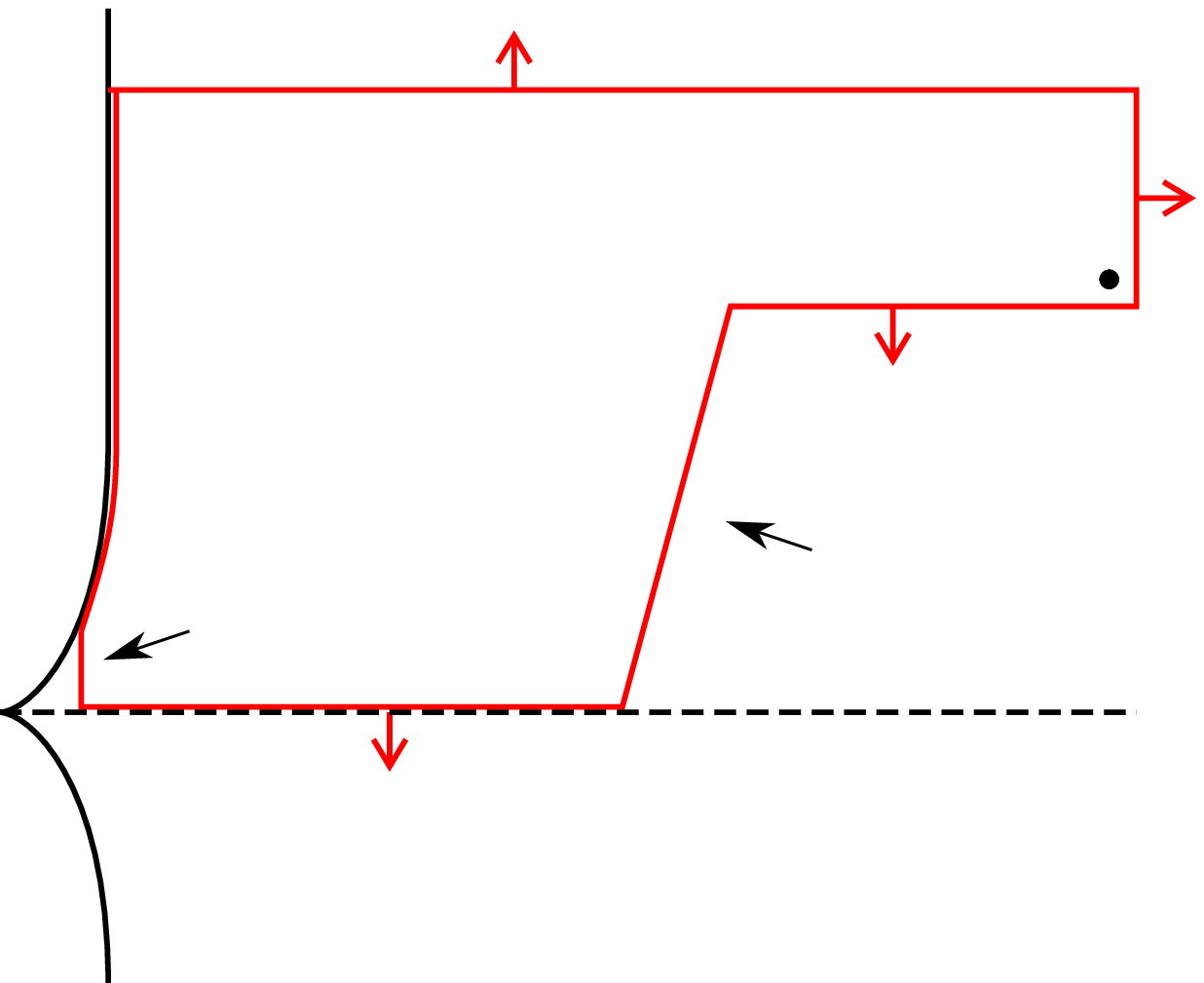}} 
\quad \quad \quad \quad \quad \quad 
\labellist
\small
\pinlabel $a^{-,-}_{i,k+1}$ [l] at 28 24
%\pinlabel $P$ [l] at 244 128
\pinlabel $c_{i,k}$ [r] at 422 318
%\pinlabel $x_2=\beta^R_{k+1,k+2}-\e$ [l] at 316 200
%\pinlabel $x_1=\beta^U_{k+1,k+2}+\e$ [tl] at 288 170
\endlabellist
\includegraphics[scale=.5]{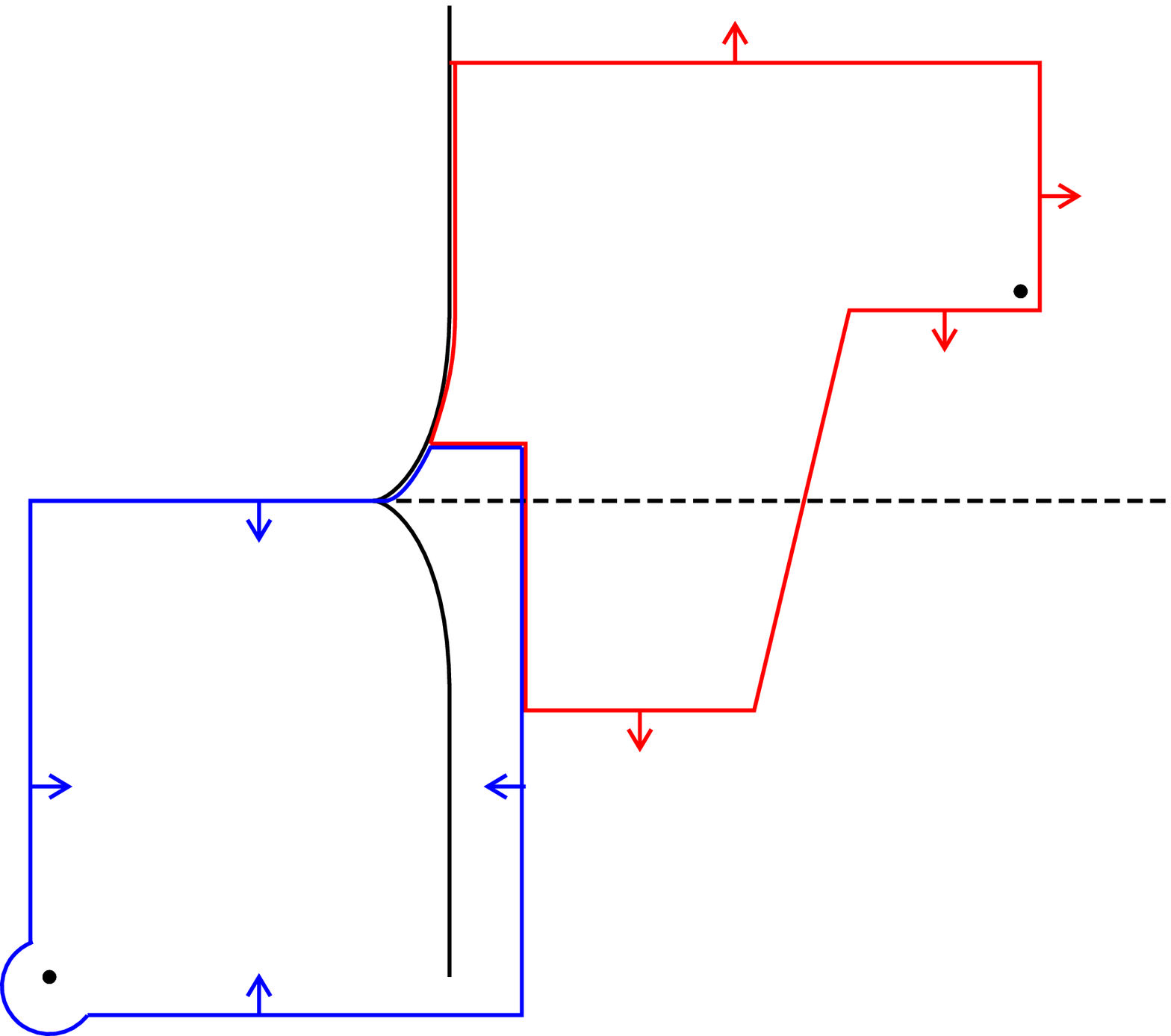} 
}
\caption{(left) The region $R(\mathit{sw}_{k+1,k+2})$.  (right)  The regions $R(\mathit{sw}_{i,k})$ (in red) and $C_{i,k}$ (in blue).
%The arrows indicate the direction along edges of all $(i,j)$-flows (negative gradients) other than the $(k+2,k+1)$ flow.  
%(right) The partial flow trees in $A$ described in (2) of Lemma \ref{lem:RegionA}.
}  
\label{fig:Rswk1k2}
\end{figure}

\begin{lemma} \label{lem:Rswk1k2} The $(k+1,k+2)$-switch flow line is entirely contained in $R(\mathit{sw}_{i,k})$.
\end{lemma}
\begin{proof}
Note that $-\nabla F_{k+1,k+2}$ points outward along all segments of the boundary of $R(\mathit{sw}_{i,k})$.  [Working clockwise from $L_{k+1,k+2}$, use Property \ref{pr:SwitchBarriers}, Corollary \ref{cor:mnPQ}, \ref{pr:1cells},
Property \ref{pr:monotonicityI} (twice), Property \ref{pr:STBnew}, and that $F_{k+1,k+2}$ goes from being positive to being negative when we pass the crossing locus.]  Since the $(k+1,k+2)$-switch flow line begins at the upper endpoint of $L_{k+1,k+2}$ it will remain within $R(\mathit{sw}_{i,k})$ as $t$ decreases.
\end{proof}

Next, we study the $(i,k)$-switch flow line.  
Define a closed region $R(\mathit{sw}_{i,k}) \subset \widehat{N}(e^2_\alpha)$ to be bounded by the curve that, beginning at the $(i,k)$-switch point and proceeding clock-wise, consists of
\begin{enumerate}
\item the upper part of the $(k,k+1)$-cusp locus;
\item part of the upper boundary of $\widehat{N}(e^2_\alpha)$;
\item the vertical line $\{x_1= \beta^U_{i,k}+\e\}$;
\item the horizontal line segment $\{x_2 = \beta^R_{i,k} - \e\}$;
\item the part of the Swallowtail Barrier, $P$, with $x_2 \leq \beta^R_{i,k} - \e$ and $x_1 \geq -17/64$;
\item the vertical segment $\{x_1=-17/64\}$ from $P$ to the $x_2$-coordinate of the $(i,k)$-switch point; and
\item the horizontal segment from the $(i,k)$-switch point to $x_1= -17/64$.
\end{enumerate}
See Figure \ref{fig:Rswk1k2}.  

\begin{lemma}  \label{lem:ikGFT}
For $i<k$, 
\begin{enumerate}
\item the $(i,k)$-flow line ending at the $(i,k)$-switch point (i.e. the $(i,k)$-switch flow line) is entirely contained in $R(\mathit{sw}_{i,k})$ and limits to $c_{i,k}$ as $t \rightarrow -\infty$;
\item the $(i,k+1)$-flow line beginning at the $(i,k)$-switch point limits to $a^{-,-}_{i,k+1}$.  
%Any other partial flow tree, $\Gamma$, starting with this flowline has $D(\Gamma) \geq 2$.
\end{enumerate}
\end{lemma}
\begin{proof}
Along each of the boundary segments (1)-(7) of $R(\mathit{sw}_{i,k})$ we have that $-\nabla F_{i,k}$ points outward.  [Verify with Corollary \ref{cor:mnPQ}, Property \ref{pr:1cells}, Property \ref{pr:monotonicityI} (used twice), then Properties \ref{pr:STBnew}, \ref{pr:leftC}, and \ref{pr:monoLtilde}.]  Thus, as $t$ decreases the $(i,k)$-switch flow line remains in $R(\mathit{sw}_{i,k})$ and can therefore only limit to $c_{i,k}$.  

To verify the claim concerning the $(i,k+1)$-flow line, $\gamma$, that begins at the $(i,k)$-switch point, we consider a region $C_{i,k}$ 
%\footnote{\ms{7/22/15: It would be good to have $C_{i,k}$ defined outside of the proof environment: the other two regions in the figure are; also $C_{i,k}$ is used in a later Lemma.} \dr{7/26:Probably true.  Not done yet.}}
bounded by:
\begin{enumerate}
\item the horizontal segment from the $(i,k)$-switch point to $x_1=-17/64$;
\item the vertical segment $x_1=-17/64$ from the $x_2$-coordinate of the $(i,k)$-switch point to the lower boundary of $N(e^2_\alpha)$;  
\item the portion of $\partial N(e^2_\alpha)$ running clock-wise from the bottom side at $x_1 = -17/64$ to the left side at $x_2=0$;
\item the horizontal segment from the left side of $\partial N(e^2_\alpha)$ at $x_2=0$ to the swallowtail point $Q$; and
\item the portion of the $(k,k+1)$-cusp locus between $Q$ and the $(i,k)$-switch point.
\end{enumerate}
Note that $-\nabla F_{i,k+1}$ and $-\nabla(F_i - \widetilde{F}_k)$ point inward along those parts of $\partial C_{i,k}$ where they are defined.  [Verify along (1)-(5) using Properties \ref{pr:monoLtilde}, \ref{pr:leftC}, \ref{pr:1cells}, \ref{pr:monoLtilde} (again), and Corollary \ref{cor:mnPQ}.]  Thus, even though $\gamma$ may change from being an $(i,k+1)$-flow line to an $(i,\tilde{k})$ flow line if it crosses the $(k,k+2)$-cusp locus,  as $t$ increases $\gamma$ must be entirely contained in $C_{i,k}$.  The only Reeb chord in $C_{i,k}$ with upper endpoint on $S_i$ and lower endpoint on $S_{k+1}$ or $\widetilde{S}_k$ is $a^{-,-}_{i,k+1}$ (which in fact has its lower endpoint on $\widetilde{S}_k$).  Thus, $\gamma$ must limit to $a^{-,-}_{i,k+1}$ as $t \rightarrow +\infty$.
\end{proof}

We refer to the gradient flow tree obtained by appending the flow lines from (1) and (2) together at a switch vertex as the {\bf $(i,k)$-switch GFT}.

\subsubsection{The $(k+1,k+2)$ and $(k+2,k+1)$ partial flow trees}

Partial flow trees beginning with certain $(k+1,k+2)$- or $(k+2,k+1)$-flow lines serve as a base case for inductive arguments used later in the section.  In this subsection we establish some preliminary results about such PFTs.

\begin{lemma}  \label{lem:PFTrees}  The only non-constant PFTs starting with a $(k+2,k+1)$-flow at a point in $\widehat{N}(e^2_\alpha)$ (to the right of the cusp edge and below the crossing locus) are 
\begin{enumerate}
\item  PFTs consisting of a single edge with output vertex limiting to $b^D_{k+2,k+1}$ or $\tilde{b}^R_{k+2,k+1}$.  These flow trees are subsets of unique flow lines that exist from $\tilde{c}_{k+2,k+1}$ to $b^D_{k+2,k+1}$ or $\tilde{b}^R_{k+2,k+1}$, and are respectively contained in the vertical strip 
\[
V = \{ (x_1,x_2) \in \widetilde{N}(e^2_\alpha) \, \left|\,  x_1 \in (\beta^D_{k+2,k+1}- \e,\beta^D_{k+2,k+1}+ \e)  \mbox{  and  } x_2 \leq \tilde{\beta}^R_{k+2,k+1} + \e \right.\}
\]
or the horizontal strip
\[
H = \{ (x_1,x_2) \in \widetilde{N}(e^2_\alpha) \, \left|\,  x_1 \geq  \beta^D_{k+2,k+1}- \e  \mbox{  and  } x_2 \in ( \tilde{\beta}^R_{k+2,k+1} - \e, \tilde{\beta}^R_{k+2,k+1} + \e) \right.\}.
\]

\item PFTs consisting of a single edge with output at $a^{+,-}_{k+2,k+1}$.  These PFTs are contained entirely in 
\[
\{ (x_1,x_2) \in \widetilde{N}(e^2_\alpha) \, \left| \, x_1 \geq \beta^D_{k+2,k+1} - \e \mbox{  and  }  x_2 \leq \tilde{\beta}^R_{k+2,k+1} + \e  \right.\} \bigcup N(e^0_{+,-}).
\]
%in the interior of the square with vertices at $\tilde{c}_{k+2,k+1}$, $b^D_{k+2,k+1}$, $\tilde{b}^R_{k+2,k+1}$, and $(+1,-1)$
\item PFTs that form some subset of $(k+2,k+1)$-switch GFT.
%the flow tree that starts at $\tilde{c}_{k+2,k+1}$, has a switch at the unique $(k+2,k+1)$-switch point, and then ends at an $e$-vertex somewhere along the $(k,k+1)$ cusp locus (as in Proposition \ref{prop:k1k2switchflow}).

%Partial flow trees that form some subset of the flow tree that starts at $\tilde{c}_{k+2,k+1}$, has a switch at the unique $(k+2,k+1)$-switch point, and then ends at an $e$-vertex somewhere along the $(k,k+1)$ cusp locus (as in Proposition \ref{prop:k1k2switchflow}).
\end{enumerate}
\end{lemma}

\begin{proof}  First, we show that any PFT, $\Gamma$, without internal vertices and beginning with a $(k+2,k+1)$ flow in $\widehat{N}(e^2_\alpha)$ must be as in (1) or (2).  Since $\Gamma$ does not have internal vertices it must simply be a portion of a flow line that limits to a Reeb chord as $t\rightarrow +\infty$.  [Ending at an $e$-vertex is impossible since there is no $(k+1,k+2)$-cusp edge.]  The only $(k+2,k+1)$ Reeb chords are $\tilde{c}_{k+2,k+1}, \tilde{b}^R_{k+2,k+1}, b^D_{k+2,k+1},$ and $a^{+,-}_{k+2,k+1}$, so $\Gamma$ must begin on a point belonging to the stable manifolds of one of these Reeb chords. The stable manifolds of $b^D_{k+2,k+1}$ and $\tilde{b}^R_{k+2,k+1}$  must be contained in the strips $V$ and $H$ respectively, by Properties \ref{pr:monotonicityI} and \ref{pr:1cmono} (and therefore both limit to $\tilde{c}_{k+2,k+1}$ respectively as claimed).  On the other hand, a point in the intersection of the stable manifold of $a^{+,-}_{k+2,k+1}$ with $\widehat{N}(e^2_\alpha)$ must be in the region bounded by the stable and unstable manifolds of $b^D_{k+2,k+1}$ and $\tilde{b}^R_{k+2,k+1}$, since flow lines cannot cross.

%These Reeb chords are all located in $[1/4,+1] \times [-1,-1/4]$ where $F_{k+2,k+1}(x_1,x_2) = f^{Cu}_{k+1,k+2}(x_1) +f^{1Cr}_{k+2,k+1}(x_2)$.  In particular, the $x_1$ (resp. $x_2$) coordinate of $\tilde{c}_{k+2,k+1}$ agrees with the $x_1$ coordinate of $b^D_{k+2,k+1}$ (resp. the $x_2$ coordinate of $\tilde{b}^R_{k+2,k+1}$), and the vertical and  horizontal segments conecting  $\tilde{c}_{k+2,k+1}$ with these Reeb chords are indeed flow lines.  Since flow lines cannot cross, the only $(k+2,k+1)$ flow lines beginning in $[-1,1]\times[-1,1]$ that can terminate at $a^{+,-}_{k+2,k+1}$ are those that begin within the square described in (2) of the statement of the Lemma.  We have therefore established that the only partial flow trees without internal vertices and beginning with a $(k+2,k+1)$-flow are as described in items (1) and (2).  [A $(k+2,k+1)$ flow line starting somewhere outside of the square from (2) must eventually run into the $(k+2,k+1)$ crossing arc or into the $(k,k+2)$ cusp edge since there are no other Reeb chords for it to terminate at.]

It remains to show that the only PFTs starting with a $(k+2,k+1)$-flow in $\widehat{N}(e^2_\alpha)$ and having internal vertices are as in (3).  Consider the first internal vertex of such a PFT, $\Gamma$.  Since sheets $S_{k+2}$ and $S_{k+1}$ are adjacent in the entire region where $S_{k+2}$ is above $S_{k+1}$, it follows that this vertex cannot be a $Y_1$ or $Y_0$-vertex.  Therefore, the vertex can only be a switch at the unique $(k+2,k+1)$-switch point, so it follows that the initial edge of the PFT is indeed a subset of the $(k+2,k+1)$-switch flow line.
 %from $\tilde{c}_{k+2,k+1}$ to this switch point.   %that was identified in Proposition \ref{prop:k1k2switchflow}.  
The edge of $\Gamma$ that begins after the switch vertex is the $(k,k+1)$-flow line 
%We have already shown in the proof of Proposition \ref{prop:T13ExcGen} that the $(k,k+1)$ flow line 
that begins at the $(k+2,k+1)$-switch point.  
%Proposition \ref{prop:k1k2switchflow} shows that this flow line must terminate somewhere along the $(k,k+1)$ cusp edge.  
We need to show that an internal vertex cannot occur along this edge.  
%this is the only way that there our PFT can be completed after the switch point.  

Suppose instead that $\Gamma$ has a second internal vertex occurs after the switch point.  The only possibility is a $Y_0$ vertex where the $(k,k+1)$-flow splits into a $(k,k+2)$ flow and a $(k+2,k+1)$-flow.  Furthermore, the image of this $Y_0$ vertex must be located below the $(k+1,k+2)$-crossing locus (otherwise there are no intermediate sheets between $S_k$ and $S_{k+1}$), and to the left of $L_{k+2,k+1}$ (since by Property \ref{pr:SwitchBarriers} $-\nabla F_{k,k+1}$ points to the left side of $L_{k+2,k+1}$).  However, the outgoing branch of the $Y_0$ that is a $(k+2,k+1)$-flow line cannot be completed to a PFT since we have already seen that the only PFTs starting with a $(k+2,k+1)$-flows must have their initial point either with $x_1 \geq \beta^D_{k+2,k+1} -\e$ or somewhere along the 
$(k+2,k+1)$-switch flow line.
%flow line from $\tilde{c}_{k+2,k+1}$ to the switch point.  
Moreover, by Property \ref{pr:SwitchBarriers} the $(k+2,k+1)$-switch flow line is contained entirely to the right of $L_{k+2,k+1}$ except at its endpoint. 
% shows that the flow line from $\tilde{c}_{k+2,k+1}$ to the $(k+2,k+1)$-switch point must be entirely to the right of the switch point except at its endpoint.
% can be\footnote{How do we know the $Y_0$ is not along the switch flow.  Extend the monotonicity along the vertical line through the switch to the $k+2,k+1$ gradient.} completed to PFTs.
\end{proof}

Recall the notations $D(1/2) = \{(x_1,x_2) \in \widetilde{N}(e^2_\alpha) \, |\, x_2 \leq 1/2\} \bigcup N(e^0_{-,-}) \bigcup N(e^0_{+,-})$ and $L(1/2) = \{(x_1,x_2) \in \widetilde{N}(e^2_\alpha) \, |\, x_1 \leq 1/2\} \bigcup N(e^0_{-,-}) \bigcup N(e^0_{-,+})$.

%\dr{ISSUE:  1.  What about $(k,k+1)$?  2. Really, what about any $(i,j)$ with $i,j \in \{k,k+1,k+2\}$.  Monotonicity 2 probably fails for these.  In fact, Monotonicity 1 may fail in $LU$ and $LD$ square.  Currently, it is used in $LD$.}
%\begin{lemma} \label{lem:T13D12}
%For the Type (13) square we have:
%\begin{enumerate}
%\item  For any $i,j$ with $F_{i} > F_{j}$, $D(1/2)$ is invariant for the $(i,j)$-flow, i.e. trajectories of $-\nabla F_{i,j}$ that begin in $D(1/2)$ remain there as $t$ increases provided $F_{i,j}$ remains positive.  This statement includes trajectories of $-\nabla(\tilde{F}_k -F_j)$ and $-\nabla(F_i- \tilde{F}_k)$.
%\item  For any $i,j$ with $F_{i}> F_{j}$ except $(k+1,k+2)$ or $(k+2,k+1)$, $L(1/2)$ is invariant for the $(i,j)$-flow.
%\end{enumerate}
%\end{lemma}
%\begin{proof}
%This follows from Properties \ref{pr:monotonicityII} and \ref{pr:1cmono}.

%\end{proof}

\begin{lemma}  \label{lem:PFTreesk1k2}  The only PFTs starting with a $(k+1,k+2)$-flow in $D(1/2) \cup L(1/2)$ are subsets of the $(k+1,k+2)$-switch GFT.
%flow tree that starts at 
%$c_{k+1,k+2}$, has a switch at the unique $(k+1,k+2)$-switch point, and then terminates somewhere along the $(k,k+2)$ cusp locus. 
\end{lemma}
\begin{proof}
%By Lemma \ref{lem:T13D12}, 
Any $(k+1,k+2)$-flow line that starts in $D(1/2)\cup L(1/2)$ must remain in $D(1/2)\cup L(1/2)$ since Property \ref{pr:monotonicityI} 
 shows that $-\nabla F_{k+1,k+2}$ points inward along the boundary segments $\{(x_1,x_2) \, | \, x_1 \geq 1/2, x_2 = 1/2\}$ and  $\{(x_1,x_2) \, | \, x_1 = 1/2, x_2 \geq 1/2\}$.  There are no $(k+1,k+2)$ Reeb chords in $D(1/2)\cup L(1/2)$.  Thus any such PFT must have an internal vertex.  Since $S_{k+1}$ and $S_{k+2}$ are always adjacent the first vertex can only be a switch at the unique $(k+1,k+2)$-switch point. After the switch, an argument similar to that used in Lemma \ref{lem:PFTrees} shows that the outgoing $(k,k+2)$-flow line 
%must terminate somewhere along the $(k,k+2)$ cusp locus 
cannot end at an internal vertex.  
%[Consider Figure \ref{fig:FlowDirect}, but with the graphic reflected across $x_2=0$.]
\end{proof}

\subsubsection{ The invariant region  $A$}

%In the following Lemmas \ref{lem:T13D12}, \ref{lem:Aprime}, and \ref{lem:RegionA}, 

%we identify some regions of $N(e^2_\alpha)$ that are invariant for particular gradient flows and classify certain degree $0$ PFTs. 

%Next, recall the roughly vertical, polygonal path $P$ identified in Property \ref{pr:STBarrier}.   
Let $A$ denote the subset of $D(1/2) \cap L(1/2)$ obtained by removing those points that lie strictly above %the line $\{x_2 = -1/4\}$ 
and to the left of the Swallowtail Barrier $P$.  Note that the boundary of $A$ consists of:
 %$\bigcup_{i=1}^5 B_i$ where the $B_i$ are specified by:
\begin{enumerate}
\item the vertical segment $\{x_1=1/2\} \cap \{x_2 \leq 1/2\}$;
\item the segment of $\{x_2= 1/2\}$ with endpoints on  $P$ and $x_1=1/2$;
\item the portion of $P$ below $x_2 =1/2$;
% $x_2 = 1/2$ and $x_2= -1/4$ is $B_3$.
%\item The segment of $\{x_2= -1/4\}$ with endpoints on the left boundary of $N(e^2_\alpha)$ and $P$ is $B_4$.
\item the portion of $\partial N(e^2_\alpha)$ that sits below $P$ and to the left of $\{x_1 = 1/2\}$.
\end{enumerate}

%Let $A$ denote the closed subset of $N(e^2_\alpha)$ whose upper upper and right boundaries is given by the line segment from 
%$(-1-1/32,-1/4)$ to $(-1,-1)$ to 
%$(1/2,-1-1/32)$ to $(1/2,1/2)$, a horizontal segment from $(1/2,1/2)$ to $P$, the portion of $P$ below $x_2=1/2$, and then a horizontal segment from the lower endpoint of $P$ to $(-1-1/32,-1/4)$.  The lower left portion of the boundary of $A$ is given by an appropriate subset of $\partial N(e^2_\alpha)$.   
See Figure \ref{fig:RegionA}.

% $A$ is the subset of $D(1/2) \cap L(1/2)$ obtained by removing those points that lie strictly above the line $\{x_2 = -1/4\}$ and to the left of $P$.
%dr{Make this the definition then list the boundary segments}
\begin{figure}
\labellist
\small
%\pinlabel $R_-$ [c] at 350 140
\endlabellist
\centerline{ %\includegraphics[scale=.6]{images/RegionA'} \quad 
\includegraphics[scale=.6]{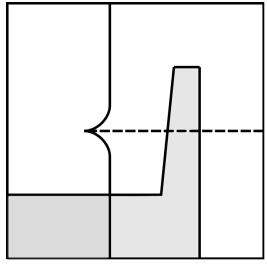} }
\caption{%(left)  The region $A'$.  (right) 
The region $A$.  
%The polygonal path $P$ from Property \ref{pr:STBnew}, 
% is pictured as simply a vertical line.
}  
\label{fig:RegionA}
\end{figure}

\begin{lemma} \label{lem:AAinvariance}
Suppose $(i,j) \neq (k+1,k+2)$, and $(i,j) \neq (k+2,k+1)$.  Any $(i,j)$-flow line that begins in $A$ remains in $A$ as $t$ increases (as long as it is defined).  This statement includes $(i,\tilde{k})$- and $(\tilde{k},j)$-flow lines.
\end{lemma}  

\begin{proof}
This follows since $-\nabla F_{i,j}$ points into $A$ along all segments comprising the boundary of $A$.  [Check with Property \ref{pr:monotonicityIIST} for (1) and (2), Property \ref{pr:STBnew} for (3), and
% For $B_4$,  we can apply Property \ref{pr:monotonicityI} (again, note the adjustment to this Property discussed in \dr{????}).  Finally, for $B_5$
%For the intersection of $A$ with the boundary of $N(e^2_\alpha)$ 
Properties \ref{pr:0cells} and \ref{pr:1cells} for (4).]  
\end{proof}

\subsection{Enumeration of flow trees with at least one $c_{i,j}$ output}

For a Type (13) square, $e^2_\alpha$, we again write 
%let us write the differential in  $(\lchA(e^2_\alpha), \partial)$ of a generator $c_{i,j}$  in the form
\begin{equation}\label{eq:dcdb}
\partial_c c_{i,j} = \sum_\Gamma w(\Gamma) 
%+ \partial_b c_{i,j} + \partial_x c_{i,j}
\end{equation}
%where $\partial_c c_{i,j}$ arises from 
with the sum over rigid GFTs that begin at $c_{i,j}$ and contain at least one puncture at a Reeb chord of the form $c_{r,s}$ or $\tilde{c}_{k+2,k+1}$.

%%;  $\partial_x c_{i,j}$ is the sum over rigid GFTs with at least one output at $\tilde{c}_{i,j}$ or $\tilde{b}^R_{k+2,k+1}$; and $\partial_b c_{i,j}$ denotes the sum over remaining rigid GFTs.  

%As we are working in $(\widetilde{\mathcal{A}}_{\mathit{LCH}}, \partial)$, all occurences of $b^D_{k+2,k+1}$ are replaced with $1$ and the other exceptional generators are replaced with $0$.

Recall that we have labeled the Reeb chords in $N(e^0_{-,-})$ with the index $k+1$ used for chords that begin or end on $\tilde{S}_k$.
%, so the  indices $k$ and $k+2$ do not appear.

\begin{theorem} \label{thm:dc}
We have
\begin{equation}  \label{eq:cterms}
%\partial_c c_{i,j} =  X_{I} + X_{II} + X_{III} + X_{IV} + X_{V} 
\partial_c c_{i,j} = \sum_{i<m<j} a^{+,+}_{i,m} c_{m,j} +  \sum_{i<m<j} c_{i,m} y_{m,j} +  x
\end{equation}
where  
$y_{m,j}$ denotes the $(m,j)$-entry of the matrix $(I+E_{k+2,k+1}) A_{-,-} (I+E_{k+2,k+1})$ from Theorem \ref{thm:SwallowComp} and $x$ belongs to the two-sided ideal generated by $\tilde{c}_{k+2,k+1}$ and $\tilde{b}^R_{k+2,k+1}$.  That is,
\begin{itemize}
%\item $\displaystyle X_{I} = \sum_{i<m<j} a^{+,+}_{i,m} c_{m,j}$
%\item $\displaystyle X_{II} = \sum_{i<m<j} c_{i,m} a^{-,-}_{m,j}$\footnote{How to label $a^{-,-}$ reeb chords for the easiest formula?}
\item for $\{m,j\} \cap \{k,k+2\}= \emptyset$, $y_{m,j} = a^{-,-}_{m,j}$;
\item $y_{k,k+2} = 1$;
\item $y_{k,k+1} = 1$;
\item for $k+2<j$, $y_{k+2,j} = a^{-,-}_{k+1,j}$;
\item all remaining $y_{m,j}$ are zero.
\end{itemize}

\end{theorem}  

The proof of Theorem \ref{thm:dc}  consists of two parts.  In Section \ref{sec:listGFTs},  we identify an odd number of rigid GFTs corresponding to each of the terms in the sum.  Then, in Section \ref{sec:estunique}, after some preliminary results, we show that any rigid GFT with a puncture at a $c_{r,s}$ must be one of the GFTs from our list.

%The proof is then completed by showing that these are the only rigid flow trees.  \dr{Before turning to the proof we establish several preliminary results on the behavior of flow trees in a Type (13) square. MOVED.}

\subsubsection{Degree of a partial flow tree}

%The $3$-rd Quadrant Lemma (Lemma \ref{lem:3rdQuad}) that holds for Type (1)-(12) squares fails for the Type (13) square.  

%In establishing properties of flow trees in a Type (13) square, it is useful to restrict our attention to those PFTs that may appear as parts of rigid GFTs.

Rigid GFTs can be identified by the following criterion.

\begin{lemma} \label{lem:aceyswST} Let $e^2_\alpha$ be a Type (13) square.  A GFT $\Gamma$ in $N(e^2_\alpha)$ that begins at some $c_{i,j}$ or $\tilde{c}_{k+2,k+1}$ is rigid if and only if 
\[
A-C+E -Y_1-SW=0
\]
where $A$  (resp. $C$) denotes the number of endpoints at Reeb chords of the form $a^{\pm,\pm}_{m,n}$ (resp. $c_{m,n}$ or $\tilde{c}_{k+2,k+1})$; $E$ denotes the number of endpoints at $e$-vertices; $Y_1$ denotes the number of $Y_1$-vertices; and $SW$ denotes the number of switch vertices.
\end{lemma}

(Recall from Theorem \ref{thm:1regV}, that since $\tilde{L}$ is $1$-regular all vertices of $\tilde{L}$ are either inputs, outputs, $e$-, $Y_0$-, $Y_1$-, or $sw$-vertices.)

\begin{proof}
This is just Proposition \ref{prop:EY1SWFormula} 
since all $a$, $b$, and $c$ Reeb chords are respectively local minima, saddles, and local maxima.
\end{proof}

We make the following definition to help identify PFTs that may appear as part of some rigid GFT.

\begin{definition}  
For a PFT, $\Gamma$, contained in  $N(e^2_\alpha)$ we define the {\bf degree of $\Gamma$} to be
\[
D(\Gamma) = A-C+E -Y_1-SW
\]
with $A, C, E, Y_1,$ and $SW$ as in Lemma \ref{lem:aceyswST}.
\end{definition}

\subsubsection{List of rigid GFTs with punctures at $c$'s.}  \label{sec:listGFTs}

%1.  DESCRIBE ALL TREES WITH PICTURES.

%2.  Prove existence of each flow tree.

%3.  Prove that these are the only flow trees.

For each term in the formula for $\partial_c c_{i,j}$ given in Theorem \ref{thm:dc} we identify an odd number of corresponding rigid GFTs.

\medskip

\noindent {\bf Term 1: $a^{+,+}_{i,m} c_{m,j}$ for $i < m < j$.}   The GFT begins by following the unique $(i,j)$-flow line that limits to $c_{i,j}$ as $t \rightarrow -\infty$ and passes the location of $c_{m,j}$.  When the flowline reaches $c_{m,j}$ a puncture ($Y_0$-vertex) occurs.  The non-constant branch after the puncture is the $(i,m)$-flow line  
beginning at $c_{m,j}$.  In view of the lexicographic ordering of the $c_{i,j}$ along the diagonal $x_1=x_2$, (this follows from Properties \ref{pr:Reeb2} and \ref{pr:Location1}), we have $c_{m,j} \in N_\e Q1(i,m)$, 
 i.e. $c_{m,j}$ is to the right and above $c_{i,m}$.  Therefore, by Lemma \ref{lem:1stQuad} (see also Lemma \ref{lem:holdovers}), this flowline limits to $a^{+,+}_{i,m}$ as $t \rightarrow + \infty$.  See Figure \ref{fig:Term1}.

\begin{figure}

\quad

\labellist
\small
\pinlabel $c_{i,j}$ [b] at 40 164
\pinlabel $c_{m,j}$ [tl] at 58 94
\pinlabel $a^{+,+}_{i,m}$ [tr] at -2 70
\pinlabel $c_{i,j}$ [br] at 164 52
\pinlabel $c_{m,j}$ [br] at 221 107
\pinlabel $a^{+,+}_{i,m}$ [bl] at 283 163
\endlabellist
\centerline{ \includegraphics[scale=.6]{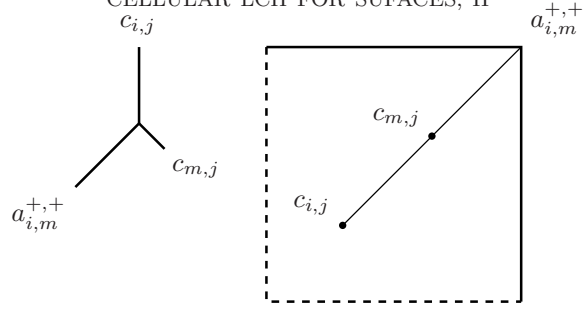} }
\caption{The domain and image of the trees from Term 1.}
%The arrows indicate the direction along edges of all $(i,j)$-flows (negative gradients) other than the $(k+2,k+1)$ flow.  
%(right) The partial flow trees in $A$ described in (2) of Lemma \ref{lem:RegionA}.  
\label{fig:Term1}
\end{figure}

\medskip

\noindent {\bf Term 2: $c_{i,m} a^{-,-}_{m,j}$ for $i<m<j$ with $\{m,j\} \cap \{k,k+2\}= \emptyset$.}   Begin by following the unique $(i,j)$-flow line from $c_{i,j}$ to $c_{i,m}$.  When the flow line reaches $c_{i,m}$ a puncture occurs.  After the puncture the non-constant branch is the $(m,j)$-flow line starting at $c_{i,m}$.  As stated in Lemma \ref{lem:holdovers}, the Lemma \ref{lem:14121} is applicable in this situation, and it shows that this flow line will pass into $[1/4,1/2]\times[1/4,1/2] \subset A$ where $A$ is the region from Lemma \ref{lem:AAinvariance}.  Once this occurs,  Lemma \ref{lem:AAinvariance} shows that the limit of this flow line can only be the Reeb chord $a^{-,-}_{m,j}$ as this is the only $(m,j)$-Reeb chord in $A$.  [Note that because of the restriction $\{m,j\} \cap \{k,k+2\}= \emptyset$, neither the $m$ or $j$ sheet has a cusp locus in $A$.  If one of $m$ or $j$ is $k+1$, then the Convention \ref{c:Sktilde} is relevant since the $S_{k+1}$ sheet
%relevant since one of the sheets $S_{m}$ or $S_{j}$ will 
becomes $\tilde{S}_k$ when the flow line crosses the $(k,k+2)$-cusp locus.]  See Figure \ref{fig:Term2}.

%beginning at $c_{m,j}$.  According to Lemma STAIRCASE, $c_{m,j}$ is in the $1$-st $(i,m)$-quadrant, i.e. is to the right and above $c_{i,m}$.  Therefore, this flowline limits to $a^{+,+}_{i,m}$ as $t \rightarrow + \infty$.

\begin{figure}
\labellist
\small
\pinlabel $c_{i,j}$ [b] at 16 162
\pinlabel $c_{i,m}$ [tr] at -2 94
\pinlabel $a^{-,-}_{m,j}$ [tl] at 57 71
\pinlabel $c_{i,j}$ [br] at 260 142
\pinlabel $c_{i,m}$ [br] at 245 126
\pinlabel $a^{-,-}_{m,j}$ [r] at 118 0
\endlabellist
\centerline{ \includegraphics[scale=.6]{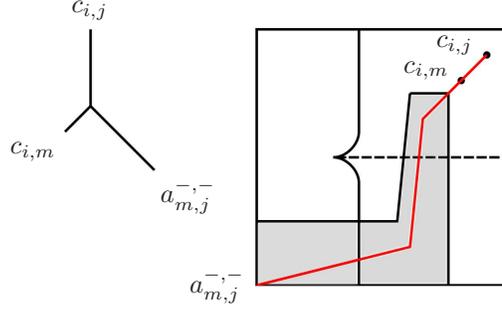} }
\caption{The domain and image (in red) of the trees from Term 2.}
%The arrows indicate the direction along edges of all $(i,j)$-flows (negative gradients) other than the $(k+2,k+1)$ flow.  
%(right) The partial flow trees in $A$ described in (2) of Lemma \ref{lem:RegionA}.  
\label{fig:Term2}
\end{figure}

\medskip

\noindent {\bf Term 3: $c_{i,k}\cdot 1$ for $i<k<j=k+2$.}
Begin with the unique $(i,k+2)$-flow line from $c_{i,k+2}$ to $c_{i,k}$, and then puncture at $c_{i,k}$.  By Lemma \ref{lem:14121} and Lemma \ref{lem:AAinvariance}, the $(k,k+2)$-flow line  that begins at $c_{i,k}$ must then enter and remain in $A$ where it can only end at an $e$-vertex along the lower half of the cusp edge where sheets $S_{k}$ and $S_{k+2}$ meet;  this produces a flow tree as in Figure \ref{fig:Term3}.

\begin{figure}
\labellist
\small
\pinlabel $c_{i,k+2}$ [b] at 16 162
\pinlabel $c_{i,k}$ [tr] at -2 94
\pinlabel $1$ [tl] at 57 71
\pinlabel $c_{i,k+2}$ [br] at 260 142
\pinlabel $c_{i,k}$ [br] at 245 126
\pinlabel $1$ [r] at 179 19
\endlabellist
\centerline{ \includegraphics[scale=.6]{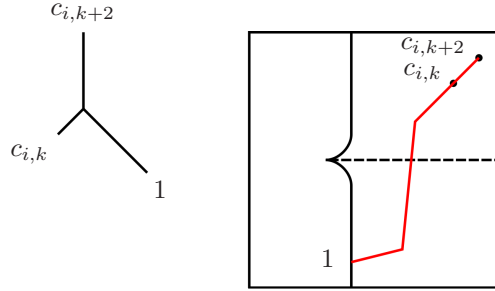} }
\caption{The domain and image (in red) of the trees from Term 3.}
%The arrows indicate the direction along edges of all $(i,j)$-flows (negative gradients) other than the $(k+2,k+1)$ flow.  
%(right) The partial flow trees in $A$ described in (2) of Lemma \ref{lem:RegionA}.  
\label{fig:Term3}
\end{figure}

\medskip

\noindent {\bf Term 4: $c_{i,k}\cdot 1$ for $i<k<j=k+1$.}

Again by Lemmas \ref{lem:14121} and Lemma \ref{lem:AAinvariance}, the $(k,k+1)$-flow line $\gamma$,that  starts at $c_{i,k}$ enters and remains in $A$ where it can only terminate at the $(k,k+2)$-cusp locus.

Consider the closed subset  $D \subset N(e^2_\alpha)$, consisting of the part of $N(e^2_\alpha)$ that sits to the right of the cusp locus with the upper right corner removed just below and to left of  $c_{i,k}$, and the lower right corner removed just above and to left of $\tilde{c}_{k+1,k+2}$.  More precisely,  remove 
\[
\{ (x_1,x_2) \in \widetilde{N}(e^2_\alpha) \, | \, x_{1} > \beta^U_{i,k} - \e, \, x_2 > \beta^R_{i,k}-\e\} \bigcup N(e^0_{+,+}),
\]  
and
\[
\{ (x_1,x_2) \in \widetilde{N}(e^2_\alpha) \, | \, x_{1} > \beta^D_{k+2,k+1} - \e, \, x_2 < \tilde{\beta}^R_{k+2,k+1}+\e\} \bigcup N(e^0_{+,-}).
\]  
Now, $D$ is homeomorphic to $D^2$, and the intersection of $\gamma$ and the $(k+2,k+1)$-switch flow line, $\lambda$, with $D$ are closed intervals with their endpoints on $\partial D$.  [That the intersections of these flow lines with $D$ consist of single arcs follows from Property \ref{pr:monotonicityI}.]  Moreover, the endpoints of $\gamma$, $A_1,A_2$,  and the endpoints of $\lambda$, $B_1,B_2$, appear cyclically in the order $A_iB_kA_jB_l$.  As $\gamma$ and $\lambda$ intersect transversally (by the $1$-regularity of $\tilde{L}$), standard differential topology shows that they intersect in an odd number of points, $w_1, \ldots, w_{2r+1}$.

To produce an odd number of GFTs:  Begin with the $(i,k+1)$-flow line from $c_{i,k+1}$ to $c_{i,k}$, and puncture at $c_{i,k}$.  Follow $\gamma$ until it reaches one of the $w_i$.  At $w_i$, a $Y_0$-vertex occurs so that the outgoing edges are a $(k,k+2)$-flow line, $\gamma_a$, and a $(k+2,k+1)$-flow line, $\gamma_b$.  Since $w_{i} \in A$, $\gamma_a$ must remain in $A$ and will thus terminate at an $e$-vertex along the $(k,k+2)$-cusp edge; $\gamma_b$ ends at a switch vertex at the $(k+2,k+1)$-switch point that is followed by a $(k,k+1)$-flow line that terminates at the cusp edge.

See Figure \ref{fig:Term4}.

\medskip

\begin{figure}
%\labellist
%\small
%\pinlabel $c_{i,k+1}$ [b] at 16 162
%\pinlabel $c_{i,k}$ [tr] at -2 94
%\pinlabel $1$ [tl] at 57 71
%\pinlabel $c_{i,k+1}$ [br] at 260 142
%\pinlabel $c_{i,k}$ [br] at 245 126
%\pinlabel $1$ [r] at 186 108
%\endlabellist
\centerline{ 
\labellist
\small
\pinlabel $c_{i,k+1}$ [b] at 16 162
\pinlabel $c_{i,k}$ [tr] at -2 94
\pinlabel $k,k+1$ [l] at 24 112
\pinlabel $1$ [tl] at 58 30
\pinlabel $1$ [tr] at 16 56
\pinlabel $c_{i,k+1}$ [br] at 260 142
\pinlabel $c_{i,k}$ [br] at 245 126
\pinlabel $1$ [r] at 186 100
\pinlabel $1$ [r] at 186 24
\pinlabel $\tilde{c}_{k+2,k+1}$ [t] at 252 24
\endlabellist
\includegraphics[scale=.6]{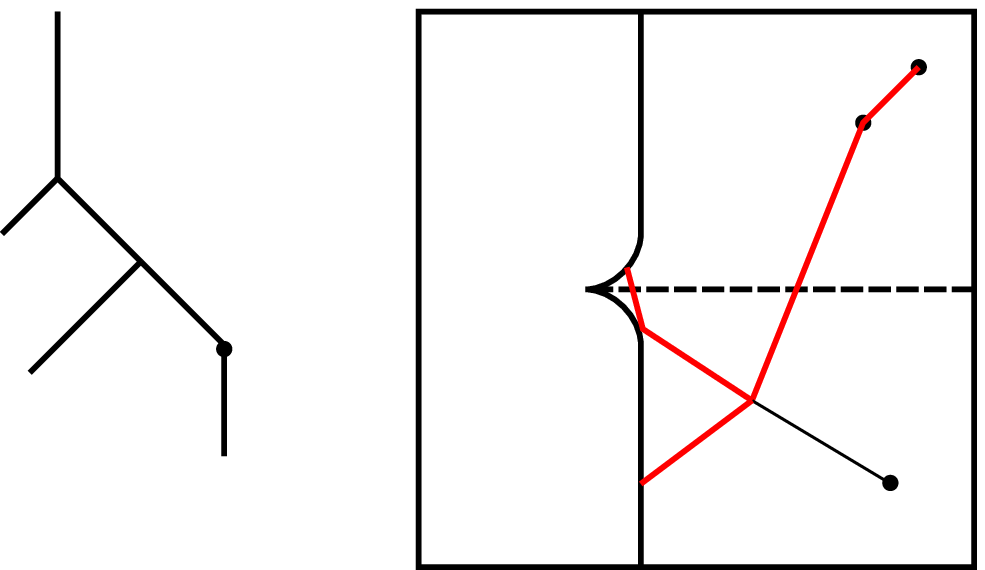} }
\caption{A tree from Term 4.  More intersections between the $(k,k+1)$-flow line starting at $c_{i,k}$ and the $(k+2,k+1)$-switch flow line may occur, resulting in an even number of additional trees.}
\label{fig:Term4}
\end{figure}

\noindent {\bf Term 5: $c_{i,k+2} a^{-,-}_{k+1,j}$ for $i<k+2<j$.}

We identify an odd number of GFTs, each of which has initial branch consisting of the $(i,j)$-flow line from $c_{i,j}$ to $c_{i,k+2}.$
%, and ending with a puncture at $c_{i,k+2}$.  
The $(k+2,j)$-flow line beginning at $c_{i,k+2}$ will enter $A$ (by Lemma \ref{lem:14121}).  Next, by Lemma \ref{lem:AAinvariance}  this flow line will proceed to terminate at a point on the $(k,k+2)$-cusp locus.  Along its way it will intersect the $(k+2,k+1)$ switch line transversally in an odd number of points, $w_1, \ldots, w_{2r+1}$.  [The reasoning is as in Term 4.]
%in a unique point, by Property \ref{pr:UniqueInt}.  
To construct GFTs $\Gamma_1, \ldots, \Gamma_{2r+1}$, after the puncture we follow the $(k+2,j)$-flow line starting at $c_{i,k+2}$ to one of the $w_s$ where a $Y_0$ occurs.  The branch that is a $(k+1,j)$ flow will then proceed to $a^{-,-}_{k+1,j}$ (again by Lemma \ref{lem:AAinvariance}), so a GFT of the form pictured in Figure \ref{fig:Term5} results.

\begin{figure}
\labellist
\small
\pinlabel $c_{i,j}$ [b] at 16 162
\pinlabel $c_{i,k+2}$ [tr] at -2 94
\pinlabel $k+2,j$ [l] at 24 112
\pinlabel $a^{-,-}_{k+1,j}$ [t] at 72 53 
\pinlabel $1$ [t] at 16 28
\pinlabel $c_{i,j}$ [br] at 260 142
\pinlabel $c_{i,k+2}$ [br] at 245 126
\pinlabel $1$ [r] at 186 100
\pinlabel $a^{-,-}_{k+1,j}$ [tr] at 120 0
\pinlabel $\tilde{c}_{k+2,k+1}$ [t] at 252 24
\endlabellist
\includegraphics[scale=.6]{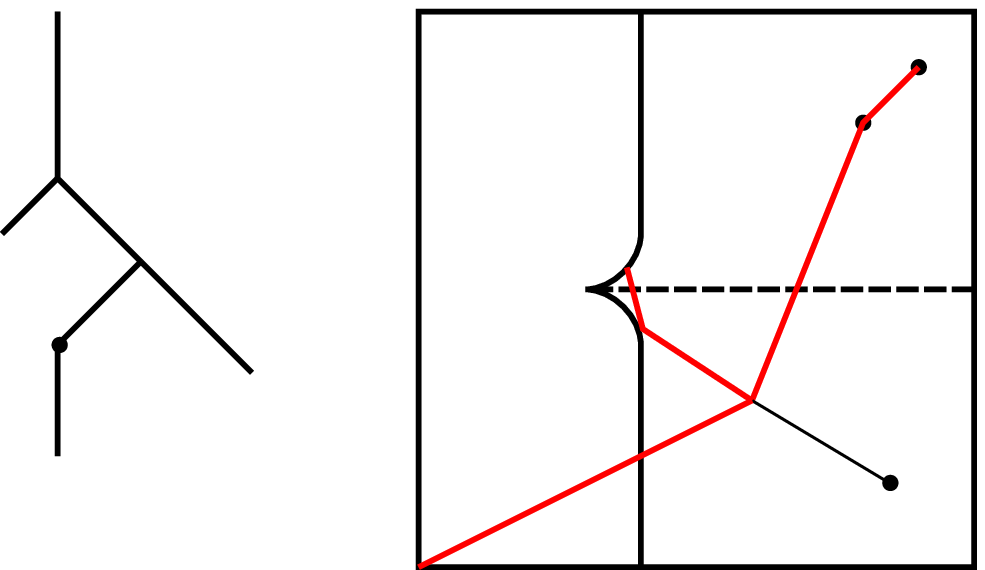} 
\caption{The domain and image of trees from Term 5.  More intersections between the $(k+2,j)$-flow line starting at $c_{i,k+2}$ and the $(k+2,k+1)$-switch flow line may occur, resulting in an even number of additional trees.}
%The arrows indicate the direction along edges of all $(i,j)$-flows (negative gradients) other than the $(k+2,k+1)$ flow.  
%(right) The partial flow trees in $A$ described in (2) of Lemma \ref{lem:RegionA}.  
\label{fig:Term5}
\end{figure}

\medskip

Note that by Lemma \ref{lem:aceyswST}, the GFTs identified in Terms 1-5 are all rigid.

\subsubsection{Establishing uniqueness}  \label{sec:estunique}

%\dr{Once exposition is completed, do a careful re-evaluation of which statements to call Lemma, Proposition, Theorem.}

The following Lemmas \ref{lem:RegionA} and \ref{lem:3Q} serve as a replacement for the $3$-rd Quadrant Lemma.

%Recall the regions $A'$ and $A$ appearing in Lemma \ref{lem:AAinvariance}.  

%\dr{Since all of these Lemmas have the same conclusion, could be best to convert them to one Proposition with a (1)-(3) hypothesis.}

\begin{lemma}  \label{lem:RegionA}
Any PFT, $\Gamma$, beginning at a point 
%with an $(i,j)$-flow 
in $A$  satisfies $D(\Gamma) \geq 0$.  Moreover, equality holds only for those PFTs that are subsets of the 
%$(k+1,k+2)$- or 
$(k+2,k+1)$- switch GFT beginning at a point on the initial edge, i.e. beginning with part of the 
%$(k+1,k+2)$- or 
$(k+2,k+1)$- switch flow line. 
\end{lemma}
\begin{proof}
 
Note that the cases of a PFT beginning with a $(i,j)$-flow in $A$ where $(i,j) = (k+2,k+1)$ or $(i,j) = (k+1,k+2)$ are covered by Lemmas \ref{lem:PFTrees} and \ref{lem:PFTreesk1k2}.  [Note that the PFTs identified in (1) and (2) of Lemma \ref{lem:PFTrees} are entirely disjoint from $A$, and the $(k+1,k+2)$-switch flow line is disjoint from $A$ via Lemma \ref{lem:Rswk1k2}.]

%$\{i,j\} \subset \{k,k+1,k+2\}$ is covered by Lemma \ref{lem:Aprime} since in this case the intersection of $A$ with the domain of $-\nabla F_{i,j}$ is contained in $A'$.  

Thus, we consider the case of a PFT, $\Gamma$, beginning with an $(i,j)$-flow starting in $A$ with $(i,j) \neq (k+1,k+2)$ and $(i,j) \neq (k+2,k+1)$, and  %$\{i,j\} \not \subset \{k,k+1,k+2\}$ 
we use induction on the number of $Y_0$-vertices
% and $Y_1$-
 in $\Gamma$.   
 
 The initial edge of $\Gamma$ remains in $A$ (by Lemma \ref{lem:AAinvariance}), so the first internal vertex of $\Gamma$ can only be a $Y_0$-vertex by the following:

\begin{lemma}  \label{lem:NoY1swinA} No $Y_1$ or switch vertex can have its image in $A.$
%\footnote{\ms{7/22/15: I might move Lemma \ref{lem:NoY1swinA} outside the proof of Lemma \ref{lem:RegionA} since it is used later} \dr{7/26: This could be done, but hasn't been yet.}}
\end{lemma}  
\begin{proof}
The only portion of the cusp locus that intersects $A$ is along the vertical segment $\{x_1=-3/8\}$ below the Swallowtail Barrier $P$.  For $i<k$ and $j=k+1$ or $j\geq k+3$, so that $S_i$ and $S_j$ are respectively above and below the $(k,k+2)$-cusp locus, Corollary \ref{cor:mnPQ} gives $-\grad_{x_1} F_{i,k}<0$ and $-\grad_{x_1} F_{k+2,j}<0$, so that no switches are possible.  Moreover, since $F_k$ and $F_{k+2}$ agree to first order along the $(k,k+2)$-cusp edge, we have 
 \[
 -\grad_{x_2}F_{i,j} = -\grad_{x_2}F_{i,k} -\grad_{x_2}F_{k+2,j} <0
 \]
 so that no $Y_1$ is possible.
%  Along this segment the $x_1$ component of $-\nabla F_{i,j}$ is negative by Property \ref{pr:monotonicityI} (which applies since $\{i,j\} \not \subset \{k,k+1,k+2\}$).\dr{WAS THE STATEMENT OF THIS PROPERTY ADJUSTED?}  Thus, the direction of $-\nabla F_{i,j}$ across the cusp locus forbids a switch or $Y_1$-vertex.  
\end{proof}

Now, in the base case where $\Gamma$ has no $Y_0$ vertices, we have that $\Gamma$ has no internal vertices at all and can only limit to an $(i,j)$-Reeb chord within $A$ or end at an $e$-vertex.  (If $i$ or $j$ is $k+1$, this Reeb chord will have an endpoint on $\widetilde{S}_k$.)   
The only $(i,j)$-Reeb chord in $A$ is $a^{-,-}_{i,j}$, so we have $D(\Gamma) = A- C + E - Y_1 - SW = 1$.  For the inductive step, assume $\Gamma$ has at least one $Y_0$-vertex, and let $\Gamma_1$ and $\Gamma_2$ denote the PFTs that begin with the two branches of $\Gamma$ that immediately follow the first vertex.  The image of this vertex must lie in $A$ by Lemma \ref{lem:AAinvariance}, so the inductive hypothesis applies to both $\Gamma_1$ and $\Gamma_2$.  Therefore, %using the notation $D(\Gamma_i) = A + K - sw - Y_1$, 
we have $D(\Gamma_i) \geq 0$ for $i=1,2$.  It is impossible that we have equality in both cases since at most one of the $\Gamma_i$ can begin with the $(k+2,k+1)$-switch flow line.
 %that %precede their respective switch points 
%would then have to cross one another at the location of the $Y_0$.  
%This is impossible because all $(k+1,k+2)$ flow lines are entirely above the crossing locus and all $(k+2,k+1)$ flow lines are entirely below the crossing locus.
  Therefore, $D(\Gamma) = D(\Gamma_1) + D(\Gamma_2) > 0$.
\end{proof}

\begin{lemma} \label{lem:3Q}
Let $i< m_1 \leq m_2 <j$.  Any PFT $\Gamma$ that begins with an $(m_2,j)$-flow line starting at a point $x \in N_\e Q3(i,m_1) \cap [1/2,1] \times [1/2,1]$ 
satisfies $D(\Gamma) > 0$.  
%Moreover, equality holds only for those PFTs that are subsets of the $(k+1,k+2)$- or $(k+2,k+1)$- switch GFT beginning at a point on the initial edge, i.e. beginning with part of the $(k+1,k+2)$- or $(k+2,k+1)$- switch flow line. 
%has $D(\Gamma) \geq 1$ unless $\Gamma$ is a subset of the GFT from identified in Lemma \ref{pr:PFTreesk1k2} that consists of the $(k+1,k+2)$-switch flow line followed by a switch and a $(k,k+2)$ edge that terminates at an $e$-vertex, in which case $D(\Gamma) = 0$.
\end{lemma}

\begin{proof}

Applying Lemma \ref{lem:14121}, repeatedly if necessary, we can locate some number of points $x_1, \ldots, x_r$ in the domain of $\Gamma$ with the following properties.  
\begin{enumerate}
\item  For some $m_2 = s_0 < s_1 < \ldots < s_r= j$, the branch of the tree that $x_i$ sits on parametrizes an $(s_{i-1}, s_{i})$-flow.
\item The image of each $x_i$ is in $[1/4,1/2] \times [1/4,1/2]$.
\item We have $\Gamma = \gamma_0 \sqcup \Gamma_1 \sqcup \cdots \sqcup \Gamma_r$ where $\Gamma_i$ denotes the PFT consisting of the portion of $\Gamma$ starting at and below $x_i$ and $\gamma_0$ is some portion of the tree that does not have any output vertices and has image contained in $N_{\e}Q3(i,m_1) \cap [1/4,1]\times [1/4,1]$. 
\end{enumerate}
 Note in particular that $\gamma_0$ can only have $Y_0$ vertices because the cusp locus is disjoint from $[1/4,1]^2$.  See Figure \ref{fig:x1x2xn}. [One locates the $x_i$ as follows.  Suppose the initial branch of $\Gamma$ 
 %the tree below the puncture at $c_{r,s}$ 
 leaves $N_{\e}Q3(i,m_1)$.  Then, we apply Lemma \ref{lem:14121} to locate a suitable point $x_1$ with image in $[1/4,1/2]^2$, and take $r=1$.  If the initial branch of the tree does not leave $N_{\e}Q3(i,m_1)$, then it can only end at a $Y_0$-vertex.  In this case, we treat the two outgoing branches of the tree at the $Y_0$ inductively.]  

Since $x_i \in [1/4,1/2]\times [1/4,1/2] \subset A,$ 
 we can now apply Lemma  \ref{lem:RegionA} to each of the $\Gamma_i$ to conclude that for all $i$, $D(\Gamma_i) > 0$.  [None of the $\Gamma_i$ begins with the $(k+2,k+1)$-switch flow line since we are above the crossing locus.]
% Moreover, the Lemma tell us that the inequality must be strict except possibly in the case that the branch that contains $x_i$ is part of the $(k+1,k+2)$-switch flow line.  This can happen for at most one $i$ (by property (1) of the $x_i$).  The result follows immediately if $n=1$, and if $n \geq 2$ 
Thus, we have $D(\Gamma') = \sum_{i} D(\Gamma_i) \geq 1$ as desired.  
%Now, if $n = 2$, then $\gamma_0$ consists of the $(s,j)$ flowline that starts at $c_{r,s}$ followed by a $Y_0$.  According to 
%Proposition \ref{lem:k1k2No}
%Property \ref{pr:k1k2No} (1), it is impossible that either of the branches coming out of the $Y_0$ is part of the $(k+1,k+2)$-switch flow, so again we have $D(\Gamma') = D(\Gamma_1)+D(\Gamma_2) \geq 2$.  
%f $n=1$, then  $x_0$ simply sits on the initial edge of $\Gamma'$.  This edge cannot be part of the $(k+1,k+2)$ switch flow by Property \ref{pr:k1k2No} (1), so it follows that $D(\Gamma') = D(\Gamma_1) \geq 1$.  

\end{proof}

\begin{figure}
\labellist
\small
\pinlabel $\gamma_0$ [t] at 164 168
\pinlabel $x_1$ [b] at 26 46
\pinlabel $x_2$ [b] at 98 46
\pinlabel $x_3$ [b] at 192 46
\pinlabel $\Gamma_1$ [t] at 8 -2
\pinlabel $\Gamma_2$ [t] at 116 -2
\pinlabel $\Gamma_3$ [t] at 212 -2
\endlabellist
\centerline{ \includegraphics[scale=.6]{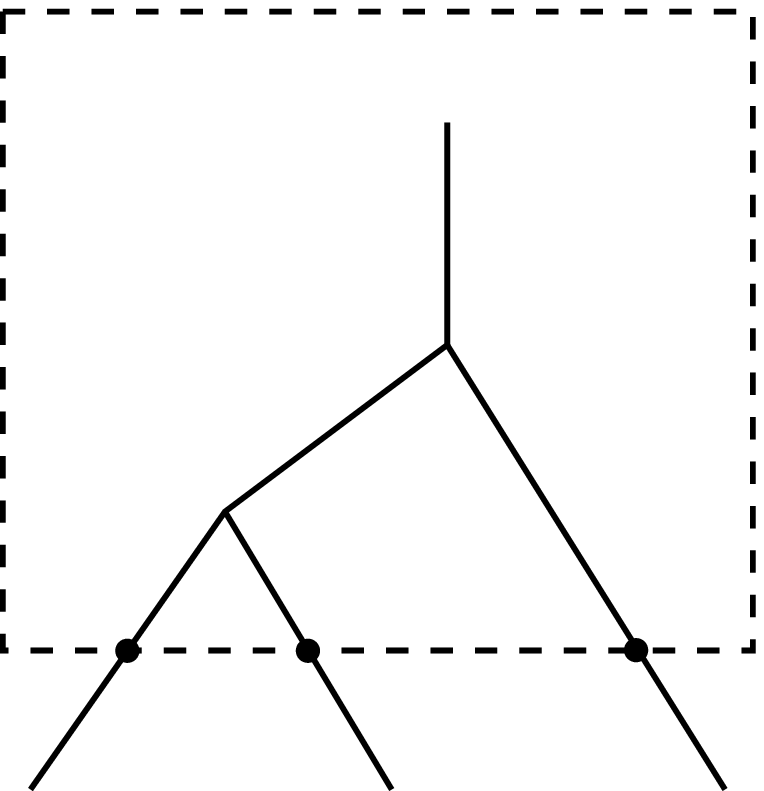} }
\caption{Decomposing the partial flow tree $\Gamma' = \gamma_0 \sqcup \Gamma_1 \sqcup \cdots \sqcup \Gamma_r$.}
%The arrows indicate the direction along edges of all $(i,j)$-flows (negative gradients) other than the $(k+2,k+1)$ flow.  
%(right) The partial flow trees in $A$ described in (2) of Lemma \ref{lem:RegionA}.  
\label{fig:x1x2xn}
\end{figure}

\begin{proof}[Conclusion of Proof of Theorem \ref{thm:dc}]  Let $\Gamma$ be a rigid GFT beginning at $c_{i,j}$ that contains at least one puncture at a Reeb chord of the form $c_{r,s}$.  (If $\Gamma$ instead contains a puncture at $\tilde{c}_{k+2,k+1}$, then its contribution to $\partial_c c_{i,j}$ is absorbed into the term $x$ from (\ref{eq:cterms})).  
By Lemma \ref{lem:aceyswST},
\[
D(\Gamma) = A- C +E- Y_1 -SW= 0.
\]
We will show that $\Gamma$ is one of the GFTs identified in Terms 1-5.

First, choose one of the punctures of $\Gamma$ at a Reeb chord of the form $c_{r,s}$ and consider the PFT, $\Gamma'$, consisting of the portion of $\Gamma$ beginning just below this puncture. 

\medskip

\noindent \textbf{Step 1.}  The PFT, $\Gamma'$, has $D(\Gamma') \geq 1$ with equality if and only if $\Gamma'$ agrees with the portion of one of the trees described in Terms 1-5 that follows the puncture.

\medskip

We verify Step 1 in two cases.

\medskip

\textbf{Case 1.}  For some $i<r<s$, an $(i,s)$-flow becomes an $(i,r)$-flow at the puncture.

\medskip

In this case, the $1$-st Quadrant Lemma applies (see Lemmas \ref{lem:holdovers} and \ref{lem:1stQuad}).  This lemma states that the image of $\Gamma'$  is contained in $N_\e Q1(r,s)$, so all internal vertices of $\Gamma'$ must be $Y_0$'s.  Thus, $D(\Gamma')$ agrees with the number of endpoints of $\Gamma'$.  If $\Gamma'$ has only one endpoint, then it is simply a portion of the $(i,r)$-flow line from $c_{r,s}$ to $a^{+,+}_{i,r}$.  In this case, $\Gamma'$ agrees with the post-puncture part of a tree from Term 1.  

\medskip

\textbf{Case 2.}  For some $r<s<j$, an $(r,j)$-flow becomes an $(s,j)$-flow at the puncture.

\medskip
Lemma \ref{lem:3Q} applies to show $D(\Gamma') \geq 1$.  

%%%%%% Moreover, the Lemma \ref{lem:RegionA} tells us that the inequality must be strict except possibly in the case that the branch that contains $x_i$ is part of the $(k+1,k+2)$-switch flow line.  This can happen for at most one $i$ (by property (1) of the $x_i$).  Thus, if $p \geq 3$ we must have $D(\Gamma') = \sum_{i} D(\Gamma_i) \geq 2$.  Now, if $p = 2$, then $\gamma_0$ consists of the $(s,j)$ flowline that starts at $c_{r,s}$ followed by a $Y_0$.  According to 
%Proposition \ref{lem:k1k2No}
%%%%%%Lemma \ref{lem:k1k2NoEtc} (1), it is impossible that either of the new branches that start at the $Y_0$ is part of the $(k+1,k+2)$-switch flow line, so again we have $D(\Gamma') = D(\Gamma_1)+D(\Gamma_2) \geq 2$.  Finally, if $p=1$, then  $x_1$ simply sits on the initial edge of $\Gamma'$.  This edge cannot be part of the $(k+1,k+2)$-switch flow line by Lemma \ref{lem:k1k2NoEtc} (1), so it follows that 

It remains to establish the claim about equality in Step 1.  Suppose that $D(\Gamma')=1$.  As in Proof of Lemma \ref{lem:3Q}, we can write $\Gamma' = \gamma_0 \sqcup \Gamma_1 \sqcup \cdots \sqcup \Gamma_p$, where the decomposition has the properties discussed in Proof of Lemma \ref{lem:3Q}.  
Again, we can now apply Lemma \ref{lem:RegionA} to each of the $\Gamma_i$ to conclude that for all $i$, $D(\Gamma_i) > 0$.  Thus, we must have  $p=1$ so the initial edge of $\Gamma'$ is in $[1/4,1/2]^2$ at the image of $x_1$.

First, {\bf suppose $\Gamma'$ has no internal vertices}.   Then, $\Gamma'$ is a single $(s,j)$-flow line.  It must be the case that $(s,j) \neq (k+1,k+2)$ (by Lemma \ref{lem:PFTreesk1k2}), so $\Gamma_1$ has image contained in $A$ (by Lemma \ref{lem:AAinvariance}).  Since this flow line, $\Gamma_1$, is an actual PFT, it must end at a Reeb chord or at an $e$-vertex.  

\begin{itemize}
\item If $\Gamma_1$ ends at a Reeb chord,  the only possible Reeb chord in $A$ is $a^{-,-}_{s,j}$.  Therefore, it must be that $\{s,j\} \cap \{k,k+2\}= \emptyset$ because otherwise $a^{-,-}_{s,j}$ does not exist. 
%the flowline would cease to exist at some point of the $(k,k+2)$-cusp locus.  
Thus, $\Gamma'$ agrees with the post-puncture part of a tree from Term 2 (as ordered in Section \ref{sec:listGFTs}).  

\item If $\Gamma_1$ ends at an $e$-vertex, then $\Gamma_1$ must be a $(k,k+2)$-flow line terminating somewhere along the $(k,k+2)$-cusp edge.  [The $(k,k+1)$-cusp locus is disjoint from $A$.] In this case, $\Gamma'$ agrees with the post-puncture part of a tree from Term 3.
\end{itemize}

%\begin{itemize}

%\item  If $\{s,j\} \not \subset \{k,k+1,k+2\}$, then $\Gamma_1$ has image contained in $A$ (by Lemma \ref{lem:AAinvariance}).  Since this flow line is an actual PFT, it must terminate at a Reeb chord, and the only possible Reeb chord in $A$ is $a^{-,-}_{s,j}$.  Therefore, it must be that $\{s,j\} \cap \{k,k+2\}= \emptyset$ because otherwise $a^{-,-}_{s,j}$ does not exist. 
%%%%%%the flowline would cease to exist at some point of the $(k,k+2)$-cusp locus.  
%Thus, $\Gamma'$ agrees with the post-puncture part of a tree from Term 2.  

%\item If instead it were the case that $\{s,j\} \subset \{k,k+1,k+2\}$, then $\Gamma'$ could only be either a single $(k,k+1)$-flow line terminating somewhere along the upper half of the cusp edge, or a $(k,k+2)$ flow line terminating somewhere along the lower half of the cusp edge.  In these cases, $\Gamma'$ agrees with the post-puncture part of trees discussed in Term 4, Case 1 (i) or Term 3 respectively.  [We cannot have $(s,j) = (k+1,k+2)$ since a PFT starting with a $(k+1,k+2)$-flow line in $A'$ must have an internal vertex by Proposition \ref{prop:PFTreesk1k2}.]

%\end{itemize}

Next, {\bf suppose that $\Gamma'$ has at least $1$ internal vertex}.  
The first internal vertex is located somewhere on the initial edge of  $\Gamma_1$ which is an $(s,j)$-flow line with $(s,j) \neq (k+1,k+2)$ (by Lemma \ref{lem:PFTreesk1k2}).  Thus,  the vertex is located somewhere in $A$ (by Lemma \ref{lem:AAinvariance}) and is a $Y_0$ (by Lemma \ref{lem:NoY1swinA}).

Consider the PFTs $\Gamma_a$ and $\Gamma_b$ that begin with the branches of $\Gamma'$ that follow this first $Y_0$ vertex, labeled so the lower sheet of $\Gamma_a$ agrees with the upper sheet of $\Gamma_b$.  By Lemma \ref{lem:RegionA}, it is the case that $D(\Gamma_a) \geq 0$ (resp. $D(\Gamma_b) \geq 0$)  with equality only if $\Gamma_a$ (resp. $\Gamma_b$) is a subset of the $(k+2,k+1)$-switch GFT starting on the $(k+2,k+1)$-switch flow line.  Since $1 = D(\Gamma') = D(\Gamma_a) + D(\Gamma_b)$, one of the following two cases holds:

\medskip

{\bf Case:  $\Gamma_a$ is a subset of the $(k+2,k+1)$-switch flow.}  Then, we have $s = k+2$ and $j>k+2$, and 
%Lemma \ref{lem:AAinvariance} shows that 
$\Gamma_b$ would begin with a $(k+1,j)$-flow line, starting in the part of $A$ that is below the crossing locus.    We claim that there cannot be any internal vertices in $\Gamma_b$.  [{\it Proof:} Suppose this is not the case and consider the first such vertex, $z$, which (due to the location of the image in $A$) must also be a $Y_0$.  
%Note that the image of $z$ cannot be at a point on the $(k+1,k+2)$-switch flow line, since Property \ref{pr:x1axisST} (2) the initial edge of $\Gamma_b$ from passing into the region where $F_{k+1} > F_{k+2}$.  
Now, Lemma \ref{lem:RegionA} shows that the PFTs, $\Gamma_{z,1}$ and $\Gamma_{z,2}$, beginning with the outgoing edges at $z$ both have $D(\Gamma_{z,i}) \geq 1$.  (Neither one can start with a $(k+2,k+1)$-flow line since $\Gamma_b$ began with a $(k+1,j)$-flow line.)  This implies 
\[
D(\Gamma') = D(\Gamma_a)+ D(\Gamma_b) = D(\Gamma_a)+ D(\Gamma_{z,1}) +D(\Gamma_{z,2}) \geq  0 + 1+1=2,
\]  which contradicts that $D(\Gamma') = 1$.] 
%prevents the $Y_0$ from being located above the line $x_2=0$, and the $(k+1,k+2)$-switch flow line is entirely above this line by \dr{improved version of} Property \ref{pr:k1k2No}.

%(***) \footnote{\dr{Why can't the $(k+1, j)$ flow cross to back above the crossing locus and then have another $Y_0$ at the $(k+1,k+2)$ switch line?  In $\tilde{L}$ this is prevented by $\partial F_{k+2,j}>0$ as shown in Claim 2 of no Reeb chords proof.}  Lemma {lem:x1axisST} shows that $d_{x_2}F_{i,j} >0$ along crossing locus for $(i,j)\neq(k,k+1)$.  Also, true in $\tilde{L}$ by construction section.  }

Therefore, $\Gamma_b$ must be a single flow line that can only end at $a^{-,-}_{k+1,j}$.  It follows that $\Gamma'$ agrees with the post-puncture portion of a flow tree from Term 5.

\medskip

{\bf Case:  $\Gamma_b$ is a subset of the $(k+2,k+1)$-switch flow.}  %In this case, it is impossible that $D(\Gamma_a)=1$.
Then, $\Gamma_a$ would have to begin with an $(s,k+2)$-flow line for $s \leq k$.  

First, {\bf suppose $\Gamma_a$ has no internal vertices.}  Then, by Lemma \ref{lem:AAinvariance} the $(s,k+2)$-flow line would have to end at the $(k,k+2)$-cusp edge, 
 and could only be part of a PFT if $s=k$.  Then, the initial branch of $\Gamma'$ that precedes $\Gamma_a$ and $\Gamma_b$ would be the $(k,k+1)$-flow line that starts at $c_{r,k}$.  Thus, $\Gamma'$ would agree with the post-puncture part of a GFT from Term 4.

%However, in the description of Term 4, it was shown that this flowline remains above the $(k+1,k+2)$-crossing locus.  This contradicts the assumption that $\Gamma_b$ is a subset of the $(k+2,k+1)$-switch flow.
%and it terminates on the lower half of the cusp edge.  In this case, $\Gamma'$ agrees with the post-puncture part of the GFT from \footnote{Actually, we now know this doesn't happen because of Lemma {lem:kk1}.} from Term 4, Option 2. 

Finally, {\bf suppose $\Gamma_a$ has at least one internal vertex.}  Applying Lemmas \ref{lem:AAinvariance} and \ref{lem:NoY1swinA}, we conclude that the first internal vertex of $\Gamma_a$ must be a $Y_0$ in $A$.  Neither outgoing branch of this $Y_0$ could begin with a $(k+2,k+1)$-flow line (since the incoming branch is an $(s,k+2)$-flow line), so Lemma \ref{lem:RegionA} would give $D(\Gamma') >1$, a contradiction.

% located at a point on the $(k+1,k+2)$-switch flow line, but Property \ref{pr:x1axisST} (2) prevents $\Gamma_a$ from passing into the region where $F_{k+1} > F_{k+2}$ as $t$ increases.  [Initially, $\Gamma_a$ is below the crossing locus since $\Gamma_b$ is a subset of the $(k+2,k+1)$-switch flow line.]

\medskip

\noindent \textbf{Step 2.}  Above the puncture at $c_{r,s}$, $\Gamma$ must consist of a single edge that limits to $c_{i,j}$.  

\medskip

%We have shown above that any $1$-st or $3$-rd quadrant lemma flow has degree $\geq 1$, so just apply this as in Step 2 from Type (1)-(12) proof.  

This is similar to Step 2 from the uniqueness proof of Theorem \ref{thm:SquareComp}, although the implementation is more involved.  

Consider the path, $\beta$, within the domain of $\Gamma$, that starts just above the puncture at $c_{r,s}$ and follows edges in a manner that is reverse to their orientation until the the input vertex (assumed to be at $c_{i,j}$) is reached.  In other words, $\beta$, consists of several consecutive edges $B_1, \ldots, B_N$ of $\Gamma$ so that, with respect to the orientation of $\Gamma$, the endpoint of $B_1$ is at $c_{r,s}$; the initial point of $B_i$ is the endpoint of $B_{i+1}$ for $1 \leq i \leq N-1$; and $B_N$ limits to $c_{i,j}$ as $t \rightarrow - \infty$.  Let $y_1, \ldots, y_{N-1}$ denote the internal vertices where $B_{i}$ and $B_{i+1}$ meet for $1 \leq i \leq N-1$.  See Figure \ref{fig:B1B2}.

\begin{figure}
\labellist
\small
\pinlabel $c_{r,s}$ [t] at 32 0
\pinlabel $B_1$ [br] at 56 43
\pinlabel $B_2$ [bl] at 57 119
\pinlabel $B_{N-1}$ [br] at 92 240
\pinlabel $B_N$ [l] at 117 278
\pinlabel $\Gamma_1$ [bl] at 105 64
\pinlabel $\Gamma_2$ [br] at 7 126
\pinlabel $\Gamma_{N-1}$ [bl] at 134 244
\pinlabel $c_{i,j}$ [br] at 112 304
\pinlabel $\vdots$ [t] at 113 40
\pinlabel $\vdots$ [t] at 0 102
\pinlabel $\vdots$ [t] at 137 231
\pinlabel $\vdots$ at 70 202
\endlabellist
\centerline{ \includegraphics[scale=.6]{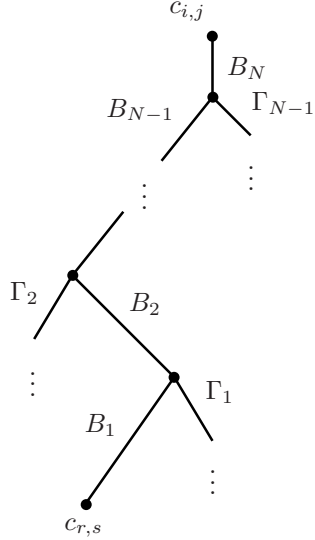} }
\caption{The sequence of edges $B_1, B_2, \ldots, B_N$ and the partial flow trees $\Gamma_1, \ldots, \Gamma_{N-1}$.}
%The arrows indicate the direction along edges of all $(i,j)$-flows (negative gradients) other than the $(k+2,k+1)$ flow.  
%(right) The partial flow trees in $A$ described in (2) of Lemma \ref{lem:RegionA}.  
\label{fig:B1B2}
\end{figure}

\begin{lemma} \label{lem:step2tail}
The following properties hold:
\begin{enumerate}
\item  The image of each $B_i$ is contained in $[1/2,3/4]\times[1/2,3/4]$.
\item  Each $y_i$ is a $Y_0$-vertex.
\item  At each $y_i$ for $1 \leq i \leq N-1$, let $\Gamma_i$ denote the PFT that begins with the edge of $\Gamma$ at $y_i$ that is not $B_{i}$ or $B_{i+1}$.  Then, $D(\Gamma_i) \geq 1$.
% unless $\Gamma_i$ is a subset of the $(k+1,k+2)$-switch GFT that begins with part of the $(k+1,k+2)$-switch flow line, in which case $D(\Gamma_i) = 0$.
\end{enumerate}
\end{lemma}

\begin{proof}[Proof of Lemma \ref{lem:step2tail}]
Recall the notation
\[
\sq(c_{m_1,n_1}, c_{m_2,n_2}) = N_\e Q1(m_1,n_1) \cap N_\e Q3(m_2,n_2).
\]

We establish the following statement:

\medskip

Suppose that for some $1\leq i \leq n-1$, $B_i$ satisfies the following.

\begin{enumerate}
\item[(A)]  There exists $l,l',m,m'$ satisfying $l < l' \leq m$, and  $l\leq m' <m$ such that $B_i$ is an $(l,m)$-flow line and has its endpoint in $\sq(c_{l,l'},c_{m',m})$.
% or in $\sq(c_{l,q},c_{p,q})$.
\end{enumerate}
\medskip
Then, (1), (2), and (3) hold for $B_i$, $y_i$ and $\Gamma_i$, and $B_{i+1}$ also satisfies (A).

\medskip

%Since the endpoint of $B_1$ is at $c_{r,s}$ and $B_1$ is either an $(r_1,s)$-flow or a $(r,s_1)$ for some $r_1 < r$ or $s <s_1$, it is clear that $B_1$ satisfies (A).  

It is straightforward that $B_1$ satisfies (A). [The endpoint of $B_1$ is at $c_{r,s}$, and $B_1$ is either an $(r_1,s)$-flow or a $(r,s_1)$ for some $r_1 < r$ or $s <s_1$.  In the first case, put $l = r_1$, $m' =r$, and $l'=m =s$; in the second case, put $m'=l = r$, $l' =s$, and $m=s_1$.]
Thus we can apply the statement repeatedly to deduce the Lemma.  

To prove the statement, observe that $-\nabla F_{l,m}$ points outward along all four edges of $\sq(c_{l,l'},c_{m',m})$ %(resp. $\sq(c_{l,q},c_{p,q})$) 
by Property \ref{pr:monotonicityI} (since $(l,m)$ is between $(l,l')$ and $(m',m)$ with respect to lexicographic order).  Thus, if the endpoint of $B_i$ is in $\sq(c_{l,l'},c_{m',m})$, %(resp. $\sq(c_{l,q},c_{p,q})$), 
then $B_i$ remains there as $t$ decreases.  Thus, (1) holds since $\sq(c_{l,l'},c_{m',m}) \subset [1/2,3/4]^2$.   Next, (2) holds since the cusp locus is disjoint from $[1/2,3/4]^2$.  Now, consider cases
\begin{enumerate}
\item[Case 1.]  $\Gamma_i$ begins with a $(m,p)$-flow line for some $m < p$.   Since %both of the squares under consideration are contained in either $N_\e Q3(l,q)$ or 
$\sq(c_{l,l'},c_{m',m}) \subset N_\e Q3(m',m)$, we can apply Lemma \ref{lem:3Q} to deduce that $D(\Gamma_i) \geq 1$.
 %unless $\Gamma_i$ is part of the GFT that begins with the $(k+1,k+2)$-switch flow line.  
 Moreover, $B_{i+1}$ then begins with an $(l,p)$-flow line in $\sq(c_{l,l'}, c_{m',m}) \subset \sq(c_{l,l'}, c_{m',p})$, so $B_{i+1}$ satisfies (A). 
%or $\sq(c_{l,q}, c_{l,q_1})$ 
\item[Case 2.]  $\Gamma_i$ begins with a $(h,l)$-flow line for some $h < l$.   Since %both of the squares under consideration are contained in either $N_\e Q3(l,q)$ or 
$\sq(c_{l,l'},c_{m',m}) \subset N_\e Q1(l,l')$, we can apply Lemmas  \ref{lem:1stQuad} and \ref{lem:holdovers} to deduce that $D(\Gamma_i) \geq 1$ (all endpoints are at $a^{+,+}$ Reeb chords and, since the image of $D(\Gamma_i)$ is contained in a region that is disjoint from the cusp locus,
%in addition, the proof\footnote{\dr{2-25: Should add this into the statement of that Lemma.}} of Lemma \ref{lem:1stQuad} shows that 
all internal vertices of $\Gamma_i$ are $Y_0$'s).  Moreover, $B_{i+1}$ then begins with an $(h,m)$-flow line in $\sq(c_{l,l'}, c_{m',m}) \subset \sq(c_{h,l'}, c_{m',m})$, so $B_{i+1}$ satisfies (A).
\end{enumerate}

\end{proof}

We now deduce Step 2 from  Lemma \ref{lem:step2tail}.  
%From (1) and (2), we have that at most one of the $\Gamma_i$ can begin with a $(k+1,k+2)$-flow line.  [To verify, note that if $B_{i}$ is an $(r_i,s_i)$-flow, than for $1 \leq i \leq N-1$ we have either $r_{i+1}=r_i$ and $s_i<s_{i+1}$ or $r_{i+1}< r_i$ and $s_{i}=s_{i+1}$.  In particular, the $(r_i,s_i)$ are all distinct.]  Therefore, 
Using Step 1, we can compute
\begin{align*}
0= D(\Gamma) & = D(\Gamma') -1 + \sum_{i=1}^{N-1} D(\Gamma_i)\\
  & \geq 1-1 + (N-1), 
\end{align*}
Thus, we must have $N=1$ so that the edge $B_1$ runs from $c_{r,s}$ to its limit at $c_{i,j}$ without internal vertices.  This is Step 2.  Moreover, we see that $D(\Gamma') =1$ must hold, so that Step 1 and Step 2 together show that the portions of $\Gamma$ below and above the puncture at $c_{r,s}$ agree with a particular GFT from Terms 1-5.   This completes the proof of Theorem \ref{thm:dc}.

\end{proof}

\subsection{Flow trees without punctures at $c_{i,j}$'s} \label{ssec:Nocij}
With $\partial_c c_{i,j}$ understood, we turn now to the remaining terms in $\partial c_{i,j}$.  To this end, let 
\[
\partial_b c_{i,j} = \sum_{\Gamma} w(\Gamma)
\]
with the sum over rigid GFTs beginning at $c_{i,j}$ that do not have any outputs at Reeb chords of the form $c_{r,s}$ or $\tilde{c}_{k+2,k+1}$.  In the remainder of \ref{ssec:Nocij}, we work to obtain  a characterization of these GFTs that will allow computation of $\partial_b c_{i,j}$. 
%a more useful characterization of the GFTs that contribute to $\partial_b c_{i,j}$.  
See Proposition \ref{prop:dbc}.  The computation of $\partial_b c_{i,j}$ is then carried out in Section \ref{ssec:enumerbtrees}.

%Recall that in (\ref{eq:dcdb}) we have written $\partial c_{i,j}= \partial_c c_{i,j} + \partial_b c_{i,j} + \partial_x c_{i,j}$.  With the calculation of $\partial_c c_{i,j}$ now complete, we turn to the $\partial_b c_{i,j}$ term that is a sum over GFTs that begin at $c_{i,j}$ and do not have output vertices at other $c$ Reeb chords or at $\tilde{c}_{k+2,k+1}$ or $\tilde{b}^R_{k+2,k+1}$.  After some preliminary work, we obtain in Proposition \ref{prop:dbc}, below, a characterization of these GFTs that will allow computation of $\partial_b c_{i,j}$.  

\begin{lemma} \label{lem:swPFT}
No edge in a partial flow tree can have both of its endpoints at switch vertices.
%\item For $k+2<j$, no partial flow tree can have a switch at the $(k+1,j)$-switch point.  PROBABLY UNTRUE
%\item  Any partial flow tree, $\Gamma$, whose initial edge is the  $(k,j)$-flowline beginning at the $(k+1,j)$-switch point has $D(\Gamma) \geq 2$.  (We do not include the $(k+1,j)$-switch in the degree calculation.)   
%\item The $(k,k+2)$ (resp. $(k,k+1)$) flowline beginning at the $(k+1,k+2)$-switch (resp. $(k+2,k+1)$-switch) terminates along the $(k,k+2)$ (resp. $(k,k+1)$) cusp edge.  Any other partial flow tree, $\Gamma$, starting with this flowline has $D(\Gamma) \geq 2$.
\end{lemma}

\begin{proof}
Note that by Property \ref{pr:switches} edges that end at a switch vertex can only be either $(i,k)$-, $(k+1,j)$-, $(k+1,k+2)$-, or $(k+2,k+1)$-flow lines for some $i<k$ or $k+2<j$;  edges that begin at a switch point can only be $(i,k+1)$-, $(k,j)$-, $(k,k+2)$-, and $(k,k+1)$-flow lines.  Recall, that within a Type (13) square, a single edge of a PFT may change from being an $(i,j)$-flow line to an $(i',j')$-flow line if it has an endpoint on $\tilde{S}_k, S_{k+1},$ or $S_{k+2}$ and crosses the portion of the cusp locus where these sheets are identified.  Thus, an index that begins as $k+1$ or $k+2$ could possibly become  $k+2$ or $k+1$ during the course of the edge.  However, even with this point in mind, considering the indices of edges  that may begin or end with a switch vertex shows that no edge can both begin and end with a switch vertex.

%Thus now edge can both begin and end at a switch point.

%by invariance of $A$, it seems we just need to show the flow does not terminate at the upper branch of the cusp locus.  Extending the $\partial_{x_2} F_{i,j} >0$ to the horizontal line through the switch point would do it.  We already have this in $\tilde{L}$, so maybe the swallowtail barrier can be used to show that the line must pass into $LD$ where flow is monotonic decreasing in both coordinates (double check with properties of $\widehat{f}$.
%May be easier to show that $\partial_{x_1}>0$ in $\tilde{L} \setminus O_2$.

%would be nice to have a similar picture to Figure 5 for the $(i,k+1)$-flow.  Then, combine this with invariance of $A$. Or, maybe the swallowtail barrier can be used.

%and $(k+1,j)$ flows

\end{proof}

\begin{lemma} \label{lem:AprimeNew}
%\begin{enumerate}
%\item Any $(k,k+1)$- or $(k,k+2)$- flow line that begins in $A'$ remains in $A'$ as long as it is defined.  
 For $\{i,j\} \subset \{k,k+1,k+2\}$, any PFT, $\Gamma$, starting with an $(i,j)$-flow in $L(1/2)$  satisfies $D(\Gamma) \geq 0$.  Moreover, equality holds only for those PFTs that are subsets of the $(k+1,k+2)$- or $(k+2,k+1)$- switch GFT beginning at a point on the initial edge, i.e. beginning with part of the $(k+1,k+2)$- or $(k+2,k+1)$- switch flow line. 
 %and terminate at an $e$-vertex (as in (3) of Proposition \ref{prop:PFTrees}).
%\end{enumerate}
\end{lemma}
\begin{proof}
%(1):  This follows immediately From Lemma \ref{lem:T13D12}.

We establish the result by induction on the number of $Y_0$-vertices in the flow tree.

For PFTs beginning with a $(k+1,k+2)$ or $(k+2,k+1)$-flow line starting in $L(1/2)$ the result follows from Lemmas \ref{lem:PFTrees} and \ref{lem:PFTreesk1k2}.  

Thus, we can consider a PFT, $\Gamma$, beginning with a $(k,k+1)$- or $(k,k+2)$- flow line in $L(1/2)$.    Note that along the initial edge no $Y_1$ or switch vertices are possible because the upper sheet of the flow is the upper sheet of the cusp locus.  In the base case where there are no $Y_0$'s, $\Gamma$ can only terminate at the cusp edge since its image is contained in $L(1/2)$ (by Property \ref{pr:monotonicityIIST}) and there are no $(k,k+1)$- or $(k,k+2)$- Reeb chords in $L(1/2)$.  
In this case, 
\[
D(\Gamma) = A-C+E-Y_1-SW = 0- 0 + 1 - 0-0 = 1.
\]  
For the inductive step, suppose $\Gamma$ has at least one $Y_0$.  At the first $Y_0$, which is located in $L(1/2)$ (by Property \ref{pr:monotonicityIIST}), the initial  $(k,k+1)$-flow line  (resp. $(k,k+2)$-flow line), splits into a $(k,k+2)$-flow line and a $(k+2,k+1)$-flow line (resp. a $(k,k+1)$-flow line and a $(k+1,k+2)$-flow line).  Let $\Gamma_1$ and $\Gamma_2$ denote the two PFTs consisting of the portion of $\Gamma$ that begins with these two branches of the tree.
%, and let $D(\Gamma_i)$ denote $A+E-SW-Y$.  
The inductive hypothesis gives $D(\Gamma_1) >0$, and 
%that the result is known for trees beginning with $(k+2,k+1)$ or $(k+1,k+2)$ flows lines gives $
$D(\Gamma_2) \geq 0$.  Thus, $D(\Gamma) = D(\Gamma_1) + D(\Gamma_2) > 0$.
\end{proof}

Let    
$B = D(1/2) \cap L(1/2)$.

\begin{lemma}  \label{lem:RegionB} Any PFT $\Gamma$ starting in $B$ satisfies $D(\Gamma) \geq 0$.  Moreover, equality holds only for trees with a single internal vertex at a switch.  Any such partial flow tree can be described in one of the following ways:
\begin{enumerate}
\item   A subset of the $(i,k)$-switch GFT beginning on the branch that precedes the switch.
%A portion of the unique $(i,k)$-flow line from $c_{i,k}$ to the $(i,k)$-switch point followed by the $(i,k+1)$-flow line from the switch point to $a^{-,-}_{i,k+1}$.
\item  A subset of the $(k+1,k+2)$-switch GFT beginning on the branch that precedes the switch.
%A portion of the  unique $(k+1,k+2)$-flow line from $c_{k+1,k+2}$ to the $(k+1,k+2)$-switch point followed by the $(k,k+2)$-flow line that starts at the switch point and terminates somewhere along the $(k,k+2)$-cusp edge.
\item  A subset of the $(k+2,k+1)$-switch GFT beginning on the branch that precedes the switch.
%A portion of the  unique $(k+2,k+1)$-flow from $\tilde{c}_{k+2,k+1}$ to the $(k+2,k+1)$-switch point followed by the $(k,k+1)$ flow line that starts at the switch point and terminates somewhere along the $(k,k+1)$ cusp edge.
\end{enumerate}

\end{lemma}

(By the definition of PFT (see Definition \ref{def:GFT}), in (1)-(3) $\Gamma$ contains the entire portion of the appropriate switch GFT that follows the initial point of $\Gamma$, including the switch vertex and the last edge of the GFT.)

\begin{proof}
First, note that for PFTs beginning with an $(i,j)$-flow line with $\{i,j\} \subset \{k,k+1,k+2\}$, the result follows from Lemma \ref{lem:AprimeNew}.  
%This is because the intersection of $B$ with the domain of $F_{i,j}$ is contained in $A'$.  

We now assume $\Gamma$ begins  in $B$ with an $(i,j)$-flow line with $\{i,j\} \not \subset \{k,k+1,k+2\}$, and
%$(k+2,k+1)$ flow in $B$ the result follows from Proposition \ref{prop:PFTrees} since the trees described in (2) and (3) of that Lemma have images disjoint from $B$.  Also, for PFTs beginning with a $(k+1,k+2)$ flow in $B$ the result follows from Proposition \ref{prop:PFTreesk1k2}.  In addition, notice that, for $(i,j)$ different from $(k+1,k+2)$ and $(k+2,k+1)$, Lemma \ref{lem:T13D12}   \dr{As stated, the lemma fails for $i,j \in \{k,k+1,k+2\}$, but the result is already known for these flows, right?}  implies that $B$ is invariant for the $(i,j)$-flow.
prove the Lemma by induction on the sum of the number of $Y_0$ and $Y_1$ vertices in $\Gamma$.  (We do allow the possibility that $\Gamma$ begins with a flowline involving the sheet $\tilde{S}_k$.) 

For the {\bf base case}, 
%suppose $\Gamma$ begins in $B$ with an $(i,j)$-flow with $(i,j) \neq (k+1,k+2)$ and $(i,j) \neq (k+2,k+1)$, and 
suppose that $\Gamma$ has no $Y_0$'s and $Y_1$'s.  The tree must have precisely one output vertex possibly preceded by a switch vertex.  Since $\{i,j\} \not \subset \{k,k+1,k+2\}$, the Property \ref{pr:monotonicityIIST} shows that the entire image of $\Gamma$ is in the region $B$.  As $B$ does not contain $b_{i,j}$ or $c_{i,j}$ Reeb chords, the output must be at an $a^{-,-}$ Reeb chord or at an $e$-vertex, so we have
\[
1 = A-C+E.    
\]
Since there are no $Y_1$ vertices, and $\Gamma$ can have at most one switch by Lemma \ref{lem:swPFT}, the inequality
\[
D(\Gamma) = A-C+E-SW -Y \geq 0
\]
follows.  Those PFTs that contain exactly one switch will produce equality.   %Correspondingly, partial flow trees with a single switch at the $(i,k)$ switch point are listed in (1), (2), and (3) in the statement of the theorem. According to Proposition \ref{prop:SwPoints}, the only remaining possibility for a PFT with a single switch would be if the switch is at the $(k+1,j)$ switch point.  
We claim that without $Y_0$'s or $Y_1$'s, there are no such trees with a switch at the $(k+1,j)$-switch point for $k+2<j$.  Thus, in view of Property \ref{pr:switches}, trees without $Y_0$'s and $Y_1$'s and with $D(\Gamma) =0$ match the description given in (1)-(3) of the statement of Lemma \ref{lem:RegionB}.   [To verify the claim, note that the $(k,j)$-flowline that we are left with after a switch at the $(k+1,j)$-switch point must remain in $B$.  However,  there are no $(k,j)$ Reeb chords in $B$, and sheets $k$ and $j$ do not meet at any cusp edge.]   

%where we used that $\Gamma$ does not have $Y_1$ vertices.  Thus, $\Gamma$ has at least one switch, and (1) of Lemma \ref{lem:swPFT} shows that this is the only internal vertex of $\Gamma$.  The claim now follows from the remaining parts (2), (3), and (4) of Lemma \ref{lem:swPFT}.   
%(Actually, what about the other switch points, and did we need to mention them earlier?)

For the {\bf inductive step}, suppose that $\Gamma$ has at least one $Y_0$- or $Y_1$-vertex.  
 Consider the first $Y_0$- or $Y_1$-vertex, $x$, that is encountered when traveling away from the initial point of $\Gamma$.  A priori, the part of the tree, $\alpha \subset \Gamma$, that connects $x$ to the start of $\Gamma$ may contain switches.  However, Lemma \ref{lem:swPFT} implies that $\alpha$ can  have at most one such switch.

Let $\Gamma_a$ and $\Gamma_b$ denote the PFTs that begin with the two branches of $\Gamma$ that follow $x$, notated so that, for some $i,j, m, l$, (some of which may be equal to $\tilde{k}$), the incoming edge is a $(i,j)$-flow line at $x$, $\Gamma_a$ begins with an $(i,m)$-flow line,  and $\Gamma_b$ begins with an $(l,j)$-flow.  
Note that $i$ and $j$ used here may be different from their values at the initial point of $\Gamma$, due to $\alpha$  crossing the boundary between sheets $\widetilde{S}_k$ and one of $S_{k+1}$ or $S_{k+2}$, or due to a switch occurring along $\alpha$.  However, note that at all points of $\alpha$ we still  have $\{i,j\} \not \subset \{k,k+1,k+2\}$, so that the image of $\alpha$ is entirely contained in $B$ (by Property \ref{pr:monotonicityIIST}).  Thus, the inductive hypothesis applies to both $\Gamma_a$ and $\Gamma_b$.  See Figure \ref{fig:GammAB}.

%of these trees since, except possibly for $(k+1,k+2)$ and $(k+2,k+1)$-flow lines, any edge that  begins in $B$ remains in $B$. [In case a switch occurs before $x$, note that the possible indices for edges beginning with a switch vertex were listed in Proof of Lemma \ref{lem:swPFT}.  Here, it is only relevant that $(k+1,k+2)$ and $(k+2,k+1)$ do not occur.]  See Figure \ref{fig:GammAB}.

\begin{figure}
\labellist
\small
\pinlabel $\Gamma$ at 76 108
\pinlabel $\alpha$ [r] at 28 80
\pinlabel $x$ [bl] at 41 47
\pinlabel $\Gamma_b$ [bl] at 61 16
\pinlabel $\Gamma_a$ [br] at 5 16
%\pinlabel $\vdots$ [t] at 63 -2
%\pinlabel $\vdots$ [t] at 0 -2
\endlabellist
\centerline{ \includegraphics[scale=.6]{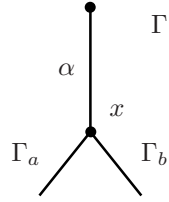} }
\caption{The PFT $\Gamma$ begins in $B$ and has its first $Y_0$- or $Y_1$-vertex at $x \in B$.  The segment $\alpha$ begins with the initial edge of $\Gamma$ and may possibly contain a $2$-valent switch vertex.}
%The arrows indicate the direction along edges of all $(i,j)$-flows (negative gradients) other than the $(k+2,k+1)$ flow.  
%(right) The partial flow trees in $A$ described in (2) of Lemma \ref{lem:RegionA}.  
\label{fig:GammAB}
\end{figure}

\medskip

%\begin{itemize}

%\item[Case 1:]

\noindent {\bf Case 1:}  The vertex $x$ is a $Y_1$.  
%Let $\Gamma_a$ and $\Gamma_b$ denote the PFTs that begin with the two branch of $\Gamma$ that follow $x$.  

\medskip

First, {\bf suppose that $\alpha$ does not contain a switch}.
Supposing the $Y_1$ is along the $(k,k+1)$-cusp edge, $\Gamma_a$ and $\Gamma_b$ would respectively begin with $(i,k+1)$- and $(k, j)$- flow lines, and we must have $i<k$ and $k+1<j$.  If instead the $Y_1$ occurs along the $(k,k+2)$-cusp edge, $\Gamma_a$ and $\Gamma_b$ would respectively begin with $(i,k+2)$- and $(k, j)$- flow lines with $i<k$ and $j \in \{k+1, k+3, \ldots\}$.    Note that none of the degree $0$  PFTs listed in (1), (2), (3) of the statement begin with an edge that is a $(i,k+1)$-, $(i,k+2)$-, or $(k,j)$-flow line for $i<k, j \geq k+1$.  As $\Gamma_a$ and $\Gamma_b$ both begin with edges of one of these types, the inductive hypothesis shows that $D(\Gamma_a)\geq 1$ and $D(\Gamma_b)\geq 1$, and we can estimate
\[
D(\Gamma) = D(\Gamma_a)+ D(\Gamma_b)- 1 \geq 1+1-1 =1.
\]

Next, {\bf suppose that $\alpha$ does contain a switch}.  We will show that this case cannot actually occur.  The only possibility is that the switch occurs at the $(i,k)$-switch point with the incoming edge an $(i,k)$-flow line and the outgoing edge initially an $(i,k+1)$-flow line.  [For the other possible switch vertices identified in Property \ref{pr:switches}, the outgoing edge must be either a $(k,k+1)$-, $(k,k+2)$-, or $(k,j)$-flow line for some $k+2<j$.  However, the internal vertex that occurs at the end of this edge was assumed to be the $Y_1$ at $x$.  This is impossible since the entire edge must have $S_k$ as the upper sheet, and the incoming edge at a $Y_1$ vertex must have its upper sheet above the sheets that meet at the cusp edge.]  

Thus, the outgoing edge at the switch is the $(i,k+1)$-flow line, $\nu$, that begins at the $(i,k)$-switch point.  
Recall that in the proof of Lemma \ref{lem:ikGFT}, a region $C_{i,k}$ was defined (see  Figures \ref{fig:Rswk1k2} and \ref{fig:RegionCC})
and shown to satisfy the following: 

%%%%%%%contain the entire flow line $\nu$.   
%%%%%%%%%both $-\nabla F_{i,k+1}$ and $-\nabla(F_i - \tilde{F}_k)$ pointing inward along $\partial C_{i,k}$.  Co
%%%%%%%%Note that the lower sheet of this flow line may become $\tilde{S}_k$ if the image crosses the $(k,k+2)$-cusp locus.  We will show that this edge must remain in the region, $C$, that consists of the lower left portion of $N(e^2_\alpha)$ that is bounded by 
%%%%%%%%%\begin{enumerate}
%%%%%%%%%\item $\{x_1 = -17/64\}$; 
%%%%%%%%%\item the horizontal line segment from $x_1=-17/64$ to the $(i,k)$-switch point;  
%\item the portion of the $(k,k+1)$-cusp locus that lies between the $(i,k)$-switch point and the swallowtail point; and
%\item the portion of $\{x_2 = 0\}$ between the swallowtail point and the left hand boundary of $N(e^2_\alpha)$.  
%\end{enumerate}

\begin{figure}
%\labellist
%\small
%\pinlabel $\Gamma$ at 76 108
%\pinlabel $\alpha$ [r] at 28 80
%\pinlabel $x$ [bl] at 41 47
%\pinlabel $\Gamma_b$ [bl] at 61 16
%\pinlabel $\Gamma_a$ [br] at 5 16
%\pinlabel $\vdots$ [t] at 63 -2
%\pinlabel $\vdots$ [t] at 0 -2
%\endlabellist
\centerline{ \includegraphics[scale=.6]{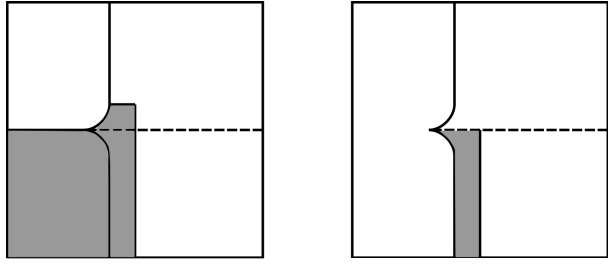} }
\caption{The regions $C_{i,k}$ and $C'$ from Claims A and B.}
%The arrows indicate the direction along edges of all $(i,j)$-flows (negative gradients) other than the $(k+2,k+1)$ flow.  
%(right) The partial flow trees in $A$ described in (2) of Lemma \ref{lem:RegionA}.  
\label{fig:RegionCC}
\end{figure}

\medskip

\noindent \textbf{Claim A.}  For $i<k$, the  $(i,k+1)$-flow line, $\nu$, beginning at the $(i,k)$-switch point  
%any edge that starts with an $(i,k)$- or $(i,k+1)$-flow line somewhere in $C$ 
has its entire image contained in $C_{i,k}$.  (The parts of $\nu$ that lie to the left of the $(k,k+2)$-cusp locus are $(i,\tilde{k})$-flow lines and are also contained in $C_{i,k}$.)

\medskip

%\noindent \textit{Proof of Claim A.}  Along (1) $-\nabla F_{i,k+1}$ and $-\nabla F_{i,k}$ point left by Property \ref{pr:leftC}; along (2) $-\nabla F_{i,k+1}$ and $-\nabla F_{i,k}$ point down by Property \ref{pr:monoLtilde}; along (3) $-\nabla F_{i,k+1}= \nabla F_{i,k}$ because (3) is a subset of the $(k,k+1)$-cusp edge, and they both point in the direction where the number of sheets increases by Corollary \ref{cor:mnPQ}; and along (4) $-\nabla(F_i - \tilde{F}_k)$ points down by Property \ref{pr:monoLtilde}.  In addition, an $(i,k)$-flow line ceases to exist if it reaches any portion of the cusp locus, and thus will not every reach (4).

%\medskip

Thus, the $Y_1$ vertex that occurs at $x$ could only be located somewhere on the $(k,k+2)$ cusp edge
 since the incoming edge must be either a $(i,k+1)$- or $(i,\tilde{k})$-flow line contained in $C_{i,k}$.  
 After the $Y_1$, the outgoing edges are an $(i,k+2)$-flow line and a $(k,k+1)$-flow line.  We show that the PFT beginning with the edge that is the $(i,k+2)$-flow line cannot exist.  (Our line of reasoning is summarized in Figure \ref{fig:GammAB2}.)

\begin{figure}
\labellist
\small
%\pinlabel $\Gamma$ at 76 108
\pinlabel $(i,k)$-switch~point [l] at 40 118
\pinlabel $Y_1$-vertex~in~$C'$ [l] at 78 64
\pinlabel $(k,k+1)$ [bl] at 58 16
\pinlabel $(i,k+2)$ [br] at 6 16
\pinlabel Does~not [t] at 0 -2
\pinlabel exist [t] at 0 -16
%\pinlabel $\vdots$ [t] at 63 -2
%\pinlabel $\vdots$ [t] at 0 -2
\endlabellist
\centerline{ \includegraphics[scale=.6]{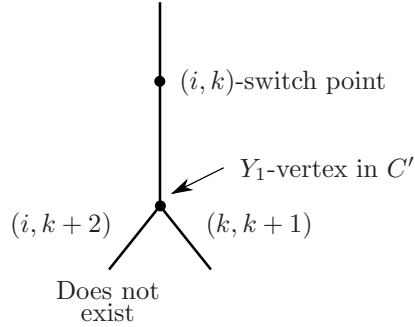} }

\quad

\quad

\quad

\caption{Summary of reasoning applied for the inductive step in Case 1 when the initial branch of $\Gamma$  ends at a switch.}
%The arrows indicate the direction along edges of all $(i,j)$-flows (negative gradients) other than the $(k+2,k+1)$ flow.  
%(right) The partial flow trees in $A$ described in (2) of Lemma \ref{lem:RegionA}.  
\label{fig:GammAB2}
\end{figure}

\medskip

\noindent \textbf{Claim B.}  For $i<k$, there are no PFTs that begin with an $(i,k+2)$-flow line starting in the region $C'$ that lies below the crossing locus; to the right of the $(k,k+2)$-cusp locus; and to the left of $x_1 = -17/64$.  

\medskip

%Note that $C'$ that is the subset of $C_{i,k}$ that lies below the crossing locus and to the left of the cusp locus.
 
\noindent \textit{Proof of Claim B.}   The region $C'$ is pictured in Figure \ref{fig:RegionCC}.
Suppose such a PFT does exist.  
%and consider the region $E$ bounded above by $x_2=0$ and on the right by $x_1=-17/64$ and on the left by the $(k,k+2)$ cusp locus.  
The $(i,k+2)$-flow line cannot leave $C'$ (by Properties \ref{pr:leftC} and \ref{pr:monoLtilde}, and the fact that the flow line would terminate if it reaches the $(k,k+2)$ cusp edge since $i<k$.)  Thus, since there are no $(i,k+2)$-Reeb chords in $C'$, and $S_i$ and $S_{k+2}$ do not meet at a cusp edge in $C'$, this edge must end at an internal  vertex which can only be a $Y_0$.  The outgoing edges will be a $(i,i')$-flow line and a $(i',k+2)$-flowline for some $i<i'\leq k$.  If $i' <k$, then we repeat this argument with the $(i',k+2)$-edge to locate another $Y_0$.  If $i' =k$, then the $(i,k)$-edge must remain in  $C'$ (by the same reasoning that shows the $(i,k+2)$-flow line cannot leave). Then, we can apply a similar argument to see that this edge must end at a $Y_0$.  (Note, that the edge cannot end at the $(i, k)$-switch point which does not belong to $C'$.) One of the outgoing edges at this $Y_0$ would be an $(i'', k)$-flow line with $i<i''<k$ beginning in $C'$, and so we could repeat the same argument to find yet another $Y_0$.  However, this process cannot go on indefinitely since there are only finitely many numbers between $i$ and $k$, so we reach a contradiction. 

\medskip

%  along this edge could only be a $Y_0$-vertex where ------------One of the bottom branches of the tree would have to be either an $(i',k)$ or $(i',k+2)$ flow beginning in this region and with a terminal point in this region.

%\dr{The case where the segment $\alpha$ does contain a switch seems to be omitted.  The only possibility is the $(i,k)$ switch point followed by an $(i,k+1)$ flow that then becomes a $\tilde{S}_k$ flow before a $Y_1$.  We use that the horizontal line through the swallowtail point cannot be crossed, so the only possibility is for the $Y_1$ to be located on the lower half of the cusp edge.  After the $Y_1$, the flows will be an $(i,k+2)$ flow and a $(k,k+1)$ flow.  We show that the $(i,k+2)$-flow cannot be completed to a partial flow tree.  For this, consider the region bounded above by $x_2=0$ and on the right by $x_1=-1/4$ (for instance).  One of the bottom branches of the tree would have to be either an $(i',k)$ or $(i',k+2)$ flow beginning in this region and with a terminal point in this region.}

%\item[Case 2:]  

\noindent {\bf Case 2:} The vertex $x$ is a $Y_0$.  

\medskip

First, {\bf suppose that $\alpha$ does not contain a switch}.  Then, 
\[
D(\Gamma) = D(\Gamma_a) + D(\Gamma_b).
\]
The inductive hypothesis gives $D(\Gamma_a) \geq 0$ and $D(\Gamma_b) \geq 0$. Note that because $\Gamma_a$ and $\Gamma_b$ must respectively start with an $(i,m)$-flow line and an $(m,j)$-flow line, it is impossible that they are both degree $0$ PFTs as described in (1), (2), (3) in the statement of this Lemma.  Thus, $D(\Gamma_a)> 0$ or $D(\Gamma_b)> 0$, so  $D(\Gamma) \geq 1$ holds.

Finally, {\bf suppose that $\alpha$ does contain a single switch}, $y$, so that 
\[
D(\Gamma) = D(\Gamma_a)+ D(\Gamma_b) -1.
\]
We do not need to consider the case where $y$ is at the $(k+1,k+2)$- or $(k+2,k+1)$-switch point (since $\Gamma$ does not begin with an $(i,j)$-flow line with $\{i,j\} \subset \{k,k+1,k+2\}$), so we are left to consider two cases.

\begin{itemize}
\item[Subcase A:]  The vertex $y$ is an $(i,k)$-switch point with $i<k$.  Then, the edge of the tree that directly follows $y$ and  precedes $x$ begins is the $(i,k+1)$-flow line, $\nu$, starting at the switch.

%Recall the region  $C \subset N(e^2_\alpha)$ defined above Claim A.
In view of Claim A, the incoming edge at $x$ is a flow line in $C_{i,k}$ with lower sheet either $\tilde{S}_k$ or $S_{k+1}$, so we see that either 
\begin{itemize}
\item[(a)] $\Gamma_a$ begins with an $(i,m)$-flow line and $\Gamma_b$ begins with an $(m,k+1)$-flow line  for some $m$ such that sheet $S_m$ lies between sheets $S_i$ and $S_{k+1}$ at the image of $x$, or  
\item[(b)] $\Gamma_a$ begins with an $(i,m)$-flow line and $\Gamma_b$ begins with an $(m, \tilde{k})$-flow line (i.e. a flow line for  
$-\nabla(F_{m}- \tilde{F}_k)$) for some $m$ such that sheet $S_m$ lies between sheets $S_i$ and $\widetilde{S}_k$ at the image of $x$.
\end{itemize}
Here, (a) occurs if $x$ is to the right of the cusp locus and (b) occurs if $x$ is to the left of the cusp locus.  By the description of degree $0$ PFTs found in (1)-(3)  of this Lemma, we know from the inductive hypothesis that both $\Gamma_a$ and $\Gamma_b$ have degree strictly greater than $0$ in all cases except possibly when (a) occurs and $m = k$ or $m=k+2$.  We address these exceptional cases separately.

\begin{itemize} 
\item If $m =k$, note that $x$ cannot occur at a point belonging to the $(i,k)$-switch flow line that precedes the $(i,k)$-switch point.  [A switch at the $(i,k)$ switch point, $y$, has already occured in the edge from $y$ to $x$, and the difference functions $F_{i,j}$ strictly decrease along all edges of a PFT and when passing a $Y_0$ vertex.]     Thus, $\Gamma_a$ and $\Gamma_b$ both have strictly positive degree so that $D(\Gamma) \geq 1$ holds.

\item The case $m=k+2$ cannot actually occur.  
%If $m=k+2$, we show that $D(\Gamma_a) \geq 2$, so that applying the inductive hypothesis to $\Gamma_b$ gives $D(\Gamma) = D(\Gamma_a) + D(\Gamma_b) -1 \geq 2 + 0 -1$.

As $F_{k+2,k+1}$ is only positive below the crossing locus, and $\Gamma_b$ begins with a $(k+2,k+1)$-flow, it must be the case that the image of $x$ which is in $C_{i,k}$ (by Claim A) lies below the crossing locus and right of the cusp locus, i.e. in the region $C'$ from Claim B.  However, the PFT $\Gamma_a$ then starts with an $(i,k+2)$-flow line in $C'$, with $i<k$, and cannot exist by Claim B.

%$\{x_2 < 0\}$.  In combination with the Claim, we see that  $\Gamma_a$ begins with an $(i,k+2)$ flow located in  $C'$ which we define to be the portion of $\{x_2 \leq 0 \} \cap C$ to the right of the cusp locus.  See Figure ???.   Note that $C'$ is invariant for $(i,k+2)$ flows in positive time.  [Using Claim 2 in proof of Theorem \ref{thm:NoReebChordsST}, $-\nabla F_{i,k+2}$ points down along the upper boundary of $C'$.  Lemma \ref{lem:leftC} shows that $-\nabla F_{i,k+2}$ points left along the right boundary of $C'$.  The remaining portions of $\partial(C')$ are made up by the $(k,k+2)$ cusp edge and part of $\partial ( [-1,1]^2)$. ]  There are no $(i,k+2)$ Reeb chords or cusp edges in $C'$.  Thus, in order to be completed to a partial flow tree, the $(i,k+2)$ flow must terminate at an internal vertex of $\Gamma_a$, that can only be a $Y_0$.  [Switches and $Y_1$ vertices are prohibited since $i<k$ and the only cusp edge in $C'$ is a $(k,k+2)$-cusp edge.]  The partial flow trees, $\Gamma_c$ and $\Gamma_d$, that begin immediately after this first $Y_0$ must have strictly positive degree.  This is because $\Gamma_c$ and $\Gamma_d$ respectively begin with $(i,l)$ and $(l,k+2)$ flows in $C'$ for some $l \leq k$.  By the inductive hypothesis, only the possibility for either $\Gamma_c$ or $\Gamma_d$ to have degree $0$ is for $\Gamma_c$ if $l =k$.  However, as in the previous case, the $(i,k)$ flow line that reaches the $(i,k)$ switch point cannot intersect $C'$.

\end{itemize}

\item[Subcase B:]  The vertex $y$ is a $(k+1,j)$-switch point with $k+2<j$.  The edge that follows $y$ and ends at $x$ is then a $(k,j)$-flow line.  Thus, $\Gamma_a$ and $\Gamma_b$ will begin respectively with $(k, l)$- and $(l,j)$-flow lines for some $k<l$.  The  characterization of degree $0$ PFTs from (1)-(3) applies to show that both $D(\Gamma_a) \geq 1$ and $D(\Gamma_b) \geq 1$, and $D(\Gamma) \geq 1$ follows.  

\end{itemize}

%\end{itemize}

\end{proof}

\begin{lemma} \label{lem:Btree} Let $\Gamma$ be a PFT with $D(\Gamma) = 0$ and without endpoints at Reeb chords of the form $c_{i,j}$ or $\tilde{c}_{k+2,k+1}$.  Then,
\begin{enumerate}
\item $\Gamma$ has no $Y_1$ vertices,  
\item any edge that starts with a switch vertex ends at a Reeb chord or $e$-vertex, and
\item all endpoints of $\Gamma$ occur either at $b$ Reeb chords or at the end of an edge that starts with a switch at either the $(i,k)$-, $(k+1,k+2)$-, or $(k+2,k+1)$-switch point.
%and ends at an $a$ Reeb chord or an  $e$-vertex.
\end{enumerate}
\end{lemma}

\begin{proof}
Let $y_1, \ldots, y_m$ denote a collection of $Y_1$ and switch vertices of $\Gamma$ such that
\begin{enumerate}
\item  the PFTs $\Gamma_1,\ldots, \Gamma_m$ that begin just above the $y_i$ are disjoint from one another (as subsets of the domain of $\Gamma$),  and 
\item all $Y_1$ and switch vertices of $\Gamma$ are contained in the union of the $\Gamma_i$.
\end{enumerate} See Figure \ref{fig:yigammai}.  [To construct such a collection $y_1, \ldots, y_m$, note that if two PFTs intersect (in the domain of $\Gamma$) than one is contained in the other.  Begin by listing all the switches and $Y_1$ vertices as $z_1, \ldots, z_M$.  Start with $y^1_1 = z_1$, and build collections $\{y^i_1, \ldots, y^i_{m_i}\}$ inductively for $1 \leq i \leq M$ so that at the $i$-th step (1) holds and all vertices $z_1, \ldots, z_i$ are contained in the PFTs starting above $y^i_1, \ldots, y^i_{m_i}$.  For the inductive step, if $z_{i+1}$ is already contained in the union of PFTs starting at $y^i_1, \ldots, y^i_{m_i}$, then do not change the collection.  Else, add $z_{i+1}$ to the $y^i_1, \ldots, y^i_{m_i}$ and throw out all $y^i_j$ that are contained in the PFT starting at $z_{i+1}$.]

\begin{figure}
\labellist
\small
\pinlabel $y_1$ [r] at 34 80
\pinlabel $y_2$ [r] at 150 98
\pinlabel $y_3$ [l] at 392 120
\pinlabel $\Gamma_1$ [t] at 50 10
\pinlabel $\Gamma_2$ [t] at 162 -2
\pinlabel $\Gamma_3$ [t] at 378 60
\pinlabel $\Gamma$ at 240 236
\endlabellist
\centerline{ \includegraphics[scale=.6]{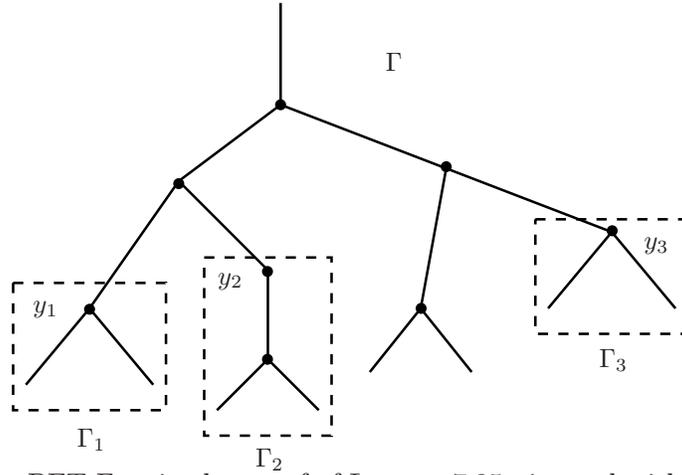} }
\caption{The PFT $\Gamma$ as in the proof of Lemma \ref{lem:Btree} pictured with $m=3$.   The vertices $y_1, \ldots, y_m$ are $Y_1$- and $\mathit{sw}$- vertices chosen so that all vertices outside of $\Gamma_1, \ldots, \Gamma_m$ are $Y_0$'s. }
%The arrows indicate the direction along edges of all $(i,j)$-flows (negative gradients) other than the $(k+2,k+1)$ flow.  
%(right) The partial flow trees in $A$ described in (2) of Lemma \ref{lem:RegionA}.  
\label{fig:yigammai}
\end{figure}

Since $\Gamma$ was assumed to not have punctures at $c_{i,j}$'s, we then have
\begin{equation} \label{eq:DGamma}
0 = D(\Gamma) = \sum_{i} D(\Gamma_i) + A' + E',
\end{equation}
where $A'$ and $E'$ respectively denote the number of output vertices at $a_{i,j}$'s and at $e$-vertices that are located in $\Gamma \setminus (\bigcup_{i} \Gamma_i)$.  

Each of the $\Gamma_i$  starts in $B$.  [Property \ref{pr:monotonicityI} shows that $Y_1$'s cannot occur where $x_{2} \geq 1/2$ since, for all $S_{i}$ and $S_{j}$ that are above and below the cusp edge,  $-\nabla F_{i,j}$ points to the side of the cusp locus with fewer sheets.  In addition, all switch points are located in $B$.]   Thus, by Lemma \ref{lem:RegionB} each $\Gamma_i$ has $D(\Gamma_i) \geq 0$.  Therefore, all terms on the right hand side of (\ref{eq:DGamma}) are non-negative, and we can conclude that they all vanish.  Consequently, all of the $y_i$ must be switches with the flow trees $\Gamma_i$ as described in (1)-(3) of Lemma \ref{lem:RegionB}.  It follows that $\Gamma$ has no $Y_1$-vertices, and there are no other switches in $\Gamma$.  Since $A'=E'=0$, all endpoints of $\Gamma$ that do not follow switches are at $b_{i,j}$ generators.
%can by $Y_1$-vertices   
%\[
%D(\Gamma) = \sum_{i} D(\Gamma_i) + A' + E' - SW' \geq A' + E' - SW'
%\]
%where $A'$, $E'$, and $SW'$ denote the number of output vertices at $a_{i,j}$'s, at cusp edges, and the number of switches that are located in $\Gamma \setminus (\bigcup_{i} \Gamma_i)$.  Now, let $\Lambda_1, \ldots, \Lambda_\ell$ denote PFTs starting at switches...  Better to take at the first step a collection of uppermost switches and $Y_0$'s  this gives the desired result almost immediately.  

%Do we need that any flow tree below a switch has deg $\geq 1$ with equality only if there are no internal vertices?

\end{proof}

\subsection{Enumeration of $b$-trees}  \label{ssec:enumerbtrees}

With Lemma \ref{lem:Btree} in hand, we can now compute $\partial_b c_{i,j}$ using a method that is parallel to that of Section \ref{ssec:112btrees}.  We begin by modifying some of the terminology from \ref{ssec:112btrees} to our current setting.

\begin{definition}  In a Type (13) square, a {\bf $b$-tree} is a PFT or GFT that satisfies the conclusions (1)-(3) of Lemma \ref{lem:Btree} and also  
\begin{enumerate}
\item[(4)]  does not have any endpoints at $\tilde{b}^R_{k+2,k+1}$.
\end{enumerate}
\end{definition}

Let
\begin{equation}  \label{eq:defdbST}
\partial_b c_{i,j} = \sum w(\Gamma) 
\end{equation}
where the summation is over $b$-trees that begin with a positive puncture at $c_{i,j}$.

The following is an immediate consequence of Lemma \ref{lem:Btree}.

\begin{proposition}  \label{prop:dbc}  For any $1 \leq i< j \leq n$,
%the term $\partial_b c_{i,j}$ from (\ref{eq:dcdb}) is given by
\[
\partial c_{i,j} = \partial_c c_{i,j} + \partial_b c_{i,j} + x
\]
where  $x$ denotes a term belonging to the two-sided ideal generated by $\tilde{b}_{k+2,k+1}^R.$
\end{proposition}

Note that branches of a $b$-tree that follow switches must be as in the $(i,k)$-, $(k+1,k+2)$-, and $(k+2,k+1)$-switch GFTs, and in particular are uniquely determined by the switch point.  In these three cases the branch that follows the switch is the flow line that starts at the switch point and respectively    
limits to $a^{-,-}_{i,k+1}$, ends at an $e$-vertex on the $(k,k+2)$-cusp edge, or ends at an $e$-vertex on the $(k,k+1)$-cusp edge.  (See Lemmas \ref{lem:k1k2switchflow} and \ref{lem:ikGFT}.)
 When such a switch occurs in a $b$-tree, we will refer to the switch and the following edge as a {\bf switch output} at $\mathit{sw}_{i,k}, \mathit{sw}_{k+1,k+2}$, or $\mathit{sw}_{k+2,k+1}$.

We consider two types of $b$-trees as follows.  Define a {\bf $b^{LU}$-tree} to be a $b$-tree that has all outputs at Reeb chords of the form $b^L_{i,j}$, $b^U_{i,j}$ or at switch outputs of the form $\mathit{sw}_{i,k}, \mathit{sw}_{k+1,k+2}$.  Define a {\bf $b^{RD}$-tree} to be a $b$-tree that has all outputs at Reeb chords of the form $b^R_{i,j}$, $b^D_{i,j}$, or at switch outputs of the form $\mathit{sw}_{k+2,k+1}$.  We write $\partial_{b^{LU}} c_{i,j}$ and $\partial_{b^{RD}} c_{i,j}$ respectively for the portion of the sum (\ref{eq:defdbST}) arising from $b^{LU}$-trees and $b^{RD}$-trees.

 In the Type (13) square, the role of $b^{LU}$-trees and $b^{RD}$-trees are no longer entirely symmetric.  Consequently, we compute $\partial_{b^{LU}} c_{i,j}$ and $\partial_{b^{RD}} c_{i,j}$ separately in 
%in separate parts the contribution to $\partial_b c_{i,j}$ coming from $b^{LU}$-trees and $b^{RD}$-trees in 
\ref{sec:bLU13trees} and 
\ref{sec:bRDtrees}, respectively. 

The following observations follow from the definition of $b$-tree.
\begin{lemma} \label{lem:observations} 
\begin{enumerate}
\item[(i)] Aside from the switches at $\mathit{sw}_{i,j}$-outputs, all other internal vertices in a $b$-tree must be $Y_0$'s. 
\item[(ii)]  For any $Y_0$ in a $b$-tree, $b^{LU}$-tree, or $b^{RU}$-tree, both of the PFTs that start with the outgoing edges of the $Y_0$ are themselves $b$-trees, $b^{LU}$-trees, or $b^{RU}$-trees, respectively.
\end{enumerate}
\end{lemma}

\subsubsection{$b^{LU}$-trees}  \label{sec:bLU13trees}

For $1 \leq i < j \leq n$, we again define regions $B^{LU}(i,j)$.   
%Take the $B^{RD}(i,j)$ to be defined exactly 
The definition is as in Section \ref{ssec:112btrees}, except that 
%For the regions $B^{LU}(i,j)$, 
we \textbf{replace all uses of the segment $B_2$ that appear in the definition from Section \ref{ssec:112btrees} with the Swallowtail Barrier $P,$} see Figure \ref{fig:BLUij}.
Recall  $R(sw_{i,j})$ from Figure \ref{fig:Rswk1k2}.

\begin{lemma} \label{lem:BLUijST}   %The intersection of any edge, $\gamma$, of a $b^{LU}$-tree (resp. a $b^{RD}$-tree) with $\widehat{N}(e^2_\alpha)$ must be an $(i,j)$-flow line for some $i<j$, and must have its image contained in the region $B^{LU}(i,j)$ (resp. $B^{RD}(i,j)$).  
 %The intersection of A
 Any edge, $\gamma$, of a $b^{LU}$-tree that does not start with a switch must 
 \begin{enumerate}
 \item be an $(i,j)$-flow line for some $i<j$ and 
\item have the intersection of its image with $\widehat{N}(e^2_\alpha)$ contained in the region 
\[
\mbox{$B^{LU}(i,j)$, if $(i,j)$ does not have the form $(k+1,k+2)$ or $(i,k)$ with $i<k$;}
\]
\[ \mbox{$R(\mathit{sw}_{i,j})$, if $(i,j)=(k+1,k+2)$ or $(i,j) = (i,k)$ with $i<k$.}
\]   
\end{enumerate}
Moreover, the image of $\gamma$ is entirely in $\widehat{N}(e^2_\alpha)$ unless $\gamma$ is part of a $b^X_{i,j}$-line for $X \in \{L,U\}$.
\end{lemma}
\begin{proof}
%Note from the definition of $b$-tree that: 
%\begin{enumerate}
%\item[(i)] Aside from the switches at $\mathit{sw}_{i,j}$-outputs, all other internal vertices in a $b$-tree must be $Y_0$'s. 
%\item[(ii)]  For any $Y_0$ in a $b$-tree, $b^{LU}$-tree, or $b^{RU}$-tree, both of the PFTs that start with the outgoing edges of the $Y_0$ are themselves $b$-trees, $b^{LU}$-trees, or $b^{RU}$-trees, respectively.
%\end{enumerate}

We consider slightly smaller regions $A^{LU}(i,j)$, defined by
\begin{enumerate}
\item $A^{LU}(i,j) = R(\mathit{sw}_{i,j})$, if $(i,j) = (i,k)$ or $(i,j) = (k+1,k+2)$;
\item $A^{LU}(i,j) = B^{LU}(i,j)\cap\{x_1\geq 1/4\}$, if $(i,j) = (k,k+1)$ or $(i,j) = (k,k+2)$; 
\item $A^{LU}(i,j) = B^{LU}(i,j)\cap\{x_1\geq -3/8\}$, if 
%precisely one of $i$ and $j$ belongs to $\{k,k+1\}$ and $\{i,j\} \not \subset \{k,k+1,k+2\}$.  (In other words, 
$\{i,j\} \cap \{k,k+1\} \neq \emptyset$,  and $(i,j)$ is not $(i,k), (k+1,k+2), (k,k+1)$, or $(k,k+2)$.
\item $A^{LU}(i,j) = B^{LU}(i,j)$, if $\{i,j\} \cap \{k,k+1\} = \emptyset$.
\end{enumerate}
Note that for any $i<j$, $-\nabla F_{i,j}$ points outwards along all parts of $\partial A^{LU}(i,j)$, with $-\nabla( F_{i}-\widetilde{F}_k)$ or $-\nabla( \widetilde{F}_k-F_j)$ pointing outward as well when $i=k+2$ or $j=k+2$.  [For $A^{LU}(i,j)$ as in (1), this is shown in Lemmas \ref{lem:Rswk1k2} and \ref{lem:ikGFT}; as in (2), use Properties \ref{pr:monotonicityI}, \ref{pr:STBnew}, and \ref{pr:1cells}; as in (4), add Property \ref{pr:monotonicityIIST} and \ref{pr:0cells}; as in (3), add Corollary \ref{cor:mnPQ}.]

We prove a slightly stronger version of the proposition, where the $B^{LU}(i,j)$ are replaced with the $A^{LU}(i,j)$, by induction on the number of $Y_0$'s in the PFT that starts with $\gamma$.  In the base case, Lemma \ref{lem:observations} (i) implies the only $b^{LU}$-trees without $Y_0$'s are the $b^L_{i,j}$-lines (here $\{i,j\} \cap \{k,k+1\} = \emptyset$); the $b^U_{i,j}$-lines, and the $(i,k)$- and $(k+1,k+2)$-switch GFTs.  
By Proposition \ref{prop:bUS}, we see that the $b^L_{i,j}$- and $b^U_{i,j}$-lines all enter $\widehat{N}(e^2_\alpha)$ along $\partial A^{LU}(i,j)$, and that the parts of their images that lie outside of $\widehat{N}(e^2_\alpha)$ are all pairwise disjoint.  Lemma \ref{lem:Rswk1k2} and Lemma \ref{lem:ikGFT} show that the branchs of the $(i,k)$- and $(k+1,k+2)$-switch GFTs that do not start with switches are contained in $A^{LU}(i,k)$ and $A^{LU}(k+1,k+2)$.

For the inductive step, Lemma \ref{lem:observations} implies that if the PFT starting with $\gamma$ contains at least $1$ $Y_0$, then $\gamma$ itself must end at a $Y_0$.  Moreover, the inductive hypothesis applies to the two outgoing edges at the $Y_0$, so, for some $i<m<j$, they must be $(i,m)$ and $(m,j)$-flow lines with the $Y_0$ point located in $A^{LU}(i,m) \cap A^{LU}(m,j)$.  Since we have seen that $-\nabla F_{i,j}$ points out along $\partial A^{LU}(i,j)$, the argument is completed by the following.

\medskip

\noindent {\bf Claim.}  For any $i<m<j$, $A^{LU}(i,m) \cap A^{LU}(m,j) \subset A^{LU}(i,j)$.

\medskip
To verify the claim, note that it is impossible that both $(i,m)$ and $(m,j)$ are of the forms $(i,k)$ or $(k+1,k+2)$.  Thus, $A^{LU}(i,m) \cap A^{LU}(m,j) \subset \{x_2 \geq 1/2\}$, so we show $\widehat{A}^{LU}(i,m) \cap \widehat{A}^{LU}(m,j) \subset  A^{LU}(i,j)$ where $\widehat{A}^{LU}(i,j) = A^{LU}(i,j) \cap \{x_2 \geq 1/2\}$.  Moreover, $\widehat{A}^{LU}(i,j) \cap \{x_1 \geq 1/4\} = B^{LU}(i,j) \cap \{x_1 \geq 1/4\}$, 
 and just as in Section \ref{ssec:112btrees}, the lexicographic ordering of the $\beta_{i,j}$ easily shows  $ B^{LU}(i,m) \cap B^{LU}(m,j) \subset B^{LU}(i,j)$.  (See Figure \ref{fig:BLUijPf}.)  It only remains to check that 
\begin{equation} \label{eq:ALUim}
\left(\widehat{A}^{LU}(i,m) \cap \widehat{A}^{LU}(m,j)\right) \cap \{x_1 < 1/4\} \subset  A^{LU}(i,j)\cap \{x_1 < 1/4\}.
\end{equation}
Note that each $\widehat{A}^{LU}(i,j) \cap \{x_1< 1/4\} = \{(x_1,x_2) \in \widehat{N}(e^2_\alpha)\, |\, \alpha_{i,j} \leq x_1 < 1/4, 1/2 \leq x_2\}$ where 
\[
\alpha_{i,j} = \left\{ \begin{array}{cl} -1, & \mbox{  if $\{i,j\}\cap\{k,k+1\}= \emptyset$,} \\
-3/8, & \mbox{  if  $\{i,j\}\cap\{k,k+1\} \neq \emptyset$ and $(i,j) \neq (k,k+1), (k,k+2)$,} \\
1/4, & \mbox{  if $(i,j)= (k,k+1)$ or $(i,j)= (k,k+2)$.}\end{array} \right.
\]
Thus, (\ref{eq:ALUim}) is equivalent to $\mathit{Max}\{ \alpha_{i,m}, \alpha_{m,j}\} \geq \alpha_{i,j}$ for all $i<m<j$.  This is easily verified by cases.  [It's clear when $ \alpha_{i,j} = -1$.  When $\alpha_{i,j} = -3/8$, since $\{i,j\}\cap\{k,k+1\} \neq \emptyset$, we also have $\{i,m,j\}\cap\{k,k+1\} \neq \emptyset$.  When $\alpha_{i,j} = 1/4$, we can only have $(i,m,j) = (k,k+1,k+2)$, and $\alpha_{k,k+1} = \alpha_{k,k+2} = 1/4$.]  
\end{proof}
 
Next, we define a disk datum $(D^{LU}, \{b^L_{i,j}, b^U_{i,j}, \mathit{sw}_{i,k}, \mathit{sw}_{k+1,k+2}\}, \{I_{i,j}\})$.   Let $R \subset\widehat{N}(e^2_\alpha)$ denote the region below $x_2=1/2$ that is  above and to the left of the Swallowtail Barrier $P$ and to the right of the cusp locus.  We define 
\[
D^{LU} = \left(\bigcup_{1\leq i < j \leq n}B^{LU}_{i,j} \bigcup R \right)\setminus \left( \bigcup_{1\leq i < j \leq n}\mathit{Sq}_{i,j} \right)  
\]
where the $\mathit{Sq}_{i,j} = (\beta^U_{i,j} -\e, \beta^U_{i,j} -\e) \times (\beta^R_{i,j} -\e, \beta^R_{i,j} -\e)$ are small squares containing $c_{i,j}$.  We use $b^L_{i,j}, b^U_{i,j}$ to denote (by slight abuse of notation) the unique intersection points of the $b^L_{i,j}$- and $b^U_{i,j}$-lines with $\partial D^{LU}$, and use $\mathit{sw}_{i,k}, \mathit{sw}_{k+1,k+2}$ to denote the $(i,k)$- and $(k+1,k+2)$-switch points which also lie on $\partial D$.  In all cases, the upper and lower indices agree with the subscripts.  The intervals $I_{i,j}$ are given by the upper and left edges of the closure $\overline{\mathit{Sq}}_{i,j}.$ 

There is a generalized $b$-manifold $A = \sqcup_r A_r$ given by paths $\gamma$ that are the intersection with $\widehat{N}(e^2_\alpha)$ of the top branch of a GFT that is a $b^{LU}$-tree beginning at some $c_{i,j}$.  (Note that except for the ends of the $b^{L}_{i,j}$- and $b^{R}_{i,j}$-lines all such top branches are in fact entirely contained in $D^{LU} \cup \mathit{Sq}_{i,j}$, by Lemma \ref{lem:BLUijST}.)  Upper and lower indices $i(\gamma)< j(\gamma)$ are assigned to each $\gamma \in A$ so that $\gamma$ is an $(i(\gamma), j(\gamma))$-flow line (or an $(i,\tilde{k})$- or $(\tilde{k},j)$-flow line if $j(\gamma)= k+2$ or $i(\gamma) = k+2$).  (Note that this definition of $(i(\gamma),j(\gamma))$ is well defined:  All $\gamma$ have there image in $D^{LU}$ and, in particular, may not cross the $(k,k+2)$-cusp locus.  Thus,  crossing the cusp locus \emph{cannot} cause, $\gamma$, to change from an $(i,\tilde{k})$-flow line to an $(i,k+1)$-flow line, or from a $(\tilde{k},j)$-flow line to a $(k+1,j)$-flow line.)

%We will define a generalized $b$-manifold for the disk datum $(D^{LU}, \{b^L_{i,j}, b^U_{i,j}, \mathit{sw}_{i,k}, \mathit{sw}_{k+1,k+2}\}, \{I_{i,j}\})$.  Suppose that  $\widehat{\gamma}$ is the top branch of a GFT that is a $b^{LU}$-tree beginning at some $c_{i,j}$, and let $\gamma = \widehat{\gamma} \cap D^{LU}$.  In addition, define upper and lower indices $1 \leq i(\gamma) < j(\gamma) \leq n$, so that $\gamma$ is a trajectory for $-\nabla F_{i(\gamma),j(\gamma)}$, i.e. $S_{i(\gamma)}$ and $S_{j(\gamma)}$ are the sheets of $\widetilde{L}$ that correspond to $\widehat{\gamma}$ in the GFT that it belongs to.   

\begin{lemma}  \label{lem:bLUgenb}
Let $A = \bigsqcup_{r} A_r$ denote the collection of such paths $\gamma$ with $i(\gamma)$ and $j(\gamma)$ as above.  Then, $A$ is a generalized $b$-manifold for the disk datum $(D^{LU}, \{b^L_{i,j}, b^U_{i,j}, \mathit{sw}_{i,k}, \mathit{sw}_{k+1,k+2}\}, \{I_{i,j}\})$.  Moreover, for $1 \leq i <j \leq n$,
\begin{equation}  \label{eq:bLUgenB} 
\varphi(\partial_A I_{i,j}) = \partial_{b^{LU}} c_{i,j}
\end{equation}
where the sum on the right is over $b^{LU}$-trees that start at $c_{i,j}$, and 
\[
\varphi: \Z/2\langle b^L_{i,j}, b^U_{i,j}, \mathit{sw}_{i,k}, \mathit{sw}_{k+1,k+2} \rangle \rightarrow \lchA(e^2_\alpha)
\]
 is the $\Z/2$-algebra homomorphism determined by
\[
\varphi(b^L_{i,j})=b^L_{i,j}; \quad  \varphi(b^U_{i,j})=b^U_{i,j}; \quad \varphi( \mathit{sw}_{i,k}) = a^{-,-}_{i,k+1}; \quad  \varphi(\mathit{sw}_{k+1,k+2}) =1.
\]
\end{lemma}
\begin{proof}
That $A = \sqcup A_r$ is indeed a generalized $b$-manifold is verified as in the proof of Proposition \ref{prop:MainGenB}.  Items (1) and (6) from Definition \ref{def:genbmfd} are verified exactly as in Proposition \ref{prop:MainGenB}, while in the verification of (2) we need to use Lemma \ref{lem:BLUijST} instead of Lemma \ref{lem:BLUij}.  Finally, to check items (3)-(5) concerning the endpoints of paths, we note that if $\widehat{\gamma}$ is the top most branch of a $b^{LU}$-tree, $\Gamma$, then Lemma \ref{lem:observations} shows that when $t \rightarrow +\infty$, $\widehat{\gamma}$ must either (a.) limit to some $b^L_{i,j}$ or $b^U_{i,j}$, (b.) end at a $Y_0$ such that the outgoing edges are themselves top branches of $b^{LU}$-trees, or (c.) end at a switch vertex at either $\mathit{sw}_{i,k}$ or $\mathit{sw}_{k+1,k+2}$.  Thus, paths in $A=\sqcup A_r$ indeed correspond to (i) the points $\{b^L_{i,j}, b^U_{i,j}, \mathit{sw}_{i,k}, \mathit{sw}_{k+1,k+2}\} \subset \partial D^{LU}$ and (ii) triples $(\gamma_1,\gamma_2,x)$ with $\gamma_1,\gamma_2 \in A$, $x \in \gamma_1 \cap \gamma_2$, and $j(\gamma_1) = i(\gamma_2)$.

To verify (\ref{eq:bLUgenB}), let $\gamma$ be the top of a $b^{LU}$-tree $\Gamma$ that starts at $c_{i,j}$.  Then, verify using induction on the number of $Y_0$'s in $\Gamma$ that
\[
\varphi(w(\gamma)) = w(\Gamma)
\]
where on the left $w(\gamma)$ is the word associated to $\gamma \in A$.  In the base case, Lemmas \ref{lem:k1k2switchflow} and \ref{lem:ikGFT} give that $w(\Gamma_{k+1,k+2}) = 1$ and $w(\Gamma_{i,k}) = a^{-,-}_{i,k+1}$ when $\Gamma_{k+1,k+2}$ and $\Gamma_{i,k}$ are respectively the $(k+1,k+2)$- and $(i,k)$-switch GFTs.  The inductive step is an immediate consequence of equation (\ref{eq:pbIij}) and the definitions of $w(\gamma)$ (defined right after Definition \ref{def:genbmfd}) and $w(\Gamma)$.
\end{proof}

%Let 
%\[
%\partial_{b^{LU}} c_{i,j} = \sum_{\Gamma} w(\Gamma)
%\]
%where the sum is over $b^{LU}$-trees beginning at $c_{i,j}$.
\begin{proposition}  \label{prop:bLUCompST}
For any $1 \leq i < j \leq n$, $\partial_{b^{LU}} c_{i,j}$ is equal to the $(i,j)$-entry of the matrix
\[
(I+B_U)(I+B_L)(I+A_{-,-}E_{k+1,k}+E_{k+1,k+2})
\]
where the matrices $B_U, B_L$ and $A_{-,-}$ are as in Theorem \ref{thm:SwallowComp}.
\end{proposition}

\begin{proof}
From Lemma \ref{lem:bLUgenb} and Proposition \ref{prop:dAind}, we see that
\[
\partial_{b^{LU}} c_{i,j} = \varphi( \partial_A I_{i,j})
\]
where $A$ is \emph{any} generalized $b$-manifold associated to the disk datum $(D^{LU}, \{b^L_{i,j}, b^U_{i,j}, \mathit{sw}_{i,k}, \mathit{sw}_{k+1,k+2}\}, \{I_{i,j}\})$. 

We construct an explicit generalized $b$-manifold $A= \sqcup A_r$ in several steps that we call {\it levels}.  At Level 0, we will begin by choosing paths to connect each of the points $\{b^L_{i,j}, b^U_{i,j}, \mathit{sw}_{i,k}, \mathit{sw}_{k+1,k+2}\}$ to the appropriate $I_{i,j}$.  Then, at Level 1 we identify all intersections between paths that share a common upper and lower index, and choose paths to connect each of these intersection points to the required $I_{i,j}$. 
The construction is completed once we reach a Level where no new intersections arise.

We say that an intersection of paths $\gamma_1,\gamma_2 \in A$ is {\bf $A$-relevant} if $j(\gamma_1) = i(\gamma_2)$
and $A$-irrelevant otherwise.

\medskip

\noindent {\bf Level 0} 
\begin{enumerate}
\item Connect each $b^{U}_{i,j}$ to $I_{i,j}$ by a straight vertical segment, $\gamma^U_{i,j}$.
\item Connect each $b^{L}_{i,j}$ to $I_{i,j}$  by a piecewise linear path, $\gamma^L_{i,j}$, (smoothed at the corner) consisting of (i) the segment from $b^{L}_{i,j}$ to $(-1/2, \beta^R_{i,j})$, then (ii) the straight horizontal segment from $(-1/2, \beta^R_{i,j})$ to $I_{i,j}$.  Here, $\{i,j\} \cap \{k,k+1\} = \emptyset$.
\item For $(i,j) = (i,k)$ or $(i,j) = (k+1,k+2)$, connect $\mathit{sw}_{i,j}$ to $I_{i,j}$ by a piecewise linear path, $\gamma^{\mathit{sw}}_{i,j}$, (smoothed at the corners) consisting of (i) a horizontal segment from $\mathit{sw}_{i,j}$ to $x_1 = -1/4$, then (ii) the straight vertical segment from this point to $(-1/4, \beta^R_{i,j})$, then (iii) the straight horizontal segment from $(-1/4, \beta^R_{i,j})$ to $I_{i,j}$.  The second segment (ii) is identical for different values of $(i,j) = (i,k), (k+1,k+2)$, so we shift the segments slightly in the horizontal direction.  [The precise manner that this is done is irrelavent, since  intersections between any two $\gamma^{\mathit{sw}}_{i,j}$ are $A$-irrelevant.]
\end{enumerate}
Note that the contribution to $\varphi(\partial_A I_{i,j})$ from paths constructed at Level 0 is the $(i,j)$-entry of the matrix
\[
B_U+B_L+ A_{-,-}E_{k+1,k}+ E_{k+1,k+2}.
\]

\medskip

\noindent {\bf Level 1} 

The $A$-relevant intersections between paths defined at Level $0$ are as follows:
\begin{enumerate}
\item  For $i<m<j$, with $\{m,j\} \cap \{k,k+1\} = \emptyset$, $\gamma^U_{i,m}$ intersects $\gamma^L_{m,j}$ in an $\e$-neighborhood of $(\beta^U_{i,m},\beta^R_{m,j})$.
\item  For $i<m<j$, with $(m,j) = (m,k)$ or $(m,j) = (k+1,k+2)$, $\gamma^U_{i,m}$ intersects $\gamma^{\mathit{sw}}_{m,j}$ in an $\e$-neighborhood of $(\beta^U_{i,m},\beta^R_{m,j})$.
\item  For $i<m<j$, with $(m,j) = (m,k)$ or $(m,j) = (k+1,k+2)$, $\gamma^L_{i,m}$ intersects $\gamma^{\mathit{sw}}_{m,j}$ at $(-1/4,\beta^R_{i,m})$.
\end{enumerate}
All of these intersections are unique.

\begin{figure}

\quad

\quad

\labellist
\small
\pinlabel $\mathit{sw}_{m,k}$ [r] at 42 2
\pinlabel $b^{L}_{i,m}$ [r] at -4 82
\pinlabel $b^U_{i,m}$ [b] at 322 240
\pinlabel $c_{i,m}$ [t] at 322 78
\pinlabel $c_{i,k}$ [t] at 372 124
\pinlabel $c_{m,k}$ [l] at 424 178
\pinlabel $\lambda_{i,m,k}$ [b] at 206 134
\pinlabel $\gamma^{L}_{i,m}$ [b] at 206 86
\pinlabel $\gamma^{\mathit{sw}}_{m,k}$ [b] at 206 182
\pinlabel $\gamma_{i,m,k}^{\mathit{sw}}$ [l] at 354 160
\endlabellist
\centerline{ \includegraphics[scale=.6]{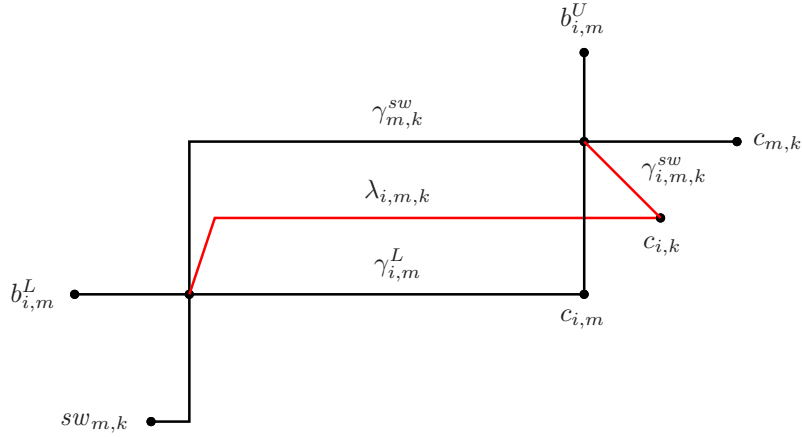} }

\quad

\caption{Several of the paths created at Levels 0 and 1.}
%The arrows indicate the direction along edges of all $(i,j)$-flows (negative gradients) other than the $(k+2,k+1)$ flow.  
%(right) The partial flow trees in $A$ described in (2) of Lemma \ref{lem:RegionA}.  
\label{fig:BLUA}
\end{figure}

Now, for each of the $A$-relevant intersection points we add a path to $A$ as follows:
\begin{enumerate}
\item  For $i<m<j$, with $\{m,j\} \cap \{k,k+1\} = \emptyset$, use a straight (negatively sloped) line segment from $x \in \gamma^U_{i,m} \cap \gamma^L_{m,j}$ to $I_{i,j}$.  Call this segment $\gamma^L_{i,m,j}$. 
\item  For $i<m<j$, with $(m,j) = (m,k)$ or $(m,j) = (k+1,k+2)$, also use a straight (negatively sloped) line segment from $x \in \gamma^U_{i,m} \cap \gamma^{sw}_{m,j}$ to $I_{i,j}$.  We  call these segments $\gamma^{\mathit{sw}}_{i,m,j}$.
\item  For $i<m<j$, with $(m,j) = (m,k)$ or $(m,j) = (k+1,k+2)$, we use a piece-wise linear path, $\lambda_{i,m,j}$, from $x \in \gamma^L_{i,m} \cap \gamma^{sw}_{m,j}$ to $I_{i,j}$.  Define $\lambda_{i,m,j}$ as (i) a vertical segment from $x$ up to $x_2 = \beta^R_{i,j}$ followed by  (ii) a horizontal segment to $I_{i,j}$.  Note that this agrees precisely with part of the previously defined $\gamma^{\mathit{sw}}_{i,j}$, so we can shift $\lambda_{i,m,j}$ slightly to make it disjoint from  $\gamma^{\mathit{sw}}_{i,j}$ except at its endpoint.  [This is not so important, as intersections between these two paths are $A$-irrelevant.]
\end{enumerate}
See Figure \ref{fig:BLUA}.

Note that since 
\[
w(\gamma^L_{i,m,j}) = b^U_{i,m} b^L_{m,j}; \quad w(\gamma^{\mathit{sw}}_{i,m,j}) = b^U_{i,m} \mathit{sw}_{m,j}; \quad \mbox{and} \quad w(\gamma^{\mathit{sw}}_{i,m,j}) = b^L_{i,m} \mathit{sw}_{m,j},
\]
    the contribution to $\varphi(\partial_A I_{i,j})$ from paths constructed at Level 1 is the $(i,j)$-entry of the matrix
\[
B_UB_L+ B_U(A_{-,-}E_{k+1,k}+ E_{k+1,k+2}) + B_L(A_{-,-}E_{k+1,k}+ E_{k+1,k+2}).
\]

\medskip

\noindent {\bf Level 2}
The only $A$-relevant intersections involving at least one path created at Level 1 are:
\begin{enumerate}
\item For $i<h<m<j$ with $(m,j) = (m,k)$ or $(m,j) = (k+1,k+2)$, $\lambda_{h,m,j}$ intersects $\gamma^U_{i,h}$ at a unique point in an $\e$-neighborhood of $(\beta^U_{i,j}, \beta^R_{h,j})$.
\end{enumerate} 
[The diagonal segments $\gamma^L_{i,m,j}$ and $\gamma^{\mathit{sw}}_{i,m,j}$ are seen to be disjoint from all other segments with lower index $i$ or upper index $j$ using the lexicographic ordering of the $c_{i,j}$ along $x_1=x_2$.  This is as in Proof of Theorem \ref{thm:112btrees}.]

We complete the construction of $A$ by adding, straight line segments, $\gamma_{i,h,m,j}$, from  $x \in \gamma^U_{i,h} \cap \lambda_{h,m,j}$ to $I_{i,j}$ for each such $i<h<m<j$.  These diagonal segments have no 
 $A$-relevant intersections with other paths in $A$ (again using the lexicographic ordering of the $c_{i,j}$), and they contribute the $(i,j)$-entry of the matrix
\[
B_UB_L(A_{-,-}E_{k+1,k}+ E_{k+1,k+2})
\]
to $\varphi(\partial_A I_{i,j})$.

To complete the proof, simply sum the contributions to $\varphi(\partial_A I_{i,j})$ that were identified at Levels 0-2.  The result is the $(i,j)$ entry of 
$(I+B_U)(I+B_L)(I+A_{-,-}E_{k+1,k}+E_{k+1,k+2})$.
\end{proof}

\subsubsection{$b^{RD}$-trees}  \label{sec:bRDtrees}

We turn now to $b^{RD}$-trees which were defined to be $b$-trees with all outputs at $b^R_{i,j}$ or $b^D_{i,j}$ Reeb chords (including possibly $b^D_{k+2,k+1}$, but not $\tilde{b}^R_{k+2,k+1}$) or at $\mathit{sw}_{k+2,k+1}$.  Due to the presence of $(k+2,k+1)$-flow lines as edges in $b^{RD}$-trees, we require a bit more preliminary work than usual before reaching a situation where we can apply Proposition \ref{prop:dAind}.

Recall the definition of regions $B^{RD}(i,j) \subset \widehat{N}(e^2_\alpha)$ for $1\leq i <j \leq n$ from Section \ref{ssec:112btrees} (as in Figure \ref{fig:BLUij}, but reflected across $x_1=x_2$).  We use the same definition here.  

Next, for $i<j$, with $(i,j) \neq (k+1,k+2)$, let 
\[
A^{RD}(i,j) = \left(B^{RD}(i,j) \cap \{x_2 \geq 1/4\} \right) \bigcup
\]
\[
 \left\{(x_1,x_2) \in \widehat{N}(e^2_\alpha) \, | \,\mathit{Min}(\beta^U_{i,j}, \beta^D_{i,j}) -\e \leq x_1 \leq \mathit{Max}(\beta^U_{i,j}, \beta^D_{i,j})+\e, \,  x_2 \leq 1/4\right\},
\]
while
\[
A^{RD}(k+1,k+2)= B^{RD}(k+1,k+2) \cap \{x_2 \geq 1/4\}.
\]
See Figure \ref{fig:BRDA}.

\begin{figure}

\quad

\labellist
\small
\pinlabel $x_1$ [b] at 282 16
\pinlabel $x_2$ [b] at 8 432
\pinlabel $1/2$ [t] at 32 -2
\pinlabel $\beta^D_{i,k+2}$ [t] at 128 -2
\pinlabel $\beta^D_{i,k+1}$ [tl] at 168 -2
\pinlabel $0$ [r] at -2 168
\pinlabel $1/4$ [r] at  -2 224
\pinlabel $\beta^R_{i,k+2}$ [r] at  -2 400 
\pinlabel $c_{i,k+2}$ [tl] at  174 394 
\pinlabel $B_1$ [r] at  80 300 
%\pinlabel $F_{k+1}=F_{k+2}$ [l] at  292 168 
\pinlabel $\partial\widehat{N}(e^2_\alpha)$ [l] at  278 322 
\endlabellist
\centerline{ \includegraphics[scale=.5]{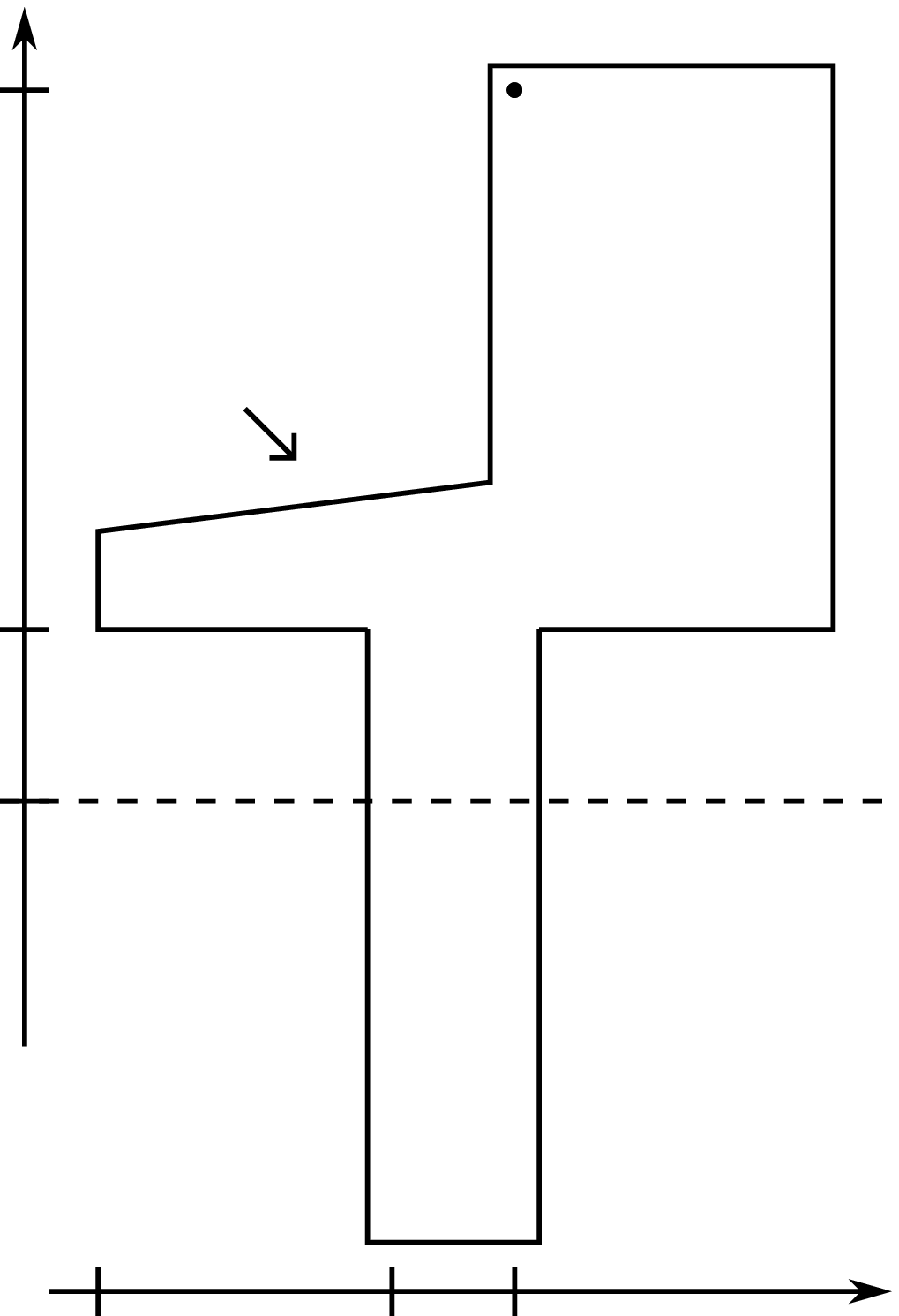} \quad \quad \quad \quad \quad
\labellist
\small
\pinlabel $x_1$ [l] at 326 46
\pinlabel $V$ [b] at 128 116
\pinlabel $\beta^D_{i,k+2}$ [t] at 40 36
\pinlabel $\beta^D_{i,k+1}$ [t] at 80 -2
\pinlabel $\beta^D_{k+2,k+1}$ [t] at 128 36
\pinlabel $\beta^D_{k+2,j}$ [t] at 176 -2
\pinlabel $\beta^D_{k+1,j}$ [t] at 272 36
\pinlabel $A^{RD}(i,k+1)$ [b] at 60 268
\pinlabel $A^{RD}(k+1,j)$ [b] at 224 268
\endlabellist
\includegraphics[scale=.5]{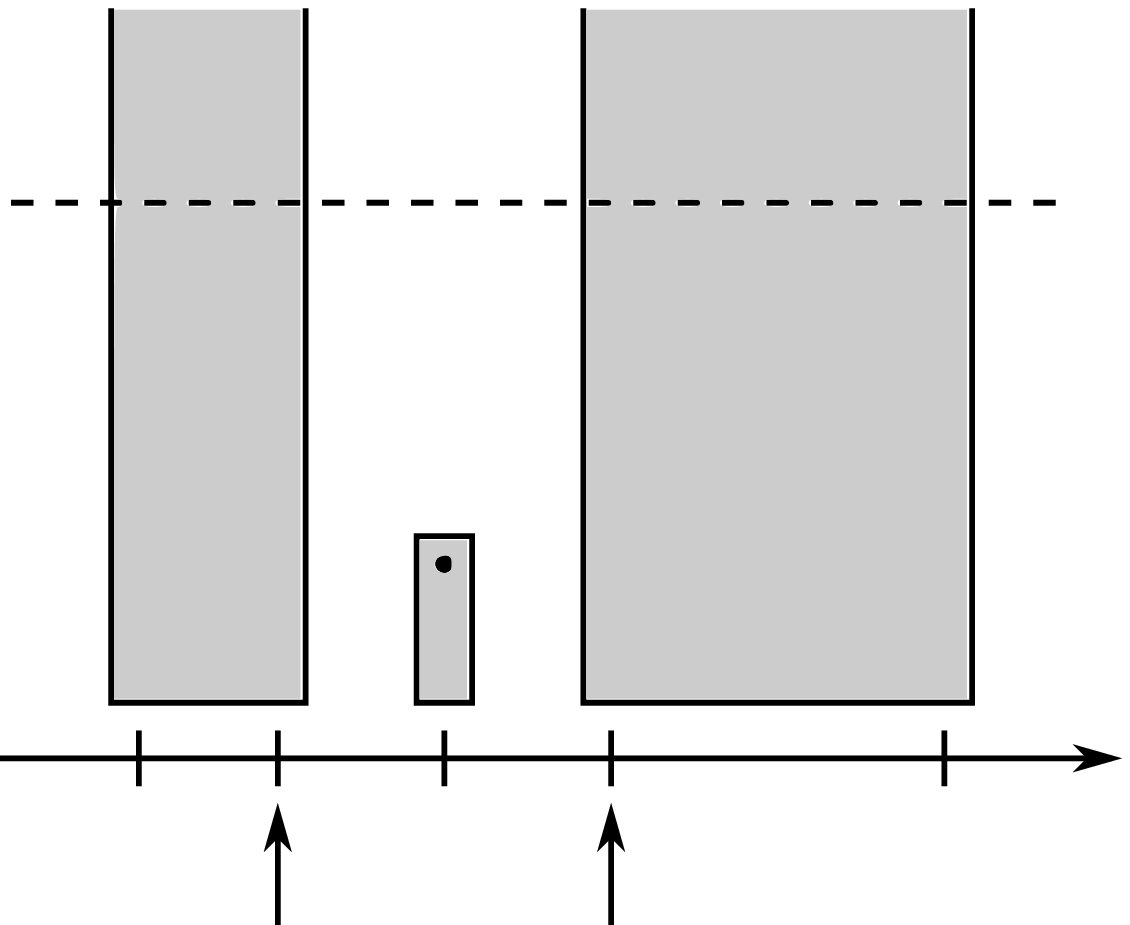} 
}

\quad

\caption{(left) The region $A^{RD}(i,k+2)$.  The $(k+1,k+2)$-crossing locus is pictured as a dotted line, and the segment $B_1$ is from Property \ref{pr:monotonicityIV}.  Note that $\beta^D_{i,k+2} = \beta^U_{i,k+1}$ and $\beta^D_{i,k+1} = \beta^U_{i,k+2}$ with the $\beta^U$ appearing in lexicographical order.  (right)  Below $x_2 = 1/4$, the regions $A^{RD}(i,k+1)$ and $A^{RD}(k+1,j)$ agree with $A^{RD}(i,k+2)$ and $A^{RD}(k+2,j)$, for $i<k+1$ and $k+2<j$.  These regions are disjoint from the strip $V$ that contains $\tilde{c}_{k+2,k+1}$.}
\label{fig:BRDA}
\end{figure}

\begin{lemma}  \label{lem:bRDijST}
Any edge, $\gamma$, of a $b^{RD}$-tree that does not begin with a switch either 
\begin{enumerate}
\item has $i<j$, and the intersection of its image with $\widehat{N}(e^2_\alpha)$ is in $A^{RD}(i,j)$, or 
\item is the $(k+2,k+1)$-switch flow line, or 
\item is the $b^D_{k+2,k+1}$-line which is contained in the strip 
\[
V = [\beta^D_{k+2,k+1}-\e,\beta^D_{k+2,k+1}+\e] \times [-1-1/32, \tilde{\beta}^R_{k+2,k+1} +\e].
\]
\end{enumerate}
Moreover, the image of $\gamma$ is entirely in $\widehat{N}(e^2_\alpha)$ unless $\gamma$ is part of a $b^X_{i,j}$-line for $X \in \{R,D\}$.
\end{lemma}
\begin{proof}
As a preliminary, note that the $b^D_{k+2,k+1}$-line is indeed contained in the strip $V$, since $-\nabla F_{k+2,k+1}$ points out along the left, right, and upper boundaries of $V$ (by Property \ref{pr:monotonicityI}).  Moreover, $V$ contains the portion of $N(e^1_D)$ where $b^{D}_{k+2,k+1}$ is located, and the $b^{D}_{k+2,k+1}$-line cannot cross the lower boundary of $\partial N(e^1_D)$ (by Proposition \ref{prop:bUS}).

In addition, note that for all $1\leq i<j \leq n$, $-\nabla F_{i,j}$ points out along all boundary segments of $A^{RD}(i,j)$.  [Along boundary segments shared with  $B^{RD}(i,j)\cap\{x_2 \geq 1/4\}$, this follows from Properties \ref{pr:1cells}, \ref{pr:monotonicityI}, and \ref{pr:monotonicityIV}.  Along the vertical lines $x_1 = \mathit{Min}(\beta^U_{i,j}, \beta^D_{i,j}) -\e$ and $x_1 = \mathit{Max}(\beta^U_{i,j}, \beta^D_{i,j})+\e$, this is from Property \ref{pr:monotonicityIIST}.]

Now, we prove the Proposition using induction on the number of $Y_0$'s in the PFT that begins with the edge $\gamma$.  In the base case, $\gamma$ must be either the $(k+2,k+1)$-switch flow line or one of the $b^R_{i,j}$- or $b^D_{i,j}$-lines.  In the latter case, the claimed bound on the location holds since, for $i<j$, as $t$ decreases from $+\infty$ these lines enter $\widehat{N}(e^2_\alpha)$ at a point on $\partial A^{RD}(i,j)$ (as specified by Proposition \ref{prop:bUS}).  Note that the segments of the $b^R_{i,j}$- or $b^D_{i,j}$-lines that lie outside of $\widehat{N}(e^2_\alpha)$ are all pairwise disjoint.

Suppose now that the PFT starting with $\gamma$ has at least one $Y_0$.  From Lemma \ref{lem:observations}, $\gamma$ must end with a $Y_0$, and the two PFTs starting with the outgoing branches are themselves $b^{RD}$-trees.  These outgoing branches, $\gamma_1$ and $\gamma_2$, are respectively $(i,m)$- and $(m,j)$-flow lines for some $i,m,j$.  We need to show that $i<j$ and that $\gamma \subset A^{RD}(i,j)$.  

\medskip

\noindent {\bf Case 1:} One of $(i,m)$ and $(m,j)$ is equal to $(k+2,k+1)$.  

\medskip

Consider the case where $(i,m) = (k+2,k+1)$; the other case is similar.  By the inductive hypothesis, $\gamma_1$ must be either the $(k+2,k+1)$-switch flow line or the $b^D_{k+2,k+1}$-line.   The latter case is impossible, since the strip $V$ is disjoint from all $A^{RD}(i,j)$.  [See Figure \ref{fig:BRDA} (right).]  Thus, $\gamma_1$ is the $(k+2,k+1)$-switch flow line.  Of course, $(m,j) = (k+1,k+2)$ is not possible, since all $(k+2,k+1)$- and $(k+1,k+2)$-flow lines are disjoint from one another.  Thus, $(m,j) = (k+1,j)$ with $k+2<j$, so $(i,j) = (k+2,j)$ satisfies $i<j$.  The $Y_0$ is located in $A^{RD}(k+1,j)$, below the crossing locus.  Note that below the crossing locus $A^{RD}(k+1,j)$ and $A^{RD}(k+2,j)$ coincide (because $\{\beta^U_{k+1,j}, \beta^D_{k+1,j}\} = \{\beta^D_{k+2,j}, \beta^U_{k+2,j}\}$), so the edge $\gamma$ has its end point in $A^{RD}(k+2,j)$ as required. [Since $-\nabla F_{k+2,j}$ points out along $\partial A^{RD}(k+2,j)$, $\gamma$ remains in $A^{RD}(i,j)$ as $t$ decreases.]

\medskip

\noindent {\bf Case 2:} Neither of $(i,m)$ or $(m,j)$ are equal to $(k+2,k+1)$.  

\medskip
Then, the inductive hypothesis shows that $i<m<j$, and the $Y_0$ at the end of $\gamma$ is in $A^{RD}(i,m) \cap A^{RD}(m,j)$.   To complete the proof we show that  $A^{RD}(i,m) \cap A^{RD}(m,j) \subset A^{RD}(i,j)$.  

First, we have
\[
A^{RD}(i,m) \cap A^{RD}(m,j) \cap\{x_2\geq 1/4\} = B^{RD}(i,m) \cap B^{RD}(m,j) \cap\{x_2\geq 1/4\} \subset 
\]
\[
B^{RD}(i,j) \cap\{x_2\geq 1/4\} = A^{RD}(i,j) \cap\{x_2\geq 1/4\}.
\]
To check that 
\[
A^{RD}(i,m) \cap A^{RD}(m,j) \cap\{x_2 <  1/4\} = \emptyset \subset A^{RD}(i,j) \cap\{x_2 < 1/4\},
\]
 note that for $(i,m), (m,j) \neq (k+1,k+2)$, the intervals $[\mathit{Min}(\beta^U_{i,m}, \beta^D_{i,m}) - \e, \mathit{Max}(\beta^U_{i,m}, \beta^D_{i,m})+\e]$ and 
$[\mathit{Min}(\beta^U_{m,j}, \beta^D_{m,j}) - \e, \mathit{Max}(\beta^U_{m,j}, \beta^D_{m,j})+\e]$  are disjoint.
\end{proof}

\begin{corollary}  \label{cor:btreebLUbRD}
Any $b$-tree is either a $b^{LU}$- or $b^{RD}$-tree.
\end{corollary}
\begin{proof}
This is verified using induction on the number of $Y_0$'s.  At the inductive step, Lemmas \ref{lem:BLUijST} and \ref{lem:bRDijST} show that except for possibly branches that are $(i,k)$- $(k+1,k+2)$- or $(k+2,k+1)$-flow lines, all edges of $b^{LU}$-trees are disjoint from all edges of $b^{RD}$-trees.  No $Y_0$ can occur between two flow lines with indices of the form $(i,k),(k+1,k+2),$ or $(k+2,k+1)$.
\end{proof}

Consider the set $X$ of ordered pairs $(x,\Gamma)$ where $\Gamma$ is a PFT that is a $b^{RD}$-tree starting on the line $x_2 =1/4$ with initial point $x = (x_1, 1/4)$.  
As we will see in Lemma \ref{lem:onlybRD} the elements $(x,\Gamma)$ of $X$ are described as follows:
\begin{enumerate}
\item For $i<j$, $(i,j) \neq (k+1,k+2)$, the portion of the $b^{D}_{i,j}$-line beginning at its unique intersection point with $x_2=1/4$.
\item For $i<k+1$, the $(k+2,k+1)$-switch flow line intersects the $b^D_{i,k+2}$-line in an \emph{odd} number of points, $q_1, \ldots, q_{2r+1}$.  [This is because the endpoints of the $(k+2,k+1)$-switch flow line are at $\tilde{c}_{k+2,k+1}$ and the $(k+2,k+1)$-switch flow line, and these lie on opposite sides of the portion of the $b^{D}_{i,k+2}$-line that is below the crossing locus.  See Figure \ref{fig:BRDtrees}.]  For each such intersection point $q_{s}$, there is a unique PFT $\Gamma^{s}_{i,k+2,k+1}$ whose top edge is the $(i,k+1)$-flow line starting at $x_{2} = 1/4$ and ending at $q_s$.  [The $(i,k+1)$-flow line through $q_s$ has a unique intersection with $x_2=1/4$, since it
 is contained in $A^{RD}(i,k+1)$ and,  in $A^{RD}(i,k+1)$, $-\nabla F_{i,k+1}$ points down along $x_2 = 1/4$.]  The PFT $\Gamma^{s}_{i,k+2,k+1}$ then has a $Y_0$ with the outgoing edges limiting to $b^{D}_{i,k+2}$, and following the remainder of the $(k+2,k+1)$-switch GFT. 
\item For $k+2<j$, the $(k+2,k+1)$-switch flow line intersects the $b^D_{k+1,j}$-line in an \emph{even} number of points, $q_1, \ldots, q_{2r}$, where it is possible that $r=0$.  [This is because the endpoints of the $(k+2,k+1)$-switch flow line both sit to the left of the portion of the $b^{D}_{k+1,j}$-line that is below the crossing locus.]  For each such intersection point $q_{s}$, there is a unique PFT $\Gamma^{s}_{k+2,k+1,j}$ whose top edge is the $(k+2,j)$-flow line starting at $x_{2} = 1/4$ and ending at $q_s$.  [This flow line has a unique intersection with $x_2=1/4$.]  The PFT $\Gamma^{s}_{k+2,k+1,j}$ then has a $Y_0$ with the outgoing edges limiting to $b^{D}_{k+1,j}$, and following the remainder of the $(k+2,k+1)$-switch GFT. 
\end{enumerate}

\begin{figure}

\quad

\quad

\labellist
\small
\pinlabel $\tilde{c}_{k+2,k+1}$ [t] at 320 72
\pinlabel $b^D_{i,k+2}$ [t] at 176 8
\pinlabel $b^D_{k+1,j}$ [t] at 432 8
\pinlabel $A^{RD}(i,k+2)$ [b] at 220 236
\pinlabel $A^{RD}(k+1,j)$ [b] at 400 236
\endlabellist
\centerline{ \includegraphics[scale=.5]{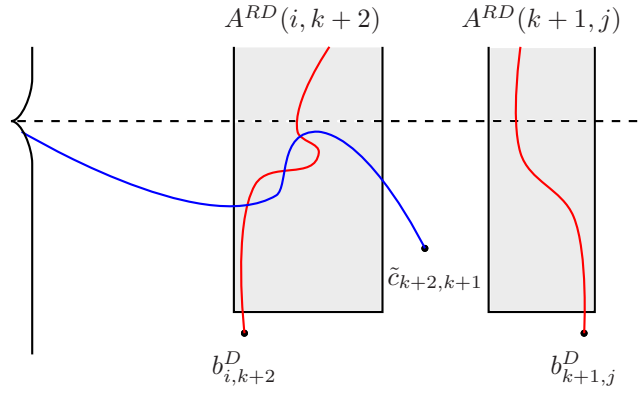} }

\quad

\caption{Topological considerations show that the $b^D_{i,k+2}$-line and the $b^D_{k+1,j}$-line (red) intersect the $(k+2,k+1)$-switch flow line (blue) in an odd and even number of points respectively.}
\label{fig:BRDtrees}
\end{figure}

\begin{lemma} \label{lem:onlybRD}
These are the only PFT $b^{RD}$-trees starting along $x_2 =1/4$.
\end{lemma}
\begin{proof}
Verify by induction on, $N$, the number of $Y_0$'s in $\Gamma$ with $(x,\Gamma) \in X$. In the base case, the only $b^{RD}$-trees without $Y_0$'s are subsets of the  $b^{R}_{i,j}$-lines, $b^D_{i,j}$-lines, and the $(k+2,k+1)$-switch GFT.  Only the $b^D_{i,j}$ lines with $i<j$ intersect $x_2=1/4$, and each one of them does indeed intersect $x_2=1/4$
 in a unique point.  

Now, suppose that $(x,\Gamma) \in X$ is such that $\Gamma$ has $N \geq 1$ $Y_0$-vertices.   

Using Lemma \ref{lem:bRDijST}, we notice that for $(x, \Gamma) \in X$, the entire image of $\Gamma$ lies strictly below $x_2=1/4$, except for at the initial point, $x$.  [For all $i<j$, $-\nabla F_{i,j}$ points down along $\{x_2=1/4\} \cap A^{RD}_{i,j}$, and any $(k+2,k+1)$-branches of $\Gamma$ lie below the crossing locus.]  Also, note that any edge of $\Gamma$, that is not part of the $(k+2,k+1)$-switch GFT or the $b^D_{k+2,k+1}$-line would limit to $c_{i,j}$ if extended to allow $t \rightarrow -\infty$.   Thus, the PFTs $\Gamma_1$ and $\Gamma_2$ that begin with the outgoing edges at the first $Y_0$ of $\Gamma$ are themselves subsets of PFTs that are $b^{RD}$-trees starting along $x_2=1/4$, or are subsets of the $(k+2,k+1)$-switch GFT or the $b^{D}_{k+2,k+1}$-line. 
When $x_2< 1/4$, the regions $A^{RD}(i,m)$ and $A^{RD}(m,j)$ are all disjoint and are disjoint from the $b^{D}_{k+2,k+1}$-line.  Thus, the only possibility is that one of $\Gamma_1$ or $\Gamma_2$ begins with the  $(k+2,k+1)$-switch flow line and the other $\Gamma_i$ begins with either a $(i,k+2)$ or $(k+1,j)$-flow line
 and has $N-1$ $Y_0$ vertices.  [No PFT starting with the $(k+2,k+1)$-switch flow line can have any $Y_0$-vertices by Lemma \ref{lem:PFTrees}.]  By the inductive hypothesis, the other $\Gamma_i$ that begins with a $(i,k+2)$ or $(k+1,j)$-flow can only be (a subset of) one of the PFTs from $X$ as in (1)-(3) above.  It cannot be that the top branch of this $\Gamma_i$ is as in (2) or (3) since then the outgoing branches of the $Y_0$ would be either a $(k+2,k+1)$- and $(i,k+1)$-flow line or a $(k+2,k+1)$- and $(k+2,j)$-flow line.  Thus, $\Gamma_i$ must be as in (1), and so $\Gamma$ must be as in (2) or (3).
\end{proof}

Now, we have a disk datum $\left(D, \{b^R_{i,j}\} \cup \{x \, | \, (x, \Gamma) \in X\}, I_{i,j} \right)$ where $D$ is $\bigcup B^{RD}(i,j) \cap \{x_2 \geq 1/4\}$ with the squares $\mathit{Sq}_{i,j}$ that contain the $c_{i,j}$ removed, and the points $x$ with $(x,\Gamma) \in X$ are assigned upper and lower indices that agree so that the first branch of $\Gamma$ is an $(i(x), j(x))$-flow line.  As usual, the $I_{i,j}$ are the lower and right edges of $\overline{\mathit{Sq}_{i,j}}$.

We associate a generalized $b$-manifold, $A$, to this disk datum consisting of paths $\gamma$ that are the intersection with $D$ of the top branches of GFTs that are $b^{RD}$-trees.  
\begin{lemma}  \label{lem:genBRDmfd}
Defined as above, $A$ is a generalized $b$-manifold for $\left(D, \{b^R_{i,j}\} \cup \{x \, | \, (x, \Gamma) \in X\}, I_{i,j} \right)$.  Moreover,
\[
\varphi( \partial_A I_{i,j}) = \partial_{b^{RD}} c_{i,j}
\]
where $\varphi$ is the $\Z/2$-algebra homomorphism defined by $\varphi(b^{R}_{i,j}) = b^{R}_{i,j}$ and $\varphi(x) = w(\Gamma)$ for $(x,\Gamma) \in X$. 
\end{lemma} 

\begin{proof}
That $A$ is a generalized $b$-manifold follows as usual, since if the top edge, $\gamma$, of a $b^{RD}$-tree, $\Gamma$, intersects $D$, then it starts at some $c_{i,j}$ and either ends at a $Y_0$ within $D\cap \{x_2 \geq 1/4\}$ or leaves this region at one of the points in  $\{b^R_{i,j}\} \cup \{x \, | \, (x, \Gamma) \in X\}$.  That $\varphi(w(\gamma))= w(\Gamma)$ follows from induction on the number of $Y_0$'s that $\Gamma$ has with image in $D \cap \{x_2 \geq 1/4\}$, where the base case is arranged by the definition of $\varphi$.   
\end{proof}

\begin{proposition}  \label{prop:bRDcomp}
For all $1 \leq i<j \leq n$, $\partial_{b^{RD}}c_{i,j}$ agrees with the $(i,j)$-entry of the matrix
\[(I+B_R)(I+B_D + B_DE_{k+2,k+1})
\]
with $B_R$ and $B_D$ as in Theorem \ref{thm:SwallowComp}.
\end{proposition}
\begin{proof}
From Lemma \ref{lem:genBRDmfd} and Proposition \ref{prop:dAind}, we see that $\varphi( \partial_{A} I_{i,j}) = \partial_{b^{RD}} c_{i,j}$ for \emph{any} generalized $b$-manifold associated to the disk datum $\left(D, \{b^R_{i,j}\} \cup \{x \, | \, (x, \Gamma) \in X\}, I_{i,j}\right)$.  Construct a generalized $b$-manifold $A = \sqcup A_r$ as follows:

\medskip

\noindent {\bf Level 0} 
\begin{enumerate}
\item  Connect each $b^R_{i,j}$ to $I_{i,j}$ via a straight horizontal segment, $\gamma^R_{i,j}$.
\item  For each $(x,\Gamma) \in X$, connect $x$ to the appropriate $I_{i,j}$ by a path, $\gamma_x$, that is (i) a line segment from $x = (x_1, 1/4)$ to $(y, 3/8)$ with $y \in (\beta^U_{i(x),j(x)} - \e, \beta^U_{i(x),j(x)} + \e)$, then (ii) a straight vertical segment connecting $(y,3/8)$ to $I_{i(x),j(x)}$. 
\end{enumerate}
Notice that any intersections between the paths $\gamma_x$ are not $A$-relevant, since all endpoints of $\gamma_x$ have $x_1$-coordinate in the interval $J_{i,j}:=[\mathit{Min}(\beta^U_{i,j}, \beta^D_{i,j}) - \e, \mathit{Max}(\beta^U_{i,j}, \beta^D_{i,j})+\e]$ where $i=i(x)$ and $j=j(x)$;  as we have already observed, for $i<m<j$, $J_{i,m} \cap J_{m,j} = \emptyset$.

Recall that above Lemma \ref{lem:onlybRD} we divided all $(x,\Gamma) \in X$ into three families (1)-(3).  The paths corresponding to $(x,\Gamma) \in X$ as in (1) contribute a $b^{D}_{i,j}$ term to $\varphi(\partial_A I_{i,j})$; as in (2) contribute an \emph{odd} number of $b^D_{i,k+2} \cdot 1$ terms to $\varphi(\partial_A I_{i,k+1})$ for $i<k+1$; as in (3) contribute an \emph{even} number of $1 \cdot b^D_{k+1,j}$ terms to 
$\varphi(\partial_A I_{k+2,j})$ for $k+2<j$.  As we use $\Z/2$-coefficients, for $i<j$, the total contribution to $\varphi(\partial_A I_{i,j})$ from paths defined at Level 0 is the $(i,j)$-entry of the matrix
\[
B_R + B_D + B_D E_{k+2,k+1}.
\]

\medskip

\noindent {\bf Level 1}

For each $i<m<j$ with $(m,j) \neq (k+1,k+2)$, every path $\gamma_x$ with $(i(x),j(x)) = (m,j)$ intersects every $\gamma^R_{i,m}$ exactly once at a point in $[\beta^U_{m,j}-\e, \beta^U_{m,j}+\e]\times [\beta^R_{i,m}-\e, \beta^R_{i,m}+\e]$.  Connect each of these points to $I_{i,j}$ using a straight line segment.  An earlier argument (see the proof of Theorem \ref{thm:112btrees}) based on the lexicographic ordering of the $\beta_{i,j}$ shows that none of the segments constructed at Level 1 has any $A$-relevant intersections with other paths from $A$.  Thus, the construction of $A$ is complete.

Keeping in mind the $\Z/2$ coefficients, the contribution to $\varphi(\partial_A I_{i,j})$ from paths defined at Level 1 is the $(i,j)$-entry of the matrix
\[
B_R(B_D + B_DE_{k+2,k+1}).
\]
Thus, in total $\varphi(\partial_A I_{i,j})$ is the $(i,j)$-entry of the matrix $(I+B_R)(I+B_D +B_DE_{k+2,k+1})$.
\end{proof}

%A GFT that is a $b$-tree with $Y_0$-vertices is uniquely determined by the first $Y_0$ that occurs in the tree together with the PFTs (also $b$-trees) that follow the $Y_0$.  Thus, Proposition \ref{prop:N1} identifies all $b$-trees starting at some $c_{i,j}$ and having $N=1$ $Y_0$ vertices.

%\begin{proposition}  \label{prop:db}
%\[
%\partial_b C = [I+B_U][I+B_L][I+A_{-,-}E_{k+1,k} + E_{k+1,k+2}] + [I+B_R][I+B_D + B_{D} E_{k+2,k+1}].
%\]
%\end{proposition}

%\begin{proof}
%Multiply this out and cancel $I+I=0$.  The remaining terms are precisely those identified in Corollaries \ref{cor:N0}-\ref{cor:N2}. 
%\end{proof}

\begin{proof}[Proof of Theorem \ref{thm:SwallowComp}]
%The formula for $\partial C$ follows from equation (\ref{eq:dcdb}) together with Theorem \ref{thm:dc} and Proposition \ref{prop:db}.
Using Proposition \ref{prop:dbc} followed by Corollary \ref{cor:btreebLUbRD}, we compute
\[
\partial C = \partial_c C + \partial_b C +X = \partial_c C + \partial_{b^{LU}} C + \partial_{b^{RD}} C +X 
\]
where $X$ has entries in the ideal generated by $\tilde{c}_{k+2,k+1}$ and $\tilde{b}^R_{k+2,k+1}$.  Substituting the computations established in Theorem \ref{thm:dc} and   Propositions \ref{prop:bLUCompST} and \ref{prop:bRDcomp}, this becomes
\[
A_{+,+}C + C(I+E_{k+2,k+1})A_{-,-}(I+E_{k+2,k+1}) + 
\]
\[
(I+B_U)(I+B_L)(I+A_{-,-}E_{k+1,k}+E_{k+1,k+2}) + (I+B_R)(I+B_D +B_DE_{k+2,k+1}).
\]
[Note that the two $I$ terms cancel out, so that the entries on and below the main diagonal are all zero, just as in $\partial C$.]
This establishes the formula for $\partial C$ as stated in Theorem \ref{thm:SwallowComp}. 
\end{proof}

\section{Isomorphism between the cellular DGA and LCH}  \label{sec:Iso}

In the preceding three sections, we have computed, at least up to terms involving the exceptional generators, the sub-DGAs $(\lchA(e^d_{\alpha}), \partial)$ of $(\lchA, \partial)$ generated by Reeb chords located in the neighborhood of a single $0$, $1$, or $2$-cell, $e^d_\alpha$, of the square decomposition $\mathcal{E}_\pitchfork$.  In the present section, we combine these local computations to give a computation of $(\lchA,\partial)$ up to stable tame isomorphism (in Proposition \ref{prop:lchAI}), and then use this result to show that $(\lchA,\partial)$ is equivalent to the Cellular DGA of $L$ (in Proposition \ref{prop:IsoProof}).  

\subsubsection{Local vs. global computation of $(\lchA, \partial)$}

%In the preceding three sections, we have computed the sub-DGAs of $(\lchA, \partial)$ generated by Reeb chords located in the neighborhood of a single $0$, $1$, or $2$-cell of the transverse square decomposition $\mathcal{E}_\pitchfork$. 

In computing the sub-DGAs $(\lchA(e^d_\alpha),\partial)$, we have used different notations for the same Reeb chords depending on which cell was under consideration.  To address this ambiguity, we take the following formal approach when considering the entire DGA of $L$, $(\lchA,\partial)$.  

For a $d$-cell $e^d_\alpha$ of $\mathcal{E}_\pitchfork$, where $d=0,1,$ or $2$, we write
\[
(x, e^d_{\alpha})
\]
to indicate the Reeb chord that is notated by $x$ when considered in the context of the neighborhood $N(e^d_\alpha)$.  Here, $x$ has one of the forms $a_{i,j}$, $b_{i,j}$, or $c_{i,j}$ possibly decorated with super-scripts and/or a tilde.  Such a notation then specifies a well-defined generator of $(\lchA,\partial)$.  However, the same generator may have several different notations.  For example, suppose that a $1$-cell $e^1_\alpha$ has intial vertex $e^0_{\beta_0}$ and terminal vertex $e^0_{\beta_1}$, and that $\tilde{L}$ has a single crossing above $e^1_\alpha$, i.e. the $1$-cell is of type (1Cr), with the crossing involving sheets $S_k$ and $S_{k+1}$ as labelled above $x=+1$.  Then, for $i<k$, we have equalities
\[
(a_{i,k},e^0_{\beta_0}) = (a^-_{i,k+1},e^1_{\alpha});  \quad (a_{i,k},e^0_{\beta_1}) = (a^+_{i,k},e^1_{\alpha}); \quad \mbox{and} \quad (a_{k,k+1}, e^0_{\beta_0}) = (a^-_{k+1,k}, e^1_\alpha). 
\]
If a Type (13) $2$-cell, $e^2_\alpha$ has boundary $1$-cells $e^1_{X}$ with location indicated by $X \in \{L, D, R, U\}$, and boundary $0$-cells $e^0_{\pm,\pm}$, then, for $i<k$,
\[
(b^D_{i,k+1}, e^2_\alpha) = (b_{i,k+2}, e^1_{D}); \quad (b^L_{i,k+2}, e^2_\alpha) = (b_{i,k}, e^1_{L});  \quad \mbox{and} \quad (a^{-,-}_{i,k+1}, e^2_\alpha) = (a_{i,k}, e^0_{-,-}).
\]
[Here, we recall the Convention \ref{c:Sktilde}.]

Consequently, the generators of $\lchA$ should be viewed as equivalence classes of pairs $(x, e^d_{\alpha})$ with $x$ a generator for $\lchA(e^d_\alpha)$, such that two pairs are equivalent if they correspond to the same Reeb chord of $\tilde{L}$.  Often, the cell $e^d_{\alpha}$ that we consider a generator of $\lchA$ in will be clear from context, and then we simply write $x$ for $(x, e^d_\alpha)$.

\begin{remark} 
% \footnote{Mike:  Do you want to extend this remark, or have another remark in the Cellular DGA section in order to: \dr{Describe the pushout square situation, when we consider the union of two overlapping cells.  Reference Sivek and Harper-Sullivan.}? \ms{8/10/15: Does $A_{LCH}/I$ have push-out square structure? If so something should definelty be said in the short paper.} \dr{Yes.  The same remark applies to the cellular DGA since we have sub-DGAs for each cell with inclusions into the global DGA.}}
  In categorical language, for each cell $e^d_\alpha$ of $\mathcal{E}_\pitchfork$, we have an algebra $\mathcal{A}(e^d_\alpha)$ generated by symbols $(x, e^d_\alpha)$.  The cells of $\mathcal{E}_\pitchfork$ form a category $\mathcal{C}$ with a unique morphism $e^d_\beta \rightarrow e^{d'}_\alpha$ whenever $e^d_\beta \subset \overline{e^{d'}_\alpha}$.  Moreover, there is a functor from $\mathcal{C}$ to DGAs that assigns to morphisms the resulting inclusions $\mathcal{A}(e^d_\beta) \hookrightarrow \mathcal{A}(e^{d'}_\alpha)$ obtained from identifying the different notations for common Reeb chords.  We have thus constructed an inductive  system of DGAs.  As all of these DGAs include into $(\lchA, \partial)$, and any generator of $(\lchA,\partial)$ belongs to one of the $\mathcal{A}(e^d_\alpha)$, it is immediate that the direct limit of this system is the DGA of $\tilde{L}$, $(\lchA, \partial).$

This can be viewed as a refinement of the push-out square (Seifert-Van Kampen Theorem) proved for $\lchA$ for one-dimensional Legendrian knots \cite{Sivek11} and two-dimensional Legendrian surfaces \cite{HarperSullivan}.
\end{remark}

\subsection{The differential ideal $I$}
%\footnote{\ms{Maybe title subsection ``A simpler DGA" and then title sub subsection ``Differential of exceptional saddles" etc?}}  
\label{ssec:DiffExceptional}

Let $I$ denote the ideal generated by
\begin{itemize}
\item $(b^D_{k+2,k+1}; e^2_\alpha) +1$ for all Type (13) squares $e^2_\alpha$; 
\item $(b^D_{l-1,l-2}; e^2_\alpha) +1$ for all Type (14) squares $e^2_\alpha$; and
\item all other Reeb chords that are exceptional generators (for at least one $2$-cell). 
\end{itemize}
(Recall the definition of {\it exceptional generator} from Definition \ref{def:ex-gen}.)

%The quotient $\lchA/I$ has generating set given by all non-exceptional Reeb chords.

\begin{theorem}  \label{thm:Exceptional}
The ideal $I \subset \lchA$ satisfies $\partial (I) \subset I$ so that $(\lchA/I, \partial)$ is a semi-free DGA with generating set given by the non-exceptional Reeb chords of $\tilde{L}$.  Moreover, $(\lchA/I, \partial)$ is stable tame isomorphic to $(\lchA,\partial)$.
\end{theorem}

%Our computation of LCH will be greatly simplified by the following Theorem.

%\begin{theorem} \label{thm:Exceptional} (Suppose $L$ does not have swallowtail\footnote{For each swallowtail point there is an exceptional generator that should be replaced by $1$ rather than $0$.  I think we should try to handle all parts of the argument for the swallowtail in a separated section. DR} points.)
%The LCH DGA of $L$ is stable tame isomorphic to a semi-free triangular DGA obtained by removing all exceptional generators from the generating set, and replacing all occurrences of exceptional generators in the differential of non-exceptional generators with $0$.
%\end{theorem}

The proof of Theorem \ref{thm:Exceptional} is given at the end of this subsection.  
In preparation for the proof, we examine the exceptional generators more closely (a square-by-square list appears below) and compute some of their
 differentials.  

\subsubsection{Differential of exceptional generators in the $1$-skeleton} \label{sssec:ExGen1}
First, we consider generators of the form $\tilde{b}_{i,j}$.  These generators are always exceptional, and they occur only in $1$-cells of type (1Cr) and (2Cr).
\begin{lemma}  \label{lem:Exc1}
\begin{itemize}
\item[(1Cr)]  
Let $e^1_\alpha$ be a $1$-cell of type (1Cr) so that sheets $S_k$ and $S_{k+1}$ cross above $N(e^1_\alpha)$.  In $\lchA$, we have
\[
\partial \tilde{b}_{k+1,k} = a^-_{k+1,k}.
\]

\item[(2Cr)] 
Let $e^1_\alpha$ be a $1$-cell of type (2Cr) so that sheet $S_{k+2}$ crosses $S_{k+1}$ and then $S_{k}$ as we follow $e^1_\alpha$ from its terminal point to its initial point.
The exceptional generators $\tilde{b}_{k+2,k}$ and $\tilde{b}_{k+2,k+1}$ satisfy   
\[\partial \tilde{b}_{k+2,k} = a^-_{k+2,k};  \quad \mbox{and} \] 
\begin{equation} \label{eq:dbtilde}
\partial \tilde{b}_{k+2,k+1} = a^-_{k+2,k+1}  +  X
\end{equation}
 where $X$ denotes a term belonging to the ideal generated by $\tilde{b}_{k+2,k}, a^-_{k+2,k}$.  
\end{itemize}
\end{lemma}

\begin{proof}
In the (1Cr) case, note that the sheets $S_{k+1}$ and $S_k$ are adjacent above the portion of $N(e^1_\alpha)$ 
where $F_{k+1,k}>0$, so there is no possibility for a GFT starting at $\tilde{b}_{k+1,k}$ to have internal vertices.  Moreover, the $(k+1,k)$-descending manifold of $\tilde{b}_{k+1,k}$ has two non-constant flowlines.  One of them limits to $a^-_{k+1,k}$ while the other terminates at the crossing locus.  [This is verified as in the proof of Proposition \ref{prop:bUS} (1):  If the two flowlines in the unstable manifold of $\tilde{b}_{k+1,k}$ either both limit to $a^-_{k+1,k}$ or both reach the crossing locus, then one of the flow lines from the stable manifold would have to start at a local maximum of $F_{k+1,k}$ in $N(e^1_\alpha)$.] 
%and the other reaches the level set $F_{k+1,k}=0$ at the crossing.

In the (2Cr) case, the first formula is verified as in the (1Cr) case. 
To verify the second formula, note that the only opportunity for a $(k+2,k+1)$-flow starting at $\tilde{b}_{k+2,k+1}$ to have a $Y_0$-vertex is if it splits into a $(k+2,k)$-flow and a $(k,k+1)$-flow somewhere in the region where $S_k$ lies between $S_{k+2}$ and $S_{k+1}$.  The edge that is a $(k+2,k)$-flow must end either at $\tilde{b}_{k+2,k}$ or $a^-_{k+2,k}$ by considerations as in the (1Cr) case above.
\end{proof}

\subsubsection{Exceptional generators in $2$-cells} \label{sssec:ExGen2}

\begin{lemma} \label{lem:ExGen2} The exceptional generators in $N(e^2_\alpha)$ for a $2$-cell $e^2_\alpha$ of type (1)-(14) are as indicated in the following table.  Note that we use the notation $b_{k+2,\{k+1,k\}}$ to denote a pair of generators of the form $b_{k+2,k+1}$ and $b_{k+2,k}$.  Moreover, the differentials of exceptional generators of the form $\tilde{c}_{i,j}$ are as indicated.

\centerline{\begin{tabular}{c|c|c}  \mbox{Square Type} & \mbox{Exceptional Generators} & \mbox{Differentials} \\ & & \\
(1) & $\emptyset$ & N/A \\ & & \\
(2) & $\begin{array}{ccc}   \tilde{b}^U_{k+1,k} &  \tilde{b}^D_{k+1,k} & \tilde{c}_{k+1,k}  \\   a^{-,+}_{k+1,k} & a^{-,-}_{k+1,k} & b^L_{k+1,k}   \end{array}$   &  $\partial \tilde{c}_{k+1,k}= \tilde{b}^U_{k+1,k} + \tilde{b}^D_{k+1,k} + b^L_{k+1,k}$    \\ & & \\
(3) & $\begin{array}{cc}   \tilde{b}^L_{k+1,k} & \tilde{c}_{k+1,k}  \\   a^{-,-}_{k+1,k} & \tilde{b}^D_{k+1,k}      \end{array}$  & $\partial \tilde{c}_{k+1,k}= \tilde{b}^L_{k+1,k} + \tilde{b}^D_{k+1,k}$ \\ & & \\
(4) & $\begin{array}{ccc}  \tilde{b}^R_{k+1,k} &  \tilde{b}^D_{k,k+1} & \tilde{c}_{k+1,k} \\   a^{+,-}_{k+1,k} & a^{-,-}_{k,k+1} & b^D_{k+1,k}    \end{array}$ &  $\partial \tilde{c}_{k+1,k}= \tilde{b}^R_{k+1,k} + b^D_{k+1,k}$    \\ & & \\
(5) & $\begin{array}{ccc}  \tilde{b}^L_{k+2, \{k+1,k\}} & \tilde{b}^R_{k+2, \{k+1,k\}} & \tilde{c}_{k+2, \{k+1,k\}} \\
  a^{-,-}_{k+2, \{k+1,k\}} & a^{+,-}_{k+2, \{k+1,k\}} & b^D_{k+2, \{k+1,k\}} \end{array}$ &  $\begin{array}{c} \partial \tilde{c}_{k+2,k} = \tilde{b}^L_{k+2,k}+\tilde{b}^R_{k+2,k}+ b^D_{k+2,k}; \\
 \partial \tilde{c}_{k+2,k+1} = \tilde{b}^L_{k+2,k+1}+\tilde{b}^R_{k+2,k+1}+ b^D_{k+2,k+1} + \\ O(k+2,k); \end{array}$    \\ & & \\
(6)  & $\begin{array}{cccc}  \tilde{b}^R_{k+2, \{k+1,k\}} & \tilde{b}^L_{k+2, k+1} & \tilde{b}^D_{k, k+2} & \tilde{c}_{k+2, \{k+1,k\}}  \\
  a^{+,-}_{k+2, \{k+1,k\}} & a^{-,-}_{k+2, k+1} & a^{-,-}_{k, k+2} & b^D_{k+2, \{k+1,k\}} \end{array}$ &  $\begin{array}{c} \partial \tilde{c}_{k+2,k} = \tilde{b}^R_{k+2,k}+ b^D_{k+2,k}; \\
 \partial \tilde{c}_{k+2,k+1} = \tilde{b}^R_{k+2,k+1}+ b^D_{k+2,k+1} + \tilde{b}^L_{k+2,k+1} + \\ O(k+2,k); \end{array}$    \\ & & \\
(7) & [Similar to (2)] $\times 2$ & \\ & & \\
(8) & $\left\{\begin{array}{ccc} \tilde{b}^L_{k+2, \{k+1,k\}} & \tilde{b}^R_{k+2, \{k+1,k\}} & \tilde{b}^U_{k+1,k} \\
  a^{-,-}_{k+2, \{k+1,k\}} & a^{+,-}_{k+2, \{k+1,k\}}& a^{-,+}_{k+1,k} \\ & & \\
  \tilde{b}^D_{k+1,k} & \tilde{c}_{k+2, \{k+1,k\}} & \tilde{c}_{k+1,k}   \\  a^{-,-}_{k+1,k} & b^D_{k+2, \{k+1,k\}} & b^L_{k+1,k}  
 \end{array} \right\}$ &  $\begin{array}{c} \partial \tilde{c}_{k+2,k} = b^D_{k+2,k} + X; \\
 \partial \tilde{c}_{k+1,k} = b^L_{k+1,k} + Y \\ \partial \tilde{c}_{k+2,k+1} = b^D_{k+2,k+1} + Z; \end{array}$    \\ & & \\
(9) & $\emptyset$ & N/A \\ & & \\
(10) & Similar to (2) & \\ & & \\
(11) & $\emptyset$ & N/A \\ & & \\
(12) & $\begin{array}{ccc} \tilde{b}^R_{k+2, \{k+1,k\}}  & \tilde{c}_{k+2, \{k+1,k\}} \\
 a^{+,-}_{k+2, \{k+1,k\}}  & b^D_{k+2, \{k+1,k\}}   \end{array}$ &  $\begin{array}{c} \partial \tilde{c}_{k+2,k} =\tilde{b}^R_{k+2,k}+ b^D_{k+2,k}; \\
 \partial \tilde{c}_{k+2,k+1} = \tilde{b}^R_{k+2,k+1}+ b^D_{k+2,k+1} + O(k+2,k). \end{array}$    \\ & & \\
(13) & $\begin{array}{ccc} \tilde{b}^R_{k+2,k+1} &  \tilde{c}_{k+2,k+1} \\
 a^{+,-}_{k+2,k+1} &  b^D_{k+2,k+1} \end{array}$ &  $\begin{array}{c} \partial \tilde{c}_{k+2,k+1}= b^D_{k+2,k+1} + \tilde{b}^R_{k+2,k+1} + 1;\end{array}$    \\ & & \\
(14) & $\begin{array}{ccc} \tilde{b}^R_{l-1,l-2} &  \tilde{c}_{l-1,l-2} \\
 a^{+,-}_{l-1,l-2} &  b^D_{l-1,l-2} \end{array}$ &  $\begin{array}{c} \partial \tilde{c}_{l-1,l-2}= b^D_{l-1,l-2} + \tilde{b}^R_{l-1,l-2} + 1.\end{array}$    \\ & & \\
\end{tabular}}
In the squares (5), (6), and (12), $O(k+2,k)$ denotes an element of the ideal generated by those exceptional Reeb chords with upper endpoint on $S_{k+2}$ and lower endpoint on $S_k$.  In square (8), the terms $X, Y, Z$  respectively belong to the $2$-sided ideals generated by 
\[
\begin{array}{cl} E_X = & E\setminus \{ \tilde{c}_{k+2, \{k+1,k\}}, b^D_{k+2, \{k+1,k\}},
 \tilde{c}_{k+1,k}, b^L_{k+1,k} \},   \\
 E_Y = & E_X \cup \{ \tilde{c}_{k+2,k} \} \\
 E_Z = & E_Y \cup \{  b^D_{k+2,k} \}
\end{array}
\]
where $E$ denotes the set of all exceptional generators of the square.  
\end{lemma}

\begin{proof}

The verification for the differentials of the $\tilde{c}$  depends on the square type.

\medskip

{\bf 1.} {\it For squares of type (2)-(4), (7) and (10).}

The sheets $S_{k+1}$ and $S_k$ that form the upper and lower endpoint of $\tilde{c}_{k+1,k}$ never have another sheet between them in the region where $F_{k+1,k} > 0$.  Therefore, no interval vertices are possible in a GFT that begins at $\tilde{c}_{k+1,k}$.    
A flow tree without internal vertices is rigid if and only if it is simply a flow line from $\tilde{c}_{k+1,k}$ to one of the $b_{k+1,k}$ or $\tilde{b}_{k+1,k}$  Reeb chords (by Proposition \ref{prop:ACE112}).  There is a unique such flow line for each $b_{k+1,k}$ and $\tilde{b}_{k+1,k}$ in $N(e^2_\alpha)$.  [In the terminology from \ref{sssec:blines}, these flow lines are the $b_{k+1,k}$- and $\tilde{b}_{k+1,k}$-lines.  That they indeed limit to $\tilde{c}_{k+1,k}$ as $t \rightarrow -\infty$ was established in Lemma \ref{lem:blimits}.]   
% [There are only two $F_{k+1,k}$ trajectories that limit to $b_{k+1,k}$; one of these lies in a different square.  The trajectory that does lie in the square must have $\tilde{c}_{k+1,k}$ as its other endpoint since $\tilde{c}_{k+1,k}$ the unique local maximum of $F_{k+1,k}$ in $R$.]  

\medskip

{\bf 2.} {\it For squares of type (5), (6), and (12).}

The formulas for $\partial \tilde{c}_{k+2,k}$ are derived as in 1., as sheets $S_{k+2}$ and $S_k$ are adjacent in the region where $F_{k+2,k} >0$.   
All terms besides $O(k+2,k)$ in $\partial \tilde{c}_{k+2,k+1}$ correspond to GFTs without internal vertices, and these correspond to the unique flow lines from $\partial \tilde{c}_{k+2,k+1}$ to each of the $b_{k+2,k+1}$ and $\tilde{b}_{k+2,k+1}$ Reeb chords in $N(e^2_\alpha)$.   
%The term $O(k+2,k)$ denotes an element of the ideal generated by those exceptional Reeb chords with upper endpoint on $S_{k+2}$ and lower endpoint on $S_k$.  
We claim that any other rigid GFT starting at $\tilde{c}_{k+2,k+1}$ must have an endpoint at a Reeb chords with upper sheet $S_{k+2}$ and lower sheet $S_k$, and thus is accounted for in the $O(k+2,k)$ term.  [Indeed, for a GFT starting at $\tilde{c}_{k+2,k+1}$, the first internal vertex must be a $Y_0$ (by Proposition \ref{prop:NoY1NoSW}) where the initial $(k+2,k+1)$-flow line splits into a $(k+2,k)$-flow line and a $(k,k+1)$-flow line (because $S_k$ is the only sheet that is ever between $S_{k+2}$ and $S_{k+1}$).  The $(k+2,k)$-branch has no possibility for further $Y_0$-vertices below it, and hence terminates at one of the $(k+2,k)$ Reeb chords.]

\medskip

{\bf 3.} {\it For the Type (8) square.}

%  Here, the terms $X, Y, Z$  respectively belong to the $2$-sided ideals generated by 
%\[
%\begin{array}{cl} E_X = & E\setminus \{ \tilde{c}_{k+2, \{k+1,k\}}, b^D_{k+2, \{k+1,k\}},
% \tilde{c}_{k+1,k}, b^L_{k+1,k} \},   \\
% E_Y = & E_X \cup \{ \tilde{c}_{k+2,k} \} \\
% E_Z = & E_Y \cup \{  b^D_{k+2,k} \}
%\end{array}
%\]
%where $E$ denotes the set of all exceptional generators of the square.  

First observe that the rigid GFTs without internal vertices that start at the $\tilde{c}$ account for the given term or belong to the appropriate ideal.  Next, assume that we have a rigid GFT $\Gamma$ starting at one of the $\tilde{c}$ with at least one $Y_0$-vertex.  We claim that we can find a path $\gamma$ within the domain of $\Gamma$ from the input at $\tilde{c}$ to an output Reeb chord such that all edges that appear along $\gamma$ are flow lines for some $F_{i,j}$ with $i >j$.  [To see this, start at $\tilde{c}$ and work down. When a $Y_0$-vertex is reached, if the edge oriented into the vertex is a flowline for $F_{i,j}$ with $i>j$, then the outputs must be flow lines for $F_{i,h}$ and $F_{h,j}$ for some $h$.  If it were the case that $i<h$ and $h<j$ then $i <j$ would hold also, so at least one of the output edges satisfies the desired inequality. Note that here, as usual, we consider punctures at $c$'s as being special cases of $Y_0$-vertices where one outgoing edge is a constant flowline.]

All Reeb chords for which the subscripts $i,j$ satisfy $i >j$ belong to the collection $E$.  Therefore, the endpoint of $\gamma$ must belong to $E$, and we just need to show that it belongs to the appropriate  one of $E_X, E_Y,$ or $E_Z$.  This follows from the following sequence of statements:

\begin{itemize}
\item The endpoint of $\gamma$ cannot be a puncture at $\tilde{c}_{k+1,k}$ or $\tilde{c}_{k+2,k+1}$.  This is  because of the ordering of sheets at these points: there is no way the branch of the path above the puncture could be a flow for $F_{i,j}$ with $i>j$.  See Figure  \ref{fig:TripPointOrder}.  

\item The endpoint of $\gamma$ cannot be at $b^D_{k+2, k+1}$ or $b^L_{k+1,k}$.   
In the first case, the branch of the flow tree immediately above such an endpoint would be a portion of the $b^D_{k+2,k+1}$-line from $b^D_{k+2,k+1}$ to $\tilde{c}_{k+2,k+1}$,  and this flow line is contained entirely in the corner $C_{R,D}$.  [The notation is as in Property \ref{pr:monotonicityI}.  By Property \ref{pr:monotonicityI}, the $b^D_{k+2,k+1}$-line is contained in a vertical strip of width $2 \e$ and centered at $x_1= \beta^D_{k+2,k+1}$ that lies below $x_2= \tilde{\beta}^R_{k+2,k+1}+\e$.]  In $C_{R,D}$, the $z$-coordinates of sheets satisfy $S_{k+2} > S_{k} > S_{k+1}$ so again a branch of the path above a $Y_0$-vertex along this flow line would violate the $i>j$ condition.  Similarly, the flow line from $b^L_{k+1,k}$ to $\tilde{c}_{k+1,k}$ belongs entirely to $C_{L,U}$ where $S_{k+1} >S_k > S_{k+2}$, so the $Y_0$-vertex preceding the edge would be impossible.

\item If the endpoint of $\gamma$ is a puncture at $\tilde{c}_{k+2,k}$, then the initial vertex of $\Gamma$ cannot also be $\tilde{c}_{k+2,k}$.  This is because the difference between $z$-coordinates of sheets decreases along all edges of a GFT.% the initial vertex cannot also be at $\tilde{c}_{k+2,k}$.

\item If the endpoint of $\gamma$ is at $b^D_{k+2,k}$, then the initial vertex would have to be at $\tilde{c}_{k+2,k+1}$.  To verify, note that the point $x$ on the $b^D_{k+2,k}$-line from $b^D_{k+2,k}$ to $\tilde{c}_{k+2,k}$ where the $Y_0$ above $b^D_{k+2,k}$ occurs must belong to $T = \{(x_1,x_2) \in \tilde{N}(e^2_\alpha) \, |  \, x_1 \geq 1/4, \,\,\mbox{and  } x_2 \leq -1/4\}$. [Since the entire $b^D_{k+2,k}$-line  remains in this region by Property \ref{pr:monotonicityI}.]    Thus, the edge in the path $\gamma$ that precedes this edge is a $(k+2,k+1)$-flow line.  Again, using Property \ref{pr:monotonicityI}, this flow line from $x$ to $\tilde{c}_{k+2,k+1}$ must remain in $T$ as $t$ decreases.   Above $T$, we have $S_{k+2} > S_{k} > S_{k+1}$, so another $Y_0$ is impossible since the edge above the $Y_0$ could not be an $(i,j)$-flow with $i>j$.  Thus, this  edge must limit to $\tilde{c}_{k+2,k+1}$, so the initial vertex of the tree is indeed $\tilde{c}_{k+2,k+1}$ as desired.

\end{itemize}

{\bf 4.} {\it For squares of type (13) and (14).}

The formula for $\partial \tilde{c}_{k+2,k+1}$ follows from Lemmas \ref{lem:PFTrees} and \ref{lem:aceyswST}.
%entire flowline $k+2,k+1$ flow from $x$ to $\tilde{c}_{k+2,k+1}$ is contained in the $1/6$-th of the square where $S_{k+2} > S_{k} > S_{k+1}$.}  Using the $i>j$ condition, the first $Y_0$-vertex above $b^D_{k+2,k}$ would occur at such an $x$ and the incoming flow to $x$ must be the $k+2,k+1$-flowline from $x$ to $\tilde{c}_{k+2,k+1}$.  The $i>j$ condition then prohibit any more $Y_0$-vertices from occuring before we reach  $\tilde{c}_{k+2,k+1}$.

\end{proof}

%\subsubsection{Cancelling pairs of exceptional generators in the $2$-skeleton}

\begin{proof}[Proof of Theorem \ref{thm:Exceptional}]

%The formulas for differentials found in \ref{sssec:ExGen1} and \ref{sssec:ExGen2}  allow us to see that within each closed $2$-cell, we can cancel all exceptional generators by a sequence of applications of Theorem \ref{thm:Alg}.  To see this,  
For each square type the exceptional generators are divided into pairs according to the way they are listed in columns of $2$ in the table from Lemma \ref{lem:ExGen2}.  We order these pairs of generators for squares of type (2)-(4), (7), and (10) as they appear from left to right.  For squares of type (5), (6), (8), and (12) the columns with subscripts of the form $k+2,\{k+1,k\}$ represent two distinct pairs of generators.  In this case, order pairs of generators by proceeding from left to right using $k+2,k$ subscripts everywhere, and then listing the remaining $k+2,k+1$ generators as they appear from left to right.  For example, in the neighborhood of a Type (6) $2$-cell there are $6$ pairs of exceptional generators ordered as:
\[
\begin{array}{cccccc}   \tilde{b}^R_{k+2,k}  &  \tilde{b}^L_{k+2,k+1} & \tilde{b}^D_{k,k+2} & \tilde{c}_{k+2,k} & \tilde{b}^R_{k+2,k+1} & \tilde{c}_{k+2,k+1} \\
a^{+,-}_{k+2,k} & a^{-,-}_{k+2,k+2} & a^{-,-}_{k,k+2} & b^D_{k+2,k} & a^{+,-}_{k+2,k+1} & b^D_{k+2,k+1} 
\end{array}
\]

\begin{lemma} \label{lem:local} Let $x$ and $y$ denote a pair of exceptional generators in $N(e^2_\alpha)$ with $x$ appearing above $y$ in the table from Lemma \ref{lem:ExGen2}.  Then, in a quotient of $(\A_{\mathit{LCH}}, \partial)$ in which all of the previous pairs of exceptional generators of the square (with respect to the ordering just described) have been set to zero we have $\partial x = y$, unless either $e^2_\alpha$ has type (13) and $x= \tilde{c}_{k+2,k+1}$ or $e^2_\alpha$ has type (14) $x= \tilde{c}_{l-1,l-2}$.  In these latter two cases, we have
\[
\partial \tilde{c}_{k+2,k+1} = b^D_{k+2,k+1} +1;  \mbox{    and    } \partial \tilde{c}_{l-1,l-2} = b^D_{l-1,l-2} +1. 
\]
\end{lemma}
\begin{proof}
This is easily verified from the formulas for differentials established in Lemmas \ref{lem:Exc1} and \ref{lem:ExGen2}.  
\end{proof}

%This is possible as one easily verifies %(using the formulas from \ref{sssec:ExGen1} for the $\tilde{b}$) 
%that, as long as all previous exceptional generators from the list have been cancelled, the differential simply takes the top entry of each column to the bottom entry.  For squares of type (5), (6), (8), and (12) a similar process works except that the subscripts of the form $k+2,\{k+1,k\}$ represent two distinct generators.  For these terms, first cancel all columns using the $k+2,k$ subscripts and then repeat to cancel the generators with $k+2,k+1$ subscripts. 

Using Lemma \ref{lem:local}, we can apply  
%within cancel all of the exceptional generators within a given $2$-cell by repeat applications of 
Theorem \thmAlg from \cite{RuSu1}
%\ref{thm:Alg} 
multiple times to see that the quotient of $(\lchA, \partial)$ by the ideal generated by those generators of $I$ that lie in a single $N(e^2_\alpha)$ is stable tame isomorphic to $(\lchA,\partial)$.  [Apply Theorem \thmAlg from \cite{RuSu1} to cancel the exceptional generators of $N(e^2_\alpha)$ one  pair at a time, in the order previously specified.  
%Each application of Theorem \ref{thm:Alg} adds $x$ and $\partial x$ to  and $y$ from the table in Proposition \ref{prop:ExGen2} one at a time.  If we cancel pairs in the order that was previously described, then 
Lemma \ref{lem:local} shows that the result of this procedure is indeed $\lchA$ quotiented by the ideal generated by those generators of $I$ that lie in $N(e^2_\alpha)$.]   However, to see that the quotient of $(\lchA,\partial)$ by all of $I$ is stable tame isomorphic to $\lchA$ we need to be slightly more careful about the order in which we cancel exceptional generators in distinct squares for the following reason:   %the proof of Theorem \ref{thm:Exceptional} is slightly more involved for the following reason.  
After cancelling all exceptional generators in a given square, one could worry that when considering a neighboring square some of the generators from the $1$-skeleton that are necessary to carry out the cancellation process may be missing.  Conditions (A2)-(A4) in the construction of $\mathcal{E}_\pitchfork$ are designed to prevent this from happening.

The procedure to globally cancel all exceptional generators using repeated applications of Theorem \thmAlg from \cite{RuSu1} is as follows:

\medskip

\noindent {\bf Step 1.}  Cancel all exceptional generators in squares of type (3).  

\medskip

The property (A3) of Proposition \ref{prop:Etra} states that no two Type (3) squares of $\mathcal{E}_\pitchfork$ can border one another.  Therefore, there is no overlap among the generators from different squares that should be cancelled in Step 1.

\medskip

\noindent {\bf Step 2.}  For all (1Cr) $1$-cells $e^1_\alpha$ that do not border a Type (3) square, cancel $\tilde{b}_{k+1,k}$ with $a^-_{k+1,k}$; for all (2Cr) $1$-cells $e^1_\alpha$ cancel  $\tilde{b}_{k+2,k}$ with $a^-_{k+2,k}$ and then cancel $\tilde{b}_{k+2,k+1}$ with $a^-_{k+2,k+1}$.
(Note that Lemma \ref{lem:Exc1} shows that in the (2Cr) case, once $\tilde{b}_{k+2,k}$ and $a^-_{k+2,k}$ are equal to $0$, $\partial \tilde{b}_{k+2,k+1} = a^-_{k+2,k+1}$ so that Theorem \thmAlg from \cite{RuSu1} applies.)

\medskip

For carrying out Step 2, we need to choose some ordering of the (1Cr) and (2Cr) $1$-cells before proceeding.  This choice of order is irrelevant.  However, since a single  $0$-cell $e^0_\alpha$ can appear as the initial vertex of more than one $1$-cell, it is important to verify the cancelling procedure does not specify that a single Reeb chord $a_{i,j} \in N(e^0_\alpha)$ should be cancelled more than once.  This holds since the property (A2) of $\mathcal{E}_\pitchfork$ (from Proposition \ref{prop:Etra})  shows that the only way this can happen is when $e^0_\alpha$ is at the lower left corner of a Type (3) square, $e^2_\beta$, where sheets $k$ and $k+1$ cross, and the Reeb chord is $(a_{k,k+1}, e^0_\alpha) = (a^{-,-}_{k+1,k},e^2_\beta)$.  However, these Reeb chords and the $\tilde{b}$ Reeb chords that lie on the left and down edges of type (3) $2$-cells have already been cancelled in Step 1, and cancelling them is not part of the Step 2 procedure.  

%the only Reeb chords $(a_{i,j}, e^0_\alpha)$ that could equal $(a^-_{k+1,k}, e^1_\beta)$, $(a^-_{k+2,k}, e^1_\beta)$, or $(a^-_{k+2,k+1}, e^1_\beta)$ for more than one $1$-cell of Type (1Cr) or (2Cr) 

%satisfies (A4) from Proposition \ref{} that states that: 

%For any $0$-cell, $e^0_\alpha$, and pair of sheets $S_{i}$, $S_{j}$ above $e^0_\alpha$  consider the set of $1$-cells, 
%\[
%T(e^0_\alpha, S_i,S_j) =\{e^1_\beta \, | \,  \mbox{$e^0_\alpha$ is the initial vertex of $e^1_\beta$; and $S_i$ and $S_j$ intersect above $e^1_\beta$}\}.
%\]
%The cardinality of $T(e^0_\alpha, S_i,S_j)$ satisfies
%\[
%|T(e^0_\alpha, S_i,S_j)| \leq 2,
%\]
%and if $|T(e^0_\alpha, S_i,S_j)| = 2$ then the two $1$-cells in $T(e^0_\alpha, S_i,S_j)$ form the bottom and left edge of a single Type (3) square that has $e^0_\alpha$ as its lower left vertex.

%Those $\tilde{b}_{k+1,k}$ and $a^-_{k+1,k}$ Reeb chords that lie within neighborhoods of Type (3) $2$-cells were alreadry cancelled in Step 1.

%this set is $\leq 2$ with equality only when the two edges 

%, whose initial point is at $e^0_\alpha$.  The only time that the same two sheets above $e^0_\alpha$ cross above more than one of the $\{e^1_\beta\}$ is

\medskip

\noindent {\bf Step 3.}  Choose some ordering of the non-Type (3) squares and cancel all remaining exceptional generators one $2$-cell at a time.  

\medskip

At the conclusion of Step 2., all exceptional generators of the form $\tilde{b}$ together with the $a$ generators that appear immediately below them in Lemma \ref{lem:ExGen2} have already been cancelled.  [Note that in all cases 
%besides 
%the Type (3) square, 
the generator $y$ that appears below a $\tilde{b}$ Reeb chord in Lemma \ref{lem:ExGen2} is precisely the $a$ Reeb chord between the same two sheets and above the initial vertex of the edge of $\tilde{b}$.  These are precisely the Reeb chords that the $\tilde{b}$ were cancelled with in Step 2.]  Therefore, we can cancel the exceptional generators in a given $2$-cell by applying Theorem \thmAlg from \cite{RuSu1} repeatedly to the pairs of exceptional generators in $N(e^2_\alpha)$ in the order described above, and skipping those (already cancelled) pairs that begin with a $\tilde{b}$ generator.  

During Step 3., the only Reeb chords to be cancelled that belong to neighborhoods of more than one $2$-cell are Reeb chords of the form $b^X_{i,j}$ (without a tilde).  Such a Reeb chord is exceptional in $N(e^2_\alpha)$ if and only if a segment of the $(i,j)$-crossing locus of $e^2_\alpha$ is homotoped to sit above the boundary $1$-cell $e^X_1$ during the procedure defined in Section \ref{sec:parallel}.  Property (A4) from Proposition \ref{prop:Etra} states that no crossing arcs from distinct $2$-cells are homotoped to the same $1$-cell.  Thus, there is no overlap among Reeb chords in distinct $2$-cells that need to be cancelled at Step 3.

At this point, we have obtained $(\lchA/I, \partial)$ from $(\lchA,\partial)$ by repeated applications of Theorem \thmAlg from \cite{RuSu1}, so it follows that these DGAs are stable tame isomorphic.
\end{proof}

%\subsection{Global cancellation of exceptional generators}

%The formulas for differentials found in \ref{sec} and \ref{sec}  allow us to see thatWithin each closed $2$-cell, we can cancel all exceptional generators by a sequence of applications of Theorem \ref{thm:Alg}.  

%Computing the differential in the quotient.
%$0$-cells
%$1$-cells
%$2$-cells

\subsection{Isomorphism between $(\lchA/I, \partial)$ and the Cellular DGA}  

For easier comparison with the cellular DGA, we record here the presentation of $(\lchA/I, \partial)$ determined by the computations from Sections  \ref{sec:CompLCH}-\ref{sec:SwallowComp}.  Recall the procedure introduced in Section \ref{sec:parallel} for homotoping the cusp and crossing loci of $\tilde{L}$ into the $1$-skeleton of $\mathcal{E}_\tra$.

%Since we want to define a global isomorphism, we now provide a fixed notation for the Reeb chords of $\tilde{L}$.   

%\begin{convention}  In the remainder of this section, we use the fixed notation for Reeb chords
%\[
%a^\alpha_{i,j} := (a_{i,j}; e^0_\alpha);
%\]
%\[
%b^\alpha_{i,j} := (b_{i,j}; e^1_\alpha);
%\]
%\[
%c^\alpha_{i,j} := (c_{i,j}; e^2_\alpha).
%\]
%That is, for a Reeb chord that is a local minimum, saddle point, or local maximum we provide subscripts by using the ordering of sheets $S_1, \ldots, S_n$ with descending $z$-coordinates as they appear above the unique $0$-cell, $1$-cell, or $2$-cell, respectively, whose neighborhood contains the Reeb chord.  (Above $1$-cells or $2$-cells we require that $z$-coordinates descend above $x=+1$ or $(x_1,x_2) = (+1,+1)$ respectively.  Since only the notation of the $c_{i,j}$ Reeb chords is based on the labeling of sheets above a Type (13) or (14) square we do not have to consider the sheet that was earlier denoted $\tilde{S}_k$.]
%\end{convention}

%In the following description of the DGA $(\lchA/I, \partial)$, 
%Recall the procedure introduced in Section \ref{sec:parallel} for homotoping the cusp and crossing loci of $\tilde{L}$ into the $1$-skeleton of $\mathcal{E}_\tra$.  

\begin{proposition} \label{prop:lchAI} The DGA $(\lchA/I, \partial)$ has the following generators.  

%Totally order the sheets above each $d$-cell of $\epsilon_\tra$ $S_1(e^d_\alpha), \ldots, S_n(e^d_\alpha)$...
For each $0$-cell; $1$-cell; or $2$-cell, $e^d_\alpha$, of $\mathcal{E}_\tra$ we respectively have generators
\[
a^\alpha_{i,j},   \quad b^\alpha_{i,j}, \quad \mbox{or  } c^\alpha_{i,j}
\]
for each $i,j$ such that 
\begin{itemize}
\item $1 \leq i < j \leq n$ where $n$ is the number of sheets above $e^d_{\alpha}$, and 
\item the crossing locus between sheets $S_{i}$ and $S_{j}$ (as notated above $e^d_\alpha$) is not placed above (the interior of) $e^\alpha_d$ when it is homotoped to sit above the $1$-skeleton as in Section \ref{sec:parallel}.
\end{itemize}
Supposing $L$ is equipped with an $\Z/m(L)$-valued Maslov potential, $\mu$, the grading of generators is
\begin{equation} \label{eq:grading4}
|a^\alpha_{i,j}| = \mu(S_i)- \mu(S_j)-1;  \quad |b^\alpha_{i,j}| = \mu(S_i)-\mu(S_j); \quad |c^\alpha_{i,j}| = \mu(S_i)-\mu(S_j)+1.
\end{equation}

Moreover, with the exception of Type (13) and (14) squares, and some Type (9) squares, the differentials are computed using the process defined in Section 3.6 of \cite{RuSu1}:

\begin{itemize}
\item For a $0$-cell $e^0_\alpha$, form the strictly upper triangular matrix with $(i,j)$-entry $a^\alpha_{i,j}$, if it exists, and $0$ otherwise.  We have
\begin{equation} \label{eq:AdefProp}
\partial A = A^2.
\end{equation}

\item Suppose there are $n$ sheets above the $1$-cell $e^1_\alpha$ with $0$-cells $e^0_\gamma$ and $e^0_\beta$ at $x=-1$ and $x =+1$.  Form $n \times n$ matrices $B$, $A_-$, $A_+$ as in Section 3.6 of \cite{RuSu1}.  That is,  the columns and rows of $A_-$ and $A_+$ in which the $a^\gamma_{ij}$ and $a^\beta_{ij}$ appear are determined by identifying the sheets above $e^0_\beta$ and $e^0_\gamma$ with a subset of the sheets above $e^1_\alpha$.  All remaining entries are $0$ except when sheets $k$ and $k+1$ of $e^1_\alpha$ meet in a cusp in which case a $1$ is placed in the $(k,k+1)$ entry of $A_-$.

%The $(i,j)$-entries of $B, A_-$ and $A_+$ are respectively  $b^\alpha_{i,j}$, $a^\gamma_{\sigma(i),\sigma(j)}$, $a^\beta_{i,j}$ when these generators exist, where $\sigma(i)$ is such that the sheet labelled $S_{\sigma(i)}$ above $e^0_\gamma$ belongs to the closure of the sheet labelled $S_i$ above $e^1_\alpha$.  All remaining entries are $0$ except for the in the case when $e^1_\alpha$ has Type (Cu) with sheets $S_k$ and $S_{k+1}$ meeting at a cusp point.  In this case, the $(k,k+1)$-entry of $A_-$ is $1$ while all other entries are $0$.

We have
\begin{equation} \label{eq:BdefProp}
\partial B = A_+ (I+B) + (I+B) A_-.
\end{equation}

\item  Consider a $2$-cell $e^2_\alpha$ with $L, D, R, U$ edges given by the $1$-cells $e^1_{\beta_L}, e^1_{\beta_D}, e^1_{\beta_R}, e^1_{\beta_U}$ and $0$-cells $e^0_{\gamma_0}$ and $e^0_{\gamma_1}$ at $(x_1,x_2) = (-1,-1)$ and $(x_1,x_2) = (+1,+1)$.  Define upper triangular matrices $C$, $B_L, B_D, B_R, B_U$, $A_{-,-}$ and $A_{+,+}$ whose entries are respectively the Reeb chords $c^\alpha_{i,j}$, $b^{\beta_L}_{i,j}$, $b^{\beta_D}_{i,j}$, $b^{\beta_R}_{i,j}$, $b^{\beta_U}_{i,j}$, $a^{\gamma_0}_{i,j}$, and $a^{\gamma_1}_{i,j}$.  The placement of the $b_{i,j}$ and $a_{i,j}$ generators in rows and columns of the $B_X$ and $A_{\pm,\pm}$ is determined as in Section 3.6 of \cite{RuSu1} by identifying sheets above boundary cells with the sheets of $e^2_\alpha$ whose closure they belong to.  Remaining entries are $0$ except that the $(k,k+1)$ entry of $A_{-,-}$ is $1$ whenever sheets $k$ and $k+1$ of $e^2_\alpha$ meet at a cusp above $e^2_\alpha$.

We have
\begin{equation}  \label{eq:LCHdiffCcomp}
\partial C = A_{+,+} C + C A_{-,-} + (I+ B_U)(I+B_L) + (I+B_R)(I+B_D).
\end{equation}
\end{itemize}

The above holds for a square of type (9), $e^2_\alpha$, as well, with the following exception.  Suppose that after the singular set $\pi_x(\Sigma)$ is homotoped into the $1$-skeleton, a swallowtail point sits above the initial point of the edge $e^1_{\beta_X}$ and the crossing arc that ends at the swallowtail point sits above $e^1_{\beta_X}$ (as in Figure \ref{fig:IsoProof2} below).  Then, (\ref{eq:LCHdiffCcomp}) holds provided that we take the $(i,j)$-entry  of $B_X$ to be $1$ if sheets $S_{i}$ and $S_j$ (as labeled above $e^2_\alpha$) cross above $e^1_{\beta_X}$.

For squares of type (13) or (14) recall that we have decomposed those sheets above $e^2_\alpha$ that contain the swallowtail point in their closure into subsets $S_k, S_{k+1}, S_{k+2}$, and $\tilde{S}_k$ (resp. $S_l, S_{l-1}, S_{l-2}$, and $\tilde{S}_l$).  Form matrices $C, B_X$ and $A_{\pm,\pm}$  by making the convention that generators $b^L_{i,j}$ whose upper (resp. lower) sheet is a portion of $\tilde{S}_k$ (resp. $\tilde{S}_l$) are placed in row  $k+2$ (resp. column $l-2$) of $B_L$ while  generators $a^{-,-}_{i,j}$ whose upper (resp. lower) sheet is a portion of $\tilde{S}_k$ (resp. $\tilde{S}_l$) are placed in row $k+1$ (resp. column $l-1$) of $A_{-,-}$.  All remaining entries are $0$ except for the $(k,k+2)$-entry  (resp. column $(l-2,l)$-entry) of $A_{-,-}$ which is $1$.  

We have
\[
\begin{array}{rl}
\partial C = & A_{+,+} C + C (I+E_{k+2,k+1}) A_{-,-} (I+E_{k+2,k+1}) + \\ & (I+B_U)(I+B_L)(I+ A_{-,-}E_{k+1,k}+ E_{k+1,k+2}) + (I+B_R)(I+B_D + B_D E_{k+2,k+1}). 
\end{array}
\]

\end{proposition}

\begin{proof}
By construction, the generating set of $(\lchA/I, \partial)$ is obtained from the generating set of $(\lchA, \partial)$ by removing all exceptional Reeb chords.  
%That is, we remove all Reeb chords of the form $\tilde{b}_{i,j}$ and $\tilde{c}_{i,j}$ as well as all Reeb chords above $0$-cells (resp. $1$-cells)  for which the sheets that form the lower and upper endpoint of the Reeb chord have their crossing locus homotoped to sit above the (entire) $0$-cell (resp. $1$-cell) when the crossing locus is homotoped as in  Section \ref{sec:parallel}.  (See Definition \ref{def:ex-gen}.) This leaves precisely those $a^\alpha_{i,j}$, $b^\alpha_{i,j}$, and $c^\alpha_{i,j}$ as in the statement.  
That the remaining  generators are as stated follows from the definition of exceptional generator (in Definition \ref{def:ex-gen}).  Note that in the statement of the Proposition, when providing subscripts for generators, we used the ordering of sheets above the unique $0$-cell, $e^0_\alpha$, containing $a^\alpha_{i,j}$, above the unique $1$-cell, $e^1_\alpha$, containing $b^\alpha_{i,j}$, and above the unique $2$-cell, $e^2_\alpha$, containing $c^\alpha_{i,j}$.  That is, 
\[
a^\alpha_{i,j}:= (a^\alpha_{i,j},e^0_\alpha);  \quad b^\alpha_{i,j} := (b^\alpha_{i,j}, e^1_\alpha); \quad \mbox{and} \quad c^\alpha_{i,j} := (c^\alpha_{i,j}, e^2_\alpha).
\] 
The mod $m(L)$ grading of generators is computed as in equation (\ref{eq:gradingdef}).

The formulas for differentials follow from Propositions \ref{prop:LCH0comp}, \ref{prop:LCH1comp}, Theorem \ref{thm:SquareComp}, and Theorem \ref{thm:SwallowComp}.

In elaborating, we first discuss cases other than Type (13) or Type (14) squares.  
%In comparing, the matrix formulas from the statement of the current Proposition \ref{prop:lchAI} with 
In the computations from Sections \ref{sec:CompLCH}-\ref{sec:SwallowComp}, when computing the sub-DGA generated by Reeb chords above $N(e^d_\alpha)$, subscripts were determined using the labelling of sheets above $e^d_\alpha$ (at the terminal vertex if $d=1$ and above the upper-right corner if $d=2$).  Moreover, when it is a Reeb chord, the $(i,j)$-entry for any of the ``boundary matrices'' $B_X$, $A_\pm$  or $A_{\pm,\pm}$ has subscripts $i$ and $j$.  That is, the Reeb chords in the boundary matrices are placed so that if the upper and lower sheets of the Reeb chord above the  boundary cell  belong to the closure of sheet $i$ and $j$ as labeled above $e^d_\alpha$, then that Reeb chord appears in row $i$ and column $j$.  This is precisely the way we place generators in the boundary matrices in the statement of the current Proposition \ref{prop:lchAI}.  Moreover, the appearance of $0$'s and $1$'s in these matrices is identical as well.  Indeed, by the definition of $I$, with a single exception in each Type (13) and (14) squares all exceptional generators equal $0$ in $\lchA/I$.  The single exceptional generator that is set to $1$ in Type (13) and (14) squares is reflected by the $1$ placed in the $(i,j)$ entry of $B_X$ whenever $e^1_{\beta_X}$ is an edge of a Type (9) square that sits under a  crossing arc between sheets $S_i$ and $S_j$ with swallowtail endpoint.  [This Reeb chord $b^X_{i,j}$ is equal to the Reeb chord notated $b^{D}_{k+2,k+1}$ or $b^D_{l-1,l-2}$ in the neighboring Type (13) or Type (14) square.]  Moreover, the extra terms, $x$ and $X$, that appear in Proposition \ref{prop:LCH1comp} and Theorem \ref{thm:SquareComp} vanish in $\lchA/I$, so that the formulas from the statement of Proposition \ref{prop:lchAI} agree precisely with the corresponding formulas from Propositions \ref{prop:LCH0comp}, \ref{prop:LCH1comp}, and Theorem \ref{thm:SquareComp}.  

Finally, we consider the differentials in a Type (13) square, $e^2_\alpha$, as similar considerations apply to the Type (14) squares.  The $X$ term in Theorem \ref{thm:SwallowComp} belongs to the ideal generated by exceptional generators of $e^2_\alpha$ other than $b^D_{k+2,k+1}$, and hence is $0$ in $\lchA/I$.  Moreover, the matrices $B_X$ and $A_{\pm,\pm}$ agree precisely with the matrices from Theorem \ref{thm:SwallowComp} since the manner in which generators are placed in the statement of Proposition \ref{prop:lchAI} is consistent with placement of generators in the matrices of Theorem \ref{thm:SwallowComp}.  In particular, the placement of generators whose upper or lower sheet is $\tilde{S}_k$ is consistent with Convention \ref{c:Sktilde}.

%if we place the Reeb chords above a boundary  

%By the definition of $I$, all exceptional generators are set to

% is a consequence of the way we labelled Reeb chords above $N(e^1_\alpha)$ in Section ...  There we used the labelling of sheets above the $1$-cell in the entire neighborhood, so that those sheets of $e^\pm_0$ were given the same labeling as the sheets that they correspond to under the injection used in defining the cellular DGA.  This is why the entries of matrix has subscript $(i,j)$ when it is non-zero.

%Perhaps it is useful to have a formula like
%$(a^{\pm,\pm}_{i,j}, e^2_\alpha) = a^\alpha_{\iota(i), \iota(j)}$
%in the absolute notation introduced here.
\end{proof}

\subsubsection{The cellular DGA $(\mathcal{A}_{||}, \partial)$}
The presentation of $(\lchA/I,\partial)$ from Proposition \ref{prop:lchAI} looks quite similar to the definition of the cellular DGA.  To realize an isomorphism between $(\lchA/I,\partial)$ and the cellular DGA of $L$, we make use of the $L$-compatible polygonal decomposition $\mathcal{E}_{||}$.

Recall the construction of $\mathcal{E}_{||}$ from Section \ref{sec:mathcalE} as well as its relation to $\mathcal{E}_\pitchfork$ as stated in Proposition \ref{prop:EparProp}.  The differential of the cellular DGA (as defined in Section 3 of \cite{RuSu1}) associated to $\mathcal{E}_{||}$ relies on a few choices which we make as follows:

\begin{itemize}
\item  Orient $1$-cells to agree with the orientation of edges of corresponding squares in $\mathcal{E}_\tra$.  There are also some extra $1$-cells added to squares of type (5), (6), (8),  and (12) that do not correspond to subsets of edges of squares in $\mathcal{E}_\tra$.  We orient these $1$-cells from left to right when viewed using the parametrizations of squares from $\mathcal{E}_\tra$.  
\item  For each $2$-cell of $\mathcal{E}_{\parallel}$ that is also a square of $\mathcal{E}_\tra$, choose the initial and terminal vertices $v_0$ and $v_1$ for the $2$-cell to be the lower left and upper right vertices of the square respectively.  Each Type (5), (6), (8), and (12) square of $\mathcal{E}_\tra$ contains a pair of $2$-cells of $\mathcal{E}_{\parallel}$.  (See Section \ref{sec:mathcalE} and Figure \ref{fig:IsoProof}.) For each of these $2$-cells use the lower left (or just the left vertex if the $2$-cell is a triangle) and upper right vertex for $v_0$ and $v_1$ respectively.
\item  For each swallowtail point, assign the decorations $S$ and $T$ so that the $S$ corresponds to a corner of a Type (13) or (14) square, and the $T$ is a corner of the Type (9) square that has the bottom edge of the Type (13) or (14) square as a border.  In the Type (9) squares containing $T$ decorations, we shift the initial vertex, $v_0$, so that it lies at the border of the $T$ edge and the left border of the square. (See Figure \ref{fig:IsoProof2} below.)
\end{itemize}

Denote the cellular DGA associated to $\mathcal{E}_{||}$ with the above choices by $(\mathcal{A}_{||}, \partial)$.  

\subsubsection{Outline of equivalence of $(\lchA/I, \partial)$ and $(\mathcal{A}_{||},\partial)$}

In comparing, $(\lchA/I, \partial)$ as presented in Proposition \ref{prop:lchAI} with the cellular DGA $(\mathcal{A}_{||}, \partial)$ we notice the following key differences:
\begin{enumerate}
\item  The cellular DGA has more generators associated to the Type (5), (6), (8),  and (12) squares of $\mathcal{E}_\tra$, since each of these squares is subdivided into two $2$-cells of $\mathcal{E}_{||}$.  In addition, in $(\lchA/I,\partial)$ the bottom edges of the Type (5), (6), (8), and (12) squares have matrices $B_D$ with two above diagonal entries equal to $0$, specifically the $(k+1,k+2)$ and $(k,k+2)$ entries.  This does not occur in any of the $B$ matrices used in defining $(\mathcal{A}_{||}, \partial)$.  
\item  The differentials in $(\mathcal{A}_{||}, \partial)$ of generators associated to the squares of $\mathcal{E}_{||}$ that contain the swallowtail decorations $S$ or $T$ appear somewhat different than the differentials of corresponding generators of $(\lchA/I, \partial)$.
\end{enumerate}

To address (1), we produce a stable tame isomorphic quotient $(\mathcal{A}_{||}/J, \partial)$ whose generators are in precise correspondence with those of $(\lchA/I, \partial)$.  We then give a (tame) DGA isomorphism between $(\mathcal{A}_{||}/J, \partial)$ and $(\lchA/I, \partial)$.  For most squares, the isomorphism simply identifies generators, but, as may be expected from (2), the isomorphism is more involved for squares that contain $S$ or $T$ decorations.

\begin{proposition}  \label{prop:IsoProof}
There exists a stable tame isomorphic quotient of $(\mathcal{A}_{||}, \partial)$ that is tame isomorphic to $(\lchA/I, \partial)$.  
\end{proposition}

\begin{proof}

{\bf Constructing the quotient $\mathcal{A}_{||}/J$.}

All cells of $\mathcal{E}_{||}$ were obtained from cells of $\mathcal{E}_\tra$ by applying a homeomorphism of the base surface $S$ and then subdividing some of the cells of $\mathcal{E}_\tra$.  Let us set notation for the cells that appear during the subdivision process (which is summarized in Proposition \ref{prop:EparProp}).

First, for each $\mathcal{E}_\tra$ square, $e^2_\alpha$, of type (5), (6), (8), and (12), the right edge of the square is subdivided into two $1$-cells by adding a new $0$-cell at the intersection of the (upper) crossing arc with the edge.  Denote the new $0$-cell by $e^0_R$ and the $1$-cells that constitute the upper and lower half of this edge by $e^1_{R,+}$ and $e^1_{R,-}$.  If $e^2_\alpha$ is a Type (6) square then we also subdivide the left edge to produce cells $e^0_L$, $e^1_{L,+}$ and $e^1_{L,-}$.  
(Note that some of these edges are shared with neighboring cells, in a manner that is visible in Figures \ref{fig:EParallel} and \ref{fig:EParallel12}.  We caution that the ``left'' and ``right'' edges that we refer to are with respect to the parametrizations by $[-1,1] \times[-1,1]$, and not with respect to the appearance in Figures \ref{fig:EParallel} and \ref{fig:EParallel12}.)

Secondly, a new $1$-cell is added that connects either $e^0_{-,-}$ to $e^0_R$, in the case of Type (5), (8), and (12) squares, or $e^0_{L}$ to $e^0_R$, in the case of a Type (6) square.  Denote this new cell by $e^1_{C}$, and the $2$-cells that respectively form the upper and lower half of $e^2_\alpha$ as $e^2_{\alpha,+}$ and $e^2_{\alpha,-}$.  See Figure \ref{fig:IsoProof}.

\begin{figure}

\quad

\centerline{   
\labellist
\small
\pinlabel $e^2_{\alpha,+}$  at 524 124
\pinlabel $e^1_C$ [br] at 536 65
\pinlabel $e^2_{\alpha,-}$   at 570 42
\pinlabel $e^1_{R,+}$  [l] at 620 144
\pinlabel $e^0_R$  [l] at 620 96
\pinlabel $e^1_{R,-}$  [l] at 620 38
\endlabellist
\includegraphics[scale=.6]{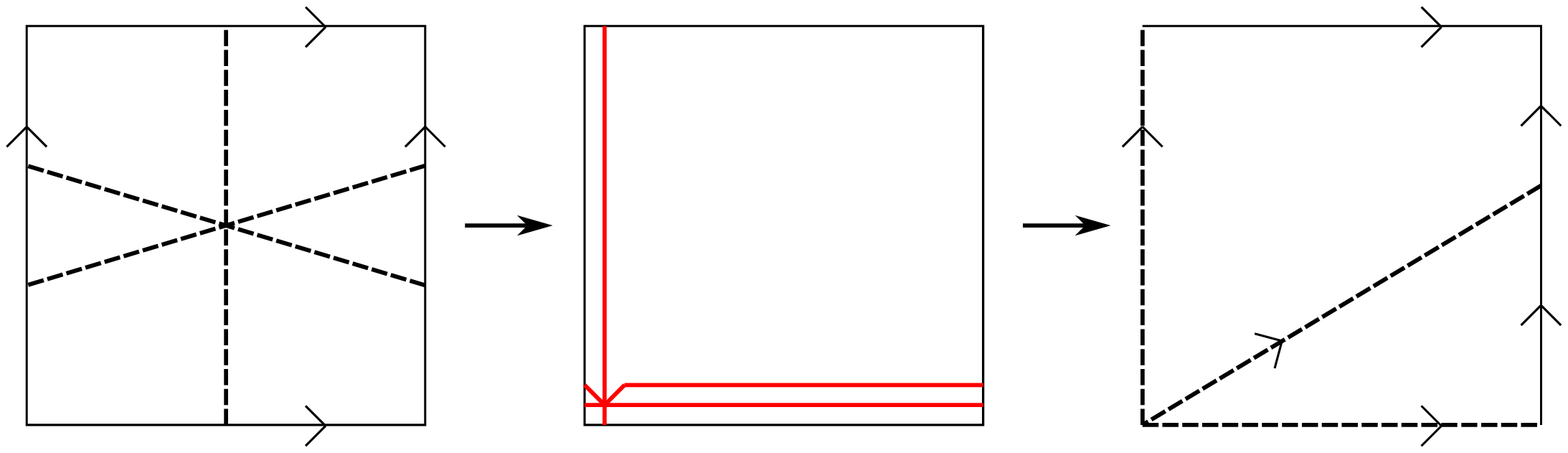}  }
%\centerline{ \includegraphics[scale=.8]{images/SubdivideSq} } %Dan had already commented out this line
\caption{Subdividing a Type (8) cell of $\mathcal{E}_\pitchfork$ to form $\mathcal{E}_{||}$.}
\label{fig:IsoProof}
\end{figure}

To form the quotient $\mathcal{A}_{||}/J$, we will apply Theorem \thmAlg from \cite{RuSu1} to cancel many generators in these squares.  The hypothesis of Theorem \thmAlg from \cite{RuSu1} requires an ordering of the generating set of $\mathcal{A}_{||}$ for which the differential is triangular.  Such orderings for the cellular DGA are discussed in Section 4.1 of \cite{RuSu1}.

\medskip

\noindent {\bf Cancellation 1.} Label sheets of all cells that belong to the closure of $e^2_{\alpha,-}$ in the order they appear above $e^2_{\alpha,-}$.  With generators associated to a cell placed into matrices with the same subscripts as the cell, we have
\[
\partial B_{R,-} = A_R[I+B_{R,-}] + [I+B_{R,-}] A_{+,-}.
\]
Where the $(k+1,k+2)$ entry of $A_R$ is $0$ and the $(k,k+1)$-entry of $A_{+,-}$ is $0$.  [The $k,k+1$, and $k+2$ sheets above $e^1_R$ in $\mathcal{E}_\tra$ appear as in  Figure \ref{fig:EdgeTypesB} (2Cr), so above $e^1_{R,-}$ in $\mathcal{E}_{||}$ they appear as in the region between the two crossings in Figure \ref{fig:EdgeTypesB} (2Cr).]  

First, since 
\[
\partial \framebox{$b^{R,-}_{k+1,k+2}$} = \framebox{$a^{+,-}_{k+1,k+2}$} 
\]
we can quotient by the $2$-sided ideal generated by $b^{R,-}_{k+1,k+2}$ and $\partial b^{R,-}_{k+1,k+2}$ which we notate as $I(b^{R,-}_{k+1,k+2}, \partial b^{R,-}_{k+1,k+2})$.  
Theorem \thmAlg from \cite{RuSu1} tells us that the quotient with the boxed generators removed from the generating set is stable tame isomorphic to $(\mathcal{A}_{||},\partial)$.

Next, quotient by $I(b^{R,-}_{i,j}, \partial b^{R,-}_{i,j})$ inductively for all remaining $i<j$, so that $|i-j|$ is non-increasing during the inductive process.  
When we quotient by $I(b^{R,-}_{i,j},\partial b^{R,-}_{i,j})$, we already have that $b^{R,-}_{i,k} = 0 = b^{R,-}_{k,j}$ for all $i<k<j$, so that
\[
\partial \framebox{$b^{R,-}_{i,j}$}= \framebox{$a^R_{i,j}$} + a^{+,-}_{i,j}
\]
for $(i,j) \neq (k,k+1)$, and 
\[
\partial \framebox{$b^{R,-}_{k,k+1}$}= \framebox{$a^R_{k,k+1}$}. 
\]
Thus, Theorem \thmAlg from \cite{RuSu1} implies that (removing boxed generators from the generating set) the result is stable tame isomorphic to $(\mathcal{A}_{||}, \partial)$.  

After the inductive process is complete, the new generating set contains no generators associated to $e^0_R$ or $e^1_{R,-}$ and has only the generators $a^{+,-}_{i,j}$ such that $S_{i}$ and $S_{j}$ do not cross above $e^1_R$; i.e. with $(i,j) \neq (k+1,k+2)$ and $(i,j) \neq (k,k+1)$.  We have relations
\begin{equation}  \label{eq:someRel1}
a^{+,-}_{i,j} = a^R_{i,j}  \quad \mbox{for $(i,j) \notin \{(k,k+1), (k+1,k+2)\}$};  \quad a^{+,-}_{k+1,k+2}=0;  \quad a^{R}_{k,k+1} = 0.
\end{equation}  

Perform this same process on the generators associated to the generators of $\mathcal{A}_{||}$ associated to $e^0_L$, $e^1_{L,+}$ and $e^1_{L,-}$ cells in Type (6) squares.

%The cell decompositions of closed squares of Types (5), (8), and (12) from $\mathcal{E}_\tra$ are subdivided by placing an extra $0$-cell in the interior of the edge of the square that corresponds to the right edge in Figure \ref{fig:generators} and then  connecting this new $0$-cell to the lower left corner of the square with a new $1$-cell.  

%\item  The cell decompositions of closed squares of Type (6) are subdivided by placing an extra $0$-cell on both the left and right edges and then connecting these new $0$-cells with a new $1$-cell.

\medskip

\noindent {\bf Cancellation 2.}  The generators associated to $e^2_{\alpha,-}$ satisfy
\begin{equation} \label{eq:can2}
\partial C \doteq A_{-,-} C + C A_R + [I+B_D] + [I+ B_C]
\end{equation}
where $\doteq$ denotes equality in the currently considered quotient.
[Note that in $(\mathcal{A}_{||}, \partial)$ there may be other factors of the form $[I+B_{L,-}]$ or $[I+B_{R,-}]$  in the final two terms of $\partial C$. However, $B_{L,-} \doteq B_{R,-} \doteq 0$ after the first cancellation procedure.]

First, cancel 
\[
\partial \framebox{$c^{\alpha,-}_{k+1,k+2}$} = \framebox{$b^{D}_{k+1,k+2}$}. 
\]
Then, inductively cancel the remaining $c^{\alpha,-}_{i,j}$ with the $b^{C}_{i,j}$ in an order so that $|i-j|$ is non-decreasing.  As we proceed, we can verify
\[
\partial \framebox{$c^{\alpha,-}_{i,j}$} \doteq \framebox{$b^C_{i,j}$} + b^{D}_{i,j}  \mbox{  for $(i,j) \notin \{(k,k+1),(k+1,k+2)\}$,  and $\partial \framebox{$c^{\alpha,-}_{k,k+1}$}= \framebox{$b^C_{k,k+1}$}$,  }  
\]
so that Theorem \thmAlg from \cite{RuSu1} implies the result is equivalent to $(\mathcal{A}_{||}, \partial)$.  [Any $a^{-,-}_{i,m} c_{m,j}$ or $c_{i,m}a^{R}_{m,j}$ terms that may appear in $\partial c_{i,j}$ are already equal to $0$ in the quotient, since $|m-j| < |i-j|$.]

After the inductive process is complete, the new generating set contains no generators associated to $e^1_C$ or $e^2_{\alpha,-}$ and has only the generators $b^D_{i,j}$ such that $S_{i}$ and $S_{j}$ do not cross above $e^1_D$ when the crossing locus of $e^2_\alpha$ is homotoped into the one skeleton; i.e. with $(i,j) \neq (k+1,k+2)$ and $(i,j) \neq (k,k+1)$.  We have relations
\begin{equation} \label{eq:someRel2}
b^{C}_{i,j} = b^D_{i,j}  \quad \mbox{for $(i,j) \notin \{(k,k+1), (k+1,k+2)\}$};  \quad b^{D}_{k+1,k+2}=0;  \quad b^{C}_{k,k+1} = 0.
\end{equation}  

Once Cancellation 2 is completed in all Type (5), (6), (8), and (12) squares, we let $J \subset \mathcal{A}_{||}$ denote the resulting ideal so that stable tame isomorphic quotient that we have constructed is $(\mathcal{A}_{||}/J, \partial)$. 

\medskip

\noindent {\bf Correspondence between generators.}  We have completely removed all generators associated to the cells of $\mathcal{E}_{||}$ that do not correspond to cells of $\mathcal{E}_\tra$ from the generating set.  
In fact, every remaining generator of $\mathcal{A}_{||}/J$ corresponds to a cell of $\mathcal{E}_{\tra}$ and a pair of sheets $S_{i}$ and $S_{j}$ above that cell that do not meet when the crossing locus is homotoped into the $1$-skeleton.  [In particular, it was noted during the construction of $J$ that this holds for the $0$- and $1$-cells along the bottom edges of Type (5), (6), (8), and (12) squares.]

In the remainder of the argument, we use this correspondence to identify the underlying algebras,
\[
\mathcal{A}_{||}/J \, \cong \, \lchA/I.
\]
Note that the grading of generators coincides; see (\ref{eq:grading4}) and  equation (1) in Section 3.4 of \cite{RuSu1}.

\medskip

\noindent {\bf Claim:}  Let $e^2_\alpha$ be a $2$-cell of $\mathcal{E}_{\tra}$ that is not one of those Type (13), (14) and (9) squares that correspond to a square of $\mathcal{E}_{||}$ containing an $S$ or $T$ decoration.  Then, the differentials from $\mathcal{A}_{||}/J$ and $\lchA/I$ agree when applied to generators associated to cells in the closure of $e^2_\alpha$.

\begin{proof}[Proof of Claim] For generators associated to $0$ and $1$-cells, note that the formulas 
%(\ref{eq:Adef}) and (\ref{eq:Bdef}) 
(2) and (3)
from \cite{RuSu1} used to define the differential on the Cellular DGA are identical to the formulas (\ref{eq:AdefProp}) and (\ref{eq:BdefProp}) and the formulas are identical.  
Moreover, the entries of matrices that appear in these formulas coincide.
 [Here, we note that in $(\mathcal{A}_{||}, \partial)$, for the matrices $A_{\pm}$ and $B$ used for the bottom edge of a Type (5), (6), (8) and (12) square the two entries that correspond to the sheets that cross are both replaced with $0$ in $\mathcal{A}_{||}/J$ due to the relations (\ref{eq:someRel1}) and (\ref{eq:someRel2}).]

For $2$-cells of $\mathcal{E}_{||}$ such that the cell decomposition of the entire $2$-cell coincides with the decomposition of an entire square of $\mathcal{E}_{\pitchfork}$, the formulas 
%(\ref{eq:Cdef}) 
(4) from \cite{RuSu1} 
and  (\ref{eq:LCHdiffCcomp}) give identical differentials.  This relies on our having chose the initial and terminal vertices in $\mathcal{E}_{||}$ to coincide with the lower left and upper right corners of squares from $\mathcal{E}_{\pitchfork}$.  

There may be some squares, $e^2_\alpha$, not of type  (5), (6), (8) or (12) that none-the-less have one of their boundary edges subdivided when we form $\mathcal{E}_{||}$ from $\mathcal{E}_{\pitchfork}$.  [Specifically, this can happen for squares that share an edge with a Type (5), (6), (8) or (12) square.]  We discuss the case where the edge $D$ is subdivided, although, similar considerations apply if other edges are subdivided.  Then, 
%(\ref{eq:Cdef}) 
(4) from \cite{RuSu1}
gives that in $(\mathcal{A}_{||},\partial)$ 
\[
\partial C = A_{+,+} C + C A_{-,-} + [I+ B_U] [I + B_L] + [I+ B_R] [ I + B_{D,+}] [I + B_{D, -}] 
\]
where $B_{D,\pm}$ contain generators associated to the two halfs of the subdivided edge.  When we formed the quotient $\mathcal{A}_{||}/J$, all generators associated to $B_{D,-}$ became $0$. Thus, in $\mathcal{A}_{||}/J$, the desired formula
\[
\partial C = A_{+,+} C + C A_{-,-} + [I+ B_U] [I + B_L] + [I+ B_R] [ I + B_{D,+}] 
\]
holds for $e^2_\alpha$.

Finally, consider a Type (5), (6), (8), or (12) square, $e^2_\alpha$.  For Type (5), (8), (12) the generators of $\mathcal{A}_{||}$ associated to $e^2_{\alpha,+}$ (which are the ones identified with generators of $e^2_\alpha$ under the isomorphism $\mathcal{A}_{||}/J \cong \lchA/I$) have their differential in $(\mathcal{A}_{||}, \partial)$ given by
\[
\partial C = A_{+,+} C + C A_{-,-} + [I+B_U][I+B_L] + [I+ B_{R,+}] [I + B_{C}].
\]
For Type (6), they satisfy
\[
\partial C = A_{+,+} C + C A_{L} + [I+B_U][I+B_{L,+}] + [I+ B_{R,+}] [I + B_{C}].
\]
In $\mathcal{A}_{||}/J$, the relations (\ref{eq:someRel1}) and (\ref{eq:someRel2}) identify the entries of the matrices $B_{C}$ and $B_D$, as well as $A_{L}$ and $A_{-,-}$ in the Type (6) case.  Note that the location of entries in the matrices $B_C$ and $A_{L}$ may be different here than in (\ref{eq:can2}) (in fact, the $k+1$ and $k+2$ rows and columns are transposed).  However, the relations (\ref{eq:someRel1}) and (\ref{eq:someRel2}) preserve the the upper and lower sheet associated to generators, thus the previous formula is indeed equivalent to 
\[
\partial C = A_{+,+} C + C A_{-,-} + [I+B_U][I+B_{L}] + [I+ B_{R}] [I + B_{D}]
\]
that holds in $(\lchA/I, \partial)$.  [Recall that it is the entries of $B_{R,+}$ and $B_{L,+}$ that are associated to the $e^1_R$ and $e^1_L$ generators of $\lchA/I$ under the isomorphism $\mathcal{A}_{||}/J \cong  \lchA/I$.]
\end{proof}

\noindent {\bf The isomorphism $\Phi: (\lchA/I, \partial) \rightarrow (\mathcal{A}_{||}/J, \partial)$.}

Adding subscripts $1$ and $2$ to distinguish the notation for differentials, we will define an isomorphism $\Phi: (\lchA/I, \partial_1) \rightarrow (\mathcal{A}_{||}/J, \partial_2)$.  On all generators belonging to the closure of a 
square without a swallowtail decoration $S$ or $T$, $\Phi$ simply identifies generators of $\lchA/I$ and $\mathcal{A}_{||}$.  In defining $\Phi$ on the $S$ and $T$ squares, we consider only the case of a  Type (13) square (upward swallowtail) such that the swallowtail involves sheets $k,k+1,$ and $k+2$, as the other case is similar. 

We fix matrices containing generators from cells belonging to the closure of squares that have $S$ and $T$ decorations.  The $S$ and $T$ squares, as they appear in $\mathcal{E}_\tra$ and $\mathcal{E}_{||}$, are pictured in Figure \ref{fig:IsoProof2} where  matrices associated to each cell are indicated.  Supposing there are $n$ sheets to the right of the cusp edge, all matrices are $n\times n$ except for $A$ which is $(n-2) \times (n-2)$.  The matrices $B_0, B_1, B_2, B_3, A_1, A_2,$ and $C_1$ always have generators placed so that the ordering of rows and columns corresponds to the ordering of sheets of $\tilde{L}$ in the $T$ square, and $B_1$ has all entries in rows and columns $k$ and $k+1$ equal to $0$.  Note that the $(k+1,k+2)$ entry of $B_0$ is $0$.  The matrices $B_4, B_5, B_6, A_3$, and $C_2$ use the ordering of sheets above $S$, and $B_4$ also has $0$'s in rows and columns $k$ and $k+1$.

\begin{figure}

\quad

\centerline{ \includegraphics[scale=.6]{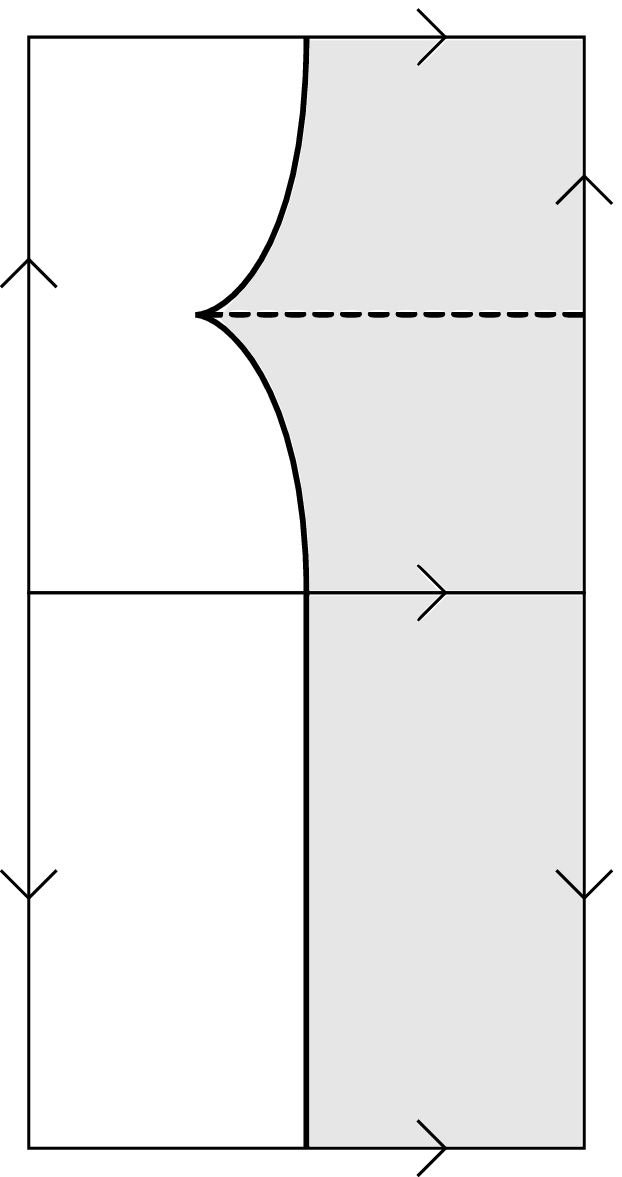} \quad \quad \quad \quad \quad \quad  
\labellist
\small
\pinlabel $C_1$  at 108 86
\pinlabel $C_2$  at 108 226
\pinlabel $A$  [r] at -2 152
\pinlabel $B_0$  [t] at 108 140
\pinlabel $B_2$  [t] at 98 0
\pinlabel $B_3$  [l] at 182 92
\pinlabel $B_1$  [r] at 26 39
\pinlabel $B_4$  [r] at 26 266
\pinlabel $B_5$  [b] at 96 302
\pinlabel $B_6$  [l] at 182 224
\pinlabel $T$  [br] at 47 125
\pinlabel $S$  [tr] at 45 185
\pinlabel $v_0$  [tr] at 26 80
\pinlabel $v_0$  [bl] at 68 159
\pinlabel $A_2$  [l] at 184 8
\pinlabel $A_3$  [l] at 184 298
\pinlabel $A_1$  [l] at 181 152
\pinlabel $v_1$  [tr] at 169 291
\pinlabel $v_1$  [br] at 169 14
\endlabellist
\includegraphics[scale=.6]{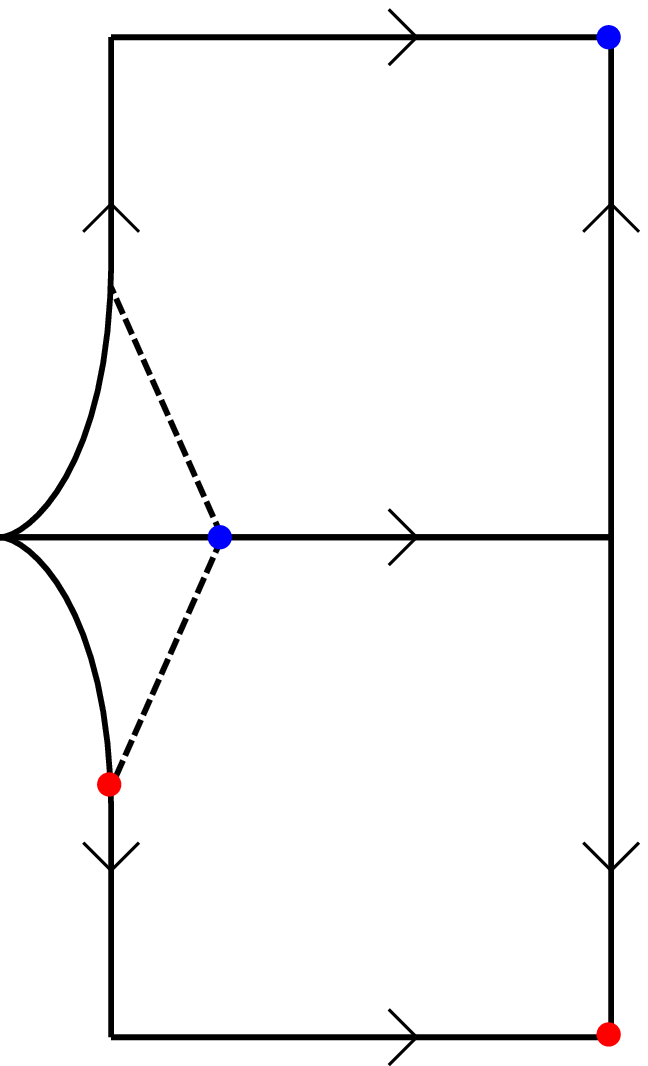}  }
%\centerline{ \includegraphics[scale=.8]{images/SubdivideSq} } %Dan had already commented out this line
\caption{Two squares of $\mathcal{E}_\pitchfork$ and $\mathcal{E}_{||}$ near a swallowtail point.  The dotted edges labelled $S$ and $T$ in the right are considered when computing the cellular differential, but are not $1$-cells of $\mathcal{E}_{||}$.  Initial and terminal vertices of the $C_1$ and $C_2$ squares are depicted in blue and red respectively.}
\label{fig:IsoProof2}
\end{figure}

As in \cite[Section 3.11 %\ref{sec:matrices}
]{RuSu1}, we let $\widehat{A}_{k,k+1}$ (resp. $\widehat{A}_{k,k+2}$) denote $A$ with two rows and columns added in positions $k$ and $k+1$ (resp. $k$ and $k+2$) with all entries $0$ except the $(k,k+1)$-entry (resp. the $(k,k+2)$-entry) which is $1$.  In addition, let $T = I+E_{k+1,k+2}$ and $S = I+ \widehat{A}_{k,k+1} E_{k+2,k} + E_{k+1,k+2},$ as in equation (7) from Section 3.11  of \cite{RuSu1}.

We complete the definition of $\Phi$ by setting
\[
\Phi(C_1) = C_1; \quad \Phi(C_2) = C_2; \mbox{    and}
\]
\[
\Phi(B_0) = B_0 [I+E_{k+1,k+2}].
\]
[All other generators belong to cells lying in the closure of squares without $S$ or $T$ decorations.]  
 Since $[I + E_{k+1,k+2}]$ is upper triangular, it is easy to verify that $\Phi$ is a {\it tame} isomorphism.  
 [For more details, compare with the proof of tameness given for the map $\psi$ from  
 % (\ref{def:psiIso})
  Theorem 4.5 from \cite{RuSu1}.]
 %\ref{thm:IndDecorate},
 Observe the formula
\begin{equation} \label{eq:PhiIB0}
\Phi(I+B_0+E_{k+1,k+2}) = (I+B_0)(I+E_{k+1,k+2}). 
\end{equation}

From Proposition \ref{prop:lchAI} and the definition of the Cellular DGA,
%\footnote{\ms{8/10/15:It would be good to offer more specific references. For example, lemma 3.5 (unlabeled), equation (7) (unlabeled), and equations ({eq:st1}), ({eq:st2}). Also, can you spell out $\partial_2 C_i$ a bit more?}}
  we have the following formulas:

\[
\begin{array}{rl} \partial_1 B_0  = & A_1[I+B_0+E_{k+1,k+2}] + [I+B_0+E_{k+1,k+2}] \widehat{A}_{k,k+1} \\
\partial_2 B_0  = & A_1[I+B_0] + [I+B_0][I+E_{k+1,k+2}] \widehat{A}_{k,k+1}[I+E_{k+1,k+2}] 
\end{array}
\]
[The $E_{k+1,k+2}$ terms appear in $\partial_1B_0$ because the $B_0$ Reeb chord between the $k+1,k+2$ sheets of the $T$ square is the exceptional generator  $b^D_{k+2,k+1}$ from the Type (13) square that satisfies $b^D_{k+2,k+1}=1$ in $\lchA/I$.]
\[
\begin{array}{rl} \partial_1 C_1  = & A_2 C_1 + C_1 \widehat{A}_{k,k+1} + [I+B_2][I+ B_1] + [I+B_3][I+B_0+E_{k+1,k+2}] \\
\partial_2 C_1  = & A_2 C_1 + C \widehat{A}_{k,k+1} + [I+B_2][I+ B_1] + [I+B_3][I+B_0] T 
\end{array}
\]
\[
\begin{array}{rl} \partial_1 C_2  = & A_3 C_2 + C_2(I+E_{k+2,k+1}) \widehat{A}_{k,k+2} (I + E_{k+2,k+1}) + 
\\  & [I+B_5][I+B_4] [I + \widehat{A}_{k,k+2} E_{k+1,k}+E_{k+1,k+2}] + [I+B_6][I+Q B_0 Q + QB_0Q E_{k+2,k+1}] \\
\partial_2 C_2  = & A_3 C_2 + C_2(I+E_{k+2,k+1}) \widehat{A}_{k,k+2} (I + E_{k+2,k+1}) + \\
  &  [I+B_5][I+B_4] S + [I+B_6][I+Q B_0Q]. 
\end{array}
\]
where $Q$ denotes the permutation matrix of the trasposition $(k+1 \,\, k+2)$.  [All occurences of $B_0$ are conjugated by $Q$ because the order of the $k+1$ and $k+2$ sheets is opposite above the $S$ square and the $T$ square.]

The verification of  $\Phi \circ \partial_1 = \partial_2 \circ \Phi$ on generators is as follows.  When applied to a generator not appearing in $B_1, C_1,$ or $C_2$, the equality follows from the Claim.   For the remaining generators, we compute using (\ref{eq:PhiIB0})
\begin{align*}
\Phi \circ \partial_1(B_0) = & A_1 \Phi(I+B_0+E_{k+1,k+2}) + \Phi(I+B_0+E_{k+1,k+2}) \widehat{A}_{k,k+1} =
\\
  &  \left(A_1[I+B_0] + [I+B_0][I+E_{k+1,k+2}] \widehat{A}_{k,k+1}[I+E_{k+1,k+2}]  \right) [I+E_{k+1,k+2}] = 
\\
  & \partial_2(B_0[I+E_{k+1,k+2}]) = \partial_2\circ \Phi(B_0);
\end{align*}
and
\[
\Phi \circ \partial_1(C_1) = A_2 C_1 + C_1 \widehat{A}_{k,k+1} + [I+B_2][I+ B_1] + [I+B_3] \Phi(I+B_0+E_{k+1,k+2}) = \partial_2 \circ \Phi(C_1).
\]
Finally, observe that $[I + \widehat{A}_{k,k+2} E_{k+1,k}+E_{k+1,k+2}] = [I + \widehat{A}_{k,k+1} E_{k+2,k}+E_{k+1,k+2}] =S$, so that all terms in $\partial_1 C_2$ and $\partial_2 C_2$ that do not involve $B_0$ are identical.  Thus, showing that $\Phi \circ \partial_1(C_2) = \partial_1\circ \Phi(C_2)$ reduces to the computation  
\[\Phi(I+Q B_0 Q + QB_0Q E_{k+2,k+1}) = Q\Phi(I+B_0[I + E_{k+1,k+2}])Q= Q(I+B_0)Q = I+ QB_0Q.
\]
\end{proof}

\begin{proof}[Proof of Theorem \ref{thm:CellularLCH}]
Stable tame isomorphism has the properties of an equivalence relation.  Thus, since $(\mathcal{A}_{||}/J, \partial)$ is stable tame isomorphic to the Cellular DGA of $L$, and $(\lchA/I, \partial)$ is stable tame isomorphic to the Legendrian contact homology DGA of $L$, it follow that the Cellular DGA is stable tame isomorphic to the LCH DGA.
\end{proof}

\section{Proof of Theorem \ref{thm:PropertiesofLtilde} Part 1: Construction of squares without swallowtails}
\label{sec:Constructions}

The remaining four sections of the paper concern the proof of Theorem \ref{thm:PropertiesofLtilde}.  We need to establish the existence of a Legendrian $\tilde{L}$ satisfying the (many) stated properties.  This task is carried out in two parts.  In Sections \ref{sec:Constructions}-\ref{sec:Properties}, we construct an initial approximation, $\tilde{L}_0$ ,to $\tilde{L}$ that satisfies all of the requirements of Theorem \ref{thm:PropertiesofLtilde} except that $\tilde{L}_0$ is not necessarily $1$-regular.  Then, in Section \ref{sec:ProofSetup} we show that by applying an appropriate perturbation to $\tilde{L}_0$ we may obtain the $1$-regular condition while preserving the other requirements of Theorem \ref{thm:PropertiesofLtilde}.  

The construction of $\tilde{L}_0$ is carried out in a fairly explicit manner using local coordinates on $S$ associated to the transverse square decomposition, $\mathcal{E}_\pitchfork$.  Recall that above each of the squares of $\mathcal{E}_\pitchfork$, the singular set of $L$ (topologically) matches one of the types (1)-(14),  and that above each $1$-cell the singular set has one of the four types (PV), (1Cr), (2Cr), and (Cu).    For each square type, we will produce a collection of  functions defined on (subsets of) $[-1,1]\times[-1,1]$ whose $1$-jets define $\tilde{L}_0$ above squares of that type.  These \emph{defining functions} have a standard form near 
%$\partial( [-1,1] \times [-1,1])$
%, so that the form of $\tilde{L}_0$ in a neighborhood of 
each boundary edge of $[-1,1]\times [-1,1]$ that only depends on the type of the edge, and this will allow us to verify that the pieces of $\tilde{L}_0$ fit together to produce a smooth Legendrian.

The construction of defining functions for squares of type (1)-(12) is done in a uniform manner and is carried out in the remainder of Section \ref{sec:Constructions}.  In Sections \ref{sec:1dim1234}-\ref{sec:1dim1} we construct functions over $[-1,1]$ that match the four possible $1$-cell types.
%models that appear as boundaries of the squares. 
%We  list in Propositions \ref{prop:PV1Cr2Cr} and\ref{lem:MainPropsCu} various features of the functions needed later.
Then, in Section \ref{ssec:Two-skeleton}  
we define two-variable functions for the Type (1)-(12) squares as  a suitable sum of interpolations between the one-variable functions associated to the boundary edges.  
%of the square.  of squares (1)-(12) to
The construction of the Type (13) and (14) squares is more involved due to the presence of swallowtail points, and is the topic of Section  \ref{sec:ConstructionsST}.   
The verification that the singular set of the resulting Legendrian has the proper form above each square is deferred until Section \ref{sec:Properties} where all of the properties of Theorem \ref{thm:PropertiesofLtilde} are verified, except for $1$-regularity.

%Section \ref{ssec:Two-skeleton} construction produces squares (1)-(12).

\subsubsection{Preliminaries}

Fix the Legendrian $L \subset J^1(S)$ and transverse square decomposition $\mathcal{E}_\pitchfork$ from the statement of Theorem \ref{thm:PropertiesofLtilde}.
Let $N$ denote the maximal number of sheets of the front projection of $L$ over any  point in $S$.

We can assume without loss of generality that above all squares of $\mathcal{E}_\pitchfork$ the lower most and upper most sheet of $L$ do not meet any other sheets at singular points, i.e. crossings, cusps, and swallowtail points.  If this is not the case, then we simply take the union of $L$ with the $1$-jets of a pair of constant functions with sufficiently large negative and positive value.  After constructing the resulting Legendrian, suitable defining functions for the original $L$ arise from discarding the two new components.

For the constructions, we  fix a constant $0 < \ea < 1$ such that
\begin{equation}
\label{eq:constants}
\ea< \left(\frac{1}{ 10N }\right)^{10}. 
\end{equation}
Later in the construction we will fix additional, even smaller constants $\ec$ and $\ed$.
%We also impose in the proof of Lemma \ref{lem:MainPropsPV} an additional upper bound (\ref{eq:epsilon3}) on $\ec$ not implied by (\ref{eq:constants}).

\begin{remark}
No explicit relation is assumed between $\ea, \ec, \ed$ and the constant $\e$ that appears in some of the Properties from  Sections \ref{sec:CompLCH}-\ref{sec:SwallowComp}, such as Property \ref{pr:Location1}.
\end{remark}

The possible (topological) appearance of the front projection of $L$ above any edge $e \cong [-1,1]$ of $\mathcal{E}_\pitchfork$ is determined by 
\begin{itemize}
\item[(i)] the number of sheets, $n$, of $L$ above $[-1,1]$;
\item[(ii)] the nature of the singular set above $[-1,1]$ which must be one of (PV), (1Cr), (2Cr), (Cu); and
\item[(iii)]  the location of crossing and cusp sheets which are specified in the (1Cr) and (Cu) case by a pair of consecutive integers $1 < k < k+1 < n$, and in the (2Cr) by $1< k< k+1<k+2< n$.
\end{itemize}
%Following the convention from Section \ref{ssec:Compute1cell}, when we label the sheets of the Legendrian above $[-1,1]$ as $S_1, \ldots, S_n$ with descending $z$-coordinate at $x=+1$.  
We refer to the data (i)-(iii) as an {\bf edge type}.
%In Sections ???, for all relevant $n \leq N,$ 
%\footnote{\ms{8/7/15: It would be better if $n +2 < N$ since we need to construct artificial sheets sometimes from above and below, e.g. in $f_{k-0.5}$ construction.}}

In Sections \ref{sec:1dim1234}-\ref{sec:1dim1}, we construct $1$-variable defining functions for $1$-dimensional Legendrians matching all possible edge types with $1 \leq n \leq N$.
%with $n$-sheets and whose singular sets match each of the $1$-cell Types (PV), (1Cr), (2Cr), (Cu), allowing for all possible locations of the crossing and cusp sheets.  We refer to the data of the number of sheets $n$ together with the specification of the singular set as matching one of (PV), (1Cr), (2Cr), (Cu) 
Following the convention from Section \ref{ssec:Compute1cell}, when considering a particular edge type we label the sheets of the Legendrian above $[-1,1]$ as $S_1, \ldots, S_n$ with descending $z$-coordinate at $x=+1$.  We will notate the corresponding defining functions  as $f_1, \ldots, f_n:[-1,1] \rightarrow \R$.  (For (Cu) edges, the defining functions $f_k$ and $f_{k+1}$ will have proper subsets of $[-1,1]$ as their domain.)  For $1 \leq i, j \leq n$, we notate \emph{difference functions} as
\[
f_{i,j} := f_i - f_j.
\]

Our constructions often make use of the following (probably standard) technical lemma.
\begin{lemma}  \label{lem:thetechlem}
Suppose we are given smooth functions $f$ and $g$ such that
$g: [a,b] \rightarrow \R$ satisfies $0< g'(x)$ for all $x \in [a,b]$,
 and for some $\delta >0$, 
$f: [a,a+\delta] \cup [b-\delta, b] \rightarrow \R$  satisfies
\[\begin{array}{ll}
0 < f'(x) < g'(x),  & \quad \quad \mbox{for all $x \in [a,a+\delta] \cup [b-\delta, b]$, and} \\
0 < f(b)-f(a) < g(b)-g(a). & \end{array}
\]
Then, there exists $0 < \delta' < \delta$ and  $\tilde{f}:[a,b] \rightarrow \R$ such that
\begin{equation} 
\begin{array}{ll}
0 < \tilde{f}'(x) < g'(x) &    \quad \quad \mbox{for all $x \in [a,b]$,  and} \\
\tilde{f}(x) = f(x) &  \quad \quad \mbox{for all $x\in [a,a+\delta'] \cup [b-\delta', b]$.}
\end{array}
\end{equation}
\end{lemma}

\begin{proof}
By continuity, we can choose $a < a' < a+\delta$ and  $b-\delta < b' < b$ so that
\[
0 < f(b')-f(a') < g(b')-g(a')
\]
continues to hold.  Start by defining $\widehat{f}: [a, b] \rightarrow \R$ via
\[
\widehat{f}(x) = \left\{ \begin{array}{cr} f(x),  &  x \in [a,a']\cup [b',b], \\
f(a') + \alpha \int_{a'}^x d_x g(s) \, ds, &  x \in[a',b'], \\
 \end{array} \right.
\]
where 
\[
\alpha = \frac{f(b')-f(a')}{g(b')-g(a')}  \quad \quad \mbox{satisfies  $0 < \alpha < 1$.}
\]
Note that $\widehat{f}$ is continuous everywhere and is smooth at points other than $x=a'$ and $x=b'$.  It is easy to verify that
$0 < d_x \widehat{f}(x) < d_x g(x)$ holds for all $x \in [a,b]$ including for the right and left derivatives of $\widehat{f}$ when $x=a'$ and $x=b'$.  Therefore, we can produce $\tilde{f}$ by choosing sufficiently small neighborhoods of $a'$ and $b'$ and smoothing $\widehat{f}$ within these neighborhoods in a manner that retains the inequality $ 0 < d_x \tilde{f}(x) < d_x g(x)$. 
\end{proof}

In typical applications of Lemma \ref{lem:thetechlem}, with $g$ already defined, and $f$ defined on $[a,b] \cup [c,d]$ satisfying $0 < f'(x) < g'(x)$ and $0 < f(c)-f(b)< g(c)-g(b)$, we extend $f$ to $[a,d]$ in a manner that preserves the derivative inequalities.  This is accomplished with Lemma \ref{lem:thetechlem} by first choosing some extension of $f$ to $[a,b+\delta]$ and $[c-\delta,d]$ and taking $\delta$ small enough so that the hypothesis of the Lemma are satisfied.  We will often make such extensions without explicit reference to Lemma \ref{lem:thetechlem}.

\subsection{$1$-dimensional defining functions in $[1/2,3/4]$}  \label{sec:1dim1234}

\begin{proposition}  \label{prop:1234def}
There exists a collection of smooth functions, 
\[
h_{i} : [1/2-\ea, 3/4+ \ea] \rightarrow \R, \quad \mbox{for $i = 1, \ldots N$,}
\]
satisfying:
\begin{enumerate}
\item  For all $1 \leq i < j \leq N$, $h_i(1/2) = h_i(3/4) = 0$, and  $h_i(x) > h_j(x)$, for all $x \in (1/2, 3/4)$.
\item  For all $1 \leq i \leq N$, $||h_i||_{C^0} < \ea/6$.
\item  For all $1 \leq i < j \leq N$, $h_{i,j}$ has a unique critical point, $ \eta_{i,j} \in [1/2-\ea, 3/4+\ea]$, that is a non-degenerate local maximum in  $(1/2,3/4)$.  Moreover, the $\eta_{i,j}$ appear in lexicographic order:
For any other $i' <j',$
\begin{equation}
\label{eq:stairwell}
i' < i, \,\, \mbox{or} \,\, i'  =i \,\, \mbox{and}\,\, j'<j
\quad \Rightarrow \quad
\eta_{i',j'} < \eta_{i,j}.
\end{equation}
\end{enumerate}
\end{proposition}

\begin{proof}
Inductively select points $p_N, p_{N-1}, \ldots, p_1$ in $(1/2,3/4)\times(0,+\infty)$ so that for all $1 \leq i < N$, we have
\begin{itemize}
\item[(i)] $x(p_i) > x(p_{i+1})$,  and
\item[(ii)] $p_i$ lies above the line through $(1/2,0)$ and $p_{i+1}$.
\end{itemize}
Define $\widehat{h}_i: [1/2-\ea, 3/4+\ea] \rightarrow \R$ to be piece-wise linear agreeing with the line $\ell_i$ through $(1/2,0)$ and $p_{i}$ on $[1/2-\ea, x(p_i)]$ and with the line $\ell_i'$ through $p_i$ and $(3/4,0)$ on $[x(p_i), 3/4 + \ea]$.  See Figure \ref{fig:LinStair}.  Next, produce $\widetilde{h}_i$ from $\widehat{h}_i$, by smoothing in a small neighborhood of $x= x(p_i)$ in a manner so that $d_x \tilde{h}_i$ decreases monotonically, and $d_x^2 \tilde{h}_i$ is non-zero where $\tilde{h}_i$ is non-linear.    It follows easily from (i) and (ii) that the critical points of $\tilde{h}_{i,j}$ are as required in item (3).  Finally, put $h_{i} = \alpha \tilde{h}_i$ where $\alpha >0$ is an overall constant chosen to be small enough so that item (2) holds.
\end{proof}

\begin{figure}

\labellist
\small
\pinlabel $\widehat{h}_1$ [bl] at 178 116
\pinlabel $\widehat{h}_2$ [bl] at 156 88
\pinlabel $\widehat{h}_3$ [bl] at 134 60
\pinlabel $\widehat{h}_4$ [tl] at 108 40
\endlabellist
\centerline{ \includegraphics[scale=.6]{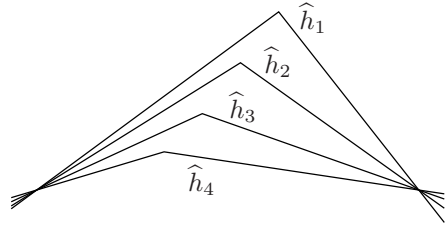} }

\caption{Graphs of the piecewise linear functions $\widehat{h}_{i}$.}
\label{fig:LinStair}
\end{figure}

 Let 
\begin{equation}  \label{eq:eeta}
{\e}_\eta = \frac{1}{3} \mbox{Min}\{ |\eta_{i,j}-\eta_{i',j'}|  \, :\, (i,j) \neq (i',j')\},
\end{equation} 
so that ${\e}_\eta$-balls centered around the $\eta_{i,j}$ are all disjoint. Let
\begin{equation} \label{eq:Mdef}
M = \mbox{Min}_{i<j}\left( \inf \{ |d_x{f}_{i,j}(x)| \, :\,  x \in [1/2,3/4] \setminus (\eta_{i,j}-{\e}_\eta, \eta_{i,j} + {\e}_\eta) \}\right) >0.
\end{equation}
Choose $\ec>0 $ such that
\begin{equation}
\label{eq:epsilon3}
\ec <  \min \left\{  \frac{ \ea^{10}}{(10N)^{10}}, \frac{4M\ea}{N} \right\}.
%\frac{\ec N}{M} < \eb.
\end{equation}

\subsection{$1$-dimensional defining functions}  \label{sec:1dim1}

For an arbitrary edge type with $1 \leq n \leq N$, we fix defining functions
%For $1 \leq n \leq N$, we consider each of the edge types (PV), (1Cr), (2Cr), (Cu) with $n$ sheets (allowing all possible locations of the crossing and cusping sheets not involving the upper-most and lower-most sheet).  For each of these cases, we will fix functions 
\[
f_1, \ldots, f_n: [-1,1] \rightarrow \R
\] 
%so that the union of their $1$-jets define a $1$-dimensional Legendrian whose front projection matches the given edge type.  
[Although the edge type is not indicated in the notation for the $f_i$, 
no relation is assumed between $f_i$ corresponding to distinct edge types.]  

%We fix some preliminary notations.

%Label the sheets of a given edge type as $S_1, \ldots, S_n$ with descending $z$-coordinate as they appear above $x=+1$.  
We let $\sigma_-(i)$ denote the ordering that sheet $S_i$ appears in above $x=-1$, where in the (Cu) case we put $\sigma_-(k) = \sigma_-(k+1) = k-0.5$ when $S_k$ and $S_{k+1}$ meet at the left cusp.  Explicitly, for edges of type (PV), (1Cr), and (2Cr), $\sigma_-$ is respectively given by permutations $id$,  $(k \,\,  k+1)$ and $(k \,\, k+1 \,\, k+2)$ while in the (Cu) case we have
\[
\sigma_-(i) = \left\{ \begin{array}{cr} i,  & i<k; \\  k-0.5, &  i = k, k+1;\\ i-2, & i >k+1. \end{array} \right.
\]
Note that $\sigma_-$ uniquely determines the edge type.  For $i<j$, sheets $S_{i}$ and $S_{j}$ cross if and only if $\sigma_-(i) > \sigma_-(j)$.  In addition, put $\sigma_+(i) := i$ for $1 \leq i \leq n$.

For each edge type, we fix  $y_{i,i+1}$,  $i = 1, \ldots, n-1$, that will correspond approximately to the maximum values of the $f_{i,i+1}$ as follows:

\medskip

\noindent (PV) and (Cu):  Set all $y_{i,i+1} =1$. 

\medskip

\noindent (1Cr):  Set $y_{i,i+1} = 1$, $i \neq k-1,k,k+1$, and 
\[
y_{k-1,k} = 1.5, \quad y_{k,k+1} = 0, \quad y_{k+1,k+2} = 1.5.
\]

\medskip

\noindent (2Cr):  Set $y_{i,i+1} = 1$, $i \neq k-1,k,k+1,k+2$, and 
\begin{eqnarray}
\label{eq:TwoCrossingConstants}
&y_{k-1,k}= 1.125, \quad y_{k,k+1} = 0.625,\quad y_{k+1,k+2} = 0.125,&\\
\notag
& y_{k+2,k+3} = (k+3) - (k-1)- \sum_{\eta=k-1}^{k+1} y _{\eta, \eta+1} = 2.125.&
\end{eqnarray}

\medskip

Finally, for $1 \leq i \leq j \leq n$, put $y_{i,j} = \sum_{l = i}^{j-1} y_{l,l+1}$.

\begin{proposition} \label{prop:PV1Cr2Cr} For any (PV),  (1Cr), or (2Cr) edge type, functions $f_i$,  $1 \leq i \leq n$, can be constructed so that for all $1 \leq i < j \leq n$ the  following properties hold:

\begin{enumerate}
\item  If $\sigma_-(i) < \sigma_-(j)$, then $f_{i,j}(x) > 0$, $\forall x\in [-1,1]$.  

If $\sigma_-(i) > \sigma_-(j)$, then $f_{i,j}$ has a unique non-degenerate $0$ located at $x=0$ in the (1Cr) case, and located in $[1/4, 1/4+\e_1]$ in the (2Cr) case.

\item  For all $x \in [1/2,3/4]$,  
\[
d_x f_{i,j}(x) = d_x h_{i,j}(x),
\]
and for all $x \in [1/4,3/4]$,
\[
|f_{i,j}(x) - y_{i,j}| < N \ea.
\]
  In particular, $f_{i,j}$ has a non-degenerate local maximum at $\eta_{i,j}$.  

 If $\sigma_-(i) < \sigma_-(j)$, then $\eta_{i,j}$ is the only critical point of $f_{i,j}$ in $(-1,1)$.  

If $\sigma_-(i) > \sigma_-(j)$, then the only critical points of $f_{i,j}$ in $(-1,1)$ are $\eta_{i,j}$ and a non-degenerate local minimum, $\tilde{\eta}_{j,i} \in [-3/4,-1/2]$.

\item For $x \in [-1,-1/4]$, 
\[
| f_{i,i+1}(x) | < \ea.
\]
%\dr{If $\sigma_-(i) < \sigma_-(i+1)$, then
%\[
%d_xf_{i,i+1}(x) \geq 1, \quad  \forall \, x \in [-3/8-\ec, -3/8+\ec].
%\]}

\item  For $|x - (\pm 1)| \le 1/8,$  
\begin{equation*}
%\label{eq:7/8}
f_{i,i+1}(x) = (\sigma_\pm(i+1) - \sigma_\pm(i)) Q_{\pm}(x)
\end{equation*}
where
\[
Q_{\pm}(x) := %\left(
\ec( x - (\pm 1))^2 +\ec
%\right)
.
\]
\item  The restriction of $f_{i,i+1}$ to $[-1/4,1/4]$ is linear with slope
\[
|d_x f_{i,i+1} - 2 y_{i,i+1} | < \ea.
\]

\item
For $x \in [1/3 - \ea, 1/3 + \ea]$, $d_xf_{i,i+1}(x) = \ec.$

%\item
%For \dr{$x \in [13/16-\ea, 13/16+\ea],$ $d_xf_{i,i+1}(x) \le -1.$}

\end{enumerate}

\end{proposition}

\begin{proof}
We construct the $f_{i}$ by setting $f_n = 0$ and specifying $f_{i,i+1} = f_i - f_{i+1}$ for $1 \leq i \leq n-1$.  

To begin, fix a smooth function $g: [-1, -1/4-\ea] \rightarrow \R$ with the following properties:
%\begin{align}
%\tag{G1}  \forall x \in [-1,-7/8], \quad  &  g(x) = \ec( x - (\pm 1))^2 +\ec; \\
%\tag{G2}  \forall x \in (-1,-1/4-\ea], \quad & g'(x) > 0; \\
%\tag{G3}   & g(-3/4) = .06\ea \quad \mbox{and} \quad g(-1/4 -\ea) = .09 \ea; \\
%\tag{G4}  \forall x \in [-3/8 -\ec, -3/8+ \ec], \quad &  g'(x) = 1.
%\end{align}

\begin{enumerate}
\item[(G1)] For $x \in [-1,-7/8]$, 
$g(x) = Q_-(x)$.
\item[(G2)] For all $x \in (-1,-1/4-\ea]$, \quad $g'(x) > 0$.
\item[(G3)] We have $g(-3/4) = .06\ea$,  $g(-1/2) = .07 \ea$, $g(-3/8)= .08\ea$, and $g(-1/4 -\ea) = .09 \ea$.
\item[(G4)] For all $x \in [-3/8 -\ec, -3/8+ \ec]$,   $g'(x) = 1$.
\end{enumerate}
It is possible to construct such a function using Lemma \ref{lem:thetechlem} since
 \[
g( -7/8) = \ec 65/64 < g(-3/4) < g(-1/2) < g(-3/8-\ec) < g(-3/8+\ec)< g(-1/4-\ea).
\]
%and the total change in $g$ on the interval $[-3/8-\ec, -3/8+\ec]$ required by (G4) is $2 \ec$ which is much less than $g(-1/4 -\ea) - g(-3/4)$.
(We use the Lemma to verify that $g'(x)>0$ is attainable with no upper bound on $g'$ required.  Thus, we just need to check that $g(b)-g(a)>0$ holds on the intervals $[a,b]$ that we extend $g$ over.)

\medskip

\noindent {\bf Defining $f_{i,i+1}$ when $\sigma_-(i) < \sigma_-(i+1)$:}
In all cases where $\sigma_-(i) < \sigma_-(i+1)$, we define $f_{i,i+1}$ in the following steps.  

Let $\sigma(i,i+1) : = \sigma_-(i+1) - \sigma_-(i)$.

\medskip

\noindent {\bf Step 1.}  Begin by making the partial definition
\begin{equation} \label{eq:initialdef}
f_{i,i+1}(x) = \left \{\begin{array}{lcl}
\sigma(i,i+1) \cdot g, & \quad & x \in [-1,-1/4-\ea], \\
.1\ea \cdot \sigma(i,i+1) + 2 y_{i,i+1}(x+1/4), & \quad & x \in [-1/4, 1/4], \\
y_{i,i+1} + (2/3)\ea + h_{i,i+1}(x),  & \quad & x \in [1/2,3/4], \\
Q_+(x). & \quad & x \in [7/8,1].
\end{array} \right.
\end{equation}

\medskip

\noindent {\bf Step 2.}  On the remaining portions of $[-1,1]$, smoothly interpolate in a manner that produces the conditions
\begin{equation} \label{eq:remaining}
\begin{array}{lcl} f_{i,i+1}'(x) > 0, & \quad & x \in (-1,1/2]; \\
f_{i,i+1}'(x) < 0, & \quad & x \in [3/4,1); \mbox{  and} \\
f_{i,i+1}'(x) = \ec, & \quad & x \in [1/3 -\ea,1/3+\ea].  
%f_{i,i+1}'(x) = -1, & \quad & \dr{x \in [13/16-\ea, 13/16+\ea]}.
\end{array}
\end{equation}
[The first two conditions required in (\ref{eq:remaining}) are obtainable since the inequalities already hold on intervals where $f_{i,i+1}$ was defined in (\ref{eq:initialdef}), and (keeping in mind that $\sigma(i,i+1) \leq 3$)
\[
f_{i,i+1}(-1/4-\ea) = .09 \ea \cdot \sigma(i,i+1) < f_{i,i+1}(-1/4) = .1 \ea \cdot \sigma(i,i+1); 
\]
\[
f_{i,i+1}(1/4) = y_{i,i+1} + .1 \ea \cdot \sigma(i,i+1) < y_{i,i+1} + (2/3) \ea = f_{i,i+1}(1/2);
\]
\[
f_{i,i+1}(3/4) = y_{i,i+1} + (2/3) \ea > (65/64)\ec = f_{i,i+1}(7/8).
\]
The third condition is obtainable since the total change to $f_{i,i+1}$ that it requires on the interval $[1/3 -\ea,1/3+\ea]$ is much smaller in magnitude than $f_{i,i+1}(1/2)- f_{i,i+1}(1/4)$ from (\ref{eq:epsilon3}).]

Next, we define $f_{i,i+1}$ in the two special cases where $\sigma_-(i) > \sigma_-(i+1)$.

\medskip

\noindent {\bf Defining $f_{k,k+1}$ for a (1Cr) edge:} Set
\begin{equation}  \label{eq:initialdef2}
f_{k,k+1}(x) = \left\{\begin{array}{cr} \sigma(k,k+1) g(x) = (-1)g(x), & x \in [-1,-3/4]; \\
-.05 \ea + .2 \ea(x+1/4), & x \in [-1/4,1/4]; \\
(5/6)\ea + h_{k,k+1}(x), & x\in [1/2,3/4]; \\
%\ec( x - 1)^2 +\ec
Q_+(x),  & x \in [7/8,1],
\end{array} \right.
\end{equation}
 then (using Lemma \ref{lem:thetechlem}) complete 
 the definition by smoothly interpolating in a manner that arranges
 \begin{itemize}
\item $d_xf_{k,k+1}$ has a unique $0$ at some $\tilde{\eta}_{k+1,k} \in (-3/4,-1/2)$ and satisfies $d_xf_{k,k+1}(x) \geq -d_xg(x)$ for $x \in (-3/4,-1/2)$; 
\item $d_xf_{k,k+1} >0$ on $[-1/2, 1/2]$;  \quad $d_xf_{k,k+1} <0$ on $[3/4,7/8]$; and
\item $d_xf_{k,k+1} = \ec$ on $[1/3-\ea, 1/3+\ea]$.
\end{itemize}
[The first condition is obtained by choosing a small interval within $(-3/4,-1/2)$ to interpolate $d_xf_{k,k+1}$ between $-d_xg(x)$ and a small positive constant.  We can then arrange that $d_xf_{k,k+1}>0$ continues to hold up to $x=-1/4$ since the local minimum value will satisfy 
\[
f_{k,k+1}(\tilde{\eta}_{k+1,k}) < f_{k,k+1}(-3/4) = -g(-3/4) = -.06\ea <  -.05 \ea = f_{k,k+1}(-1/4).
\]
The 2-nd and 3-rd conditions are obtainable using Lemma \ref{lem:thetechlem} since
\[
f_{k,k+1}(1/4) = .05 \ea < (5/6)\ea =f_{k,k+1}(1/2),  \quad  f_{k,k+1}(3/4) =(5/6)\ea > (65/64)\ec = f_{k,k+1}(7/8),
\]
\[ \quad \mbox{and} 
 \quad f_{k,k+1}(1/2) -f_{k,k+1}(1/4) > 2\ea \ec.\mbox{]}
\]

\medskip

\noindent {\bf Defining $f_{k+1,k+2}$ for a (2Cr) edge:} Set
\begin{equation}  \label{eq:initialdef3}
f_{k+1,k+2}(x) = \left\{\begin{array}{cr} \sigma(k+1,k+2) g(x) = (-2)g(x), & x \in [-1,-3/4]; \\
-.11 \ea + 2 y_{k+1,k+2} (x+1/4), & x \in [-1/4,1/4]; \\
y_{k+1,k+2} + h_{k+1,k+2}(x), & x\in [1/2,3/4]; \\
%\ec( x - 1)^2 +\ec,  
Q_+(x), & x \in [7/8,1],
\end{array} \right.
\end{equation}
 then complete the definition by smoothly interpolating in a manner that arranges
\begin{itemize}
\item $d_xf_{k+1,k+2}$ and $d_x(g+f_{k+1,k+2})$ both have unique $0$'s in $(-3/4,-1/2)$ that we denote $\tilde{\eta}_{k+2,k+1}$ and $\tilde{\eta}_{k+2,k}$. In addition, $d_xf_{k+1,k+2}(x) \geq -2 d_xg(x)$ for $x \in (-3/4,-1/2)$; 
\item $d_xf_{k+1,k+2} >0$ on $[-1/2, 1/2]$;  \quad $d_xf_{k+1,k+2} <0$ on $[3/4,7/8]$; and
\item $d_xf_{k+1,k+2} = \ec$ on $[1/3-\ea, 1/3+\ea]$.
\end{itemize}
[The uniqueness of the $0$ of $d_x(g+f_{k+1,k+2})$ is arranged by taking the $2$-nd derivative of $f_{k+1,k+2}$ to be sufficiently large when we interpolate $d_x f_{k+1,k+2}$ between $-2 d_x g$ and a small positive constant in a small subinterval of $(1/2,3/4)$.  To see that $d_xf_{k+1,k+2}$ may remain positive in $(\tilde{\eta}_{k+2,k+1}, -1/4]$ and that the 2-nd and 3-rd conditions are obtainable (using Lemma \ref{lem:thetechlem}), 
 note that
\[
f_{k+1,k+2}(\tilde{\eta}_{k+2,k+1}) < f_{k+1,k+2}(-3/4) = -2g(-3/4) = -.12\ea <  -.11 \ea = f_{k+1,k+2}(-1/4),
\]
\[
f_{k+1,k+2}(1/4) = -.11 \ea + y_{k+1,k+2} < y_{k+1,k+2} =f_{k+1,k+2}(1/2), 
\]
\[
  f_{k+1,k+2}(3/4)=y_{k+1,k+2} > (65/64)\ec = f(7/8),
\]
\[ 
\quad \mbox{and} \quad f_{k+1,k+2}(1/2) -f_{k+1,k+2}(1/4) > 2\ea \ec.]
\]

\medskip

With the definition of all $f_{i,i+1}$ complete, we verify properties (1)-(6).
 
\smallskip

Items (3)-(6):  These items all follow from (\ref{eq:initialdef}), (\ref{eq:initialdef2}), and (\ref{eq:initialdef3}) together with the additional itemized requirements imposed during the interpolation process.  Item (3) follows from verifying the inequality at $x=-1$ and $x=-1/4$, and at the critical point $\tilde{\eta}_{i+1,i}$ 
if it exists.  [Keep in mind that in all cases, $|\sigma(i,i+1)| \leq 3$.]
Item (4) follows from (G1).  Item (5) is explicit in the definitions of $f_{i,i+1}$.  Item (6) was arranged during the interpolation step of the constructions of the $f_{i,i+1}$.  

\smallskip

Item (2):  For $1 \leq i < j \leq n$,  
\begin{equation} \label{eq:1leqij}
f_{i,j} = \sum_{l=i}^{j-1} f_{l,l+1}.
\end{equation}
Thus, for $x \in [1/2,3/4]$,  
\[
f_{i,j}(x) = \sum_{l=i}^{j-1} ( C_l + h_{l,l+1}(x))= \left(\sum_{l=i}^{j-1}  C_l\right) + h_{i,j}(x),
\]
where each $C_l$ is a non-negative constant satisfying $|C_l - y_{i,j}| \leq (5/6)\ea$. [See (\ref{eq:initialdef}), (\ref{eq:initialdef2}), and (\ref{eq:initialdef3}).]  That $d_x  f_{i,j}(x) = d_x h_{i,j}(x)$ follows, and using Proposition \ref{prop:1234def} (2), we have
\begin{equation} \label{eq:fy}
|f_{i,j}(x) - y_{i,j}| \leq \sum_{l=i}^{j-1}|C_i - y_{i,i+1}|  + |h_{i,j}(x)| < \left(\sum_{l=i}^{j-1} (5/6)\ea\right) + \ea/6 \leq N \ea.
\end{equation}
Note that $f_{i,j}(x)$ is increasing on $[1/4,1/2]$, so to verify that the inequality (\ref{eq:fy}) continues to hold on $[1/4,1/2]$ it is enough to check that $|f_{i,i+1}(x)-y_{i,i+1}| \leq \ea$ holds at $x=1/4$.  This is explicitly seen from (\ref{eq:initialdef}), (\ref{eq:initialdef2}), and (\ref{eq:initialdef3}).

%In particular, the stated estimate holds when $x = \eta_{i,j}$.

Within $[1/2,3/4]$ the only critical point of $f_{i,j}(x)$ is $\eta_{i,j}$ (from Proposition \ref{prop:1234def} (3)), and in $[-1/2,1/2]$ (resp. $[3/4, 1)$) we have $d_xf_{i,i+1} >0$ (resp. $d_xf_{i,i+1}<0$) for all $i$, so that $d_xf_{i,j}>0$ (resp. $d_xf_{i,j}<0$) follows from (\ref{eq:1leqij}).  
Thus, to establish the claim about the critical points of $f_{i,j}$ we are left to consider $x \in (-1,-1/2]$.

Note that when $x \in (-1,-1/2]$, the  constructions give
\[
\quad d_x f_{i,i+1}(x) \geq \sigma(i,i+1) d_xg(x),
\]
so we can estimate
\[
d_x f_{i,j}(x) = \sum_{l=i}^{j-1} d_xf_{l,l+1} \geq \sum_{l=i}^{j-1} \sigma(l,l+1) d_xg(x) = \left( \sigma_-(j) - \sigma_-(i) \right) d_xg(x). 
\]
Thus, when $\sigma_-(i) < \sigma_-(j)$,  
\[d_x f_{i,j}(x) >0,  \quad \mbox{for $x \in (-1,-1/2]$} 
\]
 follows from (G2), so $\eta_{i,j}$ is the only critical point of $f_{i,j}$ in $(-1,1)$.
In the three cases where $\sigma_-(i) > \sigma_-(j)$, the existence of a unique critical point in $(-1,-1/2)$, denoted $\tilde{\eta}_{j,i}$, follows from the construction.  [In the case of a (2Cr) edge with $(i,j) = (k,k+2)$,  $f_{k,k+2} = -g(x)$ holds in $(-1,-3/4]$;  in $[-3/4,-1/2]$, $f_{k,k+2} = g(x) + f_{k+1,k+2}$, and the existence of a unique $0$ of $d_x(g(x) + f_{k+1,k+2})$ in $[-3/4,-1/2]$ was required in the definition of $f_{k+1,k+2}$.] 

\smallskip

Item (1):  When $\sigma_-(i) < \sigma_-(j)$, $f_{i,j}>0$ follows since it holds when $x=\pm 1$  (by Item (4)), and the unique critical point of $f_{i,j}$ in $(-1,1)$ is a local maximum.

The three cases where $\sigma_-(i) > \sigma_-(j)$ are $f_{k,k+1}$ for a (1Cr) edge, and $f_{k,k+2}$ and $f_{k+1,k+2}$ for a (2Cr) edge. All three of these functions are negative at $-1$ and $-1/4$ with a single local min in $(-1,-1/4)$.  [For $f_{k,k+2}$ use Item (4), and compute
\[
f_{k,k+2}(-1/4) = .1 \ea + (-.11 \ea) = -.01 \ea <0.] 
\]
Moreover, all three are positive at $1/4$ and $1$ with a single local max in $(1/4,1)$.  We conclude that all three are negative (resp. positive) on $[-1,-1/4]$ (resp. $[1/4,1]$).  Each function has a unique $0$ on the interval $[-1/4,1/4]$ that can be explicitly found using (\ref{eq:initialdef}), (\ref{eq:initialdef2}), and (\ref{eq:initialdef3}) with the location as specified in the statement.

\end{proof}

With defining functions now fixed for all (PV), (1Cr), and (2Cr) edge types, we let
\[
K = \inf\{ d_xf^{(1Cr)}_{i,j}(x)\}.
\]
The infimum is taken over all (1Cr) difference functions where we allow for any number of sheets $n \leq N$ and all locations of the crossing sheets $k$ and $k+1$; in addition, we require $i<j$ and $-3/8 \leq x \leq \eta_{i,j}-\e_{\eta}$ (where $\e_\eta$ was defined in (\ref{eq:eeta})).  Note that since the collection of functions involved is finite, compactness gives $K>0$.

Next, fix $\ed >0$ to satisfy
\begin{equation} \label{eq:epsilon4}
\ed < \min \left\{ \ec, \frac{16 \ea}{9 N} \cdot \min\{K, 1/5\} \right\}.
\end{equation}

\begin{proposition}  \label{prop:CuDef}
For any (Cu) edge type, we can construct functions 
\[
f_i: [-1,1] \rightarrow \R,   \quad 1 \leq i\leq n, \,  i \notin \{k,k+1\}, \quad \mbox{and} \quad f_k,f_{k+1}: [-3/8,1] \rightarrow \R,
\]
so that for all $1 \leq i < j \leq n$ the items (1)-(5) of Proposition \ref{prop:PV1Cr2Cr} hold provided that we impose the additional hypothesis that $x$ belongs to the domain of the function under consideration.   In place of item (6), we have
\begin{itemize}
\item[(6')]  For $x \in [1/3- \ea, 1/3 +\ea]$,  $d_xf_{i,i+1} = \ed$.  
\end{itemize}

In addition, we have
\begin{itemize}
\item[(7')] The function $f_{k-1,k+2}$ satisfies
%\footnote{\dr{The second property is already there from (2).  The third property can be deduced from (2) and (3).  Maybe convenient to have a corollary where the $f_{i,j}$ versions of estimates in (3) and others are stated since they may be used later.} } 
\[
\begin{array}{ll} f_{k-1,k+2} = 
%(\sigma_-(k+2)-\sigma_-(k-1))(\ec(x+1)^2+ \ec) = 
%\ec(x+1)^2+ \ec
Q_-(x), \quad & \forall \, x \in [-1,-7/8]; \\
d_x f_{k-1,k+2} > 0, & \forall \, x \in (-1,-3/8];  \\
|f_{k-1,k+2}(x)| \leq \ea, &  \forall \, x \in [-1,-3/8].  
%d_x f_{k-1,k+2} =3, & \forall \, x \in [-3/8-\ec,-3/8+\ec]. 
\end{array}
\]
\item[(8')]  For $x \in [-3/8-\ec , -3/8+\ec]$ the functions $f_{i}$ with $i \notin\{k,k+1\}$ are linear with slopes satisfying
\[\begin{array}{l}
d_x f_{i,j}(x) = j-i,  \quad \quad \mbox{for $\{i,j\}\cap \{k,k+1\} =\emptyset.$} \\
d_x f_{k-1,k}(-3/8) = d_xf_{k+1,k+2}(-3/8) = 1.5.
\end{array}
\]
For $x \in [-3/8, -3/8+\ec]$, 
\[
\begin{array}{l}
f_{k}(x) = L(x) + (x+3/8)^{3/2}, \\
f_{k+1}(x) = L(x) - (x+3/8)^{3/2},
\end{array}
\]
where $L(x)$ is a linear function.
In particular, the sheets $k$ and $k+1$ meet in a semicubical cusp point at $x=-3/8$.  The function $f_{k,k+1}$ satisfies
\[
\begin{array}{lr}
%f_{k,k+1} = 2(x+3/8)^{3/2}, & x \in [-3/8, -3/8+ \ec], \\
f_{k,k+1}(x) >0,  &  x \in (-3/8,1],  
\end{array}
\]
and $\eta_{k,k+1}$ is the only critical point of $f_{k,k+1}$ in $(-3/8,1)$.
\end{itemize}

\end{proposition}

\begin{proof}
For $i \neq k-1,k,k+1$, we define $f_{i,i+1}$ as in the proof of Proposition \ref{prop:PV1Cr2Cr} with the modification that on the interval $[1/3-\ea,1/3+\ea]$ we now require that $d_xf_{i,i+1} = \ed$.

To construct the $f_i$, we will define $f_{k-1,k}$, $f_{k,k+1}$, and $f_{k+1,k+2}$, which all have domain $[-3/8, +1]$.  
In addition, we will define $f_{k-1,k+2}$ in $[-1,-3/8+\ec]$.  Then, we set $f_n =0$, and 
\[
f_i :=  \sum_{l = i}^{n-1} f_{i,i+1}  \quad \mbox{for $k \leq i \leq n$}
\]
where when $i=k,k+1$ we restrict the domain of all functions in the summation to $[-3/8,1]$.  For $1 \leq i \leq k-1$, we set 
\begin{equation} \label{eq:deflessthan}
f_i(x) := \left\{ \begin{array}{lc} \displaystyle  \sum_{l = k+2}^{n-1} f_{l,l+1}(x) + f_{k-1,k+2}(x) + \sum_{l=k-2}^{i} f_{l,l+1}(x), & x \in [-1,-3/8 + \ec] \\
\displaystyle \sum_{l=i}^{n-1} f_{l,l+1}(x), & x \in [-3/8, +1],
\end{array} \right.
\end{equation}
and check that $f_{k-1,k+2} = f_{k-1,k}+f_{k,k+1}+f_{k+1,k+2}$ holds on $[-3/8,-3/8+\ec]$ to verify well-definedness.

Define $f_{k-1,k+2}|_{[-1,-3/8+\ec]}$ to have
\[
f_{k-1,k+2}(x) = \left\{ \begin{array}{lc} %\ec(x+1)^2+ \ec,
Q_-(x), & x \in[-1,-7/8], \\ 3 g(x), &  x \in [-3/4, -3/8+\ec],
\end{array} \right. 
\]
(where $g(x)$ is from the proof of Proposition \ref{prop:PV1Cr2Cr}), and to satisfy
\[
d_x f_{k-1,k+2}(x) > 0, \quad \mbox{for $x \in (-1,-3/8+\ec]$}.
\]
[This is possible since $f_{k-1,k+2}(-7/8) = (65/64)\ec < f_{k-1,k+2}(-3/4) = .18\ea$.]
%\footnote{\ms{3/8/16: How are there no other critical points between, say $S_{k-1}$ and $S_{k-2}?$} \dr{3/8:  $f_{k-1,k-2}$ is defined in the first line of the proof.  It is as constructed in the main case of the previous proposition, and one of its properties is that $d_xf_{i,i+1}$ has only one critical point that is a local max in $[1/2,3/4]$.  Maybe you are thinking back to an earlier version of the construction where the $f_i$ (Cu) are the same as $f_i$ (PV) except for the cusp sheets (and similar for (1Cr) and (2Cr)).  This was one major change in my revision.  Now, it is the $f_{i,i+1}$ that are the same.  The $f_i$ may be different, but we don't care.  (If you go the other way around, then you have to consider many more special cases in proofs.)}}

Next, define $f_{k+1,k+2}$, $f_{k,k+1}$ and $f_{k-1,k}$ on $[-3/8, -3/8 +\ec]$ by
\begin{equation} \label{eq:k1k2CuDef}
\begin{array}{lr}
f_{k+1,k+2}(x) = .12 \ea + 1.5 (x+ 3/8) -(x+3/8)^{3/2},  & 
f_{k,k+1}(x) = 2(x+3/8)^{3/2}, \quad \mbox{and}
\\ & \\
f_{k-1,k}(x) = f_{k-1,k+2}(x) - f_{k,k+1}(x) - f_{k+1,k+2}(x),  & \mbox{for $x \in [-3/8,-3/8+\ec]$.}
\end{array}
\end{equation}
Note that for $x \in [-3/8, -3/8+\ec]$,
%on this interval all three functions have positive derivative:
\begin{align*}
d_xf_{k+1,k+2}(x) & = 1.5-(3/2)(x+3/8)^{1/2} \geq 1.5-(3/2)\ec^{1/2} >0,
\\
d_xf_{k,k+1}(x) & = 3 (x+3/8)^{1/2}> 0, \quad \mbox{for $x \in (-3/8, -3/8+\ec]$}, \quad \mbox{and}
\\
d_xf_{k-1,k}(x) & = 3 d_x g(x) - d_x f_{k+1,k+2}(x) - d_x f_{k+1,k+2}(x) = 3 -1.5 -(3/2)(x+3/8)^{1/2} > 0
\end{align*}
where in the last equality we used the property (G4) of $g(x)$.
Moreover, 
\begin{equation} \label{eq:08ea}
f_{k+1,k+2}(-3/8) = .12 \ea, \quad f_{k,k+1}(-3/8) = 0, \quad \mbox{and} 
\end{equation}
\[
 f_{k-1,k} = 3 g(-3/8) - .12 \ea = .12\ea. 
\]
[We used properties (G2) and (G3) of $g(x)$.]  Note that $f_{k-1,k+2} = f_{k-1,k}+f_{k,k+1}+f_{k+1,k+2}$ holds on $[-3/8,-3/8+\ec]$ as required.

We now complete the definition of $f_{k+1,k+2}$, $f_{k,k+1}$ and $f_{k-1,k}$ by requiring, for $i = k-1,k,$ and $k+1$,
\[
f_{i,i+1}(x) = \left\{ \begin{array}{lc}   
.2 \ea + 2y_{k-1,k}(x+1/4), & x\in [-1/4,1/4], \\
y_{k-1,k} + (2/3)\ea + h_{i,i+1},  &  x \in [1/2,3/4], \\
Q_+(x),  & x \in [7/8,1];
\end{array} \right.
\] 
and by interpolating on the remaining intervals so that
\[
\mbox{$d_x f_{i,i+1} >0$ for $x \in (-3/8,1/2]$; \quad $d_x f_{i,i+1} < 0$ for $x \in [3/4, 1)$; and}
\]
\[
\mbox{$d_x f_{i,i+1} = \ed$, \quad for $x \in [1/3-\ea,1/3+\ea]$.}
\]
[
This is possible since we have verified above that $d_xf_{i,i+1} >0$ at $x= -3/8 +\ec$, and  $f_{i,i+1}(-3/8+\ec) < .2\ea$.  Moreover, $f_{i,i+1}(1/4) < f_{i,i+1}(1/2)$, $f_{i,i+1}(3/4) > f_{i,i+1}(7/8)$, and $f_{i,i+1}(1/2) -f_{i,i+1}(1/4) > 2\ea\ed$ are all easily verified.
]

With the definitions complete we verify items (1)-(5) from Proposition \ref{prop:PV1Cr2Cr} and (6')-(8').

\smallskip

The properties stated in items (3)-(5) and (6')-(7') and most of the conditions of (8') all follow immediately from the construction.  To verify, that in $[-3/8-\ec,-3/8+\ec]$ the $f_{i}$ with $i \neq k,k+1$ are linear use induction.  By definition, $f_n =0$; at the inductive step we use that (G4) states that $g(x)$ is linear with slope $1$ in $[-3/8-\ec, -3/8+\ec]$, and observe that $f_{i,i+1} = g(x)$ for $i\geq k+2$ and $i \leq k-2$, and $f_{k-1,k+2}= 3 g(x)$ hold in $[-3/8-\ec, -3/8 +\ec]$.  That $f_k$ and $f_{k+1}$ have the required form in $[-3/8-\ec,-3/8+\ec]$ then follows from the definition of $f_{k+1,k+2}$ and $f_{k,k+1}$. 

Item (2):  Verify as in the proof of Proposition \ref{prop:PV1Cr2Cr}.  Note that when checking $d_{x}f_{i,j} >0$ in $(-1, 1/2]$ 
 it is important that this property holds for $f_{k-1,k+2}$ and $f_{i,i+1}$, $i=k-1,k,k+1$ (aside from $d_xf_{k,k+1}(-3/8) =0$).

Item (1):  Follows from from the corresponding statement for $(i,j) = (i,i+1)$ or $(k-1,k+2)$.  In all cases, we check positivity (strict except when $(i,j) = (k,k+1)$) at the endpoints of the interval of definition.  This follows from items (4) and (7') together with equation (\ref{eq:08ea}).   Checking endpoints suffices since, in the interior of its interval of definition, each $f_{i,j}$ has only a single critical point that is a local maximum by (2). 

\end{proof}

\begin{corollary}  \label{cor:summary}
For any edge type, and any $1 \leq i < j \leq n$, the following properties hold for all $x$ belonging to the domain of $f_{i,j}$.  (The numbering follows Proposition \ref{prop:PV1Cr2Cr}.)

\begin{itemize}
\item[(3)]  If $x \in [-1, -1/4]$, then
\[
|f_{i,j}(x)| < N \ea.
\]
\item[(4)]  If $|x -(\pm1)| \leq 1/8$, then
\[
f_{i,j}(x) = \left( \sigma_{\pm}(j)- \sigma_{\pm}(i) \right) Q_{\pm}(x),
\]
and
\[
f_{i}(x) = (n_\pm - \sigma_{\pm}(j)) Q_{\pm}(x)
\] 
where $n_\pm$ is the number of sheets defined above $x= \pm1$.  
\item[(5)]  The restriction of $f_{i,j}$ to $[-1/4,1/4]$ is linear with slope
\[
|d_xf_{i,j} -2y_{i,j}| < N \ea.
\]
\item[(6)]  If $x \in [1/3-\ea, 1/3+\ea]$, then 
\[
0 < d_x f_{i,j}(x) \leq N \ec   \quad \quad \mbox{(resp.  $0 < d_x f_{i,j}(x) \leq N \ed$)}
\]
for (PV), (1Cr), and (2Cr) edge types (resp. for (Cu) edge types).
\end{itemize}

\end{corollary}

\begin{proof}
For (PV), (1Cr), and (2Cr) edge types, we have $f_{i,j} = \sum_{l=i}^{j-1} f_{l,l+1}$.  Thus, (3)-(6) follow from the corresponding properties in Proposition \ref{prop:PV1Cr2Cr} using induction and the triangle inequality.  For (Cu), $f_{i,j}$ is determined via the $f_{i,i+1}$ and $f_{k-1,k+2}$ as in  (\ref{eq:deflessthan}), so to apply induction we use Proposition \ref{prop:CuDef} including item (7').  The second equation in (4) follows from the first since $f_n = 0$.
\end{proof}

\begin{figure}

\quad

\quad

\labellist
\small
\pinlabel $C^0$-small~in~$[-1,-1/4]$ [bl] at 8 14
\pinlabel Linear~in~$[-1/4,1/4]$ [bl] at 144 104
\pinlabel slope$\,\approx\,2y_{i,j}$ [bl] at 154 76
\pinlabel Local~max~at~$\eta_{i,j}$ [tl] at 298 120
\pinlabel $f_{i,j}|_{[1/4,3/4]}\approx\,y_{i,j}$ [bl] at 298 88
\endlabellist
\centerline{ \includegraphics[scale=.9]{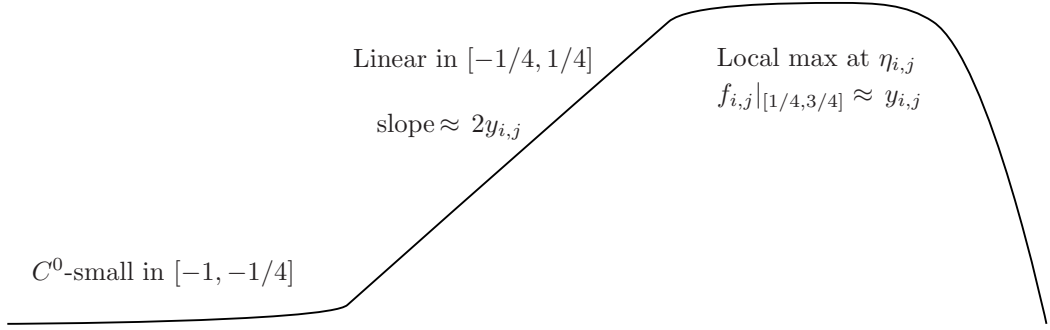} }

\caption{The key features of the difference functions $f_{i,j}$.}
\label{fig:fijfeatures}
\end{figure}

A summary of the properties of the difference functions $f_{i,j}$ appears in Figure \ref{fig:fijfeatures}.

\subsection{Two-skeleton}
\label{ssec:Two-skeleton}
We now turn to constructing functions over $[-1,1]\times[-1,1]$ for squares with type (1)-(12) as numbered in Section \ref{sec:transverse}.  In this section, when we refer to a Legendrian {\it square type} we mean a specification of the singular set as one of the squares (1)-(12), together with the number of sheets $n$ with $1 < n \leq N$ and the location of sheets that correspond to the  crossing arcs and cusp edges.  Fixing a square type specifies edge types for the edges $U, R, D, L$ of $[-1,1]\times[-1,1]$ which are respectively $[-1,1] \times \{1\}, \{1\} \times [-1,1],  [-1,1] \times \{- 1\},$ and $\{-1\} \times [-1,1]$, and we note that this (ordered) collection of edge types uniquely determines the square type.
  Denote the $1$-dimensional functions associated to each of the edge types as $f^U_i$, $f^R_i$, $f^D_i$, and $f^L_i$.

%Since we have not actually exhibited the constructions for the 12 squares yet, we think of ``square ($n$)", $1 \le n \le 12,$ as the (ordered) choice of one-cell models on the square's four boundaries.  
%\footnote{
%\dr{I'm not quite sure what the meaning/intent of this paragraph is.  When the different square types are introduced, they are referring to the form of the singular set above the square.}
%\ms{8/6/15: I mean that now square (n) is defined by the loci intersected at the boundary, indirectly by the choice of 1-dim models. Only later it is shown that the implied loci in the interior agrees with section 6 def of square (n).}
%}

%\dr{Mike:  (i)  It is extremely likely that the $\sigma_L, \sigma_D$ notation will not be used at all in computation sections.  (ii)  For proving the Properties, it may be convenient to observe that $\eta_{\sigma_L(i), \sigma_L(j)} = \beta^L_{i,j}$, at least when $i$ and $j$ are not cusp sheets.  Or, does your notation require a superscript too?}

In (\ref{eq:interpolating}), we construct a defining function $F_i: [-1,1]\times[-1,1] \rightarrow \R$ for each sheet $S_i$ of a given square type, as a  sum of interpolations of $1$-dimensional functions determined by the location of $S_i$ above each of the four boundary edges.  
%In the case that $S_i$ terminates at a cusp edge above $U$ (resp. above $R$), 
However, due to the presence of cusp edges, some sheets may not exist above all four edges of $[-1,1]\times[-1,1]$.  In response, we create, in the next Lemma, some `artificial' $1$-dimensional functions $f^L_{k-.5}$ (resp. $f^D_{k-.5}$) for the cases where $S_i$ terminates at a cusp edge above $U$ (resp. above $R$).  
They will possess properties similar to those of Proposition \ref{prop:PV1Cr2Cr}.
We focus our notation on the construction of $f^L_{k-.5}$ functions, while the construction of $f^D_{k-.5}$ is carried out in an identical manner.

%$f^U_i$ and $f^D_{\sigma_D(i)}$ as well as between $f^R_i$ and $f^L_{\sigma_L(i)}$ 

%If the sheets  exist over some but not all four edges we create an `artificial' function in Lemma \ref{lem:MainProps0.5}.
%We must do this in a manner that is consistent with other adjacent squares which may not have the same
%crossing and/or cusp loci (see Lemma \ref{lem:MainProps0.5} item (4) below).
%Suppose  $S_k,S_{k+1}$ exist at the up right corner but not at the up left corner due to a cusp, as in squares (9)-(12).
%(The discussion when the sheet does not exist at the down right corner is similar.)
%So $\sigma_L(i) = i$ for $i \le k-1$ and $\sigma_L(i) =i-2$ for $i \ge k+2.$
%Let $\sigma_L(k) =\sigma_L(k+1) =k - 0.5.$ 

Label sheets above $[-1,1]\times[-1,1]$ as $S_1, \ldots, S_n$ with descending $z$-coordinate at $(+1,+1)$.
For integers $i,j$ let $y^L_{i,j}$ 
denote $y_{i,j}$ as defined above Proposition \ref{prop:PV1Cr2Cr} based on whether edge $L$ is (PV) as in squares (9) or (12), (1Cr) as in square (10) or (Cu) as in square (11). 
Edge $L$ will never be (2Cr).
%Let $y^L_{k-0.5,j} = \frac{y^L_{\sigma_L(k-1),j} +y^L_{\sigma_L(k+2),j}}{2}$ 
%and $y^L_{i,k-0.5} = \frac{y^L_{i,\sigma_L(k-1)} +y^L_{i,\sigma_L(k+2)}}{2}.$

\begin{lemma}
\label{lem:MainProps0.5}
%\footnote{\ms{8/5/15: As structured, i need to refer to the proof of this lemma to prove the switch point lies above crossing locus for square (12).}}
Consider a (9)-(12) square type such that $S_k$ and $S_{k+1}$ meet at a cusp above the $U$ edge.  There exists $f^L_{k-.5} \in C^\infty([-1,1])$ such that for all $1 \leq i < k-.5 < j \leq n-2$ the following properties hold (with the hypothesis that $x$ belongs to the domain of the functions in question always imposed).
%These are analogous to (parts of) properties in Lemma \ref{lem:MainPropsPV}, \ref{lem:MainPropsCr} or \ref{lem:MainPropsCu}. 
%(We enumerate these  to match the corresponding ones in Lemmas \ref{lem:MainPropsPV}. 
%When both $i$ and $j$ are integers, these properties are simply restating those in Lemmas \ref{lem:MainPropsPV},
%\ref{lem:MainPropsCr} or \ref{lem:MainPropsCu}.)
\begin{enumerate}

\item
For all $x \in [-1,1]$, $f^L_{i,k-.5}(x)> 0$ and $f^L_{k-.5,j}(x)> 0$, except in the case of a Type (12) square when $j = k$.  For the Type (12) square,  $f^L_{k-.5,k}$ has a unique non-degenerate 0 in $(-1/4,0).$
%at $0.$ 

\item
For all $x \in (-1, 1/2]$,  $d_x f_{i,k-.5} >0$ and $d_x f_{k-.5,j} >0$,  except in the case of a Type (12) square when $j = k$.  For the Type (12) square, $f^L_{k-..5,k}$ has a unique critical point in $(-1,1/2]$ that is a local minimum located in $(-3/4,-1/2)$.

For all $x \in [3/4,1)$, $d_x f_{i,k-.5} <0$ and $d_x f_{k-.5,j} <0$.

For all $x \in [1/4,3/4]$,  
\[
|f_{i,k-.5}(x) - C_{i,k-.5}| < N \ea, \quad \mbox{and} \quad |f_{k-.5,j}(x) - C_{k-.5,j}| < N \ea
\]
for constants $C_{i,k-.5} \in \{y^L_{i,k-1}, y^L_{i,k}\}$ and $C_{k-.5,j} \in \{y^L_{k-1,j}, y^L_{k,j} \}$.

%$f^L_{i,j}$ has a maximum at $1/2 < \eta^L_{i,j} < 3/4$ with $|f^L_{i,j}(\eta^L_{i,j})- y^L_{i,j}| < \ea.$
%%%\footnote{\ms{8/5/15: Don't think i use ``artificially extended" staircase ordering in property proofs.}}.
%None of the $f^L_{i,j}$ have other interior critical points except
%in Square (10) where (if sheets $S_l,S_{l+1}$ are the crossing pair) $f_{l,l+1}$ has a critical point at $\tilde{\eta}_{l+1,l} \in [-3/4,-1/2]$
%and in Square (12) where $f_{k-0.5, k}$ has a critical point at $\tilde{\eta}_{k,k-0.5} \in [-3/4,-1/2].$

\item For $x \in [-1,-1/4]$,
\[
|f^L_{i,k-.5}(x)| < N \ea,  \quad \mbox{and} \quad |f^L_{k-.5,j}(x)| < N\ea.
\]

\item
Let  $\sigma_+$ and $\sigma_-$ be associated to the functions $f^L_i$, so that $\sigma_\pm(i)$ specifies the ordering of the $f^L_i(\pm1)$ with descending $z$-coordinate.\footnote{Here, we can take $\sigma_-(l)$ to be undefined if $-1$ is not in the domain of $f^L_i$ as may occur for an (11) square type.}  In addition, set $\sigma_+(k-.5) = k-.5$, and $\sigma_-(k-.5) = l-.5$ where along $D$ sheets $S_k$ and $S_{k+1}$ meet at a cusp between the sheets that appear in positions $l-1$ and $l$ above $(-1,-1)$.     
%For $i \ne k-0.5,$ let  $\sigma_-(i)$ denote the order of $S_i$ as it appears over $(x_1,x_2) = (-1,- 1).$
%Let $\sigma_-(k-0.5) = \frac{\sigma_-(\kappa)+\sigma_-(\lambda)}{2}$ where $S_\kappa$ and $S_\lambda$ are the two sheets immediately above and below the cusping sheets $S_k,S_{k+1}$ along edge $D.$
For $|x- (\pm 1)| \le 1/16,$
\begin{equation} \label{eq:altv4}
f^L_{k-.5,j}(x) = (\sigma_\pm(j)-\sigma_\pm(k-.5)  )Q_{\pm}(x),  \quad \mbox{and} \quad f^L_{i,k-.5}(x) = (\sigma_\pm(k-.5) -\sigma_\pm(i)  )Q_{\pm}(x).
\end{equation}

%Let  $\sigma_+(i) = \sigma_L(i).$ 
%For $i \ne k-0.5,$ let  $\sigma_-(i)$ denote the order of $S_i$ as it appears over $(x_1,x_2) = (-1,- 1).$
%Let $\sigma_-(k-0.5) = \frac{\sigma_-(\kappa)+\sigma_-(\lambda)}{2}$ where $S_\kappa$ and $S_\lambda$ are the two sheets immediately above and below the cusping sheets $S_k,S_{k+1}$ along edge $D.$
%For $|x- (\pm 1)| \le 1/16,$  
% $f^L_{i,j}(x) = (\sigma_\pm(j)- \sigma_\pm(i) )(\ec (x - (\pm 1))^2 + \ec).$

\item
$f_{i,k-.5}^L$ and $f_{k-.5,j}$ are linear when restricted to $[-1/4,1/4]$ with slopes
\[
 |df^L_{i,k-.5} - 2C_{i,k-.5}| < N \ea, \quad \mbox{and} \quad |df^L_{k-.5,j} - 2C_{k-.5,j}| < N \ea.
\]
\end{enumerate}

\end{lemma}
%We could construct $f^L_{k-0.5}$ so that other properties related to Lemma \ref{lem:MainPropsPV} hold; however,
%we only list those needed in subsequent sections.

\begin{proof}
First we construct $f^L_{k-.5}$ for {\bf squares of type (9)-(11)}.  Consider the sheets $S_{k-1}$ and $S_{k+2}$ that appear directly above or below the cusping sheets along edge $U$.  For any square of type (9)-(11) it is the case that at least one of $S_{k-1}$ or $S_{k+2}$ does not meet another sheet at a crossing or cusp anywhere above $[-1,1] \times[-1,1]$ which leads to cases:

\medskip

\noindent{ \bf Case 1.} $S_{k-1}$ is disjoint from all other sheets above $[-1,1]\times[-1,1]$.  
Note that since $S_{k-1}$ is directly above $S_k$ and $S_{k+1}$ above both the $U$ and $D$ edge, we have
\begin{equation} \label{eq:sigma55}
\sigma_+(k-.5) = k-.5 = \sigma_+(k-1) + .5, \quad \mbox{and} \quad \sigma_-(k-.5) = \sigma_-(k-1) + .5.
\end{equation}

Let $f^{PV}_{1,2}$ denote the difference function for the (PV) edge type with $2$ sheets.  (Actually, taking any $f_{i,i+1}$ of any (PV) edge type would work, but we consider $f^{PV}_{1,2}$ for concreteness.)
We seek to define $f^L_{k-.5}$ by requiring that 
\begin{equation}  \label{eq:flk51}
f^L_{k-.5}(-1) 
%= (n-\sigma_-(k-.5))\ec 
= f^L_{k-1}(-1) - .5 f^{PV}_{1,2}(-1), 
\end{equation}
and that
\begin{equation}  \label{eq:jj1}
d_xf^L_{k-.5} = d_xf^L_{k-1} - \beta d_x f^{PV}_{1,2}
\end{equation} 
where $\beta: [-1,1] \rightarrow \R$ is an even function satisfying 
\[
0 <\delta \leq \beta(x) \leq 1/2; \quad \beta(x)= 1/2, \, \mbox{when $|x-(\pm 1)| \leq 1/16$}; \quad \mbox{and} \quad \beta(x) = \delta, \, \mbox{when $x \in [-7/8,7/8]$}.   
\]

\medskip

\noindent {\bf Case 2:} $S_{k-1}$ is not disjoint from all other sheets.  Then, $S_{k+2}$ must be disjoint from the others.  In this case we define $f^L_{k-.5}$ to satisfy 
\[
f^L_{k-.5}(-1) 
%= (n-\sigma_-(k-.5))\ec 
= f^L_{k}(-1) + .5 f^{PV}_{1,2}(-1), 
\]
and 
\begin{equation} \label{eq:jj2}
d_xf^L_{k-.5} = d_xf^L_{k} + \beta d_x f^{PV}_{1,2}.
\end{equation}

\medskip

We show that the required conditions (1)-(5) are satisfied provided that $\delta >0$ is chosen sufficiently small.  We present the proof when $f^L_{k-.5}$ is defined as in Case 1.  With obvious modifications, the proof applies equally well in Case 2.\footnote{The change from $-$ to $+$ that occurs between equations (\ref{eq:jj1}) and (\ref{eq:jj2}) is to arrange that in Case 1 $d_x f^L_{k-1,k-.5} = \beta d_x f^{PV}_{1,2}$ and in Case 2 $d_xf^L_{k-.5,k} = \beta d_x f^{PV}_{1,2}$ so that these derivatives share their sign with $d_xf^{PV}_{1,2}$.  In verifying properties of derivatives for the other $f^L_{i,k-.5}$ and $f^L_{k-.5,j}$, the $\beta d_x f^{PV}_{1,2}$ terms are dominated in absolute value by remaining terms, so that the difference between the $-$ and the $+$ is irrelevant.} 

\begin{lemma} \label{lem:c0deltae1}  For all sufficiently small $\delta >0$,  
\[
|| f^L_{k-1} - f^L_{k-.5}||_{C^0([-1,1])} < 2\ec.
\]
\end{lemma}
\begin{proof}
Note that the combination of (\ref{eq:flk51}) and 
%$f^L_{k-.5}(-1) = f^L_{k-1}(-1) - .5 f^{PV}_{1,2}(-1)$, so 
(\ref{eq:jj1}) gives 
\[
f^L_{k-.5}(x) = f^L_{k-1}(x) -.5f^{PV}_{1,2}(-1) - \int_{-1}^x \beta(x) \cdot d_xf^{PV}_{1,2}(x) \, dx.
\]
As $-.5f^{PV}_{1,2}(-1)= -.5 \ec$, this allows us to estimate for $x \in [-1,1]$
\begin{equation} \label{eq:dxF12Est}
|f^L_{k-1}(x)-f^L_{k-.5}(x) | \leq .5\ec + |\int_{-1}^x \beta(x) \cdot d_xf^{PV}_{1,2}(x) \, dx| \leq .5 \ec + \int_{-1}^1 \beta(x) |d_xf^{PV}_{1,2}(x)| \, dx. 
\end{equation}
Now, since $\beta \leq 1/2$ is even and $d_xf^{PV}_{1,2}(-x) = -d_xf^{PV}_{1,2}(x)$ for $x \in [-1,-7/8]$, we have 
\[
\int^{1}_{7/8} \beta(x) |d_x f^{PV}_{1,2}(x) | \, dx = \int_{-1}^{-7/8} \beta(x) |d_x f^{PV}_{1,2}(x) | \, dx \leq
\]
\[ 
.5 \int_{-1}^{-7/8} d_x f^{PV}_{1,2}(x) \, dx = .5[f^{PV}_{1,2}(-7/8) - f^{PV}_{1,2}(-1)] = \ec/128. 
\]
Thus, for small enough $\delta >0$, we have 
\[
\int_{-1}^1 \beta(x) |d_xf^{PV}_{1,2}(x)| \, dx < 2(\ec/128) + \delta \int_{-7/8}^{7/8} |d_xf^{PV}_{1,2}(x)| \, dx < \ec,
\]
In conjunction with (\ref{eq:dxF12Est}), this results in
\[
|f^L_{k-1}-f^L_{k-.5}(x) | \leq .5 \ec + \ec < 2\ec.
\] 
\end{proof}

\smallskip

Item (4):  %Let $Q_{\pm}(x) = \ec (x- (\pm1))^2 + \ec$ so that item (4) requires that
%\begin{equation}  \label{eq:altv4}
%f^L_{k-.5}(x) = (n-\sigma_{\pm}(k-.5)) Q_{\pm}(x), \quad \mbox{when } |x-(\pm1)| \leq 1/16.
%\end{equation}
Note that in view of Corollary \ref{cor:summary} (4), since $f^L_{i,k-.5} = f^L_{i,k-1}+f^L_{k-1,k-.5}$ and $f^L_{k-.5,j} = f^L_{k-1,j}-f^L_{k-1,k-.5}$, it suffices to establish the case of $f^L_{k-1,k-.5}$.  
%???\footnote{Statement of quadratic form near $\pm1$ for $f^L_{k}$ as opposed to for difference functions.}, 
We have that for $|x-(\pm1)| \leq 1/16$,
%$x\in [-1,-15/16]$,
\begin{equation} \label{eq:Qpm}
d_x f^L_{k-1,k-.5}(x) = .5 d_x f^{PV}_{1,2}(x) = .5 d_xQ_\pm(x) = d_x\left[(\sigma_-(k-.5) - \sigma_-(k-1))Q_{\pm}(x)\right] 
\end{equation}
%\begin{equation} \label{eq:Qpm}
%d_x\left((n- \sigma_-(k-.5)) Q_{\pm}(x)\right).
%\end{equation}
where we used (\ref{eq:sigma55}).
In addition, we have
\[
f^L_{k-1,k-.5}(-1) = .5 f^{PV}_{1,2}(-1) = (\sigma_-(k-.5) - \sigma_-(k-1)) Q_{-}(-1),
%(\ec (x+1)^2 + \ec)|_{x=-1},
\]
so it follows that (\ref{eq:altv4}) holds when $x \in [-1, -15/16]$.
%A similar computation gives for $x \in [15/16]$,
%\begin{equation} \label{eq:nsigmak55}
%d_x f^L_{k-.5} = d_x\left((n- \sigma_+(k-.5))(\ec (x+1)^2 + \ec)\right),
%\end{equation}
At $x=1$ we have 
\begin{equation} \label{eq:betaintpv}
f^L_{k-1,k-.5}(1) = f^L_{k-1,k-.5}(-1) + \int_{-1}^1 \beta d_x f^{PV}_{1,2} \, dx 
%= \left(f^L_{k-.5}(-1) +  f^L_{k-1}(1) -f^L_{k-1}(-1) \right) - \int_{-1}^1\beta d_x f^{PV}_{1,2} \, dx.
%f^L_{k-.5}(1) = f^L_{k-.5}(-1) + \int_{-1}^1d_xf^L_{k-1} - \beta d_x f^{PV}_{1,2} \, dx = \left(f^L_{k-.5}(-1) +  f^L_{k-1}(1) -f^L_{k-1}(-1) \right) - \int_{-1}^1\beta d_x f^{PV}_{1,2} \, dx.
\end{equation}
Note that using Proposition \ref{prop:PV1Cr2Cr} (4), we have
\[
\int_{-1}^1\beta d_x f^{PV}_{1,2} \, dx = \int^{-7/8}_{-1}\beta d_x f^{PV}_{1,2} \, dx + \int_{7/8}^1 \beta d_xf^{PV}_{1,2} \, dx + \int_{-7/8}^{7/8} \delta d_x f^{PV}_{1,2} \, dx =
\]
\[
\int_{-1}^{-7/8} \beta(x)\cdot \ec 2(x+1)\, dx + \int_{7/8}^1\beta(x)\cdot \ec 2(x-1)\, dx + \delta \left(f^{PV}_{1,2}(7/8)- f^{PV}_{1,2}(-7/8)\right) = 
\]
\[
\int_{1}^{7/8} \beta(-x) \cdot \ec2(-x+1) \, (-1)dx + \int_{7/8}^1\beta(x)\cdot \ec 2(x-1)\, dx + 0 = 0.
\]
(In the last two equalities, we substituted $u = -x$, and used that $\beta$ is even.)
Thus, (\ref{eq:betaintpv}) becomes
\[
f^L_{k-1,k-.5}(1) = f^L_{k-1,k-.5}(-1) = .5 Q_-(-1) = .5 Q_+(-1) = (\sigma_+(k-.5) - \sigma_+(k-1))Q_{+}(-1)
%f^L_{k-.5}(1) =f^L_{k-.5}(-1) +  f^L_{k-1}(1) -f^L_{k-1}(-1) =  (n-\sigma_-(k-.5) + n - \sigma_+(k-1) - n + \sigma_-(k-1))\ec =
\]
%\[
%(n-\sigma_+(k-1) + .5) \ec = (n- \sigma_+(k-.5)) \ec,
%\]
where we used (\ref{eq:sigma55}).  When combined with (\ref{eq:Qpm}), this shows that item (4) holds for $x \in [15/16,1]$.

\smallskip

In the remaining items note that for $i<k-.5 < j$, we set $C_{i,k-.5} = y^L_{i,k-1}$ and $C_{k-.5,j} = y^L_{k-1,j}$.  

\smallskip

Item (2):  
For $x\in [1/4,3/4]$, using Lemma \ref{lem:c0deltae1} and item (2) of Propositions \ref{prop:PV1Cr2Cr} and \ref{prop:CuDef}  we have
\[
|f^L_{i,k-.5}(x)- C_{i,k-.5}| \leq |f^L_{k-1,k-.5}(x)| + |f^L_{i,k-1}(x)-y^L_{i,k-1}| < 2 \ec + (k-1-i)\ea  < N \ea,
\]
and $|f^L_{k-.5,j}(x)- C_{k-.5,j}| < N \ea$ holds similarly.

To address the remaining inequalities required in (2), 
note that
\[
d_xf^L_{i,k-.5} = d_xf^L_{i,k-1} + \beta d_x f^{PV}_{1,2}, \quad \mbox{and} \quad d_xf^L_{k-.5,j} = d_xf^L_{k-1,j} - \beta d_x f^{PV}_{1,2}.
\]

When $i =k-1$ the inequalities required of $d_xf^L_{k-1,k-.5}= \beta d_x f^{PV}_{1,2}$ follow from the properties of $d_{x} f^{PV}_{1,2}$ (as in Proposition \ref{prop:PV1Cr2Cr} (2)).    Thus, we may assume $i<k-1$.
For 
%$x \in [-1,-7/8] \cup [7/8,1]$ 
$0 < | x - (\pm 1)| \leq 1/8$
we use the explicit form (as in Corollary \ref{cor:summary} (4)) of the $f^L_{i,k-1}$, $f^L_{k-1,j}$, and $f^{PV}_{1,2}(x)$ to verify 
\begin{equation}  \label{eq:betaxright}
d_xf^L_{i,k-.5}(x) = \left[ \sigma_\pm(k-1) - \sigma_\pm(i) + \beta(x) \right] d_x Q_\pm(x)  \quad \mbox{and}
\end{equation}
\[
d_xf^L_{k-.5,j}(x) = \left[ \sigma_\pm(j) - \sigma_\pm(k-1) - \beta(x) \right] d_x Q_\pm(x).
\] 
Since sheet $S_{k-1}$ does not meet any crossing or cusp above the edge $L$, and $i<k-1<j$, we have
\[
\sigma_\pm(k-1) - \sigma_\pm(i) \geq 1, \quad  \mbox{and} \quad \sigma_\pm(j) - \sigma_\pm(k-1) \geq 1.
\]
As $0 <\beta \leq 1/2$, we see from (\ref{eq:betaxright}) that $d_xf^L_{i,k-.5}(x)$ and $d_xf^L_{k-.5,j}(x)$ have the same sign as $d_xQ_\pm(x)$ as required.

Now for $x \in [-7/8,1/2] \cup [3/4,7/8]$, we have
\[
d_xf^L_{i,k-.5} = d_xf^L_{i,k-1} + \delta d_x f_{1,2}^{PV}(x), \quad \mbox{and} \quad 
d_xf^L_{k-.5,j} = d_xf^L_{k-1,j} - \delta d_x f_{1,2}^{PV}(x).
\]
Thus, as long as $\delta>0$ is small enough so that for all $i< k-1<j$, 
\[
\delta \, || d_xf^{PV}_{1,2}||_{C^0} <  \inf_{x \in [-7/8,1/2] \cup [3/4,7/8]} \{ \, |d_xf^L_{i,k-1}(x)|\, ,\, |d_xf^L_{k-1,j}(x)|\,\},
\]
it follows that $d_xf^L_{i,k-.5}$ and $d_x f^L_{k-.5,j}$ have the same sign as $d_xf^L_{i,k-1}$ and $d_x f^L_{k-1,j}$.  [The infimum is positive since sheet $S_{k-1}$ does not meet other sheets at crossings or cusps, so that $f^L_{i,k-1}$ and $f^L_{k-1,j}$ have no critical points in $[-7/8,1/2] \cup [3/4,7/8]$ by Proposition \ref{prop:PV1Cr2Cr} (2).] 

\smallskip

Item (3):  For $x \in [-1,1/4]$ we have 
\[
|f^L_{i,k-.5}(x)|  \leq |f^L_{i,k-1}(x)| + ||f^L_{k-1,k-.5}||_{C^0} \leq (k-1-i)\ea + 2\ec \leq  N \ea.
\]
[We used Lemma \ref{lem:c0deltae1} and Proposition \ref{prop:PV1Cr2Cr} (3).]  A similar argument applies for $f^L_{k-.5,j}$.

\smallskip

Item (5):  For $x \in [-1/4, 1/4]$, $d_xf^L_{i,k-.5} = d_xf^L_{i,k-1} + \delta d_xf^{PV}_{1,2}$.  Both terms are constant, and we have 
\[
|d_xf^L_{i,k-.5} - 2 C_{i,k-.5}| \leq |d_xf^L_{i,k-1} - 2y^L_{i,k-1}| + \delta |d_xf^{PV}_{1,2}| < (k-1-i)\ea + \delta(2 +\ea) \leq  N \ea   
\]
where the 2nd inequality used Proposition \ref{prop:PV1Cr2Cr} (5) and the 3rd requires us to take $\delta < \ea/4$. 

\smallskip

Item (1):  To verify $f^L_{i,k-.5} >0$ and $f^L_{k-.5,j} >0$, first check that the inequality holds at the endpoints of the domain of definition of these functions.  [This follows from item (4) in all cases except when edge $L$ is a (Cu) edge.  In this case, we use that from Lemma \ref{lem:c0deltae1} $f^L_{k-.5}(-3/8)$ is within $2\ec$ from $f^L_{k-1}(-3/8)$ and this is much smaller than the difference in values between the functions that define the cusp sheets and either of the adjacent sheets--in (\ref{eq:08ea}) that distance is computed to be $.12\ea$.]  Then, (2) shows that these functions remain positive on $[-1,1/2]$ and $[3/4,1]$.  In all the cases where $C_{i,k-.5}$ (resp. $C_{k-.5,j}$) is non-zero the $C^0$ bound from (2) shows that $f^L_{i,k-.5}$ (resp. $f^L_{k-.5,j}$) cannot vanish in $[1/2,3/4]$ and hence must remain positive there.  The only case where this argument does not apply is when $C_{i,k-.5} = C_{k-1,k-.5}=0$.  Then, $d_x f^L_{k-1,k-.5} = \delta d_x f_{1,2}^{PV}$ in $[1/2,3/4]$ which has a unique critical point in $[1/2,3/4]$ that is local max, and since $f^L_{k-1,k-.5}$ has already been shown to be positive at $1/2$ and $3/4$ it follows that positivity must hold on $[1/2,3/4]$ as well.

\medskip

\noindent {\bf The Type (12) square:}  

Recall that during the proof of Proposition  \ref{prop:PV1Cr2Cr}, for the (1Cr) edge types the difference functions for the crossing sheets $f_{k,k+1}$ were constructed in a manner that is independent of $n$ and $k$ on $[-1, 1/4]$.  Denote the restriction of these difference functions to $[-1, 1/4]$ by $f^{(1Cr)}$.  Recall that $f^{(1Cr)}$ has a unique critical point in $(-1,1/4)$ that is a local minimum located at some $\tilde{\eta} \in (-3/4,-1/2)$. 

Choose $\alpha$ with $\tilde{\eta}< \alpha< -1/2$ and to be close enough to $\tilde{\eta}$ so that
\begin{equation} \label{eq:1CralphaUp}
.5 f^{(1Cr)}(\alpha) < .5 f^{(1Cr)}(-3/4) = -.03 \ea
\end{equation}
and 
\begin{equation} \label{eq:1CralphaUpderiv}
0< .5 d_xf^{(1Cr)}(x) < d_x f^L_{k-1,k}(x), \quad \mbox{for $x \in (\tilde{\eta}, \alpha]$.}
\end{equation}

Note also that since $-g(x) \leq f^{(1Cr)}$ holds in $[-1,-1/2]$ with $g$ increasing (here $g$ is the function fixed in the proof of Proposition \ref{prop:PV1Cr2Cr}, specifically (G1)-(G4)), we have the lower bound
\begin{equation}  \label{eq:1CralphaLower}
 -.035 \ea = -.5 g(-1/2)  < -.5 g(\alpha) \leq .5 f^{(1Cr)}(\alpha).
\end{equation}
Recall $h_i$ defined in Proposition \ref{prop:1234def}. Define 
\[
f^L_{k-.5,k}(x) = \left\{ \begin{array}{lc}  .5 f^{(1Cr)}(x), &  x \in [-1,\alpha], \\
-.02 \ea + \ea(x + 1/4), & x \in [-1/4,1/4], \\
.5 \ea + .5 h_{k-1,k},  & x \in[1/2,3/4], \\
.5 Q_+(x), & x \in[7/8,1],
\end{array} \right.
\]
and then extend $f^L_{k-.5,k}$ to satisfy
\begin{equation} \label{eq:type123478}
\begin{array}{ll}
0 < d_x f^L_{k-.5,k}(x) < d_x f^L_{k-1,k}(x),  & \quad x \in (\tilde{\eta}, 1/2], \\
d_x f^L_{k-1,k}(x) < d_x f^L_{k-.5,k}(x) < 0,  & \quad x \in [3/4,1).
\end{array}
\end{equation}
To verify that such an extension is possible using Lemma \ref{lem:thetechlem}, note that the $L$ edge is (PV), so $f^L_{k-1,k}$ is defined as in (\ref{eq:initialdef}). Explicit comparison of $d_x f_{k-.5,k}$ with $d_x f_{k-1,k}$ in $[-1/4,1/4]$ and $[7/8,1]$ together with the inequality (\ref{eq:1CralphaUpderiv}) in $(-\tilde{\eta}, \alpha]$, show that the inequalities already hold in these regions.  %Moreover, for $x \in (\tilde{\eta}, \alpha]$, (\ref{eq:1CralphaUpderiv}) gives that
%$0< d_x f_{k-.5,k} < d_x f^L_{k-1,k}.
 Thus, existence of such an extension follows from a check of the inequalities
\[
0 < f^L_{k-.5,k}(-1/4) - f^L_{k-.5,k}(\alpha) < f^L_{k-1,k}(-1/4) - f^L_{k-1,k}(\alpha), 
\]
\[
0 < f^L_{k-.5,k}(1/2) - f^L_{k-.5,k}(1/4) < f^L_{k-1,k}(1/2) - f^L_{k-1,k}(1/4), 
\]
\[
f^L_{k-1,k}(7/8) - f^L_{k-1,k}(3/4) < f^L_{k-.5,k}(7/8) - f^L_{k-.5,k}(3/4) < 0,
\]
by consulting the following table of values obtained from (\ref{eq:1CralphaUp}) and (\ref{eq:1CralphaLower}) and the definitions of $f^L_{k-1,k}$ and $f^L_{k-.5,k}$.  
\begin{center}
\begin{tabular}{c | c | c}
$x$ & $f^L_{k-.5,k}(x)$ & $f^L_{k-1,k}$  \\ \hline
$\alpha$ & $-.035 \ea < f^L_{k-.5,k}(\alpha) < -.03\ea$ & $.06 \ea < f^L_{k-1,k}(\alpha) < .07 \ea$ \\
$-1/4$ & $ -.02 \ea $ & $ .1 \ea$ \\
$1/4$ & $.48 \ea $ & $1+.1\ea$ \\
$1/2$ & $.5 \ea$ & $1 + (2/3) \ea$ \\
$3/4$ & $.5 \ea$ & $1 + (2/3) \ea$ \\
$7/8$ & $.5 (65/64) \ec$ & $ (65/64) \ec$ 
\end{tabular}
\end{center}
With $f^L_{k-.5,k}$ defined we set $f^L_{k-.5} = f^L_{k-.5,k} + f^L_{k}$.   

To conclude the proof we check items (1)-(5), taking $C_{i,k-.5} = y_{i,k}$ and $C_{k-.5,j} = y_{k,j}$.

\smallskip

Item (2):  In $[1/4,3/4]$, $f^L_{k-.5}(x)$ is within $.5\ea + .5 || h_{k,k+1}||_{C^0([1/2,3/4])}< \ea$ of $f^L_{k}$ (see Proposition \ref{prop:1234def} (2)), and it easily follows that $|f^L_{i,k-.5}(x) -C_{i,k-.5}|$ and $|f^L_{k-.5,j}(x) -C_{k-.5,j}|$ satisfy the required bound.

The inequalities (\ref{eq:type123478}), imply
\[
 d_xf^L_{k-.5,k} >0 , d_xf^L_{k-1,k-.5} >0  \quad \mbox{on $(\tilde{\eta},1/2]$, and} \quad d_xf^L_{k-.5,k} <0 , d_xf^L_{k-1,k-.5} <0  \quad \mbox{on $[3/4,1)$,}
\]
and it follows (using Proposition \ref{prop:PV1Cr2Cr} (2)) that for any $i \leq k-1$ and $k \leq j$, the required inequalites hold for $d_xf^L_{i,k-.5}= d_xf^L_{i,k-1} + d_xf^L_{k-1,k-.5}$ and $d_xf^L_{k-.5,j}= d_xf^L_{k,j}+ d_xf^L_{k-.5,k}$ on these intervals.  
Finally (using properties of $f^{(1Cr)}$), the unique critical point for $f^L_{k-.5,k}$ in $(-1,1/2]$ is the local minimum $\tilde{\eta}$, and for $x \in (-1, \tilde{\eta}]$ we can estimate
\[
d_xf^L_{k-.5, k+1}(x) = d_xf^L_{k-.5, k}(x) + d_xf^L_{k,k+1}(x) \geq -.5 d_xg(x) + d_xg(x) > 0,
\]
and (using that $d_xf^L_{k-.5,k}$ is non-positive in $(-1,\tilde{\eta})$)
\[
d_x f^L_{k-1,k-.5}(x) = d_xf^L_{k-1,k}(x) - d_x f^L_{k-.5,k}(x) > 0.
\]
It follows that $d_xf^L_{i,k-.5}(x) >0$ and $d_xf^L_{k-.5,j}$ holds in $(-1,1/2]$ for any $i\leq k-1$, $k+1\leq j$.

\smallskip

Item (1):  The unique non-degenerate $0$ of $f^L_{k-.5,k}$ in $(-1/4,0)$ is clear from the construction.  Positivity of remaining $f^L_{i,k-.5}$ and $f^L_{k-.5,j}$ follows as usual from item (2) and from checking positivity at endpoints.

\smallskip

Item (3):  Follows, since $|f^L_{k-.5,k}(x)| < \ea$ in $[-1,-1/4]$.

\smallskip

Item (4):  Above the edge $U$ (resp. $D$) the cusping sheets $S_k$ and $S_{k+1}$ are directly above (resp. below) $S_{k+2}$.  Therefore, we have 
\[
\sigma_+(k-.5) = k-.5, \quad \mbox{and} \quad \sigma_-(k-.5) = k+.5.
\]
For $|x-(-1)| < 1/8$, we have
\[
f^L_{k-.5,k}(x) = .5 f^{(1Cr)}(x) = -.5 Q_{-}(x),
\]
and for $|x-1| < 1/8$ we have $f^L_{k-.5,k}(x) = .5 Q_+(x)$.  Thus, the result follows.

\smallskip

Item (5):  Note that $f^L_{k-.5,k}$ is linear with slope $\ea$ in $[-1/4,1/4]$.  Since $C_{i,k-.5} = y_{i,k}$ and $C_{k-.5,j} = y_{k,j}$, the result follows easily.

\end{proof}

Given any (1)-(12) square type, recall we impose the global labeling of sheets over the square $[-1,1]_{x_1} \times [-1,1]_{x_2}$ consistent with the labeling of sheets over the top right corner $\{+1\} \times \{+1\}$, and use $f^U_i, f^R_i, f^D_i, f^L_i$ for functions associated to the  edge types along the boundary.

%Recall in Remark \ref{rem:betanotate} that 
Let $\sigma_L(i)$ denote the  ordering of $S_i$ as it appears over the up left corner 
$\{-1\} \times \{+1\},$ and $\sigma_D(i)$ denotes its ordering as it appears over the down right corner 
$\{+1\} \times \{-1\}$.  If sheets $S_k$ and $S_{k+1}$ meet at a cusp edge above $U$ (resp. $R$) then we set $\sigma_L(k) = \sigma_L(k+1) = k-.5$  (resp. $\sigma_D(k) = \sigma_D(k+1) = k-.5$). 
%Let
%$ f_i^U,  f_i^R,  f_{\sigma_D(i)}^D$ and $f_{\sigma_L(i)}^L$ denote the functions associated to  the edges
%$[-1,1] \times \{1\}, \{1\} \times [-1,1],  [-1,1] \times \{- 1\},$ and $\{-1\} \times [-1,1],$
%respectively.  
%\footnote{
%\dr{Clarify meaning of $f_{i}^U$, etc.}
%\ms{8/6/15: what do you mean? state that U stands for Up, etc?}
%}

Let $\phi:[-1,1] \rightarrow [0,1]$ be a smooth non-decreasing function such that:
\begin{eqnarray}
\label{eq:InterpolatePhi}
&\mbox{for}\,\, -1/4 + \ea < x < 1/4 - \ea, \,\,\,
\phi(x) = 2(x+1/4);&\\
\notag
&\mbox{for}\,\,-1/4 < x < 1/4, \,\,\,0< d_x\phi(x)<2 + \ec;&\\
\notag
&\phi^{-1}(0)= [-1, -1/4]; \,\,\phi^{-1}(1)= [1/4,1];&\\
\notag
&\mbox{for} \,\, 1/4-\ea \leq x \leq 1/4,\,\,\, \phi(x) \geq 2(x+1/4);& \\
\notag
&  \,\,\, \mbox{and} \,\, \phi(x) -1/2 \,\, \mbox{is an odd function}.&
\end{eqnarray}

\begin{figure}
\labellist
\small
\pinlabel $y$ [l] at 120 200
\pinlabel $x$ [l] at 232 22
\pinlabel $y$ [l] at 440 200
\pinlabel $x$ [l] at 552 22
\pinlabel $1$ [r] at 98 184
\pinlabel $-1/4$ [u] at 32 8
\pinlabel $1/4$ [u] at 192 8
\pinlabel $y=\phi(x)$ [l] at 230 184
\pinlabel $y=\psi(x)$ [l] at 474 84
\endlabellist
\centerline{ \includegraphics[scale=.6]{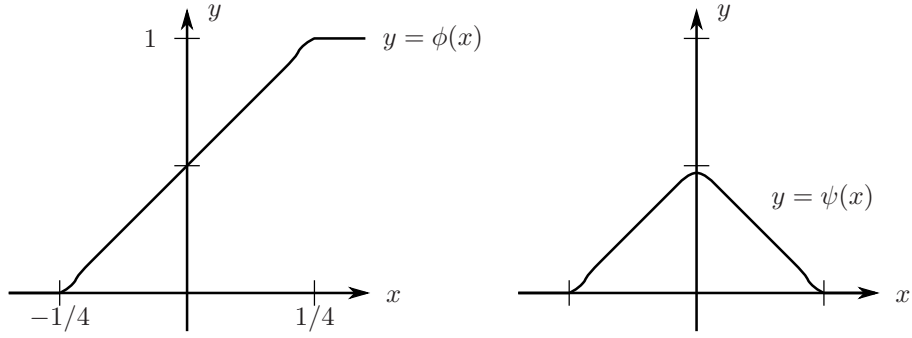} }
%\centerline{ \includegraphics[scale=.8]{images/SubdivideSq} }
\caption{Schematic graphs of $\phi$ and $\psi$ (defined in Section \ref{sec:funcO2}).  The smoothings at $-1/4$ and $1/4$ are actually carried out in the very small intervals $[-1/4, -1/4+\ea]$ and $[1/4-\ea,1/4]$.}
\label{fig:phi}
\end{figure}

See Figure \ref{fig:phi}. 
Define
\begin{eqnarray}
\label{eq:interpolating}
\notag
F_i(x_1,x_2) &= &\phi(x_2) f_i^U(x_1)+ (1-\phi(x_2)) f_{\sigma_D(i)}^D(x_1)\\
\notag
&+&\phi(x_1)f_i^R(x_2)+(1-\phi(x_1)) f_{\sigma_L(i)}^L(x_2). \\
F_{i,j} = F_i - F_j &= &\phi(x_2) f_{i,j}^U(x_1)+ (1-\phi(x_2)) f_{\sigma_D(i), \sigma_D(j)}^D(x_1)\\
\notag
&+&\phi(x_1)f_{i,j}^R(x_2)+(1-\phi(x_1)) f_{\sigma_L(i), \sigma_L(j)}^L(x_2).
\end{eqnarray}

%%%%%%%%%%%%%%%%%%%%%%%%%%%%%%%%%%%%%%%%%%%%%%%%%%%
%%%%%%%%%%%%%%%%%%%%%%%%%%%%%%%%%%%%%%%%%%%%%%%%%%%
%%%%%%%%%%%%%%%%%%%%%%%%%%%%%%%%%%%%%%%%%%%%%%%%%%%

\section{Proof of Theorem \ref{thm:PropertiesofLtilde} Part 2: Construction of squares with swallowtails}
\label{sec:ConstructionsST}

For the construction of the (13) and (14) square types, we fix concentric, closed disks $O_1$ and $O_2$ centered at $(-3/8, 0)$ with radii $R_1 = 1/16$ and $R_2 = 3/32$.  The construction then has a different character within each of the $3$ regions $O_1$, $O_2\setminus O_1$, and $[-1,1]^2 \setminus O_2$.    Within $O_1$ a suitably scaled and shifted version of a standard coordinate model for the swallowtail is used;  see Section \ref{sec:funcO1} below.  Outside of the disk $O_2$  defining functions of a
%built out of $1$-dimensional restrictions 
similar nature to those  of squares (1)-(12) are employed;  see Section \ref{sec:funcO2}.  
%To complete the construction, an interpolation is carried out in the annulus between $O_1$ and $O_2$.  
Section \ref{ssec:FormulaST} combines the models in a single formula %(\ref{eq:FHCAG}) 
that carries out an interpolation in the annulus between $O_1$ and $O_2$.

As in Section \ref{sec:SwallowComp} we consider only the case of a Type (13) (upward) swallowtail in detail.  A similar construction applies for the (14) square types that contain downward swallotail points.   

%\subsection{Upward and downward swallowtails}
%\dr{Recall from ??? that there are two types of swallowtail points that we have referred to as upward and downward swallowtails;  they are distinguished by the location of the third sheet of the swallowtail which respectively appears above or below the two sheets that cross.  To simplify exposition, we consider only the case of an upward swallowtail in detail.  An analogous construction and computation holds for a downward swallowtail.  (See Remark ??? for further discussion\footnote{Alternatively, delete this sentence.} of the downward swallowtail.)}

%\dr{THE PREVIOUS PARAGRAPH IS NOW MOVED TO THE START OF SECTION 9. Probably just start by recalling that we carry out our computation/construction for upward swallowtails only.}

%Therefore, in the following, we will construct a Legendrian above $[-1,1]\times [-1,1]$ whose front projection represents the  upward swallowtail as in a Type (13) square with sheets matching the required standard form (see Section \ref{ssec:Fitting}) near the boundary of the square.  

%We then conclude with a computation of Legendrian contact homology as stated in Theorem ???.   

We maintain Convention \ref{conv:stlabeling}, and label the sheets of a (13) square type as $S_1, \ldots, S_n, \widetilde{S}_k$ where the sheets $S_{k}, S_{k+1}, S_{k+2}, \widetilde{S}_k$ all correspond to portions of the sheets that meet at the swallowtail point.  At $(+1,+1)$, sheet $S_k$ is the top sheet of the swallow tail, and sheets $S_{k+1}$ and $S_{k+2}$ cross to the right of the swallow tail point; $\widetilde{S}_k$ lies above the closed subset of $[-1,1]\times[-1,1]$ that sits to the left of the cusp locus.

%\subsubsection{Convention for labeling sheets}
%For the Type (13) square, it is awkward to use a global labeling for sheets since the two sheets that cross  correspond to the same sheet on the other side of the cusp locus.  Correspondingly, when refering to sheets numerically, we are careful to state which portion of the square is under consideration.  Typically, to the right (resp. left) of the cusp locus we label sheets as they appear at $(+1,+1)$ (resp. $(-1,+1)$).

%Let $n$ denote the number of sheets at $(+1,+1)$, so that there are $n-2$ sheets at $(-1,-1)$.    
%Furthermore, we suppose that at $(+1,+1)$ the three swallow tail sheets are $k, k+1, k+2$ where sheet $k$ is the top of the swallow tail, and sheets $k+1$ and $k+2$ cross to the right of the swallow tail.  

\subsection{Coordinate model near a swallowtail point}    \label{sec:funcO1}    We begin by giving an explicit coordinate description of the front projection of a Legendrian containing a single swallowtail point.  Although only sheets involved with the swallowtail point are present in this model, we will still refer to them as sheets $S_k, S_{k+1}$, $S_{k+2}$, and $\widetilde{S}_k$ since we will later integrate this coordinate model with the remaining sheets of a Type (13) square.

%\footnote{The generating family for the swallowtail is $F(x_1,x_2; e) = e^4-x_1 e^2+x_2e$, where I have replaced $x_1$ in the reference with $-x_1$ here so that the swallow tail opens to the right.} 

The following description of the swallowtail is taken from\footnote{We have replaced $x_1$ in the reference with $-x_1$ here so that the swallow tail opens to the right rather than to the left.}\cite[p.47]{ArnoldGuseinZadeVarchenko}.
 %pg. 47 of ``Singularities of Differentiable Maps, Volume I''.  
 Consider the function
\[
F: \R^2 \times \R \rightarrow \R, \quad F(x_1,x_2; e) = e^4-x_1 e^2+x_2e
\]
which we view as a $2$-parameter family of functions on $\R$ denoted $f_{x_1,x_2}(e) = F(x_1,x_2; e)$.  The front projection of a Legendrian $L_{st} \subset J^1(\R^2)$ with a single swallowtail point at $(0,0)$ is given by
\[
\pi_{xz}(L_{st}) = \{  \left(x_1,x_2, f_{x_1,x_2}(e)\right) \, |\, f_{x_1,x_2}'(e)=0\},
\]
i.e.  it is the collection of critical values of the map $\R^3\rightarrow \R^3$ given by $(x_1,x_2,e) \mapsto (x_1,x_2, f_{x_1,x_2}(e))$.

\begin{remark}
The function $F$ is called a generating family for the Legendrian $L_{st}$.  The $y_1$ and $y_2$-coordinates of $L_{st}$ are given by $y_i = \partial F/\partial x_i(x_1,x_2;e)$.
\end{remark}

The number of sheets of %$\pi_{xz}(L_{st})$ 
$L_{st}$ above a point $(x_1,x_2)$ is determined by the number of real roots of 
\[
f_{x_1,x_2}'(e)= 4 e^3 - 2 x_1 e + x_2.
\]  
The projection of $L_{st}$ to the base $\R^2$ is $3$ sheeted where $f_{x_1,x_2}'$ has $3$ distinct real roots, and $1$-to-$1$ where $f_{x_1,x_2}'$ has a real root and $2$ distinct complex conjugate roots.  
Cusp edge points $(x_1,x_2)$ correspond to those points  $\left(x_1,x_2, f_{x_1,x_2}(e)\right)$ where $e$ is a multiplicity $2$ root of $f_{x_1,x_2}'$.  The swallowtail point itself is the unique value of $(x_1,x_2)$ where $f_{x_1,x_2}'$ has a root of multiplicity $3$.

\begin{proposition} \label{prop:stCompute}
Let $S$ denote the closure of the portion of $\pi_{xz}(L_{st})$ that projects in a $3$-to-$1$ manner to the $x_1x_2$-plane.  The formulas
\begin{align}  \label{eq:stParam}
x_1 = 2[r^2+s^2+rs] \\
x_2 = 4 r s [r+s] \nonumber \\
z = -r^4 +2r^3s + 2 r^2s^2 \nonumber
\end{align}
parametrize $S$ in a manner  that  is $2$-to-$1$ except along $\partial S$ where the mapping is $1$-to-$1$, provided that we view $S$ as the subset $\pi_{xz}^{-1}(S) \subset L_{st}$.
%\footnote{Minor point here.  Along the crossing locus the parametrization OF THE FRONT PROJECTION is 4 to 1.  Note that the crossing locus is $z(r,s) = z(s, r)$ union $z(r,s) = z(-s-r,-s)$...probably.}  
The inverse image of the cusp edge is
\[
\{s=r\}\cup\{s=-2r\};
\]
the inverse image of $\partial S$ is 
\[
\{s=-\frac{1}{2}r\}
\]
with the swallow tail point itself at $(r,s)=(0,0)$.

 In a neighborhood of a point not belonging to the cusp edge, $S$ is given by the graph of a locally defined function $z = z(x_1,x_2)$.  The Euclidean gradients $\nabla z = (\partial_{x_1}z,\partial_{x_2}z)$ are given by
\begin{equation} \label{eq:nablaz}
\nabla z(r,s) = \frac{J(r,s)}{J(r,s)}(-r^2, r)
\end{equation}
where $J(r,s) = 8(r-s)(2r+s)(r+2s).$
%\footnote{\ms{2/18/15: Why write $J/J$ instead of 1? Also done in proof at chain rule application.}\dr{10/12/15:  The inverse function theorem does not imply when $J =0$, so the first sentence fails, and I think it is appropriate to leave the $J$ in the denominator to indicate these values do not belong to the domain of $\nabla z$.  Of course, the continuous extension of the gradient field to the cusp edge is deduced from cancelling $J/J = 1$.}}
\end{proposition}

\begin{proof}
As discussed in the previous paragraph, $S$ is the subset of $L_{st}$ for which $f_{x_1,x_2}'$ has $3$ real roots (with possible repetitions).  For such values of $(x_1,x_2)$, we can write
\begin{equation}  \label{eq:fRoots}
f_{x_1,x_2}'(e)= 4 e^3 - 2 x_1 e + x_2 e = 4(e-r_1)(e-r_2)(e-r_3).
\end{equation}
Putting $r_1=r$ and $r_2=s$, the vanishing of the quadratic term in $f_{x_1,x_2}'$ forces $r_3 = -r-s$.  Making these substitutions and expanding the right hand side gives
\[
4 e^3 - 2 x_1 e + x_2 e = 4 e^3 - 4[r^2+s^2+rs]e +4[rs(r+s)].
\]
Equating coefficients, and setting $z= f_{x_1,x_2}(r)$ gives the parametrization (\ref{eq:stParam}).

Next we check that the parametrization is $2$-to-$1$ except along $\partial S$.   Note that replacing $r$ and $s$ with any $2$ of the $3$ roots, $r, s, -r-s$, of $f'_{x_1,x_2}$ will lead to the same product in (\ref{eq:fRoots}).  Thus, substituting any of 
\begin{equation} \label{eq:Symmetries}
(r,s), (r, -r-s), (s,r), (s,-r-s), (-r-s, r), (-r-s,s)
\end{equation}
 for $(r,s)$ in the parametrization leads to a point on the swallowtail with the same $x_1$ and $x_2$ coordinates as $(x_1,x_2,z)(r,s)$.  [Moreover, only these substitutions have this property, since a polynomial with leading coefficient $4$ is uniquely determined by its roots.]  By construction, we will have 
\begin{equation}  \label{eq:symmetry}
z(r,s)=z(r,-r-s); \quad z(s,r) = z(s,-r-s); \quad z(-r-s,r) = z(-r-s,s),
\end{equation}
and assuming the three roots of $f'_{x_1,x_2}$ are distinct and take distinct values when $f_{x_1,x_2}$ is applied, these $3$-values will all be distinct.  This shows that for values of $x_1,x_2$ in the complement of the cusp and crossing locus each of the $3$ sheets above $(x_1,x_2)$ has $2$ preimages in the $(r,s)$-plane.     
%\dr{[EDIT] The crossing locus is given by ???, note that these $2$ sheets are distinct in $L_{st}$ since...}

%[When the $3$ roots are distinct, but the value $f_{x_1,x_2}(r) = f_{x_1,x_2}(s)$ or $f_{x_1,x_2}(r) = f_{}$]  

Next, note that $f'_{x_1,x_2}$ has a multiplicity $2$ root when $r=s$, $r=-r-s$, or $s=-r-s$.  As claimed, the first two lines both parametrize the cusp edge (since $r$ is the multiplicity $2$ root).  The third line parametrizes $\partial S$ (since $r$ is the other root), and the equality of $s$ and $-r-s$ gives a $1$-to-$1$ parametrization of $\partial S$. 

To compute $\nabla z = (\partial_{x_1}z,\partial_{x_2}z)$, we use the inverse function theorem.  The differential of $(r,s) \mapsto (x_1,x_2)$ is given in matrix form by
\[
\left[ \begin{array}{cc} \displaystyle  \frac{\partial x_1}{\partial r} &  \displaystyle  \frac{\partial x_1}{\partial s} \\  \displaystyle   \frac{\partial x_2}{\partial r} &  \displaystyle   \frac{\partial x_2}{\partial s} \end{array} \right] = \left[ \begin{array}{cc} 2s+ 4r & 2 r + 4s \\ 8 r s + 4 s^2 &  8 r s + 4 r^2 \end{array} \right].
\]
When the determinant $J(r,s) = 8(r-s)(2r+s)(r+2s)$ is non-zero, we can locally write $(r,s) = (r(x_1,x_2), s(x_1,x_2))$ and compute
\[
\left[ \begin{array}{cc} \displaystyle \frac{\partial r}{\partial x_1} & \displaystyle \frac{\partial r}{\partial x_2} \\ \displaystyle \frac{\partial s}{\partial x_1} & \displaystyle \frac{\partial s}{\partial x_2} \end{array} \right]= \frac{1}{J} \left[ \begin{array}{cc} 8 r s + 4 r^2 & -[2r+4s] \\ -[8 r s + 4 s^2] & 2 s + 4 r \end{array} \right].
\]
Finally, the chain rule gives
\[
\partial_{x_1}z = \frac{\partial z}{\partial r}\frac{\partial r}{\partial x_1}+ \frac{\partial z}{\partial s}\frac{\partial s}{\partial x_1}= \frac{J}{J}\cdot(-r^2),
\]
\[
\partial_{x_2}z  = \frac{\partial z}{\partial r}\frac{\partial r}{\partial x_2}+ \frac{\partial z}{\partial s}\frac{\partial s}{\partial x_2}= \frac{J}{J}\cdot (r).
\]

%The Legendrian swallowtail is the image of the fiber critical set of $F$ which is given by $0 = 4e^3-2 x_1 e+x_2$.  Here, we can think of $x_1$ and $x_2$ as parametrizing the space of cubic polynomials with leading coefficient $4$ and no quadratic term, while $e$ takes on the values of roots.  The subset of this space corresponding to those $x_1$ and $x_2$ where the cubic has all real roots is naturally parametrized by taking $-r$ and $-s$ to be two of the roots.  (The third root is forced to be $r+s$ by the vanishing of the quadratic term.)  That is where the parametrization comes from.  The rest is just computing.

\end{proof}

Note that it follows from Proposition \ref{prop:stCompute} that the projection of the cusp edge to the $(x_1,x_2)$ plane is the subset
\begin{equation} \label{eq:6r28r3}
\Sigma_{\mathit{st}}= \{(6r^2, 8r^3) \, |\, r \in \R\} = \{(x_1,x_2) \, |\, x_1 = \frac{3}{2} (x_2)^{\frac{2}{3}}\}.
\end{equation}
Let 
\[
R_+= \{(x_1,x_2) \, |\, x_1 \geq \frac{3}{2} (x_2)^{\frac{2}{3}}\}  \quad \mbox{and} \quad R_-=\{(x_1,x_2) \, |\, x_1 \leq \frac{3}{2} (x_2)^{\frac{2}{3}}\}
\]
 denote the closed regions to the right and left of $\Sigma_{\mathit{st}}$.  Figure \ref{fig:STParam} illustrates some features of the parametrization (\ref{eq:stParam}).

\begin{figure}
\labellist
\small
\pinlabel $r$ [l] at 196 96
\pinlabel $s$ [b] at 96 194
\pinlabel $W$ [b] at 136 194
\pinlabel $(x_1,x_2)$ [b] at 246 100
\pinlabel $\Sigma_{\mathit{st}}$ [r] at 396 176
\pinlabel $R_+$ [c] at 432 140
\pinlabel $R_-$ [c] at 350 140
\endlabellist
\centerline{ \includegraphics[scale=.6]{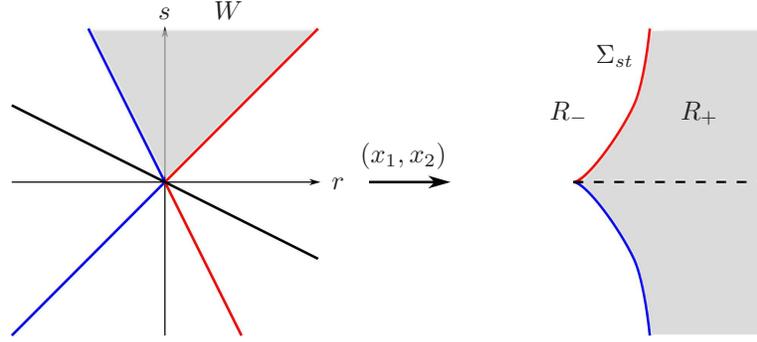} }
%\centerline{ \includegraphics[scale=.8]{images/SubdivideSq} }
\caption{The parametrization (1) with the base projection of $S$ pictured (schematically).  The pictured lines, $\{r=s\}, \{2r+s=0\}$, and $\{r+2s=0\}$ each project to the cusp locus $\Sigma_{\mathit{st}}$ in the $(x_1,x_2)$-plane.  The first two lines parametrize the cusp edge in $S$ with the direction of parametrization indicated by coloring.   The $3$ lines divide the $rs$-plane into $6$ wedges with the wedge $W$ shaded.  Any of these (closed) wedges homeomorphically parametrizes the base projection of $S$ which is the shaded region $R_+$ in the $x_1x_2$-plane.}  
\label{fig:STParam}
\end{figure}

\begin{lemma} \label{lem:crossingx2}
The crossing locus of $L_{st}$ (in the base projection) is the ray consisting of the positive $x_1$-axis.  The two arcs in $L_{st}$ that cross one another in the front projection are the image of $\{s=0\}\cup \{s+r =0\}$ under the parametrization (\ref{eq:stParam}).
\end{lemma}

\begin{proof}
The wedge $W = \{s \geq -2r\}\cap \{s\geq r\}$ homeomorphically parametrizes the base projection $\pi_x(S)$ (using $x_1 = x_1(r,s)$ and $x_2= x_2(r,s)$ from (\ref{eq:stParam})).  Indeed, those points $(u,v)$ such that $(x_1,x_2)(u,v)=(x_1,x_2)(r,s)$ were listed in (\ref{eq:Symmetries}), and for any $(r,s)$ the list provides exactly one point in $W$.  In fancier language, $W$ is a fundamental domain for the action of the group of invariants of $(x_1,x_2)$ on $\R^2$.

Following the discussion surrounding (\ref{eq:symmetry}), the $3$ sheets, $S_k,S_{k+1}$, and $S_{k+2}$ above $R_+$ are homeomorphically parametrized by
\[
W \stackrel{\cong}{\rightarrow} S_l, \quad (r,s) \mapsto (x_1(r,s), x_2(r,s), z_l(r,s)), 
\]
for $l = k,k+1,k+2$ where
\[
z_k(r,s) = z(r,s), \quad z_{k+1}(r,s) = z(s,r), \quad z_{k+2}(r,s) = z(-r-s,s)
\]
with $z(r,s)$ as in (\ref{eq:stParam}).  

A computation shows
\[
\begin{array}{ccr}
z_{k}-z_{k+1} & = & (s+r)(s-r)^3; \\
z_{k}-z_{k+2} & = & s(2r+s)^3; \\
z_{k+1}-z_{k+2} & = & r(r+2s)^3.
\end{array}
\]
It follows that for any $(r,s) \in W$, and hence any $(x_1,x_2) \in R_+$, $z_k(r,s) \geq z_{k+1}(r,s)$ (resp. $z_k(r,s) \geq z_{k+2}(r,s)$) with equality only when $s=r$ (resp. $s=-2r$) which is the upper (resp. lower) half of the cusp edge of $L_{st}$.  In addition, 
\begin{equation} \label{eq:zk1zk2}
\sgn(z_{k+1}-z_{k+2}) = \sgn(r) = \sgn(x_2),
\end{equation}
 so as claimed the crossing locus of $L_{st}$ is $\{x_2=0\} \cap R_+$.     For the statement about the preimage of the crossing sheets, note that $0=x_2 = 4 r s [r+s]$ only when $r=0$, $s=0$, or $r+s=0$.  The line $r=0$ parametrizes a portion of $S_k$ while the other two lines parametrize the crossing arc.
\end{proof}

\subsubsection{Local defining functions}

  Let $a_{k}, a_{k+1}, a_{k+2}$ denote defining functions with domain $R_+$ whose graphs are the $3$ swallowtail sheets, $S_k, S_{k+1}$, and $S_{k+2}$,  above $R_+$ so that $a_{k}(x_1,x_2) \geq a_{k+1}(x_1,x_2) \geq a_{k+2}(x_1,x_2)$ when $x_2 >0$ and $a_{k}(x_1,x_2) \geq a_{k+2}(x_1,x_2) \geq a_{k+1}(x_1,x_2)$ when $x_2 <0$.  In addition, let $b_k$ denote the defining function in $R_-$ for $\widetilde{S}_k$.   Note that along $\Sigma_{\mathit{st}} \cap \{x_2 \geq 0\}$ (resp. $\Sigma_{\mathit{st}} \cap \{x_2 \leq 0\}$) $a_k = a_{k+1}$ and $b_k= a_{k+2}$  (resp. $a_k = a_{k+2}$ and $b_k = a_{k+1}$).

As a corollary of Proposition \ref{prop:stCompute}, we can compute the gradients of local difference functions,
\[ \nabla a_{i,j} = \nabla( a_i - a_j), \quad \mbox{for $i, j \in \{k,k+1,k+2\}$}.  
\]
Recall that, using the first two coordinates from (\ref{eq:stParam}), the wedge  $W =\{s \geq -2r\}\cap \{s\geq r\}$ homeomorphically parametrizes $R_+$ in an orientation reversing manner.

%Using (\ref{eq:stParam}), the wedge, $W$, in the $(r,s)$-plane given by 
%\[
%W =\{s \geq -2r\}\cap \{s\geq r\}
%\] homeomorphically parametrizes the top sheet of the swallow tail, and hence projects in a homeomorphic manner to $R_+$.  This parametrization of $R_+$ is orientation reversing:  The bounding rays, $s = -2r$ and $s=r$, parametrize the lower and upper halfs of the cusp locus respectively.  In addition, the ray $\{s \geq 0\} \cap \{r= 0\}$ projects to the crossing locus.  See Figure ???. 
% Labeling sheets in the order they appear above $x_2=0$, we have that for $(r,s)$ belonging to $W$ 
\begin{corollary} \label{cor:aijGrad} For $(r,s) \in W$, we have
\begin{align}
 \nabla a_{k,k+1}(x_1(r,s),x_2(r,s))  =   (s-r) (s+r, -1)  \\
 \nabla a_{k,k+2}(x_1(r,s),x_2(r,s)) =  (2r+s) (s, 1)  \\
 \nabla a_{k+1,k+2}(x_1(r,s),x_2(r,s)) =   (r+ 2 s) (r, 1). 
\end{align}
\end{corollary}
\begin{proof}
%Using the parametrization (\ref{eq:stParam}) and the discussion surrounding equation (\ref{eq:symmetry}), 
As in the proof of Lemma \ref{lem:crossingx2},  we see that for $(r,s) \in W$, the $S_k$, $S_{k+1}$, and $S_{k+2}$ sheets of $L_{st}$ above $(x_1(r,s),x_2(r,s))$ are respectively the images of $(r,s)$, $(s,r)$,  and $(-r-s,s)$.  [That $(s,r)$ and $(-r-s,s)$ corresponds to the $S_{k+1}$ and $S_{k+2}$ sheets respectively, as ordered when $x_2 >0$, follows from (\ref{eq:zk1zk2}).]
%note that they agree with the $k$-th sheet $(r,s)$ along $W\cap \{r=s\}$ and $W\cap \{s=-2r\}$.  These two portions of $\partial W$ parametrize the upper and lower half of $\Sigma$ respectively.]
  The formulas then follow from  (\ref{eq:nablaz}).
\end{proof}

\subsubsection{Straightening the cusp edge} \label{sec:Straight} In this subsection, we make two changes to the model $L_{st}$.  Specifically, we place the swallowtail within $O_1$ and straighten the cusp locus in $O_2\setminus O_1$ according to the following specifications.  

Consider a diffeomorphism of the plane $S: \R^2 \rightarrow \R^2$ of the form $S(x_1,x_2) = (M x_1 + X, M x_2)$, i.e. the composition of a dialation and a horizontal translation.  We can (and do) choose the constants $X$ and $M$ so that after applying the contactomorphism of $J^1(\R^2)$ induced by $S$ to $L_{st}$
%\footnote{[OLD NOTE DELETED]... the cusp locus will bend (\ms{2/26/15: How? Why does it bend?}) back a bit before it becomes straight.  (And perhaps this can't be avoided if $b$ is to be non-decreasing and positive.)  \dr{10/12:  It is the application of $T$ that causes the cusp locus to bend back.}} we obtain a Legendrian, $L_{st}' \subset J^1(\R^2)$, whose front satisfies the properties:
\begin{enumerate}
\item[(P1)]  The swallowtail point is on the $x_1$-axis within $O_1$ and left of the center of $O_1$, $(-3/8,0)$.
\item[(P2)] The cusp locus intersects the points $(-3/8, \pm 1/32)$, and is contained in the strip where $-3/8-\ec \leq x_1 \leq -3/8 + \ec$ when $|x_2| \leq 1/16.$
%\footnote{\ms{2/20/15: In the proof of Proposition \ref{prop:SWPoints}, equation (\ref{eq:121eqx2}), $\epsilon_\bullet \le \epsilon_1 - \epsilon_3.$ But see no other determination of it.}}  
% $\partial O_1$ in two points where \footnote{The $\epsilon_1$ is unimportant here, so perhaps its presence will be misleading.  Could change to $1/100$ maybe?} $x_1 =3/8+1/100$.
\item[(P3)] When $|x_2| \geq 1/32$, 
the slope of the cusp locus is bounded below in magnitude by $9(2N+1.1)$. 
\item[(P4)]  We have  $M > (2 (2N+1) 480)^3$.
\end{enumerate}
These  conditions are obtainable by choosing $M$ suitably large since the slopes of the two halves of the cusp locus, $\Sigma_{\mathit{st}}$, go to $+ \infty$ and $- \infty$ as $x_1 \rightarrow +\infty$.  See Figure \ref{fig:STPlacement}.

\begin{figure}
\labellist
\small
\pinlabel $O_1$ [tl] at 54 174
\pinlabel $O_2$ [br] at 42 208
\pinlabel $b(x)$ [t] at 372 148
\endlabellist
\centerline{ \includegraphics[scale=.6]{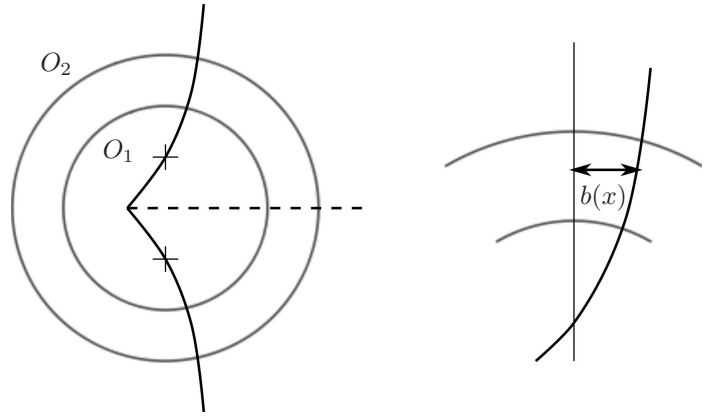} }
%\centerline{ \includegraphics[scale=.8]{images/SubdivideSq} }
\caption{(left) The placement of the cusp and crossing locus within $O_1$ after applying $S$.  The marked points are $(-3/8, \pm1/32)$.  (right) The function $b(x)$ when $x \in [1/16,+\infty)$.}
\label{fig:STPlacement}
\end{figure}

Next, we apply a further diffeomorphism $T: \R^2 \rightarrow \R^2$ to $L'_{st}$ to straighten the cusp locus within $O_2\setminus O_1$. For the purpose of defining $T$, choose a function $b: [0, +\infty) \rightarrow \R$ satisfying
\begin{itemize}
\item $0 \leq b(x)$ with $b(x) = 0$ for $x \in [0,1/32]$;
\item $b(x) \leq c(x) - (-3/8)$ for $x \in [1/32, +\infty)$ with equality when $x\geq 1/16$ where $c(x)$ is the $x_1$ coordinate of the cusp locus of $L'_{st}$ when $x=x_2$; 
\item $b$ is non-decreasing; and
\item $b'(x) < 1/(3\times (2N+1.1))$.
\end{itemize}

\begin{lemma} Such a function $b$ exists.
\end{lemma}
\begin{proof}
Let $\beta: [0,+\infty)\rightarrow \R$ be a smooth cutoff function with $\beta(x) = 0$ for $x \in[0, 1/32]$; $\beta(x) =1$ for $x \in [1/16, +\infty)$; and $ 0 \leq \beta'(x) \leq 33$ for all $x$.  Set $b(x) = \beta(x) (c(x)+ 3/8)$.  To verify the third and fourth items, compute
\begin{align*}
0 \leq b'(x) = &  \quad \beta'(x)( c(x)+3/8) + \beta(x) c'(x) \leq \\
 & 33 (1/32)(1/(9[2N+1.1])) + 1 (1/(9[2N+1.1])) < 1/(3\times (2N+1.1)).
\end{align*}
For the 1st inequality, note that for $x >0$, $c(x)$ is increasing, and $(c(x)+3/8) \geq 0$ when $x \geq 1/32$ by (P2).   For the 2nd inequality, use (P2) and (P3)
 to estimate $(c(x)+3/8)$ when $1/32 \leq x \leq 1/16$ and (P3) for $c'(x)$ when $1/32 \leq x$.
\end{proof}

Now, define $T(x_1,x_2) = (x_1 + b(|x_2|) , x_2)$.  Let $L''_{st} \subset J^1(\R^2)$ denote the Legendrian obtained from $L'_{st}$ by applying the contactomorphism induced by $T$.  
We denote the local defining functions for the front projection of $L''_{st}$ as $A_{k}, A_{k+1}, A_{k+2}$, and $B$.  They are related to the local defining functions of $L_{st}$ via 
\begin{equation}  \label{eq:Aiai}
A_i = a_{i} \circ S \circ T,  \mbox{ for $i = k,k+1,k+2$},  \quad \mbox{and} \quad B = b \circ S \circ T.
\end{equation}
The respective domains of the $A_i$ and $B$ are $T^{-1}(S^{-1}(R_+))$ and $T^{-1}(S^{-1}(R_-))$.  The cusp locus of $L''_{st}$ is 
\begin{equation} \label{eq:SigmaT1S1}
\Sigma = T^{-1}(S^{-1}(\Sigma_{st})).
\end{equation}

%We denote the local defining functions for the front projection of $L'_{st}$ as $a_{k}, a_{k+1}, a_{k+2}$, and $\tilde{a}_k$.  Here, $a_{k}, a_{k+1}, a_{k+2}$ are defined in the closure of the region of the $x_1x_2$-plane that lies to the right of the cusp locus with the subscripts denoting the ordering of the sheets when $x_2$ is above the crossing locus; $\tilde{a}_k$ is defined on the closure of the region to the left of the cusp locus.

\subsubsection{Properties of model}  \label{sec:PropertiesA}

We record and establish several properties of $L''_{st}$ for later use.  For $i,j \in \{k,k+1,k+2\}$, denote difference functions by $A_{i,j} = A_{i}-A_j$.

\begin{lemma}
\label{lem:Aprops}
The following properties of  $L''_{st}$ hold.
\begin{enumerate}
\item[A1] 
Within $O_2$, the cusp locus is contained in $\{ -3/8-\ec \leq x_1 \leq -3/8+\ec \}$, and in $O_2\setminus O_1$ it agrees with the vertical line segment $x_1 = -3/8$.  Furthermore, when $1/32 \leq |x_2| \leq 1/16$ the absolute value of the slope of the cusp locus is bounded below by $2N+1.1$.

\item[A2]
The crossing locus is the horizontal ray extending to the right from the swallowtail point.

\item[A3] For all $(x_1,x_2) \in T^{-1}(S^{-1}(R_+))$,
%\in (O_2\setminus O_1) \cap \{x_1 > -3/8\}$,
\[
\partial_{x_1} A_{k,k+1}(x_1,x_2) \geq  0;  \quad  \partial_{x_1} A_{k,k+2}(x_1,x_2) \geq 0; \quad \mbox{and  } \mbox{sgn}(\partial_{x_1} A_{k+1,k+2}(x_1,x_2)) = \mbox{sgn}(x_2).
\] 
Moreover, in the first (resp. second) inequality, we have equality only when $(x_1,x_2)$ belongs to the upper (resp. lower) half of the cusp locus.
\item[A4]

 For points $( x_1, 0)$ belonging to the crossing locus
\[
\partial_{x_2} A_{k+1,k+2}(x_1,0) \geq 0
\]
with equality only at the swallow tail point itself.

Moreover, for $(x_1,x_2) \in O_1$ with $|x_2| \leq 1/32$,
\[
\partial_{x_2} A_{k+1,k+2}(x_1,x_2) \geq 0, \quad \partial_{x_2} A_{k,k+2}(x_1,x_2) \geq 0, \quad \partial_{x_2} A_{k,k+1}(x_1,x_2) \leq 0.
\]

\item[A5]  
In $(O_2\setminus O_1) \cap \{x_2 \geq 1/32\}$,  the gradient  $\nabla A_{k,k+1}$ points into the wedge $\{x_1 \geq 0\} \cap \{x_1 \geq x_2\}$.  In $(O_2\setminus O_1) \cap \{x_2 \leq -1/32\}$,  the gradient  $\nabla A_{k,k+2}$ points into the wedge $\{x_1 \geq 0\} \cap \{x_1 \geq -x_2\}$.  %See Figure ???.
\end{enumerate}
\end{lemma}

\begin{proof} {\it A1.}  Denote the cusp locus of $L'_{st}$ by $\Sigma'_{\mathit{st}} = \{(c(x_2),x_2)\}$.  The cusp locus of $L''_{st}$ is $T^{-1}(\Sigma'_{\mathit{st}})$, and for $(c(x_2),x_2) \in \Sigma'_{\mathit{st}}$ with $|x_2| \geq 1/16$ we
 compute
\[
T^{-1}(c(x_2),x_2) = (c(x_2)-b(|x_2|), x_2).
\]
From the definition of $b$, and (P2) we have
\begin{align*}
& -3/8-\ec \leq c(0) \leq c(x_2) = c(x_2)-b(|x_2|) \leq -3/8, & \quad  \mbox{for $|x_2| \leq 1/32$;} \\
& -3/8 = c(x_2) - (c(x_2) +3/8) \leq c(x_2)-b(|x_2|) \leq c(x_2) \leq -3/8 + \ec, & \quad  \mbox{for $1/32 \leq |x_2|$}, 
\end{align*} 
with $c(x_2) - b(|x_2|) = -3/8$ when $|x_2| \geq 1/16$.

To prove the slope bound, note that the differential of $T^{-1}$ is 
\[
\left[ \begin{array}{cc} 1 & -b'(|x_2|) \cdot \sgn(x_2) \\ 0 & 1  \end{array}  \right].
\]
By property (P3), for a point $T^{-1}(a) \in T^{-1}(\Sigma'_{\mathit{st}})$ in the cusp locus of $L''_{st}$ with $1/32 \leq |x_2| \leq 1/16$, we have that the tangent to $\Sigma'_{\mathit{st}}$ at $a$ is given by a vector of the form $\left[ \begin{array}{c} 1 \\ y  \end{array}  \right]$ with $|y| > 9 (2 N + 1.1)$.  Thus, a tangent vector to $T^{-1}(\Sigma'_{\mathit{st}})$ at $T^{-1}(a)$ is given by $\left[ \begin{array}{c} 1   -b'(|x_2|) \cdot \sgn(x_2) \cdot y \\ y  \end{array}  \right]$.  The slope therefore satisfies
\[
|\mbox{slope}| = \left| \frac{y }{1  -b'(|x_2|) \cdot \sgn(x_2) \cdot y}\right| \geq \frac{|y| }{1  +|b'(|x_2|)| \cdot | y|} = \frac{1}{1/|y|+|b'(|x_2|)|}.
\]
Again using (P3) and a defining property of $b$, we have
\[
{1/|y|+|b'(|x_2|)|} \leq \frac{1}{9(2N+1.1)} + \frac{1}{3(2N+1.1)} = \frac{4}{9(2N+1.1)},
\]
so it follows that $|\mbox{slope}| \geq \frac{9}{4}(2N+1.1)> 2N+1.1 $.

{\it A2.}  This statement is clearly true for $L'_{st}$ (using Lemma \ref{lem:crossingx2}), and $T(x_1,x_2) = (x_1,x_2)$ when $ |x_2| < 1/32$.

{\it A3.}  The differentials of $S$ and $T$ are given by
\[
dS =  \left[ \begin{array}{cc} M & 0 \\ 0 &  M \end{array} \right] \quad \mbox{and} \quad dT= \left[ \begin{array}{cc} 1 & b'(|x_2|) \cdot \mbox{sgn}(x_2) \\ 0 & 1 \end{array} \right].
\]
Using the chain rule we compute the differential of $A_{i,j}= a_{i,j} \circ S \circ T$ to be
\[
[\partial_{x_1} A_{i,j}, \partial_{x_2} A_{i,j}] = d A_{i,j}=  d a_{i,j} \cdot dS \cdot dT,\]
and carrying out the matrix multiplication gives
\begin{equation} \label{eq:Dx1Aij}
\partial_{x_1}A_{i,j} =  M \cdot \partial_{x_1} a_{i,j}( S\circ T(x_1,x_2));   
\end{equation}
\begin{equation} \label{eq:Dx2Aij}
\partial_{x_2}A_{i,j} = M\cdot \left[\mbox{sgn}(x_2) \cdot b'(|x_2|) \cdot \partial_{x_1} a_{i,j}( S\circ T(x_1,x_2)) + \partial_{x_2} a_{i,j}( S\circ T(x_1,x_2)) \right].
\end{equation}

Using Corollary \ref{cor:aijGrad} and its surrounding notation,  we have that for any $(r,s) \in W=\{s \geq -2r\}\cap \{s\geq r\}$ (and hence any $(x_1,x_2) \in R_+$) 
\begin{itemize}
\item $\partial_{x_1} a_{k,k+1} = (s-r) (s+r) \geq 0$, with equality only when $r=s$;
\item  $\partial_{x_1} a_{k,k+2}= (2r+s)\cdot s \geq 0$, with equality only when $s= -2r$; and
\item $\partial_{x_1} a_{k+1,k+2}= (r+2s) \cdot r$ which shows that $\sgn(\partial_{x_1} a_{k+1,k+2}) = \sgn(r)$.
\end{itemize}
Since the images of $\{r=s\}\cap W$ and $\{s=-2r\}\cap W$ are respectively the upper and lower halves of the cusp locus the statements from {\it A3} concerning $A_{k,k+1}$ and $A_{k,k+2}$ both follow.  
To deduce the claim about $A_{k+1,k+2}$, note that $\sgn(r) = \sgn(x_2)$ and that the diffeomorphism $ S \circ T$ preserves the sign of the $x_2$ coordinate.

{\it A4.}  When $|x_2|\leq1/32$, (\ref{eq:Dx2Aij}) simplifies to 
\[
\partial_{x_2}A_{k+1,k+2}(x_1,x_2)= M \cdot \partial_{x_2} a_{k+1,k+2}( S\circ T(x_1,x_2)),
\]
and, as $r+2s$ is positive everywhere in $W$ except for $(0,0$), Corollary \ref{cor:aijGrad} shows that $\partial_{x_2} a_{k+1,k+2}$ is positive everywhere except for the swallowtail point itself.  Similar computations using Corollary \ref{cor:aijGrad} verify the signs of $\partial_{x_2} A_{k,k+1}$ and $\partial_{x_2} A_{k,k+2}$.

{\it A5.}  With  {\it A3} established we need only show that $\partial_{x_2} A_{k,k+1} \leq  \partial_{x_1} A_{k,k+1}$ when $ x_2\geq 1/32$ and $-\partial_{x_2} A_{k,k+2} \leq  \partial_{x_1} A_{k,k+2}$ 
when $x_2 \leq -1/32$.
For the first inequality, note that when $x_2 \geq 1/32$,
\[
\sgn(x_2) b'(|x_2|) \partial_{x_1} a_{k,k+1}(S\circ T(x_1,x_2)) \leq 1 \cdot \partial_{x_1} a_{k,k+1}(S\circ T(x_1,x_2))
\]
and
\[
\partial_{x_2} a_{k,k+1}(S\circ T(x_1,x_2)) \leq 0
\]
where we have used the properties of $b$ and Corollary \ref{cor:aijGrad} respectively.  In view of (\ref{eq:Dx1Aij}) and (\ref{eq:Dx2Aij}), adding these inequalities and multiplying both sides by $M$ gives $\partial_{x_2} A_{k,k+1} \leq  \partial_{x_1} A_{k,k+1}$ as desired.  When $x_2 \leq -1/32$, note that Corollary \ref{cor:aijGrad} gives $\partial_{x_2} a_{k,k+2}(S\circ T(x_1,x_2)) \geq 0$, and then estimate
\[
-\partial_{x_2} A_{k,k+2}= M[-\sgn(x_2) b'(|x_2|) \partial_{x_1} a_{k,k+2}(S\circ T(x_1,x_2))-\partial_{x_2} a_{k,k+2}(S\circ T(x_1,x_2))] \leq 
\]
\[
M[ \partial_{x_1} a_{k,k+2}(S\circ T(x_1,x_2))+0] = \partial_{x_1} A_{k,k+2}.
\]

\end{proof}

%be chosen so that Choose the constant $X$  to place ;  choose $M$ so that  the cusp locus intersects the boundary circle $\partial O_1$ in two points where \footnote{The $\epsilon_1$ is unimportant here, so perhaps its presence will be misleading.  Could change to $1/100$ maybe?} $x_1 =3/8+\epsilon_1$ with the slope of the cusp locus bounded below in magnitude by $10$ in the annulus $O_2\setminus O_1$.  

\subsection{Defining functions away from the swallowtail point} \label{sec:funcO2}

The edges of the Type (13) square have the following types:  The $L$ edge is (PV) with $n-2$ sheets.  Edges $U$ and $D$ are (Cu) with  $n$ sheets and a left cusp between sheets $k$ and $k+1$.  The $R$ edge is (1Cr) with $n$ sheets and a single crossing between the sheets $k+1$ and $k+2$.  We let $f^L_i$, $f^U_i$, $f^D_i$,  $f^R_{i}$ denote the $1$-skeleton functions associated to each of these types of edges as constructed in Section \ref{sec:Constructions}.   
%for $1$-cells of plain vanilla type; left cusp between $k,k+1$; and single crossing between $k+1, k+2$ respectively.  

In constructing the swallowtail square we have occasion to introduce $5$ new $1$-variable functions which we denote by $f^{ST}_{k}, f^{ST}_{k+1}, f^{ST}_{k+2}, \widehat{f}_{k+1}$, and $\widehat{f}_{k+2}$.  The three $f^{ST}$ will serve as substitutes for the functions $f^{L}_{k-0.5}$ that were used in defining the Type (9)-(12) squares (see Section \ref{ssec:Two-skeleton}) while $\widehat{f}_{k+1}$ and $\widehat{f}_{k+2}$  will be used in combination with $f^U_{k+1}= f^D_{k+1}$ and $f^U_{k+2} = f^D_{k+2}.$

\subsubsection{Definition of $\widehat{f}_{k+1}$ and $\widehat{f}_{k+2}$}
We construct functions $\widehat{f}_{k+1} : [-1,1] \rightarrow \R$ and $\widehat{f}_{k+2}: [-3/8,1] \rightarrow \R$ to satisfy 
\begin{equation} \label{eq:hatk1}
 \widehat{f}_{k+1}(x) = \left\{ \begin{array}{cr} f^U_{k+2}(x) & \mbox{for $x \in [-1, -7/8]$}, \\ f^U_{k+1}(x) & \mbox{for $x \in [-3/8+R_1/2,1]$,} \end{array} \right.
\end{equation}
and
\begin{equation} \label{eq:hatk2}
 \widehat{f}_{k+2}(x) = \left\{ \begin{array}{cr} f^U_{k+1}(x) + C & \mbox{near $x=-3/8$}, \\ f^U_{k+2}(x) & \mbox{for $x \in [-3/8+R_1/2,1]$,} \end{array} \right.
\end{equation}
where $C$ is some additive constant, along with several technical requirements.
A schematic depiction of the constructions appears in Figure \ref{fig:fHat}.

\begin{lemma} \label{lem:fhatexist} There exist functions $\widehat{f}_{k+1}$ and $\widehat{f}_{k+2}$ satisfying (\ref{eq:hatk1}) and (\ref{eq:hatk2}) as well as the following conditions:

\begin{itemize}
\item For $x \in [-3/8, -1/4]$, 
\begin{equation} \label{eq:Est21} d_x\widehat{f}_{k+1} < d_x f_k^U    \quad \mbox{and} \quad d_x\widehat{f}_{k+2} \leq d_xf^U_{k+1}.
\end{equation}

\item For $x \in [-3/8,1]$, 
\begin{equation} \label{eq:estfk1hat}  
f^U_{k+2}(x) \leq \widehat{f}_{k+2}(x) <  f^U_{k+1}(x) \quad \mbox{and} \quad f^U_{k+2}(x) <\widehat{f}_{k+1}(x) \leq f^U_{k+1}(x).
\end{equation}
with all inequalities strict at $x= -3/8$.
\item For $x \in [-1, -3/8]$,
\begin{equation} \label{eq:estfk1hatnext}
f^U_{k+2}(x) \leq  \widehat{f}_{k+1}(x) < f^U_{k-1}(x).
\end{equation}
\item For $x \in [ -3/8- \ec, -3/8+\ec]$,
\begin{equation} \label{eq:linearfk1hat}  \widehat{f}_{k+1} \quad \mbox{is linear and} \quad  d_x f^U_{k+2}(x) < d_x \widehat{f}_{k+1}(x) < d_x f^U_{k+1}(-3/8).
\end{equation}
\item For $(x_1,x_2) \in (\{-3/8 \leq x_1 \leq -1/4\}\cap \{x_2>0\})\setminus O_1$, 
\begin{equation} \label{eq:phix2}
\phi(x_2) d_x(f^U_{k+1} - f^U_{k+2})(x_1) \geq [1 - \phi(x_2)] d_x (\widehat{f}_{k+1} - \widehat{f}_{k+2})(x_1).
\end{equation}
\end{itemize}
\end{lemma}

\begin{proof}
Note that for $-3/8 +R_1/2 \leq x_1$ (\ref{eq:hatk1}) and (\ref{eq:hatk2}) will imply (\ref{eq:phix2}) since $\phi(x_2) \ge1-\phi(x_2)$ when $x_2 >0$. For $(x_1,x_2)$ where (\ref{eq:phix2}) is required to hold, when $-3/8 \leq  x_1\leq  -3/8+R_1/2$ we will have $x_2 \geq \sqrt{3}R_1/2= \sqrt{3}/32$ by trigonometry.  Thus, (\ref{eq:phix2}) will follow provided we have (\ref{eq:hatk1}), (\ref{eq:hatk2}) and the condition
\begin{equation} \label{eq:Newphix2}
\left[\frac{\phi(\sqrt{3}R_1/2)}{1 - \phi(\sqrt{3}R_1/2)}\right] \cdot d_x(f_{k+1}^U-f^U_{k+2})(x_1) \geq d_x(\widehat{f}_{k+1} - \widehat{f}_{k+2})(x_1)
\end{equation}
holds for $-3/8 \leq x_1 \leq -3/8+R_1/2$.  

Therefore, it suffices to obtain $\widehat{f}_{k+1}$ and $\widehat{f}_{k+2}$ satisfying  (\ref{eq:hatk1})-(\ref{eq:linearfk1hat}) and (\ref{eq:Newphix2}).    
Begin by setting $\widehat{f}_{k+1}=f^U_{k+1}$ and $\widehat{f}_{k+2}=f^U_{k+2}$ for $x \geq -3/8 + R_1/2$.  
Then, proceeding from right to left, we modify the derivatives of $\widehat{f}_{k+1}$ and $\widehat{f}_{k+2}$ in a small interval to the left of $-3/8 + R_1/2$ but to the right of $-3/8-\ec$ to momentarily make $d_x \widehat{f}_{k+1}$ slightly larger than $d_x f^U_{k+1}$ and $d_x \widehat{f}_{k+2}$ slightly smaller than $d_x f^U_{k+2}$ in such a way that (\ref{eq:Est21}) and (\ref{eq:Newphix2}) continue to hold.  
The latter is achievable because  \begin{equation} \label{eq:phiand15} 
\displaystyle\left[\frac{\phi(\sqrt{3}R_1/2)}{1 - \phi(\sqrt{3}R_1/2)}\right]= \frac{8 +\sqrt{3}}{8-\sqrt{3}} >1.5,
\end{equation}
 where we used (\ref{eq:InterpolatePhi}) to evaluate $\phi$. 
 Next, return the derivatives to equality, $d_x \widehat{f}_{l}= d_x f^U_{l}$, $l=k+1,k+2$. 
%so that, continuing from right to left, the graphs of $\widehat{f}_{k+1}$ and $\widehat{f}_{k+2}$ are parallel to $f^U_{k+1}$ and $f^U_{k+2}$ but are respectively shifted slightly up and slightly down.  
Continuing from right to left, we now have an interval where $\widehat{f}_{k+1} = f_{k+1}^U - \alpha_1$ and $\widehat{f}_{k+2}= f^U_{k+2}+\alpha_2$ for positive constants $\alpha_1$ and $\alpha_2$.  Moreover, the previous construction can be arranged so that $\alpha_1=\alpha_2 =\alpha$ with $0 < \alpha <<1$ as small as desired.
 % are parallel to $f^U_{k+1}$ and $f^U_{k+2}$ but are respectively shifted slightly up and slightly down.  

To complete the definition of $\widehat{f}_{k+2}$, just prior to $-3/8$, interpolate the derivative $d_x \widehat{f}_{k+2}$ from  $d_x f^U_{k+2}$ to $d_x f^U_{k+1}$  to arrange that $\widehat{f}_{k+2} = f^U_{k+1}(x) + C$ in a interval to the right of and including  $x=-3/8$.  This can be done to maintain the estimates for $d_x\widehat{f}_{k+2}$ and $\widehat{f}_{k+2}$ from (\ref{eq:Est21}) and (\ref{eq:estfk1hat}) provided the interpolation is carried out in a small enough neighborhood of $-3/8$ and $\alpha$ was taken to be suitably small.
%  Notice that $d_x\widehat{f}_{k+2} \geq d_xf^U_{k+2}$ holds during this portion....}

To complete the definition of $\widehat{f}_{k+1}$, just to the right of $x_1=-3/8+\ec$ interpolate  $d_x\widehat{f}_{k+1}$ from being equal to $d_xf^U_{k+1}$ to being constant with 
\[
d_x\widehat{f}_{k+1}(x) = d_xf^U_{k+1}(-3/8) - \delta 
\]
with $\delta >0$ small,  and then continue the definition of $\widehat{f}_{k+1}$ by interpolating between the linear function and $f^U_{k+2}$ somewhere on the interval $[-7/8,-3/8-\ec)$.

We now verify the required inequalities  for $\widehat{f}_{k+1}$ from (\ref{eq:Est21})-(\ref{eq:linearfk1hat}).  
Note that (\ref{eq:Est21}) continues to hold since
\[
d_x\widehat{f}_{k+1}(x) \leq d_xf^U_{k+1}(-3/8) - \delta = d_xf^U_k(-3/8) -\delta < d_xf^U_{k}(x), \quad \quad \mbox{for $x \in [-3/8, -3/8+\ec]$} 
\]
where the 2nd inequality follows from Proposition \ref{prop:CuDef} (8').  In addition, $\widehat{f}_{k+1}(x) \leq f^U_{k+1}(x)$ continues to hold (part of (\ref{eq:estfk1hat})) provided $\delta$ is chosen small enough, since $f^U_{k+1}> \widehat{f}_{k+1}(x)$ holds to the right of $-3/8+\ec$, and for $x \in[-3/8,-3/8+\ec]$
\[
d_x(f^U_{k+1}- \widehat{f}_{k+1})(x) < d_xf^{U}_{k+1}(-3/8) - (d_xf^{U}_{k+1}(-3/8) - \delta) = \delta.
\]
[When working from right to left $f^U_{k+1}-\widehat{f}_{k+1}$ could conceivably decrease, but by taking $\delta$ sufficiently small, we guarantee that the total decrease from $x=-3/8+\ec$ to $x=-3/8$ is less than the value of the difference at $-3/8+\ec$.]  To verify the bound from (\ref{eq:linearfk1hat}) for $x \in [-3/8-\ec,-3/8+\ec]$, we note that $d_xf^U_{k+2}$ is also constant (by Proposition \ref{prop:CuDef} (8'))  with
\[
d_xf^U_{k+2} < d_xf^{U}_{k+1}(-3/8).
\]
Next, note that to the right of $x=-3/8 +\ec$, 
\begin{align*}
(f^U_{k+1}-f^U_{k+2})(x) \geq (f^U_{k+1}-f^U_{k+2})(-3/8) &= .08 \ea, \quad \mbox{and} \\
(f^U_{k-1}-f^U_{k+1})(x) \geq (f^U_{k-1}-f^U_{k+1})(-3/8) & \geq .08 \ea
\end{align*}
(evaluated using (\ref{eq:k1k2CuDef})), and
$\widehat{f}_{k+1}$ is approximately equal to $f^U_{k+1}$.  Since 
\begin{align*} 
d_x(\widehat{f}_{k+1}- f^U_{k+2}) \leq d_xf_{k-1,k+2} &= 3 \\
d_x(f^U_{k-1}-\widehat{f}_{k+1}) \leq d_xf_{k-1,k+2} &= 3
\end{align*}
  on $[-3/8-\ec,-3/8+\ec]$, the total change in magnitude to both $(\widehat{f}_{k+1}-f^U_{k+2})(x)$ and $(f^U_{k-1}-\widehat{f}_{k+1})(x)$ on this interval is less than $6\ec$ which is much smaller than $.08 \ea$.  It follows that  
\[
f^U_{k+2}(x) \leq  \widehat{f}_{k+1}(x) < f^U_{k-1}(x) 
\]
holds on $[-3/8-\ec,-3/8+\ec]$ (and this gives the other part of (\ref{eq:estfk1hat})).  As long as the interpolation from the linear function defining $\widehat{f}_{k+1}(x)$ near $-3/8-\ec$ to $f^U_{k+2}(x)$ is done close enough to $-3/8-\ec$, the previous inequality remains valid on all of $[-1,-3/8]$ as required by (\ref{eq:estfk1hatnext}).

Finally,  we check (\ref{eq:phix2}), for $x \in [-3/8, -3/8+\ec]$.  We have
\[
d_x(\widehat{f}_{k+1}-\widehat{f}_{k+2})(x) \leq d_xf^U_{k+1}(-3/8)-\delta -d_xf^U_{k+2}(x) = 
\]
\[
d_x(f_{k+1}^U)(-3/8+\ec) + (3/2)(\ec)^{1/2} - \delta - d_xf^U_{k+2}(x)  \leq
d_x(f^U_{k+1}- f^U_{k+2})(x) + (3/2)(\ec)^{1/2}.
\]
[Use the definition of $d_x\widehat{f}_{k+1}$ and $d_x\widehat{f}_{k+2}$; then, the explicit formula for $d_xf^U_{k+1}$ in $[-3/8,-3/8+\ec]$ from Proposition \ref{prop:CuDef} (8'); then, that $d_xf^U_{k+1}$ decreases on $[-3/8,-3/8+\ec]$.] 
Also, using  (\ref{eq:phiand15}) and $d_x(f^U_{k+1})(x)- d_xf^U_{k+2}(x) \geq 1.5-(3/2)(\ec)^{1/2}$ (by the explicit formula for $d_xf^U_{k+1,k+2}(-3/8) = 1.5$ found in  equation (\ref{eq:k1k2CuDef})), we get 
\[
\left[\frac{\phi(\sqrt{3}R_1/2)}{1 - \phi(\sqrt{3}R_1/2)}\right] \cdot d_x(f_{k+1}^U-f^U_{k+2})(x) \geq 
(3/2) d_x(f_{k+1}^U-f^U_{k+2})(x) \geq 
\]
\[
d_x(f_{k+1}^U-f^U_{k+2})(x)+ .5 \left(1.5-(3/2)(\ec)^{1/2}\right) \geq d_x(f^U_{k+1}- f^U_{k+2})(x) + (3/2)(\ec)^{1/2}. 
\]
Combining these inequalites shows that (\ref{eq:Newphix2}) holds.

%  slope slightly less than that of $f^U_{k+1}$ prior to $-3/8,$ ensuring (\ref{eq:linearfk1hat}) \dr{and (\ref{eq:estfk1hat})} holds , and then continue the definition of $\widehat{f}_{k+1}$ by interpolating between the linear function and $f^U_{k+2}$ somewhere on the interval $[-7/8,-3/8).
%$
%Both of these modifications only decrease the right hand side of (\ref{eq:Newphix2}), and allow for (\ref{eq:Est21}) and (\ref{eq:estfk1hat}) provided the final interpolations that occur before $-3/8$ are carried out in a small enough neighborhood of $-3/8$.
\end{proof}

%\footnote{\ms{3/9/15, 6/19/15: I reduced interpolation support  from $[-1,-3/8)$ and added criterion (\ref{eq:linearfk1hat}) here, OK?}  \dr{2-5-16: The linear condition that you added in  (\ref{eq:linearfk1hat}) does appear to be used later.  When I revised this section, it led to a lengthening of the proof.}}

We will make use of the estimates
\begin{equation}  \label{eq:fhat14est}
|\widehat{f}_{k+1}(x)| < N \ea, \quad |\widehat{f}_{k+2}(x)| < N \ea,  \quad \mbox{for $x\in [-1,-1/4]$}
\end{equation}
that follow from (\ref{eq:estfk1hat}), (\ref{eq:estfk1hatnext}), and the corresponding estimate for the $f^U$ from (3) of Corollary \ref{cor:summary}.

\begin{figure}
\labellist
\small
\pinlabel $f^U_k$ [l] at 242 272
\pinlabel $f^U_{k+1}$ [l] at 242 192
\pinlabel $f^U_{k+2}$ [l] at 242 0
\pinlabel $\widehat{f}_{k+1}$ [br] at 32 42
\pinlabel $\widehat{f}_{k+2}$ [bl] at 162 34
\endlabellist
\centerline{ \includegraphics[scale=.6]{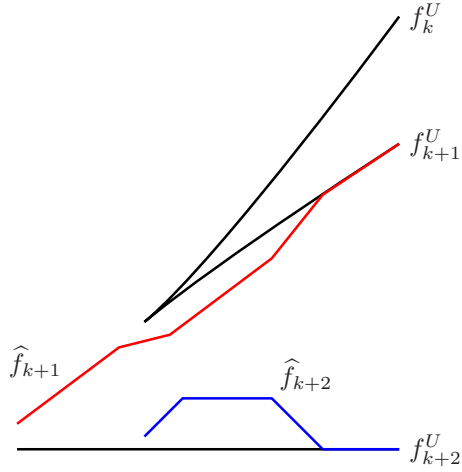} }
%\centerline{ \includegraphics[scale=.8]{images/SubdivideSq} }
\caption{Schematic graphs of $f^U_k$, $f^U_{k+1}$, $\widehat{f}_{k+1}$ (red), $\widehat{f}_{k+2}$ (blue), and $f_{k+2}$ with $f_{k+2}$ pictured as $0$.}
\label{fig:fHat}
\end{figure}

\subsubsection{Definition of the $f^{ST}$}

To begin, fix a bump function $\psi:[-1,1] \rightarrow [0,1]$  related to the cutoff function $\phi$ from Section \ref{ssec:Two-skeleton} by
\[
\psi(x)= \left\{ \begin{array}{cr} \phi(x), & x \leq-R_1/2 \\ 1-\phi(x), & x \geq +R_1/2  \end{array} \right.
\]
 and with a single local maximum on $[-1/4,1/4]$ at $x=0$.  In addition, we require that $\psi$ is an even function satisfying the  bounds
 \begin{equation}
\label{eq:psibounds}
0 \leq \psi(x) \leq \mbox{Min}\{\phi(x), 1-\phi(x)\}, \quad |\psi'(x)| <3,  \quad \mbox{and} \quad |\psi''_{[-1/16,1/16]}| <65.
\end{equation}
See Figure \ref{fig:phi}.  [To obtain $\psi(x)$, set $\psi(x) = \phi(x)$ for $x \leq-R_1/2 = -1/32$.  Then, interpolate the derivative of $\psi(x)$, which is $2$ at $x= -1/32$, to agree with $-x$
 in a neighborhood of $x=0$.  If the interpolation is done in a manner that is roughly linear, then the bound $|\psi''(x)| <65$ will hold.  Finally, extend $\psi(x)$ to be an even function and note that $\psi(x) =1-\phi(x)$ for $ x \geq R_1/2$ follows from (\ref{eq:InterpolatePhi}).]

\begin{lemma} \label{lem:flst}
There exists a function  $f^{ST}_k: [-1,1] \rightarrow \R$ satisfying
\begin{align}
& f^{ST}_k(x) - f^L_{k+1}(x) = 1.5 Q_{\pm}(x),   & \mbox{  for $|x-(\pm1)| \leq 1/16$,} \label{eq:STk1111} \\
& f^{ST}_{k}(x) := f^{L}_k(x)+ \psi(x)[ \widehat{f}_{k+2}(-3/8) - f^U_{k+1}(-3/8)]  & \mbox{  on $[-3/8,3/4]$}; \label{eq:STk1434}  
\end{align}
with
\begin{equation} \label{eq:flstc0}
\forall x \in [-1,-1/4] \cup[3/4,1],  \quad f^L_k(x) \leq f^{ST}_k(x) \leq f^L_{k-1}(x);
\end{equation}
\begin{equation} \label{eq:propflst}
\forall x \in (-1,-1/4], \quad (f^L_{k-1} -f^{ST}_k)'(x) >0, \quad \mbox{and  } (f^{ST}_{k}-f^L_{k+1})'(x) >0; 
\end{equation}
\begin{equation}  \label{eq:propflst2}
\forall x \in [3/4,1), \quad (f^L_{k-1} -f^{ST}_k)'(x) <0, \quad \mbox{and  } (f^{ST}_{k}-f^L_{k+1})'(x) <0;
\end{equation}
\begin{equation} \label{eq:C0flst}
||f^{ST}_{k}- f^L_k||_{C^0([-1,-1/4])} < 2 N \ea, \quad \mbox{and} \quad ||f^{ST}_{k}- f^L_k||_{C^0([3/4,1])} < 2 N \ea.
\end{equation}
 \end{lemma}
\begin{proof}
We need to show that the definition of $f^{ST}_k$ resulting from (\ref{eq:STk1111}) and (\ref{eq:STk1434}) can be extended over $[-15/16,-3/8] \cup [3/4,15/16]$ while retaining the properties (\ref{eq:flstc0})-(\ref{eq:C0flst}).
%\footnote{\ms{3/7/16: Should be (\ref{eq:flstc0})-(\ref{eq:C0flst})}}

\medskip

%\noindent {\bf Extending over $[-15/16,-1/4]$:}   We construct $f^{ST}_k - f^L_{k+1}$, and then set $f^{ST}_k = f^L_{k+1} + (f^{ST}_k - f^L_{k+1})$.  
%Notice that the inequalities in (\ref{eq:propflst}) are equivalent to the inequality
%\[
%0 < d_x(f^{ST}_k - f^L_{k+1})(x) < d_x(f^L_{k-1}(x) - f^L_{k+1}(x)).
%\]
%Thus, to extend the definition of $f^{ST}_k - f^L_{k+1}$ along $[-15/16,-1/4]$ to obtain (\ref{eq:propflst}) via Lemma \ref{lem:thetechlem} we just need to observe that
%\[
%0 < (f^{ST}_k - f^L_{k+1})(-1/4) - (f^{ST}_k - f^L_{k+1})(-15/16) = f^L_{k,k+1}(-1/4) - 1.5 Q_-(-15/16) = .1\ea - 1.5(257/256) \ec
%\]
%and the latter is smaller than
%\[
%f^L_{k-1,k+1}(-1/4)- f^L_{k-1,k+1}(-15/16) = .2 \ea - 2(257/256) \ec
%\]
%(since $\ec << \ea$ as in \dr{???}).
%To verify the first inequality of (\ref{eq:C0flst}), we note that 
%\[
%f^L_{k+1}(-1) < f^{ST}_k(-1) < f^L_{k-1}(-1) \quad \mbox{and} \quad \forall x \in (-1, -1/4], \,\, d_xf^L_{k+1}(x) < d_xf^{ST}_k(x) < d_xf^L_{k-1}(x)
%\]
%hold via Corollary \ref{cor:summary} (4) and (\ref{eq:propflst}).  Thus, 
%\[
%\forall x \in (-1, -1/4], \,\, f^L_{k+1}(x) < f^{ST}_k(x) < f^L_{k-1}(x),
%\]
%so that for any $x \in [-1,1/4]$ we have
%\[
%|f^{ST}_{k}(x)- f^L_k(x)| \leq |f^{ST}_k(x)| + |f^L_k(x)|\leq \mbox{Max}\left(|f^{L}_{k-1}(x)|,|f^{L}_{k+1}(x)|\right) + |f^L_k(x)| \leq N\ea + N\ea. 
%\]
%[The last inequality is from Corollary \ref{cor:summary} (3).]

\noindent {\bf Extending over $[3/4,15/16]$:}
%To arrange (\ref{eq:C0flst}) in $[3/4,1]$ requires slightly more care.  
For small $\delta >0$, fix  $\beta:[7/8,1]\rightarrow \R$ with 
\[
\beta \equiv -\delta \,\, \mbox{near $7/8$}, \quad \beta \equiv .5 \,\, \mbox{on $[15/16,1]$},
\] 
\[
\beta' \geq 0, \quad ||\beta'||_{C^0} < 9. 
\] 
In $[3/4,1]$, we will define $f^{ST}_k-f^L_{k}$.  Start by setting
\begin{equation} \label{eq:fSTbeta78}
(f^{ST}_k-f^L_{k})(x) = \beta(x) \ec (x-1)^2+.5 \ec,  \quad \forall x\in [7/8,1]. 
\end{equation} 
Notice that, since $Q_+(x) = \ec (x-1)^2 + \ec$, (\ref{eq:STk1111}) holds.  Moreover, for $x\in[7/8,1]$ we can verify (\ref{eq:propflst2}) by using Corollary \ref{cor:summary} (4) to compute
\begin{align*}
(f^L_{k-1}-f_k^{ST})(x) & = (1-\beta(x))\ec(x-1)^2 +.5 \ec; \\
(f_k^{ST}-f^L_{k+1})(x) & = (1+\beta(x)) \ec(x-1)^2+1.5 \ec.
\end{align*}
Then, estimate for $x \in [7/8,1)$,
\[
(f^L_{k-1}-f_k^{ST})'(x) = -\beta'(x) \ec(x-1)^2 + (1-\beta(x)) 2 \ec(x-1) < 0 + 0;
\]
\[
(f_k^{ST}-f^L_{k+1})'(x) = \beta'(x)\ec(x-1)^2 + (1+\beta(x))2 \ec(x-1) \leq ||\beta'||_{C^0} \ec (x-1)^2 + (1-\delta)2 \ec(x-1) \leq
\]
\[
(x-1)\ec \left(9 (x-1) + (1-\delta)2\right) <0 
\]
where the last inequality requires that $\delta < 7/16$.

With $f^{ST}_k-f^L_{k}$ now defined on $[7/8,1]$, we extend it to $[1/2,1]$ using  Lemma \ref{lem:thetechlem} to agree with $0$ on $[1/2,3/4]$ and to meet the requirements
\begin{equation} \label{eq:streq}
0 \leq (f^{ST}_k-f^L_k)'(x) < (f^L_{k+1}-f^L_{k})'(x).
\end{equation}
 To apply Lemma \ref{lem:thetechlem}, first extend $f^{ST}_k-f^L_{k}$ slightly past $3/4$ to make the first inequality strict, and then verify the two requirements:
\begin{enumerate}
\item  That (\ref{eq:streq}) already holds for $x$ near $3/4$ and $7/8$.  
This is straightforward since  $d_xf^L_{k,k+1} < 0$ near $3/4$ (by Proposition \ref{prop:PV1Cr2Cr}) 
and near $7/8$ Corollary \ref{cor:summary} (4) and (\ref{eq:fSTbeta78}) give explicitly
\[
0 < (f^{ST}_k-f^L_k)'(x) = -2 \delta \ec(x-1) < -2 \ec(x-1) = (f^L_{k+1}-f^L_{k})'(x).
\]
\item  The other required inequality is
\[
0 <  (f^{ST}_k-f^L_k)(7/8)- (f^{ST}_k-f^L_k)(3/4) < (f^L_{k+1}-f^L_{k})(7/8)-(f^L_{k+1}-f^L_{k})(3/4). 
\]
We have
\begin{equation} \label{eq:otherreq}
(f^{ST}_k-f^L_k)(7/8)- (f^{ST}_k-f^L_k)(3/4) = -\delta \ec(-1/8)^2+ .5\ec
\end{equation}
which is positive provided $\delta$ is sufficiently small and
\[
(f^L_{k+1}-f^L_{k})(7/8)-(f^L_{k+1}-f^L_{k})(3/4)   \geq -(65/64)\ec +y^L_{k,k+1} - N\ea > .5 
\]
where we bounded below the second term using Proposition \ref{prop:PV1Cr2Cr} (2).
\end{enumerate}

With the definition now complete, we note that (\ref{eq:streq}) implies both inequalities from (\ref{eq:propflst2})
 for $x \in [3/4,7/8]$.  [Adding $f^L_k$ everywhere gives
\[
d_xf^L_k(x) \leq d_xf^{ST}_k(x) < d_xf^L_{k+1}(x).
\]
Multiplying the first inequality by $-1$ and adding $d_xf^L_{k-1}$ gives
\[
(f^L_{k-1}-f^{ST}_k)'(x) \leq d_xf^L_{k-1,k}(x) < 0,
\]
and subtracting $d_xf^L_{k+1}(x)$ from the second inequality gives $(f^{ST}_k-f^L_{k+1})'(x) < 0$.]
To check the second inequality of (\ref{eq:C0flst}) note that, by (\ref{eq:streq}), $(f^{ST}_k-f^L_k)|_{[3/4,7/8]}$ has its maximum absolute value at $7/8$ and the required bound holds in $[7/8,1]$ by (\ref{eq:fSTbeta78}).  

Finally, we check (\ref{eq:flstc0}).  For small enough $\delta > 0$, check (\ref{eq:flstc0}) directly in $[7/8,1]$ using  (\ref{eq:fSTbeta78}) and the equation that immediately follows it.  To verify that $f_k^L \leq f^{ST}_k$ holds on  $[3/4,7/8]$ use the first inequality from (\ref{eq:streq}) together with $(f^{ST}_k-f^L_k)(3/4) = 0$.  To verify that $f^{ST}_k \leq f^L_{k-1}$ holds on $[3/4,7/8]$, use that $(f^L_{k-1}-f^{ST}_k)' <0$ and that the inequality holds at $x=7/8$.

\medskip

\noindent {\bf Extending over $[-15/16, -3/8]$:} Using a similar procedure, set
\[
(f^{ST}_k-f^L_{k})(x) = \beta(-x) \ec (x+1)^2+.5 \ec,  \quad \forall x\in [-1,-7/8]. 
\]
Then, extend for $x \in [-7/8,-3/8]$ in a manner that satisfies
\[
0 \leq (f^L_k-f^{ST}_k)'(x) < d_xf^L_{k,k+1}(x).
\]
The verification that this can be done and that the required properties of $f^{ST}_k$ all hold is similar to the above.
Here, we only mention that the definition of $f^L_{k,k+1}$ in (\ref{eq:initialdef}) may be consulted to obtain the inequality
\[
f^L_{k,k+1}(-3/8)-f^L_{k,k+1}(-7/8) > [f^L_{k}-f^{ST}_k](-3/8) - [f^L_{k}-f^{ST}_k](-7/8) >0
\]
that is required for the application of Lemma \ref{lem:thetechlem}.

%The extension of $f^{ST}_k - f^L_{k+1}$ to $[3/4,15/16]$ is handled in a similar manner:  It suffices to arrange
%\[
%0 > d_x(f^{ST}_k - f^L_{k+1})(x) > d_x(f^L_{k-1}(x) - f^L_{k+1}(x)).
%\]
%To use Lemma \ref{lem:thetechlem}, we check
%\[
%0 < (f^{ST}_k - f^L_{k+1})(15/16) - (f^{ST}_k - f^L_{k+1})(3/4) =  1.5 Q_+(15/16)-f^L_{k,k+1}(3/4)  = 1.5(257/256) \ec - (1+(2/3)\ea)
%\] 
%and the latter is less negative than
%\[
% f^L_{k-1,k+1}(15/16)- f^L_{k-1,k+1}(3/4) = 2(257/256) \ec - 2(1+(2/3)\ea).
%\]
\end{proof}

With $f^{ST}_k$ defined as in Lemma \ref{lem:flst}, we complete the definition of the $f^{ST}$ by setting
\begin{align}
%& f^{ST}_{k}(x) := f^{L}_k(x)+ \psi(x)[ \widehat{f}_{k+2}(-3/8) - f^U_{k+1}(-3/8)]  & \mbox{  on $[-1/4,3/4]$}; \label{eq:STk1434}\\
& f^{ST}_{k+2} := f^{ST}_k & \mbox{  on $[-1,-1/4]$}; \label{eq:STk2114}\\
& f^{ST}_{k+2} := f^{L}_k & \mbox{  on $[-1/4,1]$}; \label{eq:STk214} \\
& f^{ST}_{k+1} := f^{L}_k & \mbox{  on $[-1,3/4]$; and} \label{eq:STk134} \\
& f^{ST}_{k+1} := f^{ST}_k & \mbox{  on $[3/4,1].$} \label{eq:STk1341}
\end{align}
Note that $\psi$ vanishes near outside of  $(-1/4,1/4)$ so that the definitions of $f^{ST}_{k+1}$ and $f^{ST}_{k+2}$ piece together smoothly.

\subsubsection{Definition of $2$-variable defining functions in $I^2\setminus O_1$}  
With these $1$-variable functions in hand, we now give $2$-variable functions 
\[\begin{array}{rl}
G_k,G_{k+1},G_{k+2}: & \{x_2 \geq -3/8\} \setminus \mbox{Int}(O_1) \rightarrow \R,\\
  H: & \{x_2 \leq -3/8\} \setminus \mbox{Int}(O_1) \rightarrow \R,
  \end{array}
\] whose graphs will together form the swallow tail sheets outside of $O_2$.  The definitions are as follows:
%\footnote{\ms{2/26/15: How are $G_{k+1},G_{k+2}$ well-defined at $x_2 = 0?$ For example, in $G_{k+1},$
%why does $\widehat{f}_{k+2}(x_1) + f^U_{k+1}(x_1) = f^D_{k+2}(x_1)+ \widehat{f}_{k+1}(x_1)?$} \dr{10/17:  Changed to state the domain of the $G_l$ and $H$ up front.}}

\begin{align} \label{eq:Gkdef}
G_k(x_1,x_2) = (1-\phi(x_1)) f^{ST}_{k}(x_2)+ \phi(x_1) f^R_{k}(x_2) + f^U_k(x_1).
\end{align}
\begin{align} \label{eq:Gk1def}
G_{k+1}(x_1,x_2) = & (1- \phi(x_1)) f^{ST}_{k+1}(x_2) + \phi(x_1) f^R_{k+1}(x_2) + \\
 &  (1-\phi(x_2))\left\{ \begin{array}{cr}  \widehat{f}_{k+2}(x_1), & \mbox{if $x_2 \geq 0$} \\ f^D_{k+2}(x_1), & \mbox{if $x_2 \leq 0$}, \end{array} \right. + \phi(x_2) \left\{ \begin{array}{cr}  f^U_{k+1}(x_1), & \mbox{if $x_2 \geq 0$} \\ \widehat{f}_{k+1}(x_1), & \mbox{if $x_2 \leq 0$}. \end{array} \right. \nonumber
\end{align}
\begin{align} \label{eq:Gk2def}
G_{k+2}(x_1,x_2) = & (1- \phi(x_1)) f^{ST}_{k+2}(x_2) + \phi(x_1) f^R_{k+2}(x_2) + \\
 &  (1-\phi(x_2)) \left\{ \begin{array}{cr} \widehat{f}_{k+1}(x_1), & \mbox{if $x_2 \geq 0$} \\ f^D_{k+1}(x_1), & \mbox{if $x_2 \leq 0$} \end{array} \right. + \phi(x_2) \left\{ \begin{array}{cr} f^U_{k+2}(x_1), & \mbox{if $x_2 \geq 0$} \\ \widehat{f}_{k+2}(x_1), & \mbox{if $x_2 \leq 0$}. \end{array} \right.  \nonumber
\end{align}
\begin{equation}
\label{eq:Hdef}
H(x_1,x_2) = (1- \psi(x_2))f^U_{k+2}(x_1) + \psi(x_2) \widehat{f}_{k+1}(x_1) + f^L_{k}(x_2).
\end{equation}
%where $\psi(x) = \left\{ \begin{array}{cr} \phi(x), & x \leq-\beta \\ 1-\phi(x), & x \geq +\beta  \end{array} \right.$ and has a single local maximum on $[-1/4,1/4]$ at $x=0$.  (Visualize $\psi$ as an even bump function that is $0$ outside of $[-1/4,1/4]$ and about $1/2$ at $0$.)
%%We form the front projection defined above $I^2\setminus O_1$ by taking the union of the graphs of
%%\begin{itemize}
%%\item $G_{k}, G_{k+1}, G_{k+2}$ in $\{x_2 \geq -3/8\} \setminus O_1$ and
%%\item $H$ in $\{x_2 \leq -3/8\} \setminus O_1$.
%%\end{itemize}
(Note that $G_{k+1}$ and $G_{k+2}$ are well defined and smooth along $x_2=0$, 
since $\widehat{f}_{k+1}(x_1) = f^U_{k+1}(x_1) = f^D_{k+1}(x_1)$ and $\widehat{f}_{k+2}(x_1) = f^U_{k+2}(x_1) = f^D_{k+2}(x_1)$ for $x_1 \geq -3/8 + R_1$.  Also, the definition of $H$ makes sense for all $(x_1,x_2) \in [-1,1]^2$.)  We follow our usual notation by setting $G_{i,j} = G_i-G_j$ for $i,j \in \{k,k+1,k+2\}$.

The graphs of the $G_l$ and $H$ fit together to form the front projection of a smooth Legendrian above $[-1,1]\setminus \mbox{Int}(O_1)$ with features described in the following Proposition.  See Figure \ref{fig:GDomains}.

% The respective domains of $F_k, F_{k+1}, F_{k+2},$ and $G$.  The circle denotes $\partial O_1$.  The vertical solid lines at $x_1= 3/8$ denote cusp edges, with the $+$ or $-$ indicating whether the function is the upper or lower sheet.  The dotted horizontal lines $A$ and $B$ are at $x_2 = + \beta$ and $x_2 = -\beta$ where $0< \beta < R_1$ can be specified at some point.  The two functions defined at $A$ should agree in a neighborhood of $A$; the same statement applies to $B$.
%COMMENT ON DOMAINS The domains as in Figure \ref{fig:Domains}, according to the formulas:

\begin{figure}
\labellist
\small
\pinlabel $G_k=G_{k+1}$ [r] at -2 104
\pinlabel $G_k=G_{k+2}$ [r] at -2 24
\pinlabel $G_{k+1}=G_{k+2}$ [b] at 104 66
\pinlabel $H=G_{k+2}$ [l] at 234 104
\pinlabel $H=G_{k+1}$ [l] at 234 24
\endlabellist
\centerline{ \includegraphics[scale=.8]{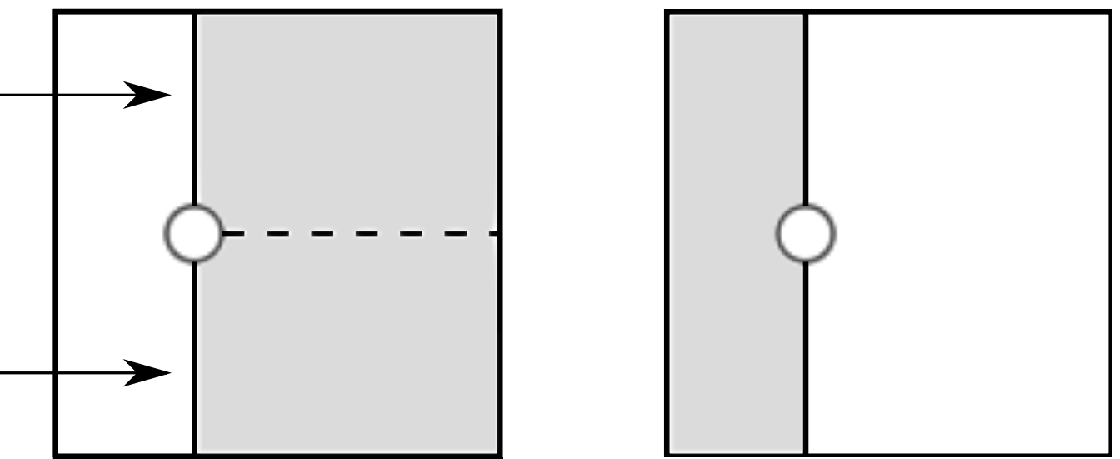} }
%\centerline{ \includegraphics[scale=.8]{images/SubdivideSq} }
\caption{The shaded region in the left (resp. right) square indicates the portion of $[-1,1]^2$ where the graphs of $G_{k},G_{k+1},$ and $G_{k+2}$ (resp. $H$) are used.  The agreement of various functions along parts of the lines $x_1=-3/8$ and $x_2=0$ is indicated.  }
\label{fig:GDomains}
\end{figure}

\begin{lemma}  \label{lem:GkGk1}
With $G_k,G_{k+1},G_{k+2},H$ defined as above,
\begin{itemize}
\item[(1)] The sheets defined by $G_{k}$ and $G_{k+1}$ (resp. $G_{k}$ and $G_{k+2}$)  line up at a cusp edge along $\{x_1=-3/8\} \cap \{x_2 \geq R_1\}$ (resp. along $\{x_1=-3/8\} \cap \{x_2 \leq -R_1\}$).  Moreover, for $x_1 \leq -1/4$ and $|x_2| \leq 1/4$, these are the only points (in the  domain of definition) where $G_{k,k+1}=0$ (resp. $G_{k,k+2}=0$).  
\item[(2)] The sheets defined by $H$ and $G_{k+1}$ (resp. $H$ and $G_{k+2}$) fit together smoothly along  $\{x_1=-3/8\} \cap \{x_2 \leq -R_1\}$ (resp. along $\{x_1=-3/8\} \cap \{x_2 \geq R_1\}$).
\item[(3)] The sheets defined by $G_{k+1}$ and $G_{k+2}$ cross along $\{x_1 \geq R_1\} \cap \{x_2=0\}$.  Moreover, for $x_1 \leq -1/4$, these are the only points (in the  domain of definition) where $G_{k+1,k+2}=0$.
%\item[(4)]  Along $\{x_1 = -3/8\} \cap \{x_2 \geq R_1\}$  (resp. $\{x_1 = -3/8\} \cap \{x_2 \leq -R_1\}$ )  $G_{k,k+2} = G_{k+1,k+2} >0$ (resp. $G_{k,k+1}= G_{k+2,k+1}>0$).
\end{itemize}
\end{lemma}
\begin{proof}  (1):  % We prove only the statement about $G_k$ and $G_{k+1}$ as an analogous proof applies for $G_{k}$ and $G_{k+2}$.  
 Consider $(x_1,x_2)$ where $G_{k,k+1}$ is defined and satisfying $x_1 \leq -1/4$.  Letting $f^{ST}_{k,k+1} = f^{ST}_k-f^{ST}_{k+1}$,  
  we have
\begin{align*}
G_{k,k+1}(x_1,x_2) =  & f^{ST}_{k,k+1}(x_2)  +\\
 & (1-\phi(x_2)) \left\{ \begin{array}{cr} (f^U_k-\widehat{f}_{k+2})(x_1),  &  x_2 \geq 0 \\ f^U_{k,k+2}(x_1), & x_2 \leq 0,  \end{array}\right. +\phi(x_2) \left\{ \begin{array}{cr} f^U_{k,k+1}(x_1),  &  x_2 \geq 0 \\ (f^U_k-\widehat{f}_{k+1})(x_1), & x_2 \leq 0. \end{array}\right.
\end{align*}
Letting $C= \widehat{f}_{k+2}(-3/8)  -f^U_{k+1}(-3/8) < 0$, using (\ref{eq:psibounds}) we have for $x_2\geq -1/4$ 
\begin{equation} \label{eq:evenness}
f^{ST}_{k,k+1}(x_2) = \psi(x_2) C \geq (1-\phi(x_2)) C
\end{equation}
with equality for $x_2 \geq R_1/2$.  From the evenness of $\psi$, and the symmetry in the formulas defining $G_{k+1}$ and $G_{k+2}$, we observe that $G_{k,k+2}(x_1,-x_2) = G_{k,k+1}(x_1,x_2)$, so it suffices to establish (1) for $G_{k,k+1}$.

The inequality (\ref{eq:evenness}) leads to 
\[
G_{k,k+1} \geq (1-\phi(x_2)) \left\{ \begin{array}{cr} (f^U_k-\widehat{f}_{k+2})(x_1) + C ,  &  x_2 \geq 0 \\ f^U_{k,k+2}(x_1) +C , & x_2 \leq 0,  \end{array}\right. 
  +\phi(x_2) \left\{ \begin{array}{cr} f^U_{k,k+1}(x_1),  &  x_2 \geq 0 \\ (f^U_k-\widehat{f}_{k+1})(x_1), & x_2 \leq 0. \end{array} \right.
\]
Note that the terms that appear when $x_2 \leq 0$ are both strictly positive.  [The left term becomes $\widehat{f}_{k+2}(-3/8)-f^U_{k+2}(-3/8) >0 $ 
 when $x_1=-3/8$ by (\ref{eq:estfk1hat}), and remains positive as $x_1$ increases since $d_xf^U_{k,k+2}(x_1) > 0$.  Positivity of the right term also follows from (\ref{eq:estfk1hat}).]  In addition, those terms that appear when $x_2 \geq 0$ are non-negative and vanish only when $x_1 = -3/8$. [The left term is $0$ when $x_1=-3/8$, and then increases by (\ref{eq:Est21}).] Thus, the only possible place where $G_{k,k+1}=0$ is along $\{x_1=-3/8\} \cap \{x_2 \geq R_1\}$. Moreover,  equality holds when $x_2 \geq R_1/2$, and both terms agree with $f^U_{k,k+1}(x_1)$ when $x_1$ is in a neighborhood of $-3/8$ by (\ref{eq:hatk2}).  
Thus, we see that the sheets indeed meet at a cusp edge at the stated location.

%For $x_2 \geq R_1$ and $x_1 \geq -3/8$ with $x_1$ near $-3/8$ we have, by definition, 
%\[
%\psi(x_2) = 1-\phi(x_2);  \quad \phi(x_1)=0; \quad  \widehat{f}_{k+2}(x_1) = f^U_{k+1}(x_1) + C; \quad \mbox{and} \quad f^{ST}_k(x_2)- f^{ST}_{k+1}(x_2)= \psi(x_2) C. 
%\]
%Using these observations, we compute 
%\begin{eqnarray*}
%G_{k}(x_1,x_2)-G_{k+1}(x_1,x_2) &= & f^{ST}_k(x_2)- f^{ST}_{k+1}(x_2) + f^U_{k}(x_2) -\left[ \phi(x_2) f^U_{k+1}(x_1) + (1-\phi(x_2)) \widehat{f}_{k+2}(x_1)\right]\\
%& =&
%\psi(x_2) C + f^U_k(x_1) - \phi(x_2) f^U_{k+1}(x_1)-\psi(x_2) [f^U_{k+1}(x_1) + C]\\
%&  = &f^U_{k}(x_1) -f^U_{k+1}(x_1).
%\end{eqnarray*}
%Since there is a left cusp between sheets $S_k$ and $S_{k+1}$ on the $U$ edge, the result follows.

(2):  When $x_2 \geq R_1/2$ and $x_1 \leq -1/4$, recall that $f_{k+2}^{ST}(x_2) = f^L_{k}(x_2)$ and $\psi(x_2) = 1-\phi(x_2)$ to easily verify that the formulas for $G_{k+2}$ and $H$ agree.  Similarly, check that $G_{k+1} =H$ when $x_{2} \leq -R_1/2$ and $ x_1 \leq -1/4$.

(3):  For $x_1 \geq R_1$, $\widehat{f}_{i}(x_1) = f^U_{i}(x_1) = f^D_{i}(x_1)$ for $i=k+1$ or $k+2$.  This allows us to evaluate
\[
G_{k+1}(x_1,0) = (1-\phi(x_1)) f^{ST}_{k+1}(0) + \phi(x_1) f^R_{k+1}(0) + \frac{1}{2}f^D_{k+2}(x_1) + \frac{1}{2}f^U_{k+1}(x_1)
\] 
and
\[
G_{k+2}(x_1,0) = (1-\phi(x_1)) f^{ST}_{k+2}(0) + \phi(x_1) f^R_{k+2}(0) + \frac{1}{2}f^D_{k+1}(x_1) + \frac{1}{2}f^U_{k+2}(x_1).
\]
Since there is the $R$ edge is (1Cr) with sheets $S_{k+1}$ and $S_{k+2}$ crossing, we have $f^R_{k+1}(0)=f^R_{k+2}(0)$ (as in Proposition \ref{prop:PV1Cr2Cr} (1)), and by definition $f^{ST}_{k+1}(0) = f^{ST}_{k+2}(0)= f^L_k(0)$.  In addition, $U$ and $D$ have the same edge type, so the result follows.

To see that $G_{k+1,k+2}(x_1,x_2) \neq 0$ when $x_2 \neq 0$ and $x_1 \leq -1/4$, we compute
\begin{align*}
G_{k+1,k+2}=  & f^{ST}_{k+1,k+2}(x_2)+  \\ 
  & (1-\phi(x_2))\left\{\begin{array}{cr} \widehat{f}_{k+2}(x_1) - \widehat{f}_{k+1}(x_1), & x_2 \geq 0 \\
f^U_{k+2}(x_1)- f^U_{k+1}(x_1), & x_2 \leq 0,  \end{array} \right.  + \phi(x_2) 
\left\{\begin{array}{cr} f^U_{k+1}(x_1) - f^U_{k+2}(x_1), & x_2 \geq 0 \\
\widehat{f}_{k+1}(x_1)- \widehat{f}_{k+2}(x_1), & x_2 \leq 0,  \end{array} \right.
\end{align*}
%We show that when $x_2>0$, with a similar argument applying to show $G_{k+1,k+2} <0$ for $x_2<0$.   since 
When $x_2 >0$,  we have $f^{ST}_{k+1,k+2}(x_2) \geq 0$  by (\ref{eq:STk2114})-(\ref{eq:STk1341}) and (\ref{eq:flstc0}); 
 $\phi(x_2)>1-\phi(x_2)$; and $f^U_{k+1}(x_1) -f^U_{k+2}(x_1) \geq \widehat{f}_{k+1}(x_1) - \widehat{f}_{k+2}(x_1)$ by (\ref{eq:estfk1hat}).  This shows that $G_{k+1,k+2}(x_1,x_2) >0$.  
When $x_2< 0$, similar reasoning gives $G_{k+1,k+2}(x_1,x_2) <0$.
\end{proof}

\subsubsection{Properties of defining functions} \label{sec:PropGk} We record some properties of $G_{k}, G_{k+1}, G_{k+2}$, $H$ and their differences for later use.    In the following, we let $K$ denote the closed subset
\[
K = \{(x_1,x_2) \, | \, -3/8  \leq x_1 \leq -1/4,  \, -1/4 \leq x_2 \leq 1/4 \} \setminus \mbox{Int}(O_1).
\]
%denote the closed subset of $\widetilde{L}$ where the functions $G_l$ are defined.
%\medskip

\begin{lemma}
\label{lem:Bprops}
The following properties of $G_{k}, G_{k+1}, G_{k+2}$, $H$ and their differences hold.

\begin{enumerate}

\item[B1]  For any $l \in \{k,k+1,k+2\}$, we have
\[
\| H - G_l \|_{C^0(K)} \leq 6 N \epsilon_1
\]
and
\[
\|\partial_{x_2}( H - G_l) \|_{C^0(K)} \leq 18 N \epsilon_1.
\]

\item[B2] For all $(x_1,x_2) \in K$,
%\{x_1 \geq -3/8\}\setminus O_1$, 
\[
\partial_{x_1} G_{k,k+1}(x_1,x_2) \geq  0;  \quad  \partial_{x_1} G_{k,k+2}(x_1,x_2) \geq 0; \quad \mbox{and  } \mbox{sgn}(\partial_{x_1} G_{k+1,k+2}(x_1,x_2)) = \mbox{sgn}(x_2).
\] 
Moreover, in the first (resp. second) inequality, we have equality only when $(x_1,x_2)$ belongs to the upper (resp. lower) half of the cusp locus.

\item[B3]   Along the crossing locus, $\{(x_1,0) \, |\, -3/8 + R_1 \leq x_1 \leq -1/4 \}$,
\[
\partial_{x_2} G_{k+1,k+2}(x_1,0) >0.
\]

\item[B4]  For $(x_1,x_2) \in K$ with $x_2 \geq R_1/2$, $\partial_{x_2} G_{k,k+1}(x_1,x_2) \leq 0$; and  \\ for $(x_1,x_2) \in K$ with $x_2 \leq -R_1/2$,  $\partial_{x_2} G_{k,k+2}(x_1,x_2) \geq 0$.
\end{enumerate}
\end{lemma}

\begin{proof} (B1) For $l \in \{k,k+1,k+2\}$ and $(x_1,x_2) \in K$, we have 
\[
H(x_1,x_2) - G_l(x_1,x_2) = 
\]
\[
f^L_k(x_2) - f^{ST}_l(x_2) + \psi(x_2) \widehat{f}_{k+1}(x_1) + [1-\psi(x_2)]f^U_{k+2}(x_1) - [1-\phi(x_2)]a_1(x_1) - \phi(x_2) a_2(x_1)
\]
where $a_1$ and $a_2$ are each one of $f^U_{k}, f^U_{k+1}, f^U_{k+2}, \widehat{f}_{k+1},$ or $\widehat{f}_{k+2}$ depending on $l$ and the location of $(x_1,x_2)$.  
Corollary \ref{cor:summary} (3) and the definitions of $\widehat{f}_{k+1}, \widehat{f}_{k+2}$  imply that all five of these functions are bounded in absolute value by $N \epsilon_1$ when $ -1 \leq x_1 \leq -1/4.$  
Moreover, each of $\psi$, $1-\psi$, $\phi$, and $1-\phi$ are bounded by $1$.  
Thus, the sum of the last $4$ terms is bounded by $4N \epsilon_1$.  
In addition, (\ref{eq:STk1434}),  (\ref{eq:STk214}), and (\ref{eq:STk134})
imply that
$f^L_k(x_2) -f^{ST}_l(x_2) = \psi(x_2) [ \widehat{f}_{k+2}(-3/8) - f^U_{k+1}(-3/8)]$ if $l =k$ and is $0$ when $l=k+1$ or $k+2$.  
Thus, $|f^L_k(x_2) - f^{ST}_l(x_2)| \leq 1 \cdot[ N \epsilon_1 + N\epsilon_1]$, so the triangle inequality gives
\[
|H(x_1,x_2) - G_l(x_1,x_2)| \leq 6N \epsilon_1.
\]  

To obtain the second inequality, compute the derivative and estimate
\[
|\partial_{x_2}(H-G_l)(x_1,x_2)| \leq 
\]
\[
|d_x(f^L_k-f^{ST}_l)(x_2)|  + |d_x\psi(x_2)|\cdot \left(|\widehat{f}_{k+1}(x_1)|+ |f^U_{k+2}(x_1)|\right) + |d_x\phi(x_2)|\cdot \left(|a_1(x_1)| + |a_2(x_1)|\right) \leq 
\]
\[
|d_x\psi(x_2)|\cdot\left(|\widehat{f}_{k+2}(-3/8)| +|f^U_{k+1}(-3/8)|\right) + 3(2N \epsilon_1) + 3 (2N \epsilon_1) \leq 18N\ea,
\]
where we used that $|d_x\psi(x_2)|, |d_x\phi(x_2)| \leq 3,$ see (\ref{eq:InterpolatePhi}), (\ref{eq:psibounds}). 

(B2)  In $K$,
\[
G_{k,k+1}(x_1,x_2) = f^{ST}_k(x_2) - f^{ST}_{k+1}(x_2) + 
\]
\[
f^U_k(x_1) - (1-\phi(x_2)) \left\{ \begin{array}{cr} \widehat{f}_{k+2}(x_1), & \mbox{if $x_2 \geq 0$}  \\ f^D_{k+2}(x_1), & \mbox{if $x_2 \leq 0$}  \end{array} \right. - \phi(x_2) \left\{ \begin{array}{cr} f^U_{k+1}(x_1), & \mbox{if $x_2 \geq 0$}  \\ \widehat{f}_{k+1}(x_1), & \mbox{if $x_2 \leq 0$} .  \end{array} \right.
\]
Thus,
\[
\partial_{x_1}G_{k,k+1}=(1-\phi(x_2))\left\{ \begin{array}{cr} d_x(f^U_k-\widehat{f}_{k+2})(x_1), & \mbox{if $x_2 \geq 0$} \\ d_x(f^U_k-f^D_{k+2})(x_1), & \mbox{if $x_2 \leq 0$} \end{array} \right. + \phi(x_2) \left\{ \begin{array}{cr} d_x(f^U_k-f^U_{k+1})(x_1), & \mbox{if $x_2 \geq 0$} \\ d_x(f^U_k-\widehat{f}_{k+1})(x_1), & \mbox{if $x_2 \leq 0$}. \end{array} \right.
\]
By %the construction of the $f^U_i= f^D_i$ 
Proposition \ref{prop:CuDef} applied to the $f^{U}_i=f^{D}_i$ and 
(\ref{eq:Est21}), all $4$ of the derivatives that appear in the $2$ piecewise defined terms of $\partial_{x_1} G_{k,k+1}$ are non-negative when $-3/8 \leq x_1 \leq -1/4$.  Moreover, of these $4$ terms only the $2$ that are used when $x_2 \geq 0$  can vanish, and this happens precisely when $x_1 = -3/8$.  The statement concerning $G_{k,k+1}$ follows.

The statement about $G_{k,k+2}$ is established in a similar manner.  In fact, $\partial_{x_1} G_{k,k+2}$ is nearly identical to $G_{k,k+1}$  except that the upper and lower rows in the piecewise definitions are interchanged.

Next, observe that $G_{k+1,k+2}(x_1,-x_2) = -G_{k+1,k+2}(x_1,x_2)$.  Thus to show $\sgn(\partial_{x_1}G_{k+1,k+2}) = \sgn(x_2)$, it suffices to show that $\partial_{x_1}G_{k+1,k+2}(x_1,x_2) >0$ when $x_2>0$.  To this end, for $x_2 >0$ we compute
\[
\partial_{x_1} G_{k+1,k+2} = (1-\phi(x_2)) d_x(\widehat{f}_{k+2}(x_1) - \widehat{f}_{k+1}(x_1))(x_1) + \phi(x_2)d_x(f^{U}_{k+1}-f^U_{k+2})(x_1)
\]
which is positive on $K \cap \{x_2>0\}$ by (\ref{eq:phix2}).
% from the definition of $\widehat{f}_{k+1}$ and $\widehat{f}_{k+2}$.

(B3)  When $-3/8 + R_1 \leq x_1 \leq -1/4$, $\widehat{f}_{k+1}(x_1) = f^U_{k+1}(x_1)$ and $\widehat{f}_{k+2}(x_1) = f^U_{k+2}(x_1).$
Near $x_2 = 0,$ $f^{ST}_{k+1}(x_2) = f^{ST}_{k+2}(x_2).$ So we compute
\[
\partial_{x_2}G_{k+1,k+2}(x_1,0) = \phi'(0)\cdot 2(f^U_{k+1}(x_1)-f^U_{k+2}(x_1)) >0.
\]

(B4)  We prove the first statement with the second statement proved by similar considerations.  Note that when $x_2 \geq R_1/2$ we have $\psi(x_2) = 1-\phi(x_2)$, and $f^{ST}_k(x_2) -f^{ST}_{k+1}(x_2)= \psi(x_2)[\widehat{f}_{k+2}(-3/8) -f^U_{k+1}(-3/8)]$.  These observations allow us to compute
\[
\partial_{x_2} G_{k,k+1} = \phi'(x_2)\left[\widehat{f}_{k+2}(x_1)- f^U_{k+1}(x_1) -(\widehat{f}_{k+2}(-3/8)- f^U_{k+1}(-3/8))  \right] \leq 0.
\] 
The inequality is deduced as follows: The multiplier $\phi'(x_2)$ is non-negative; the term in brackets is $0$ when $x_1 = -3/8$, so using (\ref{eq:Est21}) it follows that this term is non-positive for all $-3/8 \leq x_1 \leq -1/4$. 
\end{proof}

\subsection{The formula over the square}
\label{ssec:FormulaST}
 
We now complete the construction of the Type (13) squares.  
The sheets $S_k, S_{k+1}, S_{k+2}$, and $\tilde{S}_k$ that meet at the swallowtail point are to be defined as the union of graphs of functions $F_{k},F_{k+1}, F_{k+2}, \tilde{F}_k$  that will be obtained by interpolating between the defining functions of Section \ref{sec:funcO2}  and a suitable constant multiple of  those from Section \ref{sec:funcO1}.

Using polar coordinates, $(r,\theta)$, centered at $(-3/8,0)$, consider the subsets of the annulus $A=O_2 \setminus O_1$, 
\begin{eqnarray}
\label{eq:V1V2}
V_1 &= &\{(r,\theta) \, |\,  R_1 \leq r \leq R_2, \,\, \pi/3 \leq\theta \leq \pi/2 \}, \\
\notag
 V_2 &= &\{(r,\theta) \, |\,  R_1 \leq r \leq R_2, \,\, -\pi/2 \leq\theta \leq -\pi/3 \},
\end{eqnarray}
as in Figure \ref{fig:V1V2}.
Note that the closure of $A \setminus V_1$ (resp. $A \setminus V_2$) in the domain of $G_{k,k+1}$ (resp. $G_{k,k+2}$) is a compact set where $G_{k,k+1}$ (resp. $G_{k,k+2}$) is positive by Lemma \ref{lem:GkGk1} (1) (because we are away from the cusp edge). 
Thus by compactness, there exist constants $C_1,C_2 >0$ such that
\begin{equation} \label{eq:AV1Gk}
 \forall (x_1,x_2) \in A \setminus V_1, \quad  G_{k,k+1}(x_1,x_2) >C_1, \quad  \mbox{ and } \quad \forall (x_1,x_2) \in A \setminus V_2, \quad G_{k,k+2}(x_1,x_2)> C_2.
\end{equation}
Next, using Lemma \ref{lem:Aprops}.A4 and Lemma \ref{lem:Bprops}.B3 and continuity, we can find a neighborhood, $V_3$, of the crossing locus in $O_2\setminus O_1$ such that the directional derivatives in the $\theta$ direction satisfy
\begin{equation}  \label{eq:x1x2V3}
\forall (x_1,x_2) \in V_3,  \quad \partial_{\theta} A_{k+1,k+2}(x_1,x_2) > 0,  \quad \mbox{and} \quad \partial_{\theta} G_{k+1,k+2}(x_1,x_2) > 0.
\end{equation}
Within $A$, $G_{k+1,k+2}$ vanishes only at the crossing locus by Lemma \ref{lem:GkGk1} (3), so we can find $C_3 >0$ such that 
\begin{equation} \label{eq:AV3}
\forall (x_1,x_2) \in A \setminus V_3,  \quad |G_{k+1,k+2}(x_1,x_2)| > C_3. 
\end{equation}

\begin{figure}

\labellist
\small
%\pinlabel $1/4$ [t] at 338 21
\pinlabel $V_1$ [tr] at 188 228
\pinlabel $V_2$ [br] at 188 94
\pinlabel $V_3$ [l] at 330 162
\endlabellist
\centerline{ \includegraphics[scale=.3]{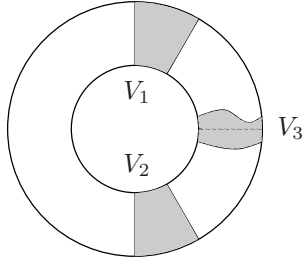} }

%%\centerline{ \includegraphics[scale=.8]{images/SubdivideSq} } %Dan had already commented out this line

\caption{The closed subsets $V_1$, $V_2$, and $V_3$.
}
\label{fig:V1V2}
\end{figure}

Using compactness of $O_2$, we can now fix $C>0$ such that for $(i,j) = (k,k+1), (k,k+2)$, or $(k+1,k+2)$ and $l = k,k+1$, or $k+2$, we have
\begin{equation} \label{eq:CAijCO2}
C\|A_{i,j}\|_{C^0(O_2)} \leq \mbox{Min}\{C_1,C_2,C_3\};
\end{equation}
\begin{equation} \label{eq:CACO2}
   C\|A_l\|_{C^0(O_2)}, C\|B\|_{C^0(O_2)} < \epsilon_1; \quad \mbox{and} \quad \mbox{Max}\{C\|\partial_{x_m}A_l\|_{C^0(O_2)}, C\|\partial_{x_m}B\|_{C^0(O_2)} \, | \, m=1,2 \} < \epsilon_1;  
\end{equation}
\begin{equation} \label{eq:CACO3}
C < \frac{1}{2(2N+1)(M+M^{4/3})}, \quad \quad \mbox{where $M$ is from Section \ref{sec:Straight}.}
\end{equation}
(In the preceding formulas, by abuse of notation,  $\| \cdot \|_{C^0(O_2)}$ indicates the sup norm of the functions on the intersection of their domain with $O_2$.)  

  In addition, we fix a non-decreasing cutoff function $\alpha: [0,+\infty) \rightarrow \R$ such that $\alpha(r) = 0$ for $r\leq R_1$ and $\alpha(r) =1$ for $r \geq R_2$ such that 
	\begin{equation} \label{eq:alphabound}
	0 \leq \alpha'(r) \leq 33.
\end{equation}

With these preparations complete, we make the definitions for $l \in \{k,k+1,k+2\}$,
\begin{eqnarray}
\label{eq:FHCAG}
F_l(x_1,x_2) & = & (1-\alpha(r))\cdot \left( H(x_1,x_2) + C\,A_l(x_1,x_2)\right) + \alpha(r) \cdot G_{l}(x_1,x_2), \\
\notag
\widetilde{F}_k(x_1,x_2) & = & (1-\alpha(r))\cdot \left(H(x_1,x_2) + C\,B(x_1,x_2)\right) + \alpha(r) \cdot H(x_1,x_2).
\end{eqnarray}
Recall $A_l, B, G_l, H$ are defined in (\ref{eq:Aiai}), (\ref{eq:Gkdef})-(\ref{eq:Hdef}).
For $i \notin \{k,k+1,k+2\}$ we define
\begin{equation}  \label{eq:FifLfR}
F_i(x_1,x_2) = f^U_i(x_1) + (1-\phi(x_1)) f^L_{\sigma_L(i)}(x_2)+ \phi(x_1) f^R_{i}(x_2). 
\end{equation}

We form the front projection of a Legendrian to be used for the Type (13) square, as follows:  Take the union of the graphs of $F_{k}, F_{k+1}, F_{k+2}$ in the region to the right of the cusp locus of $L''_{st}$   with the graph of $\widetilde{F}_k$ to the left of the cusp locus of $L''_{st}.$ 

\begin{lemma} \label{lem:smoothST}
The graphs of the $F_i$ and $\widetilde{F}_k$ fit together to form the front projection of a smooth Legendrian with a single swallowtail point.  The base projection of the cusp locus agrees with that of the Legendrian $L''_{st}$ from Section \ref{sec:funcO1}.
\end{lemma}
\begin{proof}
 The graphs of the $A_l$ and $B$ as well as the $G_l$ and $H$ fit together to form front projections of smooth Legendrians (by construction of the $A_l$ and by Lemma \ref{lem:GkGk1}).  In addition, these two collections of defining functions have the same cusp locus (in the base projection) in their common domain.  Thus, when we interpolate to form the $F_l$ and $\widetilde{F}_k$, the result is again a smooth Legendrian with the same cusp locus. 
 %  In addition, we include the graphs of the $F_i$ with $i \notin \{k,k+1,k+2\}$ as the remaining sheets of $L_{(13)}$.
Moreover, up to scaling and the addition of a common function, the sheets defined by $F_l$, $l=k,k+1,k+2$ and $\widetilde{F}_k$ agree with $L''_{st}$  within  $O_1$ (where $\alpha(r) =0$), and hence have a single swallowtail point.  
\end{proof}

\section{Proof of Theorem \ref{thm:PropertiesofLtilde} Part 3: Verification of Properties}
\label{sec:Properties}

Let us recall where we are.
We wish to construct $\tilde{L}_0$ satisfying all of the properties of Theorem \ref{thm:PropertiesofLtilde} except possibly the 1-regular condition.  
Over each square of $\mathcal{E}_\pitchfork$, we take the front projection of $\tilde{L}_0$ to be the union 
 of graphs of the functions defined using equation (\ref{eq:interpolating})  for (1)-(12) square types, or using equation (\ref{eq:FHCAG}) for square (13) (and analogous for (14)).

In Section \ref{ssec:gluing}, we 
%confirm  Properties ??? and in the process 
verify 
that $\tilde{L}_0$ pieces together over adjacent squares to form a smooth, globally defined Legendrian, and fix a metric from Construction \ref{construct:metric}  for computation of gradient vector fields. 
%and prove a lemma reducing some of the (13)-(14) case to the (1)-(12) case.
In Section \ref{ssec:Fitting}, we confirm that the singular set of $\tilde{L}_0$ above each square is as prescribed by the square types (1)-(14), and verify properties of the Reeb chords.  Properties 1, 4, 5, 6, 8, and 9 are proved in this section.
In Section \ref{ssec:PropertiesGFTs}, we verify the remaining Properties 2, 3, 7, and  10-19, which concern gradient vector fields of local difference functions.
%offer descriptions of the GFTs. 
%The proofs in Sections \ref{ssec:Fitting} and  \ref{ssec:PropertiesGFTs} are local, focusing on one square at that time. 

\subsection{Gluing the square models together}
\label{ssec:gluing}

For each square type, the defining functions correspond in an obvious way to the sheets of the Legendrian.  Explicitly, for squares with types (1)-(12) the function $F_i$ corresponds to sheets $S_i$ that are labeled in descending order as they appear above $(+1,+1)$;  for the Type (13) square functions $F_{i}$ correspond to sheets $S_i$ which, when $i=k,k+1,k+2$, are only defined to the right of the cusp locus, while the function $\widetilde{F}_k$ corresponds to the sheet $\tilde{S}_k$ defined to the left of the cusp locus.

\begin{proposition}  \label{prop:cornerC}
Fix any square type.

\begin{enumerate}

\item Assume there are $n_c$ sheets defined above the corner $c = ({c_1}, {c_2})$.  (Here, $c_i \in \{\pm 1\}$).  Suppose that for some $1 \leq j \leq n_c$, the defining function $F$ corresponds to the sheet that appears in position $j$ above $c$.  Then, for $|x_i -  c_i| \leq 1/16$, 
\begin{equation} \label{eq:cornerC}
F(x_1,x_2) = (n_c-j) \ec((x_1-c_1)^2+ (x_2-c_2)^2 + 2).
\end{equation}

\item Assume that along the edge $E$ where $x_2 = e$ equals $+1$ or $-1$, 
$n_\pm$ sheets exist above $(\pm1, e)$.  Suppose that for some $1 \leq j \leq n_+$, the defining function $F$ corresponds to the sheet above $E$ that appears in position $j$ (resp. $\sigma_-(j)$) at $(+1, e)$ (resp. at $(-1,e)$).  Then, for $|x_2 - e| \leq 1/16$ we have 
\begin{equation}  \label{eq:edgeE}
F(x_1,x_2) = (1-\phi(x_1))(n_- - \sigma_-(j))\ec ( (x_2-e)^2 +1) + \phi(x_1)(n_+ - j)\ec ( (x_2-e)^2 +1) + f_j(x_1)
\end{equation}
with $f_j$ the $1$-dimensional function associated to the edge type of $E$.  A similar formula holds along edges of the form $x_1 = e$ equals $+1$ or $-1$.

\end{enumerate}   
\end{proposition}

Note that in (2), if $E$ is a $(Cu)$ edge type with $j=k$ or $k+1$, then we take $\sigma_-(k) = \sigma_-(k+1) = k-0.5$ as usual.

\begin{proof}
Both statements follow from the formulas (\ref{eq:interpolating}) and (\ref{eq:FHCAG}), together with the standard form of the $1$-dimensional functions near $x = \pm1$ stated in Corollary \ref{cor:summary} (4) (and variants for the functions $f_{k-0.5}$, $f^{ST}_l$, and $\widehat{f}_{k+1}$). 

In more detail, consider the $U$ edge of a (1)-(12) square type, and suppose $|x_2-1| \leq 1/16$.  Then, (\ref{eq:interpolating}) simplifies to 
\[
F_j(x_1,x_2) = f^U_j(x_1) + (1-\phi(x_1)) f^L_{\sigma_L(j)}(x_2) + \phi(x_1) f^R_j(x_2).
\]   
Note that $\sigma_L$ agrees with $\sigma_-$ as associated to the edge type of $U$, so the formula follows from Corollary \ref{cor:summary} (4)  (and Lemma \ref{lem:MainProps0.5} (4) in the case that $U$ is a (Cu) edge and $j =k$ or $k+1$).  For the $D$ edge, when $|x_2-(-1)| \leq 1/16$,
\[
F_i(x_1,x_2) = f^D_{\sigma_D(i)}(x_1) + (1-\phi(x_1)) f^L_{\sigma_L(i)}(x_2) + \phi(x_1) f^R_i(x_2).
\]
Since $F_i$ corresponds to the sheet that appears in position $\sigma_D(i)$ above $D$ the $f^D_{\sigma_D(i)}(x_1)$ term is consistent with (\ref{eq:edgeE}).  The remaining two terms are seen to have the desired form via Corollary \ref{cor:summary} (4)  (and Lemma \ref{lem:MainProps0.5} (4) if necessary).

If in addition, $|x_1 - (\pm 1)| \leq 1/16$, then (\ref{eq:edgeE}) reduces to (\ref{eq:cornerC}) after another application of Corollary \ref{cor:summary} (4).

To conclude the proof we consider the (13) square type where the functions $F_l$, $l=k,k+1,k+2$ and $\widetilde{F}_k$ are defined by (\ref{eq:FHCAG}).  When $|x_i -(\pm1)| \leq 1/16$ for $i =1$ or $2$, 
 $\alpha(r) =1$ in (\ref{eq:FHCAG}), so for $l = k,k+1,k+2$, we have $F_{l} = G_l$ and $\widetilde{F}_k= H$. 
Observe that Lemma \ref{lem:flst} together with (\ref{eq:STk2114})-(\ref{eq:STk1341}) shows that, letting $n_L$ be the number of sheets to the left of the cusp locus,  for $|x_2- 1| \leq 1/16$, 
 \[
 \mbox{$f^{ST}_k(x_2) = f^{ST}_{k+1}(x_2) = (n_L - (k-0.5)) Q_+(x_2)$,    and    $f^{ST}_{k+2}(x_2) = (n_L - k) Q_+(x_2)$;}
 \]
 and for $|x_2-(-1)| \leq 1/16$,
 \[
 \mbox{$f^{ST}_k(x_2) = f^{ST}_{k+2}(x_2) = (n_L - (k-0.5)) Q_-(x_2)$,    and    $f^{ST}_{k+1}(x_2) = (n_L - k) Q_-(x_2)$.}
 \]
 Moreover, in (\ref{eq:Gkdef})-(\ref{eq:Gk2def}) and (\ref{eq:Hdef}) when $|x_2 - (\pm1)| \leq 1/16$ all appearances of the functions $\widehat{f}_{k+1}$ and $\widehat{f}_{k+2}$ are with $0$ coefficients.  It follows that $F_{l}$ and $\widetilde{F}_k$ have the required form near the edges $U$ and $D$.   

Near the $R$ edge, when $|x_1-1| \leq 1/16$,   in (\ref{eq:Gkdef})-(\ref{eq:Gk2def}), we have $\phi(x_1)=1$; and  $\widehat{f}_{k'}(x_1) = f^U_{k'}(x_1)= f^D_{k'}(x_1)$ for $k' = k+1, k+2$ so that $F_l(x_1, x_2) = f^R_{l}(x_2) + (1-\phi(x_2)) f^D_{\sigma_D(l)}(x_2) + \phi(x_2) f^U_{l}(x_1)$ as with all the other square types.  
Finally, near the $L$ edge when $|x_1-(-1)| \leq 1/16$, we have $\widehat{f}_{k+1}(x_1) = f^U_{k+2}(x_1)$ by (\ref{eq:hatk1}), so that 
\[
\widetilde{F}_k(x_1,x_2) = H(x_1,x_2) = f^U_{k+2}(x_1) + f^L_k(x_2) = f^D_{k+2}(x_1) + f^L_k(x_2).
\]

\end{proof}

Recall from Section \ref{sssec:FlowLinesBoundary} that we have for any $0$-cell a neighborhood, $N(e^0_\alpha)$.  Using polar coordinates in a smooth chart centered at $e^0_\alpha$, $N(e^0_\alpha)$ is the subset $\{(r,\theta) \, | \, 0 \leq r \leq 1/16\}$.

\begin{proposition} \label{prop:smooth0}
Suppose that $n_c$ sheets exist above $e^0_\alpha$.  Then, in coordinates on $N(e^0_\alpha)$, the local defining functions of $\tilde{L}_0$ are given by
\begin{equation}  \label{eq:Fmrtheta}
F_i(r, \theta) = (n_c-i) \ec(r^2 + 2),  \quad \mbox{for $i= 1, \ldots, n_c$}
\end{equation}
\end{proposition}

\begin{proof}
The intersection of $N(e^0_\alpha)$ with a square containing $e^0_\alpha$ is a ball of radius $1/16$ centered at the corner of $[-1,1] \times [-1,1]$ corresponding to $e^0_\alpha$, and the distance from the corner agrees with the coordinate $r$ on $N(e^0_\alpha)$.  (See Section \ref{sec:RegReq} (2).)  Thus,  (\ref{eq:Fmrtheta}) follows from Proposition \ref{prop:cornerC} (1).
\end{proof}

Any $1$-cell, $e^1_\alpha$, appears as an edge of precisely two $2$-cells of $\mathcal{E}_\pitchfork$.  Recall that the regularity requirement (3) from Section \ref{sec:RegReq} allows us to combine the coordinates of these bordering $2$-cells to smoothly parametrize a neighborhood of $e^1_\alpha$ by $(-1,1) \times (-1/16,1/16)$.  (This neighborhood consists of the portion of the interior of the bordering $2$-cells that lies within $1/16$ from $e^1_\alpha$.) 

\begin{proposition}  \label{prop:smooth1}
Let $f_i$ be the one dimensional defining functions associated to the edge type of $e^1_\alpha$, and let $n_\pm$ and $\sigma_-$ be as in Proposition \ref{prop:cornerC}.   Above the subset $(-1,1) \times (-1/16,1/16)$ the defining functions for $\tilde{L}_0$ have the form
\begin{equation} \label{eq:smooth1eq}
F_i(x_1,x_2) = f_i(x_1) + \big[(1-\phi(x_1))(n_- - \sigma_-(i))+ \phi(x_1)(n_+ - i) \big]\ec ( x_2^2 +1) 
\end{equation} 
for $1 \leq i \leq n_+$.

Moreover, the critical points of $F_{i,j}$ located along $(-1,1)\times \{0\}$ are precisely the points $(x_1,0)$ with $x_1$ a critical point of $f_{i,j}$.
\end{proposition}
\begin{proof}
When parametrizing by $(-1,1) \times(-1/16,1/16)$, the $[-1,1]\times [-1,1]$ parametrizations of the two bordering squares are shifted (and possibly reflected) so that the edge corresponding to $e^1_\alpha$ is at $x_2=0$.  Thus,  $x_2^2$ in the current formula corresponds to $(x_2-e)^2$ in  (\ref{eq:edgeE}), and (\ref{eq:smooth1eq}) follows from Proposition \ref{prop:cornerC} (2).

For the statement concerning critical points of $F_{i,j}$, note that 
\[
\partial_{x_2} F_{i,j}(x_1,0) \equiv 0 \quad \mbox{and} \quad (\partial_{x_1}F_{i,j})|_{[-1,-1/4]\cup[1/4,1]} = d_x f_{i,j}(x_1),
\]
while for $x_1 \in [-1/4,1/4]$ both $\partial_{x_1}F_{i,j}(x_1,0)$ and $d_x f_{i,j}(x_1)$ are non-vanishing.  [For $d_x f_{i,j}(x_1)$, this follows from item (2) of Propositions \ref{prop:PV1Cr2Cr} and \ref{prop:CuDef}.  For $\partial_{x_1} F_{i,j}(x_1,0)$,  
using item (5) of the same propositions (with the  values of the $y_{i,j}$ found above Proposition \ref{prop:PV1Cr2Cr}) we see that either:  
\begin{itemize}
\item[(a.)]  $d_x f_{i,j}(x_1) > 1/8$ and therefore dominates the other term of 
\begin{equation} \label{eq:Fijijijj}
\partial_{x_1}F_{i,j}(x_1,0) = d_xf_{i,j}(x_1) + d_x\phi(x_1)\left[( j - i)-(\sigma_-(j) - \sigma_-(i))\right] \ec
\end{equation}
 in absolute value.  (Using (\ref{eq:InterpolatePhi}) the second term is bounded by $3\cdot2N\ec$.) 
\item[(b.)] Or, $e^1_\alpha$ is a (1Cr) edge with $S_i$ and $S_j$ crossing above $L$.  In this case, the second term becomes $d_x\phi(x_1)2\ec$, so that both terms in   (\ref{eq:Fijijijj}) have the same  sign.
\end{itemize}
\end{proof}

\begin{corollary}
The Legendrian $\tilde{L}_0$ formed by piecing together the local defining functions assigned to the square types (1)-(14) above the transverse square decomposition $\mathcal{E}_\pitchfork$ is a smooth Legendrian submanifold of $J^1(S)$.  Moreover, there exists a Riemannian metric on $S$ such that the gradients of all local difference functions agree with the Euclidean gradients computed in $[-1,1]\times[-1,1]$ coordinates on each square.
\end{corollary}
\begin{proof}
Smoothness follows from Propositions \ref{prop:smooth0} and \ref{prop:smooth1}, since the interiors of the given neighborhoods of $0$-cells and $1$-cells, together with the $2$-cells themselves form an open cover of $S$.  In each of these regions, the coordinate formulas show that the front projection of $\tilde{L}_0$ is smooth except at cusp edges and swallowtail points which have been arranged to be standard so that the Legendrian $\tilde{L}_0$ is smooth in $J^1(S)$.  (Smoothness in the interior of Type (13)-(14) squares was established in Lemma \ref{lem:smoothST}.)

Proposition \ref{prop:smooth0} shows that Construction \ref{construct:metric} may be applied to produce the required metric on $S$. 
\end{proof}

\subsection{Properties involving Reeb chords, cusp loci and crossing loci.}
\label{ssec:Fitting}

%\dr{Consider the 14 possible square. Combinatorially, they are uniquely characterized by what the sheets are suppose to do (cross, cusp, etc) on the four edges.} 
%In particular, the four collections $\{f^X_i\}_i$ for $X = U,D, R,L$ contain this combinatorial data.

\begin{proposition}
\label{prop:interpolating}

Consider the functions $F_i: [-1,1] \times [-1,1] \rightarrow \R$  defined in (\ref{eq:interpolating}) (for (1)-(12) square types) or (\ref{eq:FHCAG}) (for (13)-(14) square types).  In each case, the singular set of the Legendrian front defined by the $F_i$ (and $\widetilde{F}_k$ in the (13)-(14) case) topologically agrees with the singular set prescribed by the square type.
Moreover the crossing and/or cusp loci (if they exist) satisfy Property \ref{pr:14models}.

\end{proposition}

Note that this proves the first itemized claim of Theorem \ref{thm:PropertiesofLtilde}.

Before we prove this proposition, we spend some time confirming that there are no unintentional Reeb chords.  Locations of Reeb chords along the boundary of a square $I^2$ are determined by Propositions \ref{prop:smooth0} and \ref{prop:smooth1}.  In Lemmas \ref{lem:c-chords} and \ref{lem:quarter-cross}, we identify all Reeb chords in the interior, $\mathit{Int}(I^2)$.
%We then complete the proof in Section \ref{ssec:InterpolatingProof}.

Let 
\begin{eqnarray}
\label{eq:Phi}
\Phi& :=  & \{(x_1,x_2) \,\,|\,\, |x_i| < 1/4\,\, \mbox{for some}\,\, i\} \cap \mathit{Int}(I^2), \\
\label{eq:Psi}
\Psi & := & \{(x_1,x_2) \,\,|\,\, |x_i| < 1/4\,\, \mbox{for both}\,\, i\}.
\end{eqnarray}
Label the four components of $Int(I^2) \setminus \Phi$ as $XY$ where $X \in \{U, D\}$ and $Y \in \{R, L\}$ according to which two edges (up, right, down, left)  of $\partial I^2$ intersect the closure of $XY.$
Label each of the four components of $\Phi \setminus \Psi$ as $\overline{U}, \overline{R}, \overline{D}, \overline{L}$ based on which boundary sits
in the component's closure. See Figure \ref{fig:PsiPhi}.  

\begin{figure}

\quad

\labellist
\small
\pinlabel $\Phi$ at 96 96
\pinlabel $UR$  at 156 156
\pinlabel $DR$  at 156 36
\pinlabel $UL$  at 36 156
\pinlabel $DL$  at 36 36
\pinlabel $\Psi$  at 360 96
\pinlabel $\overline{U}$  at 360 156
\pinlabel $\overline{D}$  at 360 36
\pinlabel $\overline{R}$  at 420 96
\pinlabel $\overline{L}$  at 300 96
\endlabellist
\centerline{ \includegraphics[scale=.6]{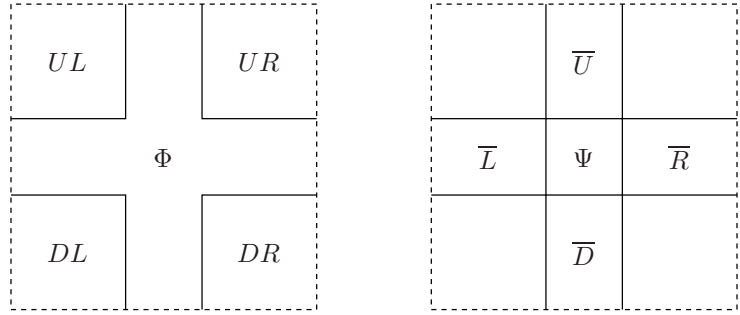} }

\caption{Notations used for different subsets of $[-1,1] \times [-1,1]$.   
%pictured with flow lines from the stable manifolds of $b^U_{i,j}$ and $b^L_{i,j}$.
}
\label{fig:PsiPhi}
\end{figure}

\begin{lemma}
\label{lem:c-chords}
The difference function $F_{i,j}$ for two sheets has a critical point at $(x_1,x_2)$ in the interior $XY$ if and only if $F_{i,j}$ has a critical point on edge $X$  at $(x_1,\pm1)$ and a critical point on edge $Y$ at $(\pm 1, x_2)$.  (This includes the statement that if $F_{i,j}$ is not defined above $X$ or $Y$ due to a cusp edge, then $F_{i,j}$ has no critical points in $XY$.)
\end{lemma}

\begin{proof}
For squares (1)-(12), since $|x_1| \ge  1/4$ and $|x_2| \ge 1/4,$
\begin{equation} \label{eq:XYform}
F_{i,j}(x_1, x_2) = (f^X_{i'}-f^X_{j'})(x_1) + (f^Y_{i''}-f^Y_{j''})(x_2)
\end{equation}
for appropriate $i',j'$ and $i'',j''$.  If sheets $S_i$ and $S_j$ exist above $X$ and $Y$, then the result follows from (\ref{eq:XYform}) and Proposition \ref{prop:smooth1}.   If either $S_i$ or $S_j$ does not exist over $X$, then (i) $Y$ must be a (Cu) edge and (ii) within $XY$ we have $-1 < x_2\leq -1/4$.  From item (2) of Proposition \ref{prop:CuDef} (and Lemma \ref{lem:MainProps0.5}, in the case of the (11) square), it follows that
\[
\partial_{x_2}F_{i,j}(x_1,x_2) = d_x(f^Y_{i''} - f^Y_{j''})(x_2)
\]
is non-zero.  A similar argument applies if $S_i$ or $S_j$ does not exist over $Y$.

For squares (13) and (14), recall $\alpha(r) = 1$ in (\ref{eq:FHCAG}). 
So (\ref{eq:XYform}) remains valid if we allow for the possibility that one or more of the $f$'s above could be 
of type 
%$\widehat{f}$ or 
$f^{ST},$ 
see (\ref{eq:Gkdef})-(\ref{eq:Hdef}).  Above $UL$ and $DL$ where some sheets are not defined and the $f^{ST}$ functions may appear, the $f^U$ and $f^D$ functions are (Cu) functions, so that $\partial_{x_2}F_{i,j}(x_1,x_2) \neq 0$ as above. 

\end{proof}

\begin{lemma}
\label{lem:quarter-cross}  
For any pairs of sheets  $(i,j)$ in any of the 14 squares,  
$F_{i,j}$ has no critical points in $\Phi.$
\end{lemma}

\begin{proof}
Sections \ref{sssec:quarter-cross} and \ref{sssec:quarter-crossST} constitute the proof of this lemma, in which we treat the non-swallowtail and swallowtail cases separately. 

For squares (1)-(12), we show in Lemmas \ref{lem:tildeU},  \ref{lem:tildeD} and \ref{lem:Psi}
 that $F_{i,j},$ for $i < j,$ has no critical points in $\overline{U},$ 
$\overline{D},$ and $\Psi.$
The arguments for $\overline{L}$ and $\overline{R}$ are similar, provided we consider all reflections in the $x_1=x_2$ line of the square types (1)-(12) .
To indicate this, let $(n')$ label the reflected square type $(n).$

We then show in Lemmas \ref{lem:tildeUtildeDtildeRST}  and \ref{lem:PsiST} how slight modifications apply to square (13) 
for $\overline{U}, \overline{R}, \overline{D}, \Psi$. The argument for $\overline{L}$ given in Lemma \ref{lem:tildeLST} is more substantial.
Square (14) is analogous to (13).

\end{proof}

\subsubsection{Proof of Lemma \ref{lem:quarter-cross} for squares (1)-(12)}
\label{sssec:quarter-cross}

First we prove a technical Lemma which will be used to prove Lemma \ref{lem:tildeU}.

\begin{lemma} 
\label{lem:18est}
Fix  $i<j$ and a square (1)-(12).  Then,
\begin{eqnarray*}
\sigma_{L}(i) < \sigma_{L}(j) &\longrightarrow &
\left.\partial_{x_1} F_{i,j}\right|_{[-1/4, 1/4]\times[1/4,3/4]} > 1/5 \\
\sigma_{D}(i) < \sigma_{D}(j) &\longrightarrow &
\left.\partial_{x_2} F_{i,j}\right|_{[1/4, 3/4]\times[-1/4,1/4]} > 1/5,
\end{eqnarray*}
with the single exception of a Type (4) square with $(i,j) = (k,k+1)$ in which case the first inequality fails.
\end{lemma}

\begin{proof}
Say $\sigma_{L}(i) < \sigma_{L}(j)$ and $(x_1, x_2) \in [-1/4, 1/4]\times[1/4,3/4].$
We prove
\[
\partial_{x_1} F_{i,j}(x_1,x_2) = d_x f_{i,j}^U(x_1) +
d_x \phi(x_1)( f_{i,j}^R(x_2) - f_{\sigma_L(i), \sigma_L(j)}^L(x_2)) > 1/5.
\]
We prove the second implication at the same time by  considering the reflected square types ($n'$) as well.

For $X = U,D,R,L,$ let $y^X_{i,j}$ denote the $y_{i,j}$ value which satisfies 
\[
|f^X_{i,j}(x)-y_{i,j}| <  N \ea,  \quad \mbox{for $x\in [1/4,3/4]$}
\] 
as stated in (2) of Propositions \ref{prop:PV1Cr2Cr} and \ref{prop:CuDef}. [The values of $y_{i,j}$ are listed above Proposition \ref{prop:PV1Cr2Cr}.] If $U$ or $R$ is a (Cu) edge we need to allow $i$ and/or $j$ to have the form $k-0.5$ when $X=L$ or $D$.  In these cases, we let $y^X_{i,j}$ denote the appropriate constant $C_{i,k-.5}$ or $C_{k-.5,j}$ from Lemma \ref{lem:MainProps0.5} (2), or $0$ if $i=j = k-0.5$.  
Since $-1/4\le x_1 \le 1/4$ and $1/4 \le x_2 \le 3/4,$ the previous inequality and item (5) of Corollary \ref{cor:summary} and Lemma \ref{lem:MainProps0.5} give
%(and analogous statements in Lemmas  \ref{lem:MainPropsCr}-\ref{lem:MainProps0.5}) imply
\begin{eqnarray*}
&|d_x f^U_{i,j}(x_1)- 2y^U_{i,j}| <N \ea , \quad  |f^R_{i,j}(x_2) - y^R_{i,j}| <  N\ea, &\\
&|f^L_{\sigma_L(i),\sigma_L(j)}(x_2) - y^L_{\sigma_L(i),\sigma_L(j)}| < N \ea.&
\end{eqnarray*}
We have that $\sigma_L(i) < \sigma_L(j)$ implies $y_{i,j}^U \ge 1/8.$  [The only case in which $y_{i,j} < 1/8$ is when sheets $i$ and $j$ cross along a $1$-crossing edge.]
So if $f_{i,j}^R(x_2) - f_{\sigma_L(i), \sigma_L(j)}^L(x_2) \ge 0,$ then the lower bound of $1/5$ holds.
If $f_{i,j}^R(x_2) - f_{\sigma_L(i), \sigma_L(j)}^L(x_2) \le 0,$ then (using also (\ref{eq:InterpolatePhi}))
\begin{eqnarray*}
d_x f_{i,j}^U(x_1) + d_x \phi(x_1)( f_{i,j}^R(x_2) - f_{\sigma_L(i), \sigma_L(j)}^L(x_2)) &\ge&
2y_{i,j}^U - N \ea + (2+\ec)((y^R_{i,j} - N\ea) - (y^L_{\sigma_L(i),\sigma_L(j)} + N\ea))
\\
&\ge& 2(y_{i,j}^U + y^R_{i,j} - y^L_{\sigma_L(i),\sigma_L(j)}) -20 N \ea. 
\end{eqnarray*}
Thus, the lemma follows once we verify
\begin{equation}
\label{eq:URL}
y^U_{i,j} + y^R_{i,j} -y^L_{\sigma_L(i),\sigma_L(j)} \ge 1/8.
\end{equation}

It suffices to show equation (\ref{eq:URL}) for all consecutive pairs, $j=i+1$ since the other cases $j > i+1$ follow.  [If $L_{i,j}$ denotes the left hand side of the inequality, then we have $L_{i,m} + L_{m,j} = L_{i,j}$.]
Moreover,  $y^U_{i,i+1} = y^R_{i,i+1} = y^L_{i,i+1}= 1$ if neither $i$ or $i+1$ is a crossing or cusp sheet, so in this case the left hand side of (\ref{eq:URL}) is $1$.  

For squares (2)-(6), (8)-(9) and (12), suppose the single crossing/cusp edges involve sheets $(S_k,S_{k+1})$ and/or $(S_{k+1},S_{k+2}),$ and the double crossing
involves sheets $(S_k,S_{k+1},S_{k+2})$.  Then, in squares (2), (3), (4), (9) (and their reflections) we need only show equation (\ref{eq:URL}) for pairs $(S_{k-1},S_k)$, $(S_k,S_{k+1})$, $(S_{k+1},S_{k+2})$, and in squares (5), (6), (8), (12) (and their reflections) we need only consider  pairs  $(S_{k-1},S_k)$, $(S_k,S_{k+1})$, $(S_{k+1},S_{k+2})$, $(S_{k+2},S_{k+3})$.

For squares (7), (10), (11) (and their reflections), assume the crossing/cusp involves two pairs of sheets, $S_k,S_{k+1}$ (vertical locus) and $S_l, S_{l+1}$ (horizontal locus). 
By interchanging the role of ($n$) and ($n'$) if necessary,   
%in the one-skeleton \dr{meaning? I'm thinking we can flip everything upside down is the reason... }, 
we may assume without loss of generality that 
$k+1 < l.$  However, this assumption requires us to consider two versions of square (10), denoted (10a) and (10b) where the vertical locus is respectively a cusp locus and a crossing locus.
Moreover, we only consider the case $k+1 = l-1,$ since for all other cases the crossing and cusp sheets are not adjacent so the calculation follows from squares (2) and (9).
%, the corresponding $y^U_{i,j}$ increases while $y^R_{i,j} -y^L_{\sigma_L(i),\sigma_L(j)}$ remains the same.
Thus for figures  (7), (10), (11) (and their reflections) we need to consider pairs of sheets 
$(S_{k-1},S_k), (S_k,S_{k+1}), (S_{k+1},S_{k+2}), (S_{k+2},S_{k+3}), (S_{k+3},S_{k+4}).$

To avoid excessive sub/superscripts, we introduce notation (in the remainder of this proof only) to represent the different $y_{i,i+1}$ values as defined above  Proposition \ref{prop:PV1Cr2Cr}.
\[
a = 0, \,\, z = 1.5, \,\, y_0 = 1.125, \,\, y_1 = 0.625, \,\, y_2 = 0.125, \,\, y_3 =2.125.
\]
The below tables list the left-hand side of equation
(\ref{eq:URL}) for different sheet pairs and different squares.  In all cases where $y^L_{\sigma_L(j), \sigma_L(j)}$ is of the form $C_{i,k-.5}$ or $C_{k-.5,j}$ we use the larger of the two possible values, and marked this term with *.  (We have used $a$ for the $0$'s that  represent
$y_{i,i+1}$ from a (1Cr) edge where $i$ crosses $i+1.$ 
Other $0$'s also appear in $y^L_{\sigma_L(k),\sigma_L(k+1)}$ when sheets $k$ and $k+1$ meet at a cusp above $U$ as in squares (9)-(12).)

$$
\begin{matrix}
\mbox{Square} \setminus \mbox{Sheet Pair} & (S_{k-1},S_k) & (S_k,S_{k+1}) & (S_{k+1},S_{k+2}) \cr
(2) & z + 1-2 & a+1 --1& z + 1-2  \cr
(2') & 1 + z -z & 1 + a - a & 1+z-z  \cr
(3)  & 1+1-z & 1+1-a & 1+1 - z  \cr
(4) & 1+z - 1 & 1+a -1 & 1+z -1  \cr
(4') & z + 1- (z + a) & a+1-({-a}) & z + 1- (a+z)  \cr
(9) & 1+1-1^* & 1+1-0 & 1 + 1 - 1^* \cr
(9') & 1+1 -1 & 1+ 1 -1 & 1+ 1 -1 
\end{matrix}
$$
$$
\begin{matrix}
& (S_{k-1},S_k) & (S_k,S_{k+1}) & (S_{k+1},S_{k+2}) & (S_{k+2},S_{k+3}) \cr
(5) & 1 + y_0 -y_0 & 1 + y_1-y_1 & 1+y_2 -y_2 & 1 + y_3 - y_3 \cr
(5') & y_0 + 1-2 & y_1 + 1 - 1 & y_2+1-({-2}) & y_3 + 1-3 \cr
(6) & 1+y_0 -1 & 1+y_1 - z & 1 + y_2 -a & 1+y_3 -z \cr
(6') & y_0 + 1 - (z+a) & y_1+1 - z & y_2+1 - ({-a-z}) & y_3 +1  - ({a+z+1}) \cr
(8) & z + y_0 - (y_0 + y_1) & a + y_1 --y_1 & z + y_2 - (y_1+y_2) & 1+y_3 - y_3 \cr
(8') & y_0 + z - (1+z) & y_1 + a - a & y_2 +z -(-z-a) & y_3 + 1-3  \cr
(12) & 1+ y_0 - 1^* & 1+ y_1 - 0 & 1 + y_2 - 1^* & 1 + y_3 - 1 \cr
(12') & y_0 + 1 - (1+1) & y_1+1-1& y_2+1 - (-1-1) & y_3 +1 - 3
\end{matrix}
$$
$$
\begin{matrix}
 & (S_{k-1},S_k) & (S_k,S_{k+1}) & (S_{k+1},S_{k+2}) & (S_{k+2},S_{k+3}) & (S_{k+3},S_{k+4}) \cr
(7) & z + 1 - 2 & a +1 - (-1) & z + z - ({z+1}) & 1 + a -a & 1 + z - z \cr
(7') & 1 + z - z & 1 +a - a & z + z - (z+1) & a + 1 -(-1) & z + 1 - 2 \cr
(10a) & 1 + 1 - z^* & 1+1-0 & 1 + z - z^* & 1 + a -a  & 1 + z - z \cr
(10a') & 1 + 1 - 1 & 1 + 1 -1 & z + 1 - 2 & a + 1 - (-1) & z + 1 -  2  \cr
(10b) & z + 1 - 2 & a + 1 - (-1) & z + 1 -  2 & 1 + 1 - 1 & 1 + 1 -1  \cr
(10b') & 1 + z - z & 1 + a - a & 1 + z -  z^* & 1 + 1 - 0 & 1 + 1 -z^* \cr
(11) & 1 + 1 - 1^* & 1 + 1 - 0 & 1 + 1- 1^* & 1 + 1 - 1 & 1 + 1 - 1 \cr
(11') & 1 + 1 - 1 & 1 + 1 - 1 & 1 + 1- 1^* & 1 + 1 - 0 & 1 + 1 - 1^*
\end{matrix}
$$

Plug in the values of $z,y_0,y_1,y_2,y_3$  to get that, with the exception of the $(S_{k},S_{k+1})$ entry for the Type (4) square, each table entry is greater than the right-hand side of equation (\ref{eq:URL}). 
$$
\begin{matrix}
& (S_{k-1},S_k) & (S_k,S_{k+1}) & (S_{k+1},S_{k+2}) & (S_{k+2},S_{k+3}) & (S_{k+3}, S_{k+4}) \cr
(2) & 0.5 & 2 & 0.5 & \mbox{N/A} & \mbox{N/A}   \cr
(2') & 1 & 1  & 1 & \mbox{N/A}   & \mbox{N/A} \cr
(3)  & 0.5 & 2 & 0.5 & \mbox{N/A}   & \mbox{N/A} \cr
(4) & 1.5 & {0} & {1.5} & \mbox{N/A}   & \mbox{N/A} \cr
(4') & 1 & {1} & {1} & \mbox{N/A}   & \mbox{N/A} \cr
(5) & 1 & 1  & 1 & 1  & \mbox{N/A} \cr
(5') & 0.125 & 0.625 & {3.125} & 0.125 & \mbox{N/A} \cr
(6) & 1.125 & 0.125 & 1.125 & 1.625 & \mbox{N/A} \cr
(6') & 0.625 & {0.125} & {2.625} & {.625} & \mbox{N/A} \cr
(7) & 0.5 & 2 & {0.5} & 1 & 1 \cr
(7') & 1 & 1 & {0.5} & 2 & 0.5 \cr
(8) & 0.875 & 1.25  & 0.875 & 1 & \mbox{N/A} \cr
(8') & 0.125 & 0.625 & 3.125 & 0.125 & \mbox{N/A} \cr
(9) & 1 & 2 & 1 & \mbox{N/A} & \mbox{N/A} \cr
(9') & 1 & 1 & 1 & \mbox{N/A} & \mbox{N/A} \cr
(10a) & 0.5 & 2 &1 &1 & 1 \cr
(10a') & 1 & 1 & 0.5 & 2 & 0.5 \cr
(10b) & 0.5 & 2 & 0.5 &1 & 1 \cr
(10b') & 1 & 1 & 1 & 2 & 0.5 \cr
(11) & 1 & 2 & 1 & 1 & 1 \cr
(11') & 1 & 1 & 1 & 2 & 1 \cr
(12) & 1.125 & 1.625 & 0.125 & 2.125 & \mbox{N/A} \cr
(12') & 0.125 & 0.625 & 3.125 & 0.125 & \mbox{N/A}
\end{matrix}
$$

\end{proof}

\begin{lemma}
\label{lem:tildeU}  
Fix  $i<j$ and a square (1)-(12).
$F_{i,j}$ has no critical points in $\overline{U}$ or $\overline{R}$.
\end{lemma}

\begin{proof}
By considering the reflected squares (1')-(12') as well, we only need to show there are no critical points in $\overline{U}$.

If $\sigma_L(i) \ge \sigma_L(j)$ and
$(x_1,x_2) \in \overline{U},$ then
the first two terms in
\begin{equation}
\notag
%\label{eq:key}
\partial_{x_1} F_{i,j}(x_1,x_2) = d_x f_{i,j}^U(x_1) +
d_x \phi(x_1)( f_{i,j}^R(x_2) - f_{\sigma_L(i), \sigma_L(j)}^L(x_2)) 
\end{equation}
 are strictly positive and 
the third term is non-negative.

If $\sigma_L(i) < \sigma_L(j)$ and $(x_1,x_2) \in  [-1/4,1/4] \times [1/2,3/4],$ then  
the result follows from Lemma \ref{lem:18est}, except when $(i,j) = (k,k+1)$ and the square is Type (4).  In this exceptional case, 
$f_{i,j}^U = f_{\sigma_L(i), \sigma_L(j)}^L = f^{PV}_{k,k+1}$ and $f_{i,j}^R(x_2) = f^{1Cr}_{k,k+1}$ where $f^{PV}_{k,k+1}$ and $f^{1Cr}_{k,k+1}$ are the $1$-dimensional functions defined in (\ref{eq:initialdef}) and (\ref{eq:initialdef2}) during the proof of Proposition \ref{prop:PV1Cr2Cr}.  Using (\ref{eq:initialdef}) and (\ref{eq:initialdef2}) we explicitly evaluate
\[
\partial_{x_1} F_{i,j}(x_1,x_2) = 2 +
d_x \phi(x_1)( -1 + (1/6)\ea),  
\]
and using (\ref{eq:InterpolatePhi}) 
\[
\partial_{x_1} F_{i,j}(x_1,x_2) > 2 + (2 +\ec)(-1 +(1/6)\ea) = (1/3)\ea - \ec + (1/6)\ea\ec \geq (1/4)\ea >0.
\]

If  $\sigma_L(i) < \sigma_L(j)$ and
$(x_1,x_2) \in \overline{U} \setminus [-1/4,1/4] \times [1/2,3/4],$
 then
\begin{equation}
%\label{eq:d_{x_2}F_{i,j}}
\notag
\partial_{x_2} F_{i,j}(x_1,x_2) = \phi(x_1) d_x f_{i,j}^R(x_2) + (1-\phi(x_1))
d_x f_{\sigma_L(i), \sigma_L(j)}^L(x_2)
\end{equation}
is a convex linear combination of strictly positive terms on $[-1/4,1/4] \times [1/4,1/2]$
and a convex linear combination of strictly negative terms on $[-1/4,1/4] \times [3/4,1).$

\end{proof}

\begin{lemma}
\label{lem:tildeD}  
Fix  $i<j$ and a square (1)-(12).
$F_{i,j}$ has no critical points in $\overline{D}$ or $\overline{L}$.
\end{lemma}

\begin{proof}
By considering the reflected squares (1')-(12') as well, we only show that there are no critical points in $\overline{D}$.  Note that if  $(x_1,x_2) \in \overline{D}$ then
\begin{equation}
\label{eq:d_1F^D} 
\partial_{x_1} F_{i,j}(x_1,x_2) = d_x f_{\sigma_D(i), \sigma_D(j)}^D(x_1) +
d_x \phi(x_1)( f_{i,j}^R(x_2) - f_{\sigma_L(i), \sigma_L(j)}^L(x_2)),
\end{equation}
\begin{equation}
\label{eq:d_2F^D} 
\partial_{x_2} F_{i,j}(x_1,x_2) = \phi(x_1) d_xf_{i,j}^R(x_2) + (1- \phi(x_1))d_xf_{\sigma_L(i), \sigma_L(j)}^L(x_2).
\end{equation}
\medskip

\noindent {\bf Case 1:}  $R$ is a (Cu) edge.  

In this case, for $x_2$ in the domain of $f^R_{i,j}$, we always have
\[
f^R_{i,j}(x_2) \geq 0, \quad \mbox{and} \quad d_x f^R_{i,j}(x_2) \geq 0
\]
with equality only when $(i,j) = (k,k+1)$ and $x_2 = -3/8$, by (1), (2), and (8') of Proposition \ref{prop:CuDef}, and
\[
\begin{array}{lr}
d_xf^D_{\sigma_D(i),\sigma_D(j)}(x_1) > 0,  & \quad \mbox{for $(i,j) \neq (k,k+1)$,} \\
d_xf^D_{\sigma_D(i), \sigma_D(j)}(x_1) = 0, & \quad \mbox{for $(i,j) = (k,k+1)$.}
\end{array}
\]
[Here, in addition to Proposition \ref{prop:PV1Cr2Cr} (2) and Proposition \ref{prop:CuDef} (2) and (8'), we use Lemma \ref{lem:MainProps0.5} (2).]

\medskip

\noindent {\bf SubCase 1:}  $d_x f^L_{\sigma_L(i), \sigma_L(j)}(x_2) > 0$.  Note that this includes the case where $(i,j) = (k,k+1)$ for $x_2 \neq 3/8$, since sheets $S_k$ and $S_{k+1}$ must meet above a cusp edge along edge $L$ as well.    

In this case, $\partial_{x_2}F_{i,j}(x_1,x_2)>0$ is a convex combination of two strictly positive values.  

\medskip

\noindent {\bf SubCase 2:}  $d_x f^L_{\sigma_L(i), \sigma_L(j)}(x_2) \leq 0$.  Since $L$ is a (Cu) edge, it must be the case that $f^L_{\sigma_L(i), \sigma_L(j)}(x_2) \leq 0$ holds as well.  As long as, $(i,j) \neq (k,k+1)$, then $\partial_{x_1}F_{i,j}(x_1,x_2) >0$ since  the first term in (\ref{eq:d_1F^D}) is strictly positive and the second is non-negative. Finally,  if $(i,j) = (k,k+1)$, then $x_2 = -3/8$ and the sheets in question meet at a cusp edge above $(x_1,x_2)$ so that there is nothing to prove. 

\medskip

\noindent {\bf Case 2:}  Sheets $S_k$ and $S_{k+1}$ cross along $D$.  
Then, $\mbox{sgn}(f_{i,j}^R(x_2))  = -\mbox{sgn}(f_{\sigma_L(i), \sigma_L(j)}^L(x_2)),$  and since 
\begin{equation}
\notag
%\label{eq:RDimplications}
f_{i,j}^R(x_2) <  0 \longleftrightarrow \sigma_D(i) > \sigma_D(j) \longleftrightarrow d_x f_{\sigma_D(i), \sigma_D(j)}^D(x_1) < 0
\end{equation}
all terms in equation (\ref{eq:d_1F^D}) are of the same sign.

\medskip

\noindent {\bf Case 3:} All remaining cases.  Since $\sigma_D(i) \neq \sigma_D(j)$; neither of these values are half integers;  and sheets $S_i$ and $S_j$ do not cross above $D$, we see from Corollary \ref{cor:summary} (5) that
\[
\left|d_x f_{\sigma_D(i), \sigma_D(j)}^D(x_1)\right| >  2y^D_{i,j}- N\ea \ge 2/8 - N\ea.
\]
In addition, (\ref{eq:InterpolatePhi}) together with Corollary \ref{cor:summary} (3) and (if necessary) Lemma \ref{lem:MainProps0.5} (3)  give
\[
|d_x \phi(x_1)| (|f_{i,j}^R(x_2)| + |f_{\sigma_L(i), \sigma_L(j)}^L(x_2)|) \le (2+\ec)( N\ea + N \ea) \leq 6N\ea.
\]
Combining these inequalities, we have
\[
|\partial_{x_1} F_{i,j}(x_1,x+2)| > 2/8- N\ea - 6N\ea > 0.
\]

\end{proof}

\begin{lemma}
\label{lem:Psi}  
Fix  $i<j$ and a square (1)-(12).  
$F_{i,j}$ has no critical points in $\Psi.$
\end{lemma}

\begin{proof}  
\noindent {\bf Main case:}  After possibly reflecting across $x_1=x_2$, suppose that  
\begin{enumerate}
\item $S_i$ and $S_j$ do not meet one another at a crossing or cusp along $U$; 
%\item $S_i$ and $S_j$ do not meet each other at a crossing above any (1Cr) edge; and
\item   $d_xf^L_{\sigma_L(i), \sigma_L(j)}(x) \geq 1/10$ for $-1/4 \leq x \leq 1/4$; and  
\item  $S_i$ and $S_j$ are not the crossing sheets in a Type (4) square.
\end{enumerate}

%or some $(X,Y) \in \{(U,L), (R,D)\}$, suppose that  
%\begin{enumerate}
%\item $S_i$ and $S_j$ do not meet one another at a crossing or cusp along $X$; 
%%\item $S_i$ and $S_j$ do not meet each other at a crossing above any (1Cr) edge; and
%\item   $d_xf^Y_{\sigma_Y(i), \sigma_Y(j)}(x) \geq 1/10$ for $-1/4 \leq x \leq 1/4$; and  
%\item  $S_i$ and $S_j$ are not the crossing sheets in a Type (4) square.
%\end{enumerate}

%We present the case when $(X,Y) =(U,L)$, while $(X,Y)=(R,D)$ is treated by interchanging the role of $x_1$ and $x_2$.  
We will show that for  $(x_1,x_2) \in \Psi,$
\begin{eqnarray}
\label{eq:dx2F}
\partial_{x_{2}} F_{i,j} & = & \phi(x_{1}) d_x f_{i,j}^\Are(x_{2}) +(1 - \phi(x_{1})) d_x f_{\sigma_\El(i), \sigma_\El(j)}^\El(x_{2}) \\
\notag
& + & d_x \phi(x_{2})( f_{i,j}^\You(x_{1}) - f_{\sigma_\Dee(i), \sigma_\Dee(j)}^\Dee(x_{1})) 
\end{eqnarray}
is positive.

Our assumption implies  
\begin{equation}
\label{eq:posDfRandfB}
f_{i,j}^U(x_{1}) > 0, \quad d_x f_{\sigma_L(i), \sigma_L(j)}^L(x_{2}) \ge 1/10.
\end{equation}
In addition, since the case of a Type (4) square with $S_i$ and $S_j$ crossing is prohibited, we see from the estimate (\ref{eq:URL}) proved in Lemma \ref{lem:18est} (with $U,D$ replacing $R,L$ and vice versa, which amounts to interchanging the estimates for the Type ($n$) and ($n'$) squares) that
 \begin{equation}
 \label{eq:new18est}
% y^L_{\sigma_L(i), \sigma_L(j)} \ge1/8, \quad 
  y^R_{i,j} + y^U_{i,j} - y^D_{\sigma_L(i),\sigma_L(j)} \ge 1/8.
  \end{equation}	
%So assume that $\sigma_\El(i) \ne \sigma_\El(j)$ and edge $\El$ is not the single-crossing model with 
% $S_i,S_j$ crossing; in particular,
% $y^\El_{\sigma_\El(i), \sigma_\El(j)} \ge1/8.$
% (Note this includes the case when $S_i$ or $S_j,$ but not both, ends at a vertical cusp locus.)

If  $-1/4 \le x_{1} \le -1/4 + \frac{1}{1000N}$ (recall $\ea < \frac{1}{1000N}$), then we estimate
 \begin{eqnarray*}
\partial_{x_{2}} F_{i,j}& \ge & (1 - \phi(x_{1}))d_x f_{\sigma_\El(i), \sigma_\El(j)}^\El(x_{2}) -
d_x \phi(x_{2})  f_{\sigma_\Dee(i), \sigma_\Dee(j)}^\Dee(x_{1}) \\
& \ge &
\left(1-2((-1/4 +\frac{1}{1000N}) +1/4)\right)
d_x f_{\sigma_\El(i), \sigma_\El(j)}^\El(x_{2}) - (2+\ec) 
\left|f_{\sigma_\Dee(i), \sigma_\Dee(j)}^\Dee(x_{1})\right|  \\
& \ge & \left(1-\frac{2}{1000N}\right) (1/10)
 - (3) \left( \frac{1}{1000N}(2y^\Dee_{\sigma_\Dee(i), \sigma_\Dee(j)}+N\ea)+ N \ea\right)\\
& \ge & 1/10 -(\frac{2}{10000N} +\frac{6}{1000} + 6 N\ea) > 0.
\end{eqnarray*}
[At the 1st inequality we used (\ref{eq:posDfRandfB}); at the 2nd inequality we used (\ref{eq:InterpolatePhi}); at the 3rd inequality we used that $f_{\sigma_\Dee(i), \sigma_\Dee(j)}^\Dee(x_{1}) = (x_{1}+1/4) d_xf_{\sigma_\Dee(i), \sigma_\Dee(j)}^\Dee(-1/4) + f_{\sigma_\Dee(i), \sigma_\Dee(j)}^\Dee(-1/4)$ and applied estimates from items (3) and (5) of Corollary \ref{cor:summary} and Lemma \ref{lem:MainProps0.5}; 
 at the 4th inequality, we used $y_{l,m} \leq N$.] 
So the lemma holds in this region.

Finally, consider the region $-1/4+ \frac{1}{1000N} < x_{1} \le 1/4.$ 
Assume $f_{i,j}^\You(x_{1}) - f_{\sigma_\Dee(i), \sigma_\Dee(j)}^\Dee(x_{1}) < 0$ since otherwise the left hand side of equation  (\ref{eq:dx2F}) is positive because of equation (\ref{eq:posDfRandfB}).  We have
%If $\Are$ is the one crossing model, and  $S_i,S_j$ cross there, let 
%$\delta = 4\sqrt{\epsilon_2} = d_xf^\Are_{i,j}|_{[-1/4,1/4]}$ (Lemma \ref{lem:MainPropsCr} item (2a)). Otherwise, let $\delta = 0.$
%By equations (\ref{eq:InterpolatePhi}), 
%(\ref{eq:posDfRandfB}), 
%\ms{(\ref{eq:new18est})},
%as well as items (2), (3), (5) of  Lemma \ref{lem:MainPropsPV} (and analogous results in Lemmas \ref{lem:MainPropsCr}-\ref{lem:MainPropsCu}),
 \begin{eqnarray*}
\partial_{x_{2}} F_{i,j} & \ge & \phi(x_{1}) d_x f_{i,j}^\Are(x_{2}) +  d_x \phi(x_{2})( f_{i,j}^\You(x_{1}) - f_{\sigma_\Dee(i), \sigma_\Dee(j)}^\Dee(x_{1})) \\
&\ge& 2(x_{1} + 1/4)(2y^\Are_{i,j} - N\ea)
+(2+\ec) ( f_{i,j}^\You(x_{1}) - f_{\sigma_\Dee(i), \sigma_\Dee(j)}^\Dee(x_{1})) \\
& \ge &  2(x_{1} + 1/4)(2y^\Are_{i,j} - N\ea)\\
& +& (2+\ec) (2y^\You_{i,j} - 2y^\Dee_{\sigma_\Dee(i), \sigma_\Dee(j)} -2N\ea) (x_{1} + 1/4)\\
& - &
(2+\ec)(|f_{i,j}^\You(-1/4)| + |f_{\sigma_\Dee(i), \sigma_\Dee(j)}(-1/4)|)\\
&=&4(x_{1} + 1/4)\left(y^{R}_{i,j} +y^\You_{i,j} - y^\Dee_{\sigma_\Dee(i), \sigma_\Dee(j)} \right) \\
&+ &(x_{1} + 1/4) \left( -2N \ea + 2y^U_{i,j} \ec - 2 y^D_{\sigma_D(i),\sigma_D(j)} \ec- (2 +\ec)(2N \ea)\right)\\
&-& (2+\ec)(|f_{i,j}^\You(-1/4)| + |f_{\sigma_\Dee(i), \sigma_\Dee(j)}(-1/4)|)\\
&\ge& 4\left(\frac{1}{1000N}\right)(1/8)  - (1/2)(10 N \ea) - (3)(2N\ea) >0.
\end{eqnarray*}
[We used in the 1st inequality (\ref{eq:posDfRandfB}); in the 2nd inequality item (5) of Corollary \ref{cor:summary} and properties of $\phi$ from (\ref{eq:InterpolatePhi}) including the estimate $\phi(x_{1}) \ge 2(x_{1} + 1/4)$ which holds since $x_1 \geq \frac{1}{1000N}$;  and in the 3rd inequality item (5) from Corollary \ref{cor:summary} and Lemma \ref{lem:MainProps0.5}.  In the 4th inequality, we used that $-1/4+ \frac{1}{1000N} < x_{1} \le 1/4$, that $|2y_{l,m} \ec| \leq N \ea$ and $2+\ec < 3$,  and applied (\ref{eq:new18est}) as well as item (3) from Corollary \ref{cor:summary} and Lemma \ref{lem:MainProps0.5}.]

\medskip

\noindent {\bf Remaining cases:}  First, we determine those cases where the requirements (1)-(3) imposed in the main case can fail.

\begin{lemma} After possibly reflecting across $x_1=x_2$, we can arrange that either $S_i$ and $S_j$ satisfy requirements (1)-(3) or one of the following conditions is satisfied:
\begin{enumerate}
\item[(A)] $S_i$ and $S_j$ cross above $L$ and $R$.
\item[(B)] $S_i$ and $S_j$ cross above $L$ and $D$.
\item[(C)] $S_i$ and $S_j$ are the crossing sheets of a Type (4) square. 
\item[(D)] $S_i$ and $S_j$ belong to a Type (11) square with a cusp edge between sheets $S_k$ and $S_{k+1}$ above $R$ and a cusp edge between $S_{k+2}$ and $S_{k+3}$ above $U$, and we have $i \in \{k,k+1\}$ and $j \in \{k+2,k+3\}$. 
\item[(E)] $S_i$ and $S_j$ belong to a Type (12) square with $j = k+2$ and $i \in \{k,k+1\}$.
\end{enumerate}
\end{lemma}

\begin{proof}
The only time when condition (3) fails is as in (C). 
 Notice that for any of the square types, it is the case that no pair of sheets meets (at a crossing or cusp) above both $R$ and $U$.  Therefore, reflecting across $x_1=x_2$ if necessary, we may assume that $S_i$ and $S_j$ do not meet above $U$, so that (1) holds.  Before considering condition (2), if either of $S_i$ or $S_j$ ends at a cusp edge, if necessary we attempt to reflect so that neither $S_i$ or $S_j$ meets a cusp edge above $U$.   Such a reflection is possible except when we have
\begin{itemize}
\item[(D')]  a Type (11) square with the pairs of sheets $S_{k}, S_{k+1}$ and $S_l, S_{l+1}$
 meeting at cusp edges and $i \in \{k,k+1\}$ and $j \in \{l,l+1\}$.
\end{itemize} 
Moreover, condition (1) will still hold after this reflection except when $S_i$ and $S_j$ are as in (E), since this is the only instance where cusp sheets have crossings with another sheet. 

Now, item (5) from Corollary \ref{cor:summary} and Lemma \ref{lem:MainProps0.5},  gives 
\[
d_xf^L_{\sigma_L(i), \sigma_L(j)} > y^L_{\sigma_L(i),\sigma_L(j)} -  N \ea.
\]
Since we have arranged that $S_i$ and $S_j$ do not meet above $U$ so that $\sigma_L(i) < \sigma_L(j)$, we have $y^L_{\sigma_L(i),\sigma_L(j)} -  N \ea  > 1/10$ as required by (2) except possibly when
\begin{itemize}
\item $L$ is a (1Cr) edge with $S_i$ and $S_j$ crossing above $L$; or 
\item $\sigma_L(i) - \sigma_L(j) = 0.5$.
\end{itemize}
In the first case, either (A) or (B) holds.  The second case, can only happen when at least one of $S_i$ or $S_j$ ends at a cusp edge above $U$, and as arranged above, this can only be when (D') is satisfied.  In order for $\sigma_L(i) -\sigma_L(j) = .5$ it must be the case that the two pairs of cusping sheets are adjacent so that $(l,l+1) = (k+2,k+3)$.
\end{proof}

We now prove Lemma \ref{lem:Psi} for each of the cases (A)-(E).

\medskip

\noindent {\bf (A):}  In this case, $d_xf^L_{\sigma_L(i),\sigma_L(j)}(x_2), d_xf^R_{i,j}(x_2), f^U_{i,j}(x_1),$ and $-f^D_{\sigma_D(i),\sigma_D(j)}(x_1)$ are all positive, so that $\partial_{x_2} F_{i,j} > 0$ follows from (\ref{eq:dx2F}).

\medskip

\noindent {\bf (B):}  This only occurs in the Type (3) square with $i =k$, and $j = k+1$.  Since edge $D$ is (1Cr),  $\mathit{sgn} \left( f^D_{\sigma_D(k), \sigma_D(k+1)}(x_1)\right) = \mathit{sgn}(x_1)$ by item (1) of Proposition \ref{prop:PV1Cr2Cr}.  Thus, when $x_1<0$, the argument from (A) applies, while when $x_1\geq 0$ (in fact when $x_1 \geq -1/4 + 1/(1000N)$) the final estimate used in the Main Case applies.  

\medskip

\noindent {\bf (C):}  For the (4) square, with $(i,j) = (k,k+1)$, the $1$-dimensional functions satisfy
\[
d_xf^L_{\sigma_L(k),\sigma_L(k+1)}(x_2) \geq 2- N \ea; \quad d_xf^R_{k,k+1}(x_2) > 0; \quad f_{i,j}^U(x_1) > 0;
\]
\[
\mathit{sgn} \left( f^D_{\sigma_D(k), \sigma_D(k+1)}(x_1)\right) = \mathit{sgn} \left( f^D_{k+1, k}(x_1)\right) = -\mathit{sgn}(x_1);   \quad \mbox{and} \quad |f^D_{\sigma_D(k), \sigma_D(k+1)}(x_1)| < N\ea.
\]
(The estimates are from items (2), (3), and (5) of Proposition \ref{prop:PV1Cr2Cr}.)  When $x_1 >0$, all terms in (\ref{eq:dx2F}) are positive.  When $x_1 \leq 0$,
\[
(1-\phi(x_1))d_xf^L_{\sigma_L(k),\sigma_L(k+1)}(x_2) \geq (1/2)(2- N \ea)
\] 
and this dominates the only negative term $-d_x\phi(x_2)  f^D_{\sigma_D(k), \sigma_D(k+1)}(x_1)$ which is bounded as
\[
|-d_x\phi(x_2)  f^D_{\sigma_D(k), \sigma_D(k+1)}(x_1)| < 3 N \ea
\]
using (\ref{eq:InterpolatePhi}).

\medskip

\noindent {\bf (D):}  In this case, $\sigma_L(i) = i \in \{k,k+1\}$ and $\sigma_L(j) = k+1.5$.  The statement of Lemma \ref{lem:MainProps0.5} only explicitly gives that $d_xf^L_{i,k+1.5} > 2C_{i,k+1.5} - N\ea$ with $C_{i,k+1.5} \in \{ y^L_{i,k+1}, y^L_{i,k+2}\}$.  However, in examining the proof, since it is the sheet in position $k+2$ above edge $L$ that is appears adjacent to the $S_{k+2}$ and $S_{k+3}$ cusp sheets above both $U$ and $D$, the construction produces $C_{i,k+1.5} = y^L_{i,k+2}$, so that for $i \in \{k,k+1\}$ the estimate
$d_xf^L_{i,k+1.5} > 1/10$ holds and the Main Case applies. 

\medskip

\noindent {\bf (E):}  For the Type (12) square (as pictured, and not reflected) with $i \in \{k,k+1\}$ and $j = k+2$, the $1$-dimensional functions satisfy
\[
d_xf^L_{\sigma_L(i), \sigma_L(j)}(x_2) = d_xf^L_{k-.5,k}(x_2) > 0, \quad d_xf^R_{i,j}(x_2) > 0,
\]
\[
f^U_{i,k+2}(x_1) > 0, \quad \mbox{and} \quad -f^D_{\sigma_D(i), \sigma_D(j)}(x_1) > 0
\]
since $\sigma_D(i) > \sigma_D(j) = k$.  Thus, all terms in (\ref{eq:dx2F}) are positive.

\end{proof}

\subsubsection{Proof of Lemma \ref{lem:quarter-cross} for squares (13) and (14)}
\label{sssec:quarter-crossST}

We focus on square (13) as the other is similar.
In Lemmas \ref{lem:URLST}-\ref{lem:tildeUtildeDtildeRST},  we prove there are no Reeb chords in $\Psi, \overline{U}, \overline{D}, \overline{R}$ by showing that for $(x_1,x_2)$ in each of the various regions considered one of the $\partial_{x_{i}} F_{i,j}$ does not vanish.  That there are no Reeb chords in $\overline{L}$ is proved in Lemma \ref{lem:tildeLST}.

Outside of $\overline{L},$ we have $F_l = G_l$ in (\ref{eq:FHCAG}),  $\widehat{f}_{k+1} = f^U_{k+1}$ in (\ref{eq:hatk1}), 
and $\widehat{f}_{k+2} = f^U_{k+2}$ in (\ref{eq:hatk2}).
Plugging these into (\ref{eq:Gkdef})-(\ref{eq:Gk2def}), we get
\begin{eqnarray}
\label{eq:ReebST}
\partial_{x_{1}} F_{i,j} & = & \phi(x_{2}) d_x f_{i,j}^U(x_{1}) +(1 - \phi(x_{2})) d_xf_{\sigma_D(i), \sigma_D(j)}^D(x_{1})\\ 
\notag
& +&   d_x \phi(x_{1})( f_{i,j}^R(x_{2}) - f^{ST}_{i,j}(x_{2})) \\
\notag
\partial_{x_{2}} F_{i,j} & = & \phi(x_{1}) d_x f_{i,j}^R(x_{2}) +(1 - \phi(x_{1})) d_x f^{ST}_{i,j}(x_2)\\
\notag
&  + &  d_x \phi(x_{2})( f_{i,j}^U(x_{1}) - f_{\sigma_D(i), \sigma_D(j)}^D(x_{1})).
\end{eqnarray}
 Here $f^{ST}_{i,j}:= f^{ST}_i - f^{ST}_j,$ and $f^{ST}_i(x)$ is defined as $f^L_{\sigma_L(i)}(x)$ if $i \ne k,k+1,k+2,$
and by (\ref{eq:STk1111})-(\ref{eq:STk1341}) otherwise.
Set $y^{ST}_{k-1,k} = 1,$  $y^{ST}_{k,k+1}=0= y^{ST}_{k+1,k+2},$ $y^{ST}_{k+2,k+3}=1$
and all other $y^{ST}_{i,i+1} = 1.$ For $i<j$, make $y^{ST}_{i,j}$ the summation as done above Proposition \ref{prop:PV1Cr2Cr}.
Note for $\{i,j \} \cap\{k,k+1,k+2\} =  \emptyset,$  $y^{ST}_{i,j} = y^L_{\sigma_L(i), \sigma_L(j)}.$

\begin{lemma}
\label{lem:URLST}
The $y^{ST}_{i,j} \ge 0$ satisfy properties analogous to  Proposition \ref{prop:PV1Cr2Cr} (2) and (5), and  
equation (\ref{eq:URL}) holds. 
Namely, 
\begin{enumerate}
\item
$|f^{ST}_{i,j}(x)| \le y^{ST}_{i,j} + N\ea$ if  $\{i,j\} \not \subset \{k,k+1,k+2\},$
\item
$|d_xf^{ST}_{i,j}|_{[-1/4,1/4]} -   2 y^{ST}_{i,j}| \le 7N \ea$ for all $i <j,$ 
\item
$y^R_{i,j} + y^U_{i,j} - y^{D}_{\sigma_D(i), \sigma_D(j)} \ge 1/2 \ge 1/8$ for all $i <j,$
\item
 $y^U_{i,j} + y^R_{i,j} - y^{ST}_{i,j} \ge 1 \ge 1/8$ for all $i <j.$
\end{enumerate}

\end{lemma}
\begin{proof}
(2):  If we set 
\begin{equation} \label{eq:sigmaprime}
\sigma'(i) = \left\{ \begin{array}{cr} i, & i \leq k-1, \\ k, &  i = k,k+1,k+2, \\ i-2,  & i \geq k+3, \end{array} \right.
\end{equation}
 then for all $1 \leq i \leq n$ on $[-1/4,1/4]$ we have
\[
f^{ST}_{i}(x) = f^L_{\sigma'(i)}(x) + \delta_{i,k}\cdot \psi(x)(\widehat{f}_{k+2}(-3/8)- {f}_{k+1}^U(\-3/8)), \quad \mbox{and} \quad y^{ST}_{i,j} = y^L_{\sigma'(i),\sigma'(j)}.
\]
Thus, using Corollary \ref{cor:summary} (5), (\ref{eq:psibounds}), and (\ref{eq:fhat14est}), we have
\begin{align*}
|d_xf^{ST}_{i,j}(x) - 2y^{ST}_{i,j}| & \leq |d_xf^L_{\sigma'(i),\sigma'(j)}(x)-2y^{L}_{\sigma'(i),\sigma'(j)}|  + |d_x\psi(x)(\widehat{f}_{k+2}(-3/8)- {f}_{k+1}^U(\-3/8))| \\ & \leq N\ea + 3 \cdot(N\ea + N\ea) = 7 N\ea.
\end{align*}

(1):  For $i<k$, $k+2 < j$, and $l \in \{k,k+1,k+2\}$, since locally $f^{ST}_l$ agrees with $f^L_k$ or $f^{ST}_k$, the combination of (\ref{eq:propflst}) and item (2) of Proposition \ref{prop:PV1Cr2Cr} show that $f^{ST}_{i,l}$ and $f^{ST}_{l,j}$ have positive (resp. negative) derivative on $(-1,1/2]$ (resp. on $[3/4, 1)$).  Within $[1/2,3/4]$ $f^{ST}_{l} = f^L_{k}$, so we see that $f^{ST}_{i,l}$, $f^{ST}_{l,j}$ and $f^{ST}_{i,j}$ all have a unique local maximum in $[1/2,3/4]$.  This maximum agrees with the maximum of $f^L_{\sigma'(i),\sigma'(j)}$ which is within $N\ea$ of $y^{ST}_{i,j} = y^L_{\sigma'(i),\sigma'(j)}$ by Proposition \ref{prop:PV1Cr2Cr} (2).  The result follows.

%The first two items follow from Proposition \ref{prop:flst} and equations (\ref{eq:STk1111})-(\ref{eq:STk1341}).
%For example, on $[-1/4,1/4],$ 
%$d_xf^{ST}_{k,k+1} - d_x\psi(x)(\widehat{f}_{k+2}(-3/8)- {f}_{k+1}^U(\-3/8))  = 0 = d_xf^{ST}_{k+1,k+2}$
%and $|d_x\psi(x)(\widehat{f}_{k+2}(-3/8)- {f}_{k+1}^U(-3/8))| < 3 \epsilon_1$ (which gets absorbed 
%into the $3N \epsilon_1$ error, see proof of Lemma \ref{lem:MainPropsPV} (5)).

(3) and (4):  We compute $y^R_{i,i+1} + y^U_{i,i+1} - y^{D}_{\sigma_D(i), \sigma_D(i+1)}.$  The general case follows.
$$
\begin{matrix}
 (i\le k-1,i+1) & (k,k+1) & (k+1, k+2) & (k+2,k+3) & (i\ge k+3, i+1) \cr
 1+1-1 & 1.5+1-2 & 0+1-(-1) & 1.5+1-2 & 1+1-1
\end{matrix}
$$
  Similarly,  $y^U_{i,i+1} + y^R_{i,i+1} - y^{ST}_{i,i+1}$ equals   
$$
\begin{matrix}
 (i\le k-1,i+1) & (k,k+1) & (k+1, k+2) & (k+2,k+3) & (i\ge k+3, i+1) \cr
 1+1-1 & 1+1.5-0 & 1+0-0 & 1+1.5-1 & 1+1-1
\end{matrix}
$$

\end{proof}

\begin{lemma}
\label{lem:PsiST}  
For any pairs of sheets  $(i,j)$ in either square (13) or (14),  
$F_{i,j}$ has no critical points in $\Psi.$
\end{lemma}

\begin{proof}

\noindent {\bf Case 1:} $(i,j) = (k+1,k+2).$   For $(x_1,x_2) \in \Psi$ we have
\begin{align*}
F_{k+1,k+2}(x_1,x_2) = & \phi(x_1) f^R_{k+1,k+2}(x_2) + (1-\phi(x_1)) f^{ST}_{k+1,k+2}(x_2) \\
& + \phi(x_2)f^U_{k+1,k+2}(x_1) + (1-\phi(x_2)) f^{D}_{k+2,k+1}(x_1)  \\ 
=& \phi(x_1) f^R_{k+1,k+2}(x_2) + (2 \phi(x_2)-1) f^U_{k+1,k+2}(x_1); \\
& \\
\partial_{x_1}F_{k+1,k+2}  = & d_x\phi(x_1)f^R_{k+1,k+2}(x_2) + (2 \phi(x_2)-1) d_xf^U_{k+1,k+2}(x_1); \\
\partial_{x_2}F_{k+1,k+2}  = & \phi(x_1) d_xf^R_{k+1,k+2}(x_2) + 2 d_x \phi(x_2) f^U_{k+1,k+2}(x_1).
\end{align*}
Since $R$ is a (1Cr) edge with sheets $k+1$ and $k+2$ crossing, we get using item (1) of Proposition \ref{prop:PV1Cr2Cr} that $\sgn \left( \partial_{x_1}F_{k+1,k+2}(x_1,x_2) \right) = \sgn (x_2)$.  Thus,  $\partial_{x_1}F_{k+1,k+2}$ vanishes only when $x_2=0$ in which case
\[
\partial_{x_2}F_{k+1,k+2}(x_1,0) \geq 2 d_x\phi(0) f^U_{k+1,k+2}(x_1) = 4 f^U_{k+1,k+2}(x_1) >0.
\]  

%Then, $- f_{\sigma_D(i), \sigma_D(j)}^D(x_{1}) >0,$ $d_xf^{ST}_{i,j}(x_2) = 0,$
% $f_{i,j}^U(x_{1})>0$ and $d_x f_{i,j}^R(x_2) > 0;$ so  $\partial_{x_{2}} F_{i,j} > 0.$

\medskip

\noindent {\bf Case 2:} $i = k$ and $j \in \{k+1,k+2\}$.   From equations (\ref{eq:InterpolatePhi}), (\ref{eq:STk2114})-(\ref{eq:STk1341}), as well as Corollary \ref{cor:summary} (5) we have
\begin{align}
\label{eq:PsiST1}
&{\phi(x_{2}) d_x f_{i,j}^U(x_{1}) +(1 - \phi(x_{2})) d_xf_{\sigma_D(i), \sigma_D(j)}(x_1)} \ge  2\min\{y^U_{i,j}, y^D_{\sigma_D(i), \sigma_D(j)}\} - N \ea =  2 -  N \ea; &\\
\label{eq:PsiST2}
&d_x \phi(x_{1}) (f_{i,j}^R(x_{2}) - f^{ST}_{i,j}(x_{2})) \ge 
-d_x \phi(x_{1})f^{ST}_{i,j}(x_{2})  \ge -(2+\ec)(\psi(x) 2 N \ea)  \ge -6N \ea.&
\end{align}
So $\partial_{x_{1}} F_{i,j} > 0.$

\medskip

\noindent {\bf Case 3:} $\{i,j\} \not \subset \{k,k+1,k+2\}$.   
We use an argument almost identical to the one from the Main Case of the proof of Lemma \ref{lem:Psi}, where
$f^{ST}_{i,j}$ replaces $f^L_{\sigma_L(i), \sigma_L(j)}$ throughout.
Recall that we split the proof into two regions: $\{ -1/4 \le x_1 \le -1/4+ \frac{1}{1000N}\}$
and $\{ -1/4+ \frac{1}{1000N} \le x_1 \le 1/4\}.$

In the first region, we 
used that $d_xf^L_{\sigma_L(i),\sigma_L(j)} \ge 1/10$  which still holds by Lemma \ref{lem:URLST} (2) along with properties of $\phi$ and the $f^D, f^U,$ and $f^R$ that remain true in square (13).  The argument applies as written.
%Note that $y_{i,j}^{ST},$
%is also bounded from below by $1/8,$ so 
%it suffices to repeat the proof substituting
%$y^L_{\sigma_L(i), \sigma_L(j)}$ with $y_{i,j}^{ST}.$
%There is only one extra lower order term ($||d_x\psi||_{C^0} N \epsilon_1 \le 3 N \epsilon_1$) which may appear when $k \in \{i,j\}.$ 

%The argument is then as follows:

In the second region, $f^L_{\sigma_L(i), \sigma_L(j)}$ is never used other than in the first step which assumes $d_xf^L_{\sigma_L(i), \sigma_L(j)}(x_2) \ge 0$.  The proof still applies with the estimates from Lemma \ref{lem:URLST} (3) and (4) used in place of (\ref{eq:URL}).

% \begin{eqnarray*}
%\partial_{x_{2}} F_{i,j}& \ge & (1 - \phi(x_{1}))d_x f^{ST}_{i,j}(x_{2}) -
%d_x \phi(x_{2})  f_{\sigma_D(i), \sigma_D(j)}^D(x_{1}) \\
%& \ge &
%\left(1-2((-1/4 +\frac{1}{1000N}) +1/4)\right)
%d_x f^{ST}_{i,j}(x_{2}) - (2+\ec) 
%\left|f_{\sigma_D(i), \sigma_D(j)}^D(x_{1})\right|  \\
%& \ge & \left(1-\frac{2}{1000N}\right) (2y_{i,j}^{ST} - 7N\epsilon_1 - ||d_x\psi||_{C^0} N \epsilon_1)) \\
%&  - & (2+\epsilon_1) \left( \frac{1}{1000N}(2y^D_{\sigma_D(i), \sigma_D(j)}+N\epsilon_1)+ N \epsilon_1\right)\\
%& \ge & 2y^L_{\sigma_L(i), \sigma_L(j)} - (3+||d_x\psi||_{C^0}+\epsilon_1)N\epsilon_1 - \frac{13 + 2||d_x\psi||_{C^0}}{1000}\\
%& > & 0.
%\end{eqnarray*}
%In the second term of the second to last inequality, we estimate $f_{\sigma_D(i), \sigma_D(j)}^D(x_{1})$ with $(x_{1}+1/4) d_xf_{\sigma_D(i), \sigma_D(j)}^D(-1/4) + f_{\sigma_D(i), \sigma_D(j)}^D(-1/4).$ 
% In the last inequality, we use $y_{l,m} \le N.$
%So the lemma holds in this region.

\end{proof}

\begin{lemma}
\label{lem:tildeUtildeDtildeRST}  
For any pairs of sheets  $(i,j)$  in either square (13) or (14),  
$F_{i,j}$ has no critical points in $\overline{U}, \overline{D}, \overline{R}.$
\end{lemma}

\begin{proof}
Note that in $\overline{R}$, the $f^{ST}_{l}$ functions are not used and, since $x_1 \geq 1/4$, the $f^U$ and $f^D$ functions share all of  the properties of the (PV) $1$-dimensional functions.  Thus, in this region, the result follows from the case of the (2') square. 

We are left to consider the regions $\overline{U}$ and $\overline{D}$.

\medskip

\noindent {\bf Case 1:}  $(i,j) = (k+1,k+2)$.  We have
\[
\forall (x_1,x_2) \in \overline{U}, \quad \partial_{x_1} F_{k+1,k+2}(x_1,x_2) = d_xf^U_{k+1,k+2}(x_1) + d_x\phi(x_1) ( f^R_{k+1,k+2}(x_2) -f^{ST}_{k+1,k+2}(x_2));
\]
\[
\forall (x_1,x_2) \in \overline{D}, \quad \partial_{x_1} F_{k+1,k+2}(x_1,x_2) = d_xf^D_{k+2,k+1}(x_1) + d_x\phi(x_1) ( f^R_{k+1,k+2}(x_2) -f^{ST}_{k+1,k+2}(x_2)).
\]
For $x_1 \in [-1/4,1/4]$, Corollary \ref{cor:summary} (5) gives
\[
|d_xf^U_{k+1,k+2}(x_1)| = |d_xf^D_{k+2,k+1}(x_1)| \geq 2-N \ea > 1.
\]
Since sheets $k+1$ and $k+2$ cross along $R$ which is a (1Cr) edge, Proposition \ref{prop:PV1Cr2Cr} gives
\[
|f^R_{k+1,k+2}(x_2)| < N\ea,  \quad  \forall x_2 \in [-1,1].
\]
From (\ref{eq:STk2114})-(\ref{eq:STk1341}) we see that 
\[
|f^{ST}_{k+1,k+2}(x_2)| = \left\{ \begin{array}{cr} |f^{ST}_{k}(x_2)-f^L_k(x_2)|, &  x_2\in [-1,-1/4]\cup[3/4,1], \\
0 & x_2 \in [-1/4,3/4],
\end{array} \right. 
\] 
so that (\ref{eq:C0flst}) gives 
\[
|f^{ST}_{k+1,k+2}(x_2)| < 2N \ea, \quad \forall x_2 \in [-1,1].
\]
Thus, for $(x_1,x_2) \in \overline{U} \cup \overline{D}$ we have
\[
|\partial_{x_1} F_{k+1,k+2}(x_1,x_2)| > 1 - 3( N \ea + 2 N \ea) > 0.
\]

\medskip

\noindent {\bf Case 2:}  $i = k$ and $j \in \{k+1,k+2\}$.
For $(x_1,x_2) \in \overline{U} \cup \overline{D}$, the inequality (\ref{eq:PsiST1}) still holds, and since 
\[
|f^{ST}_{i,j}(x_2)| \leq |f^{ST}_k(x_2)-f^L_k(x_2)| \leq 2N\ea
\]
holds by (\ref{eq:C0flst}) the inequality (\ref{eq:PsiST2}) remains valid as well.  Thus, $\partial_{x_1}F_{i,j} >0$. 
%If $i=k$ and $j\in \{k+1,k+2\},$  or if $(i,j) = (k+1,k+2),$ the arguments in Lemma
%\ref{lem:PsiST} apply verbatim. (For example, equations (\ref{eq:PsiST1}) and (\ref{eq:PsiST2}) still hold.)

\medskip 

\noindent {\bf Case 3:}   $\{i,j\} \not \subset \{k,k+1,k+2\}$.  
 We follow an argument almost identical to the proofs of Lemmas \ref{lem:18est}, \ref{lem:tildeU} and \ref{lem:tildeD}, mindful of the possible modifications done in the  proof of Lemma \ref{lem:PsiST}:
(i)  $f^{ST}_{i,j}$ and $y^{ST}_{i,j}$ replace $f^{L}_{\sigma_L(i), \sigma_L(j)}$ and $y^{L}_{\sigma_L(i), \sigma_L(j)};$ 
%(2) an occurrence of $d_x\psi(x_2)$ appears in $\partial_{x_{2}} F_{i,j}$ when $k \in \{i,j\};$ 
(ii) equation (\ref{eq:URL}) is replaced by (3) and (4) from Lemma \ref{lem:URLST}.
%We perform the modifications below, or argue why they do not apply.

For $(x_1,x_2) \in [-1/4,1/4] \times [1/2,3/4] \subset \overline{U},$ we see by applying modifications (i) and (ii)  to  the proof of Lemma \ref{lem:18est}   that
\begin{eqnarray*}
\partial_{x_{1}} F_{i,j} & = &d_x f_{i,j}^U(x_{1}) + d_x \phi(x_{1})( f_{i,j}^R(x_{2}) - f^{ST}_{i,j}(x_{2})) > 1/5.
\end{eqnarray*}
(Here, Lemma \ref{lem:URLST} (1) is used to bound $|f^{ST}_{i,j}(x_2)|$ by $y^{ST}_{i,j} + N\ea$.)
For $(x_1,x_2) \in  \overline{U} \setminus [-1/4,1/4] \times [1/2,3/4],$
\begin{eqnarray*}
\partial_{x_{2}} F_{i,j} & = & \phi(x_{1}) d_x f_{i,j}^R(x_{2}) +(1 - \phi(x_{1})) d_x f^{ST}_{i,j}(x_2)
\end{eqnarray*}
is a convex linear combination of terms of the same sign on the two components of $\overline{U} \setminus [-1/4,1/4] \times [1/2,3/4],$ 
see Lemma \ref{lem:flst} and (\ref{eq:STk1434})-(\ref{eq:STk1341}), with (\ref{eq:propflst2}) used to check that $d_xf^{ST}_{i,j}(x_2)$ is negative for $x_2 \geq 3/4$.

%\dr{Here we note  that $d_x \psi(x_2) =0$ so modification (2) does not apply to the proof of Lemma \ref{lem:tildeU}.}

%\footnote{\ms{3/8/15:  $f^{ST}_k$ is not explicitly defined for $x_2 \ge 3/4$ in Proposition \ref{prop:flst}. This makes the argument in 
%$\overline{U}$ harder. Unless Lemma \ref{lem:PreStandard} has big enough range.}}

Next, we apply modification (i) to the proof of Lemma \ref{lem:tildeD}.
For $(x_1,x_2) \in  \overline{D},$
\begin{eqnarray*}
\partial_{x_{1}} F_{i,j} & = &d_x f_{\sigma_D(i), \sigma_D(j)}^D(x_{1}) + d_x \phi(x_{1})( f_{i,j}^R(x_{2}) - f^{ST}_{i,j}(x_{2})).
\end{eqnarray*}
The right-hand side is non-zero from Corollary \ref{cor:summary} (3) and (5), plus an application of (\ref{eq:C0flst}) to bound $|f^{ST}_{i,j}(x_2)|$,
\begin{eqnarray*}
& \left|d_x f_{\sigma_D(i), \sigma_D(j)}^D(x_1)\right| >  2y^D_{i,j}- N\ea \ge 2 - N\ea,&
\\&
|d_x \phi(x_1)| (|f_{i,j}^R(x_2)| + |f^{ST}_{i,j}(x_2)|) \le (2+\ec)( N\ea + 3N \ea).&
\end{eqnarray*}

%Next, consider $(x_1,x_2) \in [1/2,3/4] \times [-1/4,1/4]  \subset \overline{R}.$
%Showing 
%\begin{eqnarray*}
%\partial_{x_{2}} F_{i,j} & = & d_x f_{i,j}^R(x_{2}) +d_x \phi(x_{2})( f_{i,j}^U(x_{1}) - f_{\sigma_D(i), \sigma_D(j)}^D(x_{1})) > 1/5
%\end{eqnarray*}
% is analogous to the $[-1/4,1/4] \times [1/2,3/4] \subset \overline{U}$ argument, with no analogous modifications necessary
%since $f^{ST}_{i,j}$ does not appear.
%Finally, if $(x_1,x_2) \in  \overline{R} \setminus  [1/2,3/4] \times [-1/4,1/4],$ then
%\begin{eqnarray*}
%\partial_{x_{2}} F_{i,j} & = & \phi(x_{2}) d_x f_{i,j}^U(x_{1}) +(1 - \phi(x_{2})) d_xf_{\sigma_D(i), \sigma_D(j)}^D(x_{1})
%\end{eqnarray*}
%is a convex linear combination of terms of the same sign on the two components (again, with no modifications to note).

\end{proof}

%Recall our notation for various portions of $[-1,1]^2$ from Section ???.  For $X= L$ or $R$ and $Y= D$ or $U$, in any of the corners $XY$, the difference functions $F_{i,j}$ have the form $a(x_1)+ b(x_2)$  where $a$ and $b$ denote the $1$-variable difference functions associated to sheets of the $1$-dimensional Legendrians that sit about the $X$ and $Y$ edges. (In the case that $X=L$, we may have $b(x_2) = f^{ST}_l(x_2)$ in some cases.)   Thus,  Reeb chords of $F_{i,j}$  in $XY$ are in bijection with pairs of Reeb chords of the $i$ and $j$ sheets above the corresponding portions of the $1$-skeleton.  This leads to Reeb chords $c_{i,j}$ in $RU$ for each $1 \leq i < j \leq n$ that, as in all other cases, have both coordinates ordered lexigraphically by Theorem (STAIRCASE).  In addition, because of the Reeb chord $\tilde{b}^R_{k+2,k+1}$  that appears below the crossing on the $R$ edge, we have one additional Reeb chord $\tilde{c}_{k+2,k+1}$ located in $RD$.

\begin{lemma}
\label{lem:tildeLST}  
For any pairs of sheets  $(i,j)$  in either square (13) or (14),  
$F_{i,j}$ has no critical points in $\overline{L}.$
\end{lemma}

\begin{proof}
This statement is established by the following Claims 1-5.

\medskip

\noindent {\bf Claim 1.}  For $i<j$ with   $\{i,j\} \cap \{k,k+1,k+2\} = \emptyset$, we have $\partial_{x_1}F_{i,j}(x_1,x_2)>0$ and $\partial_{x_2}F_{i,j}(x_1,x_2)>0$ for all $(x_1,x_2) \in \overline{L}$.

\medskip

  In $\overline{L}$,  (\ref{eq:FifLfR}) leads to $F_{i,j}(x_1,x_2)= f^L_{\sigma_L(i),\sigma_L(j)}(x_2) +  f^U_{i,j}(x_1)$ (since $f^D_{\sigma_D(i),\sigma_D(j)} = f^U_{i,j}$).  As $\sigma_L(i) < \sigma_L(j)$, $\partial_{x_2}F_{i,j} > 0$ holds by Corollary \ref{cor:summary} (5).

\medskip

\noindent {\bf Claim 2.}  For $i< k$ and $j>k+2 $, we have  $\partial_{x_2}(F_i -w)(x_1,x_2) >0$ and $\partial_{x_2}(w-F_j)(x_1,x_2) >0$ for all $(x_1,x_2) \in (\mbox{domain of $w$})\cap\overline{L}$  where $w$ is any of $F_k,F_{k+1},F_{k+2}$ or $\widetilde{F}_k$.

\medskip

We prove the statement about $F_i$ as  a similar argument establishes the  statement about $F_j$.  

To begin, we compute and estimate $\partial_{x_2}(F_i- H)$.  In $\overline{L}$, (\ref{eq:interpolating}) and (\ref{eq:Hdef})
give 
\[
F_i-H = f^L_{i}(x_2) + f^U_{i}(x_1) -f^L_k(x_2) - (1- \psi(x_2))f^U_{k+2}(x_1) - \psi(x_2) \widehat{f}_{k+1}(x_1),
\]
 so we have 
\begin{equation} \label{eq:parx2Fi}
\partial_{x_2}( F_i(x_1,x_2)- H(x_1,x_2)) = d_x f^L_{i,k}(x_2) + d_x\psi(x_2)\cdot\left(f^U_{k+2}(x_1) - \widehat{f}_{k+1}(x_1)\right)  \geq
\end{equation}
\[
1- |\psi'(x_2)|\cdot( |f^U_{k+2}(x_1)| + |\widehat{f}_{k+1}(x_1)|)  \geq 1 -3 (N \ea+ N\ea).
\] 
The bounds of $3$ for $|\psi'(x_2)|$ and $N \epsilon_1$  for $|f^U_{k+2}|$ and $|\widehat{f}_{k+1}|$ on $[-1,-1/4]$ follow from (\ref{eq:psibounds}),  (\ref{eq:fhat14est}) and item (3) of Corollary \ref{cor:summary}.
Next, (\ref{eq:FHCAG}) implies
\[
H(x_1,x_2)-w(x_1,x_2) = -(1-\alpha(r)) u(x_1,x_2) + \alpha(r)\cdot \left(H(x_1,x_2)-v(x_1,x_2)\right)
\]
 where $u$ is one of $C\,A_{l}$ or $C\,B$, and $v$ is one of $G_l$ or $H$ with $l \in \{k,k+1,k+2\}.$  (Here, $(r,\theta)$ is a polar coordinate centered at $(-3/8,0)$.)
   Thus, using that 
	\[
		|\partial_{x_2} (\alpha(r))| = \left|\alpha'(r) \cdot \frac{x_2}{\sqrt{(x_1+3/8)^2+ x_2^2}} \right| \leq |\alpha'(r)|, 
		\]
		we have
\begin{equation} \label{eq:x2HFi}
\left|\partial_{x_2}( H-w)(x_1,x_2)\right| \leq |\alpha'(r)| \cdot|u(x_1,x_2)| + (1-\alpha(r))|\partial_{x_2}u(x_1,x_2)| + 
\end{equation}
\[
|\alpha'(r)|\cdot |H(x_1,x_2)- v(x_1,x_2)| + \alpha(r)|\partial_{x_2}(H-v)(x_1,x_2)| \leq
\]
\[
33 \|u\|_{C^0(O_2)} + \|\partial_{x_2} u\|_{C^0(O_2)} + 33 \|H- v\|_{C^0(K)} + \|\partial_{x_2}(H-v)\|_{C^0(K)} \leq 
\]
\[
33\epsilon_1 + \epsilon_1 + 33(6 N \epsilon_1) + 18 N \epsilon_1 =  250 N \epsilon_1.
\]
In the second to last line, the notation $K$ is from Section \ref{sec:PropGk} and we used (\ref{eq:alphabound}).  In the last line, we used  (\ref{eq:CACO2}) and Lemma \ref{lem:Bprops}.B1.

Now, combining (\ref{eq:parx2Fi}) and (\ref{eq:x2HFi}) gives the desired inequality:
\begin{equation}
\label{eq:dx2Fiw}
\partial_{x_2}(F_i -w)(x_1,x_2)  \geq  \partial_{x_2}(F_i-H)(x_1,x_2) - |\partial_{x_2}(H-w)(x_1,x_2)| \geq 1-6N\epsilon_1 - 250N \epsilon_1 > 0.
\end{equation}

\medskip

\noindent {\bf Claim 3.}  The only critical points of $F_{k,k+1}$ in the intersection of $\bar{L}$ with the domain of definition of $F_{k,k+1}$ 
 are along the upper half of the cusp locus (where sheets $k$ and $k+1$ meet at the cusp edge).

\medskip

First, consider $(x_1,x_2)  \notin V_1$ where $V_1$ is defined in (\ref{eq:V1V2}).  
We have
\begin{eqnarray} 
\label{dx1Fkk1}
\partial_{x_1} F_{k,k+1}(x_1,x_2) & =&  
(1-\alpha(r))C  \partial_{x_1}A_{k,k+1}(x_1,x_2) + \alpha(r)\cdot \partial_{x_1}G_{k,k+1}(x_1,x_2)\\
\notag
&+& [\partial_{x_1}(\alpha(r))]\cdot(G_{k,k+1}(x_1,x_2) - CA_{k,k+1}(x_1,x_2)).
\end{eqnarray}
The combination of Lemmas \ref{lem:Aprops}.A3 and \ref{lem:Bprops}.B2 show that the sum of the first two terms is non-negative and, in fact, strictly positive everywhere except for the upper half of the cusp locus.  The values of $(x_1,x_2)$ where $\alpha'(r)$ is non-zero and $F_{k,k+1}$ is defined consist of only the right half of the annulus $A= O_2 \setminus O_1$, and in this half the chain rule gives $\partial_{x_1}(\alpha(r)) \geq 0$.  In addition, since $(x_1,x_2) \notin V_1$, (\ref{eq:AV1Gk}) and (\ref{eq:CAijCO2}) give $G_{k,k+1}(x_1,x_2) - A_{k,k+1}(x_1,x_2)>0$.  Thus, for such $(x_1,x_2)$ we have shown
\[
\partial_{x_1} F_{k,k+1}(x_1,x_2) \geq 0
\]
with equality only along the upper half of the cusp locus.

In the remaining case of $(x_1,x_2) \in V_1$ we show that  $\partial_\theta F_{k,k+1}(x_1,x_2) < 0$ unless $(x_1,x_2)$ belongs to the cusp locus.  Note that since $\partial_{\theta}(\alpha(r))=0$, we simply have
\begin{equation} \label{eq:PartialTheta}
\partial_\theta F_{k,k+1}(x_1,x_2) = (1-\alpha(r))\cdot  \partial_{\theta}A_{k,k+1}(x_1,x_2) + \alpha(r)\cdot \partial_{\theta}G_{k,k+1}(x_1,x_2).
\end{equation} 
The definition of $V_1$ and trigonometry show that $x_2 \geq \sqrt3 R_1/2 > R_1/2$, so that both 
Lemmas \ref{lem:Aprops}.A5 and \ref{lem:Bprops}.B4  apply.  Lemma \ref{lem:Bprops}.B2 and B4
   place $\nabla G_{k,k+1}(x_1,x_2)$ in the $4$-th quadrant where the angle from the positive $x$-axis is between $-\pi/2$ and $0$, while Lemmas \ref{lem:Aprops}.A5 has the angle between $\nabla A_{k,k+1}(x_1,x_2)$ and the positive $x$-axis bounded between $-\pi/2$ and $+\pi/4$.  On the other hand, in $V_1$ $\partial/\partial \theta$ has angle between $5\pi/6$ and $\pi$ and equal to $\pi$ only along the cusp edge.  Thus, except at the cusp edge the angle between $\partial/\partial \theta$ and either gradient is strictly larger than $\pi/2$, so geometric properties of the dot product give that both 
\[
\partial_{\theta}A_{k,k+1}(x_1,x_2) = \frac{\partial}{\partial \theta} \bullet \nabla A_{k,k+1}(x_1,x_2) \leq 0 \quad \mbox{  and  } \quad \partial_{\theta}G_{k,k+1}(x_1,x_2) = \frac{\partial}{\partial \theta} \bullet \nabla G_{k,k+1}(x_1,x_2) \leq 0
\]
where in both cases equality holds only along the upper half of the cusp locus.  At other points the gradient vectors are non-zero by 
Lemmas \ref{lem:Aprops}.A3 and \ref{lem:Bprops}.B2.

\medskip

\noindent {\bf Claim 4.}  The only critical points of $F_{k,k+2}$ in its domain of definition are along the lower half of the cusp locus (where sheets $k$ and $k+2$ meet at the cusp edge).

\medskip  

Arguments parallel to those used in Claim 3 show that $\partial_{x_1} F_{k,k+2}(x_1,x_2) >0$  everywhere except $V_2$ and the lower half of the cusp locus, while, in $V_2$, $\partial_\theta F_{k,k+2}(x_1,x_2) > 0$ except at the cusp locus.

\medskip

\noindent {\bf Claim 5.}  The only critical point of $F_{k+1,k+2}$ in its domain of definition is at the swallowtail point  (where sheets $k+1$ and $k+2$ merge together).

\medskip

For $(x_1,x_2) \notin V_3$ (see (\ref{eq:x1x2V3})), we check that $\sgn(\partial_{x_1} F_{k+1,k+2}(x_1,x_2)) = \sgn(x_2)$ as follows.
   The computation of $\partial_{x_1} F_{k+1,k+2}(x_1,x_2)$ is as in (\ref{dx1Fkk1}) with subscripts $k,k+1$ replaced with $k+1,k+2$.  
Lemmas \ref{lem:Aprops}.A3 and \ref{lem:Bprops}.B2 give that 
\[
\sgn\left((1-\alpha(r))\cdot  \partial_{x_2}A_{k+1,k+2}(x_1,x_2) + \alpha(r)\cdot \partial_{x_2}G_{k+1,k+2}(x_1,x_2) \right) = \sgn(x_2),
\]
while (\ref{eq:AV3}) and (\ref{eq:CAijCO2}) together with the fact that $\sgn(A_{k+1,k+2}) = \sgn(G_{k+1,k+2}) = \sgn(x_2)$ (from the location of the crossing locus) show that 
$(\partial_{x_1}(\alpha(r)))\,(G_{k+1,k+2}(x_1,x_2) - A_{k+1,k+2}(x_1,x_2))$ is greater than (resp. less  than) or equal to $0$  when $x_2>0$ (resp. when $x_2<0$).    Putting these observations together completes the check.

Finally, if $(x_1,x_2) \in V_3,$  then the defining property (\ref{eq:x1x2V3}) of $V_3$ shows that $\partial_{\theta} F_{k+1,k+2} >0$ since the computation of $\partial_{\theta} F_{k+1,k+2}$ is as in (\ref{eq:PartialTheta}).  In addition, for points not in $V_3$ but belonging to the crossing locus (where $x_2=0$), Lemma \ref{lem:Aprops}.A4 and Lemma \ref{lem:Bprops}.B3 show that $\partial_{\theta} F_{k+1,k+2} >0$.

%\medskip
%
%\noindent {\it No critical points in $\widetilde{D}\cup \Psi \cup \widetilde{U}$:}
%In this region, the only difference between the swallowtail square and the Type (2)' square (single crossing running horizontally through the square) is in that $\sigma_L$ is not the identity, and some of the functions $f^L_{\sigma_L(i)}$ are replaced with $f^{ST}_{l}$.  The proofs follow the same overall outline as in Theorem ???, but are presented here for the sake of completeness.
%
%...
%
%\noindent {\it No critical points in $\widetilde{R}$:}
%This follows directly from Theorem ??? since the defining functions for $L_{(13)}$ agree with the defining function of the  Type (2)' square (single crossing running horizontally through the square) in $\widetilde{R}$.  [To see this note, that the $1$-skeleton defining functions for  left cusp and plain vanilla edges agree when $x \in [1/4,1]$ ref????.]

\end{proof}

%\subsubsection{Proof of Proposition \ref{prop:interpolating}}
%\label{ssec:InterpolatingProof}

We now prove Proposition \ref{prop:interpolating}.

\begin{proof}

%Consider square (13) or (14).
%Lemma \ref{lem:Aprops}.A2 and Proposition \ref{prop:GkGk1} (3) imply the $(k+1,k+2)$-crossing locus is the $x_2$-axis.
%Since there are no unaccounted Reeb chords, all crossing loci must end at the cusp locus or the boundary. So the explicit models there prevent all other such crossing possibilities. We turn to squares (1)-(12).\footnote{\dr{Explicit models?  How exactly do they prevent crossings ending at the cusp locus? Will move this part of the proof to the end.}} 

For squares (1)-(12):

\medskip

\noindent {\bf Step 1:}  The cusp loci are as described in Property \ref{pr:14models}, i.e. vertical (resp. horizontal) cusp loci agree with $\{x_1= -3/8\}$  (resp. $\{x_2= -3/8\}$).

Suppose a cusp locus occurs on edges $U$ and $D$ involving  $S_k,S_{k+1},$ as in squares (9)-(12).  (The case of a cusp locus connecting edges $L$ and $R$, which only occurs in square (11), is similar.)  Then, $f_{\sigma_L(k),\sigma_L(k+1)}^L(x_2) = 0$, and edges $U$ and $D$ are (Cu) edges with cusps between $k,k+1$ and $\sigma_D(k), \sigma_D(k+1)$ respectively.  Therefore, using Proposition \ref{prop:CuDef} (8'), for $x_1 \in [-3/8, -3/8 +\ec]$, the formula for $F_{k,k+1}(x_1, x_2)$ reduces to 
\[
F_{k+1,k+2}(x_1,x_2) = 2(x_1+3/8)^{3/2},
\]
so that the sheets meet at a cusp edge along $x_1 = -3/8$.  

\medskip

\noindent{\bf Step 2:}  All crossing loci are as in Property \ref{pr:14models}, i.e. they must sit inside $\Phi.$

%Recall from the proof of Lemma \ref{lem:c-chords} that 
For $X \in \{U,D\},  Y \in \{R,L\}$, we need to show sheets $S_i$ and $S_j$ do not cross in the region $XY$ (from Figure \ref{fig:PsiPhi}).  For squares (1)-(12), we have 
\begin{equation} \label{eq:ijiijj}
F_{i,j}|_{XY}(x_1,x_2) = f_{i',j'}^X(x_1) + f_{i'',j''}^Y(x_2)
\end{equation}
 for some $i',j', i'', j''$ (some of which may be half integers).  Here, $i'=j'$ (resp. $i''=j''$) occurs only when $X= D$ (resp.  $Y = L$) and sheets $S_i$ and $S_j$ meet at a cusp  above $R$ and $L$ (resp. $U$ and $D$).  These conditions cannot simultaneously occur.  In the case $i'=j'$ (the case $i''=j''$ is similar), we get that $F_{i,j}|_{XY}(x_1,x_2)= f^Y_{i'',j''}(x_2)$ and $f^Y_{i'',j''}$ is a function for a (Cu) edge type with $i''$ and $j''$ the cusp sheets.  From Proposition \ref{prop:CuDef} (8'), we see that $(f^Y_{i'',j''})^{-1}(0) = \{-3/8\}$, so the only place that sheets $S_i$ and $S_j$ meet within $XY$ is at their cusp edge.

We are left to consider the case when $i'\neq j'$ and $i''\neq j''$.  Note that: 
\begin{enumerate}
\item[(i)] By item (1) from Propositions \ref{prop:PV1Cr2Cr} and \ref{prop:CuDef} and item (1) from Lemma \ref{lem:MainProps0.5}, we have 
 (possibly empty)
\[
(f_{i',j'}^X)^{-1}(0), (f_{i'',j''}^Y)^{-1}(0) \subset [-1/4, 1/4].
\]
\item[(ii)] For $(x_1, x_2) \in XY$,  $\mbox{sgn}(f^X_{i',j'}(x_1)) = \mbox{sgn}(f^Y_{i'',j''}(x_2))$.   [This is because $\mbox{sgn}(f^X_{i',j'}(x_1))$ (resp. $\mbox{sgn}(f^Y_{i'',j''}(x_2))$) is determined by whether $S_i$ appears above or below $S_j$ along the portion of edge $X$ (resp. edge $Y$) that borders $XY$. But, edges $X$ and $Y$ meet at the corner of $XY$, so these signs are the same.] 
\end{enumerate} 
Combining (i) and (ii) shows that (\ref{eq:ijiijj}) is never $0$ in $XY$.

%Since\footnote{\dr{Untrue. If $i,j$ meet at a cusp, then $f^L$ can be identically $0$.  Probably this is the only issue.}} (possibly empty) $(f_{i',j'}^X)^{-1}(0), (f_{i'',j''}^Y)^{-1}(0) \subset [-1/4, 1/4]$ (see Lemmas \ref{lem:MainPropsCr} item (2b), \ref{lem:MainProps2Cr} item (2a) and \ref{lem:MainProps0.5} item (1)) and $\mbox{sign}(f^X_{i',j'}(x_1)) = \mbox{sign}(f^Y_{i'',j''}(x_2))$ for $(x_1, x_2) \in XY,$ this implies that $F_{i,j}^{-1}(0) \subset \Phi$ for any of the squares we consider. So all crossing and cusp loci must sit inside $\Phi.$

\medskip

%\noindent {\bf Step 3:}  If sheets $S_i$ and $S_j$ do\footnote{\dr{Possibly all we need at this step is that cusp sheets do not have unexpected crossings in $L$ or $D$...}} not cross above $U$ (resp. $R$), then they do not cross anywhere in $\overline{U}$ (resp. $\overline{D}$).

\noindent {\bf Step 3:}  Along a cusp edge, the only time that a third sheet (not one of the cusp edge sheets)  intersects the cusp edge is in the Type (12) square where $S_{k+2}$ crosses the cusp edge of $S_k$ and $S_{k+1}$ at a single point $(-3/8,x_2)$ with $x_2 \in[-1/4,1/4]$.

In verifying, we restrict to the case of a vertical cusp edge between sheets $S_k$ and $S_{k+1}$, and make use of items (1) and (2) from Lemma \ref{lem:MainProps0.5}.  For $i \notin \{k,k+1\}$, and, in the case of square (12), $i \neq k+2$, the three terms $f^L_{\sigma_L(i),\sigma_L(k)}(x_2)$, $f^U_{i,k}(-3/8)$, and $f^D_{\sigma_D(i),\sigma_D(k)}(-3/8)$ all share the same (strict) sign which is positive if $i<k$ and negative if $k+1<i$.  Thus, 
\[
F_{i,k}(-3/8, x_2) = f^L_{\sigma_L(i),\sigma_L(k)}(x_2) + \phi(x_2) f^U_{i,k}(-3/8)+ (1-\phi(x_2)) f^D_{\sigma_D(i),\sigma_D(k)}(-3/8)
\]
is non-zero everywhere.

For square (12) with $i = k+2$, we have
\begin{equation} \label{eq:kk212}
F_{k,k+2}(-3/8,x_2) = f^L_{k-.5,k}(x_2) + \phi(x_2) f^U_{k,k+2}(-3/8) + (1-\phi(x_2))f^D_{k+1,k}(-3/8).
\end{equation}
Note that $f^U_{k,k+2}(-3/8)>0$; $f^D_{k+1,k}(-3/8)< 0$;  $f^L_{k-.5,k}(x_2)$ is positive (resp. negative) for $x_2 \geq 1/4$ (resp. $x_2 \leq -1/4$); and $d_xf^L_{k-.5,k}(x_2) >0$ for $-1/4 \leq x_2 \leq 1/4$. It follows that $F_{k,k+2}(-3/8,x_2)$ is also positive for $x_2 \geq 1/4$, negative for $x_2 \leq -1/4$, and satisfies $\partial_{x_2}F_{k,k+2}(-3/8,x_2) >0$ for $-1/4 \leq x_2 \leq 1/4$.  Thus, $F_{k,k+2}(-3/8,x_2)=0$ occurs for a unique $x_2$, and this value of $x_2$ has $x_2 \in [-1/4,1/4]$.

%Suppose $i < j$ and the sheets do not cross along edge $U,$ then for $(x_1,x_2) \in \overline{U}$,  
%both $f_{i,j}^U(x_1)$ and $f_{i,j}^R(x_2)$ are strictly positive while $f_{\sigma_L(i), \sigma_L(j)}^L(x_2)$ vanishes identically if $\sigma_L(i) = \sigma_L(j)$ and is strictly positive otherwise.  
%Thus,
%\[
%F_{i,j} |_{\overline{U}} = f_{i,j}^U(x_1)+ \phi(x_1)f_{i,j}^R(x_2)+(1-\phi(x_1)) f_{\sigma_L(i), \sigma_L(j)}^L(x_2) >0
%\]

%Since $\mbox{sign}(j-i) = \mbox{sign}(\sigma_L(j)-\sigma_L(i))$ and neither
%$f_{i,j}^R(x_2)$ nor $f_{\sigma_L(i), \sigma_L(j)}^L(x_2)$ intersect zero if $x_2 > 1/4,$ the right-hand side of the above equation
%is a (positive) linear combination of terms of the same sign.
%This implies  $S_i,S_j$ cannot cross/cusp anywhere in the component $\overline{U}$ of $\Phi \setminus \Psi.$

%A similar argument holds for each of the other 3 components, assuming the sheets do not cross/cusp on the corresponding edge.
%This argument also shows that if two sheets do not cross/cusp on any of the four edges, they do not cross/cusp anywhere in $\Phi,$ including $\Psi.$  \dr{Or, does this rely on also knowing no closed crossing loci, so should be in the next paragraph not here?}

%\footnote{\ms{8/8/15: Dan you had some Skype objection/question around this part of the proof. But I can't remember what. Do you?}}

\medskip

\noindent {\bf Step 4:}  For any pair of sheets $S_i$ and $S_j$ with $i<j$, the $(i,j)$-crossing locus is topologically as pictured in Figure \ref{fig:generators}, i.e. it is either empty or is a single arc with endpoints on the specified edges of $[-1,1]\times[-1,1]$, or in the case of square (12) on the cusp locus.   

Step 2 places $F_{i,j}^{-1}(0)$ in $\Phi$, and Lemma \ref{lem:quarter-cross} shows that $F_{i,j}$ has no critical points in $\Phi$.  Thus, $0$ is a regular value for $F_{i,j}$, so that $F_{i,j}^{-1}(0)$ is a $1$-manifold with boundary whose boundary points consist of the intersection of $F_{i,j}^{-1}(0)$ with the boundary of the domain of definition of $F_{i,j}$ in $[-1,1]\times[-1,1]$.  A priori, $F_{i,j}^{-1}(0)$ may have end points on the cusp edge, but we have shown in Step 3 that the only case in which this occurs is the $(k,k+2)$- and $(k+1,k+2)$-crossing locus in square (12).  Moreover, $F_{i,j}$ has a (unique) crossing on an edge $X \in \{U, R, L ,D \}$ if and only if the $1$-dimensional difference function $f^X_{\sigma_X(i),\sigma_X(j)}$ that corresponds to sheets $S_i$ and $S_j$ above $X$ has a crossing.  

[To verify the previous sentence, consider the case $X=L$ where, using item (4) of Corollary \ref{cor:summary}, 
\[
F_{i,j}(-1,x_2) = f^L_{\sigma_L(i),\sigma_L(j)}(x_2) + \phi(x_2)(\sigma_+(j)-\sigma_+(i))\ec + (1-\phi(x_2))(\sigma_-(j)-\sigma_-(i))\ec
\]
where $\sigma_+(i)$ and $\sigma_-(i)$ denote the order that $S_i$ appears above $(-1,+1)$ and $(-1,-1)$ respectively.  Note that
\begin{align*}
\forall x_2 \in [1/4,1], \quad &\sgn(f^L_{\sigma_L(i),\sigma_L(j)}(x_2)) = \sgn(\sigma_+(j)-\sigma_+(i)); \\
\forall x_2 \in [-1,-1/4], \quad & \sgn(f^L_{\sigma_L(i),\sigma_L(j)}(x_2)) = \sgn(\sigma_-(j)-\sigma_-(i));
\end{align*}
and $F_{i,j}(-1,x_2)$ is strictly monotonic for $x_2 \in [-1/4,1/4]$ because either  
\begin{itemize}
\item[(a.)]  by Corollary \ref{cor:summary} item (5) $d_x f^L_{\sigma_L(i),\sigma_L(j)}(x_2) > 1/8$ 
and therefore dominates the other term of 
\begin{equation} \label{eq:Fijijij}
\partial_{x_2}F_{i,j}(-1,x_2) = d_xf^L_{\sigma_L(i),\sigma_L(j)}(x_2) + d_x\phi(x_2)\ec((\sigma_+(j)-\sigma_+(i))-(\sigma_-(j)-\sigma_-(i)))
\end{equation}
 in absolute value, or 
\item[(b.)]  edge $L$ is (1Cr) with $S_i$ and $S_j$ crossing above $L$.  In this case, both terms in   (\ref{eq:Fijijij}) have the same  sign.
\end{itemize}
Therefore, $F_{i,j}$ has a unique $0$ on $L$ if and only if $f^L_{\sigma_L(i),\sigma_L(j)}$ does by Intermediate and Mean Value Theorems.]

At this point, since the $1$-dimensional functions have crossings as pictured in Figure \ref{fig:EdgeTypesB} (by (1) of Propositions \ref{prop:PV1Cr2Cr} and \ref{prop:CuDef}) we see that the non-closed components of the $(i,j)$-crossing locus, $F_{i,j}^{-1}(0)$, have their boundary points as specified in (1)-(12) of Figure \ref{fig:generators}.  To complete Step 4, we need only rule out closed components of $F_{i,j}^{-1}(0)$.  Within such a closed component, $|F_{i,j}|$ would attain a maximum value somewhere, but this is impossible since  in Lemma \ref{lem:quarter-cross} $F_{i,j}$ was shown to have no critical points in $\Phi$.

\medskip

\noindent {\bf Step 5:}  Up to ambient isotopy of $[-1,1]\times[-1,1]$, the base projections of crossing and cusp arcs appear as in Figure \ref{fig:generators}.

We need to check the following:
\begin{enumerate}
\item[(A)] For squares (5), (6), (12), the 2 crossing arcs do not intersect (in the base projection), except at their common endpoint in square (12).
\item[(B)] For squares (7), (10), (11), the intersection point of the two relevant arcs is transverse and unique.
\item[(C)] For square (8), there is a unique triple point and it is the only point in the projection of more than one crossing arc.
\end{enumerate}

Observe that for any $i < j < m$, $F_{i,j} + F_{j,m} = F_{i,m}.$  Therefore, if two of the three are zero at some $(x_1,x_2),$ then the third must also be zero at $(x_1,x_2).$  This implies (A), since the two pictured arcs are $(k+1,k+2)$- and $(k,k+2)$-crossing arcs, and they cannot intersect since there is no $(k,k+1)$-crossing arc.

Consider now a crossing locus between consecutive sheets $k,k+1$ which connects two opposite (1Cr) edges, say $U$ and $D$ (the $R,L$ case is similar), as in squares (7), (10), and (8).
In these
%\footnote{\dr{So the cases under consideration are (2), (7), (8), (10').  I'm not sure if I think the equality $f_{k,k+1}^U(x_1) = f_{\sigma_D(k),\sigma_D(k+1)}^D(x_1)$ holds exactly in cases (8) and (10').  It is true that along both the top and bottom edges we have a difference function  between the two strands that cross. However, these two strands may be numbered differently.  The construction then leads to some minor differences--for instance, we move the location of the local maximum of $f_{i,i+1}$ to the left as $i$ increases.  However, I don't think this is a problem.  In all four cases, I do believe the next equality, $f_{k,k+1}^R(x_2) = -f_{\sigma_L(k),\sigma_L(k+1)}^L(x_2)$, is exactly correct.  Moreover, for the computation of $F_{k,k+1}(0,x_2)$ all you need to know about $f_{k,k+1}^U(x_1)$ and $f_{\sigma_D(k),\sigma_D(k+1)}^D(x_1)$ is that $f_{k,k+1}^U(0) = f_{\sigma_D(k),\sigma_D(k+1)}^D(0)=0$. }} 
cases, item (1) of Proposition \ref{prop:PV1Cr2Cr} states that 
\[
f_{k,k+1}^U(0) =  0 = f_{\sigma_D(k),\sigma_D(k+1)}^D(0).
\]
Moreover, our assumption implies that the edge types of $R$ and $L$ are the same (including the number of strands and location of crossing or cusp sheets); 
thus, $f_{k,k+1}^R(x_2) = -f_{\sigma_L(k),\sigma_L(k+1)}^L(x_2).$ 
By equation (\ref{eq:InterpolatePhi}), $\phi(0) = 1/2.$ 
So we compute
\begin{eqnarray*}
F_{k,k+1}(0, x_2) & = &\phi(x_2)f_{k,k+1}^U(0) + (1 - \phi(x_2))f_{\sigma_D(k),\sigma_D(k+1)}^D(0) \\
& +& 1/2(  f_{k,k+1}^R(x_2) + f_{\sigma_L(k),\sigma_L(k+1)}^L(x_2))\\
&= & 1 \times 0 + 1/2 \times 0.
\end{eqnarray*}
Thus, the crossing loci in squares (7), (10), as well as the vertical crossing locus in square (8) are all straight lines.  Therefore, (B) follows with the unique intersection actually orthogonal.

Finally, we need to exclude multiple triple points in square (8).
All such points must lie in the $(k,k+1)$-crossing locus which is the line $\{x_1 = 0\}.$
Computing
\begin{eqnarray*}
 F_{k,k+2}(0,x_2) & = &  \phi(x_2) f_{k,k+2}^U(0) + (1-\phi(x_2))f_{k+1,k}^D(0) \\
& + & (1/2)\left( f_{k,k+2}^R(x_2) +  f_{k+1,k+2}^L(x_2)\right)
\end{eqnarray*}
and keeping in mind that the $D$ edge is (1Cr) with the crossing between the sheets in position $k+1$ and $k+2$ above $D$, we see that
 $F_{k,k+2}(0,x_2) >0$ for $x_2 \geq 1/4$, and $F_{k,k+2}(0,x_2) <0$ for $x_2 \leq 1/4$.  In addition,
\begin{eqnarray*}
\partial_{x_2} F_{k,k+2}(0,x_2) & = & d_x \phi(x_2) (f_{k,k+2}^U(0) - f_{k+1,k}^D(0)) \\
& + & \phi(0) d_x f_{k,k+2}^R(x_2) + (1-\phi(0)) d_x f_{k+1,k+2}^L(x_2),
\end{eqnarray*}
and for $x_2 \in [-1/4,1/4]$ this is a sum of four positive terms.  It follows that $F_{k,k+2}|_{\{x_1 = 0\}}$ must have exactly one zero. 

\medskip

For squares (13) and (14):  We follow the same steps.  

At Step 1, the cusp locus was already identified in Lemma \ref{lem:smoothST}, and is as in Property \ref{pr:14models}.

%we consider only the $(k,k+1)$-cusp locus with symmetric considerations applying to the $(k,k+2)$-cusp locus.   From (\ref{eq:FHCAG}),
%\begin{equation} \label{eq:radialkk1}
%F_{k,k+1} = (1-\alpha(r))(C A_{k,k+1}) + \alpha(r) G_{k,k+1},
%\end{equation}
%where $r$ is a radial coordinate centered at $(-3/8,0)$, and $0 \leq \alpha(r) \leq 1$ is such that $F_{k,k+1} = CA_{k,k+1}$ holds in $O_1$, $F_{k,k+1}= G_{k,k+1}$ holds outside of $O_2$, and in the annulus $O_2\setminus O_1$ we interpolate between these two formulas.  (Here, $O_1$ and $O_2$ were disks centered at $(-3/8,0)$ with radii $1/16$ and $3/32$.)  The cusp locus of $A_{k,k+1}$ starts at the swallowtail point, and agrees with $x_1=-3/8$ within $O_2\setminus O_1$ by Lemma \ref{lem:Aprops}.A1.  In addition, the cusp locus of $G_{k,k+1}$ agrees with $\{x_1=-3/8\}\cap \{x_2 \geq 1/16\}$ by Proposition \ref{prop:GkGk1} (1).  Thus, the location of the cusp locus is as specified in Property \ref{pr:14models}.  
 
At Step 2, when considering (\ref{eq:ijiijj}), the $f^{ST}_{i,j}$ functions appear if $Y=L$ as  
\[
F_{i,j}|_{UL} = f^U_{i,j}(x_1) + f^{ST}_{i,j}(x_2);  \quad \quad F_{i,j}|_{DL} = f^D_{\sigma_D(i),\sigma_D(j)}(x_1) + f^{ST}_{i,j}(x_2).
\]
For $i<j$, away from the cusp edge $f^U_{i,j}(x_1)$ is strictly positive, as is $f^D_{\sigma_D(i),\sigma_D(j)}(x_1)$ except for $(i,j)= (k+1,k+2)$ in which case $f^D_{\sigma_D(i),\sigma_D(j)}= f^D_{k+2,k+1}$ is negative.  Using (\ref{eq:flstc0}), we see that $f^{ST}_{i,j}(x_2)$ has the same sign, but in a non-strict manner.  For instance, 
\[
f^{ST}_{k+1,k+2}(x_2) = \left\{ \begin{array}{cr} f^{ST}_k(x_2)-f^L_k(x_2), & \mbox{for $x_2 \geq 1/4$}, \\
f^L_{k}(x_2)-f^{ST}_k(x_2), & \mbox{for $x_2 \leq -1/4$}, \end{array} \right. 
\]
and this is greater (resp. less) than or equal to $0$ when $x_2 \geq 1/4$ (resp. $x_2 \leq -1/4$) by (\ref{eq:flstc0}) together with (\ref{eq:STk2114})-(\ref{eq:STk1341}).

We verify a slightly strengthened Step 3. in three parts:
\begin{enumerate}
\item $S_{k+2}$ does not intersect $S_k$ or $S_{k+1}$ along the $(k,k+1)$-cusp edge.  

The formula for $F_{k,k+2}$ from (\ref{eq:FHCAG}) is 
%as in (\ref{eq:radialkk1}) with $k+2$ replacing all occurences of $k+1$.  
\begin{equation} \label{eq:radialkk1}
F_{k,k+2} = (1-\alpha(r))(C A_{k,k+2}) + \alpha(r) G_{k,k+2}.
\end{equation}
That $A_{k,k+2}(x_1,x_2) \geq 0$ holds with equality only along the $(k,k+2)$-cusp edge follows from Lemma \ref{lem:Aprops}.A2.  The corresponding statement for $G_{k,k+2}(x_1,x_2)$ follows from Lemma \ref{lem:GkGk1} (3).  Thus, above the $(k,k+1)$-cusp locus, $F_{k,k+2}(x_1,x_2)$ is an interpolation of positive terms.

\item $S_{k+1}$ does not intersect  $S_k$ or $S_{k+2}$ along the $(k,k+2)$-cusp edge.  

This is similar to (1).

\item For $i \notin \{k,k+1,k+2\}$, $S_i$ does not intersect any of $S_k, S_{k+1},$ or $S_{k+2}$ above the $(k,k+1)$- or $(k,k+2)$-cusp locus.    

From (\ref{eq:FHCAG}), for $l \in \{k,k+1,k+2\}$,
\begin{equation} 
F_{i,l} = (1-\alpha(r))(F_i-H - C A_{l}) + \alpha(r)(F_i-G_l).
\end{equation}
Using the definition of $H$ and $G_l$ in (\ref{eq:Gkdef})-(\ref{eq:Hdef}), and the observation that $f^D_{\sigma_D(i)} = f^D_i = f^U_i$,
\[
F_i-H = (1-\psi(x_2))f^U_{i,k+2}(x_1) + \psi(x_2)(f^U_i-\widehat{f}_{k+1})(x_1) + f^L_{\sigma_L(i),k}(x_2)
\]
all of these terms have the same sign (for the middle term apply  (\ref{eq:estfk1hatnext})), and $|f^L_{\sigma_L(i),k}(x_2)| \geq 1/4$ (by item (5) of Proposition \ref{prop:PV1Cr2Cr}) dominates $||C A_l||_{C^0(O_2)} < \ea$ (by (\ref{eq:CACO2})).  Thus, where $(1-\alpha(r)) \neq 0$, $(F_i-H - C A_{l})>0$ holds if $i<k$, while $(F_i-H - C A_{l})<0$ holds if $k+2<i$.  

Moreover, 
\[
F_i-G_l= f^{ST}_{i,l}(x_2) +(1-\phi(x_2))(f^U_i - u)(x_1) + \phi(x_2) (f^U_{i}-v)(x_1)
\]
for appropriate functions $u,v \in \{f^U_k, f^U_{k+1}, f^U_{k+2}, \widehat{f}_{k+1}, \widehat{f}_{k+2}\}$.  For $i <k$ (resp. $k+2<i$) all of these terms are positive (resp. negative).  [This is verified for the $\widehat{f}$ functions using (\ref{eq:estfk1hat}) and (\ref{eq:estfk1hatnext}). 
\end{enumerate}

With Steps 1-3 in place the argument is completed as in the case of squares (1)-(12).  Note that we use the strengthened statement of Step 3 to ensure that crossings between sheets $S_{k+1}$ or $S_{k+2}$ and another sheet $S_i$ do not cross the border where these sheets merge with $\tilde{S}_k$.  This allows us to see that closed components of the crossing locus must be in the zero set of a single difference function of the form $F_{i,j}$ or possibly $F_i - \tilde{F}_k$, and therefore would produce a critical point that is forbidden by Lemma \ref{lem:quarter-cross}.

\end{proof}

%\subsubsection{Properties of Reeb chords}

Having eliminated any unwanted Reeb chords, we can now prove that Reeb chords exist in the desired places.

\begin{proposition}
\label{prop:PropertiesRCstaircase}
Properties \ref{pr:Reeb0}, \ref{pr:Reeb1}, \ref{pr:Location1}, \ref{pr:Reeb2}, \ref{pr:Location2}, as well as their swallowtail counterparts (see Section \ref{sec:61}) hold. 
We set the constants $\beta_{i,j}$, $\tilde{\beta}_{j,i}$, and $\epsilon$  that appear in Property \ref{pr:Location1} to be 
\[
\beta_{i,j} = \eta_{i,j}, \quad \tilde{\beta}_{j,i}= \tilde{\eta}_{j,i}, \quad \epsilon = \epsilon_\eta
\]
where $\epsilon_\eta$ was defined 
in equation (\ref{eq:eeta}).
%in the proof of Lemma \ref{lem:MainPropsPV}.
\end{proposition}

\begin{proof}
Property \ref{pr:Reeb0} follows from Proposition \ref{prop:smooth0} which shows that within $N(e^0_\alpha)$ each $F_{i,j}$ has a unique critical point at $e^0_\alpha$ itself.  

With $(-1,1)\times(-1/16,1/16)$ coordinates obtained from gluing two adjacent squares used in a neighborhood of $e^1_\alpha$, Proposition \ref{prop:smooth1} identifies the critical points of $F_{i,j}$ that appear above the $1$-cell itself as being located at points $(x_1,0)$ with $x_1$ a critical point of the $1$-dimensional function $f_{i,j}$ associated to the edge type of $e^1_\alpha$.  That these critical points match the description from Properties \ref{pr:Reeb1} and \ref{pr:Location1} follows from item (2) of Propositions \ref{prop:PV1Cr2Cr} and \ref{prop:CuDef} with Proposition \ref{prop:1234def} (3) used to determine the relative location of the $\beta_{i,j}$.  In the (2Cr) case, that $\tilde{\beta}_{k,k+2} < \tilde{\beta}_{k+1,k+2}$ follows since $d_x f_{k,k+2} = d_xf_{k,k+1} + d_xf_{k+1,k+2}$ will be positive at $\tilde{\beta}_{k+1,k+2}$.  [That these are the only critical points in $\widehat{N}(e^1_\alpha)$ follows from Lemmas \ref{lem:c-chords} and \ref{lem:quarter-cross}.]

Finally, Properties \ref{pr:Reeb2} and \ref{pr:Location2} follow from Lemmas \ref{lem:c-chords} and \ref{lem:quarter-cross} together with the location of critical points along $\partial I^2$.

 % That the locations of the Reeb chords $b_{i,j}$ are as specified in Prop.............  Expository issue:  The proof of no Reeb chords in $\Phi$ does not apply on boundary of $I^2$.  Also, the lemma about reeb chords in $XY$ needs to include the case where $L$ doesn't exist over $X$ or $Y$.  

%Let $\{m,\bar{m}\} = \{1,2\},$ and $\{x_m = \pm1\}$ denote the edge $X$ of the square.
%For points in $X \setminus \Phi,$  $\partial_{x_m} F_{i,j} = 0$ and $\partial_{x_{\bar{m}}} F_{i,j} = d_x f^X_{i',j'}$ for some $i',j'.$
%On this edge then, $F_{i,j}$ has critical points at $x_m  = -1,\eta^X_{i',j'},1$ (and $\tilde{\eta}^X_{i',j'}$ if it exists).
%The proof then follows from Lemmas \ref{lem:c-chords} and \ref{lem:quarter-cross}, as well as from the critical point locations of $f_{i,j}$ specified in Lemmas \ref{lem:MainPropsPV}-\ref{lem:MainPropsCu}. 
%[Lemma \ref{lem:MainProps0.5} is not needed here since the Properties do not make claims about Reeb chords involving ``artificial" sheets.]
 
%Recall the  $\epsilon$ in Properties \ref{pr:Location1}, \ref{pr:Reeb2} and \ref{pr:Location2} is set in Section \ref{ssec:epsilon}.
\end{proof}

\begin{remark}
\label{rem:etanotate}
As mentioned in Remark  \ref{rem:betanotate}, there is some potentially confusing notation.  The $\beta^X_{i,j}$ defined above Property \ref{pr:Reeb2} have $\beta^X_{i,j} = \beta_{\sigma_X(i), \sigma_X(j)}$, so that $\beta^X_{i,j} \neq \beta_{i,j}$ can occur when $X \in \{D,L\}.$
To minimize this potential confusion, we denote critical points for $f_{i,j}^X$ by $\eta^X_{i,j}$. So in particular, 
$\beta^X_{i,j} = \eta^X_{\sigma_X(i), \sigma_X(j)}$ is the maximum for $f_{\sigma_X(i), \sigma_X(j)}^X$ when $X \in \{D,L\}.$
\end{remark}

\subsection{Properties of GFTs}
\label{ssec:PropertiesGFTs}

\begin{proposition}
\label{prop:Properties0cells}
Property \ref{pr:0cells} holds.
\end{proposition}

\begin{proof} 

From Proposition \ref{prop:smooth0}, we have that, for $i<j$, above $N(e^0_\alpha)$, $F_{i,j} = (j-i)\ec(r^2+2)$.  Thus, $\partial N(e^0_\alpha) = \{r = 1/16\}$ is a level set of $F_{i,j}$ along which $-\nabla F_{i,j}$ points orthogonally into  $N(e^0_\alpha)$.

%If $|x_1|, |x_2| \ge 15/16,$ then $-\nabla F_{i,j}(x_1,x_2) = \langle -d_xf_{i',j'}(x_1), -d_xf_{i'',j''}(x_2) \rangle$
%(for some pairs $(i',j'), (i'',j'')$) has non-vanishing coordinates, whose\footnote{\dr{Proof doesn't seem to take in to account that the metric is non-Euclidean near corners.}} absolute values are bounded away from 0 by 
%$2 \epsilon_3(1- |x_1|)$ and $2 \epsilon_3(1- |x_2|).$ 
%[See Lemma \ref{lem:PreStandard},  item (4) of Lemma \ref{lem:MainPropsPV}, and analogous Lemma \ref{lem:MainPropsCr}-\ref{lem:MainPropsCu} statements. Lemma \ref{lem:MainProps0.5} is not needed near the 0-cells.]
%Checking the signs, we see that these non-zero bounds imply Property \ref{pr:0cells}.

\end{proof}

\begin{proposition}
\label{prop:Properties1cells}
Property \ref{pr:1cells} holds.
\end{proposition}

\begin{proof}

We use $(-1,1) \times (-1/16, 1/16)$ coordinates as above in a neighborhood of $e^1_\alpha \cong (-1,1) \times \{0\}$, and we denote by $e^0_\pm$ the $0$-cells that bound $e^1_\alpha$ at $x_1 = \pm1$.  We will construct the piecewise linear path $P_1$ of Property \ref{pr:1cells} in $(-1,1) \times[-1/32, 0)$, and then take $P_2$ to be the reflection of $P_1$ across the $x_1$-axis.  

To define $P_1$ we start at the point on $\partial N(e^0_+)$ with $x_2 = -1/32 + .01$ and then specify the slope of the various segments of $P_1$.  All slopes will belong to $(-.001,.001)$ so that $P_1 \subset (-1,1)\times[-1/32,-1/32+.02]$ will hold, and it follows that the neighborhood $N(e^1_\alpha)$ (bounded below and above by $P_1$ and $P_2$ and on the left and right end by portions of $\partial N(e^0_-)$ and $\partial N(e^0_+)$ will meet the requirements (1) and (2) of Property \ref{pr:1cells}.  

%To simplify notation, we construct the piecewise linear path $P_1$ of Property \ref{pr:1cells} when the $1$-cell $e^1_\alpha$ is the upper edge ($X=U$) of its adjacent square containing $P_1.$ 
%So $e^0_\pm = (\pm 1, 1).$ 
%The other cases are similar, even for Squares (13) and (14) (see Lemma \ref{lem:PreStandard}).

Before specifying $P_1$, we record the signs of components of the $-\nabla F_{i,j} = (-\partial_{x_1} F_{i,j}, -\partial_{x_2} F_{i,j})$ in $(-1,1)\times(-1/32,0)$.  Recall that for a given edge type (PV), (Cu), (1Cr), (2Cr) with $n \leq N$ sheets and $1$-dimensional functions $f_1, \ldots, f_n$, we label sheets as they appear above $x_1 = 1$,  with $S_k$, $S_{k+1}$ crossing or meeting at a cusp for a (1Cr) and (Cu) edge, and $S_{k+2}$ crossing sheets $S_{k+1}$ and $S_{k}$ for a (2Cr) edge.  Moreover, $\sigma_-(i)$ denotes the position of $S_i$ above $x_1=-1$ with $\sigma_-(k)=\sigma_-(k+1) = k-.5$ in the (Cu) case.  Proposition \ref{prop:smooth1}, gives for $(x_1,x_2) \in (-1,1)\times(-1/32,1/32)$
\[
F_{i,j}(x_1,x_2) = f_{i,j}(x_1) + \left[ (1-\phi(x_1))(\sigma_-(j)-\sigma_-(i)) + \phi(x_1)(j-i)\right] \ec(x_2^2+1)
\]
leading to
\begin{align}
\partial_{x_1}F_{i,j}(x_1,x_2) & = d_x f_{i,j}(x_1) + d_x\phi(x_1) \left[(j-i)- (\sigma_-(j)-\sigma_-(i))\right] \ec(x_2^2+1); \label{eq:ecx21} \\
\partial_{x_2}F_{i,j}(x_1,x_2) & = \left[ (1-\phi(x_1))(\sigma_-(j)-\sigma_-(i)) + \phi(x_1)(j-i)\right] \ec(2x_2). \label{eq:ecx22}
\end{align}

Let $i<j$.  For $(x_1,x_2) \in (-1,1) \times (-1/32,0)$ and belonging to the domain of $F_{i,j}$, we observe the following sign behavior with the interval restricting the value of $x_1$:
\[
\begin{array}{lccc}
\mbox{If $S_i$ and $S_j$ do not cross}, &   &  (-1,\eta_{i,j}) & (\eta_{i,j},1) \\
 &  & -\partial_{x_1}F_{i,j} < 0 & -\partial_{x_1}F_{i,j}>0 \\
 & & & \\
\mbox{If $S_i$ and $S_j$ do cross}, & (-1, \tilde{\eta}_{j,i}) & (\tilde{\eta}_{j,i}, \eta_{i,j}) & (\eta_{i,j},1) \\
 & -\partial_{x_1}F_{i,j} >0 & -\partial_{x_1}F_{i,j} <0 & -\partial_{x_1}F_{i,j} >0  \\
 & & & \\
\mbox{If $\sigma_-(i) < \sigma_-(j)$},  &  &  &  (-1,1)  \\
 & & &  -\partial_{x_2}F_{i,j} > 0 \\
 & & & \\
\mbox{(Cu) with $(i,j) = (k,k+1)$},  &  &  (-3/8,1/4)  & [1/4,1) \\
  & & -\partial_{x_2}F_{k,k+1} \geq 0 & -\partial_{x_2}F_{k,k+1} > 0 \\
& & & \\
\mbox{(1Cr) with $(i,j) = (k,k+1)$},  &  &  (-1,0)  & (0,1) \\
  & & -\partial_{x_2}F_{k,k+1} < 0 & -\partial_{x_2}F_{k,k+1} > 0 \\
& & & \\
\mbox{(2Cr) with $(i,j) = (k+1,k+2)$},  &  &  (-1,1/12)  & (1/12,1) \\
  & & -\partial_{x_2}F_{k+1,k+2} < 0 & -\partial_{x_2}F_{k+1,k+2} > 0 \\
& & & \\
\mbox{(2Cr) with $(i,j) = (k,k+2)$},  &  &  (-1,-1/12)  & (-1/12,1) \\
  & & -\partial_{x_2}F_{k,k+2} < 0 & -\partial_{x_2}F_{k,k+2} > 0 \\
& & & \\
\end{array}
\]
[In verifying the inequalities for $\partial_{x_1}F_{i,j}$, note that $\partial_{x_1}F_{i,j}$ is non-vanishing for $x_1 \in [-1/4,1/4]$ since $d_xf_{i,j}(x_1)$ dominates the second term of (\ref{eq:ecx21}) in all cases besides (1Cr) with $(i,j) = (k,k+1)$, while in this latter case both terms have the same sign.  To verify the inequalites for $\partial_{x_2}F_{i,j}$, in the special cases where $\sigma_-(i) > \sigma_-(j)$,  substitute the values of $\sigma_-(i)$ and $\sigma_-(j)$ into (\ref{eq:ecx22}) and use that $\phi(x_1) = 2(x_1+1/4)$ in $(-1/4+\ea, 1/4-\ea)$ to explicitly locate where $\partial_{x_2}F_{i,j} = 0$.]

We now construct $P_1$ and verify that (3) and (4) of Property \ref{pr:1cells} hold.  We use the notation $\varphi(x_1,x_2)$ for the slope of $P_1$ at $(x_1,x_2) \in P_1$, noting that $\varphi$ is multivalued at non-smooth points.  In all cases, we set $\varphi(x_1,x_2) = 0$ for $1/4\leq x_1$.  Note that $-\nabla F_{i,j}$ points transversally into the region above $P_1$ 
when $1/4 \le x_1$
since $-\partial_{x_2}F_{i,j}(x_1,x_2)>0$, so (3) holds.  For $x_1 \leq 1/4$, the definition of $\varphi(x_1,x_2)$ depends on the edge type.

\medskip

\noindent {\bf (PV)}:  Take $ \phi(x_1,x_2) =0$ for $x_1 \leq 1/4$;  (3) follows since $-\partial_{x_2}F_{i,j}> 0$ everywhere.

\medskip

\noindent {\bf (Cu)}:  Take $\phi(x_1,x_2) = .001$ for $x_1 \leq 1/4$;  (3) holds since all $-\nabla F_{i,j}$ belong belong to the closed upper left quadrant of the plane.  

\medskip

\noindent {\bf (1Cr)}: For small $0 < \delta \ll .001$, take 
\[
\phi(x_1,x_2) = \left\{ \begin{array}{cr} 0 & \mbox{for $|x_1| \geq .1$}, \\ 
-\delta & \mbox{for $-.1 \leq x \leq 0$}, \\
\delta & \mbox{for $0 \leq x \leq .1$}.
\end{array}\right.
\]
See Figure \ref{fig:Slopes}.
Note that by compactness and the computation of the sign of $\partial_{x_l} F_{i,j}$ above, we can choose $\delta$ so that for all $i<j$ with $(i,j) \neq (k,k+1)$, 
\[
 0 < \delta < \left|\frac{\partial_{x_2} F_{i,j}}{\partial_{x_2} F_{i,j}}\right|,  \quad \forall (x_1,x_2) \in [-1/4,1/4] \times [-1/32, -1/32+.02].
\]

We verify that (3) holds as follows.  For $(i,j) \neq (k,k+1)$, along segments with non-negative slope, this is because $-\nabla F_{i,j}$ points into the open upper left quadrant; along the segment with negative slope the choice of $\delta$ has $-\nabla F_{i,j}$ pointing transversally above $P_1$.  For $(i,j) = (k,k+1)$, $-\nabla F_{k,k+1}$ points into the open upper left quadrant for $x_1 >0$ (verifying (3)) and along the negative $x_1$-axis for $x_1=0$ (verifying (4)), while $-\nabla F_{k+1,k}$ points into the open upper right quadrant for $x_1<0$ and along the positive $x_1$-axis for $x_1=0$.

\medskip

\noindent {\bf (2Cr)}:  Define  $P_1$ for $x_1 \leq 1/4$, so that their are two non-smooth points which are located at the intersection of $P_1$ with the $(k,k+2)$ and $(k+1,k+2)$ crossing locus, and so that working from left to right $\varphi(x_1,x_2)$ takes the values $0$, $\gamma$, and $.001$ where the value $0 < \gamma < .001$ will be specified momentarily.  Note that for $(i,j) \notin \{(k,k+1), (k,k+2)\}$, $-\nabla F_{i,j}$ points into the open upper left quadrant, so (3) holds in this case.  

Note that both crossing loci (the $(k,k+2)$ locus is to the left of the $(k+1,k+2)$ locus) will intersect $P_1$ in the region where $-1/4 < x_1 < -1/4+10N\ea$.   [This is because for $(x_1,x_2) \in [-1/4,1/4] \times [-1/32, -1/32+.02]$, since $d_xf_{i,j}(x_1) \geq y_{i,j} - N\ea \geq  1/5$ and $f_{i,j}(-1/4) \geq -N\ea$
 (using item (3) and (5) of Corollary \ref{cor:summary})
\[
F_{i,j}(x_1,x_2) \geq f_{i,j}(x_1) - 2N\ec((1/32)^2+1) \geq \left(-N\ea + 1/5(x_1+1/4)\right) - N\ea
\]
which is positive for $x_1 > -1/4+10N\ea$.]

For $(i,j) \in \{(k,k+1), (k,k+2)\}$, although $-\partial_{x_2}F_{i,j}$ becomes negative (well to the right of the crossing loci) so that $-\partial_{x_2}F_{i,j}$ points into the lower left quadrant, from (\ref{eq:ecx21}) and (\ref{eq:ecx22}) we have the slope bound
\[
\left|\frac{\partial_{x_2} F_{i,j}}{\partial_{x_1}F_{i,j}}\right| < \frac{(N)(\ec)(2/32)}{(1/5)- (3)(2N)(\ec)(2)} < .001
\]
so in the region where $\varphi = .001$ (3) and (4) hold.  Also, to the left of both crossing loci where $F_{k+2,k}, F_{k+2,k+1} >0$, we have   $-\partial_{x_2}F_{k+2,k}, -\partial_{x_2}F_{k+2,k+1} >0$, so that (3) and (4) hold when $\varphi = 0$.  

Finally, between the two crossing loci where $\varphi= \gamma$, $F_{k+2,k+1}> 0$ and $F_{k,k+2} > 0$.  In this region,  $-\nabla F_{k+2,k+1}$ (resp. $-\nabla F_{k,k+2}$) points into the upper right (resp. lower left) quadrant so that for (3) and (4) to hold we  choose $0 < \gamma< .001$ such that for all such $(x_1,x_2)$, 
\[
\left|\frac{\partial_{x_2} F_{k,k+2}}{\partial_{x_1} F_{k,k+2}}\right| < \gamma < \left|\frac{\partial_{x_2} F_{k+2,k+1}}{\partial_{x_1} F_{k+2,k+1}}\right|. 
\]
To see that this is possible, use (\ref{eq:ecx21}) and (\ref{eq:ecx22}) to compute the following bounds when $(x_1,x_2) \in [-1/4,-1/4+10N\ea] \times [-1/32, -1/32+.02]$,
\begin{align*}
| \partial_{x_1}F_{k,k+2}| & \geq (2 y_{k,k+2} - N \ea) - (3)(2N)(\ec)(2) \geq 2(.625+.125) - .001 \\
|\partial_{x_2}F_{k,k+2}| & = 2 \ec|x_2| \cdot \left|-1 + 3 \phi(x_1)\right| \leq \frac{2\ec}{32}\cdot 1 \\ 
| \partial_{x_1}F_{k+2,k+1}| & = | \partial_{x_1}F_{k+1,k+2}| \leq (2 y_{k+1,k+2} +N\ea) + (3)(2N)(\ec)(2)(.125) +.001 \\
| \partial_{x_2}F_{k+2,k+1}| & = 2 \ec |x_2| \cdot \left|-2 + 3 \phi(x_1)\right| \geq 2 \ec (1/32 -.02) 1.999
\end{align*}

  \end{proof}

\begin{figure}
\labellist
\small
\pinlabel $F_{k,k+1}=0$ [t] at 96 -2
\pinlabel $F_{k,k+2}=0$ [t] at 296 -2
\pinlabel $F_{k+1,k+2}=0$ [t] at 400 -2
%\pinlabel $x$ [l] at 552 22
%\pinlabel $1$ [r] at 98 184
%\pinlabel $-1/4$ [u] at 32 8
%\pinlabel $1/4$ [u] at 192 8
%\pinlabel $y=\phi(x)$ [l] at 230 184
%\pinlabel $y=\psi(x)$ [l] at 474 84
\endlabellist
\centerline{ \includegraphics[scale=.6]{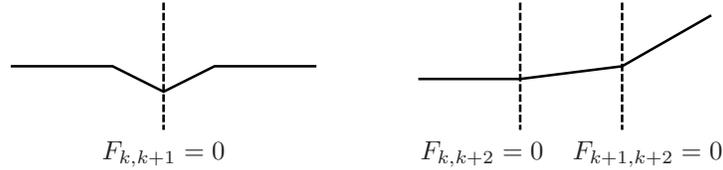} }
%\centerline{ \includegraphics[scale=.8]{images/SubdivideSq} }

\quad

\caption{The path $P_1$ pictured schematically (slopes are exaggerated) near the crossing locus along  a (1Cr) edge (left) and a (2Cr) edge (right).}
\label{fig:Slopes}
\end{figure}

 \begin{proposition}
\label{prop:1cmono}
Property \ref{pr:1cmono} holds.
%\footnote{\ms{5/28/15: Commented out below is the short proof of former(?) Property \ref{pr:IntersectB} which discussed the  stable and unstable manifolds of $b_{i,j}$ in $\hat{N}(e^1_\alpha).$}}
\end{proposition}

\begin{proof}
%From the defining equation (\ref{eq:interpolating}) for squares (1)-(12), and from Lemma \ref{lem:PreStandard}
%for squares (13)-(14), we have that if $(x_1, x_2) \in \hat{N}(e^1_\alpha),$
%\begin{eqnarray*} F_{i,j}(x_1,x_2)  & = & \phi(x_2) f^U_{i,j}(x_1) + (1 -\phi(x_2)) f_{\sigma_D(i), \sigma_D(j)}^D(x_{1}) + 
% \phi(x_1) f^R_{i,j}(x_2) + (1 -\phi(x_1)) f_{\sigma_L(i, \cdot), \sigma_L(j, \cdot)}^L(x_{1})
% \end{eqnarray*}
%where $\sigma_L(j, \cdot)$ is the appropriate one of $\sigma_L(j), \sigma_L(j, +), \sigma_L(j, -)$ (see Lemma \ref{lem:PreStandard}).

From Proposition \ref{prop:smooth1}, we have that for $(x_1,x_2)  \in \widehat{N}(e^1_\alpha) \subset (-1,1) \times [-1/32,1/32]$ with $x_1 \in (-1,-1/4] \cup [1/4,1)$,
\[
\partial_{x_1}F_{i,j}(x_1,x_2) = d_xf_{i,j}(x_1)
\]
where $f_{i,j}$ is the $1$-dimensional difference function associated to the edge type of $e^1_\alpha$.  Therefore, (keeping in mind that $\beta_{i,j} = \eta_{i,j}$ here)  Property \ref{pr:1cmono} follows immediately from  Propositions \ref{prop:PV1Cr2Cr} and \ref{prop:CuDef} which specify the sign of $d_xf_{i,j}$ everywhere on $[-1,1]$. 

%For $(x_1,x_2) \in \hat{N}(e^1_\alpha) \cap ({\{1/4 \le x_1 \le 1\}} \cup {\{-1 \le x_1 \le -1/4\}}),$
% $\partial_{x_1} F_{i,j}(x_1,x_2) = d_x f^X_{i',j'}(x_1)$  where $i',j'$ is the ordering of sheets $S_i,S_j$ as they appear over 
% $x_1 = 1,$ and $X \in \{PV, 1Cr, 2Cr, Cu\}$ is determined by which model the sheets follow over the $1$-cell $e^1_\alpha.$
% Property \ref{pr:1cmono}  follows from one of Lemmas \ref{lem:MainPropsPV} - \ref{lem:MainPropsCu}, where $X$ determines which Lemma to apply. [Lemma \ref{lem:MainProps0.5} is not used near the 1-cells.]

%The local formulation above implies the stable and unstable manifolds of $b_{i,j},$ intersected with $\hat{N}(e^1_\alpha),$ 
%are respectively, $\{x_1 = \eta_{i,j}\}$ and $\{x_2 = 0\}.$ Thus Property \ref{pr:IntersectB} holds in the the unperturbed case.
%The general case follows from the transversality of the intersection.

\end{proof}

%\subsection{Properties of GFTs}
%\label{ssec:PropertiesGFTs}

%\dr{Mike:  In case it is useful (possibly not), proofs of SOME of the monotonicity statements for the swallowtail case are contained in proofs of  Lemmas 2.4 and 2.7 in SwallowTailSection(6).pdf}

%Although Section \ref{sec:Comp2Cells} is for squares (1)-(12), many of its properties are also needed for squares (13)-(14) as stated in Section \ref{sec:SwallowComp}. In particular, Properties \ref{pr:monotonicityI},  
%\ms{\ref{pr:monotonicityII}}\ and \ref{pr:monotonicityIV} are needed for (1)-(14).
%So we will prove these for all squares, and Property \ref{pr:CuspTransversality} for just (1)-(12).

\begin{proposition}
\label{prop:PropertiesMonotonicityIandII}
Properties \ref{pr:monotonicityI} holds, including its extensions to the swallowtail case.
\end{proposition}

\begin{proof}
%This is similar to the proof of Lemma \ref{lem:c-chords}.
Recall from Remark \ref{rem:etanotate} the notation comparison:   $\eta^L_{\sigma_{L}(i),\sigma_{L}(j)}, \eta^D_{\sigma_{D}(i),\sigma_{D}(j)}$ in this section are  $\beta^L_{i,j}, \beta^D_{i,j}$ in the property statements.

For $X \in \{U, D\},  Y \in \{R,L\}$, the regions $C_{X,Y}$ appearing in Property \ref{pr:monotonicityI} satisfy $C_{X,Y} \subset XY$ (with $XY$ as in Figure \ref{fig:PsiPhi}).  For any Type (1)-(12) square and $i<j$, within $XY$,  
\[
F_{i,j}(x_1,x_2) = f^X_{i',j'}(x_1) + f^Y_{i'',j''}(x_2), 
\] for some $i',j',i'',j''$ which, because of the the assumption that neither $S_i$ or $S_j$ meet a cusp edge above $XY$, are both integers (i.e., the functions $f^L_{k-.5}$ and $f^D_{k-.5}$ of Lemma \ref{lem:MainProps0.5} do not appear here).  
In the case of the (13) square, the same considerations apply to the functions $F_{i,j}$ as well as $F_{i}-\widetilde{F}_k$ and $\widetilde{F}_k-F_j$.  The key observation is that for the sheets that do not meet a cusp edge in $C_{L,U}$ and $C_{D,U}$ (so, for $F_{k+2}$ and $\widetilde{F}_k$ in $C_{L,U}$ and $F_{k+1}$ and $\widetilde{F}_k$ in $C_{L,D}$), the definition of the $f^{ST}_l$ is such that the $f^Y(x_2)$ term arising from any of these functions will be $f^L_k(x_2)$.  [See (\ref{eq:FHCAG}) where $\alpha(r) =1$, then (\ref{eq:Gk1def})-(\ref{eq:Hdef}), then (\ref{eq:STk214})-(\ref{eq:STk134}).]

 Thus, the signs of the $\partial_{x_1}$ and $\partial_{x_2}$ partial derivatives respectively agree with the signs of the $1$-dimensional functions $d_xf^X_{i',j'}(x_1)$ and $d_xf^Y_{i'',j''}(x_2)$.  In particular, by Propositions \ref{prop:PV1Cr2Cr} and \ref{prop:CuDef}, the sign of $\partial_{x_1} F_{i,j}$ (and also $\partial_{x_2}F_{i,j}$) is as specified in Property \ref{pr:monotonicityI} with changes in sign of $\partial_{x_1}F_{i,j}$  occurring only when $x_1= \eta^X_{i',j'}= \beta^X_{i,j}$ or $x_1= \tilde{\beta}^X_{i,j}$.  

\end{proof}

%In the case of the (13) square, the same considerations apply to the functions $F_{i,j}$ and also $F_{i}-\widetilde{F}_k$ and $\widetilde{F}_k-F_j$.  The key observation is that for the sheets that do not meet at a cusp edge in $C_{L,U}$ and $C_{D,U}$ (so, for $F_{k+2}$ and $\widetilde{F}_k$ in $C_{L,U}$ and $F_{k+1}$ and $\widetilde{F}_k$ in $C_{L,D}$), the definition of the $f^{ST}$ is such that they equal $f^L_k(x_2) + f^U_{k+2}(x_1)$. 

\begin{proposition}
 Properties \ref{pr:monotonicityII} and \ref{pr:monotonicityIIST} hold.
\end{proposition}

\begin{proof}
For Property \ref{pr:monotonicityII} we establish the statement about $\partial_{x_2}F_{i,j}$ as the statement about $\partial_{x_1}F_{i,j}$ is similar.  For $i<j$ and $(x_1,x_2) \in \widehat{N}(e^2_\alpha) \subset [-1,1]\times[-1,1]$, when $x_2 =1/2$ we have 
\begin{equation} \label{eq:proofpropII}
\partial_{x_2}F_{i,j} = \phi(x_1)d_xf^R_{i,j}(x_2) + (1-\phi(x_1))d_xf^L_{\sigma_L(i),\sigma_L(j)}(x_2).
\end{equation}
The assumption that $S_i$ and $S_j$ do not cross or meet at a cusp edge above $U$ implies that $\sigma_L(i) < \sigma_L(j)$, so that both terms are positive and $-\partial_{x_2}F_{i,j}(x_1,1/2)<0$ follows.

For the first statement of Property \ref{pr:monotonicityIIST}, note that for any of the difference functions $F_{i,j}, F_i-\widetilde{F}_k, \widetilde{F}_k-F_j$ considered there, since all $f^{ST}_{l}$ functions agree with $f^L_k$ in a neighborhood of $x_2=1/2$ (see (\ref{eq:STk1434}) and (\ref{eq:STk214})-(\ref{eq:STk134})), (\ref{eq:proofpropII}) holds with $\sigma_L$ redefined as 
\[
\sigma_L(i) = \left\{ \begin{array}{cr} i, & i<k \\ k, & i = k,k+1,k+2, \\ i-2,& i>k+2. \end{array} \right.
\] 
From the restrictions placed on $i,j$ in the statement of Property \ref{pr:monotonicityIIST}, (\ref{eq:proofpropII}) is still a convex combination of positive terms.

To prove  the second statement of Property \ref{pr:monotonicityIIST}, note that for $1/4 \leq x_1 \leq 3/4$,
\[
\partial_{x_1}F_{i,j}(x_1,x_2) = \phi(x_2)d_xf^U_{i,j}(x_1) + (1-\phi(x_2))d_xf^D_{\sigma_D(i),\sigma_D(j)}(x_1).
\]
For $(i,j)\neq (k+1,k+2)$, we have $\sigma_D(i)<\sigma_D(j)$, so 
\[
\begin{array}{llc}
 \mbox{$d_xf^U_{i,j}(x_1) >0$ (resp. $d_xf^D_{\sigma_D(i),\sigma_D(j)}(x_1) >0$) } & \mbox{ for $x_1 <\eta^U_{i,j}$ (resp. $x_1 <\eta^D_{\sigma_D(i),\sigma_D(j)}$) } & \mbox{ and } \\ 
\mbox{$d_xf^U_{i,j}(x_1) <0$ (resp. $d_xf^D_{\sigma_D(i),\sigma_D(j)}(x_1) <0$) } & \mbox{ for $x_1 >\eta^U_{i,j}$ (resp. $x_1 >\eta^D_{\sigma_D(i),\sigma_D(j)}$). } & 
\end{array}
\]

\end{proof}

\begin{proposition}
\label{prop:PropertiesMonotonicityIV}
Property \ref{pr:monotonicityIV} holds.
\end{proposition}

\begin{proof}
Recall that the constants $\beta^R_{i,j}$, $\beta^U_{i,j}$, and $\e$ that appear in the statement of Property \ref{pr:monotonicityIV} have been defined as $\beta^R_{i,j} = \eta^R_{i,j}$, $\beta^U_{i,j} = \eta^U_{i,j}$, and $\e =\e_\eta$.

For $x_1,x_2 \ge 1/4$, $F_{i,j}(x_1, x_2) = f_{i,j}^U(x_1) + f_{i,j}^R(x_2).$ 
(For squares (13) and (14), $\alpha(r) = 1$ in (\ref{eq:FHCAG}), so the equation follows from (\ref{eq:Gkdef})-(\ref{eq:Gk2def}).)

Let $B_2$ be the line segment with endpoints $(1/3 - \ea/2,1/2)$ and $(1/3 + \ea/2 , 3/4).$
 Item (6) from Propositions \ref{prop:PV1Cr2Cr} and \ref{prop:CuDef}, 
%(since we are in the up right corner, we do not need Lemma \ref{lem:MainProps0.5}), 
as well as the definition of $M$ in equation (\ref{eq:Mdef}) imply for 
$\{(x_1,x_2) \in B_2\,\, | \,\, x_2 \le \eta^R_{i,j} -\epsilon\},$
\begin{eqnarray*}
\partial_{x_1} F_{i,j}(x_1,x_2) & = & \left\{\begin{array}{cr}(j-i)\ec, & \mbox{if $U$ is (PV), (1Cr), or (2Cr),} \\ (j-i)\ed, & \mbox{if $U$ is (Cu),} \end{array} \right. 
\\
\partial_{x_2} F_{i,j}(x_1,x_2) & \ge & M.
\end{eqnarray*}
The definitions of $\ec$ and $\ed$ from equations (\ref{eq:epsilon3}) and (\ref{eq:epsilon4}) then give  
$$
\partial_{x_2} F_{i,j}/ \partial_{x_1} F_{i,j} \ge \frac{M}{N\ec} > \frac{1}{4\ea}
= \frac{3/4-1/2}{\ea}
$$ 
which is the slope of $B_2.$
As both components of $-\nabla F_{i,j}$ are negative along $B_2$, the inequality shows that $-\nabla F_{i,j}$ points transversally into the region to the right of $B_2.$

Let $B_1$ be the reflection of $B_2$ through $\{x_1 = x_2\}.$ The calculation is similar.
\end{proof}

\begin{proposition}
\label{prop:CuspTransversality}
Property \ref{pr:CuspTransversality} holds.
\end{proposition}

\begin{proof}
Suppose $S_k,S_{k+1}$ cusp along $x_1 = -3/8$ for squares (1)-(12). 
The $x_2 = -3/8$ cusp locus case is similar. 
Fix distinct sheets $S_i,S_j$ with $i \le  k$, \, $k+1 \le j$, and $(i,j) \ne (k,k+1).$ 
Along the cusp locus, 
\[
\partial_{x_1} F_{i,j}  =  \phi(x_2) d_x f^U_{i,j}(-3/8) + (1 - \phi(x_2)) d_x f^D_{\sigma_D(i), \sigma_D(j)}(-3/8).
\]
For all cases except when $(i,j) = (k,k+2)$ in square (12),  we have $\sigma_D(i) < \sigma_D(j).$
Thus $\partial_{x_1} F_{i,j}$ (where defined) is a convex combination of two positive terms; see Proposition \ref{prop:CuDef}.

In the exceptional case, at the cusp point
$$
d_x f^U_{k,k+2}(-3/8) = 1.5 =-d_xf^D_{k+1,k}(-3/8) =-d_x f^D_{\sigma_D(k), \sigma_D(k+2)}(-3/8).
$$
[The first (resp. second) equality holds since $U$ (resp.  $D$) is a (Cu) edge with the cusp occurring between the sheets that appear above $U$ (resp. $D$) in positions $k$ and $k+1$ (resp. $k+1$ and $k+2$).  The exact evaluation of these derivatives at $-3/8$ is in item (8) of Proposition \ref{prop:CuDef}.]
Since $\phi^{-1}(0) = \{1/2\},$ see (\ref{eq:InterpolatePhi}), this implies there is a unique $(k+1,k+2)$-switch point $P = \partial_{x_1} F_{k,k+2}^{-1}(0) \cap \{x_1 = -3/8\} = (-3/8,0)$ with $-\nabla  F_{k+1,k+2}$ pointing left (resp. right) of the cusp locus above (resp. below) $P$. (Recall that $F_{k,k+2}$ and $F_{k+1,k+2}$ agree to first order along the cusp edge.)
Similarly, using (\ref{eq:08ea}) we evaluate $f^U_{k,k+2}(-3/8) = .12\ea = -f^D_{k+1,k}(-3/8)=-f^D_{\sigma_D(k),\sigma_D(k+2)}(-3/8);$ so  
%(\ref{eq:fkk0.5.12}) implies
\begin{eqnarray*}
F_{k,k+2}(-3/8,0) & = & f^L_{\sigma_L(k),\sigma_L(k+2)}(0) + \phi(0) f^U_{k,k+2}(-3/8) + (1-\phi(0))f^D_{\sigma_D(k),\sigma_D(k+2)}(-3/8)\\
&=&  f^L_{k-0.5, k}(0) >0
% \\
%& = & \frac{\tilde{\phi}(x_2)f^L_{k-1}(x_2) + (1-\tilde{\phi}(x_2))f^L_{k+1}(x_2) -f^L_k(x_2)}{2} + (2\phi(x_2) - 1)f^U_{k,k+2}(-3/8)
\end{eqnarray*}
where the inequality is from item (1) of Lemma \ref{lem:MainProps0.5}.  
%Note $(2\phi(x_2) - 1)f^U_{k,k+2}(-3/8)$ has a unique zero at $x_2 = 0$ and
% $f^L_{k-0.5, k}(x_2)$ has a unique zero at some value $\lambda <0,$ by Lemma \ref{lem:MainProps0.5} item (1).
% Moreover, both expressions are monotonic increasing for $\lambda \le x_2 \le 0.$ 
 Thus, the unique intersection of the crossing locus and the cusp locus, $Q = F_{k,k+2}^{-1}(0) \cap \{x_1 = -3/8\},$
  occurs below $P$ on the cusp locus.

%To prove the last claim, for points $(-3/8,x_2)$ which lie between $P$ and $Q,$ since $x_2 < 0$ we see that 
%$\partial_{x_1} F_{k+1,k+2}(-3/8,x_2) = (2\phi(x_2) - 1)d_xf^U_{k,k+2}(-3/8)$ is negative and so 
%$-\nabla F_{k+1,k+2}(-3/8,x_2)$ points in the direction with more sheets.

\end{proof}

\begin{proposition}
\label{prop:Switches}
Property \ref{pr:switches} holds.
\end{proposition}

\begin{proof}
The claimed properties of the cusp locus $\Sigma \subset N(e^2_\alpha)$ are easily verified from the construction of Sections \ref{sec:funcO1} and \ref{sec:Straight}.  [In Section \ref{sec:funcO1}, we started with an explicit model for the swallow tail with cusp locus 
\begin{equation} \label{eq:firstSTlocus}
\Sigma_{st}= \{(6r^2,8r^3) \, | \, r \in \R\}.
\end{equation}
Then, $\Sigma$ is obtained in Section \ref{sec:Straight} by straightening $\Sigma_{st}$ outside of $O_1$ via the application of two diffeomorphisms, $S$ then $T$.  Here, $S$ is a composition of a dialation and translation in the $x_1$-direction that relocates the swallowtail point to $Q= (q_1,0)$ with $-3/8 - \ec< q_1 < -3/8$.  Then, leaving $\Sigma$ fixed for $|x_2| \leq 1/32$, $T$ translates points on the cusp locus in $N(e^2_\alpha) \cap \{|x_2| \geq 1/32\}$ horizontally so that outside of $O_1$ they agree with $x_1 = -1/32$.  That the $x_2$-coordinate is monotonic along all of $\Sigma$ is easily verified in (\ref{eq:firstSTlocus}), and is preserved by $S$ and $T$.  That the $x_1$-coordinate is monotonic along either branch of $\Sigma$ when $|x_2| \leq 1/32$ is checked for (\ref{eq:firstSTlocus}) and is preserved by $S$, with $T$ leaving this portion of $\Sigma$ fixed.  That the minimum of the $x_1$-coordinate is at $Q$ follows when $|x_2|\leq 1/32$, and for $|x_2| \geq 1/32$ the cusp locus is at or to the right 
of $x_1 = -3/8$.]    

We establish the claims about switch points in the following Steps 1-3.

\medskip

\noindent {\bf Step 1.}  {\it There are no switch points outside of $O_1$.}

\medskip

Along the $(k,k+1)$- and $(k,k+2)$-cusp edges, $\nabla F_k = \nabla F_{k+1}$ and $\nabla F_{k} = \nabla F_{k+2}$ respectively.  Outside of $O_1$, $\Sigma$ agrees with the vertical line $x_1= -3/8$.  Therefore, it suffices to check the following. 
\begin{enumerate}
\item[(i)] For any $i < k$ or $j> k+2$, $\partial_{x_1} F_{i,k}$ and $\partial_{x_1} F_{k,j}$ are positive when $x_1 =-3/8$ and $|x_2| \geq R_1$.
\item[(ii)] When $x_1 =-3/8$ and $|x_2| \geq R_1$, $\partial_{x_1} F_{k+1,k+2}$ is non-vanishing.
\end{enumerate}

 Outside of $\overline{L}$, the result is clear since there $F_{i,k}$, $F_{k,j}$, and $F_{k+1,k+2}$ can each be written as the sum of a function of $x_2$ and a function of $x_1$ that does not have a critical point at $x_1=-3/8$.  Within $\overline{L}$, (ii) was established in the proof of Lemma \ref{lem:tildeLST}, Claim 5 where it was shown that $\sgn(\partial_{x_1}F_{k+1,k+2}) = \sgn(x_2)$ in a region of $\overline{L}$ that includes $x_1 =-3/8$.

For (i), suppose $i <k$ and $k+2 <j$. 
% Outside of $\overline{L}$, the result is clear since $F_{i,k}$ is a sum of a function of $x_2$ and a function of $x_1$ that does not have a critical point at $x_1=-3/8$.   
Assuming $(x_1,x_2) \in \overline{L}\cap\{x_1=-3/8\}$, we write
\begin{equation} \label{eq:FiHHFk}
F_{i,k} = (F_i - H)+ (H-F_{k}), \quad \mbox{and} \quad F_{k, j} = (F_k -H)+ (H-F_j). 
\end{equation}
Computing from (\ref{eq:interpolating}) and (\ref{eq:Hdef}) (keeping in mind that $f^U_i=f^D_i$) gives
\[
\begin{array}{cl} \partial_{x_1}(F_i-H)(-3/8,x_1)  & = (1 - \psi(x_2))d_xf^U_{i,k+2}(-3/8) + \psi(x_2) \cdot d_x(f^U_{i} - \widehat{f}_{k+1})(-3/8) \\
& \geq (1 - \psi(x_2))d_xf^U_{i,k+2}(-3/8) + \psi(x_2) \cdot d_xf^U_{i,k}(-3/8) 
\end{array}
\]
where we used that $d_x\widehat{f}_{k+1}(-3/8) \leq d_xf^U_{k}(-3/8)$ as in (\ref{eq:Est21}).  
Item (8') of Proposition \ref{prop:CuDef} 
implies $d_x f^U_{i,k+2}(-3/8) = k+2-i$ and $d_xf^U_{i,k} = (-3/8) = k+1/2-i.$
Thus,
\begin{equation}  \label{eq:x1FiHgeq}
\begin{array}{cl} \partial_{x_1}(F_i-H)(-3/8,x_2)  & \geq (1-\psi(x_2))(k+2-i) + \psi(x_2) \cdot ( k+1/2 -i)  \\
 & \geq (1-\psi(x_2))(3) + \psi(x_2) \cdot ( 3/2) \geq 9/4.
\end{array}
\end{equation}
where in the last inequality we used that $(1-\psi(x_2))\geq 1/2$;  see (\ref{eq:psibounds}).

In a similar manner, with the estimate
%\footnote{\dr{This requirement needs to be added to (21).}\ms{2/26/15: Do you mean equation (\ref{eq:estfk1hat})? 
%Should Proposition \ref{prop:fhatexist} be rewritten or is it `obvious' enough from Figure \ref{fig:fHat}?}} 
$d_x\widehat{f}_{k+1}(-3/8)  \geq d_x f^U_{k+2}(-3/8)$  from (\ref{eq:linearfk1hat}) 
%(\ref{eq:Est21}), 
we find
\begin{equation}  \label{eq:HFjjk2124}
\begin{array}{cl} \partial_{x_1}(H-F_j)(-3/8,x_2)  & \geq [1 - \psi(x_2)]d_xf^U_{k+2,j}(-3/8) + \psi(x_2) \cdot d_xf^U_{k+2,j}(-3/8) \\
& = (j-(k+2)) \geq 1.
\end{array}
\end{equation}
%when $x_1 = -3/8$.

For the other terms in (\ref{eq:FiHHFk}), we can compute from (\ref{eq:FHCAG})
\[
H-F_k =  -[1-\alpha(r)] \cdot C A_k + \alpha(r) \cdot ( H - G_k).
\]
Since $r$ is a polar coordinate centered at $(-3/8,0),$ 
when $x_1= -3/8$, $\dd r/\dd x_1 = 0$, and we have
\begin{equation} \label{eq:HGkpsi1}
\begin{array}{cl} \partial_{x_1}(H-F_k)(-3/8,x_2)  & =
-[1-\alpha(r)] \cdot C \partial_{x_1}A_k(-3/8,x_2) + \alpha(r) \cdot \partial_{x_1}( H - G_k)(-3/8,x_2) \\
 &  =  O(\ea) + \alpha(r) \partial_{x_1}( H - G_k)(-3/8,x_2)
\end{array}
\end{equation}
where $O(\ea)$
 indicates a term with absolute value bounded by $\ea$; see 
(\ref{eq:CACO2}).
Equations (\ref{eq:Gkdef})-(\ref{eq:Hdef}) imply
\begin{equation} \label{eq:HGkpsi}
\partial_{x_1}( H - G_k)(-3/8,x_2)= (1-\psi(x_2))\cdot d_x f^U_{k+2,k}(-3/8) + \psi(x_2) \cdot d_x(\widehat{f}_{k+1} - f^U_k)(-3/8)
\end{equation}
Note that for all $x_2$, $1/2 \leq (1-\psi(x_2))\leq 1$, so since $d_x f^U_{k+2,k}(-3/8)= -1.5$ (by (8') of Proposition \ref{prop:CuDef})  and $d_x(\widehat{f}_{k+1} - f^U_k)(-3/8) \leq 0$ (by (\ref{eq:Est21})) we deduce from (\ref{eq:HGkpsi1}) and (\ref{eq:HGkpsi}) that $\partial_{x_1}(H-F_k)(-3/8,x_2) <\ea$.  
Combining this observation with (\ref{eq:HFjjk2124}) and (\ref{eq:FiHHFk}) then establishes 
\[
\partial_{x_1} F_{k,j} = \partial_{x_1}(F_k -H)+ \partial_{x_1}(H-F_j) \geq -\ea + 1 > 0.
\]

The inequality $\partial_{x_1} F_{i,k} >0$ remains to be proven.  For this we use  (\ref{eq:HGkpsi}), (\ref{eq:Est21}), and (\ref{eq:linearfk1hat}) to bound
\[
|\partial_{x_1}(H - G_k)(-3/8,x_2)| \leq  |d_x f^U_{k+2,k}(x_2)| = 1.5.
\]
In view of (\ref{eq:x1FiHgeq}) and (\ref{eq:HGkpsi1}), we then achieve the desired inequality by 
\[
\partial_{x_1} F_{i,k}(-3/8,x_2) = \partial_{x_1}(F_i - H)(-3/8,x_2)+ (H-F_{k})(-3/8,x_2) \geq 9/4 - (1.5 +\epsilon_1) > 0.
\]

\medskip

\noindent {\bf Step 2.}  {\it Existence and uniqueness of $(i,k)$ and $(k+1,j)$ switch points within $O_1 \cap \{x_2 \geq 0\}$ for $i<k$ and $k+2 < j$.}

\medskip

Our outline for 
%establishing statements (1) and (2) of Proposition \ref{prop:SwPoints} concerning existence and uniqueness of $(i,k)$ and $(k+1,j)$ switch points within the upper-half of $O_1$ when $i < k$ and $k+2< j$ 
proving Step 2 is as follows.  We consider functions $\theta_1$ and $\theta_2$ defined on the upper branch of the cusp locus, $\Sigma_+ := \Sigma \cap O_1 \cap \{x_2 \geq 0\}$, that are respectively the slope of the tangent to $\Sigma_+$ and the slope of $\nabla F_{i,k}$ or $\nabla F_{k+1,j}= \nabla F_{k,j}$ along $\Sigma_+$.  Our goal is to show that there is a unique point in $\Sigma_+$ where $\theta_1$ and $\theta_2$ agree.

We begin by establishing some slightly more detailed bounds for partial derivatives of $F_{i,k}$ 
in the strip $O_1 \cap \left([-3/8 - \ec, -3/8 + \ec] \times [-1,1] \right) $ that contains $\Sigma$ (by Lemma \ref{lem:Aprops}.A1).
%\footnote{Need to right this property into the Straightening the cusp edge section.\ms{2/26/15:Where is this section?}} ???. 
The $F_{k,j}$ case is similar.
 Note that, within $O_1$, the formula (\ref{eq:FHCAG}) for $F_k$ simplifies to $F_k = H + C \, A_{k}$. Thus, for any $(x_1,x_2) \in O_1$, 
%with $|x_1-(-3/8)| < \epsilon_\bullet$,
 we have
\[
F_{i,k} = f^L_{i,k}(x_2) + [1-\psi(x_2)]\cdot f^U_{i,k+2}(x_1) + \psi(x_2)\cdot [ f^U_i(x_1) -\widehat{f}_{k+1}(x_1)] - C \, A_k.
\]
Now, for $-1/4 \leq x_2 \leq 1/4$, Corollary \ref{cor:summary} item (5) gives $|d_x f^L_{i,k}(x_2) - 2(k-i)| \le N \epsilon_1.$   In addition, using  item (3) of Corollary \ref{cor:summary} and (\ref{eq:fhat14est}), for $x_1 \leq -1/4$ we have $|f^U_{i,k+2}(x_1)|<  N \epsilon_1$ and $|f^U_i(x_1) -\widehat{f}_{k+1}(x_1)| <  N \epsilon_1+N\epsilon_1$.  Combining these observations with $|\psi'(x_2)| <3$ from (\ref{eq:psibounds}) as well as (\ref{eq:CACO2}) gives that
\[
\begin{array}{cl} \partial_{x_2} F_{i,k}(x_1,x_2) & = d_{x}f^L_{i,k}(x_2) + \psi'(x_2) \cdot [ -f^U_{i,k+2}(x_1) + f^U_i(x_1) -\widehat{f}_{k+1}(x_1)] - C \, \partial_{x_2} A_k \\
   & \leq  2(k-i) + N\epsilon_1 + 3 \cdot[ 3N \epsilon_1] + \epsilon_1 \leq  2 N + 1.
\end{array}
\]
On the other hand, the previously established estimate (\ref{eq:parx2Fi}) together with (\ref{eq:CACO2}) gives the lower bound $1/2 \leq \partial_{x_2} F_{i,k}(x_1,x_2)$.  In summary, we have
\begin{equation} \label{eq:12leqx2}
1/2 \leq \partial_{x_2} F_{i,k}(x_1,x_2) \leq 2 N + 1, \quad \forall (x_1,x_2) \in O_1.
\end{equation}

For the purpose of bounding 
\[
\partial_{x_1} F_{i,k}= \partial_{x_1}(F_i-H) - C\,\partial_{x_1}A_k = (1 - \psi(x_2))d_xf^U_{i,k+2}(x_1) + \psi(x_2) \cdot d_x(f^U_{i} - \widehat{f}_{k+1})(x_1)-C\,\partial_{x_1}A_k,
\]
 we restrict attention to  $(x_1,x_2) \in O_1$ with $x_1 \in [-3/8 - \ec, -3/8 + \ec].$ 
  Proposition \ref{prop:CuDef}, item (8') implies
%  \[
% d_xf^U_{i,k+2}(x_1) = k+2 -i  \quad \mbox{while}  \quad d_x(f^U_{i} - \widehat{f}_{k+1})(x_1)= K
% \] 
% \footnote{\ms{6/18/15: $d_x(f^U_{i} - \widehat{f}_{k+1})(x_1)= K$ uses that $d_x\widehat{f}_{k+1})$ is constant in a neighborhood of $-3/8.$ This is explicit in the proof of Proposition \ref{prop:fhatexist}.
% Also maybe right after (\ref{eq:Est21})  there should be the claim that $k-i+1/2 = d_x f_{i,k}(-3/8).$
% Then we get that this is less than $d_x f_{i,k+2}(-3/8) = k+2-i.$}}
 %\footnote{\dr{This used the revised statement about $\widehat{f}_{k+1}$ is linear in $[-3/8 - \epsilon_\bullet, -3/8 + \epsilon_\bullet]$ with slope $\ell$ satisfying $d_x f^{PV}_{k+1} \leq \ell < d_xf^{CUSP}_{k+1}(-3/8)$.  The linear assumption is mainly there for the $2$nd derivative estimates but can be done without if necessary. (DR)}}
 %\ms{2/26/15: It seems like you use $B:= d_x(f^U_{i} - \widehat{f}_{k+1}) \ge d_xf_{i,k+1} = k-i+1$ from Figure in Section \ref{sec:Straight}. Is this correct?} } 
 $d_xf^U_{i,k+2}(x_1) = k+2 -i$, and in conjunction with (\ref{eq:linearfk1hat}) implies that $d_x(f^U_{i} - \widehat{f}_{k+1})(x_1)= K$
where  $K$ is constant with $k+1/2-i \leq K \leq k+2-i$.  
Thus, again using (\ref{eq:CACO2}) to bound the $C\,\partial_{x_1}A_k$ term, we have
\begin{equation}  \label{eq:12leqx1}
1 \leq k+1/2-i- \epsilon_1 \leq \partial_{x_1}F_{i,k} \leq k+2-i+\epsilon_1 \leq N ,  
\end{equation} 
for all $(x_1,x_2) \in O_1$ with  $x_1 \in [-3/8 - \ec, -3/8 + \ec].$

In proving uniqueness, we will also use some bounds on second
%\footnote{Used a new bound $|\psi''(x) \leq 65|$.\ms{2/26/15: This is already incorporated when you wrote ST construction section.}} 
order partial derivatives.  
Note that defining functions of sheets that meet a cusp edge are $C^1$ but not $C^2$ along the cusp edge; see the standard form in Proposition \ref{prop:CuDef}.  However, we can write
\begin{equation}  \label{eq:E}
F_{i,k} = E - C A_k  \quad \mbox{where} \quad E = F_i-H
\end{equation}
and $E$ is $C^\infty$.  
Consider $(x_1,x_2) \in O_1 \cap \{-3/8-\ec \le x_1 \le -3/8+\ec\}.$
The linearity of $f^U_i(x_1)$ (when $i \ne k,k+1$) and $\widehat{f}_{k+1}(x_1)$ from Proposition \ref{prop:CuDef} item (8') and (\ref{eq:linearfk1hat}), the bounds on the derivative of $\psi$ in (\ref{eq:psibounds}), and the above computations  imply the following upper bounds:
%\footnote{\ms{6/19/15: Do we need to observe (\ref{eq:phix2}) here as well? Is this for $O_1$ or $O_1 \cap \{|x_1 --3/8| <\epsilon_2\}$?}}
\[
\left|\frac{\partial^2E}{\partial x_1^2} (x_1,x_2)\right| = 0;
%C \left|\frac{\partial^2A_{k}}{\partial x_1^2}\right| \leq \epsilon_1 ;
\]
\[
\left|\frac{\partial^2 E}{\partial x_2^2} (x_1,x_2)\right| = \left|\psi''(x_2)\cdot [ -f^U_{i,k+2}(x_1) + f^U_i(x_1) -\widehat{f}_{k+1}(x_1)]
%- C \, \frac{\partial^2 A_k}{\partial x_2^2}  
\right |  \leq 65 \cdot [ 3N \epsilon_1]  < 1/2;
\]
\[
\left|\frac{\partial^2 E}{\partial x_1\partial x_2} (x_1,x_2)\right| = \left|  \psi'(x_2)\cdot[K- k+2 -i ] 
%- C\,\frac{\partial^2 A_k}{\partial x_1\partial x_2} (x_1,x_2)
\right| \leq 3 \cdot 3/2 <5. 
\]

With these preliminaries out of the way, we now examine the functions $\theta_1$ and $\theta_2$.  For $(x_1,x_2) \in \Sigma_+$, by definition, we have
\[
\theta_2(x_1,x_2) = \frac{ \partial_{x_2}F_{i,k}}{\partial_{x_1}F_{i,k}}. 
\]
Using the inequalities (\ref{eq:12leqx2}) and (\ref{eq:12leqx1}) shows that for any $(x_1,x_2) \in \Sigma_+$,
\[
 \frac{1}{2 \,N}\leq \theta_2(x_1,x_2) \leq 2 N + 1.
\]
On the other hand, $\theta_1(x_1,x_2)$ is the slope of the tangent vector to $\Sigma_+$.  For $0 \leq x_2 \leq R_1/2 = 1/32$, $\theta_1$ takes non-negative real number values.  Moreover, at the swallow tail point, $\theta_1$ is $0$ while by  Lemma \ref{lem:Aprops}.A1  we 
%\dr{at by ???, we}\footnote{This bound and the next needs to be written into straightening section\ms{2/26/15:What straightening section?}. \dr{See 13.2.2 ``Straightening the cusp edge''.}} 
have  $\theta_1(-3/8,1/32) > 2N+1.1$.  Therefore, the Intermediate Value Theorem implies that there is at least $1$ point in $\Sigma_+ \cap \{ x_2 \leq 1/32\}$ where $\theta_1 = \theta_2$, i.e. there is at least one $(i,k)$-switch point.

Observe that Lemma \ref{lem:Aprops}.A1 gives  
\[
\theta_1(x_1,x_2) \in [2N+1.1, +\infty) \cup \{\infty\} \cup (-\infty, 0), \mbox{   when $x_2 \geq 1/32$,}
\]
so to verify the uniqueness claim of Step 2 we need only check that there is only one point in $\Sigma_+ \cap \{0 \leq x_2 \leq 1/32\}$ where $\theta_1$ and $\theta_2$ agree.  For this purpose, we use the parametrization 
\begin{equation}  \label{eq:paramSigma}
\left[\, 0\, ,\, \frac{M^{1/3}}{4^{4/3}} \,\right] \stackrel{\cong}{\rightarrow} \Sigma_+ \cap \{ x_2 \leq 1/32\}, \quad t \mapsto \left( x_1(t),x_2(t) \right) = \left( \frac{1}{M}(6 t^2 -X), \frac{1}{M}(8 t^3) \right).
\end{equation} 
Here, the constants $X$ and $M>0$ are those used in the definition of the diffeomorphism,  $S(x_1,x_2) = (M\,x_1+X, M\, x_2)$, from Section \ref{sec:Straight}.  Now, the second diffeomorphism $T$ used in that section is the identity when $|x_2| \leq 1/32$, so from (\ref{eq:6r28r3}) and (\ref{eq:SigmaT1S1}) we have
\[
\Sigma_+ = S^{-1}(\{(6 t^2, 8 t^3)\, |\, t\geq 0 \ \})\cap \{ x_2 \leq 1/32\}
\]
so that allowing $0 \leq t\leq \frac{M^{1/3}}{4^{4/3}}$ as in (\ref{eq:paramSigma}) does in fact parametrize $\Sigma_+$.

Now, as 
\[
x_1'(t) = \frac{12 \, t}{M}, \quad \mbox{and} \quad x_2'(t) = \frac{24 \, t^2}{M},
\]
we compute
\begin{equation}  \label{eq:theta1}
\theta_1(t) = \frac{x_2'(t)}{x_1'(t)} = 2 t, \quad \mbox{and} \quad \theta_1'(t) = 2.
\end{equation}
Therefore, the uniqueness statement will follow from the Mean Value Theorem provided we can show that $|\theta_2'(t)| < 2$ for all $0 \leq t\leq \frac{M^{1/3}}{4^{4/3}}$.  

In verifying that $|\theta_2'(t)|<2$, we will shorten notation to $F = F_{i,k}$.
% and let $|| u || = \sup \{|u(x_1,x_2)|\,  :\, (x_1,x_2) \in \Sigma_+\}$ for arbitrary $u.$  
We have
\[
\theta_2(t) = \frac{\partial_{x_2} F(x(t))}{\partial_{x_1} F(x(t))}.
\]
Moreover, using notation as in (\ref{eq:E}), 
\begin{equation} \label{eq:E2}
\partial_{x_m} F(x(t)) = \partial_{x_m} E(x(t)) - C \partial_{x_m} A_k(x(t)).
\end{equation}
The second term can be computed quite explicitly using Proposition \ref{prop:stCompute}.  There, a parametrization of the preliminary version of the swallowtail $L_{st}$ by variables $(r,s)$ is given, and the upper half of the cusp edge is given by setting $r=s$ with the gradient of the local defining function along the cusp edge given by
\[
\nabla a_{k}(x_1(s,s),x_2(s,s)) = (-s^2,s).
\]
(as in equation (\ref{eq:nablaz})).
After the modification of Section \ref{sec:Straight}, for $0 \leq x_2 \leq 1/32$, the local defining function becomes
\[
A_{k}(x_1,x_2) = a_{k}\circ S(x_1,x_2).
\]  
Moreover, the parametrization of the cusp edge used in (\ref{eq:paramSigma}) is related to the parametrization from Proposition \ref{prop:stCompute} by 
\[
(x_1(t),x_2(t)) = S^{-1}(x_1(t,t),x_2(t,t)).   
\]
Therefore, since the chain rule gives
\[
\nabla A_{k}(x_1,x_2) = M \cdot \nabla a_{k}(S(x_1,x_2)), 
\]
we have
\begin{equation} \label{eq:E3}
\nabla A_{k}(x(t)) = (\partial_{x_1}A_{k}(x(t)), \partial_{x_2}A_{k}(x(t))) = M\cdot(-t^2,t).
\end{equation}

Now, combining (\ref{eq:E2}) and (\ref{eq:E3})   gives 
\[
\theta_2(t) = \frac{\partial_{x_2} F(x(t))}{\partial_{x_1} F(x(t))} = \frac{\partial_{x_2} E(x(t))- M \cdot C \,t}{\partial_{x_1} E(x(t))+ M \cdot C \,t^2}
\]
We can then compute $\theta'_2(t)$ using the quotient and chain rules
\begin{align*}
\theta_2'(t)  = & \frac{\left[\frac{\partial^2E}{\partial x_1 \partial x_2} \cdot x'_1(t) + \frac{\partial^2E}{\partial x_2^2}\cdot x_2'(t) - M\cdot C\right] \cdot \partial_{x_1}F - \partial_{x_2}F \cdot \left[ \frac{\partial^2E}{\partial x_1^2}\cdot x_1'(t) + \frac{\partial^2E}{\partial x_1\partial x_2} \cdot x_2'(t) + 2 M \cdot C \, t\right]}{\partial_{x_1}F(x(t))^2}. \\
\end{align*}
Using that the absolute values of the second partial derivatives of $E$ are bounded by $5$, together with (\ref{eq:12leqx2}) and (\ref{eq:12leqx1}),
% = \sup \{ |u(x_1,x_2)| \, : \, (x_1,x_2) \in \Sigma_+, \,\, 0 \leq x_2 \leq 1/32\}$, 
we have
\begin{align*}
|\theta_2'(t)| \leq & \left[5|x'_1(t)| + 5| x_2'(t)| + M\cdot C\right] \cdot N + (2N+1)\cdot \left[ 5|x_1'(t)| + 5| x_2'(t)| + |2 M \cdot C \, t|\right]. \\
\end{align*}

%\begin{align*}
%|\theta_2'(s)|  = & \left| \frac{[\partial_{x_2} F(x(s))]' \cdot \partial_{x_1} F(x(s)) - \partial_{x_2} F(x(s)) \cdot[\partial_{x_1} F(x(s))]'}{(\partial_{x_1} F(x(s)))^2} \right| \\
 %\leq & \frac{\left[ \|\frac{\partial^2F}{\partial x_1 \partial x_2}\| \cdot |x'_1(s)| + \|\frac{\partial^2F}{\partial x_2^2}\|\cdot |x_2'(s)| \right] \cdot \| %\partial_{x_1}F\| + \|\partial_{x_2}F\|\cdot \left[ \|\frac{\partial^2F}{\partial x_1^2}\|\cdot|x_1'(s)| + \|\frac{\partial^2F}{\partial x_1\partial x_2}\|\cdot|x_2'(s)| \right]}{|\partial_{x_1}F(x(s))|^2} \\
% \leq & 2\left( \frac{24}{4^{8/3} M^{1/3}}\right) \frac{N(1/24) + 2 N + 1 }{(1/24)^2}.
% \leq & \left(\left[ \|\frac{\partial^2F}{\partial x_1 \partial x_2}\| \cdot |x'_1(s)| + \|\frac{\partial^2F}{\partial x_2^2}\|\cdot |x_2'(s)| \right] \cdot \| \partial_{x_1}F\| + \right. \\
% & \left. \|\partial_{x_2}F\|\cdot \left[ \|\frac{\partial^2F}{\partial x_1^2}\|\cdot|x_1'(s)| + \|\frac{\partial^2F}{\partial x_1\partial x_2}\|\cdot|x_2'(s)| \right]\right)/|\partial_{x_1}F(x(s))|^2 
%\end{align*}
%[To obtain the final inequality, we used the bound of $1$ for all norms of second derivatives; used that 
%\[
%\mbox{Max}\{|x'_1(s)|,|x'_2(s)| \} \leq \mbox{Max}\{|x'_1\left(\frac{M^{1/3}}{4^{4/3}}\right)|,|x'_2\left(\frac{M^{1/3}}{4^{4/3}}\right)|\} = \frac{24}{4^{8/3} M^{1/3}};
%\]
%and used bounds from (\ref{eq:12leqx2}) and (\ref{eq:12leqx1}).]  
%Note that\footnote{\ms{2/26/15: Setting $x'_1(s) = 12s/M, x_2'(s)=24s^2/M,$ I get $24s^2/M =24/(4^{2/3}M^{1/3}).$} \dr{1/27/16: I get what is written. }}
\[
\mbox{Max}\{|x'_1(t)|,|x'_2(t)| \} \leq \mbox{Max}\left\{\left|x'_1\left(\frac{M^{1/3}}{4^{4/3}}\right)\right|,\left|x'_2\left(\frac{M^{1/3}}{4^{4/3}}\right)\right|\right\} = \frac{24}{4^{8/3} M^{1/3}} < \frac{24}{M^{1/3}},
\]
so we can estimate
\begin{align*}
|\theta_2'(t)| & \leq \left[ \frac{240}{M^{1/3}} + M\cdot C\right]N + (2 N +1) \left[ \frac{240}{M^{1/3}} + 2 M C\left(\frac{M^{1/3}}{4^{4/3}} \right) \right] \\
 & \leq (2 N+1) \left( \frac{480}{M^{1/3}} + (M+M^{4/3})C \right) < 1/2 + 1/2.
\end{align*}
where at the last inequality we used the lower bound on $M$ from (P4) in Section \ref{sec:Straight} and the upper bound (\ref{eq:CACO3}) on $C$.
%\footnote{\ms{2/26/15:Is this suppose to be (\ref{eq:CAijCO2}) and (\ref{eq:CACO2})? How does $C$ control $u?$} \dr{1/27:  It was supposed to be $M$ is really big, so it makes the $x_m'$ terms small, and $C$ is even smaller than $M$ is big.  The second part needed to be added to the definition of $C$.}} 
This gives  the desired inequality $|\theta_2'(t)|<2$.

%Now, as $M$ was chosen to satisfy\footnote{Write this in.} (\ref{}), it follows that $|\theta_2'(s)|<2$ which completes the proof of (1) and (2).

%The statement (3) is established as follows.  With $\theta_1$ and $\theta_2$ now considered on $\Sigma \cap \{ x_2 \leq 0 \}$.  

\medskip

\noindent {\bf Step 3.}  {\it Non-existence of $(i,k)$ and $(k+2,j)$ switch points within $O_1 \cap \{x_2 \leq 0\}$ for $i<k$ and $k+2 < j$.}

\medskip

We establish the non-existence of $(i,k)$ switch points. The $(k+2,j)$ switch points are similar.
 When $x_2 \leq 0$, the slope of the cusp locus is negative or, when $x_2 \leq -1/32$, contained in $(-\infty, -(2N+1.1)] \cup \{\infty\} \cup [2N+1.1, +\infty)$ by Lemma \ref{lem:Aprops}.A1.  In contrast, (\ref{eq:12leqx2}) and (\ref{eq:12leqx1}) show that the slope of $\nabla F_{i,k}$ is positive and bounded above by $2 N+1 < 2N+1.1$.  

\medskip

\noindent {\bf Step 4.}  {\it Existence and uniqueness of $(k+1,k+2)$ and $(k+2,k+1)$ switch points.}

\medskip

Again, we let $\theta_1(t)$ and $\theta_2(t)$ denote the slope of the cusp locus and $\nabla F_{k+1,k+2}$ respectively at the point 
\[
(x_1(t),x_2(t)) = T^{-1} \circ S^{-1} ( 6 t^2, 8 t^3).
\]  To obtain the entire portion of the cusp locus within $O_1$ where $-1/16 \leq x_2 \leq 1/16$ we need to consider $-M^{1/3}/2^{7/3} \leq t \leq M^{1/3}/2^{7/3}$.  Within $O_1$,  we have
\[
F_{k+1,k+2} = C A_{k+1,k+2} = C\cdot a_{k+1,k+2} \circ S \circ T.
\]
Thus, for a point $(w_1,w_2) = T^{-1} \circ S^{-1} (  x_1,  x_2)$ in $O_1$, we can compute the differential of $\left(\frac{1}{C}\right) F_{k+1,k+2}$ to be
\[
d(a_{k+1,k+2} \circ S \circ T)(T^{-1} \circ S^{-1} ( x_1, x_2)) = d a_{k+1,k+2} (x_1,x_2) \cdot dS ( S^{-1}(x_1,x_2)) \cdot d T (w_1, w_2).
\]
Now, from definitions of $S$ and $T$ in Section \ref{sec:Straight}, we have 
\begin{equation} \label{eq:dScc}
dS = \left[ \begin{array}{cc} M & 0 \\ 0 & M \end{array} \right] \quad \mbox{and} \quad d T(w_1,w_2) = \left[ \begin{array}{cc} 1 & b'(|w_2|) \cdot \sgn(w_2) \\ 0 & 1 \end{array} \right].
\end{equation}

To evaluate $d a_{k+1,k+2} (x_1,x_2)$ using Corollary \ref{cor:aijGrad} we need to make use of the parametrization from (\ref{eq:stParam}) with parameter values belonging to the wedge $W$ (see Figure \ref{fig:STParam}). To parametrize the upper half of the cusp locus we take $(r,s) = (t,t)$ with $0 \leq t \leq M^{1/3}/2^{7/3}$, and for the bottom half we take $(r,s) = (t, -2t)$ with  $-M^{1/3}/2^{7/3} \leq  t \leq 0$.  Notice that these substitutions into (\ref{eq:stParam}) produce $(x_1(t),x_2(t)) = ( 6 t^2, 8 t^3)$, so that plugging $(r,s) = (t,t)$ or $(r,s) = (t, -2t)$ into Corollary \ref{cor:aijGrad} to evaluate the  $d a_{k+1,k+2} (x_1,x_2)$ term from  (\ref{eq:dScc}) and then transposing the result will give the vector whose slope is $\theta_2(t)$.  In this manner, Corollary \ref{cor:aijGrad} produces
\[
d a_{k+1,k+2}(x_1(t),x_2(t)) = \left\{ \begin{array}{cr} (t+2t)(t,1) & \mbox{if $t \geq 0$,} \\
(t+2(-2t)) (t,1) & \mbox{if $t \leq 0$.} \end{array} \right.
\]
Thus, regardless of the sign of $t$, using (\ref{eq:dScc}) we have
\[
\theta_2(t) = \frac{t \cdot b'(|w_2|) \cdot \sgn(w_2) + 1}{t}.
\]
When $|w_2| \leq 1/32$, this simplifies to $\theta_2(t) = 1/t$, and in this range we have already computed in (\ref{eq:theta1}) that $\theta_1(t) = 2t$.  Therefore, there is a $(k+1,k+2)$-switch point at $t= 1/\sqrt{2}$, and a $(k+2,k+1)$-switch point $t= -1/\sqrt{2}$.   These points are both in the range $|t|  \leq M^{1/3}/2^{8/3}$ where $|w_2| \leq 1/32$ by the requirement (P4) from Section \ref{sec:Straight}.

On the other hand, when $1/32 \leq |w_2| \leq 1/16$,  Lemma \ref{lem:Aprops}.A1 gives that
\[
|\theta_1(t)| \geq 2N+1.1,
\]  
while we can use the defining properties of $b$ in Section \ref{sec:Straight} to estimate
\[
|\theta_2(t)| \leq |b'(|w_2|)| + \frac{1}{|t|} \leq \frac{1}{3(2N+1.1)} + \frac{1}{|t|} \leq 1 
\]
where we have used $M^{1/3}/2^{8/3} \leq |t| \leq M^{1/3}/2^{7/3}$ and
%\footnote{\ms{2/26/15: I assume second $2^{8/3}$ is $2^{7/3}$.} \dr{1-27.  Right, thanks!}} 
the lower bound on $M$ from (P4) of \ref{sec:Straight} to obtain the last inequality.
%\end{proof}

\end{proof}

%\subsubsection{Proofs of Properties \ref{pr:STBnew},  \ref{pr:leftC}, \ref{pr:monoLtilde} and \ref{pr:SwitchBarriers}}
%We prove the remaining properties of Section \ref{sec:SwallowComp} which have shorter proofs.

\begin{proposition}
\label{prop:PropertySTBnew}
Property \ref{pr:STBnew} holds.
%\footnote{\ms{5/31/15: corrections needed in region containing $P_2$ as written in property statement. I will change the region to $[-1-1/32, 1/2] \times [-17/64, -1/4].$}}
\end{proposition}

\begin{proof}
Let $P_1$ be the line segment from $(1/3 - \epsilon_1, -3/8)$ to $(1/3 + \epsilon_1, 3/4).$
To verify item (2) of the property, fix a pair of sheets $S_i,S_j$ with $i<j$.
For any $(x_1,x_2) \in P_1,$ since $x_1 \geq 1/4$, %(\ref{eq:FHCAG})
%, the second formula in
% (\ref{eq:FHCAG}) is not applicable, and
%$\widehat{f}_{k+1}(x_1) = f^U_{k+1}(x_1) = f^D_{k+1}(x_1), \widehat{f}_{k+2}(x_1) = f^U_{k+2}(x_1) =f^D_{k+2}(x_1)$
%in (\ref{eq:Gk1def})-(\ref{eq:Gk2def}).
%Also, $\phi(x_1) = 0$ wherever it appears in the formulae.
%
equations (\ref{eq:FHCAG}) and (\ref{eq:interpolating}) simplify to
\[
F_{i,j}(x_1,x_2) = f^R_{i,j}(x_2) + \phi(x_2) f^U_{i,j}(x_1)
+(1-\phi(x_2)) f^D_{\sigma_D(i),\sigma_D(j)}(x_1).
\]
Using item (6) of Corollary \ref{cor:summary} gives 
\[
|\partial_{x_1}F_{i,j}(x_1,x_2)| \leq  \phi(x_2) |d_xf^U_{i,j}(x_1)|
+ (1-\phi(x_2)) |d_xf^D_{\sigma_D(i),\sigma_D(j)}(x_1)| \leq N \ed.
\]
On the other hand, for $(i,j) \neq (k+1,k+2)$, Lemma \ref{lem:18est} (which, as observed in the proof of Lemma \ref{lem:tildeUtildeDtildeRST}, remains true in the region $\overline{R}$ for the (13) square), combined with the fact that, when $|x_2| \geq 1/4$, $\partial_{x_2}F_{i,j} = d_xf^R_{i,j}(x_2)$ gives
\[
\partial_{x_2}F_{i,j}(x_1,x_2) \geq \mbox{Min}\{ d_xf^R_{i,j}(x_2), 1/5\}. 
\]
Thus, when $x_2 < \beta^R_{i,j} -\e = \eta^R_{i,j} - \e_\eta$, the $x_2$-component of $-\nabla F_{i,j}$ is negative, and to see that $-\nabla F_{i,j}$ points to the right of $P_1$ it suffices to verify that
\[
\frac{|\partial_{x_2}F_{i,j}(x_1,x_2)|}{|\partial_{x_1}F_{i,j}|} \geq \frac{\mbox{Min}\{ d_xf^R_{i,j}(x_2), 1/5\}}{N\ed} > \frac{9}{16 \ea} = \mbox{Slope($P_1$)}.
\]
[At the second inequality we used the definition of $\ed$ from (\ref{eq:epsilon4}).]

%
%In order to write all cases down in one equation, we define some new notation.
%Recall $\sigma_D(l)$ denotes the ordering of sheet $S_l$ as it appears over $x_1 = 1, x_2 = -1.$
%(So $\sigma_D(l) = l$ if $l \ne k_1,k+2.$)
%Define $D(l,x_2) \in \{D,U\}$ to be: 
%$D$ if $l \ne k+1,k+2,$ or if $j \in \{k+1,k+2\}$ and $x_2 \ge 0;$ $U$ if $l \in \{k+1,k+2\}$ and $x_2 < 0.$
%Equations (\ref{eq:FHCAG}) (for $F_{k}, F_{k+1}, F_{k+2}$) and (\ref{eq:interpolating}) (for all other $F_l$)
% imply that
%\[
%F_{i,j}(x_1,x_2) = f^R_{i,j}(x_2) + 
%(1-\phi(x_2)) (f^{D(i,x_2)}_{\sigma_D(i)}(x_1)-f^{D(j,x_2)}_{\sigma_D(j)}(x_1)) + \phi(x_2) f^U_{i,j}(x_1).
%\]
%
%Let $y^{D(i,j)}$ denote the maximum value, over $-1 \le x_2 \le 1,$ of  
%$f^{D(i,x_2)}_{\sigma_D(i)}(x_2)-f^{D(j,x_2)}_{\sigma_D(j)}(x_2).$
%Note that $y^U_{i,j} - y^{D(i,j)} =0$ if $i,j \notin \{k+1,k+2\}.$
%
%
%
%If $-1/4 \le x_2 \le 1/4$ then 
%the Lemma \ref{lem:MainPropsCr}  and \ref{lem:MainPropsCu} analogues of 
%Lemma \ref{lem:MainPropsPV} items (2),(5), (6), as well as equation (\ref{eq:InterpolatePhi}) imply
%\begin{eqnarray*}
%\partial_{x_2} F_{i,j}(x_1,x_2) & =  & d_x f^R_{i,j}(x_2) + 
%d_x \phi(x_2) \left(f^U_{i,j}(x_1) - (f^{D(i,x_2)}_{\sigma_D(i)}(x_1)-f^{D(j,x_2)}_{\sigma_D(j)}(x_1))\right),\\
%& \ge & 2y^R_{i,j} - 3N \epsilon_1  - (2+\epsilon_2)(|y^U_{i,j} - y^{D(i,j)}|+ \epsilon_1)
%\end{eqnarray*}

Next, define $P_2$ to go from the lower endpoint of $P_1$ to $x_1= -7/8$ with slope $-\ea$, and then continue with slope $0$ to the left boundary of $N(e^2_\alpha)$.
For any $(x_1,x_2)$, since $-3/8 \leq x_2 \leq -1/4$ and all $f^{ST}_l$ functions agree with $f^L_k$ in this region, we have
\[
F_{i,j}(x_1,x_2) = f^D_{\sigma_D(i),\sigma_D(j)}(x_1) +(1-\phi(x_1)) f^L_{\sigma'(i),\sigma'(j)}(x_2) + \phi(x_1) f^R_{i,j}(x_2).
\]
where $\sigma'$ is as in (\ref{eq:sigmaprime}).
In verifying item (4), we need only consider the case where $i<j$ and $(i,j) \neq (k+1,k+2)$, and this implies that for $-7/8\leq x_1 \leq 1/2$,
\[
\partial_{x_1}F_{i,j}(x_1,x_2) >0
\]
since when $-1/4 \leq x_1 \leq 1/4$, $d_xf^D_{\sigma_D(i),\sigma_D(j)}$ dominates the other terms which are each bounded by $4N\ea$ in absolute value via items (3) and (5) of Corollary \ref{cor:summary} and (\ref{eq:C0flst}).  In addition, $\partial_{x_2}F_{i,j}(x_1,x_2)$ is an interpolation of non-negative terms.  We conclude that for $-7/8\leq x_1 \leq 1/2$, $-\nabla F_{i,j}(x_1,x_2)$ points into the closed lower left quadrant, and therefore into the region below $P_2$. 

Now, when $x_1 \leq -7/8$ there is only one sheet (defined by $\tilde{F}_k$) that can produce an $f^L_k$ term in $f^L_{\sigma'(i),\sigma'(j)}$.  Therefore, when $0\leq x_1 \leq-7/8$, we have strict inequality $-\partial_{x_2}F_{i,j}(x_1,x_2)<0$ as required.  
Finally, by Property \ref{pr:1cmono}, $-\nabla F_{i,j}$ continues to point downward when $P_2$ passes into the part of $N(e^1_L)$ that lies in a neighboring square.

\end{proof}

\begin{proposition}
\label{prop:leftC}
Property \ref{pr:leftC} holds.
\end{proposition}

\begin{proof}
Let $j \in \{k,k+1,k+2\}.$
The segment $x_1 = -17/64$ is to the left of $x_1 = -1/4$ (so $\phi(-17/64) =0$), to the right of $x_1 = -3/8+R_1/2$ (see (\ref{eq:hatk1})-(\ref{eq:hatk2})), 
and is disjoint from $O_2.$ 
So (\ref{eq:hatk1})-(\ref{eq:hatk2}),  (\ref{eq:Gkdef})-(\ref{eq:Gk2def})  and (\ref{eq:FHCAG}) imply 
%(*almost)
that in a neighborhood of the segment
\[
 F_{i,j}(x_1,x_2)  =   F_i -  G_j = 
f^{L}_i(x_2) - f^{ST}_{j}(x_2) + \phi(x_2) f^{U}_{i,j}(x_1) + (1-\phi(x_2)) f^{D}_{i,\sigma_D(j)}(x_1)
\]
%& =&  \phi(x_2) d_x f^U_{i,l}(x_1) + (1-\phi(x_2)) d_x f^{D}_{i,\sigma_D(l)}(x_1) + \partial_{x_1}(f_i^L(x_2) - f_l^{ST}(x_2))
%\end{eqnarray*}
where $\sigma_D$ transposes $k+1$ and $k+2.$
We compute
\[
\partial_{x_1} F_{i,j}(x_1,x_2) = \phi(x_2) d_xf^{U}_{i,j}(x_1) + (1-\phi(x_2)) d_xf^{D}_{i,\sigma_D(j)}(x_1) >0.
\]
The sum is a convex combination of positive terms since $i<j$ and $i< \sigma_D(j).$

%*The exception  to this is when $l = k+2$ and $x_2 > 0,$ in which case we may not be able to substitute 
%$(1-\phi(x_2)) f^{Cu}_{i,\sigma_D(k+2)}(x_1)$ as above for the $(1-\phi(x_2)) (f^U_i(x_1) - \widehat{f}_{k+1}(x_1))$ term coming from (\ref{eq:Gk2def}). This is because (\ref{eq:hatk2}) is not specified at $x_1 = -3/8.$ 
%Nonetheless, from (\ref{eq:Est21}) we get $d_x f^U_i(x_1) > d_x f^U_k(x_1) >
%d_x \widehat{f}_{k+1}(x_1),$ so our estimate still holds.
\end{proof}

\begin{proposition}
\label{prop:PropertymonoLtilde}
Property \ref{pr:monoLtilde} holds.
\end{proposition}

\begin{proof}

For any $(x_1,x_2)$ considered in the Property  \ref{pr:monoLtilde} statement, we showed in (\ref{eq:dx2Fiw}) that
$\partial_{x_2} (F_i  - w) \ge 1-6\epsilon_1-250\epsilon_1.$ The case for 
$\partial_{x_2} (w-F_j )$ is similar. This proves the first statement. Since $F_{k+1,k+2}^{-1}(0) \cap \{x_1 \leq -1/4\} \subset \{x_2 = 0\},$ this also proves
the crossing locus statement.

\end{proof}

\begin{proposition}
\label{prop:PropertySwitchBarriers}
Property \ref{pr:SwitchBarriers} holds.
%\footnote{\ms{4/1/15: I Assume $P_{k+1,k+2}, P_{k+2,k+1}$ are NOT vertical, as indicated by DR's footnote in the property statement.}}
\end{proposition}

\begin{proof}
Take $L_{k+1,k+2}$ to be a segment connecting the $(k+1,k+2)$-switch point to the $(k,k+2)$-cusp edge with sufficiently negative slope (but non-vertical) so that $L_{k+1,k+2}$ is contained in the portion of $O_1$ with $|x_2| \leq 1/32$.  [Such an $L_{k+1,k+2}$ exists by Property \ref{pr:switches}.]  Then, Lemma \ref{lem:Aprops} items (A3) and (A4) implies that at points of $L_{k+1,k+2}$ above or on the cusp locus $-\nabla F_{k+1,k+2}$ points into the closed lower left quadrant, and hence to the left side of $L_{k+1,k+2}$.  Also, at all points of $L_{k+1,k+2}$, except for the endpoint on the $(k,k+2)$-cusp locus where $\nabla F_{k,k+2} = 0$, $-\nabla F_{k,k+2}$ points into the closed lower left quadrant.   

For $L_{k+2,k+1}$, take a segment from the $(k+2,k+1)$-switch point to the $(k,k+1)$-cusp edge that has positive slope and is contained in $|x_2| \leq 1/32$, and use Lemma \ref{lem:Aprops} item (A3) and (A4) in a similar manner to verify the directionality of $-\nabla F_{k+2,k+1}$ and $-\nabla F_{k,k+1}$.

%Equation (\ref{eq:FHCAG}) implies $-\nabla F_{i,j} = -C \nabla A_{i,j}.$ 
% Lemma \ref{lem:Aprops} item (A3)  then implies the property.\footnote{\ms{6/19/15: I thought this was more subtle. Has the property changed recently? For example, originally I had written that item (A4) was also used...And why can't $L_{k+1,k+2}$ and $L_{k+2,k+1}$ just be the same vertical line?} \dr{2-1: I forgot to write into the property that transversality continues to hold for $-\nabla F_{k+1,k+2}$ at the point where $L_{k+1,k+2}$ intersects the crossing locus.  This seems to be needed for perturbation invariance of the property.}}
 \end{proof}

\section{Proof of Theorem \ref{thm:PropertiesofLtilde} Part 4: Transversality}
\label{sec:ProofSetup}

%\footnote{\ms{6/26/15: The following paragraph should make up part of proof environment to appear right after Thm 7.1 statement
%``
%The proof involves lots of basic but technical calculus to construct functions whose one-jets represent the Legendrian
%sheets in the square models (1)-(14) of Figure \ref{fig:generators}.
%In Section \ref{sec:Constructions} we construct the square models (1)-(12), establishing a few lemmas to be used later.
%In Section \ref{sec:ConstructionsST} we complete the construction in square models (13)-(14) when there is a swallowtail singularity present.
%In Section \ref{sec:Properties}, we show that our construction satisfies Properties 1-19.
%In Section \ref{sec:ProofSetup}, we show that we can perturb our explicit construction to achieve transversality of GFTs, while
%preserving the properties.
%}}

The Legendrian metric pair $(\tilde{L}_0, g_0)$ constructed in Sections \ref{sec:Constructions} and \ref{sec:ConstructionsST} has now been verified to satisfy all requirements of Theorem \ref{thm:PropertiesofLtilde} except for the $1$-regular condition.  In combination, the following Propositions \ref{prop:LemmaA1} and \ref{prop:LemmaA2} show that the $1$-regular condition may be obtained by a perturbation that preserves the other properties required in Theorem \ref{thm:PropertiesofLtilde}.  
%In Proposition \ref{prop:LemmaA1}, we give a mild extension of  \cite[Theorem 1.1 (a)]{Ekholm}  to show that $(\tilde{L}_0,g_0)$ can be made $1$-regular after a $C^\infty$-small perturbation that leaves $\tilde{L}_0$ fixed in a neighborhood of the cusp edges of $\tilde{L}_0$.  We then show in Proposition \ref{prop:LemmaA2} that such a perturbation can be assumed to preserve we extend the To complete the proof of Theorem \ref{thm:PropertiesofLtilde}

\begin{convention} In this final section, the singular set of $L$ refers to cusp edges and swallow tail points only.  \emph{Note that this differs from earlier sections, where we included crossings arcs in the singular set.}
\end{convention}

\begin{proposition}  \label{prop:LemmaA1}
Consider the Legendrian $L \subset J^1(S)$ with metric $g$ on $S$ as constructed in Sections \ref{sec:Constructions}  and \ref{sec:ConstructionsST}.
There exists an open neighborhood $U \subset L$ of the singular set of $L$ (pre-image of cusp edges and swallowtail points) such that within any neighborhood of $(L,g)$ in the $C^\infty$-topology, there exist $(L',g')$ such that 
$L$ agrees with $L'$ on $U$ and $L'$ is $1$-regular with respect to $g'.$
\end{proposition}

\begin{proof}
%Summary:  Ekholm shows in \cite{Ekholm07} how to arrange that an arbitrary Legendrian metric pair $(L,g)$ become $1$-regular after a $C^\infty$-small perturbation.  This involves an initial perturbation to arrange Preliminary transversality conditions...describe.  In our setting this is unnecessary for the following reason.  

%Given a Legendrian with PTC, GFTs are made to be transversally cut out by arranging an inductively defined sequence of Secondary Transversality Conditions.  Starting with all intersections are made to be transverse.  Then the flow outs of these intersections are made to be transverse.  As we are only interested in GFTs with a single positive puncture, He establishes an A Priori bound that shows the process cannot go on forever.  

%At the base case and each inductive step, transversality of a given flow out with all existing flow outs is arranged via perturbation of $(L,g)$ in a neighborhood of critical points or in neighborhood of the flow out.  

This is a modification of \cite[Theorem 1.1a]{Ekholm07} which states that any Legendrian metric pair $(L,g)$ can be made $1$-regular by a $C^\infty$-small perturbation.  
%, which follows from  \cite[Lemma 3.7 and Proposition 3.14]{Ekholm07}, with a modification of the first part of \cite[Proposition 3.14]{Ekholm07}, specifically, pages 1113--1114.  
We will show that in our setting the perturbation procedure used in the proof of \cite[Theorem 1.1a]{Ekholm07} can be carried out leaving a neighborhood $U \subset L$ of the cusp locus fixed. 

We review the outline of the proof of \cite[Theorem 1.1a]{Ekholm07}:

\begin{enumerate}
\item   After making a small perturbation if necessary, it is assumed that $(L,g)$ satisfies certain Preliminary Transversality Conditions (PTC).
\item  When $(L,g)$ satisfies PTC, any GFT, $\Gamma$, is assigned a {\it geometric dimension}, $\mathit{gdim}(\Gamma)$, and \cite[Lemma 3.7]{Ekholm07} establishes the inequality
\[
\mathit{gdim}(\Gamma) \leq \mathit{fdim}(\Gamma),
\]  
where $\mathit{fdim}(\Gamma)$ is the formal dimension of $\Gamma$.  (See Section \ref{sec:Background}.)
%Recall that a GFT has a the formal and geometric dimensions of a GFT from Section \ref{sec:Background}. 
\item \cite[Proposition 3.14]{Ekholm07} states that for any $D>0$, there exists a $C^\infty$-small perturbation of $(L,g)$ after which any GFT with $\mathit{fdim}(\Gamma) \leq D$ is transversally cut out, and a space of nearby GFTs with the same geometric properties as $\Gamma$ (e.g. homeomorphic domain trees) forms a manifold of dimension $\mathit{gdim}(\Gamma)$.  
%whose (true) dimension equal their geometric dimensions. 
\end{enumerate}
Taking $D =1$, the Legendrian resulting after (3) is $1$-regular.  [A GFT of negative formal dimension would have to belong to a manifold of negative dimension and therefore cannot exist.]  
%And if a GFT has formal dimension 0,  its true dimension is no bigger than 0; thus, the GFT exists, is rigid, and is transversely cut-out, if and only if it has the appropriate vertices.

To establish our strengthened result, we show that 
\begin{itemize}
\item[(i)] our $(L,g)$ already satisfies PTC, and hence does not need to be perturbed at part (1) of the argument, and 
\item[(ii)] there is a neighborhood $U$ of the cusp locus that can be left fixed during the perturbation at part (3).
\end{itemize}

\medskip

%Second, we briefly review some conditions on $L$ needed before modifying the \cite[p. 1113--1114]{Ekholm07} argument: ``front-genericity" \cite[p.1093]{Ekholm07} and ``preliminary transversality conditions" 
%\cite[p.1101--1103]{Ekholm07}.
%Let  $\Sigma \subset S$ denote the cusp locus, that is, the set of points over which at least one sheet fails to be immersed. 
%This set has a stratification  $\Sigma = \Sigma_1 \supset \Sigma_2 \supset \ldots.$ 
%A Legendrian surface is front generic if $\Sigma_1$ is an immersed one-manifold representing where  two sheets come together in standard cuspoidal model; $\Sigma_2$ is an isolated set of points representing transverse self-intersections of $\Sigma_1$ and standard swallowtail points; and $\Sigma_3 = \emptyset.$
%Property \ref{pr:14models} and the construction of the swallowtail point in Section \ref{sec:funcO1} imply $L$ is front-generic with no perturbation required.
%For $L$ front generic, 
{\bf (i):}  We list the Preliminary Transversality Conditions  following  
\cite[p.1101--1103]{Ekholm07} simplified to our two-dimensional set-up.  First $L$ should have generic base and front projections.  Then, the singular set $\Sigma \subset S$, that is, the set of points over which at least one sheet fails to be immersed has a stratification  $\Sigma = \Sigma_1 \supset \Sigma_2.$  Here, $\Sigma_2$ is an isolated set of points representing projections of standard swallowtail points and transverse intersections of the projection of two cusp edges.
The Preliminary Transversality Conditions are:
\begin{enumerate}
\item 
All switch points belong to $\Sigma_1\setminus \Sigma_2$ and are non-degenerate.  
%Moreover, if a switch point involves sheets $S_i,S_k,S_k+1$ with $S_{k}, S_{k+1}$ forming the cups pair, then $\nabla F_{i,k}$ has a first order tangency to $\Sigma_1$at that point.
\item
Let $\Sigma$ and $\Sigma'$ be the cusp loci of the two locally defined sheet pairs $S_{k},S_{k+1}$ and $S_{l},S_{l+1}.$ 
For $x \in \Sigma \cap \Sigma',$ $\nabla F_{k, l}(x)$ is non-vanishing. 
\item
At a swallowtail point, all $\nabla F_{i,j}$ (involving at least one sheet that does not meet the swallowtail point) are transverse to the two cusp loci.
\end{enumerate}
That $(L,g)$ satisfies Condition 1 follows from Property \ref{pr:CuspTransversality} and \ref{pr:switches}.
Condition 2 follows from Property \ref{pr:CuspTransversality}.
Condition 3 follows from Claim 1 and 2 in proof of Lemma \ref{lem:tildeLST} since the tangent to the cusp locus is horizontal at the swallowtail point. 
%Alternatively, Condition 2, 3 and all cases of Condition 1 aside from the switch point involving the three sheets that make the swallowtail, follow after a generic small perturbation of the Legendrian away from its cusp edges.

\medskip

{\bf (ii):}  The outline of the proof of \cite[Proposition 3.14]{Ekholm07} is as follows.  Spaces of PFTs and GFTs are made to be transversally cut out through an inductive procedure.  At the inductive step, we are given a collection of spaces of PFTs and GFTs that have already been made transversally cut out.  By a perturbation, transversality is then arranged for new spaces of GFTs arising from joining existing PFTs together at a vertex as allowed by the definition.  (The perturbation is taken to be small enough to preserve transversality of existing spaces as well.)  An a priori bound on the number of edges and vertices for GFTs satisfying a given bound on formal dimension and number of positive punctures is established in \cite[Lemma 3.12]{Ekholm07}, and this guarantees that the inductive procedure is completed after a finite number of steps.

It remains to verify that holding $L$ fixed near its cusps does not disrupt the above transversality argument.
From this Proposition's statement, let $U$  be a small neighborhood of the cusp loci so that if $x \in U$ is close to the cusping sheets 
$S_k,S_{k+1}$ then $F_{k,k+1}(x)< \delta$, where $\delta >0$ is much smaller than the positive critical value for the smallest Reeb chord of $L$
%(which is $\epsilon_3$ by Lemma \ref{lem:MainPropsPV} item (4) and analogous statements in Lemmas \ref{lem:MainPropsCr}-\ref{lem:MainProps2Cr}) 
and the value of the smallest switch point.  
By this second condition, we mean if $\nabla F_{i,j}$ is tangent to the cusp locus of $S_j, S_l$ at $Q$ then  $\delta < |F_{i,j}(Q)|.$  Moreover, we can assume that the base projection of $U$ sits in (one side of) a small neighborhood of $\Sigma$, and $\pi_x|_U$ is at most $2$-to-$1$ except 
\begin{itemize}
\item[(a.)] in neighborhoods of swallow tail points that we take to be small enough so the difference between any of the $3$ swallowtail sheets is less than $\delta$, and 
\item[(b.)] in Type (11) squares in a small region $W$ near the intersection point of the two cusp loci at $(-3/8,-3/8)$.  We now place some restrictions on $W$.  At $(-3/8,-3/8)$, for all $F_{i}> F_{j}$ that do not meet one another at a cusp edge, $-\nabla F_{i,j}$ has both components strictly negative.  [This is easily checked using (\ref{eq:interpolating}).]  Therefore, by compactness, there exist $0<m<M$ such that in a square $R$ centered at $(-3/8,-3/8)$ all such $F_{i}>F_{j}$ satisfy
\[
m < \frac{\partial_{x_2} F_{i,j}}{\partial_{x_1} F_{i,j}} < M;
\]
i.e., in $R$, $-\nabla F_{i,j}$ points into the wedge in the down left quadrant bounded by lines through the origin of slope $m$ and $M$.  
Consider a point $P=(-3/8+\alpha, -3/8+\alpha)$ with $\alpha$ chosen so that lines $l_m$ and $l_M$ of slope $m$ and $M$ through $P$ intersect both cusp loci, $\{x_1= -3/8\}$ and $\{x_2=-3/8\}$, within $R$.  Then, we require that the region $W$ is contained in the closed subset with bounding segments $\{x_1 = -3/8\}$ (on left), $\{x_1 = -3/8\}$ (below), $l_M$ (on right), and $l_m$ (above).  See Figure \ref{fig:W}.

%points strictly closed down left quadrant for all $F_i<F_j$ 
%These neighborhood should be small enough so that in their closures $-\nabla F_{i,j}$ points strictly closed down left quadrant for all $F_i<F_j$ such that $F_{i}$ and $F_{j}$ do not meet one another at a cusp edge.  (Such a neighborhood exists since $\partial_{x_i}F_{i,j}(-3/8,-3/8)  <0$ is easily checked from (\ref{eq:interpolating}).)
\end{itemize}
During the upcoming inductive perturbation, we can take all perturbations small enough to preserve the condition from (b.) that $-\nabla F_{i,j}$ points strictly into the wedge bounded by $l_m$ and $l_M$ within $R$. 

%\dr{Take a closed rectangle containing $(-3/8,-3/8)$.  By compactness, there exists some bound such that all of the mentioned $-\nabla F_{i,j}$ point into a wedge in the down left quadrant with upper boundary a line of positive slope.  and right boundary with large positive slope.  Require that the $4$-to-$1$ map of is in a region bounded by such a wedge and the cusp locus within that rectangle.  Then, the argument below goes through. }

\begin{figure}
\labellist
\small
\pinlabel $P$ [bl] at 180 180
\pinlabel $l_m$ [tr] at -2 110
\pinlabel $l_M$ [br] at 112 14
\pinlabel $R$ [b] at 70 228
\pinlabel $W$ at 156 156
\endlabellist
\centerline{ \includegraphics[scale=.6]{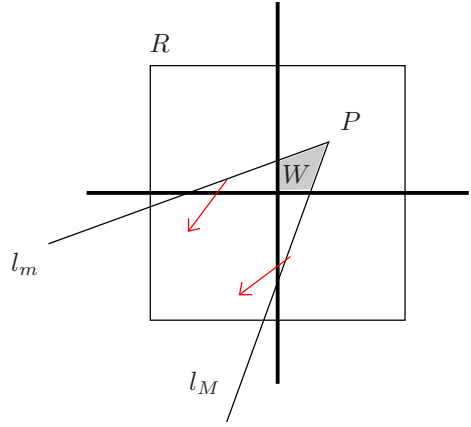} }
%\centerline{ \includegraphics[scale=.8]{images/SubdivideSq} }
\caption{The projection of $U$ to $S$ is $2$-to-$1$ or less outside of the shaded region $W$.  The thickened lines represent $\Sigma$ while the red arrows indicate the directionality of $-\nabla F_{i,j}$ (with $F_i$ and $F_j$ not meeting at a cusp) along the bordering segments $l_m$ and $l_M$ within the square $R$.}
\label{fig:W}
\end{figure}

%We will show that the perturbations used in this argument can be assumed to leave the 
As we consider only the case of GFTs with $1$ positive puncture, we can carry out the inductive procedure starting with the output vertices and working up.  To provide a slightly more detailed discussion of the induction from this point of view, we first introduce some terminology.  

Recall that the projection of the singular set is stratified as $\Sigma = \Sigma_1 \supset \Sigma_2$, and via the Preliminary Transversality Conditions
%,  $\Sigma_1\setminus \Sigma_2$ is has a further stratifications with codimension $1$ stratum the 
we can further refine $\Sigma$ by adding the image of switch points into $\Sigma_2$.
Given a subset $X \subset S$ with an upper and lower sheet of $S$  assigned to each point of $X$, the {\bf upward} (resp. {\bf downward}) {\bf flow out} manifold of $X$ is the union of all flow lines for $-\nabla F_{i,j}$  
starting at $x \in X$ when $t=0$ and continued in negative (resp. positive) time for their maximal interval of definition, where $F_i$ and $F_j$ define the upper and lower sheet assigned to $x$.

We make some observations and technical remarks: 
\begin{enumerate}
\item Flow out manifolds are not embedded since for instance, a flow line for two sheets may follow the Legendrian around and return to the same point above $S$ but with different upper and/or lower sheet.  To address this issue, one can consider flow out manifolds as belonging to the subset of the fiber product $L*L =\{(a,b) \in L \times L \,|\, \pi_x(a) = \pi_x(b)\}$ with $z(a) \geq z(b)$.   However, two or more of the initial points in $X$ may belong to the same flow line of $-\nabla F_{i,j}$, so that even in $L*L$ flow outs are non-embedded.  In the following, we (implicitly) view flow outs as maps to $L* L$.
\item Flow lines may reach the end of their domain of definition in finite time.  As such, flow out manifolds are stratified by their intersections with the various strata of $\Sigma$.
\item When $X \subset \Sigma$ has points on a cusp edge for the upper or lower sheet, there are multiple  upward/downward flow out manifolds determined by a choice of upper or lower sheet bordering the cusp edge.
%Converse statement.  
\item The sheet difference $F_{i,j}(x)$ decreases as $t$ increases.  In particular, at any point of an upward flow out manifold of $X$ the sheet difference is larger than at the corresponding initial point on $X$.  
\end{enumerate}

Returning now to the inductive procedure, at the base case we begin with the ascending manifolds of critical points, $X_1, \ldots, X_m$, and the upward flow out of all cusp edges (with respect to $-\nabla F_{k,k+1}$ where $F_k$ and $F_{k+1}$ are local defining functions for the upper and lower sheet at a cusp).  Denote the latter flow out manifold by $X_\Sigma$.  We arrange that $X_1,\ldots, X_m, X_\Sigma$  intersect the descending manifolds of critical points and descending manifolds of the stratified parts of the cusp locus transversally. (The latter condition assures that the stratified components of this initial collection of manifolds, as in (2) are again manifolds.)  In addition, we arrange that $X_1, \ldots, X_m, X_\Sigma$ all intersect one another transversally, and that their stratified components intersect transversally in the various strata of $\Sigma$.

 As discussed in \cite[pg. 1113-1114]{Ekholm07}, this is all accomplished by perturbing $(L,g)$ in a neighborhood of the critical points.  At any critical point of some $F_{i,j}$, at least one of the sheets is disjoint from $U$; perturbing this sheet accomplishes an arbitrary perturbation of $F_{i,j}$, so we can leave $U$ fixed at this step.  Note that no perturbation of $X_\Sigma$ is required since it is top-dimensional and hence has intersections that are automatically transverse.    

The intersection of the $X_i$ (resp. $X_\Sigma$) with descending manifolds of critical points are identified with GFTs with a single edge ending at a Reeb chord (resp. $e$-vertex).  The (iterated) intersections of (components of) the $X_i$ and $X_\Sigma$ with one another, restricted to points where the upper sheet of the first manifold agrees with the lower sheet of the second manifold,  specify possible  locations of a lower most vertex in a PFT or GFT.  We write this collection of intersections as $Z_1, \ldots, Z_n$, where if $Z_i = X_{i_1} \cap X_{i_2}$ with the lower sheet of $X_{i_1}$ agreeing with the upper sheet of $X_{i_2}$, then the upper (resp. lower) sheet assigned to $Z_i$ is the corresponding upper sheet of $X_{i_1}$ (resp. lower sheet of $X_{i_2}$).

Now, we inductively arrange that the upward flow out of each $Z_i$ is transverse to (i) descending manifolds of critical points, (ii) downward flow outs of the stratified components of $\Sigma$, (iii) all components of each of the $X_i$, $X_\Sigma$, and upward flow outs of $Z_l$ with $l<i$ at points where lower (resp. upper) sheet of $Z_i$ agrees with the upper (resp. lower) sheet of the other manifold.  As discussed in \cite[pages 1113--1114]{Ekholm07} this may be achieved by an arbitrarily small perturbation of $(L,g)$ along the upward flow out of $Z_i$.  We then append each non-empty intersection as in (iii), to the list of $Y_j$ with the upper and lower sheet assigned as above, and proceed inductively with  \cite[Lemma 3.12]{Ekholm07} showing that the process is completed after finitely many steps. 

Notice that since the sheet difference increases along ascending manifolds of critical points and upward flow out manifolds, we see that either: 
\begin{enumerate} 
\item The $Z_i$ arises (inductively) from some intersections of $X_{\Sigma}$ with itself.  In this case, since such intersections are top dimensional,  there is no need to perturb $Z_i$. 
\item Or, the sheet difference of the upper sheet and lower sheet of $Z_i$ must be larger than the value $\delta$ used to define $U$.
\end{enumerate}
  In case (2), the only way that both the upper and lower sheet of a point
     on $Z_i$ or its upward flow out belong to $U$ is if the endpoints are associated to two distinct cusp edges in a square of type (11) as in (b.) from the definition of $U$.  
	%In the definition of $U$, we arranged that this could only occur in squares of type (11) and (13)-(14) as prescribed in (a.) and (b.) above.  
	We now rule out the possibility of such a $Z_i$ existing, so that at this perturbation step we can always perturb a sheet disjoint from $U$.

A point on the upward flow out of $Z_i$ is the initial point of a PFT.  Thus, it suffices to show there is no PFT, $\Gamma$, that starts in the region $W$ from (b.) with upper endpoint belonging to one of $S_{k}, S_{k+1}$ and lower endpoint belonging to one of $S_{l}, S_{l+1}$ where $S_k,S_{k+1}$ and $S_{l},S_{l+1}$ are the two pairs of cusp sheets. 

Supposing $\Gamma$ exists, we build a path $\gamma$ from a sequence of consecutive edges of $\Gamma$ as follows.  Begin at the inital point in $W$, and follow the domain of $\Gamma$ according to its orientation. 
%we can find a sequence of edges  
% from its initial point in $W$ we 
Note that as long as the image remains in the square $R$, whenever a vertex is reached, at least one
 outgoing edge can be found that is a flowline involving one of the sheets $S_k,S_{k+1}, S_{l},S_{l+1}$, but is not a $(k,k+1)$-flow line or a $(l,l+1)$-flow line.  We continue $\gamma$ by following this edge.   
%Following such a sequence of edges, produces a path, $\gamma$, that is a union of edges (with their orientation) of $\Gamma$ starting in $W$ that while it 
Note that while $\gamma$ remains in $R$, it cannot cross the lines $l_m$ or $l_M$ (by the directionality of the $-\nabla F_{i,j}$).  Moreover, $\gamma$ cannot terminate at an $e$-vertex in $R$ since none of its edges are $(k,k+1)$ or $(l,l+1)$ flow lines.  Therefore,  
$\gamma$ must eventually leave $R$, but doing so without crossing $l_m$ or $l_M$ forces the image to pass into the region where none of the sheets $S_{k},S_{k+1},S_{l},S_{l+1}$ exist.  (See Figure \ref{fig:W}.)  This is impossible since along any edge of $\gamma$ the upper or lower sheet belongs to $S_k,S_{k+1},S_{l},S_{l+1}$.

\end{proof}

\begin{proposition}  \label{prop:LemmaA2}
Let $(L,g)$ denote the Legendrian and metric constructed in Sections \ref{sec:Constructions} and \ref{sec:ConstructionsST} that satisfies Properties 1-19.  Then, for any given open neighborhood, $U \subset L$, of the singular set, there exists some neighborhood $\mathcal{V}$ of $(L,g)$ in the $C^\infty$-topology such that if, 
\begin{enumerate}
\item $(L',g') \in \mathcal{V}$; and
\item $L'$ agrees with $L$ on $U$,
%\item $g'$ agrees with $g$ on $\pi_x(U)$,  
\end{enumerate}
then $(L',g')$ also satisfies Properties 1-19.
\end{proposition}

Given a neighborhood, $\mathcal{V}$, of $(L,g)$ in the $C^\infty$-topology, and an open set $U \subset L$, we write $\mathcal{V}_U$ to denote those $(L',g') \in \mathcal{V}$ such that $L'$ agrees with $L$ on $U.$

\begin{proof}
For a suitably chosen small neighborhood $\mathcal{V}$ of $(L,g)$, it is the case that for $(L',g') \in \mathcal{V}_U$, $L'$  
has the same singular set as $L$ (cusp edges and swallow tail points). We can further arrange that: 
\begin{enumerate}
\item[(I)] The local defining functions of $L'$ belong to chosen $C^2$-neighborhoods of the defining functions of $L$. 
\item[(II)] Gradient vector fields of defining functions for $L'$ with respect to $g'$ belong to chosen $C^2$-neighborhoods of the gradient vector fields of defining functions of $L$ with respect to $g$.  
\item[(III)] For gradients of difference functions $-\nabla F_{k,k+1}$ that correspond to two sheets within $U$, we can arrange a uniform bound between the direction of $-\nabla_{g'} F_{k,k+1}$ (where non-zero),
%viewed as a function form the base surface to $S^1$ belongs 
and the direction of $-\nabla_{g} F_{k,k+1}$.  (We will need to use this property for two sheets that meet at a cusp edge.  Since $\nabla_gF_{k,k+1}$ vanishes along the cusp edge, in a neighborhood of the cusp edge the uniform perturbation of direction does not follow from just the statement that $-\nabla_{g'} F_{k,k+1}$ $C^2$-approximates $-\nabla_{g} F_{k,k+1}$.)  
\end{enumerate}

[To verify (III), note that for $(L,g)$ and $(L',g') \in \mathcal{V}_U$, the differentials $dF_{k,k+1}$ and $dF'_{k,k+1}$ are identical in $U$.  Thus, the change between $-\nabla_g F_{k,k+1}$ and $-\nabla_{g'} F'_{k,k+1}$ is as a result of the isomorphism
\[
TS \stackrel{g^\#}{\longrightarrow} T^*S \stackrel{(g'^\#)^{-1}}{\longrightarrow} TS. 
\]
This isomorphism induces a map between the unit sphere bundles $S(TS) \rightarrow S(TS)$, that by taking $g'$ close enough to $g$ can be made $C^0$ close to the identity on the base projection of $U$ (which is a relatively compact subset of $S$).]

More concretely, for arbitrarily fixed $\e>0$, we can take $\mathcal{V}$ so that on any finite number of  compact subsets $K_1, \ldots, K_r \subset S$, contained in the domains of local defining functions $F_1, \ldots, F_r$, we have that for any $(L',g')$ in $\mathcal{V}_U$ the local defining functions for the same sheets $F_1', F_2', \ldots, F'_r$, satisfy, $||F_i'-F_i||_{C^2(K_i)}<\e$ and $||\nabla F'_{i}- \nabla  F_i||_{C^2(K_i)} < \e$, for all $1 \leq i \leq r$.  (We can use any choice of local coordinate on $S$ to set up the $C^2(K_i)$-norms.)  Moreover, $F_i=F_i'$ holds identically on $U$, and when $F_{k,k+1}$ is a difference function for two sheets that meet at a cusp edge, in the base projection of $U$, $\nabla F'_{k,k+1} \neq 0$ exactly where $\nabla F_{k,k+1} \neq 0$, and the angle between $\nabla F'_{k,k+1}$ and $\nabla F_{k,k+1}$ is less than $\e$.

%Supposing $(L,g)$ satisfies Properties 1-19, 
We need to show that we can make all $(L',g') \in \mathcal{V}_U$ satisfy Properties 1-19 by an appropriate choice of $K_i$ and $\e$.

\medskip

\noindent {\bf Property 1.}  The cusp locus of $L$ and $L'$ agree exactly. In any particular square, the crossing locus between sheets defined by $F_i$ and $F_j$ is the zero set of $F_i-F_j$, and it is transversally cut out, i.e. $0$ is a regular value, since $L$ is embedded.   Thus, a suitably $C^2$-small perturbation of $F_i-F_j$, results in the crossing locus of $L'$ remaining within any chosen $C^1$-neighborhood of the crossing locus of $L$. 

[By taking the $F'_i$ and $F'_j$ suitably close to $F_i$ and $F_j$ on $L\setminus U$, we can arrange that the crossing locus of $F'_i$ and $F'_j$ is contained in any given set of the form $\{|F_i-F_j| < \delta\}$, and hence keep the crossing locus of $L'$ within an arbitrary neighborhood of the crossing locus of $L$.  
%Taking $\delta$ small enough keeps the crossing locus within a $C^0$ neighborhood of itself.  
Moreover, the tangent to the crossing locus of $L'$ is the orthogonal complement of $\nabla_{g'} (F'_i-F'_j)$.  By keeping $\nabla_{g'}(F'_i-F'_j)$ close enough to $\nabla_{g}(F_i-F_j)$ and $g'$ close enough to $g$, we can keep $\nabla_{g'}(F'_i-F'_j)$ non-zero along the crossing locus of $L'$ and keep the respective $g$-orthogonal complements as close as desired.]

\medskip

\noindent {\bf Properties 2, 3, 7, 10, 11, 12, 14, 16, 17, 18, 19.}  All of these properties involve the existence of certain compact subsets (sometimes explicitly defined)
along which the gradients of particular difference functions have their directions constrained.  For instance, in Properties 2 and 3, gradients of positive local difference functions need to point strictly inward along the paths that form the boundary of $N(e^0_\alpha)$ and $N(e^1_\alpha)$, while Property 10 specifies that the $x_1$- and $x_2$-components of $\nabla (F_i-F_j)$  should have a particular (strict) sign along certain $2$-dimensional compact subsets of $N(e^2_\alpha)$.  In all cases, the constraints imposed on $-\nabla (F_i-F_j)$ can be rephrased as a finite number of properties of the form 
\begin{equation} \label{eq:strictgFX}
\forall \, x \in K,   \quad g_0(-\nabla (F_i-F_j), X) >0 
\end{equation}
with $K \subset N(e^d_\alpha)$ some compact set; $X$ is an appropriately defined (continuous) vector field (with image in $T S$) defined along  $K$, and $g_0$ is the euclidean metric in coordinates on some $N(e^2_\alpha)$ or $N(e^1_\alpha)$.   

With properties 2, 7, 10, 11, 12, 14, 16, 17, there is no issue with using the same $K$ for both $(L,g)$ and $(L',g)$, and since the inequality is strict on a compact set, it is preserved by suitably small perturbation of $(L,g)$.  
Properties 3, 18, and 19 are handled similarly except that some extra subtleties need to be addressed.  

Property 3 asserts the existence of polygonal paths $P_1$ and $P_2$ that in coordinates lie within $1/32$ of boundary edges of the squares $[-1,1] \times[-1,1]$ that parametrize $2$-cells.  Along, $P_1$ and $P_2$, all $-\nabla F_{i,j}$ must point into $N(e^1_\alpha)$ at points where $F_{i}>F_j$.  

In the case of a (PV) edge, this Property is maintained exactly as above.  For a (Cu) edge, the only subtlety is that for the sheets $F_k$ and $F_{k+1}$, the inequality (\ref{eq:strictgFX}) becomes non-strict at the point of $P_1$ that intersects the cusp edge.  However, the angle between $-\nabla F_{k,k+1}$ and $P_1$ remains bounded away from $0$ up to the cusp edge itself.  [In equation (\ref{eq:smooth1eq}) from Proposition \ref{prop:smooth1}, we will have $\sigma_{-}(k) = \sigma_{-}(k+1) = k-.5$.  Thus, for $x_1 \leq -1/4$, $F_{k,k+1}(x_1,x_2) = f_{k,k+1}(x_1)$, so that the gradient $-\nabla F_{k,k+1}$ points in the negative $x_1$ direction everywhere along the cusp edge.  From the proof of Proposition \ref{prop:Properties1cells}, $P_1$ has slope $.001$.]    Therefore, we can instead use (III) to see that this condition can be maintained.  

In the case of a (1Cr) or (2Cr) edge, any point $(x_1,x_2) \in P_l$ where distinct sheets $F_i$ and $F_j$ meet at a crossing, must be a vertex of $P_l$ where edges $Q_1$ and $Q_2$ meet.  Supposing $F_i\geq F_j$ above $Q_1$ and $F_j \geq F_i$ above $Q_2$, it is required that $-\nabla F_{i,j}$ (resp. $-\nabla F_{j,i}$) continues to point transversally to $Q_1$ (resp. $Q_2$) at $(x_1,x_2)$. 

To maintain this condition, for $(L',g') \in \mathcal{V}_U$ (for appropriate $\mathcal{V}$), we must modify $P_l$ to $P_l'$ because the crossing loci of $L$ and $L'$ are distinct (but $C^1$-close).  This is done by, at each such crossing point $(x_1,x_2)$, replacing $Q_1$ and $Q_2$ with new segments $Q_1'$ and $Q_2'$ as follows.  Since $-\nabla F_{i,j}$ is transverse to $Q_1$ at the crossing point $(x_1,x_2)$ it (as well as all other $-\nabla F_{k,l}$ with $F_k >F_l$) remains transverse to a small extension of the line segment $Q_1$ passed the crossing locus, call it $\widehat{Q}_1$, and to all close enough parallel translates of $\widehat{Q}_1$; a similar statement applies to $\nabla F_{j,i}$ and a small extension $\widehat{Q}_2$ of $Q_2$.  Take $\mathcal{V}$ small enough so that (along with satisfying the other requirements) all $L' \in \mathcal{V}_U$ have (i) $F'_{i,j}$ and $F'_{j,i}$ are still transverse to translations of $\widehat{Q}_1$ and $\widehat{Q}_2$ and (ii) replace $Q_1$ and $Q_2$ with portions of parallel translates of the $\widehat{Q}_{i}$ chosen to intersect the crossing locus of $L'$ at a common point.  With this modification $P_l'$, $(L',g)$ also satisfies Property 3.

%, an extra subtlety is that even when sheets defined by $F_k$ and $F_{k+1}$ crossings above $N(e^1_\alpha)$, we include the requirement that for $x \in \partial N(e^1_\alpha)$ where $F_k-F_{k+1} >0$, we include the requirement that $-\nabla (F_k-F_{k+1})$ points into $N(e^1_\alpha)$ along the polygonal path $P$ that forms part of the boundary of $N(e^1_\alpha)$ at the crossing locus, up to and including the crossing locus...  As observed above, the crossing locus undergoes a $C^1$-small perturbation when we change the local defining functions from $L$ to $L' \in \mathcal{V}_U$.  At the crossing locus of $L$ two line segments of $P$ meet at a vertex.  Say that $P_1$ is the segment that sits on the side where $F_{k} >F_{k+1}$.  Then, we can extend $P_1$ to a longer line segment, so that $F_{k}-F_{k+1}$ continues to point transversaly to $P_1$ even in the part where $F_k< F_{k+1}$.  Taking $\mathcal{V}_U$ small enough, we know that the crossing locus intersects $P_1$ in a unique spot, and we take this for the new $P_1$.  Do a similar procedure with $P_2$.  Thus, we have shown that for any $L'$ in appropriate $V_U$ we can redefine $P$ so that the Property continues to hold.

In Property 18, the only extra subtlety is in the statement about the directionality of $-\nabla (F_i -w)$  along the crossing locus.  This can be preserved  by choosing $\mathcal{V}$ so that the crossing loci of $L' \in \mathcal{V}_U$ remain suitably $C^1$-close to that of $L$.

In Property 19, there are two extra issues to address:  (i) The segments are required to start at switch points, and these locations may be distinct for $(L,g)$ and $(L',g')$.  (ii)  Again, we require an inequality to hold along the segments $L_{k+1,k+2}$ and $L_{k+2,k+1}$ up to the crossing locus.    

We discuss how to modify $L_{k+1,k+2}$ to $L'_{k+1,k+2}$ so that $L'_{k+1,k+2}$ starts at the switch point of $(L',g')$ and retains the desired properties.  (A similar argument applies to modify the segment $L_{k+2,k+1}$.)  
%Since the $L_{k+1,k+2} \cap \pi_x(L \setminus U)$ is compact, and  $-\nabla F_{k+1,k+2}$ and $-\nabla F_{k,k+2}$ are strictly transverse to $L_{k+1,k+2}$ on $L \setminus U$,  
Using a suitable combination of (II) and (III) ((III) is used since $-\nabla F_{k,k+2}$ becomes $0$ at the cusp edge), 
we can find a compact neighborhood of $L_{k+1,k+2}$, denoted by $N$, that is a union of line segments parallel to $L_{k+1,k+2}$ and such that the required transversality of $-\nabla F_{k+1,k+2}$ and $-\nabla F_{k,k+2}$ holds along each such segment.  
%(Transversality can be preserved in $U$ 
%(Again, it was important to remove $U$ since $-\nabla F_{k,k+2}=0$ at the $(k,k+2)$-cusp locus.)  
Taking  $\mathcal{V}$ to be small enough, the $(k+1,k+2)$-switch point of any $L' \in \mathcal{V}_U$ will continue to lie in $N$, and the required transversality will still hold on all segments in $N$.  Then, we can take $L'_{k+1,k+2}$ to be the line segment that is parallel to $L_{k+1,k+2}$ and starts at the $(k+1,k+2)$-switch point of $(L',g')$.

%the continue to hold on a 
%the location of switch points can be kept within any desired neighborhood of the original switch points by a suitable choice of $\mathcal{V}_U$.  Thus, we take the new line segments $L_{k+1,k+2}'$ and $L_{k+2,k+}'$ to be parallel to original $L_{k+1,k+2}$ and $L_{k+2,k+1}$.  Of course, the strict inequality will hold for all such segments in some open neighborhood of $L_{k+1,k+2}$ consisting of parallel translates of the original $L_{k+1,k+2}$.  Thus, it will hold on $L_{k+1,k+2}'$.

\medskip

\noindent {\bf Properties 13, 15.}  Statements about the location of the cusp locus of  $L$ hold as well for $L'$ since the cusp loci of all $(L',g') \in \mathcal{V}_U$ agree.  All remaining statements concern the direction of transversality of $-\nabla (F_i-F_k)$ and $-\nabla (F_{k+1}- F_j)$ to some $(k,k+1)$-cusp locus and the existence and location of non-degenerate switch points (which are precisely the tangencies of $-\nabla (F_i-F_k)$ and $-\nabla (F_{k+1}- F_j)$ to the cusp locus).  Here, $F_i$ and $F_j$ are defining functions for sheets located above or below the cusp edge respectively.  For notational convenience, we restrict attention to the case of $-\nabla(F_i-F_k)$.  

As a preliminary, note that from the standard form for defining functions near a cusp edge,  $\nabla F_k$ and $\nabla F_{k+1}$ are not differentiable along the cusp locus when considered as functions of two variables.  However, their restrictions to the cusp edge are $C^\infty$.  Since for $(L',g') \in \mathcal{V}_U$, $F_k'=F_k$ and $F_{k+1}'=F_{k+1}$ in $\pi_x(U)$, a suitable choice of $\mathcal{V}$ results in $-\nabla (F_i-F_k)$ and $-\nabla(F_i'-F_k')$ remaining $C^2$-close along the cusp locus.  
In particular, for any chosen neighborhood, $N$, of  the $(i,k)$-switch point, we can preserve the property of having a unique $(i,k)$-switch point in $N$, and maintain the direction of transversality of $-\nabla (F_{i}-F_k)$ to the cusp locus outside of $N$.  [Within $N$ the $(i,k)$-switch point is the unique zero of a lift to $\R$ of the circle-valued difference $\theta_1 -\theta_2$ where $\theta_1$ is the slope of the cusp locus and $\theta_2$ is the slope of $-\nabla(F_i-F_k)$.  The statement that switch points are non-degenerate is equivalent to $0$ being a regular value, so the existence of a unique, non-degenerate $0$ of $\theta_1 -\theta_2$ in $N$ is preserved as long as the perturbation of $-\nabla(F_i-F_k)$ is suitably $C^1$-small.]

%%In a neighborhood of a non-degenerate switch point, 
%
%%\footnote{Perhaps  a footnote in the background:  There is a subtle point, that the Gradient of a cusp sheet is not differentiable as a function of two variables because the local defining function is $C^1$ but not $C^2$.  However, the restriction of the gradient of a cusp edge sheet to the cusp edge it self is a $C^-\infty$ vector field (with value in $TS$) on the cusp locus.  Or, not.  A comment like this may already be written into the proof of the switch point property for the swallowtail square.}   
%
%%Explain why the locations of sw points is perturbed.  (Really it is similar to the Reeb chord claim, just in positive codimension.)
%
%%Recall the setup with the two angle functions $\theta_1$ and $\theta_2$, one for the tangent vector to the cusp edge, and one for the restriction of the local difference gradient.  The switch points are $0$'s of $\theta_1-\theta_2$, are non-degenerate provided $(\theta_1-\theta_2)'\neq 0$.  The location of such non-degenerate $0$'s is of course perturbed when $\theta_1-\theta_2$ is perturbed, and when we replace $F$ with $F'$, $\theta_1$ remains fixed while $\theta_2$ is perturbed.  

\medskip

\noindent {\bf Properties 4, 5, 6, 8, 9.}  These properties concern the location and index of critical points of the difference functions $F_i-F_j$.  Since all critical points are non-degenerate, it is standard that a suitably $C^2$-small perturbation of defining functions will result in a perturbation of the critical points of the $F_i-F_j$ that will preserve Morse indices.
\end{proof}

%\begin{proposition} \label{prop:lemmaA}
%Suppose that there exists some $\tilde{L}$ satisfying the preliminary\footnote{\dr{Double check that this terminology matches the background.}} transversality conditions and all of the conditions of Theorem \ref{thm:PropertiesofLtilde} except that $\tilde{L}$ may not be $1$-regular.  Then, Theorem \ref{thm:PropertiesofLtilde} holds.
%\end{proposition}

\begin{proof}[Conclusion of the proof of Theorem \ref{thm:PropertiesofLtilde}]
With $(\tilde{L}_0,g_0)$ as constructed in Sections \ref{sec:Constructions} and \ref{sec:ConstructionsST}, let $U \subset \tilde{L}_0$ be a neighborhood of the singular set of $L$ satisfying the conclusion of Proposition \ref{prop:LemmaA1}.  Then, let $\mathcal{V}$ be a neighborhood of $(\tilde{L}_0,g_0)$, as in Proposition \ref{prop:LemmaA2}, with respect to $U$.  The two Propositions show that within $\mathcal{V}$ there is a Legendrian metric pair, $(\tilde{L},g)$, that satisfies the Properties 1-19 and such that $\tilde{L}$ is $1$-regular with respect to $g$.  That $\tilde{L}$ was smooth with crossing arcs and cusp arcs as prescribed by the square types (1)-(14) was verified in Section \ref{sec:Properties}, and these conditions also continue to hold provided we take $\mathcal{V}$ small enough.
%, $(L',g)$ will also satisfy the first condition stated in Theorem \ref{thm:PropertiesofLtilde}.
\end{proof}


\begin{thebibliography}{99}

\bibitem{ArnoldGuseinZadeVarchenko}
V.I. Arnold, S. M. Gusein-Zade and A.N. Varchenko.
\newblock Singularities of differentiable maps. Volume 1.
{\em Modern Birkhauser Classics}, Monogr. Math., 82, Birkhäuser Boston, Boston, MA, 1985

%\bibitem{BanyagaHurtubise}
%Augustin Banyaga and David Hurtubise.
%\newblock Lectures on Morse Homology.
%{\em Kluwer Texts in the Mathematical Sciences}, vol. 29, Kluwer Academic Publishers Group, Dordrecht, 2004.



\bibitem{Chekanov02}
Yuri Chekanov.
\newblock Differential algebra of Legendrian links.
{\em Invent. Math.}, 150 (2002), no. 3, 441--483.


\bibitem{Rizell11}
Georgios Dimitroglou Rizell.
\newblock{Knotted Legendrian surfaces with few Reeb chords,} 
{\em Algebr. Geom. Topol.} 11 (2011), no. 5, 2903--2936.



\bibitem{Ekholm07}
Tobias Ekholm.
\newblock Morse flow trees and {L}egendrian contact homology in 1-jet spaces.
\newblock {\em Geom. Topol.}, 11:1083--1224, 2007.


\bibitem{EkholmEtnyreNgSullivan11}
Tobias Ekholm, John Etnyre, Lenny Ng, and Michael Sullivan.
\newblock{Knot contact homology}.
\newblock {\em Geom. Topol.} 17 (2013), no. 2, 975--1112.

\bibitem{EkholmEtnyreSullivan05b}
Tobias Ekholm, John Etnyre, and Michael Sullivan.
\newblock The contact homology of {L}egendrian submanifolds in {${\mathbb R}\sp
  {2n+1}$}.
\newblock {\em J. Differential Geom.}, 71(2):177--305, 2005.

\bibitem{EkholmEtnyreSullivan05a}
Tobias Ekholm, John Etnyre, and Michael Sullivan.
\newblock Non-isotopic {L}egendrian submanifolds in {$\mathbb R\sp {2n+1}$}.
\newblock {\em J. Differential Geom.}, 71(1):85--128, 2005.



\bibitem{EkholmEtnyreSullivan07}
Tobias Ekholm, John Etnyre, and Michael Sullivan.
\newblock Legendrian contact homology in {$P\times\mathbb R$}.
\newblock {\em Trans. Amer. Math. Soc.}, 359(7):3301--3335 (electronic), 2007.



\bibitem{Eliashberg00}
Yakov Eliashberg.
\newblock Invariants in contact topology,
\newblock  {\em Proceedings of the International Congress of Mathematicians}, 
Vol.~II (Berlin, 1998).  Doc. Math.  1998,  Extra Vol. II, 327--338.

\bibitem{EliashbergGiventalHofer00}
Yakov Eliashberg, Alexander Givental, and Helmut Hofer.
\newblock Introduction to symplectic field theory
\newblock {\em Geom. Funct. Anal.}, 2000,  Special Volume, Part II, 560--673.

\bibitem{Entov99}
Michael Entov.
\newblock Surgery on Lagrangian and Legendrian singularities
\newblock {\em Geom. Funct. Anal.}, 1999, Vol. 9, 298--352.


\bibitem{Fl}
Andreas Floer.
\newblock WittenÕs complex and infinite-dimensional Morse theory, 
\newblock {\em J. Differential Geom.}, 30 (1989) 207--221.

\bibitem{FukayaOh}
Kenji Fukaya and Yong-Geun Oh.
\newblock Zero-loop open strings in the cotangent bundle and Morse homotopy. 
\newblock {\em Asian J. Math.}, 1 (1997), no. 1, 96--180.



\bibitem{HarperSullivan}
John Harper and Michael Sullivan.
\newblock{A bordered Legendrian contact algebra},
\newblock {\em J. Symp. Geom.},  Volume 12 (2014), no.  2, 237--255.


%\bibitem{Kupka}
%Ivan Kupka.
%\newblock Contribution a la theorie des champs generiques, 
%{\em Contributions to Differential Equations} 2 (1963), 457--484 (French). 

%\bibitem{PalisdeMelo}
%Jacob Palis Jr. and Welington de Melo.
%\newblock Geometric theory of dynamical systems: an introduction,
%Springer-Verlag, New York-Berlin, 1982. xii+198 pp. ISBN: 0-387-90668-1 
%
\bibitem{RuSu1}
Dan Rutherford and Michael Sullivan.
\newblock{Cellular computation of Legendrian contact homology for surfaces, Part I.}
\newblock {\em Preprint} (2016).

\bibitem{RuSu3}
Dan Rutherford and Michael Sullivan.
In preparation.


\bibitem{Sivek11}
Steven Sivek.
\newblock{A bordered Chekanov-Eliashberg algebra,}
\newblock {\em J. Topology} 4 (2011), no. 1, 73--104.

%\bibitem{Smale}
%Stephen Smale.
%\newblock Stable manifolds for differential equations and diffeomorphisms, 
%{\em Ann. Scuola Norm. Sup.} Pisa (3) 17 (1963), 9--116.


\end{thebibliography}
\end{document}